\documentclass{amsart}
\usepackage[all]{xy}

\usepackage[a4paper, top=2cm, bottom=2cm, left=3cm, right=3cm, heightrounded, marginparwidth=3cm, centering]{geometry}
\usepackage{hyperref}
\usepackage{calligra,mathrsfs}
\usepackage{amsthm}
\usepackage[capitalise]{cleveref}
\usepackage{verbatim}
\usepackage{tikz-cd}
\usepackage{bbold}
\usepackage{amssymb}
\usepackage{amsaddr}
\usepackage[colorinlistoftodos, shadow]{todonotes}
\usepackage[T1]{fontenc}
\usepackage{enumerate}

\crefname{equation}{}{}
\Crefname{equation}{}{}
\crefname{align}{}{}
\Crefname{align}{}{}
\crefname{alignat}{}{}
\Crefname{alignat}{}{}
\crefname{gather}{}{}
\Crefname{gather}{}{}
\crefname{multline}{}{}
\Crefname{multline}{}{}
\crefname{flalign}{}{}
\Crefname{flalign}{}{}

\usepackage{adjustbox}

\usepackage[bbgreekl]{mathbbol}
\usepackage{amsfonts, amsmath}
\usepackage{stmaryrd}
\usepackage{scalerel}

\let\oldtocsubsubsection=\tocsubsubsection

\renewcommand{\tocsubsubsection}[2]{ \small\hspace{0em}\oldtocsubsubsection{#1}{#2}}

\numberwithin{equation}{subsection}

\DeclareSymbolFontAlphabet{\mathbb}{AMSb}
\DeclareSymbolFontAlphabet{\mathbbl}{bbold}

\DeclareRobustCommand{\SkipTocEntry}[5]{}

\theoremstyle{plain}
\newtheorem{theorem}{Theorem}[subsection]
\newtheorem{proposition}[theorem]{Proposition}
\newtheorem{lemma}[theorem]{Lemma}
\newtheorem*{claim*}{Claim}
\newtheorem{corollary}[theorem]{Corollary}

\newtheorem{question}[theorem]{Question}
\theoremstyle{definition}
\newtheorem{construction}[theorem]{Construction}

\newtheorem{definition}[theorem]{Definition}
\newtheorem*{claim}{Claim}
\newtheorem{example}[theorem]{Example}
\newtheorem{examples}[theorem]{Examples}
\newtheorem{remark}[theorem]{Remark}

\newcommand{\A}{\mathbb{A}}
\newcommand{\T}{\mathbb{T}}

\newcommand{\Yc}{\mathcal{Y}}

\newcommand{\an}{\mathrm{an}}
\newcommand{\An}{\mathrm{An}}

\newcommand{\AnSpec}{\mathrm{AnSpec}}
\newcommand{\GSpec}{\mathrm{GSpec}}
\newcommand{\Marc}{\mathcal{M}_{\mathrm{arc}}}

\newcommand{\arc}{\mathrm{arc}}

\newcommand{\C}{\mathbb{C}}

\newcommand{\cyc}{\mathrm{cyc}}

\newcommand{\DD}{\mathbb{D}}

\newcommand{\diam}{\diamond}
\newcommand{\dR}{\mathrm{dR}}
\newcommand{\dRall}{\mathrm{DR}}
\newcommand{\HK}{\mathrm{HK}}

\newcommand{\End}{\mathrm{End}}

\newcommand{\et}{\mathrm{{\acute{e}t}}}
\newcommand{\Ext}{\mathrm{Ext}}
\newcommand{\F}{\mathbb{F}}
\newcommand{\FF}{\mathrm{FF}}

\newcommand{\Fil}{\mathrm{Fil}}

\newcommand{\Gal}{\mathrm{Gal}}
\newcommand{\Gm}{{\mathbb{G}_m}}
\newcommand{\Ga}{\mathbb{G}_a}

\newcommand{\GL}{\mathrm{GL}}

\newcommand{\Hom}{\mathrm{Hom}}

\renewcommand{\Im}{\textrm{Im}}
\renewcommand{\inf}{\mathrm{inf}}
\newcommand{\id}{\operatorname{id}}

\newcommand{\Kum}{\mathrm{Kum}}

\newcommand{\Mod}{\mathrm{Mod}}
\newcommand{\Map}{\mathrm{Map}}
\newcommand{\N}{\mathbb{N}}

\newcommand{\Nil}{\mathrm{Nil}}

\newcommand{\Sh}{\mathrm{Sh}}

\newcommand{\op}{\mathrm{op}}
\newcommand{\Perf}{\mathrm{Perf}}

\newcommand{\Prof}{\mathrm{Prof}}
\newcommand{\Q}{\mathbb{Q}}

\newcommand{\red}{\mathrm{red}}

\newcommand{\Shv}{\mathrm{Shv}}

\newcommand{\sm}{\mathrm{sm}}
\newcommand{\lc}{\mathrm{lc}}
\newcommand{\Spa}{\mathrm{Spa}}

\newcommand{\Spf}{\mathrm{Spf}}
\newcommand{\Spec}{\mathrm{Spec}}
\newcommand{\solid}{{\scalebox{0.5}{$\square$}}}

\newcommand{\Gelf}{\mathrm{Gelf}}

\newcommand{\perf}{{\mathrm{perf}}}

\newcommand{\Z}{\mathbb{Z}}
\newcommand{\R}{\mathbb{R}}

\renewcommand{\O}{\mathcal{O}}

\newcommand{\un}{\mathrm{un}}
\newcommand{\pr}{\mathrm{pr}}
\newcommand{\Betti}{\mathrm{Betti}}

\newcommand{\fet}{\mathrm{f\acute{e}t}}
\newcommand{\proet}{\mathrm{pro\acute{e}t}}

\newcommand{\iHom}{\underline{\mathrm{Hom}}}

\newcommand{\Cat}[1]{\mathsf{#1}}
\newcommand{\ob}[1]{\mathrm{#1}}
\newcommand{\shf}[1]{\mathcal{#1}}

\newcommand{\VB}{\mathrm{VB}}

\newcommand{\qfd}{\mathrm{qfd}}
\newcommand{\Font}{\mathrm{Ft}}

\title{Analytic de Rham stacks of Fargues--Fontaine curves}
\author{Johannes Ansch\"utz, Guido Bosco, Arthur-C\'esar Le Bras, Juan Esteban Rodr\'iguez Camargo, Peter Scholze}
\date{\today}

\begin{document}

\begin{abstract}
  We define and initiate the study of analytic de Rham stacks of relative Fargues--Fontaine curves. To this end, we develop a theory of analytic de Rham stacks with sufficiently strong descent and approximation properties. Specializing to the de Rham stack of the Fargues–Fontaine curve attached to $\C_p$, we apply the general theory to obtain a new geometric proof of the $p$-adic monodromy theorem, avoiding any reliance on earlier results on $p$-adic differential equations. Building on the foundations established here, we plan in a sequel to investigate the cohomology of de Rham stacks of relative Fargues--Fontaine curves in geometric situations and, in particular, provide a stack-theoretic definition of Hyodo--Kato cohomology.
\end{abstract}

\maketitle

\tableofcontents

\section{Introduction}
\label{sec:introduction}

Let $p$ be a prime number.

\subsection{$p$-adic geometry, beyond perfectoid spaces}
\label{sec:new-defin-analyt}

A large part of this paper is devoted to developing some foundational work on derived $p$-adic geometry within the world of analytic stacks.
 This framework will be used critically by some of us to define an analogue of prismatic/syntomic cohomology (incarnated as coherent cohomology of certain stacks) for rigid analytic varieties.
 However, in the context of the present paper, its relevance lies in the fact that it serves to provide a new definition of analytic de Rham stacks in $p$-adic geometry.
 This extends the results of \cite{camargo2024analytic}, which developed a powerful theory of the analytic de Rham stack, providing the necessary tools for a robust theory of analytic $D$-modules on rigid analytic varieties (including, among other statements, Poincar\'e duality).\footnote{We stress right away that in this paper we use the terminology ``analytic $D$-module'' as a synonym for ``quasi-coherent sheaf'' on the analytic de Rham stack. The exact relation to $D$-modules will appear in \cite{AnDModRJRC}.}
 However, when applying \cite{camargo2024analytic} to more exotic spaces like Fargues--Fontaine curves (as we aim to do in this paper), one runs into technical issues making \cite{camargo2024analytic} not directly applicable.
 These technical questions circle around the descent of the functor $X\mapsto X^{\dR}$, or its commutation with limits.
 The new definition that we present in \cref{sec:perf-analyt-de} will solve these issues in a satisfactory manner.

We now provide a motivation for our construction.
 Let us recall that for a scheme $X$, its algebraic de Rham stack $X^{\dR}_{{\mathrm{alg}}}$ is a \emph{sheafification} of the functor $R\mapsto X(R_{\mathrm{red}})$.
 Sheafifying for the fpqc-topology and working in characteristic $p$, the functor $X\mapsto X^{\dR}_{\mathrm{alg}}$ is essentially\footnote{One needs to be careful with the exact definitions of the categories involved, but we content ourselves with the discussion that follows.} right adjoint to the perfection functor $Y\mapsto Y^{\mathrm{perf}}$.\footnote{Note that this notion of de Rham stack is different from that one of Drinfeld and Bhatt--Lurie in the theory of prismatization of formal schemes.}
One can observe that for semiperfect rings, which form a basis for the fpqc topology, we have $R_{\rm red}=R_{\rm perf}$ and that an fpqc-cover $R\to S$ of $\F_p$-algebras induces an exact sequence
\[
  0\to R_{\mathrm{perf}}\to S_{\mathrm{perf}}\to S_{{\mathrm{perf}}}\otimes_{R_{\mathrm{perf}}}S_{\mathrm{perf}}\to \ldots
\]
by passing to the filtered colimit over Frobenius. 
This implies that, for a scheme $X$,  the fpqc sheafification of the functor $R\mapsto X(R_{\red})$ is the functor $R\mapsto X(R_{\perf})$. 
This justifies  the adjunction formula
\[
  \Hom(\Spec(R),X^\dR_{\mathrm{alg}})=\Hom(\Spec(R_{\mathrm{perf}}),X).
\]
In particular, we see that $X\mapsto X^{\dR}_{\mathrm{alg}}$ commutes with limits (which a priori is non-obvious because of the involved sheafification).
 One can then also observe that $X^\dR_{\mathrm{alg}}$ depends itself only on the perfection of $X$, i.e., the functor represented by $X$ on perfect rings.

Our definition of an analytic de Rham stack is modelled on the above adjunction between perfection and the algebraic de Rham stack by replacing the perfection with ``perfectoidization''.
 More precisely, let $\Cat{AnStk}$ be the category of analytic stacks (as defined in \cite{AnStacks}, in particular in the ``light'' setting), and let $\Cat{Perfd}_{\Q_p}$ be the category of perfectoid spaces over $\Q_p$.
 It is not difficult to construct a functor (a naive version of ``perfectoidization'' or ``associated diamond'')
\[
  F\colon \Cat{AnStk}\to \Sh_v(\Cat{Perfd}_{\Q_p}),
\]
by sending an analytic stack $Y$ to the $v$-sheafification of $\Spa(R,R^+)\mapsto Y(\AnSpec((R,R^+)_\solid)$.
 One might hope that $F$ admits a right adjoint $G$, and that $X\mapsto G\circ F(X)$ realizes analytic de Rham stacks.
 But this hope is too naive: a basic desideratum for any reasonable theory of analytic de Rham stacks is that the analytic de Rham stack of the analytic affine line $\A^{1,\an}_{\Q_p}$ over $\Q_{p,\solid}$ is given by the quotient $\A^{1,\an}_{\Q_p}/\Ga^{\dagger}$, where $\Ga^\dagger$ is the overconvergent neighborhood of the origin in $\A^{1,\an}_{\Q_p}$ acting on $\A^{1,\an}_{\Q_p}$ by addition.
 This however is not realized by the functor $G\circ F$.
 Namely, the functor $F$ identifies the algebraic affine line $\AnSpec(\Q_p[T])$ with the analytic affine line $\A^{1,\an}_{\Q_p}$, and thus
$$G\circ F(\AnSpec(\Q_p[T]))=G\circ F(\A^{1,\an}_{\Q_p}),$$
i.e.\ the natural map from  $\A^{1,\an}_{\Q_p}$ to its analytic de Rham stack would extend to the algebraic $\A^{1}$, something we do not want. We can follow \cite{camargo2024analytic} here, and attack this problem (in a first approximation) by replacing the source of the functor $F$ by sheaves on so-called bounded $\Q_p$-algebras (\Cref{ss:BoundedRing}).
 These are $\Q_p$-algebras, which are roughly given by ``non-complete Banach algebras''.
 By definition, the algebraic and analytic affine line are identified as functors on bounded $\Q_p$-algebras, and moreover each bounded $\Q_p$-algebra $A$ admits a distinguished ideal, its $\dagger$-nilradical $\Nil^\dagger(A)$, which intuitively consists of all elements of spectral norm $0$.
 For a bounded $\Q_p$-algebra set $A^{\dagger-\red}:=A/\Nil^\dagger(A)$, the $\dagger$-reduction of $A$.
 Building on the theory of bounded $\Q_p$-algebras, \cite{camargo2024analytic} has developed the aforementioned theory of analytic de Rham stacks by $!$-sheafifying for a given $X$ the functor $A\mapsto X(A^{\dagger-\red})$ on bounded $\Q_p$-algebras.

However, for general bounded $\Q_p$-algebras $A$, the $\dagger$-reduction $A^{\dagger-\red}$ is still a rather general type of solid $\Q_p$-algebra. We will introduce a notion of \textit{Gelfand rings}, which ensures that the (uniform) completion of $A^{\dagger-\red}$ is a Banach $\Q_p$-algebra. These always admit nontrivial maps to perfectoid $\Q_p$-algebras, and hence play better in the desired adjunction.

Another simplification we make in this paper is that we endow all $\Q_{p,\solid}$-algebras with the induced analytic ring structure. This has two major effects. On the one hand, it makes many proofs easier as descendability becomes easier to check.\footnote{As far as only categories of quasi-coherent sheaves are concerned, the $D_{\hat\solid}$-theory of \cite{AMdescendPerfd} offers an option to allow more general descent on perfectoid spaces, but this does not seem to help with generalizing geometric constructions such as of the analytic de Rham stacks.} In fact, to further facilitate such descendability proofs, we make countability/separability/lightness assumptions throughout the paper.\footnote{Recently, Zelich and Aoki have constructed faithfully flat maps, which are not descendable, \cite{zelich2024faithfully}, \cite{aoki2024cohomology}. Hence, some restriction here like countable generation is necessary.} On the other hand, working only with induced analytic ring structures has the effect of identifying the usual rigid unit disc $\mathbb D_{\mathbb Q_p}=\Spa(\Q_p\langle T\rangle,\Z_p\langle T\rangle)$ with its Huber compactification $\overline{\mathbb D}_{\mathbb Q_p}$. This may seem undesired, but in fact on the level of de Rham cohomology it is helpful: namely, the well-behaved de Rham cohomology of the disc is its overconvergent de Rham cohomology. Usually, this is subtle to define as it seems to require the choice of an overconvergent structure. But in fact we show here that the analytic de Rham stack of $\overline{\mathbb D}_{\mathbb Q_p}$ -- which is also the analytic de Rham stack of $\mathbb D_{\mathbb Q_p}$ as defined in this paper -- computes overconvergent de Rham cohomology, thereby giving a canonical definition of the latter, \cref{sec:dagg-form-smooth-example-overconvergent-cohomology-of-rigid-disc}. In particular, when discussing adic spaces in this paper, we will always have to assume that they are partially proper (or work with $\dagger$-rigid spaces in the sense of \cref{DefDaggerRigid}).

The choice of using only induced analytic ring structures has as another consequence. Namely, we will usually replace localizing to \emph{open} subsets on the adic spectrum by localizing to \emph{closed} subsets of the Berkovich spectrum. This is perfectly feasible in the context of analytic stacks.

We are now in a position to define the target category of perfectoidization that we will use.

\begin{definition}
  \label{sec:new-defin-analyt-2-definition-light-perfectoid-tate-algebra-light-arc-sheaf}\
  \begin{enumerate}
  \item A perfectoid Tate algebra $A$ over $\Q_p$ is called \textit{separable} if $A$ is a separable Banach algebra over $\Q_p$. That is, $A$ admits a countable dense subset.
    \item If $A$ is a perfectoid Tate algebra over $\Q_p$, we let $\mathcal{M}(A)$ be the maximal Hausdorff quotient of $\Spa(A,A^\circ)$.\footnote{As for example discussed in \cite[Proposition 13.9]{scholze_etale_cohomology_of_diamonds}.} Equivalently, $\mathcal{M}(A)$ is the Berkovich spectrum of $A$ consisting of equivalence classes of continuous multiplicative seminorms $|-|\colon A\to \R_{\geq 0}$ with $|1|=1$.
    \item A morphism $A\to B$ of perfectoid Tate algebras over $\Q_p$ is an \textit{arc-cover} if the morphism $\mathcal{M}(B)\to \mathcal{M}(A)$ on Berkovich spaces is surjective. 
   \item We let $\Cat{ArcStk}_{\Q_p}$ be the category of hypersheaves of anima on the arc-site of separable affinoid perfectoid spaces over $\Q_p$.
  \end{enumerate}
\end{definition}

The category $\Cat{ArcStk}_{\Q_p}$ will serve as our replacement for the category of $v$-sheaves on perfectoid spaces over $\Q_p$, and we can now move to our restricted version of the category of analytic stacks.

 As mentioned before, we will have to restrict ourselves to bounded $\Q_p$-algebras $A$, in fact some full subcategory thereof.
 Associated to a bounded $\Q_p$-algebra $A$ (always seen as an analytic ring over $\Q_{p,\solid}$ through the induced analytic ring structure) are:
\begin{itemize}
\item its $\dagger$-nilradical $\Nil^\dagger(A)\subseteq A$,
\item its $\dagger$-reduction $A^{\dagger-\red}:=A/\Nil^\dagger(A)$,
\item its subring $A^{\leq 1}\subseteq A$ of (overconvergent) powerbounded elements, \cref{xjs82k},\footnote{In \cite[Definition 2.6.1]{camargo2024analytic} the ring of power-bounded elements is, on underlying anima, given by the subanima of maps $\Q_{p}\langle T\rangle\to A$. As the Tate algebra $\Q_p\langle T\rangle$ is uncountably generated over $\Q_p$ as a solid algebra, we replace it by its better behaved overconvergent variant $\Q_p\langle T\rangle_{\leq 1}:=\bigcup_{r>0} \Q_p\langle p^{r}T\rangle$.}

  \item its uniform completion $A^u:=(A^{\leq 1})^{\wedge_p}[1/p]$ (where $M^{\wedge_p}$ denotes the $p$-adic completion), \cref{xwj29e}.
\end{itemize}

Of course, we can map a bounded ring $A$ to the arc-sheafification of the functor $$R\mapsto \mathrm{Map}_{\Q_{p,\solid}-\textrm{alg.}}(A,R)$$ on separable perfectoid rings.
 It is however not necessarily true that a given $!$-cover $A\to B$ of bounded rings induces a surjection on associated arc-sheaves.\footnote{One can hope that a generalized notion of perfectoid Tate rings might solve this issue. However, in order to relate our theory to \cite{scholze2025geometrization} and \cite{scholze2024berkovichmotives} it seems easier to restrict instead the class of bounded $\Q_p$-algebras.} To solve this issue we study the following notion, which might be of independent interest.

\begin{definition}
  \label{sec:new-defin-analyt-1-gelfand-ring-intro}
  A bounded $\Q_p$-algebra $A$ is called \textit{Gelfand} if $A^u$ is a Banach $\Q_p$-algebra.
\end{definition}

We note that the condition that the bounded $\Q_p$-algebra $B:=A^u$ is Banach implies that $B^{\leq 1}/p$ is discrete (in contrast to being a rather general solid $\F_p$-algebra). Hence, being Gelfand puts a strong constraint on a bounded $\Q_p$-algebra. For a partial justification of our choice of terminology (motivated by conversations with Clausen), see \Cref{rk:uniformity-for-banach-rings}. 

\begin{example} 
\label{intro:ex-gelfand-ring}
An example of a bounded $\Q_p$-algebra which is not Gelfand is the ring $A=\Z_p[\![T]\!][1/p]$ (with topology induced by the $(p,T)$-adic topology on $\Z_p[\![T]\!]$). Indeed, $A^{\leq 1}/p=\mathbb{F}_p[\![T]\!]$ is not discrete. 
\end{example}

If $A$ is a Gelfand ring, we define its \textit{Berkovich spectrum} $\mathcal{M}(A)$ as the Berkovich spectrum of the Banach algebra $A^u$. We show, following \cite{AnStacks}:

\begin{proposition}[\Cref{xhs82j}]
Let $A$ be a Gelfand $\mathbb{Q}_{p}$-algebra with $\mathcal{M}(A)$ a metrizable compact Hausdorff space with finite cohomological dimension. There is a natural map of analytic stacks over $\mathbb{Q}_{p,\solid}$
\begin{equation*}
\ob{AnSpec}(A)\to \mathcal{M}(A)_{\ob{Betti}}\times_{\ob{AnSpec(\Z)}}{\ob{AnSpec}(\Q_{p,\solid})}
\end{equation*}
satisfying universal $!$-descent.  
\end{proposition}

To supplement the usefulness of Gelfand rings, we also prove the following theorem.
\begin{theorem}[\cref{xh28sj}]
  \label{sec:new-defin-analyt-2}
  Let $A$ be a Gelfand $\Q_p$-algebra.
 Then $A$ is Fredholm, i.e., each dualizable object in $\ob{D}(A)=\mathrm{Mod}_A\big(\ob{D}(\Q_{p,\solid})\big)$ is of the form $M\otimes_{A(\ast)}A$ for a perfect module $M$ over the animated ring $A(\ast)$.
\end{theorem}

As with perfectoid rings, we put separability assumptions. More precisely, we introduce the following notions.

\begin{definition}[\Cref{xhs9wj}]\
  \begin{enumerate}
  \item A Gelfand ring $A$ is called \textit{separable} if $A^u$ is a separable Banach algebra over $\Q_p$.
  \item The category of \textit{Gelfand stacks} $\Cat{GelfStk}$ is defined as the category of analytic stacks on (the opposite of) the category of separable Gelfand rings $\Cat{GelfRing_{\omega_1}}\subseteq \Cat{AnRing}_{\Q_{p,\solid}}$.
    \item For $A\in \Cat{GelfRing}_{\omega_1}$, we denote by $\ob{GSpec}(A)$ the functor corepresented by $A$ on separable Gelfand rings.
  \end{enumerate}
\end{definition}

In this definition, we used the notion of analytic stacks over (the opposite of) a suitable small subcategory of analytic rings; we refer to \Cref{sec:tdstacks} for the details.

The category of Gelfand stacks inherits a $6$-functor formalism from the category of analytic stacks over $\Q_p$, which we will use heavily.
 It contains fully faithfully the category of partially proper rigid analytic spaces over $\Q_p$, and even the larger category of \textit{derived Berkovich spaces} over $\Q_p$, a category which we define in \Cref{ss:GelfandAnalyticSpaces} and which generalizes the category of (good) Berkovich spaces over $\Q_p$;\footnote{For the notion of \textit{good} Berkovich space, see \cite{berkovich1993etale}.} this construction is analogous to the one of derived Tate adic spaces of \cite[Section 2.7]{camargo2024analytic}.

\begin{remark} 
  It is a natural question to wonder how the category $\Cat{GelfStk}$ relates to the category $\Cat{AnStk}_{\Q_{p,\solid}}$ of analytic stacks over $\Q_{p,\solid}$.
  There is a unique left exact colimit preserving functor $(-)_{\rm An}: \Cat{GelfStk} \to \Cat{AnStk}_{\Q_{p,\solid}}$, sending $\ob{GSpec}(A)$ to $\ob{AnSpec}(A)$ for $A$ a separable Gelfand ring.

  This functor may not be fully faithful on general Gelfand stacks, but it is on derived Berkovich spaces over $\Q_p$, which will be all we need in practice.

  The functor $(-)_{\rm An}$ has a right adjoint $(-)_{\rm Gelf}:  \Cat{AnStk}_{\Q_{p,\solid}} \to \Cat{GelfStk}$. To illustrate what it does, here are some simple examples: 
  \begin{itemize}
  \item $(\AnSpec(\Q_p[T]))_{\rm Gelf} = \mathbb{A}_{\Q_p}^{1, \mathrm{an}}:= \bigcup_{n} \GSpec(\Q_p\langle p^n T\rangle)$ is (the Gelfand stack incarnation of) the analytic affine line.
  \item $(\AnSpec(\Z_p((T))^{\wedge_p}[1/p]))_{\rm Gelf} = \emptyset$, namely, $\Z_p((T))^{\wedge_p}[1/p]$ does not admit any morphism as a solid $\Q_p$-algebra to a non-archimedean field extension of $\Q_p$
 (the ring $\Z_p((T))^{\wedge_p}[1/p]$ is known as the Amice ring in $p$-adic Hodge theory and arises geometrically as the completed residue field at the type (5) point in the closure of the affinoid unit disc, up to replacing $T$ by $T^{-1}$).
  \item Continuing \Cref{intro:ex-gelfand-ring}, $(\mathrm{AnSpec}(\Z_p[\![T]\!][1/p]))_{\rm Gelf}$ is the open unit disc over $\Q_p$ (seen as a Gelfand stack, via the fully faithful embedding of (partially proper) rigid analytic spaces in Gelfand stacks). In other words, the generic fiber of $\AnSpec(\Z_p[\![T]\!])$ in Gelfand stacks coincides with the rigid analytic generic fiber of the formal scheme $\mathrm{Spf}(\Z_p[\![T]\!])$ in the sense of Berthelot (one could generalize this statement to formal schemes locally of finite type over $\Z_p$).
  \end{itemize}
  We note that for the first example, working with stacks on bounded rings would have produced the same answer, but not for the third example (indeed, $\Z_p[\![T]\!][1/p]$ is already a bounded $\Q_p$-algebra). 
\end{remark}

For the study of Gelfand stacks, it is useful to replace the category of all separable Gelfand rings by a subclass of bounded rings forming a basis of the $!$-topology on separable Gelfand rings.
 With this in mind, we introduce the following notions.

\begin{definition}[\Cref{xkwjw8}]
Let $A$ be a Gelfand ring.

\begin{enumerate}

\item We say that $A$ is \textit{nilperfectoid} if $A^{\dagger-\ob{red}}$ is a perfectoid ring. We let $\Cat{GelfRing}^{\ob{nilperfd}}_{\Q_p}\subset \Cat{GelfRing}_{\Q_p}$ be the full subcategory of nilperfectoid Gelfand rings. 

\item  We say that $A$ is \textit{totally disconnected} if $\mathcal{M}(A)$ is a profinite set.
 We say that $A$ is \textit{strictly totally disconnected} if it is totally disconnected and all its completed residue fields are algebraically closed fields.
 We let $\Cat{GelfRing}^{\ob{td}}_{\Q_p}\subset \Cat{GelfRing}_{\Q_p}$ (resp.\ $\Cat{GelfRing}^{\ob{std}}_{\Q_p}$) be the full subcategory of (strictly) totally disconnected Gelfand rings.   
\end{enumerate}
\end{definition}

 We prove that for any separable Gelfand ring $A$, there exists a descendable map $A\to A^\prime$, which is quasi-pro-\'etale on uniform completions, with $A^\prime$ a separable nilperfectoid ring (\cref{LemBasisTopologyGelfandRings}). In particular, nilperfectoid rings form a basis of the $!$-topology. They play a role analogous to the role played by semiperfectoid rings in integral $p$-adic Hodge theory. In fact, they are in some sense a generalization of them: generic fibers of semiperfectoid rings are examples of nilperfectoid rings (see \Cref{sec:defin-main-prop-uniform-completion-and-perfectoidization}). We defined Gelfand stacks using all (separable) Gelfand rings rather than only (separable) nilperfectoid rings, to facilitate the comparison with $\dagger$-rigid spaces and more generally (derived) Berkovich spaces over $\Q_p$.

\begin{remark}
The definition of the diamond associated to a rigid analytic space $X$ over $\Q_p$ is based on the idea of only remembering this space through the collection of all maps from perfectoid spaces to it.
 This works as long as one only cares about ``topological'' information about $X$ (such as its \'etale cohomology), because $X$ becomes perfectoid locally for the pro-\'etale and $v$-topology.
 This substitution is very useful since perfectoid rings have nice cohomological properties (such as almost vanishing of higher cohomology of the integral structure sheaf).
 On the other hand, if one is interested in things such as (filtered) de Rham cohomology of smooth rigid spaces or locally analytic representations of $p$-adic Lie groups, one needs to consider functors defined on a category of rings with a possibly non-trivial (dagger) nilradical.
 Thus, we see that nilperfectoid rings are so to say the minimal choice of a category of rings combining both features, i.e.\ the notion of Gelfand stack is as close as possible to the notion of arc-stack (over $\Q_p$) while allowing non-trivial overconvergent thickenings; and we will see next that the difference is in some sense exactly measured by the construction of the analytic de Rham stack.
\end{remark}

\subsection{A new definition of analytic de Rham stacks}
With these preparations, we can come to (a first version of) the definition of the analytic de Rham stack.
 We define in \Cref{sec:perf-analyt-de-1} the \textit{perfectoidization} functor
$$ (-)^\diamond:   \Cat{GelfStk} \to \Cat{ArcStk}_{\Q_p},$$
as the colimit preserving functor sending $\ob{GSpec}(A),\ A\in \Cat{GelfRing}_{\omega_1}$ to the functor $\Marc(A):=(\ob{GSpec}(A))^\diamond$, which sends a separable perfectoid $\Q_p$-algebra $B$ to the anima $\Hom_{\Cat{GelfRing}_{\omega_1}}(A,B)$.
 Its right adjoint is the \textit{big analytic de Rham stack} functor
$$
(-)^{\rm DR}: \Cat{ArcStk}_{\Q_p} \to \Cat{GelfStk}.
$$

\begin{remark}
In fact, $(-)^\diamond$ is itself the \emph{right} adjoint of a functor $\widehat{(-)}\colon \Cat{ArcStk}_{\Q_p}\to \Cat{GelfStk}$ that realizes light arc-stacks (over $\Q_p$) geometrically.
 The functor $\widehat{(-)}$ is the left Kan extension of the functor that sends the arc-stack $\mathcal{M}_{\ob{arc}}(A)$ represented by a totally disconnected perfectoid ring $A$ to $\ob{GSpec} A$.
 	Via the functor $\widehat{(-)}$ and the $6$-functor formalism on $\Cat{GelfStk}$, one obtains a $6$-functor formalism for $\widehat{\mathcal{O}}$-cohomology on arc-stacks over $\Q_p$; this justifies the notation.
 See the end of \Cref{sec:arc-and-!-topologies-on-perfectoid-rings} for more on the relation to \cite{AMdescendPerfd} and \cite{ALBMFFCoho}. 
	\end{remark}

	\begin{example}
	Here is an example illustrating how these functors play with each other.
 Let $G$ be a compact $p$-adic Lie group.
 Let $G^{\rm la} = \GSpec(C^{\rm la}(G,\Q_p)) \in \Cat{GelfStk}$, with $C^{\rm la}(G,\Q_p)$ the ring of locally analytic functions on the group $G$.
 We have $(G^{\rm la})^\diamond = \underline{G}$, $\widehat{(G^{\rm la})^\diamond}= \GSpec(C(G,\Q_p))$ and the counit morphism $\widehat{(G^{\rm la})^\diamond} \to G^{\rm la}$ is the map induced by the inclusion
	$ C^{\rm la}(G,\Q_p) \to C(G,\Q_p)$.
	We also have $\underline{G}^{\rm DR} = G^{\rm sm}=\GSpec(C^{\rm lc}(G,\Q_p))$ (locally constant functions) and the unit morphism $G^{\rm la} \to ((G^{\rm la})^\diamond)^{\rm DR}$ is the map induced by the inclusion
	$ C^{\rm lc}(G,\Q_p) \to C^{\rm la}(G,\Q_p)$, \cref{ExamCondensedAnimadeRham}, \cref{sec:main-descent-result-1-descent-for-the-big-de-rham-stack}.
	\end{example}

	As the capital letters notation may suggest, $(-)^{\rm DR}$ is not what we will consider as the ``right'' version of the analytic de Rham stack functor in the present paper\footnote{We believe each version of the de Rham stack has its own advantages and drawbacks. We chose the one best suited to our needs, but for other questions the $\rm DR$ choice would have been reasonable too (and for many questions both versions can be used).}.
        One important, or at least very useful, aspect for the applications of the theory both in this text (\cref{sec:p-adic-monodromy}) and in future work is that the functor sending $X$ to its analytic de Rham stack should have strong descent properties: if $X \to X^\prime$ is an epimorphism of (light) arc-stacks over $\Q_p$, we want the induced map $X^{\dRall} \to X^{\prime,\dRall}$ to be an epimorphism of Gelfand stacks.
        It is however false in this general setting\footnote{Indeed, if it were the case we would have $X^{\dRall}=X_{\Betti}$ for $X$ a compact Hausdorff space (realized as an arc-stack on the left hand side); however this is not true, see \cref{ExampleCondensedAnimanodeRham}.}.
        For reasons to be explained in the main body of the paper, the situation would simplify considerably if separable strictly totally disconnected rings would form a basis of the $!$-topology on separable Gelfand rings.
        But it is unfortunately not true that all separable Gelfand rings $R$ admit a descendable map to such a ring (\cref{sec:-covers-perfectoids-1-counter-example-to-cover-by-strictly-totally-disconnected}).
        Nevertheless, we do prove such a statement under some sort of finiteness assumption on $R$: we show it when $R$ is \textit{quasi-finite dimensional} (\textit{qfd} for short), i.e., its uniform completion is quasi-pro-\'etale over a finite-dimensional analytic affine space over $\Q_p$.
        
        \begin{remark}
          \label{sec:new-defin-analyt-1-remark-finite-dimensionality}
          To ensure the existence of a descendable cover by a strictly totally disconnected ring, we will see that finite cohomological dimension of the Berkovich space is enough (\cref{LemBasisTopologyGelfandRings}).
          But we want to work with a class of separable Gelfand rings having such covers and which is stable by fiber products, and the last condition would not be guaranteed if we considered only the condition of having a Berkovich space of finite cohomological dimension.
          One could instead restrict to separable Gelfand rings $A$ such that each Berkovich residue field of $A^u$ has finite geometric transcendence degree over $\Q_p$, which implies finite cohomological dimension of the Berkovich space.
          But to get hyperdescent, rather than mere descent, we will need to use the stronger qfd condition in \cref{Subsection:HyperdescentdR}.
    \end{remark} 

        For this reason, we replace in the above the categories $\Cat{ArcStk}_{\Q_p}$, $\Cat{GelfStk}$ by their versions 
 $$
 \Cat{ArcStk}_{\Q_p}^{\rm qfd}, ~ \Cat{GelfStk}^{\rm qfd}
 $$
 defined on qfd separable perfectoid rings, resp.\ qfd separable Gelfand rings,  instead of all separable perfectoid rings, resp.\ all separable Gelfand rings.
  We still have a similarly defined perfectoidization functor (denoted by the same letter):
 $$ (-)^\diamond:   \Cat{GelfStk}^{\rm qfd} \to \Cat{ArcStk}_{\Q_p}^{\rm qfd},$$
 
 \begin{definition}
The qfd perfectoidization functor has a right adjoint, that we call the \textit{analytic de Rham stack} functor, and denote
 $$ (-)^\dR:     \Cat{ArcStk}_{\Q_p}^{\rm qfd} \to \Cat{GelfStk}^{\rm qfd}.
 $$
 If $X$ is a qfd Gelfand stack, we will abuse notation and write 
 $$
 X^\dR= (X^\diamond)^\dR.
 $$
 There is also a relative version: if $X \to Y$ is a morphism of qfd Gelfand stacks, we denote by
 $$
 X^{\dR/Y} = X^\dR \times_{Y^\dR} Y
 $$
 the \textit{relative analytic de Rham stack of $X$ over $Y$}.
 \end{definition}

Here are some important properties of the analytic de Rham stack that we prove.

\begin{theorem}\label{main-thm-intro-dr-stack}
\begin{enumerate}
\item (\Cref{TheoMaindeRham2})
The functor $(-)^{\dR}$ commutes with limits and colimits.
 In particular, if $Y\to X$ is an epimorphism of qfd Gelfand stacks, the map $Y^{\dR}\to X^{\dR}$ is an epimorphism of qfd Gelfand stacks.
\item (\Cref{PropdeRhamBerkovich})
Let $X$ be a qfd derived Berkovich space over $\Q_p$, in the sense of \Cref{DefinitionBerkovichSpaces} (for example, a classical Berkovich space is qfd as an arc-stack).
 Then $X^{\dR}$ is the sheafification for the $!$-topology of the qfd Gelfand stack sending $A\in \Cat{GelfRing}^{\qfd}_{\omega_1}$ to $X(A^{\dagger-\red})$.
 Furthermore, if $X\to  Y$ is a map of qfd derived Berkovich spaces whose associated map $X^{\diamond}\to Y^{\diamond}$ of arc-stacks is an immersion, $X^{\dR/Y}= Y^{\dagger_{X}}$.
 In particular, $X^{\dR}=(Y^{\dagger_X})^{\dR}$, where $Y^{\dagger_{X}}$ is the overconvergent neighbourhood of $X$ in $Y$ (cf. \Cref{DefDaggerNeigh}).
\item (\Cref{thm:prim-descendable-berkovich}) Let $K$ be a qfd complete non-archimedean field over $\Q_p$, separable as a $\Q_p$-Banach space, and $X$ a $\dagger$-rigid space over $K$ (see \Cref{DefDaggerRigid}; equivalently a dagger space in the sense of Gro\ss e-Kl\"onne \cite{grosse2000rigid}).
 Then the natural map $X\to X^{\dR/K}$ is an epimorphism of Gelfand stacks and we have a presentation as qfd Gelfand stacks
\[
X^{\dR/K}= \varinjlim_{[n]\in \Delta^{\op}} (X^{\times_K n+1})^{\dagger_{\Delta X}}.
\]
In particular, 
$$
(\mathbb{A}_{\Q_p}^{1, \rm an})^{\rm dR}= \mathbb{A}_{\Q_p}^{1, \rm an}/\mathbb{G}_a^\dagger,
$$
with $\mathbb{G}_a^\dagger=\GSpec(\Q_p\langle T \rangle_{\leq 0})$, the overconvergent neighborhood of the origin, acting by translations.

\end{enumerate}
\end{theorem}

Some remarks are in order.

\begin{remark}
\label{intro-rmk-after-main-thm1}\
\begin{itemize}
\item The assertions of the theorem also hold for the big de Rham stack without any qfd assumption, with the critical exception of commutation with colimits in point (1). 
Point (1) of \Cref{main-thm-intro-dr-stack} allows one to compute some examples. For example, it shows that 
$\Marc(\mathbb{C}_p)^\dR = \GSpec(\overline{\Q}_p)$.
 It also implies that if  $G$ is a locally profinite group, and  $B\underline{G} \in  \Cat{ArcStk}_{\Q_p}^{\rm qfd}$ its classifying (arc-)stack, then
$$(B\underline{G})^{\dR} = BG^{\rm sm},$$
with $G^{\rm sm}$ the Berkovich space with underlying topological space $G$ and sheaf of functions given by the locally constant functions of $G$ with values in $\Q_p$.
This example works the same for the big analytic de Rham stack.

\item Point (2) of \Cref{main-thm-intro-dr-stack} shows that for qfd derived Berkovich spaces over $\Q_p$, the analytic de Rham stack as defined in this paper agrees with the one defined in \cite{camargo2024analytic}, in a sense made precise in \Cref{ss:GelfandAnalyticSpaces}.

\item The last assertion of point (2) of \Cref{main-thm-intro-dr-stack} can be regarded as a general version of Kashiwara's lemma for the analytic, rather than algebraic, de Rham stack (overconvergent completions replacing formal completions). 

\end{itemize}
\end{remark}

\subsection{$6$-functor formalism for the analytic de Rham stack}
One important benefit of defining the analytic de Rham stack as a (certain kind of) analytic stack is that one gets for free the associated cohomology theory, i.e.\ (overconvergent) de Rham cohomology in this case (cf. \Cref{sec:coh-dR-stack-and-dR-coh}), a category of coefficients equipped with 6 operations, namely quasi-coherent sheaves on this stack.
 As usual, to make it useful, one needs to exhibit examples of morphisms which behave well from the point of view of the 6 operations like smooth or proper morphisms should: in technical parlance, cohomologically smooth and proper morphisms (see e.g.  \cite{heyer20246functorformalismssmoothrepresentations} for precise definitions).
 We first introduce the following convenient notion.
 
 \begin{definition}[\Cref{DefLQFDim}, \Cref{DefRigSmoothRigEtale}]
  Let $f\colon Y\to X$ be a morphism of qfd arc-stacks over $\Q_p$. 
 \begin{enumerate}
 \item  We say that $f$ is \textit{locally of quasi-finite dimension} (or \textit{lqfd}) if there is a strict closed cover $Y_i \subset Y$ (i.e.\ a cover by closed subspaces that can be refined by an open cover) such that the maps $Y_i\to X$ are qcqs and quasi-pro-\'etale over some relative affine space $Y_i\to \mathbb{A}^{d,\diamond}_{X}$ (with $d$ possibly depending on $i$).
\item We say that $f$ is \textit{\'etale} if, locally in the arc-topology of  $X$ and the analytic topology of $Y$, $f$ factors as a composite of open immersions (\cite[Definition 4.21]{scholze2024berkovichmotives}) and finite \'etale maps.
  We say that $f$ is \textit{smooth} if, locally in the arc-topology of $X$ and the analytic topology of $Y$, $f$ factors as a composite $ Y\xrightarrow{g} \mathbb{A}^{d,\diamond}_{X}\to X$ where $g$ factors as a composite of finite \'etale maps and open immersions. 
  \end{enumerate}
\end{definition}

For instance, usual smooth, resp.\ \'etale, maps of classical Berkovich spaces over $\Q_p$ induce smooth, resp. \'etale maps on the associated arc-stacks over $\Q_p$.

We then prove the following theorem.

\begin{theorem}
\label{intro-main-th-6ff-for-dr-stacks}\
\begin{enumerate}
\item (\Cref{LemmLqfd}) Let $f\colon Y\to X$ be a morphism of qfd arc-stacks over $\Q_p$, which is lqfd arc-locally on $X$.
 Then  $f^{\dR}\colon Y^\dR\to X^\dR$ is $!$-able. If in addition $f$ is proper, then $f^{\dR}$ is cohomologically proper.
\item (\Cref{sec:geom-prop-analyt-1-cohomological-smoothness-for-smooth-rigid-spaces})
   Let $f\colon Y\to X$ be a smooth morphism of qfd arc-stacks over $\Q_p$. Then the map $Y^{\dR}\to X^{\dR}$ is cohomologically smooth, with dualizing sheaf given by $1_{Y^\dR}[2d]$ if $f$ is of pure dimension $d$. If $f$ is \'etale, then $f^{\dR}$ is in addition cohomologically \'etale. 
  \item (\Cref{sec:geom-prop-analyt-1-cohom-smoothness-of-perfectoid-unit-disc})
    Let $X=\mathring{\mathbb{D}}^{\times,\diamond}_{\infty}=\varprojlim_{x\mapsto x^p} \mathring{\mathbb{D}}^{\times,\diamond}_{\mathbb{Q}_p}$ be the open punctured pre-perfectoid unit disc seen as an arc-stack. Then, the map $X^{\dR}\to \GSpec(\mathbb{Q}_p)$ is cohomologically smooth with dualizing sheaf isomorphic to $1_{X}[2]$.
\end{enumerate}
\end{theorem}

Once again, the statement of the theorem deserves some remarks. 

\begin{remark}\
\begin{itemize}
\item Points (1) and (2) of \Cref{intro-main-th-6ff-for-dr-stacks} give in particular finiteness and Poincar\'e duality for de Rham cohomology of smooth proper rigid varieties over $\Q_p$, even with coefficients (``analytic $D$-modules'').
 More generally, one gets Poincar\'e duality in the non-proper case, using compactly supported de Rham cohomology (realized by the lower-shriek functor along de Rham stacks).
\item Let $X$ be a partially proper smooth rigid variety over $\Q_p$.
 In this paper, we do not make precise the relation between quasi-coherent sheaves on $X^{\rm dR}$ and analytic $D$-modules, e.g.\ as defined in the work of Ardakov--Wadsley, whence the quotation marks in the previous item.
 This relationship will be analyzed in detail in \cite{AnDModRJRC}.
\item Continuing the previous item, if one thinks to quasi-coherent sheaves on $X^\dR$ for $X$ a smooth rigid space over $\Q_p$ as some category of $D$-modules on $X$, it might be surprising at first sight  that such a notion extends to much more general Berkovich spaces, such as perfectoid spaces, and only depends on the underlying arc-stack! The situation should be compared with the situation over the complex numbers, where the analytic de Rham stack of a complex manifold is literally isomorphic to its Betti stack (\cite[Theorem II.3.1]{ScholzeRLL}).
 In both cases, it is really the ``analyticity'' condition on $D$-modules that makes this phenomenon possible.  
\item The example from point (3) of \Cref{intro-main-th-6ff-for-dr-stacks} is used later in the paper in \Cref{sec:de-rham-fargues}, see \Cref{intro:sec-drff-and-p-adic-monodromy}, to show cohomological smoothness results for the de Rham stack of the relative Fargues--Fontaine curve.
 Cohomological smoothness of $X^\dR$ in this case might seem unexpected: inverse limits of cohomologically smooth maps are rarely cohomologically smooth again (e.g., the uncompleted perfectoid open unit disc is \textit{not} cohomologically smooth).
\end{itemize}
\end{remark}

\subsection{de Rham stacks of Fargues--Fontaine curves and the $p$-adic monodromy theorem}
\label{intro:sec-drff-and-p-adic-monodromy}

Since the analytic de Rham stack can be defined for any (light, qfd) arc-stack, it makes sense to consider more exotic examples than smooth rigid varieties over $\Q_p$, even though the latter, such as flag varieties relevant among others for versions of Beilinson--Bernstein localization, are the first to come to mind.
 One key class of examples for this paper are relative Fargues--Fontaine curves: while they are honest Berkovich (usually, adic) spaces for partially proper perfectoid spaces of characteristic $p$, they are for more general inputs only defined as arc-stacks over $\Q_p$.
 More precisely, we introduce the following definition, where we denote by $\Cat{ArcStk}^{\qfd}_{\F_p}$ the category of qfd arc-stacks over $\F_p$.

\begin{definition}
  
  Let $X\in \Cat{ArcStk}_{\F_p}^{\rm qfd}$.
 We define the following qfd arc-stacks over $\Q_p$:
  \begin{enumerate}
 
  \item The open punctured curve\footnote{Also called the ``punctured Fargues--Fontaine disc''.} $\Yc_{X}:=X\times_{\Marc(\F_p)} \Marc(\Q_p)$.
  
  \item The Fargues--Fontaine curve $\FF_{X}:= \Yc_{X}/\varphi_{X}^\Z$, where $\varphi_X$ is the Frobenius of $X$.
  \end{enumerate}

Given $X$ as before, we can form the analytic de Rham stack $\mathcal{Y}_X^{\dR}$, and the analytic de Rham stack of $\FF_{X}$, called the \textit{Hyodo--Kato stack of $X$} and denoted
$$X^\HK:=\FF_X^\dR.$$

 For $X\in \Cat{GelfStk}^{\qfd}$ we write $\mathcal{Y}_X^{\dR}:=\mathcal{Y}_{X^{\diamond}}^{\dR}$ and $X^\HK:= (X^{\diamond})^{\HK}$.
 \end{definition}
 
 \begin{remark}
 \label{rel:cn}
 In a sequel to this paper, we will show that Hyodo--Kato stacks geometrize Colmez--Nizio\l{}'s Hyodo--Kato cohomology for rigid analytic varieties over $p$-adic fields, \cite{colmez_niziol_basic}, in the sense that the coherent cohomology of the Hyodo--Kato stack computes Hyodo--Kato cohomology. This explains our choice of terminology.
 \end{remark}

 We establish some basic properties of this construction. For example, we prove:

 \begin{theorem}[\Cref{sec:defin-first-prop-3-smoothness-primness-of-analytic-de-rham-stacks}]
 \label{thm:intro-hk}
  \label{sec:intro-defin-first-prop-3-smoothness-primness-of-analytic-de-rham-stacks}
  Let $f:Y\to X$ be a map of qfd arc-stacks over $\Z_p$ and let $f^{\HK}\colon Y^{\HK}\to X^{\HK}$ be its associated morphism of Hyodo--Kato stacks.
  \begin{enumerate}
   \item Suppose that, locally in the arc-topology of $X$, $f$ is lqfd (see \cref{DefLQFDim}), then $f^{\HK}$ is $!$-able.
 Furthermore, if $f$ is proper then $f^{\HK}$ is cohomologically proper.
\item Suppose that $f$ is smooth of pure relative dimension $d$, then $f^{\HK}$ is cohomologically smooth, that is, it is suave and its dualizing sheaf is invertible. Furthermore, the dualizing sheaf is given by the $d$-th tensor power of the Tate twist (\cref{Tatetwist}).
 Moreover, if $f$ is \'etale, then it is cohomologically \'etale.
  \end{enumerate}

\end{theorem}

The statement is not difficult to prove after the work of the previous sections.
In particular, the cohomologically smoothness of (2) reduces to \Cref{intro-main-th-6ff-for-dr-stacks}(3), as was already pointed out.
The computation of the dualizing sheaf requires some more work, and it will follow from the general Poincar\'e duality statement of \cite[Theorem 1.2.12]{zavyalov2023poincaredualityabstract6functor} after proving that the Hyodo--Kato cohomology is ball-invariant, satisfies excision, and after constructing a theory of first Chern classes.
A similar statement holds for the induced morphism $\mathcal{Y}_Y^{\dR}\to \mathcal{Y}_X^{\dR}$, and in fact, it follows from it.

\medskip

Finally, we specialize the discussion to the case of a field extension $L\subseteq \C_p$ of $\Q_p$.
For such an $L$, we simply write $L^\HK=\Marc(L)^\HK$ (noting that this only depends on the completion of $L$).
We construct in \Cref{sec:p-adic-monodromy} a $\Gal_{\Q_p}^{\rm sm}$-equivariant morphism
\[
  \Psi\colon \C_p^\HK \to B_{\overline{\F}_p^{\HK}}\mathbb{V}(-1)
\]
with $B_{\overline{\F}_p^{\HK}}\mathbb{V}(-1) $ the classifying stack over $\overline{\F}_p^{\HK}\cong \GSpec(\Q_p^\un)/\varphi^\Z$ of the geometric vector bundle $\mathbb{V}(-1)$ associated with the simple isocrystal $D_{-1}=(\Q_p^\un,p\varphi)$ of slope $-1$.
Note that vector bundles on $B_{\overline{\F}_p^{\HK}}\mathbb{V}(-1) $ are naturally equivalent to $(\varphi,N)$-modules over $\Q_p^\un$, by Cartier duality.

\begin{remark}
Concretely, to define $\Psi$, we need to construct a (non-split) $\Gal_{\Q_p}$-equivariant extension of vector bundles of $\mathcal{O}(-1)$ by $\mathcal{O}$ on $\C_p^\HK$. One way to obtain it (or rather its opposite) is by considering the first-degree cohomology of the Hyodo--Kato stack of Tate's elliptic curve (\cref{LemComputationTateCurve}). It can also be described very concretely, which happens to be useful for computations (\cref{ConstructionMonodromy}).
\end{remark}

We are now in a position to state our version of the $p$-adic monodromy theorem precisely.

\begin{theorem}[\Cref{localmonodromy}]\label{intro:localmonodromy}
  The pullback along $\Psi\colon \C_p^\HK\to B_{\overline{\F}_p^{\HK}}\mathbb{V}(-1)  $ defines a $t$-exact equivalence
  \[
    \Psi^\ast\colon \Perf(B_{\overline{\F}_p^{\HK}}\mathbb{V}(-1))\overset{\sim}{\longrightarrow} \Perf(\C_p^\HK).
  \]
  In particular, for any $\Q_p\subseteq L\subseteq \overline{\Q}_p$ the category of vector bundles on $L^\HK$ is naturally equivalent to the category of $(\varphi,N,\Gal_L)$-modules over $\Q_p^\un$.
 \end{theorem}
 
 We refer to \Cref{sec:p-adic-monodromy-2-the-p-adic-monodromy-theorem} for a more detailed discussion of the relation between \Cref{intro:localmonodromy} and the $p$-adic monodromy theorem of Andr\'e, Kedlaya, and Mebkhout, and for a summary of its proof.
 Here, we would only like to stress that both the statement and the proof of the theorem make no use of $p$-adic differential equations: it is only when unraveling the statement (via some explicit description of $\Q_p^{\mathrm{\cyc},\HK}$) that they appear (\cref{lemma:relation-p-adic-differential-equations}).
 Rather, our proof makes extensive use of the formalism developed in this paper, such as descent and the $6$-functor formalism of de Rham stacks, and the geometry of Banach--Colmez spaces. We note, in addition, that by a spreading-out argument, \cref{intro:localmonodromy} formally implies a version of the $p$-adic monodromy theorem locally at classical points of a rigid space over $\Q_p$ (\cref{atclassicalpoints}).
 
 \begin{remark}
 Continuing \Cref{rel:cn}, we recall that Colmez--Nizio\l{}'s Hyodo--Kato cohomology takes values in $(\varphi,N)$-modules over $\Q_p^{\rm un}$ (with a Galois action when the rigid space is defined over a finite extension of $\Q_p$). This was for us a strong indication that our stacky construction should relate to Hyodo--Kato cohomology.
 \end{remark}

 \begin{remark}\label{RemarkRigidFIsocristals}
 Let $X$ be a variety over $\F_p$, and regard it as a qfd arc-stack over $\F_p$. We expect the category of quasi-coherent sheaves $\ob{D}(X^{\HK})$ to be a category of coefficients for overconvergent $F$-isocrystals on $X$; in particular, this stack should compute the rigid cohomology of $X$. This will be explored in the forthcoming PhD thesis of Junhui Qin. \end{remark}

\subsection{Plan of the paper} 
\Cref{s:BoundedRing} starts with some general results about bounded rings, recalling and complementing the results of \cite{camargo2024analytic}. In \Cref{s:gelfand-rings}, we discuss the important notion of Gelfand ring and show in particular that one can attach to them a Berkovich spectrum with good properties. These two sections, which are more dry than the others, develop the necessary (derived, topological) algebra needed for the rest of the paper. Many of the ideas presented here arose from the joint project between A., L.B., R.C. and S. on the analytic prismatization, and the work of one of us on the geometrization of the real local Langlands correspondence \cite{ScholzeRLL}.

\Cref{sec:perf-analyt-de} and \Cref{sec:d-modules-quasi} construct a robust theory of analytic de Rham stacks in $p$-adic geometry. In \Cref{sec:perf-analyt-de}, the important notion of (qfd) Gelfand stack is introduced and the analytic de Rham stack functor is defined as a right adjoint to the perfectoidization functor. We also introduce the category of derived Berkovich spaces over $\Q_p$ and explain some properties of their analytic de Rham stacks. Finally, we show a descent result for the formation of the analytic de Rham stack. At the end of \Cref{sec:perf-analyt-de}, \Cref{main-thm-intro-dr-stack} is proved, with the exception of the commutation of the analytic de Rham stack functor with colimits, shown in \Cref{sec:d-modules-quasi}. This section is devoted to the $6$-functor formalism of analytic de Rham stacks. We provide examples of cohomologically proper and smooth morphisms, proving in particular \Cref{intro-main-th-6ff-for-dr-stacks}, and show various properties of the cohomology of the analytic de Rham stack, such as its coincidence with de Rham cohomology for smooth overconvergent rigid spaces. \Cref{sec:d-modules-quasi} ends with a proof of arc-hyperdescent for the de Rham stack, the part of the paper where the qfd assumption is most critically used. 

Hyodo--Kato stacks and their $6$-functor formalism are the subject of \Cref{sec:de-rham-fargues}. This includes \Cref{thm:intro-hk} above. Our version of the $p$-adic monodromy theorem, cf. \Cref{intro:localmonodromy}, is proved in \Cref{sec:p-adic-monodromy}.

Finally, \cref{appendix:solid-functional-analysis} contains basic facts about nuclear and $\omega_1$-compact solid modules, used in  \Cref{sec:d-modules-quasi}.

\subsection{Conventions and notations}
\label{sec:conventions}

In general, we work in the context of $\infty$-categories (following \cite{lurie_higher_topos_theory}, \cite{lurie_higher_algebra}) and we consider \emph{derived} constructions, e.g., for global sections, tensor products or quotients.
Moreover, we work in the context of \emph{light} condensed mathematics, and more specifically in the \emph{light} solid setup, e.g., classical Banach spaces over $\Q_p$ are seen as solid $\Q_p$-modules.
Occasionally, we refer to papers based on \emph{non-light} condensed mathematics.
We explain in the beginning of \cref{ss:BoundedRing} how this reasoning can be justified.

Given a condensed anima $X$, a subobject of $X$ is by definition a condensed anima $Y$ with a monomorphism $Y\to X$.
In this case, we also write $Y\subseteq X$. Given $R$ a discrete animated ring (or discrete $\mathbb{E}_{\infty}$-ring spectrum), we let $\ob{D}^{\delta}(R)$ denote the symmetric monoidal category of $R$-modules in spectra $\Cat{Sp}$, and let  $\ob{D}(R)$ be the symmetric monoidal category of $R$-modules in condensed spectra $\ob{Cond}(\Cat{Sp})$. Given an analytic ring $A$, we let $A^{\triangleright}$ denote its underlying condensed ring. We say that $A$ is an \textit{analytic ring structure on $A^{\triangleright}$}.  We have a colimit preserving symmetric monoidal fully faithful embedding $\ob{D}^{\delta}(A(*))\to \ob{D}(A)$; an $A$-module $M$ is called \textit{discrete over $A$} if it belongs to the essential image of this functor.

\subsection{Acknowledgements}  For useful discussions related to the content of this paper, we are grateful to Ko Aoki,  Bhargav Bhatt, Dustin Clausen, Pierre Colmez, Wies{\l}awa Nizio{\l}, C\'edric P\'epin, Alexander Petrov, Andrea Pulita, Junhui Qin,  and Alberto Vezzani. The third author would like to thank Akhil Mathew for giving him the opportunity to present the results of this paper at the University of Chicago, and for providing very useful feedback on it.

Some of the ideas in this paper were conceived during the Hausdorff trimester program ``The arithmetic of the Langlands program'' in 2023, and the authors would like to thank the HCM for making such an event possible. Work on this project was carried out at several institutions, including the Max Planck Institute for Mathematics in Bonn, the Institute for Advanced Study, Princeton University, l'IRMA (Universit\'e de Strasbourg), and LMO (Universit\'e d'Orsay); we thank these institutions for their hospitality and support. The fourth author would also like to thank Columbia University and the Simons Foundation for supporting him during the academic year 2023--2024 as a Junior Simons Fellow.

Special thanks go to the participants of the ARGOS seminar held at the Max Planck Institute in Bonn during the winter term 2025/2026 for all their valuable feedback and criticism, and to Vincent Pilloni for his careful remarks on a previous version of the paper!

\newpage
\section{Generalities on bounded rings}\label{s:BoundedRing}
In this section we review and expand the theory of bounded rings over $\mathbb{Q}_p$ introduced in \cite{camargo2024analytic}.

\subsection{Bounded rings}\label{ss:BoundedRing}

In the following we work within the theory of light condensed mathematics\footnote{We will usually say \textit{light profinite set}, but otherwise the adjective ``light'' will always be implicit.} and  analytic stacks as in \cite{AnStacks}.
 All rings are supposed animated unless otherwise specified.
 Given an analytic ring $A$ we denote by $\ob{D}(A)$ its derived $\infty$-category of \textit{complete} $A$-modules.

Unfortunately, some of our references work in a ``light setup'', while others in a ``non-light'' one.
 We will often use references to the latter ones, because in many situations the proofs go through verbatim.
 We remark that in the solid theory the category $D(\Z_\solid)$ (defined in the light setting) embeds fully faithfully as a symmetric monoidal $\infty$-category into its non-light version, and this embedding preserves colimits and countable limits.
 In particular, certain statements in the light setup can literally be reduced to their non-light counterparts (without having to imitate the proof).
 With this general way of reasoning in mind, we don't go into the details of how to transfer particular statements to the light setting.

Let $ \mathbb{Z}_{p,\solid}$ be the analytic ring of solid $p$-adic integers and let $\mathbb{Q}_{p,\solid}$ be the analytic ring of solid $p$-adic numbers.
 We denote by $\Cat{Ring}_{\Z_{p,\solid}}$ the category of solid $\Z_{p, \solid}$-algebras and by $\Cat{Ring}_{\Q_{p,\solid}}$ the category of solid $\Q_{p, \solid}$-algebras.
 Let us start by recalling the definition of the full subcategory of \textit{bounded $\Q_{p, \solid}$-algebras} $$\Cat{Ring}_{\Q_{p,\solid}}^b\subset \Cat{Ring}_{\Q_{p,\solid}}$$ introduced in \cite[Section 2.6]{camargo2024analytic}.\footnote{Up to the issue of ``light vs. non-light'' that we clarified before.}
\medskip

Before spelling out the definition, let us recall that a basic property of a bounded $\Q_p$-algebra $A$ is that each element $f\in A(\ast)$ is \textit{bounded}, i.e.\ the morphism of solid $\Q_{p,\solid}$-algebras $\Q_p[T]\to A$ sending $T$ to $f$,  extends, for some $n\in \mathbb{N}$, to a morphism $\Q_p\langle p^n T\rangle \to A$ of $\Q_p[T]$-algebras, where $\Q_p\langle p^n T\rangle\cong \Z_p[p^nT]^{\wedge_p}[1/p]$ denotes the Tate $\Q_p$-algebra in $p^nT$.
 We note that any such extension is in fact unique due to the idempotency of $\Q_p\langle p^n T\rangle$ over $\Q_p[T]$.
  However, being bounded is not just a condition on $A(\ast)$, but on $A(S)$ for any light profinite set $S$.

Given a light profinite set $S$, we define the solid algebra $\mathbb{Z}_{p,\solid}\langle \mathbb{N}[S] \rangle $ as  the $p$-adic completion of the free solid $\mathbb{Z}_p$-algebra $\mathbb{Z}_{p,\solid}[\mathbb{N}[S]]$ generated by $S$.

\begin{definition}[{\cite[Definition 2.6.1]{camargo2024analytic}}] Let $A\in \Cat{Ring}_{\mathbb{Q}_{p,\solid}}$.
 We define the \textit{subring $A^{\circ}\subset A$ of power bounded elements} as the subring with $S$-points, for any light profinite set $S$, given by
$$A^{\circ}(S):=\ob{Map}_{\Cat{Ring}_{\Z_{p,\solid}}}(\mathbb{Z}_{p,\solid}\langle \mathbb{N}[S] \rangle, A).$$
 The \textit{subring of the bounded elements} $A^b\subset A$ is defined as $$A^b=A^{\circ}[1/p].$$ We say that $A$ is \textit{bounded}  if the natural map $A^b \to A$ is an equivalence.
\end{definition}

In other words, given $A\in \Cat{Ring}_{\mathbb{Q}_{p,\solid}}$,  the $S$-points of $A^b$ are given by maps $f: S\to A$ such that there exists some $n\in \mathbb{N}$ such that the map $p^nf :S\to A$ extends to a morphism of algebras
\[
\mathbb{Z}_{p,\solid}\langle \mathbb{N}[S]\rangle \to A. 
\]

This definition is  sensible thanks to the idempotency properties of the $\mathbb{Z}_{p,\solid}[\mathbb{N}[S]]$-algebra $\mathbb{Z}_{p,\solid}\langle \mathbb{N}[S]\rangle$, \cite[Lemma 2.4.7]{camargo2024analytic}.
 This means that the  functor $A\mapsto A^b$ is idempotent, and being bounded is a property on solid $\mathbb{Q}_p$-algebras and no extra structure.
   The fact that $A^b$ is indeed a solid $\mathbb{Q}_p$-algebra follows from \cite[Lemma 2.4.8]{camargo2024analytic} as proven in Proposition 2.6.9 of \textit{loc.cit.}.

\begin{lemma}
  \label{lem-bounded-subrings-provides-right-adjoint}
  The inclusion $\Cat{Ring}^b_{\Q_{p,\solid}}\to \Cat{Ring}_{\Q_{p,\solid}}$ of bounded $\Q_p$-algebras into solid $\Q_{p}$-algebras, has as a right adjoint the functor $A\mapsto A^b$.
  In particular, the full subcategory $\Cat{Ring}^b_{\Q_{p,\solid}}\subset \Cat{Ring}_{\Q_{p,\solid}}$ is stable under all colimits.
\end{lemma}
\begin{proof}
  These assertions are formally implied by the definition of a bounded $\Q_p$-algebra. 
\end{proof}

\begin{remark}
 By construction, the algebraic affine line $\AnSpec(\Q_p[T])$ and the analytic affine line $\A^{1,\mathrm{an}}_{\Q_p}$ get identified as functors on bounded $\Q_p$-algebras.
\end{remark}

\begin{examples}
  \label{exam-classical-tate-rings-are-bounded} \
  \begin{enumerate}
   \item If $A$ is a classical Banach $\Q_p$-algebra viewed as a solid $\Q_p$-algebra, then $A$ is bounded, \cite[Lemma 2.6.5]{camargo2024analytic}.
   \item Taking care of non-trivial induced ring structures as in \cite{camargo2024analytic}   the category of bounded affinoid $\Q_p$-algebras, as defined in \cite[Definition 2.6.10.(3)]{camargo2024analytic}, contains all classical Tate Huber pairs $(A,A^+)_{\solid}$ over $\mathbb{Q}_p$. Namely, the condition of being bounded for a solid affinoid ring \cite[Definition 2.6.6]{camargo2024analytic} only depends on its underlying solid algebra. 
  \end{enumerate}
\end{examples}

\begin{remark}
  \label{sec:bounded-rings}
  Let $A\in \Cat{Ring}_{\Q_{p,\solid}}$ be a solid $\Q_{p}$-algebra, and let $S$ be a light profinite set. We can define a new solid $\Q_p$-algebra $A_S:=\underline{\mathrm{Hom}}_{\Q_{p,\solid}}(\Q_{p,\solid}[S],A)$, i.e., $A_S$ is the condensed ring sending a light profinite set $T$ to $A(T\times S)$. Let us say that $A$ is weakly bounded if for any light profinite set $S$ each element $f\in A_S(\ast)$ is bounded. 
  This condition is in general different from being bounded.
   As a counterexample one has \cite[Example 2.5.1]{camargo2024analytic}, namely, let $A= \mathbb{Q}_{p,\solid}[\N[S]]$ be  the free solid $\Q_p$-algebra on a profinite set $S$, let $n> 1$ and consider the map $f_n\colon S\to S^n\to A$ where the first is the diagonal map and the second the natural map. We have an induced map $f_n\colon A\to A$, we let $B= \pi_0(A\otimes_{f_n,A} \Q_p)$, that is, the quotient of $A$ obtained by killing the solid module generated by $f_n(S)$. The algebra $B$ is the initial static solid algebra that parametrizes maps $g\colon S\to R$ with $R$ a solid $\Q_p$-algebra such that $f^n\colon S\to R$ is the zero map. Applying this construction not just for $S\to S^n$ but for all $S^k\to S^{kn}$, one can construct an augmented solid $\Q_p$-algebra $B'$ such that its augmentation ideal $I$ is nilpotent after evaluating at any profinite set $T$. However, one can show that the natural map $A\to B'$ will not extend to the algebra $\Q_{p,\solid}\langle \N[S]\rangle$.
\end{remark}

There are two main relevant features of bounded rings that we shall recall next.

\subsection{The $\dagger$-nilradical}\label{sss:Nilrad}  The first new feature in the category of bounded rings  is the existence of a new nilradical $\ob{Nil}^{\dagger}$, called the  \textit{$\dagger$-nilradical} which heuristically consists of \textit{elements of norm zero}, see \cite[Definition 2.6.1]{camargo2024analytic}.  Let us recall its definition since it will play an important role in the rest of the paper. We need some preparations. Given a light profinite set $S$ we let $\mathbb{Z}_{p,\solid}\llbracket\mathbb{N}[S]\rrbracket$ denote the completion of $\mathbb{Z}_{p,\solid}[\mathbb{N}[S]]$ with respect to its natural filtration, see \cite[Definition 2.4.5]{camargo2024analytic}. It has the explicit description
\[
\mathbb{Z}_{p,\solid}\llbracket\mathbb{N}[S]\rrbracket = \prod_{n\in \mathbb{N}} \mathbb{Z}_{p,\solid} [S^n/\Sigma_n]
\]
where  $S^n/\Sigma_n$ is the light profinite set given by the quotient of $S^n$ by the natural action of the symmetric group on $n$-letters (see \cite[Lemma 2.4.3]{camargo2024analytic}).

\begin{lemma}\label{xnsuw8}
Let $S,S'$ be light profinite sets. Then $\mathbb{Z}_{p,\solid}\langle \mathbb{N}[S] \rangle$  is the subring of $\mathbb{Z}_{p,\solid}\llbracket\mathbb{N}[S]\rrbracket$ whose $S'$-points consists of formal series
\[
\sum_{n\in \mathbb{N}} a_{n} 
\]
where $a_{n}\in \mathbb{Z}_{p,\solid}[S^n/\Sigma_n](S')$ are elements such that $ a_n\to 0$ in the $p$-adic topology of the ring $\mathbb{Z}_{p,\solid}\llbracket \mathbb{N}[S]\rrbracket(S')$.
\end{lemma}
\begin{proof}
Set $A:=\mathbb{Z}_{p,\solid}\langle \mathbb{N}[S] \rangle$. By definition $A$ is the $p$-adic completion of the algebra $B:=\mathbb{Z}_{p,\solid}[\mathbb{N}[S]]$. Therefore, for a light profinite set $S'$, we have that $A(S')$ is the $p$-adic completion of $B(S')$ (because the evaluation at $S'$ commutes with limits). But $B=\bigoplus_{n\in \mathbb{N}} \mathbb{Z}_{p,\solid}[S^n/\Sigma_n]$, so an element in $B(S')$ consists of series $\sum_{n} a_n$ with $a_{n}\in \mathbb{Z}_{p,\solid}[S^n/\Sigma^n](S')$  that is eventually zero. The lemma follows.
\end{proof}

For the next lemma, we recall that a morphism $f\colon A\to B$ of solid $\Q_p$-algebras satisfies $\ast$-descent, if the natural functor $\ob{D}(A)\to \ob{Tot}\big( \ob{D}^\ast(B^{\otimes_A \bullet+1})\big)$ is an equivalence. Here, the $\ast$ indicates that the implicit morphisms in the limit are given by tensoring, i.e., $\ast$-pullback.

\begin{lemma}\label{xjsiwn}
Let $f:A\to B$ be a morphism of solid $\mathbb{Q}_p$-algebras such that $f$  satisfies $*$-descent. If $B$ is bounded, then  $A$ is bounded. 
\end{lemma}
\begin{proof}
Let $S$ be a light profinite set and consider a map $S\to A$. We want to show that there is some $n\gg 0$ such that the natural map $A\to A\otimes_{\mathbb{Q}_{p,\solid}[\mathbb{N}[S]]} \mathbb{Q}_{p,\solid}\langle \mathbb{N}[p^n S] \rangle$ is an equivalence. By taking $n$ such that the map $B\to B\otimes_{\mathbb{Q}_{p,\solid}[\mathbb{N}[S]]} \mathbb{Q}_{p,\solid}\langle\mathbb{N}[ p^n S]\rangle$ is an equivalence (which exists as $B$ is bounded) the lemma follows by $*$-descent. 
\end{proof}

\begin{definition}\label{xhjs8wj}
Let $r\in \mathbb{R}_{\geq 0}$. We define the ring $\mathbb{Q}_{p,\solid}\langle \mathbb{N}[S] \rangle_{\leq r}$ as the solid subring of $\mathbb{Q}_{p,\solid}\llbracket\mathbb{N}[S]\rrbracket$ whose $S'$-points are series
\[
\sum_{n} a_{n}
\]  
 with $a_n\in \mathbb{Q}_{p,\solid}[S^n/\Sigma^n] (S')$ such that there is $r'>r$ with  $|a_{n}| r'^{n}\to 0$ where the norm $|-|$ is taken with respect to the $p$-adic topology of $\mathbb{Q}_{p,\solid}[\mathbb{N}[S]](S')$ (we normalize $|-|$ so that $|p|=p^{-1}$). 
\end{definition}

In other words, $\Q_{p,\solid}\langle \N[S]\rangle_{\leq r}(\ast)$ is the ring of \emph{overconvergent} functions on the $S$-dimensional closed disc of radius $r$.

Clearly, we have $\Q_{p,\solid}\langle \N[S]\rangle_{\leq s}\subseteq \Q_{p,\solid}\langle \N[S]\rangle_{\leq r}$ if $r\leq s$.

\begin{remark}\label{xj29jws}
A priori $\mathbb{Q}_{p,\solid}\langle \mathbb{N}[S] \rangle_{\leq r}$ is only defined as a subpresheaf of $\Q_p$-algebras of $\mathbb{Q}_{p,\solid}\llbracket\mathbb{N}[S]\rrbracket$. By \cref{x02kd76} (1) below it is indeed a sheaf on light profinite sets and even a solid $\mathbb{Q}_p$-module.
\end{remark}

\begin{lemma}\label{x02kd76}
Let $S$ be a light profinite set and $r\in \mathbb{R}_{\geq 0}$. The following hold:

\begin{enumerate}

\item\label{x1} The subpresheaf on light profinite sets  $\mathbb{Q}_{p,\solid}\langle \mathbb{N}[S] \rangle_{\leq r}\subset \mathbb{Q}_{p,\solid}\llbracket\mathbb{N}[S]\rrbracket$ is indeed a condensed set and a solid $\mathbb{Q}_{p,\solid}$-algebra.

\item\label{x2} Let $K$ be an infinitely ramified algebraic extension of $\mathbb{Q}_p$. We have the description
\begin{equation}\label{xns72hj}
K \otimes_{\Q_p}  \mathbb{Q}_{p,\solid}\langle \mathbb{N}[S] \rangle_{\leq r} = \varinjlim_{\substack{b\in K \\ |b|> r}} K_{\solid} \langle \mathbb{N}[b S] \rangle
\end{equation}
where the left-hand side is a priori understood as the tensor product of presheaves.  In particular, $\mathbb{Q}_{p,\solid} \langle \mathbb{N}[S] \rangle_{\leq r}$ is a bounded $\mathbb{Q}_{p,\solid}$-algebra and idempotent over $\mathbb{Q}_{p,\solid}[\mathbb{N}[S]]$.
Here, $K_\solid$ denotes $K$ with its induced analytic ring structure from $\Q_p$, and $K_\solid\langle \N[bS]\rangle\cong K_\solid\langle \N[S]\rangle$, and the transition maps for $r<|b'|<|b|$ in the colimit are given by the unique morphisms of $K_\solid$-algebras, which sends $S\subseteq K_\solid\langle \N[S]\rangle$ to $\frac{b'}{b}S\subseteq K_\solid[S]\subseteq K_\solid\langle \N[S]\rangle$.
\item We have that $\mathbb{Q}_{p,\solid}\langle \mathbb{N}[S]\rangle_{\leq 1}\subset \mathbb{Q}_{p,\solid}\langle \mathbb{N}[S]\rangle$ as subalgebras of $\mathbb{Q}_{p,\solid}\llbracket\mathbb{N}[S]\rrbracket$.

\end{enumerate}
\end{lemma}
\begin{proof}
Part (2) implies (1) as then $\mathbb{Q}_{p,\solid}\langle \mathbb{N}[S] \rangle_{\leq r}$ will be a retract of a solid $\mathbb{Q}_{p,\solid}$-module. It will also show part (3) as we can check this inclusion after base change to $K$. Let us then prove (2). It suffices to prove \cref{xns72hj}, in fact bounded algebras are stable under retracts of solid $\mathbb{Q}_p$-modules since they  are stable under $*$-descent by  \cref{xjsiwn}  and the right hand side term of the equation is clearly bounded. Similarly, the idempotency over $\mathbb{Q}_{p,\solid}[\mathbb{N}[S]]$ follows since this can be checked after base change to $K$ by descent, where it follows by \cite[Lemma 2.4.7]{camargo2024analytic}.

Thus, let us prove \cref{xns72hj}. Let $S'$ be a light profinite set, denote $A=\mathbb{Q}_{p,\solid}\langle \mathbb{N}[S] \rangle_{\leq r}$ and $B= \mathbb{Q}_{p,\solid}\llbracket\mathbb{N}[S]\rrbracket$ (seen as presheaves on light profinite sets). By definition, $A(S')\subset B(S')$ consists of sums
\[
\sum_{n} a_{n}
\]
with $a_{n}\in \mathbb{Q}_{p,\solid}[S^n/\Sigma_n](S')$ such that there is $r'>r$ with  $|a_{n}|r'^{n}\to 0$ for the $p$-adic norm of $ \mathbb{Q}_{p,\solid}[S^n/\Sigma_n](S')$. The triangular inequality shows that $A(S')$ is indeed a $\mathbb{Q}_{p}$-subalgebra   of $B(S')$ (even a $\mathbb{Q}_p(S')$-algebra). Thus, we can write $A=\varinjlim_{r'>r} A_{r'}$ where $A_{r'}$ is the subpresheaf of $B$ whose $S'$-valued points are precisely those sums $\sum_{n} a_{n}$ as before such that $|a_n|r'^n\to 0$. By the ultrametric inequality $A_{r'}(S')\subset B(S')$ is a $\mathbb{Q}_p$-algebra. Let us take $r'$ to be a rational number and let $F/\mathbb{Q}_p$ be a finite extension admitting an element $b$ of norm $r'$. Then, we can describe $(A\otimes_{\mathbb{Q}_p} F)(S')= A(S')\otimes_{\mathbb{Q}_p} F$ as the sums $\sum_{n} a_n $ with $a_{n}\in F_{\solid}[S^n/\Sigma^n]$ such that $|b^{n} a_n |\to 0$ for the $p$-adic valuation. The latter identifies with the $S'$-valued points of $F_{\solid}\langle  \mathbb{N}[b S] \rangle$ (cf. \cref{xnsuw8}). This implies that
\[
A_{r'}\otimes_{\mathbb{Q}_p} F =F\langle \mathbb{N}[bS] \rangle 
\]  
where the LHS tensor product is understood as a tensor product of presheaves. Taking base change to $K$ and knowing that the valuations of elements of $K$ are dense in $\mathbb{R}$ we have the desired description of \cref{xns72hj}.
\end{proof}

\begin{remark}\label{xbnsw7}
A consequence of \cref{x02kd76} is that we could have used the algebras $\mathbb{Q}_{p,\solid}\langle \mathbb{N}[S] \rangle_{\leq 1}$ instead of $\mathbb{Q}_{p,\solid}\langle \mathbb{N}[S] \rangle$ to define bounded rings as both define the same notion of a bounded element. This is the correct point of view to extend the theory of bounded rings from solid rings to gaseous rings, see \cite[Lecture V.2]{ScholzeRLL}, as the latter seems to behave well only for overconvergent functions. Indeed, for the gaseous tensor product, the algebras $\mathbb{Q}_{p,\mathrm{gas}}\langle \mathbb{N}[S] \rangle_{\leq 1}$ are idempotent over $\mathbb{Q}_{p,\mathrm{gas}}[ \mathbb{N}[S] ]$, but not the algebras $\mathbb{Q}_{p,\mathrm{gas}}\langle \mathbb{N}[S] \rangle$.
\end{remark}

\begin{definition}\label{x20ojew}
Given a solid $\mathbb{Q}_{p,\solid}$-algebra $A$ we write
 \[
 A\langle \mathbb{N}[S] \rangle_{\leq r}:= A\otimes_{\mathbb{Z}_{p,\solid}} \mathbb{Q}_{p,\solid} \langle \mathbb{N}[S] \rangle_{\leq r}.
 \] 
\end{definition}

We can now define the following subobjects represented by the algebras of \cref{x02kd76} (the definition is producing  subobjects thanks to the idempotency statements proved in that lemma).

\begin{definition}\label{xjs82k}
Let $A$ be a solid ring over $\mathbb{Q}_{p,\solid}$ and $r\in \mathbb{R}_{\geq 0}$. We define the following objects:
\begin{enumerate}

\item The $\mathbb{Z}_{p,\solid}$-subring of \textit{norm-$1$-elements} $A^{\leq 1}\subset A$ is the subpresheaf whose $S$-valued points are given by 
\[
A^{\leq 1}(S)=\ob{Map}_{\Cat{Ring}_{\Q_{p,\solid}}}(\mathbb{Q}_{p,\solid}\langle \mathbb{N}[S] \rangle_{\leq 1}, A).
\]
More generally, for $r\in \mathbb{R}_{\geq 0}$ we define the solid $A^{\leq 1}$-submodule $A^{\leq r}\subset A$ to be the subpresheaf whose $S$-valued points are 
\[
A^{\leq r}(S)=\ob{Map}_{\Cat{Ring}_{\Q_{p,\solid}}}(\mathbb{Q}_{p,\solid}\langle \mathbb{N}[S] \rangle_{\leq r}, A).
\]

\item  When $r=0$, we denote $\ob{Nil}^{\dagger}(A)=A^{\leq 0}$ and call it the \textit{$\dagger$-nilradical of $A$}.

\item We define the $A^{\leq 1}$-ideal of  $p$-adically nilpotent elements to be $A^{<1}=\bigcup_{r<1} A^{\leq r}$.  More generally for $r>0$, we define $A^{<r}:=\bigcup_{r'<r} A^{\leq r'}$. 
\end{enumerate}
\end{definition}

\begin{remark}
  \label{rema-sec:dagger-nilradical-different-notions-of-topological-nilp}\
  \begin{enumerate}
   \item The definition of the $\dagger$-nilradical given above agrees with the one introduced \cite[Definition 2.6.1]{camargo2024analytic}. In fact, $\mathbb{Q}_p\langle \mathbb{N}[S]\rangle_{\leq 0} = \mathbb{Q}_{p}\{\mathbb{N}[S]\}^{\dagger}$ in the notation of \textit{loc. cit.}.

   \item In \cite[Definition 2.6.1.(2)]{camargo2024analytic} the subgroup $A^{\circ\circ}\subseteq A$ of topologically nilpotent elements is defined. However, the natural inclusion $A^{<1}\subseteq A^{\circ\circ}$ is strict in general: take $A= \Z_p\llbracket T \rrbracket[\frac{1}{p}]$, then the element $T$ is in $A^{\circ\circ}$ by definition, but there is no $r<1$ such that the map $\Q_p[T]\to A$ extends to $\Q_p\langle T \rangle_{\leq r}$. Indeed, we actually have maps the other way around $A\to \Q_p\langle T \rangle_{\leq r}$  as $T$ is topologically nilpotent in the second algebra, and both are idempotent over $\Q_p[T]$.

     Alternatively, one can note that the subgroup of topological nilpotent elements in the sense of  \cref{xjs82k} has as underlying points elements $a\in A$ such that for some $s\in \mathbb{N}_{\geq 1}$ the sequence $ (p^{-n}a^{ns})_{n}$ is a null-sequence in $A$. On the other hand, the ideal of \cite[Definition 2.6.1(2)]{camargo2024analytic} only asks for $(a^n)_{n}$ to be a null sequence in $A$.
     This variant of the ideal will be more convenient for us in the rest of the paper (and also behaves much better in the gaseous setting).
     For Gelfand rings (\cref{sss:gelfand-rings}), this discrepancy disappears.
  \end{enumerate}
\end{remark}

 Next, we want to show that the presheaves of \cref{xjs82k} are indeed condensed anima.

 \begin{lemma}\label{x29jdm}
Let $S'\to S$ be a surjection of light profinite sets. Then the natural map
\begin{equation}\label{eqhskjdfsql}
\mathbb{Q}_{p,\solid}[\mathbb{N}[S]] \otimes_{\mathbb{Q}_{p,\solid}[\mathbb{N}[S']] } \mathbb{Q}_{p,\solid}\langle \mathbb{N}[S'] \rangle_{\leq 1} \xrightarrow{\sim} \mathbb{Q}_{p}\langle \mathbb{N}[S] \rangle_{\leq 1}
\end{equation}
is an equivalence.  In particular, the presheaves of \cref{xjs82k} are condensed anima.
\end{lemma}
\begin{proof}
 By the  projectivity of free solid $\Z_{p}$-modules, we have an epimorphism of solid $\mathbb{Z}_p$-modules $\mathbb{Z}_{p,\solid}[S']\to \mathbb{Z}_{p,\solid}[S]$, this map has a section $s: \mathbb{Z}_{p,\solid}[S]\to \mathbb{Z}_{p,\solid}[S']$  which gives rise to a map of $\mathbb{Q}_{p}$-algebras
\[
\widetilde{s}:\mathbb{Q}_{p,\solid}[\mathbb{N}[S]]\to \mathbb{Q}_{p,\solid}[\mathbb{N}[S']].
\]
Using the presentation of $\mathbb{Q}_{p,\solid}\langle \mathbb{N}[S']\rangle_{\leq 1}$ as convergent sequences of \cref{xhjs8wj}, one verifies that this map extends to
\[
\widetilde{s}:\mathbb{Q}_{p,\solid}\langle \mathbb{N}[S]\rangle_{\leq 1} \to \mathbb{Q}_{p,\solid}\langle \mathbb{N}[S']\rangle_{\leq 1}.
\]
Composing with the natural map
\[
\mathbb{Q}_{p,\solid}\langle \mathbb{N}[S']\rangle_{\leq 1}\to  \mathbb{Q}_{p,\solid}[\mathbb{N}[S]] \otimes_{\mathbb{Q}_{p,\solid}[\mathbb{N}[S']] } \mathbb{Q}_{p,\solid}\langle \mathbb{N}[S'] \rangle_{\leq 1},
\]
this produces a morphism of idempotent $\mathbb{Q}_{p,\solid}[\mathbb{N}[S]]$-algebras
\[
\mathbb{Q}_{p,\solid}\langle \mathbb{N}[S]\rangle_{\leq 1} \to  \mathbb{Q}_{p,\solid}[\mathbb{N}[S]] \otimes_{\mathbb{Q}_{p,\solid}[\mathbb{N}[S']] } \mathbb{Q}_{p,\solid}\langle \mathbb{N}[S'] \rangle_{\leq 1}.
\] Since the category of idempotent algebras is a poset and there exists a morphism in the other direction, one deduces that it must be an equivalence.  For the final statement, it suffices to show that $A^{\leq 1}$ defines an hypersheaf in light profinite sets, as the other cases can be reduced to this by \cref{LemmaAcircicrr} below. Now, given a surjection $S'\to S$ as before, the equivalence \cref{eqhskjdfsql} yields the cartesian square
\[
\begin{tikzcd}
A^{\leq 1}(S) \ar[r] \ar[d] & A^{\leq 1} (S') \ar[d] \\
A(S) \ar[r] & A(S') .
\end{tikzcd}
\]
Then $A^{\leq 1}$ is a hypersheaf as $A$ is so.
\end{proof}

\begin{lemma}\label{LemmaAcircicrr}
Let $A$  be a bounded ring. We have the following relations:
\begin{enumerate}
\item Let $r\geq 0$. Then  $A^{\leq r}=\bigcap_{r'>r} A^{\leq r'}=\bigcap_{r'> r} A^{<r'}$.
\item  Let $r>0$. Then  $A^{<r}=\bigcup_{r'<r} A^{\leq r'}=\bigcup_{r'<r} A^{<r'}$. 
\item Let $r,r'\geq 0$. Then the composition $A^{\leq r}\times A^{\leq r'}\subseteq A\times A\overset{\mathrm{mult.}}{\to} A$ factors over $A^{\leq rr'}\subseteq A$.
\item Let $r,r'\geq 0$. Then the composition $A^{\leq r}\times A^{\leq r'}\subseteq A\times A\overset{\mathrm{add.}}{\to} A$ factors over $A^{\leq {\max\{r, r'\}}}\subseteq A$.
\end{enumerate}
\end{lemma}
\begin{proof}
  For (1), the inclusion $A^{\leq r}\subseteq \bigcap_{r'>r} A^{\leq r'}$ is clear as $\Q_{p,\solid}\langle \N[S]\rangle_{\leq r'}\subseteq \Q_{p,\solid}\langle \N[S]\rangle_{\leq r}$ if $r'>r$. Conversely, one notes that, by definition, $\Q_{p,\solid}\langle \N[S]\rangle_{\leq r}=\bigcup_{r'>r}\Q_{p,\solid}\langle \N[S]\rangle_{\leq r'}$, which implies $\bigcap_{r'>r}A^{\leq r'}\subseteq A^{\leq r}$. Clearly, the systems $A^{\leq r'}, A^{<r'}$ are cofinal among each other, which implies $\bigcap_{r'>r} A^{\leq r'}=\bigcap_{r'> r} A^{<r'}$.

  For (2), the first equality is the definition of $A^{<r}$. The second equality follows again by cofinality.

  For (3), let $S$ be a light profinite set, and let $f\in A^{\leq r}(S),\ g\in A^{\leq r'}(S)$. We need to see that the product $fg\in A(S)$ lies in $A^{\leq rr'}(S)$. This reduces to the universal situation, which is given by $A=\Q_{p,\solid}\langle \N[S]\rangle_{\leq r}\otimes_{\Q_{p,\solid}} \Q_{p,\solid}\langle \N[S]\rangle_{\leq r'}$, the canonical morphisms $f,g$, and the universal product $fg$, itself given by the composition
  \[
    \Q_{p,\solid}[\N[S]]\overset{\alpha}{\to}\Q_{p,\solid}[\N[S]]\otimes_{\Q_{p,\solid}}\Q_{p,\solid}[\N[S]]\to A.
  \]
  Here, the first morphism $\alpha$ is induced by the diagonal $S\mapsto S\times S\subseteq \N[S\coprod S]$ using $\Q_{p,\solid}[\N[S]]\otimes_{\Q_{p,\solid}}\Q_{p,\solid}[\N[S]]\cong \Q_{p,\solid}[\N[S\coprod S]]$. Now, $A$ embeds into $\Q_{p,\solid}\llbracket\N[S\coprod S]\rrbracket$, and by \cref{xhjs8wj} the assertion reduces to the following observation: Let $x=\sum_{n\in \N} a_n\in \Q_{p,\solid}\langle S\rangle_{\leq rr'}$, and assume that $|a_n|s^n\to 0$ for some $s>rr'$. Write $s=tt'$ with $t>r$ and $t'>r'$. Write $\alpha(x)=\sum_{n,m}b_{n,m}$ with $b_{n,m}\in \Q_{p,\solid}[S^n/\Sigma_n\times S^m/\Sigma_m]\subseteq \Q_{p,\solid}\llbracket\N[S\coprod S]\rrbracket$ (i.e., $b_{n,m}=0$ if $n\neq m$, and $b_{n,n}$ is the image of $a_n$ under the diagonal embedding $S^n/\Sigma_n\to S^n/\Sigma_n\times S^n/\Sigma_n$). Then $|b_{n,m}|t^n(t')^m\to 0$ as $|a_n|s^n\to 0$.

  Lastly, the proof of the assertion in (4) is similar, and reduces eventually to the ultrametric inequality of the $p$-adic norm.
\end{proof}

\begin{remark}\label{xbs82kj}\
\begin{enumerate}
\item Keep the notation of \cref{xjs82k}. Given \cref{LemmaAcircicrr}, the fact that $A^{\leq 1}$ is a solid $\mathbb{Z}_{p,\solid}$-algebra and that $A^{\leq r}$ for $r\geq 0$ are solid $A^{\leq 1}$-modules follow by the same argument as in \cite[Proposition 2.6.9]{camargo2024analytic}. Alternatively one can argue via descent from an infinitely ramified extension of $\Q_{p}$ to use \cref{x02kd76}, and by writing $A^{\leq 0}$ or $A^{\leq 1}$ as a suitable intersection of multiplicative translates of the ring $A^{\circ}$ there defined. Note also that $$A^{b}=\bigcup_{r} A^{\leq r}$$ is the subring of bounded elements of $A$, while $A^\circ\subset  A^{\leq 1}$ (as $\Q_{p,\solid}\langle \N[S]\rangle_{\leq 1}\subseteq \Q_{p,\solid}\langle \N[S]\rangle$).

\item We recall that the idempotency of $\Q_{p,\solid}\langle \N[S]\rangle_{\leq r}$ implies that $\pi_i(A^{\leq r})=\pi_i(A)$ for $i>0$, i.e., $A^{\leq r},\ A^{<1},\ \mathrm{Nil}^\dagger(A)$ are condensed subanima of $A$. This also implies that the property of being bounded only depends on the static quotient $\pi_0(A)$ of $A$.
\end{enumerate}
\end{remark}

We aim to prove useful criteria to check containment in $A^{\leq 1}$. For this reason, we establish the following technical lemma.

\begin{lemma}\label{xsjw83h}
Let $S$ be a light profinite set. Then for all $n\geq 1$ the natural map
\[
\mathbb{Q}_{p,\solid}[\mathbb{N}[S]]\to \mathbb{Q}_{p,\solid}[\mathbb{N}[S]]\otimes_{\mathbb{Q}_{p,\solid}[\mathbb{N}[S^n]]}    \mathbb{Q}_{p,\solid}\langle \mathbb{N}[S^n]\rangle_{\leq 1},
\]
where the tensor in the second term is induced by $S^n\to S^n/\Sigma_n \to \mathbb{Q}_{p,\solid}[\mathbb{N}[S]]$, factors through $\mathbb{Q}_{p,\solid}\langle \mathbb{N}[S]\rangle_{\leq 1}$. In particular, the natural map $$\Q_{p,\solid}[\N[S]]\otimes_{\Q_{p,\solid}[\N[S^n]]} \Q_{p,\solid}\langle \N[S^n]\rangle_{\leq 1}\to \Q_{p,\solid}\langle \N[S]\rangle_{\leq 1}$$ is an isomorphism of idempotent $\Q_{p,\solid}[\N[S]]$-algebras.
\end{lemma}
\begin{proof}
It suffices to prove that the map 
\[
\mathbb{Z}_{p,\solid}[\mathbb{N}[S]] \to \mathbb{Z}_{p,\solid}[\mathbb{N}[S]]\otimes_{\mathbb{Z}_{p,\solid}[\mathbb{N}[S^n]]} \mathbb{Z}_{p,\solid}\langle \mathbb{N}[S^n] \rangle
\]
 factors through $\mathbb{Z}_{p,\solid} \langle \mathbb{N}[S] \rangle$ as then the other statement will follow. Indeed, using \cref{x02kd76}(2), we may pass to an infinitely ramified extension of $\Q_{p}$, then write the base change of $\Q_{p}\langle \N[S]\rangle_{\leq 1}$ as a colimit of $\Q_p\langle \N[S]\rangle$.
 
Let $A=\mathbb{Z}_{p,\solid}[\mathbb{N}[S]]\otimes_{\mathbb{Z}_{p,\solid}[\mathbb{N}[S^n]]} \mathbb{Z}_{p,\solid}\langle \mathbb{N}[S^n] \rangle$, it suffices to show that $\pi_0(A)$ is $p$-adically complete. Indeed, this would show that $\pi_0(A)$ is a $\mathbb{Z}_{p,\solid}\langle \mathbb{N}[S] \rangle$-algebra, but the latter is idempotent over $\mathbb{Z}_{p,\solid}[\mathbb{N}[S]]$, therefore $A$ is a $\mathbb{Z}_{p,\solid}\langle \mathbb{N}[S] \rangle$-algebra (indeed, all $\pi_i(A),\ i>0$ will be modules over $\Z_{p,\solid}\langle \N[S]\rangle$).

To prove the claim, it suffices to show that $\mathbb{Z}_{p,\solid}[\mathbb{N}[S]]$ is a static $\mathbb{Z}_{p,\solid}[\mathbb{N}[S^n]]$-module of finite presentation, in fact this would imply that $\pi_0(A)$ is a static $\mathbb{Z}_{p,\solid}\langle \mathbb{N}[S^n]\rangle$-module of finite presentation and any such module is $p$-adically complete. For a light profinite set $S'$, let $\ob{Free}_{\mathbb{Z}_{p,\solid}}(S')$ be the free associative $\mathbb{Z}_{p,\solid}$-algebra generated by $S'$. We have a commutative diagram of associative algebras with horizontal surjective maps
\[
\begin{tikzcd}
\ob{Free}_{\mathbb{Z}_{p,\solid}}(S)  \ar[r] & \mathbb{Z}_{p,\solid}[\mathbb{N}[S]]\\ 
\ob{Free}_{\mathbb{Z}_{p,\solid}}(S^n) \ar[r] \ar[u] &\mathbb{Z}_{p,\solid}[\mathbb{N}[S^n]] \ar[u].
\end{tikzcd}
\]
Thus, to show that $\mathbb{Z}_{p,\solid}[\mathbb{N}[S]]$ is of finite presentation as $\mathbb{Z}_{p,\solid}[\mathbb{N}[S^n]]$-module it suffices to prove that it is of finite presentation as a $\ob{Free}_{\mathbb{Z}_{p,\solid}}(S^n)$-module. On the other hand, we have that
\[
\ob{Free}_{\mathbb{Z}_{p,\solid}}(S)=\bigoplus_{i=0}^{n-1} \ob{Free}_{\mathbb{Z}_{p,\solid}}(S^n)\otimes_{\mathbb{Z}_{p,\solid}} \mathbb{Z}_{p,\solid}[S^i]
\]
as left $\ob{Free}_{\mathbb{Z}_{p,\solid}}(S^n)$-modules, therefore $\ob{Free}_{\mathbb{Z}_{p,\solid}}(S)$ is a finitely presented  $\ob{Free}_{\mathbb{Z}_{p,\solid}}(S^n)$-module (as each $\Z_{p,\solid}[S^i]$ is a compact projective $\Z_{p,\solid}$-module). Thus, it suffices to show that $\mathbb{Z}_{p,\solid}[\mathbb{N}[S]]$ is a finitely presented $\ob{Free}_{\mathbb{Z}_{p,\solid}}(S)$-module. For proving this, first note that given a finite set $S$ we have a right exact sequence of graded modules
\begin{equation}\label{eqnx92jn}
 \ob{Free}_{\mathbb{Z}_{p,\solid}}(S)  \otimes_{\mathbb{Z}_p} \bigwedge^2 \mathbb{Z}_{p}[S] \to \ob{Free}_{\mathbb{Z}_{p,\solid}}(S) \to \ob{Sym}_{\mathbb{Z}_p} \mathbb{Z}_p[S]\to 0
\end{equation}
functorial in $S$. Now, if $S=\varprojlim_i S_i$ is a countable limit of  finite sets, by  taking  limits of graded modules on \cref{eqnx92jn}, we get a right exact sequence
\[
\ob{Free}_{\mathbb{Z}_{p,\solid}}(S)\otimes_{\mathbb{Z}_{p,\solid}}\bigwedge^2 \mathbb{Z}_{p,\solid}[S] \to \ob{Free}_{\mathbb{Z}_{p,\solid}}(S) \to \mathbb{Z}_{p,\solid}[\mathbb{N}[S]]\to 0
\]
proving the claim, and so finishing the proof of the lemma.
\end{proof}

We can now verify the desired criterion.

\begin{corollary}\label{x92jhwn}
Let $S$ be a light profinite set, $A$ a bounded $\Q_p$-algebra and $f: S\to A$  a map  such that the $n$-th power $f^{(n)}:S^n\to A$ factors through $A^{\leq r}$, then $f$ factors through $A^{\leq r^{1/n}}$.  Moreover, the following hold: 

\begin{enumerate}

\item  $f$ factors through $A^{\leq r}$ for some $r>0$ if and only if for all   $r'>r$ of the form $r'=p^{a/b}$ with $a\in \mathbb{Z}$ and $b\in \mathbb{N}_{>0}$ the $b$-th power $f^{(b)}: S^b\to A$ factors through $p^{-a}A^{<1}$. In the last condition, $p^{-a}A^{<1}$ can be replaced by $p^{-a}A^{\leq 1}$.

\item  $f$ factors through $A^{<r}$ for some $r>0$ if and only if there is some $r'<r$ of the form $r'=p^{a/b}$ with $a\in \mathbb{Z}$ and $b\in \mathbb{N}_{>0}$ whose $b$-th power $f^{(b)}: S^b\to A$ factors through $p^{-a}A^{<1}$. The the last condition, $p^{-a}A^{<1}$ can be replaced by $p^{-a}A^{\leq 1}$. 
  
\end{enumerate}

\end{corollary}
\begin{proof}

We prove the first claim. Let $f:S\to A$ be such that $f^{(n)}: S^n\to A$ factors through $A^{\leq r}$. We want to show that $f$ factors through $A^{\leq r^{1/n}}$. Without loss of generality we can assume that $r$ and $r^{1/n}$ are rational (by writing $A^{\leq r}=\bigcap_{r<r'}A^{\leq r'}$ with $r'$ rational, \cref{LemmaAcircicrr}). Multiplying by some power of $p$ we may also assume that $r< 1$. We need to see that the map
\[
A\to A\otimes_{\mathbb{Q}_{p,\solid}[\mathbb{N}[S]]}\mathbb{Q}_{p,\solid}\langle \mathbb{N}[S]\rangle_{\leq r^{1/n}}
\]
is an isomorphism provided that 
\[
A\to A\otimes_{\mathbb{Q}_{p,\solid}[\mathbb{N}[S^n]]}\mathbb{Q}_{p,\solid}\langle \mathbb{N}[S^n]\rangle_{\leq r}
\]
is an isomorphism. We may base change (using $\ast$-descent) to a finite extension $K$ of $\mathbb{Q}_p$. Therefore, we can assume that there is a pseudo-uniformizer $\pi$ in $A$ with $|\pi|_p=r^{1/n}$ (using that $r^{1/n}$ is rational, $r<1$ and non-zero). In this case, we can write 
\[
K_{\solid}\langle \mathbb{N}[S]\rangle_{\leq r^{1/n}}=K_{\solid}\langle \mathbb{N}[\pi S]\rangle_{\leq 1}
\]
and 
\[
K_{\solid}\langle \mathbb{N}[S^n]\rangle_{\leq r}=K_{\solid}\langle \mathbb{N}[(\pi S)^n]\rangle_{\leq 1}.
\]
Then, by rescaling we can assume that $r=1$. In this case the claim follows from \cref{xsjw83h}.

We now prove the next claims.  In the forward direction of (1) it suffices to consider the universal case, namely the map $f\colon S\to A=\Q_{p,\solid}\langle \N[S]\rangle_{\leq r}$. Then $f^{(b)}\colon S^b\to \Q_{p,\solid}\langle \N[S]\rangle_{\leq r}$ extends to a map $\Q_{p,\solid}\langle \N[S^n]\rangle_{\leq r^b}\to A$ (this follows from the definition of these algebras, \cref{xhjs8wj}), and thus $f^{(b)}\in A^{\leq r^b}(S^b)$. Therefore, $p^{a}f^{(b)}\in A^{\leq p^{-a}r^b}(S^b)$ (as $p\in A^{\leq 1/p}$ and $p^{-1}\in A^{\leq p}$.). But $r'=p^{a/b}>r$ implies $p^{-a}r^b<1$, and thus $p^{a}f^{(b)}\in A^{<1}(S^b)$, i.e., $f^{(b)}$ factors through $p^{-a}A^{<1}$. Conversely, if $p^af^{(b)}$ factors through $A^{<1}$ (or just $A^{\leq 1}$) for all $r'=p^{a/b}>r$ then in particular  $f^{(b)}$ factors through $A^{\leq p^a}$. By the first part of the corollary the map $f$ factors therefore through $A^{\leq  r'}$. Taking limits as $r'\to r$ we  deduce that $f$ factors through $A^{\leq 1}$ as wanted.

Now, we check (2). Assume that $f$ factors through $A^{\leq s}\subseteq A^{<r}$. Then take some $s<r'=p^{a/b}<r$. By (1) $f^{(b)}$ factors through $p^{-a}A^{<1}$ as desired. Conversely, assume that $f^{(b)}$ factors through $p^{-a}A^{\leq 1}=A^{\leq p^a}$ for some $r'=p^{a/b}<r$. Then $f$ factors through $A^{\leq r'}$ by the first part of the corollary, and thus through $A^{<r}$ as $r'<r$.
\end{proof}

\begin{example}
  \label{sec:dagger-nilradical-example-power-bounded-elements}
  Let $A$ be a classical Banach algebra over $\Q_p$, seen as a solid $\Q_p$-algebra. Then $A^{\leq 1}$ is the classical $\Z_p$-algebra $A'$ of power-bounded elements in $A$ (for a similar assertion, see \cite[Lemma 2.6.5]{camargo2024analytic}). Indeed, let $A_0\subseteq A$ be a ring of definition. Then $A_0\subseteq A$ is open with $A/A_0$ discrete. This implies that for each continuous map $f\colon S\to A$ with $S$ a light profinite set, and such that $f(S)\subseteq A'$ the image of $S$ lies in a finitely generated $A_0$-subalgebra $B\subseteq A'$. But then $B$ is bounded (hence $p$-adically complete), and $S^n\to B$ for all $n\geq 0$. Thus, the map $\Z_{p,\solid}[\N[S]]\to A$ classified by $f$ extends to $\Z_{p,\solid}\langle \N[S]\rangle$, and hence a fortiori to $\Q_{p,\solid}\langle \N[S]\rangle_{\leq 1}$.
  Conversely, \cref{x92jhwn} implies that $A^{\leq 1}\subseteq A'$ if we can show $A^{<1}=A^{\circ\circ}$.
  If $f\colon S\to A$ lands in $A^{<1}$, then $f(s)\in A$ has to be topological nilpotent for any $s$, and therefore $A^{<1}\subseteq A^{\circ \circ}$.
  For the other inclusion $A^{\circ\circ}\subseteq A^{<1}$, given a map $g\colon S\to A^{\circ\circ}$, we can note that there exists an open and bounded subring $B\subseteq A$ such that $\Im(S)\subseteq B^{\circ\circ}$ (indeed, $g$ factors over $A'$, which is the union of such $B$'s).
  We can conclude that the topology on $B$ is $p$-adic and that the morphism $S\to B/p$ factors through a finite quotient.
  As $f(s)\in A^{\circ\circ}$ for all $s\in S$, there exists some $n\geq 0$ such that $S^n$ maps to $pB$.
  This implies that $S^n\to B$ extends to $\Q_p\langle \N[S^n]\rangle_{\leq 1/p}\to B$ (this uses that $B$ has the $p$-adic topology).
  By \Cref{x92jhwn} we can conclude that $S\to B\to A^{\leq 1}$ factors over $A^{<1}$.
\end{example}

\begin{lemma}\label{LemDescentAr}
Let $f\colon A\to B$ be a morphism of bounded $\Q_p$-algebras satisfying $*$-descent. Let $B^{\bullet}$ be the \v{C}ech conerve of $f$. Then for $r\in [0,\infty)$ the natural maps
\[
A^{\leq r} \to \ob{Tot}(B^{\bullet,\leq r}) \mbox{ and } A^{< r} \to \ob{Tot}(B^{\bullet,< r})
\]
are equivalences. 
\end{lemma}
\begin{proof}
As the map $f$ satisfies $*$-descent, we have that $A=\ob{Tot}(B^{\bullet})$. Thus, all terms in the lemma are subojects of $A$ and it suffices to prove the following fact: let $g\colon S\to A$ be a map whose composite $S\to B$ factors through $B^{\leq r}$ (resp. $B^{<r}$), then $g$ factors through $A^{\leq r}$ (resp. $A^{<r}$). The first case follows by $*$-descent  from the fact that $S\to A$ factors through $A^{\leq r}$ if and only if it extends to a map $\Q_{p,\solid}\langle \N[S] \rangle_{\leq r}\to A$. For the second case, by \cref{x92jhwn}, $f$ factors through $A^{<r}$ if and only if there is $r'=\frac{a}{b}<r$ such that $f^{(b)}\colon S^b\to A$ factors through $A^{\leq a}$, which then can be tested by $*$-descent. 
\end{proof}

\begin{lemma}\label{xsu29s}
Let $\{A_i\}_{i\in I}$ be a filtered diagram of bounded $\Q_p$-algebras with colimit $A$.
Then for any $r>0$ the natural map
\[
\varinjlim_{i} A_i^{<r} \to  A^{<r}
\]
is an equivalence. 
\end{lemma}
\begin{proof}
  Using descent from an infinitely ramified algebraic extension of $\Q_p$, we can reduce to the case that $r=1$, see \cref{LemDescentAr}.  Both terms are full condensed subanima of $A$, therefore we can assume without loss of generality that the $A_i$ and so $A$ are static. Let $S$ be a light profinite set and let $f:S\to A^{<1},$ then there is some $r<1$ such that $f$ lifts to 
\[
\widetilde{f}:\mathbb{Q}_p\langle \mathbb{N}[S] \rangle_{\leq r} \to A.
\]
For any $r< r'' <r'<1$ the map $\mathbb{Q}_p\langle  \mathbb{N}[S]  \rangle_{\leq r'} \to \mathbb{Q}_p\langle  \mathbb{N}[S]  \rangle_{\leq r}$ factors through the condensed subalgebra $B_{r''}\subset \mathbb{Q}_{p,\solid}\llbracket\mathbb{N}[S]\rrbracket$ consisting of sums $\sum_n a_n$ with $a_n\in \mathbb{Q}_{p,\solid}[S^n/\Sigma_n]$ such that $(|a_n|(r'')^n)_{n}$ is uniformly bounded for the $p$-adic norm of  (the points of)  $\mathbb{Q}_{p,\solid}[S^n/\Sigma_n]$. This algebra $B_{r''}$ is a light Smith space (\cref{defSmith}), so in particular compact. Thus, we can lift the map $\widetilde{f}$ to some map of $\mathbb{Q}_p$-modules $g: B_{r''}\to A_i$ for some $i$. We want to see that after enlarging $i$ this map becomes a morphism of algebras. Indeed, being a morphism of algebras means $1$ goes to $1$ (which we can guarantee as this holds on $A$), and that the multiplication maps are compatible. This last compatibility is guaranteed since $B_{r''}\otimes_{\mathbb{Q}_{p,\solid}} B_{r''}$ is still a light Smith space, so compact, and so the defect of $g$ being a multiplicative map can be tested via a linear map $B_{r''}\otimes_{\mathbb{Q}_{p,\solid}} B_{r''}\to A_{i}$. Since the composite to $A$ is zero (as $B_{r''}\to A$ is an algebra morphism) there is some index $i'>i$ such that $B_{r''}\otimes_{\mathbb{Q}_{p,\solid}} B_{r''}\to A_{i} \to A_{i'}$ is an algebra morphism. This provides a lift of $f$ to  $S\to A_{i'}$ that can be extended to $\mathbb{Q}_{p,\solid}\langle \mathbb{N}[S] \rangle_{\leq r'}$ proving what we wanted.
\end{proof}

\begin{remark}
The analog of \Cref{xsu29s} for $A^{\leq 1}$ is false: given a light profinite set $S$, e.g., $S=\{\ast\}$, the identity $\Q_{p,\solid}\langle \N[S]\rangle_{\leq 1}\to \Q_{p,\solid}\langle\N[S]\rangle_{\leq 1}=\varinjlim_{r>1} \Q_{p,\solid}\langle \N[S]\rangle_{\leq r}$ does not factor through some $\Q_{p,\solid}\langle \N[S]\rangle_{\leq r}$ as otherwise the inclusion $\Q_{p,\solid}\langle \N[S]\rangle_{\leq r}\to \Q_{p,\solid}\langle \N[S]\rangle_{\leq 1}$ would be an isomorphism for some $r>1$.
\end{remark}

One of the most important features of the $\dagger$-nilradical is that it is an ideal for bounded rings, see \cite[Proposition 2.6.9 (3)]{camargo2024analytic}. We can then define the $\dagger$-reduction of a bounded algebra:

\begin{definition}[{\cite[Definition 2.6.10]{camargo2024analytic}}]\label{xsj82k}
Let $A$ be a bounded algebra. Its $\dagger$-reduction is the bounded $\mathbb{Q}_{p,\solid}$-algebra $A^{\dagger-\ob{red}}:=A/\ob{Nil}^{\dagger}(A)$. We say that a bounded algebra is $\dagger$-reduced if the map $A\to A^{\dagger-\ob{red}}$ is an equivalence.
\end{definition}

We note that the cofiber of $\ob{Nil}^{\dagger}(A)\to A$ is automatically concentrated in degree $0$ (as $\ob{Nil}^\dagger(A)$ defines a condensed subanima). Hence, $A^{\dagger-\ob{red}}$ is static, and acquires the structure of a condensed $\Q_{p,\solid}$-algebra (even solid, cf.\ \cref{xsjwi4}).

The following is an analog of \cite[Proposition 2.6.16]{camargo2024analytic}.\footnote{We note that \cite[Proposition 2.6.16.(3)]{camargo2024analytic} is \emph{wrong}, and that this is a major motivation for us to use $A^{\leq 1}$ instead of $A^\circ$, cf., \Cref{sec:dagg-form-smooth-example-overconvergent-cohomology-of-rigid-disc}, \Cref{sec:de-rham-stacks-counterexample-to-smoothness}.
}

\begin{proposition}\label{xjs9jw}
Let $A$ be a bounded $\Q_p$-algebra, then we have a pullback square
\[
\begin{tikzcd}
A^{\leq 1} \ar[r] \ar[d]& A^{\dagger-\ob{red},\leq 1} \ar[d] \\ 
A\ar[r] & A^{\dagger-\ob{red}}.
\end{tikzcd}
\]
The same holds for $A^{\leq r}$ for $r\in \mathbb{R}_{\geq 0}$ and $A^{<1}$. In particular, we have that $\ob{Nil}^{\dagger}(A^{\dagger-\ob{red}})=0$ and so $A^{\dagger-\ob{red}}$ is $\dagger$-reduced.
\end{proposition}
\begin{proof}
It suffices to prove the first part for the $A^{<r}$ elements ($r>0$), since the others follow from \cref{LemmaAcircicrr}. Set $B=A^{\dagger-\ob{red}}$. Let $S$ be a light profinite set and $f:S\to A$ a map whose  composite $g:S\to A\to B$ factors through $B^{<r}$. By \cref{x92jhwn} there is some $r'=p^{a/b}<r$ such that the map $g^{(b)}:S^b\to B$ factors through $p^{-a}B^{<1}$. So, by letting $S':= (\bigsqcup_{k\in \mathbb{N}} S^{bk})\cup \{\infty\}$, we can extend the sequence
\[
p^{-ak}g^{(bk)}:S^{bk}\to  B
\] 
to a null sequence $\widetilde{g}: S'\to B$. Then, we can find a lift $\widetilde{f}: S'\to A$ of $\widetilde{g}$ which will land in some $A^{\leq s}$ for some $s>0$ (as $A$ is bounded). Then, the difference between $p^{-ak}f^{(bk)}:S^{bk}\to A$ and the restriction $\widetilde{f}|_{S^{bk}}:S^{bk}\to A$ lands in $\ob{Nil}^{\dagger}(A)$. Therefore, $p^{-ak}f^{(bk)}$ lands in $A^{\leq s}$ for all $k\geq 1$ and by \cref{xsjw83h} one has that $f$ lands in $A^{\leq (s|p|^{ak})^{1/bk}}= A^{\leq s^{1/bk} |p|^{a/b}}=A^{\leq s^{1/bk}r'}$. Since this holds for all $k\geq 1$, by taking limits as $k\to \infty$ one deduces that $f$ factors through $A^{\leq r'}\subset A^{<r}$ as wanted. The claim $\ob{Nil}^\dagger(A^{\dagger-\red})=0$ follows then since $\ob{Nil}^{\dagger}(A)=A^{\leq 0}$ and by definition we have a fibration 
\[
\ob{Nil}^{\dagger}(A)\to A\to A^{\dagger-\ob{red}}.\qedhere
\]
\end{proof}

\begin{remark}
  \label{sec:dagger-nilradical-quotient-by-dagger-nilpotent-kernel}
  The proof of \cref{xjs9jw} shows that for any quotient $A\to B$ with $\mathrm{fib}(A\to B)\subseteq \Nil^{\dagger}(A)$, one has a pullback square on $\leq r$ or $<r$ elements as in \cref{xjs9jw}.
\end{remark}

The following is a version of \cite[Lemma 2.6.18]{camargo2024analytic} (up to the change from $A^\circ$ to $A^{\leq 1}$ and the fact that $A$ must be static in \textit{loc. cit.}).

\begin{lemma}
  \label{sec:dagger-nilradical-a-leq-1-p-adically-separated}
  Let $A$ be a static bounded $\Q_p$-algebra. Then $A$ is $\dagger$-reduced if and only if $A^{\leq 1}$ is $p$-adically separated.
\end{lemma}
\begin{proof}
  Assume that $A$ is $\dagger$-reduced. Then $\bigcap_{n\in \N} p^nA^{\leq 1}=\ob{Nil}^\dagger(A)=0$ and $A^{\leq 1}$ is $p$-adically separated. Conversely, assume that $A^{\leq 1}$ is $p$-adically separated. Let $S$ be a light profinite set, and let $S\to \ob{Nil}^\dagger(A)$ be a map. Then for any $n\geq 1$, this map lies in $A^{\leq 1/p^n}=p^nA^{\leq 1}$. Thus, this map factors through $\bigcap_{n\in \N}p^nA^{\leq 1}=0$, and therefore $\ob{Nil}^\dagger(A)=0$.
\end{proof}

\begin{remark}\label{xns82j}
  The current definition of the nilradical suffices for all the applications in $p$-adic Hodge theory we have encountered so far.
  However, it has a technical disadvantage: given a $\dagger$-reduced bounded ring $A$ it could be that $A(S)$ has non-trivial nilpotent elements for some light profinite set $S$.
  This does not happen though if $A$ is Gelfand in the sense of \cref{xnshw8} below, see \cref{xsu8rf}.
  A fix to this problem involves a  different definition of the $\dagger$-nilradical which agrees with the former definition in the Gelfand situation. Since this deficiency of the current definition of $\ob{Nil}^{\dagger}$ will not play a role in this paper we will not comment more about it, except in \cref{xs2d93}.
\end{remark}

The following statement should not be a surprise.

\begin{lemma}\label{xsh89j}
Let $A$ be a bounded ring, then for any $r<1$ the pair  $(A^{\leq 1}(*), A^{\leq r}(*) )$ is a Henselian pair.  
\end{lemma}
\begin{proof}
We have to show that $A^{\leq r}$ is contained in the Jacobson radical of $A^{\leq 1}(*)$ and that for any monic polynomial $P(T)\in A^{\leq 1}[T]$ such that $\overline{P}(T)\in A^{\leq 1}/A^{\leq r}[T]$ has a root $\overline{\alpha}$, we can lift $\overline{\alpha}$  to a root $\alpha$ in $A$. Without loss of generality we can take $r=|p|$. Let $a\in p A^{\leq 1}=A^{\leq |p|}$, we have a map $\mathbb{Z}_p\langle \frac{T}{p} \rangle_{\leq 1}\to A^{\leq 1}$, since $1+T$ is invertible in $\mathbb{Z}_p\langle \frac{T}{p} \rangle_{\leq 1}$ then $1+a$ is invertible in $A^{\leq 1}$ proving that $a$ is in the Jacobson radical. Now, consider a monic polynomial $P$ in $A^{\leq 1}$ with a root $\overline{\alpha}$ modulo $p$. Fix $\alpha'$ a lift of $\overline{\alpha}$. Write $P(T)=T^n+a_{n-1}T^{n-1}+\cdots +a_0$, we then have a map of rings $f:B\to A^{\leq 1}$ with $B=\mathbb{Z}_p\langle X_0,\ldots, X_{n-1} , Y\rangle^{\leq 1}\to A^{\leq 1}$ mapping $X_i$ to $a_{i}$ and $Y_i$ to $\alpha'$. Let $\widetilde{P}(T)=T^n+X_{n-1}T^{n-1}+\cdots+X_0$, the map $f$ factors through $B\langle \frac{\widetilde{P}(T)-Y}{p} \rangle^{\leq 1}$. It is easy to see using Newton's method that there is a root $\widetilde{Y}$ of $\widetilde{P}(T)$ in $B\langle \frac{\widetilde{P}(T)-Y}{p} \rangle^{\leq 1}$ that agrees with $Y$ modulo $p$. The image $\alpha$ of $\widetilde{Y}$ in $A^{\leq 1}$ produces the desired root. 
\end{proof}

\subsection{Uniform completion of bounded rings}
We discuss another important construction on bounded $\Q_p$-algebras, which generalizes the uniform completion of Huber rings.

\begin{definition}\label{xwj29e}
Let $A$ be a bounded $\Q_{p}$-algebra. The \textit{uniform completion} of $A$ is the bounded $A$-algebra
\[
A^{u} = (\varprojlim_{r} A^{\leq 1}/A^{\leq r})[1/p]
\]
(with the limit taken in solid $\Z_p$-algebras) endowed with the induced analytic ring structure from $A$.  We say that a bounded $\Q_p$-algebra $A$  is \textit{uniform} if the map $A\to A^u$ is an equivalence.\footnote{We recall that a non-archimedean Banach ring $A$ is uniform iff its topology can be defined by a power-multiplicative norm iff its spectral seminorm is a norm iff its subring $A^0$ of power-bounded elements is bounded iff $A^0$ is $p$-adically complete, if $A$ is a Banach $\Q_p$-algebra. See \cite[Proposition 2.21]{scholze2024berkovichmotives}  and \cite[Definition 2.8.1]{kedlaya_liu_relative_p_adic_hodge_theory_foundations}.}
\end{definition}

\begin{remark}
 In other words,
$$A^u= (A^{\leq 1})^{\wedge_p}[1/p].
$$
We note that the natural morphism $A\to A^u$ factors over $A^{\dagger-\red}$, and that $A^u$ is automatically static.
\end{remark}

\begin{lemma}\label{xks0jw}
 Let $A$ be a bounded $\mathbb{Q}_p$-algebra. The inclusion 
\[
(A^{\leq 1})^{\wedge_p} \hookrightarrow  \varprojlim_{r\to 0} A/A^{\leq r}
\] 
induces an isomorphism 
\[
A^u\xrightarrow{\sim} \varprojlim_{r\to 0} A/A^{\leq r}.
\]
\end{lemma}
\begin{proof}
First, note that multiplication by $p$ on $\varprojlim_{r} A/A^{\leq r}$ is invertible: we already have an isomorphism 
\[
p \cdot: A/A^{\leq r}\xrightarrow{\sim } A/A^{\leq r|p|}.  
\]
This gives an injective map $A^u\to \varprojlim_{r} A/A^{\leq r}$. Next, we want to show that this map is also a surjection. Let $S$ be a light profinite set with a map $f:S\to  \varprojlim_r A/A^{\leq r}$. If $r<1$, then we can lift  the projection $f_r:S\to A/A^{\leq r}$ to a map $\widetilde{f}_r: S\to A$ (using that $\Z_{p,\solid}[S]$ is a projective solid $\Z_{p,\solid}$-module). Since $A$ is bounded, there is some $s>r$ such that $\widetilde{f}_r$ factors through $A^{\leq s}$. We claim that for any $r'<r$ the projection $f_{r'}:S\to A/A^{\leq r'}$ factors through $A^{\leq s}/A^{\leq r'}$.  Suppose the claim holds, then for $|p|^{-k}>s$ we have that $p^kf$ factors through $\varprojlim_{r} A^{\leq 1}/A^{\leq r}=(A^{\leq 1})^{\wedge_p}$ proving the surjectivity.  Let us now show the claim: let $\widetilde{f}_{r'}:S\to A$ be a lift of $f_{r'}$, then we know that $\widetilde{f}_{r}-\widetilde{f}_{r'} \in A^{\leq r}$, by the ultrametric inequality (\cref{LemmaAcircicrr}), and since $s>r$, we get that $\widetilde{f}_{r'}$ lands in $A^{\leq s}$, as desired.
\end{proof}

\begin{remark}\label{xsjwi4}
The algebra $A^{\leq 1}$ is $\Z_{p,\solid}$-solid: indeed, this can be checked after base change to the ring of integers of a ramified extension of $\Q_p$ containing $p^{-1/n}$ for all $n>0$ and then
$$
A^{\leq 1} = \cap_n p^{-1/n} A^\circ
$$
 which is solid by \cite[Proposition 2.6.9 (3)]{camargo2024analytic}. This immediately implies that $A^{u}$ is $\Z_{p}$-solid. 
\end{remark}

\begin{remark}\label{xs2d93}
For having a definition of the uniform completion of a general bounded $\Q_p$-algebra $A$, such that for all $S$ profinite $A^u(S)$ is a classical uniform Banach ring, one should use a variant $\tilde{A}^{\leq r}$ of the subspaces $A^{\leq r}$ for $r\geq 0$. By definition, $\tilde{A}^{\leq r}(S)$ will consist of those maps $*\to C(S,A)=A_S$ (in the notation of \cref{sec:bounded-rings}) that extend to $\mathbb{Q}_p\langle T \rangle_{\leq r}$. This variant contains $A^{\leq r}$  but  is in general  much larger, and it is not obvious at first sight that this definition should give a solid module (though it turns out to be the case). The variant of the nilradical mentioned in \cref{xns82j} is then the intersection
$$\widetilde{\ob{Nil}}^{\dagger}(A)=\cap_{r} \widetilde{A}^{\leq r}.$$
However, \cref{xsu8rf} below shows that this difference is irrelevant for Gelfand rings (that will be introduced in \cref{xnshw8}).
\end{remark}

\begin{proposition}\label{xsu8rf}
Let $A$ be a bounded $\Q_p$-algebra. The following hold:
\begin{enumerate}
\item $A^{u}$ is a bounded $\Q_p$-algebra and we have a pullback square in $\ob{D}(\Z_{p,\solid})$
\begin{equation}
\label{xjsiw0}
\begin{tikzcd}
A^{\leq  1} \ar[r]\ar[d] & A^{u,\leq 1} \ar[d] \\
A \ar[r] & A^{u}.
\end{tikzcd}
\end{equation}
Similarly for $A^{\leq r}$ for $r\in \mathbb{R}_{\geq 0}$ and $A^{<1}$.  In particular, taking $r=0$, the natural map $A^{\dagger-\red}\to A^u$ is injective and the natural map $(A^{\dagger-\red})^u\to A^u$ is an isomorphism.

\item We have  $A^{u,\leq 1} = \varprojlim_{r<1} A^{\leq 1}/ A^{\leq r}$ and $A^{u,<1}= \varprojlim_{r<1} A^{<1}/A^{\leq r}$. In particular, the map $(A^{u})^u\to A^u$ is an isomorphism, so $A^u$ is uniform, and the uniform completion of $A$ only depends on its $\dagger$-reduction $A^{\dagger-\ob{red}}$, and $A^u$ is static.

\item We have $\ob{Nil}^{\dagger}(A^u)=0$, i.e. $A^u$ is $\dagger$-reduced.

\item An object $f\in A$ is invertible if and only if its image in $A^u$ is so.

\end{enumerate}
\end{proposition}
\begin{proof}

We prove (1). We may assume that $A$ is static. Let us first show that $A^u$ is bounded.  This follows from the fact that $A^u= (A^{\leq 1})^{\wedge_p}[\frac{1}{p}]$ is the generic fiber of a $p$-adically complete solid ring, see \cite[Lemma 2.6.13]{camargo2024analytic}. Let $S$ be a light profinite set and let $g: S \to A$ be a map such that the composite $f:S\xrightarrow{g} A\to A^{u}$ factors through $A^{u,<1}$ (by \cref{x92jhwn}). There is a power $f^{(s)}:S^s\to A^{u,<1}$ that factors through $p A^{u,<1}$.  By \cref{x92jhwn} it suffices to show that there is some power $g^{(s)}:S^s\to A$ that factors through $A^{<1}$, by changing $S$ to $S^{2s}$ we can then assume that $f$ factors through $p^2A^{u,<1}$.

The map $p^{-1}f:S\to A^{u,<1}$ extends to a morphism of algebras 
\[
H:\mathbb{Z}_{p,\solid}\llbracket\mathbb{N}[p^{-1} S]\rrbracket\to A^{u,<1}.
\]
In particular, we have an induced map
\[
\widetilde{f}:S':= \bigsqcup_n (p^{-1} S)^n\cup\{\infty\}\to A^{u,<1}
\]
mapping $\infty$ to zero. Recall that $A^u=\varprojlim_{r} A/A^{\leq r}$ by \cref{xks0jw}.   Let $r<1$ and consider the composite 
\[
\widetilde{f}_r: S'\to A^u\to A/A^{\leq r}.
\]
We can find a lift $\widetilde{g}:S'\to A$ of $\widetilde{f}_r$.  Note that for all $n\geq 1$ the restriction of $\widetilde{g}|_{S^n}$ agrees with $p^{-n}g^{(n)}:S^n\to A$ modulo $A^{\leq r}\subset A^{<1}$.  There exists some $s>r$ such that $\widetilde{g}$ factors through $A^{\leq s}$. Since $r<s$ the ultrametric inequality  yields (\Cref{LemmaAcircicrr} (4)) that $p^{-n} g^{(n)}$ lands in $A^{\leq s}$ for all $n\geq 1$ and by \cref{x92jhwn} one sees that $g$ lands in $pA^{\leq s^{1/n}}$ for all $n\in \mathbb{N}$, and so in $pA^{\leq 1}\subset A^{<1}$ as wanted.
The same argument works as well for $A^{\leq 1}$ replaced by $A^{\leq r'}$ for some $r'\geq 0$.

 It is left to show that $A/A^{\leq r}\cong A^u/A^{u,\leq r}$ for $r\geq 0$. For this it suffices to see that the map $A/A^{\leq r}\to A^u/A^{u,\leq r}$ is surjective as injectivity follows by what we have already proven. By multiplying with some power of $p$, we may assume that $r\geq 1$, in which case the desired surjectivity is implied by the one for $r=1$. Hence, we assume $r=1$. By \cref{xks0jw} we have a map $\pi:A^u\to A/A^{\leq 1}$. We have that $\ker(\pi)= \varprojlim_{r<1} A^{\leq 1}/A^{\leq r}=(A^{\leq 1})^{\wedge_p}$ and so $\ker(\pi)\subset A^{u,\leq 1}$. Therefore, given $S$ a light profinite set and a map $f:S\to A^u$, we can take a lift $\widetilde{g}:S\to A$ with composite $\widetilde{f}:S\to A\to A^{u}$ such that the difference $f-\widetilde{f}$ lands in $\ker(\pi)\subset A^{u,\leq 1}$, proving the desired surjectivity.
 
 We now continue with part (2). It is clear that $A^u$ is static since $A^{\leq 1}\subset A$ is a full subanima, in particular $\pi_i(A^{\leq 1})= \pi_i(A)$ for $i\geq 1$ is a $\mathbb{Q}_p$-module which is killed after $p$-completion. From \cref{xks0jw,LemmaAcircicrr}  one deduces that  $\pi: A^u\to A/A^{<1}$ has kernel $\ker(\pi)=(A^{<1})^{\wedge_p}$.  Indeed, for $r<1$ we have a short exact sequence
 \[
0\to A^{<1}/A^{\leq r}   \to A/A^{\leq r} \to A/A^{<1}\to 0,
 \]
 taking limits  one has a short exact sequence
 \[
 0\to (A^{<1})^{\wedge_p}\to A^u\to A/A^{<1}\to 0.
\]

Thus we have an inclusion $(A^{<1})^{\wedge_p}\subset A^{u,<1}$.  We want to show that this inclusion is an equality. Let $f:S\to A^{u,<1}$. As before, we can find a map $\widetilde{g}:S\to A$ with composite $\widetilde{f}:S\to A\to A^u$ such that the difference $f-\widetilde{f}$ lands in $(A^{<1})^{\wedge_p}$. In particular, $\widetilde{f}$ lands in $A^{u,<1}$ and by part (1) $\widetilde{g}$ lands in $A^{<1}$. This implies that $f$ itself is in $(A^{<1})^{\wedge_p}$ proving what we wanted. 

Next we prove part (3). By (2) we know that $A^{u,\leq 1}=(A^{\leq 1})^{\wedge_p}$ is classically $p$-adically complete, so $p$-adically separated and by \cref{sec:dagger-nilradical-a-leq-1-p-adically-separated} $A^u$ is $\dagger$-reduced.

Finally, we prove (4). Let $f\in A$ be an element such that the reduction $\overline{f}\in A^u$ is invertible, let $s\in \R_{>0}$ be such that $f\in A^{\leq s}$. Let $g\in A^u$ be an inverse of $\overline{f}$. By part (2), we can find an element $\widetilde{g}\in A$ such that $\widetilde{g}= g+h$ with $h\in A^{u,\leq 1/2s}$. Hence, we have that $1-\widetilde{g}f$ maps to $-h\overline{f}\in A^{u,\leq 1/2}$ and so by part (1) that $1-\widetilde{g}f\in A^{\leq 1/2}$. Hence, the element $1-\widetilde{g}f$ defines a map $\Q_p\langle T\rangle_{\leq 1/2}\to  A$ mapping $T\mapsto 1-\widetilde{g}f$. Since $1-T$ is invertible in $\Q_p\langle T\rangle_{\leq 1/2}$, one has that $\widetilde{g}f$ is invertible in $A$  which implies that $f$ is so.   
\end{proof}

\begin{corollary}\label{xhs892k} The inclusion of uniform bounded $\Q_p$-algebras into all bounded $\Q_p$-algebras has $A\mapsto A^u$ as a left adjoint. In particular, for any diagram $\{A_i\}_{i\in I}$ of bounded $\Q_p$-algebras, the natural map $(\varinjlim_i A_{i})^u\to (\varinjlim_i A_i^u)^u$ is an equivalence.
\end{corollary}
\begin{proof}
  This follows from the facts that $A^u$ is uniform and that each morphism of bounded $\Q_p$-algebras $A\to B$ induces a morphism $A^{\leq 1}\to B^{\leq 1}$.
\end{proof}

\begin{example}
  \label{sec:adic-berk-spectra-exam-uniform-completion}
  Let $S$ be a light profinite set and $A:=\Q_{p,\solid}\langle \N[S]\rangle_{\leq 1}$. Then $A^{\leq 1}=A\cap \Z_{p,\solid}\langle \N[S]\rangle$, and $A^{u,\leq 1}=\Z_{p,\solid}\langle \N[S]\rangle$. 
\end{example}

\newpage
\section{Gelfand rings and their Berkovich spectrum}\label{s:gelfand-rings}

This  chapter introduces the notion of Gelfand ring, a suitable class of bounded rings. Using Gelfand rings, we will define in the next chapter a variant of analytic stacks, called Gelfand stacks, of which analytic de Rham stacks will be examples. The definition of Gelfand rings was motivated by conversations of the fourth author and fifth author with Clausen about different aspects of the theory of bounded rings; in particular, it was inspired by an analogous definition given by Clausen in the complex situation.

A main advantage of Gelfand rings is that there is a good theory of Berkovich spectrum for them, generalizing the one of Banach algebras.

\subsection{Definition and main properties}
\label{sss:gelfand-rings}

\begin{definition}\label{xnshw8}
A bounded $\Q_p$-algebra $A$ is said to be \textit{Gelfand} if $A/A^{\leq 1}$ is a discrete condensed set.  We let $\Cat{GelfRing}_{\mathbb{Q}_{p}}\subset \Cat{Ring}^b_{\mathbb{Q}_{p}}$ be the full subcategory of Gelfand $\Q_p$-algebras (also called Gelfand rings).
\end{definition}

\begin{remark}
  \label{sec:defin-main-prop-equivalent-characterizations-for-gelfand-property}
  Filtered colimits of discrete condensed sets, subobjects of discrete condensed sets, and quotients of discrete condensed sets are again discrete. Moreover, extensions of discrete condensed abelian groups are again discrete. This implies that a bounded $\Q_p$-algebra is Gelfand if and only if one of following objects is a discrete condensed set:
  \[
    A/A^{\leq r},\ A/A^{<r}, A^{\leq r}/A^{\leq s},\ A^{<r}/A^{<s},
  \]
  for some $r>s>0$.
\end{remark}

\begin{example}
  \label{sec:defin-main-prop-example-gelfand-rings}
  For a light profinite set $S$ and $r\geq 0$ the bounded $\Q_p$-algebra $A=\Q_{p,\solid}\langle \N[S]\rangle_{\leq r}$ is Gelfand if and only if $S$ is finite. Indeed, $A^{\leq 1}/A^{\leq |p|}$ contains $\F_{p,\solid}[S]$ as a direct summand, and this $\F_{p,\solid}$-module is discrete if and only if $S$ is finite. For the converse, one notes that $A$ is a colimit of Banach algebras over $\Q_p$, and one can apply \cref{xnbsyw}.
\end{example}

The following lemma provides enough stability properties of the category of Gelfand $\Q_p$-algebras:

\begin{proposition}\label{xnbsyw}
The category $\Cat{GelfRing}_{\mathbb{Q}_p}$ is stable under colimits in bounded $\Q_p$-algebras. Moreover, a bounded ring $A$ is Gelfand if and only if $A^u$ is a Banach $\mathbb{Q}_p$-algebra. 
\end{proposition}
\begin{proof}
By \cref{xsu8rf} we have an isomorphism $A/A^{\leq 1} = A^{u}/A^{u,\leq 1}$. Hence, a bounded ring $A$ is Gelfand if and only if $A^u$ is Gelfand. Now, for $A$ a bounded ring, one has $A^u=A^{u,\leq 1}[\frac{1}{p}]$ where $A^{\leq 1, u}$ is a $p$-adically complete and separated solid abelian group. Thus, $A^u$ is a Banach $\Q_p$-algebra if and only if $A^{u,\leq 1}/p=A^{\leq 1} /p$ is a discrete $\F_p$-vector space. Since  $|p|=p^{-1}$, $A^u$ is Banach if and only if $A^{\leq 1}/ A^{\leq p^{-1}}$ is discrete. By \cref{sec:defin-main-prop-equivalent-characterizations-for-gelfand-property} this last happens if and only if $A$ is Gelfand. 

The stability of Gelfand rings under colimits in bounded $\Q_p$-algebras follows from the compatibility of uniform completion under colimits (\Cref{xhs892k}), and the fact that the uniform completion of a colimit of Banach algebras is a Banach algebra. 
\end{proof}

With the previous proposition in mind, we give the following definition of the Berkovich spectrum of a Gelfand ring.

\begin{definition}\label{DefBerkovichGelfand}\
\begin{enumerate}
\item Let $B$ be a Banach algebra over $\Q_p$. The \textit{Berkovich spectrum} $\mathcal{M}(B)$ of $B$ is the topological space whose underlying set is the set of all continuous multiplicative seminorms $|-|\colon B\to \R_{\geq 0}$ such that $|p|=p^{-1}$. Given $x\in \mathcal{M}(B)$ an element of the Berkovich spectrum, we denote $|-|_x$ its associated multiplicative seminorm. Consider the inclusion map
\begin{equation}\label{eq0q3kp121e}
\mathcal{M}(B)\hookrightarrow \prod_{f\in B} \R_{\geq 0}
\end{equation}
sending $x\in \mathcal{M}(B)$ to the tuple $(|f|_x)_{f\in B}$. The topology of $\mathcal{M}(B)$ is defined to be the induced topology from the product topology in \cref{eq0q3kp121e}. In other words, the topology of $\mathcal{M}(B)$ is the coarsest topology making the maps   $\mathcal{M}(B)\to \R_{\geq 0}  $ sending  $ x\mapsto |f|_x$   for $f\in B$ a continuous map.  

\item Let $A$ be a Gelfand ring. We define the \textit{Berkovich spectrum of $A$} to be $$\mathcal{M}(A):= \mathcal{M}(A^u).$$

\item Let $A$ be a Gelfand ring, and $(f_1,\ldots, f_n,g)$ a tuple of elements in $A$ generating the unit ideal. Let $X=\mathcal{M}(A)$ be its Berkovich spectrum. We define the rational subspace $X(\frac{f_1,\ldots, f_n}{g})\subset X$ to be the closed subspace of multiplicative seminorms $|-|\colon A^u\to \R_{\geq 0}$ such that $|f_i|\leq |g|\neq 0$ for $i=1,\ldots, n$. A closed subspace $Z\subset X$ is called a \textit{rational subspace} if it admits a presentation of the form $Z=X(\frac{f_1,\ldots, f_n}{g})$ as before.
\end{enumerate}
\end{definition}

Let us recall some basic properties of the Berkovich spectrum.

\begin{remark}\label{RemarkBasicPropertiesBerkovichSpectrum}Let $B$ be a Banach $\Q_p$-algebra.
\begin{enumerate}
\item  The Berkovich spectrum  $\mathcal{M}(B)$ is a compact Hausdorff space. Indeed, any continuous multiplicative seminorm $|-|\colon B\to \R_{\geq 0}$ must send $B^{\leq 1}$ to the interval $[0,1]$, being a power bounded element. Hence, one has a closed embedding $\mathcal{M}(B)\hookrightarrow \prod_{f\in B^{\leq 1}} [0,1]$ making $\mathcal{M}(B)$ a closed subspace of a compact Hausdorff space and hence itself compact.

\item Let $S\subset B^{\leq 1}$ be a dense subspace. Then the map $\mathcal{M}(B)\to \prod_{f\in S} [0,1]$ sending $x$ to $(|f|_x)_{f\in S}$ is a closed embedding: this follows by continuity of the multiplicative seminorm. In particular, if $B$ admits a countable dense subspace (i.e. it is a separable Banach algebra), $\mathcal{M}(B)$ is a metrizable compact Hausdorff space.

\end{enumerate}

\end{remark}

\begin{remark}\label{rk:uniformity-for-banach-rings}
Recall that a Banach ring $B$ is said to be \textit{uniform} if the \textit{Gelfand transform} 
$$
B \to \prod_{x\in \mathcal{M}(B)}^{\mathrm{Ban}} \kappa(x)
$$
is an isometric embedding, where $\kappa(x)$ stands for the completion of the fraction field of the completion of $B$ with respect to the semi-norm corresponding to $x$ (a Banach field) and the product is taken in the category of Banach rings.  Equivalently (\cite[Proposition 2.21]{scholze2024berkovichmotives}), $B$ is uniform if the norm defining the Banach ring structure is power multiplicative. For a non-archimedean Banach ring which is Tate (i.e. admits a topologically nilpotent unit), this is also equivalent to the condition that the subring of power-bounded elements $B^\circ$ is bounded.  Over $\Q_p$, the category of uniform $\Q_p$-Banach algebras $B$, with continuous algebra morphisms, is equivalent to the category of $p$-adically complete and $p$-torsion free $\Z_p$-algebras $R$ with $R$ totally integrally closed in $R[1/p]$, with functors given by $B \mapsto B^\circ$ and $R \mapsto R[1/p]$ (with norm given by $|f|=\min \{ |p^n|, f\in p^n R\}$), see \cite[Proposition 5.2.5]{bhatt2017lecture}. In particular, for a Banach $\Q_p$-algebra $B$, its uniform completion (in the sense of Huber rings) coincides with the uniform completion as defined in \Cref{xwj29e}.

As a consequence, the uniform completion of a Gelfand ring is a uniform $\Q_p$-Banach algebra, i.e. one for which its Gelfand transform is an isometric embedding. This motivates our choice of terminology.  \end{remark}

\begin{lemma}\label{LemmaRationalLocalizations}
Let $f\colon A\to B$ be a morphism of Gelfand rings. Then we have a continuous map $\mathcal{M}(f)\colon \mathcal{M}(B)\to \mathcal{M}(A)$ that pullbacks rational subspaces to rational subspaces. Furthermore, if $A^u\xrightarrow{\sim}B^u$, then the map  $\mathcal{M}(f)$ identifies rational subspaces. 
\end{lemma}
\begin{proof}
The continuity  and stability of rational subspaces under pullbacks is classical from the Banach case. It is left to see that if $A$ and $B$ have the same uniform completions then rational subspaces are identified. We can assume without loss of generality that $B=A^u$. Let $X=\mathcal{M}(A)=\mathcal{M}(A^u)$ and  let $f_1,\ldots, f_n,g\in A^u$ be elements generating the unit ideal. We claim that there is $\epsilon >0$ such that   for any choice of elements $\widetilde{f}_1,\ldots, \widetilde{f}_n, \widetilde{g}$ in $A$ such that $f_i-\widetilde{f}_i\in A^{u,\leq \epsilon}$ and $g-\widetilde{g} \in A^{u,\leq \epsilon}$, we have an equality of rational subspaces
\[
X(\frac{f_1,\ldots, f_n}{g})= X(\frac{\widetilde{f}_1,\ldots, \widetilde{f}_n}{\widetilde{g}}).
\]
This implies the lemma as the image of $A$ in $A^u$ is dense. To prove the claim, after multiplying all $f_i$ and $g$ by a pseudo-uniformizer, we can assume without loss of generality that they belong to $A^{u,\leq 1}$.  We note that  as $(f_1,\ldots, f_n,g)$ generate the unit ideal, there exists some pseudo-uniformizer $\pi\in \Q_p$ such that $X(\frac{f_1,\ldots, f_n}{g})\subset X(\frac{\pi}{g})$. We claim that any $\epsilon>0$ such that $\epsilon<|\pi|$ works. Indeed, if $g-\widetilde{g}\in A^{u,\leq \epsilon}$, then $x\in X$ satisfies $|g|_x\geq |\pi|$ if and only  $|\widetilde{g}|_{x} \geq |\pi|$ and for such seminorms one has $|g|_x=|\widetilde{g}|_x$. Similarly, if  $f_i-\widetilde{f}_i\in A^{u,\leq \epsilon}$ and if $x\in X$ is such that $|f_i|_x\leq |g|_x$ and $|g|_x\geq |\pi|$, then $|\widetilde{f}_i|_x= |f_i+ (\widetilde{f}_i-f_i)|\leq \sup\{|\widetilde{f}_i|, \epsilon\}$, and so $|\widetilde{f}_i|_x\leq |g|_x$.
\end{proof}

\begin{proposition}\label{PropUniversalPropertyRational}
Let $A$ be a Gelfand ring with Berkovich spectrum $X=\mathcal{M}(A)$. Let $Z\subset X$ be a rational subspace. Then there is an idempotent morphism  $A\to A_{Z}$ of Gelfand rings such that a map $A\to B$ of Gelfand rings induces a factorization $A\to A_{Z}\to B$ if and only if the map of Berkovich spectra $\mathcal{M}(B)\to X$ factors through $Z$. More explicitly, if $Z=X(\frac{f_1,\ldots, f_n}{g})$ is a presentation of $Z$ as a rational subspace with $f_i,g\in A$ generating the unit ideal, we have 
\begin{equation}\label{eqw01i2k3ew3}
A_{Z}= A[\frac{1}{g}]\otimes_{\Q_p} \Q_p\langle  T_1,\ldots, T_n \rangle_{\leq 1}/^{\mathbb{L}}(gT_i-f_i )= A\otimes_{\Q_{p}} \Q_{p}\langle T_1,\ldots, T_n \rangle_{\leq 1}/^{\mathbb{L}}(gT_i-f_i). 
\end{equation}
In that case, the map $\mathcal{M}(A_Z)\to \mathcal{M}(A)$ is an homeomorphism onto $Z$ and identifies rational subspaces of $\mathcal{M}(A_Z)$ and rational subspaces of $A$ contained in $Z$.  We call $A_Z$ the \rm{Berkovich rational localization} of $A$ along $Z$.
\end{proposition}
\begin{proof}
Let us show that the two presentations of the algebra \eqref{eqw01i2k3ew3} agree.  For this, it suffices to see that $g$ is invertible in the presentation $B=A\otimes_{\Q_{p}} \Q_{p}\langle T_1,\ldots, T_n \rangle_{\leq 1}/^{\mathbb{L}}(gT_i-f_i)$. Since $(f_1,\ldots, f_n,g)$ generate the unit ideal, there are elements $a_i,b\in A$ such that $\sum_{i} a_if_i + bg=1$, hence in the algebra $B$ we have $1=\sum_{i} a_i g T_i +bg= g(\sum_i a_i T_i+b)$ proving that $g$ is invertible as wanted.

The first presentation of \eqref{eqw01i2k3ew3} is  clearly idempotent: it is the composite of the idempotent maps 
\[
A\to A[\frac{1}{g}]\to A[\frac{1}{g}]\otimes_{\Q_p[T_1,\ldots, T_n]} \Q_p\langle T_1,\ldots, T_n\rangle_{\leq 1}
\] 
where $T_i\mapsto f_i/g$.  The second presentation shows that $A_Z$ is a Gelfand ring.  We will show that the algebra $A_Z$  satisfies the universal property stated in the lemma. The uniform completion $A_Z^u$ is the uniform completion of the classical rational localization of the Banach algebra $A^u$ along $Z$, thus,  it is classical that   $\mathcal{M}(A_Z)\to X$ induces an homeomorphism with $Z$ identifying rational localizations.  Let $A\to B$ be a morphism of Gelfand $\Q_p$-algebras such that $\mathcal{M}(B)\to X$ factors through $Z$. We claim that the map $B\to B\otimes_{A} A_{Z}$ is an isomorphism, as $A_Z$ is idempotent this  produces a  uniquely defined map $A_Z\to B$, establishing its universal property.   By the universal property of rational localizations for Banach algebras, which is classical,  we have a unique factorization $A^u\to A_Z^u\to B^u$ and therefore an isomorphism $B^u\xrightarrow{\sim} (B\otimes_A A_Z)^u$. By \cref{LemmaIndependenceRationalLocalizations} below we have an isomorphism $B\xrightarrow{\sim} B\otimes_{A} A_Z$ proving what we wanted. 
 \end{proof}

 \begin{lemma}\label{LemmaIndependenceRationalLocalizations}
Let $A$ be a Gelfand ring and $A\to A_Z$ a Berkovich  rational localization as in \Cref{eqw01i2k3ew3}. Suppose that $A^u\xrightarrow{\sim } A_Z^u$ is an isomorphism. Then $A\xrightarrow{\sim} A_Z$ is an isomorphism. 
\end{lemma}
\begin{proof}
The rational subspace $Z\subset \mathcal{M}(A)$ is the locus $\{|f_i|\leq |g|\neq 0\}$. It suffices to show that the following hold: 
\begin{enumerate}[(i)]
 \item  $g\in A$ is invertible.
 \item $\frac{f_i}{g}\in A^{\leq 1}$.
\end{enumerate}
Indeed, if this holds, then we have a map $A_Z=A[\frac{1}{g}] \otimes_{\Q_p }\Q_{p}\langle  T_1,\ldots, T_n \rangle_{\leq 1}/^{\mathbb{L}} (gT_i-f_i) \to A$, and since $A_Z$ is idempotent over $A$, we get $A\cong A_Z$. The claim (i) follows from the fact that $g\in A$ is invertible if and only if its image in $A^u$ is so (see \cref{xsu8rf} (4)). The claim (ii) follows from \cref{xsu8rf} (1) and the fact that the image of $f_i/g$ are in $A^{u,\leq 1}$.
\end{proof}

\begin{lemma}\label{LemmaVanishingSpectrum}
Let $A$ be a Gelfand ring, the following are equivalent: 

\begin{enumerate}
\item $A=0$.

\item $A^{\dagger-\red}=0$.

\item $A^u=0$. 

\item $\mathcal{M}(A)=\emptyset$.

\end{enumerate}
\end{lemma}
\begin{proof}
It is clear that (1)$\Rightarrow$(2)$\Rightarrow$(3)$\Rightarrow$ (4). To see that (4)$\Rightarrow$(1),  note that $\mathcal{M}(A)=\emptyset$ implies that $\mathcal{M}(A)=\{|1|\leq 1/2\}=\emptyset$, and by \cref{LemmaIndependenceRationalLocalizations} that the map $1\to A$ factors through $\Q_p\langle T\rangle_{\leq 1/2}\to A$, but $1-T$ is a unit in $\Q_p\langle T\rangle_{\leq 1/2}$ and so it is $0=1-1$ in $A$ proving that $A=0$ as wanted.  See also \cite[Theorem 2.14]{scholze2024berkovichmotives}.
\end{proof}

\begin{lemma}\label{LemmaFilteredColimits}
Let $A=\varinjlim_i A_i$ be a filtered colimit of Gelfand rings. Then the natural map of Berkovich spectra
\[
\mathcal{M}(A)\to \varprojlim_{i} \mathcal{M}(A_i)
\]
is an homeomorphism. Furthermore, given $Z\subset \mathcal{M}(A)$ a rational subspace, there is some $i$ and a rational subspace $Z_i\subset \mathcal{M}(A)$ such that $Z$ is the pullback of $Z_i$ along $\mathcal{M}(A)\to \mathcal{M}(A_i)$. 
\end{lemma}
\begin{proof}
This follows from the analogous statement for Banach spaces, see \cite[Proposition 3.2]{scholze2024berkovichmotives}.
\end{proof}

\begin{lemma}\label{LemmaNilradical}
Let $A$ be a Gelfand ring, $f\in A$ and let $X=\mathcal{M}(A)$. The following are equivalent: 
\begin{enumerate}

\item $f\in \Nil^{\dagger}(A)$.

\item The locus $X(f\neq 0)\subset X$ is empty.

\end{enumerate}
\end{lemma}
\begin{proof}
Suppose that $f\in \Nil^{\dagger}(A)$, then $f=0$ in $A^u$ and so $X(f\neq 0)=\emptyset$. Conversely, if $X(f\neq 0)=\emptyset$ one has that $X=X(|f|=0)=\bigcap_{r>0} X(|f|\leq r)$ which by \cref{PropUniversalPropertyRational} translates to the condition that $f\in \bigcap_{r>0} A^{\leq r}=\Nil^{\dagger}(A)$. 
\end{proof}

\begin{lemma}\label{LemmaPointsBerkovichSpectrum}
Let $A$ be a Gelfand ring. The following are equivalent:
\begin{enumerate}
\item Given $f\in A$, either $f\in \Nil^{\dagger}(A)$ or $f$ is invertible. 

\item $A^{\dagger-\red}$ is a field.

\item $A^u$ is a field. 

\item $\mathcal{M}(A)$ is a point.

\end{enumerate}
\end{lemma}
\begin{proof}
Since $\Nil^{\dagger}(A)$ is the kernel of $A\to A^{\dagger-\red}$, (1)$\Rightarrow$(2). It is clear that (2)$\Rightarrow$(3) as the  completion of a normed field is also a field. The implication (3)$\Rightarrow$(4) follows from the fact that the Berkovich spectrum of a field is just a point. Finally, for (4)$\Rightarrow$(1), let $f\in A$, and denote $X=\mathcal{M}(A)$. As a set we have that $X= X(|f|\neq 0) \bigsqcup X(|f|=0)$, since $X$ consists of a point this yields that either $X=X(|f|\neq 0)$ or $X=X(|f|=0)$. By \cref{PropUniversalPropertyRational} this means precisely that either $f$ is invertible or $f\in \Nil^{\dagger}(A)$ respectively. 
\end{proof}

\begin{definition}\label{DefinitionResidueField}
Let $A$ be a Gelfand ring and let $x\in \mathcal{M}(A)$ be a point in its Berkovich spectrum. The \textit{stalk of $A$ at $x$} is the Gelfand ring $$A_x=\varinjlim_{x\in Z\subset \mathcal{M}(A)} A_Z$$ defined as the filtered colimit of rational localizations of $A$ containing $x$ (note that  $\mathcal{M}(A_x)$ is a point by  \Cref{LemmaFilteredColimits}).  We define the \textit{Berkovich residue field of $A$ at $x$} to be $$\kappa(x):= (A_x)^u$$ (note that $\kappa(x)$ is a field thanks to \Cref{LemmaPointsBerkovichSpectrum}), equivalently we define $\kappa(x)$ to be the completed residue field at $x\in\mathcal{M}(A)= \mathcal{M}(A^u)$ of the Banach algebra $A^u$.
\end{definition}

\begin{proposition}\label{PropMapSmashing}
Let $A$ be a Gelfand ring and $X=\mathcal{M}(A)$ its Berkovich spectrum. The functor sending a rational subspace $Z\subset X$ to the idempotent algebra $A_Z$ of \cref{PropUniversalPropertyRational} promotes uniquely to a   surjective morphism of locales $F\colon \mathrm{Sm}(\ob{D}(A))\to \mathcal{M}(A)$ from the smashing spectrum of $\ob{D}(A)$ to $\mathcal{M}(A)$, such that for $Z\subset \mathcal{M}(A)$ a rational localization one has that $F^{-1}(Z)=\ob{D}(A_Z)$.  Moreover, if $A\to B$ is a morphism of Gelfand rings, we have a (necessarily unique) commutative diagram of locales
\begin{equation}\label{eqwpqp3nqef}
\begin{tikzcd}
\mathrm{Sm}(\ob{D}(B)) \ar[r] \ar[d] & \mathcal{M}(B) \ar[d] \\ 
\mathrm{Sm}(\ob{D}(A)) \ar[r] & \mathcal{M}(A) .
\end{tikzcd}
\end{equation}
\end{proposition}
\begin{proof}
Let $X=\mathcal{M}(A)$ and consider the site $X_{\mathrm{rat}}$ of rational subspaces of $X$ with covers given by finitely many jointly surjective rational covers. By \cref{PropUniversalPropertyRational} we have a functor 
\[
\ob{D}\colon X_{\mathrm{rat}}^{\op}\to \Cat{Pr}^L_{\ob{D}(A)}
\]
sending $Z\subset X$ to $\ob{D}(A_Z)$ with morphisms given by pullbacks.  We claim that $\ob{D}$ satisfies descent.  It suffices to show that $\ob{D}$ satisfies descent along  a rational cover $\{X_i\to X\}$.  By \cite[Proposition 4.3]{andreychev2021pseudocoherent} any rational cover is refined by a composite of either simple Laurent covers of the form $X=X(\frac{1}{f}) \cup X(\frac{f}{1})$ or simple balanced covers $X=X(\frac{1}{f})\cup X(\frac{1}{1-f})$, hence, by a d\'evissage argument it suffices to consider these two cases. For  both situations, it suffices to consider the  universal case $A=\Q_p\langle T\rangle_{\leq r}$ with $r>0$, and the statement reduces to show that the following sequences are exact 
\[
0\to \Q_p\langle T \rangle_{\leq r} \to \Q_p\langle T \rangle_{\leq 1} \oplus \Q_p\langle T \rangle_{\leq r} \langle T^{-1} \rangle_{\leq 1} \to \Q_p\langle T^{\pm 1} \rangle_{\leq 1}\to 0 
\]
and 
\[
0\to \Q_p\langle T \rangle_{\leq r} \to \Q_p\langle T \rangle_{\leq r} \langle (1-T)^{-1} \rangle_{\leq 1}  \oplus \Q_p\langle T \rangle_{\leq r} \langle T^{-1} \rangle_{\leq 1} \to \Q_p\langle T  \rangle_{\leq r}\langle T^{-1}, (1-T)^{-1} \rangle_{\leq 1}\to 0, 
\]
which is a standard computation.

Next, we construct the map of locales $F\colon \ob{Sm}(\ob{D}(A))\to X$. We proceed in several steps:

\textit{Step 1.} Consider the ring $A=\mathbb{Q}_p\langle T \rangle_{\leq r}$. We  claim that the collection of rational subspaces $\{a \leq |T|\leq b\}$ for $0 \leq a< b\leq r$ gives rise to a map of locales 
\[
|T|\colon  \ob{Sm}(\ob{D}(\mathbb{Q}_p\langle T \rangle_{\leq r}))\to [0,r].
\]
Indeed, given a closed interval $[a,b]\subset [0,r]$ we have the idempotent morphism of algebras 
\[
A\to A([a,b])
\]
where $A([a,b])$ is the rational localization given by $a\leq |T|\leq b$.   Using the considerations at the beginning of the proof and direct computations, one sees that:

\begin{enumerate}[(a)]
\item If  $[a,b]\cap [c,d]=\emptyset$ then $A([a,b])\otimes_A A([c,d])=0$.

\item If $a<c<b<d$ then $A([a,b])\otimes_A A([c,d])=A([c,b])$.

\item If $[a,b]=\varprojlim_i [a_i,b_i]$ then  $A([a,b])=\varinjlim_{i} A([a_i,b_i])$.

\item If $[a,b]=[c,d]\cup [f,g]$ then $A([a,b])= A([c,d])\times_{A([c,d]\otimes_A A[f,g])} A([f,g])$.

\end{enumerate}

Consider a finite disjoint union of intervals $I=\bigsqcup_i I_i$, by the previous points we have an idempotent $A$-algebra $A_I=\prod_{i} A_I$. Let $U=[0,r]\backslash I$  be the complement of $I$ in $[0,r]$, and define the coidempotent coalgebra $C_U=\ob{fib}(A\to A_I)$. For general open subspace $V\subset [0,r]$, we define the coidempotent coalgebra 
\[
C_V= \varinjlim_{U \subset V} C_U
\]
where $U$ runs over the basis of complements of finite disjoint unions of closed  intervals in $[0,r]$ contained in $V$. Then $C_V$ is a coidempotent coalgebra  of $A$ (being a filtered colimit of such), and it verifies the following properties (as consequence of (a)-(d) above):
\begin{enumerate}[(i)]
\item If $U=\bigcup_i U_i$ is a union open subspaces then $C_U=\varinjlim_i U_i$.

\item $C_U\otimes_A C_{U'}= C_{U\cap U'}$.

\item For $U=[0,r]\backslash I$  the complement of a finite disjoint union of intervals, we have that $C_U=\ob{fib}(A\to A_I)$, recovering the original definition for $U$.
\end{enumerate}

The datum of the coidempotent coalgebras $C_U$ for $U\subset [0,r]$ open is precisely the datum of a map of locales 
\[
F\colon \ob{Sm}(\ob{D}(A))\to [0,r].  
\]
It is also clear from the construction that $F^{-1}([a,b])=\ob{D}(A([a,b]))$.

\textit{Step 2.} Let $A$  be a general Gelfand algebra. Thanks to Step 1, given $f\in A^{\leq 1}$ we have a natural morphism of locales
\[
|f|\colon \ob{Sm}(\ob{D}(A))\to \ob{Sm}(\ob{D}(\Q_p\langle T \rangle_{\leq 1})) \to [0,1].
\] 
Taking the product along all the $f\in A^{\leq 1}$ we get a map 
\[
F\colon \ob{Sm}(\ob{D}(A))\to \prod_{f\in A^{\leq 1}} [0,1].
\] 
We claim that the map $F$ factors through the Berkovich spectrum $\mathcal{M}(A)\subset \prod_{f\in A^{\leq 1}} [0,1]$, where the inclusion is given by evaluating $f$ in the norm associated to the point $x\in \mathcal{M}(A)$, see \cref{RemarkBasicPropertiesBerkovichSpectrum} (2). To prove this, we have  to see that if $U\subset \prod_{f\in A^{\leq 1}} [0,1]$  is an open  subspace disjoint to $\mathcal{M}(A)$, then $F^{-1}(U)=\emptyset$.  Let $f_1,\ldots, f_n\in A^{\leq 1}$ and consider the projection 
\[
(|f_i|)_i\colon \mathcal{M}(A)\to \prod_{i=1}^n [0,1].
\]
It suffices to show that if the pre-image of the  locus $\bigcap_{i=1}^n \{ a_i \leq |f|\leq b_i   \}\subset \prod_{i=1}^n [0,1]$ (with $a_i<b_i$) in $ \mathcal{M}(A)$ is empty, then its pre-image in $\ob{Sm}(\ob{D}(A))$ is also empty. This translates to the fact that if $\mathcal{M}(A)(a_i \leq  |f_i|\leq b_i )=\emptyset$ then $A\otimes_{\Q_p[T_1,\ldots, T_n]} B=0$, where $B$ is the pushout of the overconvergent algebras $\Q_p\langle T_i :  a_i\leq |T_i| \leq b_i   \rangle^{\dagger}$, and $T_i\mapsto f_i$.  But this is a rational localization of $A$ with empty Berkovich spectrum which must vanish thanks to  \cref{LemmaVanishingSpectrum}. This proves the claimed factorization
\[
F\colon \ob{Sm}(\ob{D}(A))\to \mathcal{M}(A). 
\]
It is clear from the construction that for $Z\subset \mathcal{M}(A)$ a rational subspace one has $F^{-1}(Z)=\ob{D}(A_Z)$ with $A_Z$ as in \cref{PropUniversalPropertyRational}.  The functoriality with respect to maps of Gelfand rings is also clear.

Finally,  we want to see that $F$ is surjective, for that it suffices to see that the pre-image of any point is non-empty, this follows from the fact that given $x\in \mathcal{M}(A)$ the map $A\to \kappa(x)$ towards the residue field induces the inclusion $x\in \mathcal{M}(A)$, and so $\kappa(x)$ is an $A$-module supported on $x\in \mathcal{M}(A)$. 
\end{proof}

The following lemma discusses a particular case where the uniform completion of a Gelfand ring is perfectoid.

\begin{lemma}\label{xhsjwks}\label{lemma:berkovich-spaceprofinite-residue-fields-closed-implies-perfectoid}
Let $A$ be a Gelfand ring such that $\mathcal{M}(A)$ is profinite and such that each  Berkovich  residue field is perfectoid. Then $A^u$ is perfectoid.
\end{lemma}
\begin{proof}
We can assume without loss of generality that $A=A^u$ is uniform. We want to show that $A^u$  is a perfectoid ring. Since $\mathcal{M}(A)$  is a profinite set, the presheaf mapping a clopen subspace $U\subset \mathcal{M}(A)$ to $(A_U)^{\leq 1}$ is  a sheaf of condensed rings, so it suffices to prove that $A$ is perfectoid locally in its analytic topology. We first show that $A$ admits a pseudo-uniformizer $\pi\in A^{<1}$ (locally in the analytic topology) such that  $\pi^p| p$. Let $x\in \mathcal{M}(A)$, we know that $\kappa(x)$ is perfectoid. So   $\kappa(x)$ admits such pseudo-uniformizer $\pi$ and by $p$-adic approximation we can assume that $\pi$ extends to a function in a neighbourhood $x\in U\subset \mathcal{M}(A)$. By further refining the neighbourhood we we can even assume that $\pi^p| p$ in $A_U$ obtaining the desired pseudo-uniformizer. By replacing $A$ by $A_U$ we can assume that $\pi$ is defined over  $A$. It is left to show that the Frobenius map $\varpi: A^{\leq 1}/p \to A^{\leq 1}/p$ is surjective, since the formation $\mathcal{F}:U\mapsto (A_U)^{\leq 1}/p$ is a sheaf over $\mathcal{M}(A)$ and this space is acyclic  (again, since it is profinite)  we can check the surjectivity on stalks. But the stalk at $x$ of $\mathcal{F}$ is just $\kappa(x)^{\leq 1}/p$ by \cref{xsu29s}, since  $\kappa(x)$ is perfectoid we know that Frobenius is surjective proving what we wanted.
\end{proof}

We finish this subsection with a statement relating uniform completion and perfectoidization, \Cref{sec:defin-main-prop-uniform-completion-and-perfectoidization}.

\begin{lemma}\label{x5134hs}
Let $A\to B$ be a morphism of Gelfand rings such that the map of uniform completions $A^u\to B^u$ is a surjection. Then $F\colon \mathcal{M}(B)\to \mathcal{M}(A)$ is a Zariski closed immersion given by the vanishing locus of the ideal $I=\ob{ker}(A\to B)$. If in addition $F$ is an isomorphism then $A^u\overset{\sim}{\to}B^u$.
\end{lemma}
\begin{proof}
Let $I^u:= \ker(A^u\to B^u)$. We know that $\mathcal{M}(B)$ is the zero locus of $I^u$ by taking Hausdoff quotients of \cite[Corollary 2.7.15]{camargo2024analytic}.  Thus, it suffices to show that the zero locus of $I$ and $I^u$ agree. Without loss of generality we can replace $A$ and $B$ by their $\dagger$-reductions, in this situation $A\subset A^u$ and $B^u$, so that $I=A\cap I^u$. Let $I^{\leq 1}= A^{\leq 1} \cap I^u$, then the $p$-completion of $I^{\leq 1}$ gives rise to a $\Z_p$-Banach lattice $I^{u,\leq 1}\subset I^u$. Thus, for all $n\in \N$ and  $f\in I^u$ there is $f_{n}\in I$ such that $f-f_n\in p^n I^{u,\leq 1}$. This implies that the loci $|f_n|\leq 1/p^n$ and $|f|\leq 1/p^n$ in $\mathcal{M}(A^u)$ are the same. Taking $n\to \infty$ one deduces that $|f|=0$ in the vanishing locus of $I$, proving what we wanted.  

 Suppose now that $F$ is an homeomorphism. By the open mapping theorem, to prove that $A^u\to B^u$ is an isomorphism it suffices to show that it is a bijection. Since it is already surjective we only need to prove injectivity. Let $f\in \ker(A^u\to B^u)$, then the locus where $\{|f|\neq 0\}$ is empty and thus $f\in \ob{Nil}^{\dagger}(A^u)$. But since $A^u$ is an uniform Banach algebra it has trivial $\dagger$-radical and so $f=0$.
\end{proof}

We use the previous discussion and ``Zariski closed implies strongly Zariski closed'' to prove the following link between uniform completion and perfectoidization.

\begin{proposition}
  \label{sec:defin-main-prop-uniform-completion-and-perfectoidization}
Let $A$ be a perfectoid Banach ring and let $A\to B$ be a quotient. Then the natural map $B^u\to B_{\ob{perfd}}$ from the uniform completion to its perfectoidization, i.e., the initial perfectoid Banach ring under $B$, is an isomorphism. Furthermore, the map $B\to B^u$ is surjective.   
\end{proposition}
\begin{proof}
We may replace $B$ by $\pi_0(B)$ and assume that $B$ is static (this does not change $B^{\dagger-\red},\ B^u, B_{\ob{perfd}}$). We now clarify the existence of $B_{\ob{perfd}}$. Let $I:=\ob{ker}(A\to B)$ be the ideal defining $B$, and $I^{\leq 1}:=I\cap A^{\leq 1},\ B_0:=A^{\leq 1}/I^{\leq 1}$. Then the $p$-completion of $B_0$ is semiperfectoid and given by $B_0^{\wedge_p}= A^{\leq 1}/\overline{I}$, where $\overline{I}$ is the completion of $I$ in $A^{\leq 1}$ for the $p$-adic topology. In particular, $A^{\leq 1}\to (B_0)^{\wedge_p}$ is a surjection of $p$-complete rings.  Let $C$ be a perfectoid ring and consider a map $g\colon B\to C$, then, as $C$ is a Banach ring,  $g$ musth factor through $B\to B_0^{\wedge_p}[\frac{1}{p}]\to C$ and the perfectoidization  of $B$ and $(B_0)^{\wedge_p}[\frac{1}{p}]$ must agree. This last is nothing but the generic fiber of the perfectoidization of $(B_0)^{\wedge_p}$ in the sense of \cite[Corollary 7.3]{Bhatta}.   By \cite[Theorem 7.4 and Remark 7.5]{Bhatta}, the map $A\to B_{\ob{perfd}}$ is surjective and the associated map of adic spaces
\[
Z=\Spa (B_{\ob{perfd}})\to \Spa( A)
\]
is a Zariski closed embedding given by the vanishing locus of $\overline{I}$ which agrees with the vanishing locus of $I$  by \cref{x5134hs}, the same holds for Berkovich spaces $\mathcal{M}(B_{\ob{perfd}})\to \mathcal{M}(A)$.  But then, \cref{x5134hs} implies that   the uniform completion of $B$ and of $A^{\dagger-I}=\varinjlim_{Z\subset U,\ U\textrm{ rational open}} A_U$ agree and are equal to $B_{\ob{perfd}}$, proving what we wanted.
\end{proof}

\subsection{The morphism to the Berkovich spectrum}\label{sss:map-to-the-berkovich-spectrum}
In this subsection, we explain (following \cite{AnStacks}) the construction of a natural morphism
$$
\ob{AnSpec}(A)\to \mathcal{M}(A)_{\mathrm{Betti}}
$$
of analytic stacks from the analytic spectrum of a Gelfand $\Q_p$-algebra $A$ to the Betti incarnation of its Berkovich spectrum, under some finite dimensional hypothesis. The main result is \Cref{xhs82j}. \medskip

Throughout, we write $\Cat{AnRing}$ for the category of analytic rings and $\Cat{AnStk}$ for the category of analytic stacks. For the definition of the latter, we refer the reader to \cref{sec:tdstacks}.
\medskip

In order to construct a map from $A$ towards its Berkovich spectrum as analytic stacks we need to recall a few facts about the Betti stacks. We start with the functor $\Cat{Prof}^{\rm light} \to \Cat{AnStk}$, sending a light profinite set $S$ to $S_{\rm Betti}=\AnSpec(C(S,\Z))$. We have the following lemma.

\begin{lemma}\label{xs2hd9}
Let $\Cat{Ring}^{\delta}_{\omega_1}$ be the $\infty$-category of countable animated discrete rings. Consider the functor
$$
\ob{D}^!: \Cat{Ring}^{\delta,\op}_{\omega_1}\to \Cat{Cat}_{\infty}
$$
given by mapping a ring $A\in \Cat{Ring}^{\delta}_{\omega_1}$ to its derived category of (condensed) modules with transition maps given by upper $!$-maps. Then $D^!$ is a flat hypersheaf.
\end{lemma}

\begin{proof}
  By \cite[Proposition 3.31]{mathew2016galois} flat covers in $\Cat{Ring}^{\delta}_{\omega_1}$ are $2$-descendable, so the functor $\ob{D}^!$ is a sheaf for the flat topology.
  We want to show that it is an hypersheaf.  By \cite[Proposition A.3.21]{mann2022p} it suffices to show that for any flat hypercover $A\to A_{\bullet}$ in $\Cat{Ring}^{\delta}_{\omega_1}$, the functor 
\begin{equation}\label{eqnsjs8} 
\ob{D}(A)\to \varprojlim_{[n]\in \Delta} \ob{D}^!(A_n)
\end{equation}
is fully faithful. By applying the (dual of the)  strategy of \cite[Proposition 3.1.26]{mann2022p} it suffices to prove that given a flat map of countable discrete rings $f:A\to B$ the functor $f^!(-)=\Hom_{A}(B,-)$ has cohomological dimension $\leq 1$.\footnote{It suffices to have an universal bound, but it turns out that $1$ is optimal.} By Lazard's theorem  we can write $B$ as a filtered colimit of finite projective $A$-modules (\cite[Theorem 7.2.2.15]{lurie_higher_algebra}). As $A, B$ are countable (hence $B$ is an $\omega_1$-compact $A$-module), we can assume that the filtered colimit is countable. Then the cohomological bound follows.
\end{proof}

Since light condensed anima are hypersheaves on light profinite sets, using \cref{xs2hd9} we can extend the above functor $\Cat{Prof}^{\rm light} \to \Cat{AnStk}$ to a left exact colimit preserving functor\footnote{Note how lightness is used crucially here. Also, note that the desire to construct such a functor is the reason for requiring that analytic stacks are not merely sheaves on analytic rings for the $!$-topology, but a stronger condition.}
$$
\Cat{CondAni}^{\rm light} \to \Cat{AnStk}, X \mapsto X_{\rm Betti}
$$
\begin{remark}
By \cite[Remark 3.5.15]{heyer20246functorformalismssmoothrepresentations}, if $X$ is a locally compact Hausdorff space, $\ob{D}(X_{\Betti})$ is the left completion of the derived category $\Shv(X,\Z^{\mathrm{cond}})$ of \textit{condensed abelian sheaves} on $X$. If $X$ is paracompact and has local finite $\Z$-cohomological dimension, the natural symmetric monoidal functor $\Shv(X,\Z^{\mathrm{cond}})\to \ob{D}(X_{\Betti})$ is an equivalence. 

 Notice that for $X$ a condensed anima,  $\ob{D}(X_{\Betti})$ has a full subcategory $\ob{D}^{\delta}(X_{\Betti})$ obtained by hyperdescent from discrete modules over the rings $C(S,\Z)$ with $S$ light profinite, we call $\ob{D}^{\delta}(X_{\Betti})$ the full subcategory of \textit{discrete sheaves on $X_{\Betti}$}. Similarly as before, if $X$ is a locally compact Hausdorff space,  $\ob{D}^{\delta}(X_{\Betti})$ is the left completion of the category of abelian sheaves $\Shv(X, \Z)$ on $X$, and, if $X$ is paracompact and has local finite $\Z$-cohomological dimension, then $\Shv(X,\Z)=\ob{D}^{\delta}(X_{\Betti})$.
\end{remark}

As we now explain, we can give a functor of points description of Betti stacks, cf.\ \cite[Lecture 20]{AnStacks}. Recall from \cite[Theorem A]{aoki2023sheavesspectrumadjunction} that the functor
\[
\ob{Shv}(-; \Cat{Sp}): \Cat{Loc}^{\op}\to \ob{CAlg}(\Cat{Pr}^L_{\ob{st}})
\]
sending a locale to its category of $\mathrm{Sp}$-valued sheaves has a right adjoint given by the smashing spectrum $\mathrm{Sm}(-)$, i.e., the locale of idempotent algebras.  
 This  gives for  $\mathcal{C}$ a presentably symmetric monoidal stable $\infty$-category and a locale $X$  a natural equivalence of anima 
\begin{equation}\label{eqsns3k}
\ob{Map}_{\Cat{Loc}^{\op}}(X,\mathrm{Sm}(\mathcal{C})) = \ob{Map}_{\ob{CAlg}(\Cat{Pr}^L_{\ob{st}})}(\ob{Shv}(X;\ob{Sp})^{\otimes}, \mathcal{C}).
\end{equation}

We note that connective idempotent algebras over analytic rings, which a priori are only $\mathbb{E}_\infty$-algebras, naturally acquire the structure of an animated ring. This crucially uses lightness.

\begin{lemma}\label{xi2msnm}
Let $A$ be an analytic ring and let $D\in \ob{CAlg}(\ob{D}_{\geq 0}(A))$ be a connective commutative algebra. Suppose that $A\to D$ is idempotent, then $D$ has a natural structure of an animated ring making $A\to D$ a localization of analytic rings. 
\end{lemma}
\begin{proof}
This follows form the fact that $\ob{Mod}_{D}(\ob{D}_{\geq 0}(A))\subset \ob{D}_{\geq 0}(A)$ is a full subcategory stable under all limits, colimits and internal Homs. Indeed, this category defines an uncompleted analytic ring structure $D_{A/}$ on $A$ by \cite[Definition 4.1.1]{SolidNotes}, and the animated  ring structure on $D_{A/}[*]=D$ is automatic in the light set up by \cite[Corollary 4.2.7]{SolidNotes}.
\end{proof}

\begin{lemma}\label{xhsw82h}
Let $X$ be a metrizable compact Hausdorff space of finite cohomological dimension, let $S$ be a light profinite set and $f:S\to X$ a surjection. Then the map of analytic stacks $f:S_{\ob{Betti}}\to X_{\ob{Betti}}$ is proper and descendable, i.e, $f_\ast(1)\in \ob{D}(X_{\ob{Betti}})$ is a descendable algebra.
\end{lemma}
\begin{proof} 
Properness follows from \cite[Lemma 4.8.2 (ii)]{heyer20246functorformalismssmoothrepresentations}, descendability is proven in \cite[Proposition II.1.1]{ScholzeRLL}.
\end{proof}

\begin{proposition}\label{xj29sj}
Let $X$ be a metrizable compact Hausdorff space. For $A$ an analytic ring let $\ob{Map}_{\Cat{Loc}^{\op}}^+(X,\mathrm{Sm}(\ob{D}(A))) \subset \ob{Map}_{\Cat{Loc}^{\op}}(X,\mathrm{Sm}(\ob{D}(A))) $ be the full subspace of maps of locales satisfying the following condition:  there is a surjective map from a light profinite set $S\to X $, a $!$-cover $A\to B$ and a morphism of analytic rings $C(S,\mathbb{Z})\to B$  that induces a  commutative diagram of locales
\begin{equation}\label{eqdj2idj}
\begin{tikzcd}
\mathrm{Sm}(\ob{D}(B)) \ar[r] \ar[d]& \mathrm{Sm}(\ob{D}(A)) \ar[d] \\ 
S \ar[r] & X .
\end{tikzcd}
\end{equation}
Then the natural map 
\begin{equation}\label{eqj2jdsn}
\ob{Map}(\ob{AnSpec}(A), X_{\ob{Betti}})\to  \ob{Map}_{\ob{CAlg}(\Cat{Pr}^L_{\ob{D}(\mathbb{Z})})}(\ob{Shv}(X;\mathbb{Z}^{\ob{cond}}), \ob{D}(A)) =\ob{Map}_{\Cat{Loc}^{\op}}(X,\ob{Sm}(\ob{D}(A))).
\end{equation}
 induces an equivalence onto $\ob{Map}_{\Cat{Loc}^{\op}}^+(X,\mathrm{Sm}(\ob{D}(A)))$ (here $\ob{D}(\Z)$ is the derived category of condensed abelian groups). \end{proposition}
 
\begin{remark}\label{RemarkBaseChangeTopoi}
The mapping space in the middle of \cref{eqj2jdsn} involves sheaves on $X$ valued in condensed abelian groups, while the adjunction of \cite[Theorem A]{aoki2023sheavesspectrumadjunction} involves sheaves valued in spectra. This does not cause any problem since, given an $\infty$-topos $X$ and a presentable category $\mathcal{C}$, one has a natural equivalence of presentable categories
\[
X\otimes \mathcal{C}\cong \ob{Shv}(X, \mathcal{C})
\]
where the left term is Lurie's tensor product, and the right term is the category of $\mathcal{C}$-valued sheaves on $X$. Indeed, by \cite[Proposition 4.8.1.17]{lurie_higher_algebra}, one has a natural equivalence
\[
X\otimes \mathcal{C} \cong \ob{RFun}(X^{\op}, \mathcal{C})
\]
where the right term is the category of functors $X^{\op}\to \mathcal{C}$  which admit left adjoints.  As $X$ is presentable, the latter identifies with the category of limit preserving functors $X^{\op}\to \mathcal{C}$ which is, by definition, the category $\ob{Shv}(X,\mathcal{C})$ of $\mathcal{C}$-valued sheaves on $X$.
\end{remark}

\begin{proof}[Proof of \Cref{xj29sj}]
Let us first suppose that $X=S$ is profinite. Then $S_{\ob{Betti}}=\ob{AnSpec} ( C(S,\mathbb{Z}) )$ is an affinoid stack and then the LHS term of \cref{eqj2jdsn} is just the same as morphisms of analytic rings $C(S,\mathbb{Z})\to A$, which, as $C(S,\Z)$ is discrete, is the same as morphism of animated rings $C(S,\Z)\to A^{\triangleright}(*)$. We want to see that the natural map
\begin{equation}\label{xj2jd09}
\ob{Map}_{\Cat{AnRing}}(C(S,\mathbb{Z}), A)\to \ob{Map}_{\Cat{Loc}^{\op}}(S,\mathrm{Sm}(\ob{D}(A)))
\end{equation}
already induces a bijection. For proving this, we compose with the map \cref{eqsns3k}. For any clopen subspace $U\subset S$ the map $C(S,\mathbb{Z})\to C(U,\mathbb{Z})$  is idempotent. Since $\Z\to C(S,\Z)$ is and ind-\'etale map of $\mathbb{E}_{\infty}$-rings (resp. of animated rings), we have that the map
\[
\begin{aligned}
\ob{Map}_{\Cat{AnRing}}(C(S,\mathbb{Z}), A)=\ob{Map}_{\Cat{Ani(Ring)}}(C(S,\mathbb{Z}), A^{\triangleright}) \\  \to \ob{Map}_{\ob{CAlg}_{\mathbb{Z}}}(C(S,\mathbb{Z}), A^{\triangleright})=\ob{Map}_{\Cat{AnRing}_{\mathbb{Z}-\mathbb{E}_{\infty}}}(C(S,\mathbb{Z}), A)
\end{aligned}
\]
from morphisms between analytic rings to morphisms between the underlying $\mathbb{Z}-\mathbb{E}_{\infty}$-rings is an equivalence; here, we denoted by $A^{\triangleright}$ the condensed ring underlying $A$. By \cite[Lemma 2.1.3]{camargo2024analytic} the functor $\ob{D}: \Cat{AnRing}_{\Z-\mathbb{E}_{\infty}}\to \ob{CAlg}(\Cat{Pr}^L_{\ob{D}(\Z)})$ from $\Z-\mathbb{E}_{\infty}$-analytic rings to condensed $\mathbb{Z}$-linear presentable symmetric monoidal categories is fully faithful. By \cref{eqsns3k} we deduce that \cref{xj2jd09} is an equivalence as wanted.

For proving the statement about general $X$, note first that the functors sending an analytic ring $A$ to $\ob{Map}^+_{\Cat{Loc}^{\op}}(X,\mathrm{Sm}(\ob{D}(A)))$  and $\ob{Map}^+_{\Cat{Loc}^{\op}}(X,\mathrm{Sm}(\ob{D}(A)))$  satisfy universal $*$-descent (and so in particular $!$-descent). Thus, it suffices to prove the desired claim locally in the $!$-topology of $A$.  First, we show that \cref{xj2jd09} factors through $\ob{Map}^+_{\Cat{Loc}^{\op}}(X,\mathrm{Sm}(\ob{D}(A)))$. Indeed, let $f:\ob{AnSpec}(A) \to X_{\ob{Betti}}$ be a morphism of analytic stacks and let $S\to X$ be an epimorphism of condensed anima with $S$ a light profinite set. Then $S_{\ob{Betti}}\to  X_{\ob{Betti}}$ is an epimorphism of analytic stacks and then so is $S_{\ob{Betti}}\times_{X_{\ob{Betti}}} \ob{AnSpec}(A)\to \ob{AnSpec}(A)$. Thus, there is a $!$-cover $A\to B$ and a lift $\ob{AnSpec}(B) \to S_{\ob{Betti}}\times_{X_{\ob{Betti}}} \ob{AnSpec }(A)$. This provides the map desired map $C(S,\mathbb{Z})\to B$ making \cref{eqdj2idj} commute.  Now, let us show that  the map
\[
\ob{Map}(\ob{AnSpec}(A), X_{\ob{Betti}})\to  \ob{Map}_{\Cat{Loc}^{\op}}^+(X,\mathrm{Sm}(\ob{D}(A)))
\]
is surjective. Let $\mathrm{Sm}(\ob{D}(A))\to X $ be a morphism of locales in $\ob{Map}_{\Cat{Loc}^{\op}}^+(X,\mathrm{Sm}(\ob{D}(A)))$ corresponding to a $\ob{D}(\Z)$-linear symmetric monoidal functor $\ob{Shv}(X, \mathbb{Z}^{\ob{cond}})\to \ob{D}(A)$, by hypothesis there is a profinite set $S$, a surjective map $S\to X$ and  $!$-cover $A\to B$  with  a map $C(S,\mathbb{Z})\to B$ making \cref{eqdj2idj} commute.  By \cite[Theorem A]{aoki2023sheavesspectrumadjunction} this gives rise to a commutative square of $\ob{D}(\Z)$-linear symmetric monoidal categories
\[
\begin{tikzcd}
\ob{D}(B)  & \ob{D}(A)  \ar[l] \\ 
\ob{Shv}(S,\mathbb{Z}^{\ob{cond}}) \ar[u] & \ob{Shv}(X,\mathbb{Z}^{\ob{cond}}) \ar[u] \ar[l]
\end{tikzcd}
\]
Taking \v{C}ech nerves, and knowing that the functor $\Cat{AnRing} \to \ob{CAlg}(\Cat{Pr}^L_{\ob{D}(\Z)}), A \mapsto \ob{D}(A)$, preserves colimits (see \cite[Proposition 4.1.14]{SolidNotes}), we have a  cosimplicial diagram of symmetric monoidal categories
\[
( \ob{Shv}(S^{\times_X n+1 },\mathbb{Z}^{\ob{cond}}) )_{[n]\in \Delta} \to (\ob{D}(B^{\otimes_A n+1}))_{[n]\in \Delta}
\]
which by the case of profinite sets already proven corresponds to a  simplicial diagram of morphisms of analytic stacks
\[
(\ob{AnSpec}(B^{\otimes_A n+1}))_{[n]\in \Delta^{\op}} \to (S^{\times_X n+1}_{\ob{Betti}})_{[n]\in \Delta^{\op}}.
\]
Taking geometric realizations, and knowing that $\ob{AnSpec}(B)\to \ob{AnSpec}(A) $ is an epimorphism,  we get the desired map of analytic stacks
\[
\ob{AnSpec}(A)\to X_{\ob{Betti}}.
\]

It is  left to show that \cref{eqj2jdsn} is fully faithful. Since $X_{\ob{Betti}}$ is $0$-truncated (as it arises  from a $0$-truncated condensed anima), we only have to show that two maps $f,g:\ob{AnSpec}(A)\to X_{\ob{Betti}}$ inducing the same morphism of locales $\ob{Sm}(D(A))\to X$ must agree. Consider the map
\[
(f,g):\ob{AnSpec}(A) \to (X\times X)_{\ob{Betti}}=X_{\ob{Betti}}\times X_{\ob{Betti}}
\]
we want to see that this map factors through the diagonal. We know that this holds at the level of locales. Therefore, it suffices to prove the following claim: let $f: \ob{AnSpec}(A)\to X_{\ob{Betti}}$ be a morphism of analytic stacks such that the morphism of locales $\ob{Sm}(\ob{D}(A))\to X$ factors through a closed subspace $Z\subset X$, then $f$ factors through $f:\ob{AnSpec}(A)\to Z_{\ob{Betti}}\subset X_{\ob{Betti}}$. We can prove this claim locally after a $!$-cover of $A$ and after taking pullbacks along an epimorphism $S\to X$ with $S$ a light profinite set. We can then assume that $X=S$ is profinite in which case the claim is clear.
\end{proof}

\begin{corollary}\label{xjj93jd}
Let $X$ be a metrizable compact Hausdorff space of finite cohomological dimension and let $A$ be an analytic ring. The set $\ob{Map}(A,X_{\ob{Betti}})$ of morphisms of analytic stacks is naturally equivalent to the set of morphisms of locales
\[
F:\mathrm{Sm}(\ob{D}(A))\to X
\]
such that there is a $!$-cover $A\to B$  such that the functor $\ob{Shv}(X,\mathbb{Z})\to \ob{D}(B)$ preserves connective objects. 
\end{corollary}
\begin{proof}
By \cref{xj29sj} we must identify the space $\ob{Map}^+_{\Cat{Loc}^+}(X,\mathrm{Sm}(\ob{D}(A)))$ with those maps $F$ as above that locally on the $!$-topology of $A$ preserve connective covers. Let $F:\mathrm{Sm}(\ob{D}(A))\to X$ be a morphism of locales and let $A\to B$ be a $!$-cover such that the composite $G: \mathrm{Sm}(\ob{D}(B))\to \mathrm{Sm}(\ob{D}(A))\to X$ induces a morphism of symmetric monoidal categories that preserves connective objects. Let $S\to X$ be an epimorphism from a  profinite set, then $h:S_{\ob{Betti}}\to X_{\ob{Betti}}$ is prim and descendable by \cref{xhsw82h}. Then, \cite[Proposition 6.19]{scholze6functors} implies that $\ob{AnSpec}(B) \times_{X_{\ob{Betti}}} \ob{S}_{\ob{Betti}}\to \ob{AnSpec}(B)$ is also prim and descendable and so a $!$-cover. On the other hand, we can find $S$ such that  $h_* 1_{S}=\varinjlim_{i} C_i$ is a filtered colimit of finite products of connective idempotent algebras in $\ob{Shv}(X,\mathbb{Z})$ (consider $S$ to be the limit of the poset of finite closed covers of $X$), by \cref{xi2msnm}, given a connective idempotent algebra $\mathbb{Z}_{X}\to C'$ the pullback $B\otimes_{\mathbb{Z}_X} C'$ is connective and so has a natural animated ring structure. It follows that the fiber product $\ob{AnSpec}(B) \times_{X_{\ob{Betti}}} \ob{S}_{\ob{Betti}}$ is the analytic spectrum of the analytic ring $ C=\varinjlim_{i} B\otimes_{\mathbb{Z}_{X}} C_i$, and as $\ob{AnSpec}(C)\to \ob{AnSpec}(B)$ is a $!$-cover, the map $F$ fits in a commutative square
\[
\begin{tikzcd}
\ob{AnSpec}(C) \ar[r] \ar[d] & S_{\ob{Betti}} \ar[d] \\
\ob{AnSpec}(A) \ar[r] & X_{\ob{Betti}}
\end{tikzcd}
\]
where the left vertical arrow is a $!$-cover, so it belongs to $\ob{Map}^+_{\Cat{Loc}^+}(X,\mathrm{Sm}(\ob{D}(A)))$.  Conversely, suppose that  $F\colon \mathrm{Sm}(\ob{D}(A))\to X$ is in $\ob{Map}^+_{\Cat{Loc}^+}(X,\mathrm{Sm}(\ob{D}(A)))$, then we can find a $!$-cover $A\to B$ and a surjective map from a profinite set $S\to X$ that fits in the diagram \cref{eqdj2idj}. Since the pullback along $\ob{Shv}(X,\mathbb{Z})\to \ob{Shv}(S,\mathbb{Z})$ preserves connective objects (and is even $t$-exact), the functor $\ob{Shv}(X,\mathbb{Z})\to \ob{D}(B)$ also preserves connective objects proving what we wanted. 
\end{proof}

We apply the previous abstract nonsense in order to construct a map from the analytic spectrum of (suitable) bounded $\Q_p$-algebras to their Berkovich spectrum. We need a preliminary lemma, which will allow us to invoke the previous corollary.

\begin{definition}\label{xjw92j}
A morphism of Gelfand $\Q_p$-algebras $A\to B$ is called \textit{Berkovich \'etale} if, locally after a strict rational cover of $\mathcal{M}(B)$ (that is, a cover that can be refined by an open subcover), it factors as a composite of finite \'etale\footnote{By definition, a finite \'etale map $A\to B$ is one induced by a finite \'etale map $A(\ast)\to C$, i.e., $B=C\otimes_{A(\ast)} A$.} maps and rational localizations on the Berkovich spectrum.  A morphism $A\to B$ of Gelfand algebras  is called \textit{Berkovich pro-\'etale} if $B=\varinjlim_{i} A_i$ is a filtered colimit of Berkovich \'etale maps.  We say that it is \textit{uniformly Berkovich \'etale}, resp.\ \textit{uniformly Berkovich pro-\'etale}, if its uniform completion $A^u\to B^u$ is the uniform completion of a Berkovich \'etale, resp. Berkovich pro-\'etale map $A'\to B'$.
\end{definition}

\begin{lemma}\label{xk29sm}
  Let $A$ be a  Gelfand $\mathbb{Q}_{p}$-algebra.
  One can construct a Berkovich pro-\'etale map $A\to A^w$ such that $\mathcal{M}(A^w)$ is profinite and $\mathcal{M}(A^w)\to \mathcal{M}(A)$ is surjective.
  Moreover, if $\mathcal{M}(A)$ is metrizable, $A^w$ is such that $\mathcal{M}(A^w)$ is a light profinite set.

\end{lemma}
\begin{proof}
Consider the filtered diagram $\mathcal{J}$ consisting on finite covers of $\mathcal{M}(A)$ by Berkovich rational subspaces and let $\{A_{I}\}_{I\in \mathcal{J}}$ be the associated filtered system of bounded $\Q_p$-algebras where $A_I=\prod_{i\in I} A_i$ and $A\to A_i$ is the corresponding Berkovich rational localization.\footnote{More precisely, $\mathcal{J}$ consists of all morphisms $A\to B$ where $B=\prod_{i\in I} A_i$ for some rational localizations $A\to A_i$. This class of $A$-algebras is stable under finite colimits, and hence filtered.} We claim that $A^w=\varinjlim_{i} A_i$ does the job. First, by \cref{LemmaFilteredColimits}, we have that 
\[
\mathcal{M}(A^w)=\varprojlim_{I\in \mathcal{J}} \mathcal{M}(A_I)=\varprojlim_{I\in \mathcal{J}} \coprod_{i\in I} \mathcal{M}(A_i)
\]
so it is indeed a profinite set that surjects onto $\mathcal{M(A)}$ (this is a standard construction of a profinite cover of a compact Hausdorff space). Note that if $\mathcal{M}(A)$ is metrizable the poset $\mathcal{J}$ has a cofinal countable subdiagram, so $\mathcal{M}(A^w)$ is also light (using that the inverse limit does not change if one replaces $\mathcal{M}(A_i)$ by its profinite set of connected components, which is again light as it admits a surjection from a metrizable compact Hausdorff space). 
\end{proof}

\begin{proposition}\label{xhs82j}\label{xhs83j}
Let $A$ be a Gelfand $\mathbb{Q}_{p}$-algebra with $\mathcal{M}(A)$ a metrizable compact Hausdorff space with finite cohomological dimension. The natural map of locales $\mathrm{Sm}(\ob{D}(A))\to \mathcal{M}(A)$ of \cref{PropMapSmashing} induces a natural map of analytic stacks over $\mathbb{Q}_{p,\solid}$
\begin{equation}\label{eqsjjkri}
\ob{AnSpec}(A)\to \mathcal{M}(A)_{\ob{Betti}}\times_{\ob{AnSpec(\Z)}}{\ob{AnSpec}(\Q_{p,\solid})}
\end{equation}
satisfying universal $!$-descent.  Furthermore, there is a descendable Berkovich pro-\'etale cover $A\to B$ such that $\mathcal{M}(B)$ is a light  profinite set. 
\end{proposition}
\begin{proof}
Consider the light profinite set $S$ constructed from the poset of Berkovich rational localizations of $\mathcal{M}(A)$ as in \cref{xk29sm}. Then the map $f:S\to \mathcal{M}(A)$ is surjective, and by the proof of \cref{xjj93jd} the pullback  $B= A\otimes_{\mathbb{Z}_{\mathcal{M}(A)}} f_* 1_{S}$ is an analytic ring which is actually Berkovich pro-\'etale by the construction in the proof of \cref{xk29sm}, and that is descendable (so a $!$-cover) by \cref{xhsw82h}. It is clear that the map
\[
\ob{Shv}(\mathcal{M}(A),\mathbb{Z})\to \ob{D}(B)
\] 
preserves connective objects, so by \cref{xjj93jd} we obtain the desired map 
\[
\ob{AnSpec}(A)\to \mathcal{M}(A)_{\ob{Betti}}
\]
 of analytic stacks.
 
It is left to see that \cref{eqsjjkri} satisfies universal $!$-descent.    It suffices to check the statement after base change along a proper and descendable morphism $S_{\ob{Betti}}\to \mathcal{M}(A)_{\ob{Betti}}$ of Betti stacks with $S$ a light profinite set. As $\mathcal{M}(A)$ is a metrizable compact Hausdorff space of finite cohomological dimension, any surjection from $S$ light profinite is proper and descendable.

Because $\ob{AnSpec}(A)\times_{\mathcal{M}(A)_{\ob{Betti}}} S_{\ob{Betti}}$ is again represented by a Gelfand ring $B$ by the previous discussion, and $\mathcal{M}(B)=S$,    we can assume  without loss of generality that $\mathcal{M}(A)$ is light profinite. It suffices to show that the morphism $\ob{AnSpec}(A^u)\to \ob{AnSpec}(A)\to \mathcal{M}(A)=\mathcal{M}(A^u)$ satisfies $!$-descent. Hence, we can further assume that $A$ is a Banach $\mathbb{Q}_p$-algebra. Write $S=\varprojlim_{n} S_n$ as a countable limit of finite sets with surjective transition maps, it suffices to show that each map $C(S_n,\mathbb{Q}_p)\to A$ is $1$-descendable, then \cite[Proposition 2.7.2]{mann2022p} will show that $C(S,\mathbb{Z}) \otimes \mathbb{Q}_p \to A$ is $2$-descendable. To prove that $C(S_n, \mathbb{Q}_p)\to A$ is $1$-descendable it suffices to find a section as $C(S_n,\mathbb{Q}_p)$-modules. Now the map $C(S_n,\mathbb{Q}_p)\to A$ factors as a product of maps $\mathbb{Q}_p\to A_s$ with $s\in S_n$ corresponding to the fibers of $S\to S_n$. Then, it suffices to show that for a Banach $\mathbb{Q}_p$-algebra $A$ one has a section of $\mathbb{Q}_p\to A$ as solid $\Q_p$-modules. This is just a consequence of the Hahn--Banach theorem.
 \end{proof}

 One has the expected equivalence between \'etale sites on Gelfand ring and its uniform completion:
  
\begin{proposition}\label{x290sm}
Let $A$ be a Gelfand ring with uniform completion $A^u$. Let $X=\ob{AnSpec}(A)$ and $Y=\AnSpec(A^u)$. Then the natural maps of sites
\[
 Y_{\fet}\to X_{\fet} \mbox{ and } Y_{\mathrm{Berk}-\et}\to X_{\mathrm{Berk}-\et} \mbox{ and } Y_{\mathrm{Berk}-\proet}\to X_{\mathrm{Berk}-\proet} 
\]
are equivalences. 
\end{proposition}
\begin{proof}
  By the same argument as in \cite[Lemma 15.6]{scholze_etale_cohomology_of_diamonds} it suffices to deal with the finite \'etale case, namely, Berkovich \'etale maps are locally in the Berkovich topology compositions of rational localizations and finite \'etale maps, we  have $\mathcal{M}(A^u)=\mathcal{M}(A)$ by definition (identifying rational localizations by \cref{LemmaRationalLocalizations}) and the case of Berkovich pro-\'etale maps follows by passing to limits. For the equivalences between finite \'etale maps,  by \cite[Proposition 5.4.5.3]{gabber2003almost} it suffices to see that $A^{\leq 1}(\ast)$ is Henselian along $(p)$, which follows by \cref{xsh89j}.
\end{proof} 

\subsection{Fredholm property of Gelfand rings}\label{sss:Fredholm}

Let us recall the definition of \textit{Fredholm} analytic ring. In the following, for an analytic ring $A$, we denote by $A^{\triangleright}$ its underlying condensed ring.

\begin{definition}\label{xye73ie}
An analytic ring $A$ is said to be \textit{Fredholm} if all dualizable $A$-module are discrete, i.e.\ if they arise from perfect $A^{\triangleright}(*)$-modules via the fully-faithful embedding
\[
\ob{D}^{\delta}(A^{\triangleright}(*)) \hookrightarrow \ob{D}(A)
\]
from the derived $\infty$-category $\ob{D}^{\delta}(A^{\triangleright}(*))$ of classical (i.e. non-condensed) $A^{\triangleright}(*)$-modules.  
\end{definition}

In this section, we allow non-induced analytic ring structures on bounded rings in the sense of \cite[Definition 2.6.10]{camargo2024analytic}, i.e., the analytic ring structure is assumed to be obtained by making elements in the underlying ring solid. To have a clear terminology,  we call these analytic rings over $\Q_{p,\solid}$ bounded affinoid $\Q_p$-algebras (as in \cite[Definition 2.6.10]{camargo2024analytic}), reserving the notion ``bounded $\Q_p$-algebras'' for bounded affinoid rings over $\Q_{p,\solid}$ whose analytic ring structure is induced from $\Q_{p,\solid}$.

\begin{lemma}\label{sec:fredh-prop-gelf-nak}
Let $A$ be a bounded affinoid $\Q_p$-algebra and let $K\in D(A)$ be static. Assume that $K$ is a finitely presented solid $\pi_0(A)$-module. If $\pi_0(K\otimes_{A}A^{\dagger-\ob{red}})=0$, then $K=0$. 
\end{lemma}
\begin{proof}
We can assume that $A$ is static and work within the abelian category of $A$-modules. The statement of the lemma is equivalent to the following: let $f: A[S']\to A[S]$ be a map of $A$-modules (for $S',S$ light profinite sets) whose base change $\overline{f}$ to $A^{\dagger-\ob{red}}$ is an epimorphism, then $f$ is an epimorphism. By replacing $S'$ and $S$ by larger profinite sets if necessary we can assume that $S'=S$, so that $f$ is an endomorphism of $A[S]$ (e.g., $S=S'$ can be taken to be the Cantor set). Let
\[
\ob{Nil}^{\dagger}(A)[S]=\ob{fib}(A[S]\to A^{\dagger-\ob{red}}[S])=\ob{Nil}^{\dagger}(A)\otimes_{A} A[S]
\]
and suppose that $\overline{f}$ is surjective.  Then, by projectivity  $\overline{f}$ has a section $\overline{g}: A^{\dagger-\ob{red}}[S]\to A^{\dagger-\ob{red}}[S]$ which can be lifted to a map $g:A[S]\to A[S]$ such that $\id_{A[S]}-fg$ factors through $A[S]\to \ob{Nil}^{\dagger}(A)[S]$. Thus, we are left to prove the following claim: 

\begin{claim*}
Let $A$ be a bounded affinoid $\Q_p$-algebra, $S$  a light profinite set and $h:A[S]\to \ob{Nil}^{\dagger}(A)[S]$ an $A$-linear map, then $\ob{id}_{A[S]}-h:A[S]\to A[S]$ is an isomorphism. 
\end{claim*}

We will reduce the proof of this claim to a concrete case where we can do an explicit computation. The analytic ring structure of $A$ is given by the set of solid elements $A^+$ consisting on those $a\in A^{\triangleright}$ such that the map $\Z[T]\to A$ depicted by $a$ extends to $\Z[T]_{\solid}\to A$. Writing $A^+$ as a union of finitely generated subrings $A_i^+$, we can write $A=\varinjlim_i A^{\triangleright}_{A_{i,\solid}^+/}$, and as $A[S]$ is compact,  we can  assume without loss of generality that there is a map $\Z[T_1,\ldots, T_n]_{\solid}\to A$ such that $A$ has the induced analytic ring structure.

To continue with the computation, we need to understand how bounded rings can be described from sifted colimits of easier rings. For this, consider $\Cat{Pair}_{\Z[T_1,\ldots, T_n]_{\solid}}\subset \ob{Fun}(\Delta^1,\ob{Ring}_{\Z[T_1,\ldots, T_n]_{\solid}})$ the full subcategory of maps of animated  solid $\Z[T_1,\ldots, T_n]_{\solid}$-algebras consisting on those maps $A\to B$ are surjections on connected components. The category $\Cat{Pair}_{\Z[T_1,\ldots, T_n]_{\solid}}$ has compact projective generators given by base changes along $\Z_\solid\to \Z[T_1,\ldots, T_n]_{\solid}$ of finite coproducts of pairs of the form $\mathbb{Z}_{\solid}[\mathbb{N}[S]]\xrightarrow{\ob{id}} \mathbb{Z}_{\solid}[\mathbb{N}[S]]$ (corepresenting the underlying ring of $A$), and $\mathbb{Z}_{\solid}[\mathbb{N}[S]]\to \mathbb{Z}$ (correpresenting the kernel $I=\ob{fib}(A\to B)$).

Now, let $A$ be a bounded affinoid $\Q_p$-algebra with induced analytic ring structure from a map $\Z[T_1,\ldots, T_n]_{\solid}\to A$,  and consider the pair $C\to D$ given by  $A^{\circ}\to A^\circ/\ob{Nil}^{\dagger}(A)$ (with $A^\circ$ being defined as in \cite[Definition 2.6.1]{camargo2024analytic}). Then any map from $\Z[T_1,\ldots, T_n]_{\solid}[\mathbb{N}[S]]\xrightarrow{\id} \Z[T_1,\ldots, T_n]_{\solid}[\mathbb{N}[S]]$  to $C\to D$ will factor through the idempotent pair $\mathbb{Z}_p \langle T_1,\ldots, T_n\rangle_{\solid}\langle  \mathbb{N}[S] \rangle\to \mathbb{Z}_p \langle T_1,\ldots, T_n\rangle_{\solid}\langle  \mathbb{N}[S] \rangle$ (note that colimits in $\ob{Fun}(\Delta^1,\ob{Ring}_{\Z[T_1,\ldots, T_n]_{\solid}})$ are calculated pointwise). Similarly, any map $\Z[T_1,\ldots, T_n]_{\solid}[\mathbb{N}[S]] \to \Z[T_1,\ldots, T_n]$  towards $(C[1/p]\to D[1/p])=(A\to A^{\dagger-\mathrm{red}})$ will factor through the map $\mathbb{Q}_p \langle T_1,\ldots, T_n\rangle_{\solid}\langle \mathbb{N}[S] \rangle_{\leq 0}\to \mathbb{Q}_p \langle T_1,\ldots, T_n\rangle_{\solid}$. Putting all together, one deduces that the pair $A\to A^{\dagger-\ob{red}}$ can be written as a sifted colimit of base changes to $\Q_{p,\solid}$  of finite coproducts in $\Cat{Pair}_{\Z[T_1,\ldots, T_n]_{\solid}}$ of the form

\begin{itemize}
\item[(a)] $\mathbb{Q}_p \langle T_1,\ldots, T_n\rangle_{\solid}\langle  \mathbb{N}[S] \rangle\to \mathbb{Q}_p \langle T_1,\ldots, T_n\rangle_{\solid}\langle  \mathbb{N}[S] \rangle$  and

\item[(b)] $\mathbb{Q}_p \langle T_1,\ldots, T_n\rangle_{\solid}\langle \mathbb{N}[S] \rangle_{\leq 0}\to \mathbb{Q}_p \langle T_1,\ldots, T_n\rangle_{\solid}$. 
\end{itemize}

Now, let $\{A_i\to B_i\}_i$ be a sifted diagram in $\Cat{Pair}_{\Z[T_1,\ldots, T_n]_{\solid}}$ whose colimit is $A\to A^{\dagger-\ob{red}}$.  Let $h:A[S]\to \ob{Nil}^{\dagger}(A)[S]$ be an $A$-linear map corresponding to a map of condensed sets $S\to \ob{Nil}^{\dagger}(A)[S]$. Let $I_i=\ob{fib}(A_i\to B_i)$, then $\ob{Nil}^{\dagger}(A)[S]=\varinjlim_i I_i[S]$ (for $I_i[S]=I_i\otimes_{A_i} A_i[S]$).  Since we are working with solid abelian groups and the diagram is sifted, there are finitely many $i_1,\ldots, i_k$ such that $h$ factors through $S\to I_{i_1,\ldots, i_k}[S]$ where $I_{i_1,\ldots, i_k}$ is the fiber of 
\[
A_{i_1}\otimes \cdots \otimes  A_{i_k}\to B_{i_1}\otimes \cdots  \otimes B_{i_k}
\]
with tensor product over $\Z[T_1,\ldots, T_n]_{\solid}$.     By the previous discussion we can take  pairs $A_i\to B_i$  which are finite coproducts  (over $\Z[T_1,\ldots, T_n]_{\solid}$ and so over $\mathbb{Q}_p \langle T_1,\ldots, T_n\rangle_{\solid}$ by idempotency) of the pairs of the form (a) and (b) above, in that situation we even have that $I_i=\ob{Nil}^{\dagger}(A_i)$. Hence, we have reduced the previous claim to the following one:

\begin{claim*}

Consider the ring $A=\mathbb{Q}_{p}\langle T_1,\ldots,T_n \rangle_{\solid}\langle \mathbb{N}[X] \rangle\langle \mathbb{N}[Y]  \rangle_{\leq 0}$ for $\{T_1,\ldots,  T_n\}$ a finite set of variables and $X$ and $Y$ light profinite sets. Let $S$ be a light profinite set and let $h:A[S]\to \ob{Nil}^{\dagger}(A)[S]$ be an $A$-linear map, then $1-h:A[S]\to A[S]$ is an isomorphism. 
\end{claim*}

We finally give the proof of this claim. Note that 
\[
\ob{Nil}^{\dagger}(A)=\ob{ker}( \mathbb{Q}_{p}\langle T_1,\ldots,T_n \rangle_{\solid}\langle \mathbb{N}[X] \rangle\langle \mathbb{N}[Y]  \rangle_{\leq 0}\to \mathbb{Q}_{p}\langle T_1,\ldots,T_n \rangle_{\solid}\langle \mathbb{N}[X] \rangle).
\]
 We can write $A=\varinjlim_{r>0} A_r$ with $A_r=\mathbb{Q}_{p}\langle T_1,\ldots,T_n \rangle_{\solid}\langle \mathbb{N}[X] \rangle\langle \mathbb{N}[Y]  \rangle_{\leq r}$ and set 
 \[
 I_r=\ker( A_r\to \mathbb{Q}_{p}\langle T_1,\ldots,T_n \rangle_{\solid}\langle \mathbb{N}[X] \rangle).
 \] 
 Then $\ob{Nil}^{\dagger}(A)=\varinjlim_{r>0} I_r$. Thus, there is some $r$ such that the map $h:S\to \ob{Nil}^{\dagger}(A)[S]$ factors through $I_r[S]$. Note that rescaling $Y$ by  $p$ does nothing to $A$ but multiplies the $r$'s in the indexes by the norm of $p$, so by taking $r$ close enough to $0$ and modifying $Y$ by some power of $p$ we can even assume that $h$ factors through $I^{<1}[S]$ with
 \[
 I^{<1}_1:= \ob{ker}(\mathbb{Z}_{p}\langle T_1,\ldots,T_n \rangle_{\solid}\langle \mathbb{N}[X] \rangle\llbracket\mathbb{N}[pY]\rrbracket\to  \mathbb{Z}_{p}\langle T_1,\ldots,T_n \rangle_{\solid}\langle \mathbb{N}[X] \rangle).
 \]
 But then the composite $h:S\to I^{<1}_1[S] \to \mathbb{Z}_{p}\langle T_1,\ldots,T_n \rangle_{\solid}\langle \mathbb{N}[X] \rangle\langle \mathbb{N}[Y] \rangle[S] =:C[S]$ is $0$ modulo $p$ and so the map $1-h:C[S]\to C[S]$ is an isomorphism, since the map $1-h:A[S]\to A[S]$ is just a base change of this  we deduce that it is an isomorphism as well proving what we wanted. 
\end{proof}

Next, we improve \cref{sec:fredh-prop-gelf-nak} from the $\dagger$-reduction to the uniform completion. 

\begin{lemma}\label{xj29ejs}
Let $A$ be a bounded affinoid $\Q_p$-algebra and let $K\in \ob{D}(A)$ be static. Assume that $K$ is a finitely presented solid $\pi_0(A)$-module. If $\pi_0(K\otimes_{A}A^u)=0$, then $K=0$. 
\end{lemma}
\begin{proof}
 By \cref{sec:fredh-prop-gelf-nak} we can assume that $A$ is $\dagger$-reduced, so it is static and $A^{\triangleright}\hookrightarrow A^u$ is injective as $\bigcap_{n} p^nA^{\leq 1}=\Nil^\dagger(A)=\{0\}$. By an approximation argument, and writing $A=\varinjlim_{B\subset A^+}(A^{\triangleright},B)_{\solid}$ as a filtered colimit where $B$ runs over finitely generated subrings of $A^+$, we can assume that $A$ is of the form $A=A^{\triangleright}_{\mathbb{Z}[s_1,\ldots, s_n]_{\solid}/}$ for certain elements $s_i\in A^+$. Set $R=\mathbb{Z}[s_1,\ldots, s_n]_{\solid}$, by flatness of $R[S]_{\solid}$ as solid $R$-module (for $S$ a light profinite set), we know that the map $A[S]=A\otimes_R R[S]\hookrightarrow A^u[S]=A^u\otimes_R R[S]$ is injective. Since $R$ is finitely generated we also know by \cite[Proposition 2.12.10]{mann2022p} that $A^{u,\leq 1}[S]$ is $p$-adically complete and so equal to the $p$-adic completion of $A^{\leq 1}[S]$. 

Similarly as in the proof of \cref{sec:fredh-prop-gelf-nak}, it suffices to show the following claim:

\begin{claim*}
Let $S$ be a light profinite set and $f:A[S]\to A[S]$ such that the base change $f^u$ to $A^u[S]$ is surjective, then $f$ is surjective. 
\end{claim*}

Let $S$ and $f:A[S]\to A[S]$ be as before. Since $A^u[S]$ is a projective $A^u$-module we have a section $g^u:A^u[S]\to A^u[S]$ of $f^u$. So that $f^ug^u=\ob{id}_{A^u[S]}$. Since $A^u[S]=\varprojlim_{r} A[S]/A^{\leq r}[S]$, for an arbitrary $r$ we can find $g_r:A[S]\to A[S]$ such that $g^u-g_{r}$ sends $S$ to $A^{u,\leq r}[S]$. In particular, there is some $r$ such that  $1-f^ug^u_r$ sends $S$ to $A^{u,<1}[S]$ and so $f^u g_r^u:A^{u}[S]\to  A^u[S]$ is invertible. Therefore,  we can construct a map $g: A[S]\to A[S]$ such that the base change of $fg:A[S]\to A[S]$ to $A^u$ is an isomorphism, even better, such that $1-fg$ maps $S$ to $A^{\leq r}[S]$ for a fixed but arbitrary $0<r<1$. Now, as $A$ is bounded, we can write $A^{\leq 1}$ as a filtered colimit of $p$-complete rings $A_i^{\circ}$, and so, by taking $r=|p|$, there is some $i$ such that the map $\frac{1}{p}(1-fg):S\to
A^{\leq 1}[S]$ factors through a map $h:S\to A_i^{\circ}[S]$. In particular, $1-ph:A_i^{\circ}[S]\to A_i^{\circ}[S]$ is an isomorphism by Nakayama's lemma, and since $fg:A[S]\to A[S]$ is just the base change of $1-ph$ we get that $fg$ is an isomorphism proving that $f$ is surjective and finishing the proof of the lemma. 
\end{proof}

\begin{remark}
In the notation of \cref{xj29ejs},  let $K^u=\pi_0(A^u\otimes_A K)$. It is natural to expect that the map $K\to K^u$ is injective as $K$ is finitely presented, but we do not know how to prove it. 
\end{remark}

\begin{proposition}\label{xh28sj}
Let $A$ be a bounded affinoid $\Q_p$-algebra. If $A^u$ is Fredholm, then $A$ is Fredholm. In particular, Gelfand rings are Fredholm.
\end{proposition}
\begin{proof}
 The proof takes inspiration from \cite[Proposition VI.1.7]{ScholzeRLL}. Thanks to \cite[Proposition VI.1.5]{ScholzeRLL}, it suffices to show that if $K$ is a compact complex such that $K\otimes_{A} A^u$ is perfect then $K$ is perfect. We can assume that $K$ is connective, and by an induction argument it suffices to show that there is a map $f: A^m\to K$ such that the cofiber $\ob{cofib}(f)$ is $1$-connective. By \cref{sec:fredh-prop-gelf-nak} and \cref{xj29ejs} it suffices to find a map $f$ such that the base change  $f^u: A^{u,m}\to K\otimes_{A} A^u$ is surjective on $\pi_0$. Thus, by writing $K$ as a retract of a finite complex  with terms $A[S]$ for $S$  a profinite set, we can assume that $K$ is  a retract of a compact object $C^{\leq 1} [\frac{1}{p}]$ with $C^{\leq 1}$ a complex of  $A^{\leq 1}$-modules, and even that the map $C^{\leq 1}[\frac{1}{p}]\to K$ is an isomorphism on $\pi_0$. By approximating along subrings $B\subset A^+$ of finite type, we can assume that $A=A_{B/}$. Then $C\otimes_A A^u$ is the generic fiber of the $p$-adic completion of $C^{\leq 1}$. Thus, given a map $h^u:A^{u,\leq 1, n}\to C^{\leq 1}$ we can always  find a map $g:A^{\leq 1,n}\to C^{\leq 1}$   such that $h^u-g^{u}$ is $0$ modulo $pA^{u, \leq 1}$. Applying this to a surjection $A^{u,m}\to K\otimes_{A} A^u\to C^{\leq 1}\otimes_{A^{\leq 1,u}} A^u$ (after rescaling for it to factors through $C^{\leq 1}$), we get the desired map $f: A^{m}\to K$.

The final assertion for Gelfand rings follows from the previous assertion, \Cref{xnbsyw}, and \cite[Theorem  5.50]{andreychev2021pseudocoherent}.
\end{proof}

\newpage
\section{Perfectoidization and analytic de Rham stacks}
\label{sec:perf-analyt-de}

In this section, we discuss (a version of) perfectoidization, and then we construct analytic de Rham stacks through the right adjoint of perfectoidization (as outlined in \cref{sec:new-defin-analyt}). Let us explain this more precisely. For such a picture to make sense while giving the ``expected'' analytic de Rham stack, we need, as already mentioned in the introduction, to first replace 
$$\Cat{AnStk}_{\Q_{p,\solid},}$$
the category of \emph{analytic stacks}  over $\Q_{p,\solid}$, in the sense of Clausen--Scholze (\cref{xhs92mk}), by
$$\Cat{GelfStk},$$
the category of \emph{Gelfand stacks}, that we introduce in \cref{xhs9wj} by replacing analytic rings over $\Q_{p,\solid}$ with separable Gelfand rings (\cref{xhs9wj}). The relation between these two categories of geometric objects is encoded by an adjunction\footnote{Upper morphisms indicate the right adjoints, and the lower the left adjoints.}:

\[\begin{tikzcd}
	{\Cat{AnStk}_{\Q_{p,\solid}}} & {\Cat{GelfStk}}
		\arrow["{(-)^\Gelf}", from=1-1, to=1-2]
	\arrow["{(-)_{\ob{An}}}", shift left=3, from=1-2, to=1-1]
	\end{tikzcd}\]
Namely, $(-)^\Gelf$ maps an analytic stack $X$ to the functor it represents on separable Gelfand rings, and its left adjoint sends $\ob{GSpec}(A)$ for a separable Gelfand ring $A$ to $\AnSpec(A)$ (\cref{sec:totally-disc-stacks-morphism-of-topoi}). The separability assumption on the rings is there to make several descendability arguments work.

The target of the perfectoidization functor should now be the category of arc-stacks over $\Q_p$. For the same technical reasons as with Gelfand rings, it is useful to impose separability; this leads us to the introduction of $\Cat{ArcStk}_{\Q_p}$, the category of \emph{light arc-stacks} over $\Q_p$ (\cref{sec:light-arc-stacks-2-light-arc-stacks-over-q-p}). The perfectoidization functor $(-)^\diamond$ sends $\ob{GSpec}(A)$ for a separable Gelfand ring $A$, to the functor it represents on separable totally disconnected perfectoid rings. One would like the analytic de Rham stack functor to be its (fully faithful) right adjoint, so that we have an adjunction (in fact, a morphism of $\infty$-topoi):
\[\begin{tikzcd}
	\Cat{ArcStk}_{\Q_p} & {\Cat{GelfStk}}.
		\arrow["  (-)^{\dRall} ", from=1-1, to=1-2]
	\arrow["{(-)^{\diamond}}", shift left=3, from=1-2, to=1-1]
	\end{tikzcd}\]
 This right adjoint does exist, but will not be our  definition of the de Rham stack. Instead, we will call it the \textit{big de Rham stack}, as it is defined on the \textit{big} category of all separable perfectoid rings, in contrast to what we will call simply the \textit{de Rham stack}.
 
 \begin{remark}
  Before saying why we do not define the de Rham stack functor as this right adjoint, let us point out that we show in \cref{sec:perf-analyt-de-2-left-adjoint-to-perfectoidization} that $(-)^\diamond$ is itself a \emph{right} adjoint of a functor $\widehat{(-)}\colon \Cat{ArcStk}_{\Q_p}\to \Cat{GelfStk}$ that realizes light arc-stacks (over $\Q_p$) geometrically. The functor $\widehat{(-)}$ is the left Kan extension of the functor that sends the arc-stack $\mathcal{M}_{\ob{arc}}(A)$ represented by a totally disconnected perfectoid ring $A$ to $\ob{GSpec} A$. We thus have another adjunction pair:
    \[\begin{tikzcd}
	  {\Cat{GelfStk}} & \Cat{ArcStk}_{\Q_p}.
		\arrow["{(-)^{\diamond}}", from=1-1, to=1-2]
	\arrow["{\widehat{(-)}}", shift left=3, from=1-2, to=1-1]
	\end{tikzcd}\]
	Via the functor $\widehat{(-)}$ and the $6$-functor formalism on $\Cat{GelfStk}$ (inherited from $\Cat{AnStk}_{\Q_p}$), one obtains a $6$-functor formalism for $\widehat{\mathcal{O}}$-cohomology on arc-stacks over $\Q_p$; this justifies the notation. See the end of \Cref{sec:arc-and-!-topologies-on-perfectoid-rings} for more on the relation to \cite{AMdescendPerfd} and \cite{ALBMFFCoho}. 
	\end{remark}

Now, let us briefly discuss why the de Rham stack functor defined in this way does not have all the desired properties and how we change its definition. One important aspect for the applications of the theory both in this paper (\cref{sec:p-adic-monodromy}) and in future work  is that the functor sending $X$ to its analytic de Rham stack should have strong descent properties: if $X \to X^\prime$ is an epimorphism of (light) arc-stacks over $\Q_p$, we want the induced map $X^{\dR} \to X^{\prime,\dR}$ to be an epimorphism of Gelfand stacks. It is however false in this general setting, indeed, if it were the case then $X^{\dRall}=X_{\Betti}$ for $X$ a compact Hausdorff space and this is not true, see \cref{ExampleCondensedAnimanodeRham}.  As will be explained in more detail below (see the beginning of \Cref{sec:the-analytic-dR-stack}), the situation would simplify considerably if separable Gelfand rings whose uniform completion is strictly totally disconnected perfectoid would form a basis of the $!$-topology on separable Gelfand rings. But it is unfortunately not true that all separable Gelfand rings $R$ admit a descendable map to such a ring, see \cref{sec:-covers-perfectoids-1-counter-example-to-cover-by-strictly-totally-disconnected} for a more precise counterexample. Nevertheless, we do manage to prove such a statement under some sort of finiteness assumption on $R$: we show it when $R$ is \textit{quasi-finite dimensional} (\textit{qfd} for short), i.e.\ its uniform completion is quasi-pro-\'etale over a finite-dimensional analytic affine space over $\Q_p$. For this reason, we replace in the above definitions the categories $\Cat{ArcStk}_{\Q_p}$, $\Cat{GelfStk}$ by their versions
 $$
 \Cat{ArcStk}_{\Q_p}^{\rm qfd}, ~ \Cat{GelfStk}^{\rm qfd}
 $$
 defined on qfd separable perfectoid rings, resp. qfd separable Gelfand rings,  instead of all separable perfectoid rings, resp. all separable Gelfand rings. 
 
 \begin{remark}
 To keep the notation light, the analogues of the functors $\widehat{(-)}$ and $(-)^\diamond$ for the categories $ \Cat{ArcStk}_{\Q_p}^{\rm qfd}, \Cat{GelfStk}^{\rm qfd}$ will still be denoted in the same way. We hope that this does not create confusion. Note that one has natural functors
 $$
 \Cat{ArcStk}_{\Q_p}^{\rm qfd} \to  \Cat{ArcStk}_{\Q_p}, ~ \Cat{GelfStk}^{\rm qfd} \to \Cat{GelfStk},
 $$
  defined by left Kan extensions. Since the functors $(-)^\diamond$ and $(-)^{\mathrm{Gelf}-\diamond}$ and their qfd variants are left adjoints, they commute with colimits, and so it does not matter whether one left Kan extends first and then applies them, or the other way around; this justifies this abuse of notation. The same wouldn't be true for the right adjoints of $(-)^\diamond$ and its qfd version, and for this reason we refrain from calling $(-)^{\dR}$ the right adjoint of $(-)^\diamond$ in all Gelfand stacks and reserve this name to the right adjoint of its qfd version.  \end{remark}
  
  \begin{remark}
 In the theory of the \textit{algebraic} de Rham stack\footnote{For us, the algebraic de Rham stack of $X$ is the sheafification for the descendable topology of the functor $A\mapsto X(A^{\red})$ where $A^{\red}$ is the reduction of $A$. This agrees with Simpson's de Rham stack in characteristic zero, but, in characteristic $p$ it differs from the de Rham stack appearing in the works of Bhatt--Lurie and Drinfeld.}, a similar ``finite dimensionality'' assumption is required to obtain descent of the formation of the de Rham stack for the $h$-topology\footnote{Recall that a morphism of Noetherian schemes $Y\to X$ of finite type is an $h$-cover if and only if it is descendable, see \cite[Theorem 11.26]{bhatt_scholze_projectivity_of_the_witt_vector_affine_grassmannian}}. In this case one reduces to showing that the morphism from the scheme to its algebraic de Rham stack is an epimorphism for the descendable topology, and this is where finite type assumptions are used. In general, the sheafification of $X$ for the $h$-topology is its absolute weak normalization $X^{\mathrm{awn}}$, see \cite[Tag 0EVS]{stacks-project}. If $X$ is a scheme of finite type over a field $K$ of characteristic $0$, then $X^{\mathrm{awn}}$ is a scheme of finite type over $K$ and $X^{\mathrm{awn}}\to X$  is the initial object in the category of universal homeomorphisms over $X$. If instead $X$ is of characteristic $p$, then $X^{\mathrm{awn}}$ is its perfection.
 \end{remark}

\subsection{Light arc-stacks}
\label{sec:light-arc-stacks}

We extend the definition of light arc-stacks from \cref{sec:introduction} to allow perfectoid Tate rings in characteristic $p$ (this will be used when constructing the Hyodo--Kato stacks in \cref{sec:de-rham-fargues}). We also introduce the quasi-finite dimensional variant that will be used later.

We first recall the arc-topology from \cite{scholze2024berkovichmotives}.

\begin{definition}
  \label{sec:light-arc-stacks-1-arc-cover}
  A map $A\to B$ of perfectoid Tate rings is an arc-cover if the morphism $\mathcal{M}(B)\to \mathcal{M}(A)$ is surjective.
\end{definition}

We let $\Cat{AffPerfd}$ be the opposite category of the category of perfectoid Tate rings (we stress that we do not consider perfectoid Tate-Huber pairs, i.e., we don't require the choice of a ring of integral elements).

\begin{remark}
  \label{sec:light-arc-stacks-2-remark-arc-topology}\
  \begin{enumerate}
  \item By \cite[Theorem 3.14]{scholze2024berkovichmotives} the arc-topology on the category $\Cat{AffPerfd}$ is subcanonical.
  \item A perfectoid Tate ring $A$ is called \textit{totally disconnected} if $\mathcal{M}(A)$ is a profinite set. If moreover each residue field $k(x)$ for $x\in \mathcal{M}(A)$ is algebraically closed, then $A$ is called \textit{strictly totally disconnected} (\cite[Definition 3.10]{scholze2024berkovichmotives}).
    \item Each perfectoid Tate algebra $A$ admits an arc-cover $A\to B$ with $B$ strictly totally disconnected (\cite[Proposition 3.11]{scholze2024berkovichmotives}).
  \end{enumerate}
\end{remark}

For technical reasons, we cannot work with arbitrary perfectoid Tate algebras.

\begin{definition}
  \label{sec:light-arc-stacks-1-separable-perfectoid-tate-algebra}
  A  Tate ring $A$ is called \emph{separable} if $A^\circ/\pi$ is countable for a (equiv. any) pseudo-uniformizer $\pi$.
\end{definition}

Clearly, a perfectoid Tate algebra $A$ is separable if and only if its tilt $A^\flat$ is separable. Another equivalent characterization of separable perfectoid rings is as $\omega_1$-compact perfectoid rings (\cref{sec:light-arc-stacks-3-light-vs-all-arc-sheaves}).  We let
\[
  \Cat{AffPerfd}_{\omega_1}\subseteq \Cat{AffPerfd}
\]
be the full subcategory of separable perfectoid Tate algebras. This restriction is very mild and mostly used through \cref{sec:perf-analyt-de-2-arc-covers-are-good} (or related assertions needing that an arc-cover is a $!$-cover).

\begin{lemma}
  \label{sec:light-arc-stacks-1-cover-by-separable-totally-disconnected}
  Let $A\in \Cat{AffPerfd}_{\omega_1}$ (or more generally $A$ is a separable Tate ring).
  \begin{enumerate}
  \item $\mathcal{M}(A)$ is a metrizable compact Hausdorff space.
  \item There exists an arc-cover $A\to B$ with $B$ separable strictly totally disconnected. Furthermore, if $A$ is totally disconnected we can take the cover to be pro-finite \'etale. 
  \end{enumerate}
\end{lemma}
\begin{proof}
  Let $f_i$, $i\in I$ be a countable dense subset of $A^\circ$ (e.g., products of a countable set of representatives of $A^\circ/\varpi$ for a pseudo-uniformizer $\varpi\in A$). Then $\mathcal{M}(A)\to \prod_{i\in I}[0,1],\ x\mapsto (|f_i(x))|)_{i\in I}$ is a closed embedding, and hence $\mathcal{M}(A)$ is metrizable.
  Given that $\mathcal{M}(A)$ is metrizable, the argument of \cref{xk29sm} provides an arc-cover $A\to B'$ with $\mathcal{M}(B')$ light profinite and $B'$ separable. We can replace $A$ by $B'$. Let $x\in \mathcal{M}(A)$. Then each finite \'etale extension of $k(x)$ can be extended to a rational neighborhood of $x\in \mathcal{M}(A)$, e.g., by \cite[Lemma 7.4.6.]{scholze2020berkeley}. As $\mathcal{M}(A)$ is light profinite, we can further extend to $A$. By metrizability of $\mathcal{M}(A)$, we can observe that there exists a countable, filtered system $A\to B_i$ of finite \'etale $A$-algebras, such that every $A$-algebra $C$, which is finite \'etale over a rational localization of $A$, admits a morphism to some $B_i$ over $A$. Let $A\to C$ be a morphism with $C$ strictly totally disconnected, e.g., $C$ could be a Banach product of completed algebraic closures of the residue fields of $A$ as in \cite[Proposition 3.11]{scholze2024berkovichmotives}. Then we can construct $B$ as the filtered colimit over the countable category of factorizations $A\to B_i\to C, i\in I$.
\end{proof}

\begin{definition}
  \label{defi-sec:light-arc-stacks-1-light-arc-stacks}
  We let
  \[
    \Cat{ArcStk}:=\widehat{\Shv}(\Cat{AffPerfd}_{\omega_1}^\op)
  \]
  be the category of hypersheaves on the arc-site of separable perfectoid spaces. Given a Tate ring $A$, we let $\mathcal{M}_{\ob{arc}}(A)$ be the light arc-stack given by arc-sheafification of the presheaf $B\mapsto \Hom(A, B)$.
\end{definition}

We let $\Cat{ArcStk}^{\mathrm{big}}$ be the category of arc-hypersheaves on $\Cat{AffPerfd}^\op$. Clearly, we have a morphism of topoi
\[
  \nu\colon \Cat{ArcStk}^{\mathrm{big}}\to \Cat{ArcStk}.
\]

The two topoi are concretely related as follows.

\begin{lemma}
  \label{sec:light-arc-stacks-3-light-vs-all-arc-sheaves}\
  \begin{enumerate}
  \item Each perfectoid Tate algebra is an $\omega_1$-filtered colimit of separable perfectoid Tate algebras. In fact, $\Cat{AffPerfd}\cong\mathrm{Ind}_{\omega_1}(\Cat{AffPerfd}_{\omega_1})$.
    \item The pullback $\nu^{-1}\colon \Cat{ArcStk}\to \Cat{ArcStk}^{\mathrm{big}}$ is fully faithful and commutes with countable limits. The essential image of $\nu^{-1}$ consists of those sheaves $\mathcal{F}$ which send $\omega_1$-filtered colimits of  totally disconnected perfectoid Tate algebras in $\Cat{AffPerfd}$ to filtered colimits in $\ob{Ani}$.
\end{enumerate}
\end{lemma}
The lemma is equally true for the slice topoi over some $X\in \Cat{ArcStk}$.
\begin{proof}
  Assume $A\in \Cat{AffPerfd}$, and let $\varpi^\flat\in A^\flat$ be a pseudo-uniformizer, we can consider a subring $B$ containing  $\Z_p[((\varpi^\flat)^{1/p^n})^\sharp\ |\ n\in \N]$, such that $p|\varpi:=((\varpi)^\flat)^\sharp$ and such that $B/\varpi$ is countable. Now, $A^\circ$ is an $\omega_1$-filtered colimit of countably generated $B$-subalgebras $C$, and among these the ones with Frobenius surjective on $C/\varpi$ are cofinal. Thus, by $\varpi$-completing these $C$'s and inverting $\varpi$ we can write $A^\circ$ as an $\omega_1$-filtered colimit of separable perfectoid Tate algebras. To show that $\Cat{AffPerfd}\cong\mathrm{Ind}_{\omega_1}(\ob{AffPerfd}_{\omega_1})$ it suffices to see that separable perfectoid Tate algebras are $\omega_1$-compact in $\ob{AffPerfd}$. This in turn is implied by countability modulo a pseudo-uniformizer.

  To see the second assertion, we first show that each arc-cover of  totally disconnected perfectoid Tate rings is an $\omega_1$-filtered colimit of arc-covers of separable  totally disconnected perfectoid Tate algebras. Indeed, any morphism of  totally disconnected perfectoid Tate rings is an $\omega_1$-filtered colimit of morphisms of separable  totally disconnected perfectoid Tate algebras, by the previous assertion. If the morphism one starts with is an arc-cover, one can ensure that the morphisms in the colimit are also arc-covers, by the following observation. If $C\to A$ is a morphism with $C\in \Cat{AffPerfd}_{\omega_1}$ and $C$ totally disconnected, then $\mathcal{M}(A)\to \mathcal{M}(C)$ has a closed image, which is a countable intersection of finite unions of rational localizations (as $\mathcal{M}(C)$ is a metrizable compact Hausdorff space by \cref{sec:light-arc-stacks-1-cover-by-separable-totally-disconnected}). As $C$ is totally disconnected, i.e., $\mathcal{M}(C)$ is profinite, we may even assume that these finite unions are again affinoid and totally disconnected. 

  We can conclude that each arc-cover of  totally disconnected perfectoid Tate rings in $\Cat{AffPerfd}$ is an $\omega_1$-filtered colimit of arc-covers of  totally disconnected perfectoid Tate rings in $\Cat{AffPerfd}_{\omega_1}$. From here, we can conclude that $\nu^{-1}$ is fully faithful and commutes with countable limits (namely, the presheaf pullback to  totally disconnected perfectoid Tate rings of some object $X\in \Cat{ArcStk}$ is already an hypersheaf). The description of the essential image follows by the definition of the presheaf pullback.
\end{proof}

\begin{remark}\label{RemarkRelevantarcStacks}
In particular, given a light arc-stack $X$ the finitary arc-sheaves over $\nu^{-1}X$ considered in \cite[Section 4]{scholze2024berkovichmotives} are pulled back from their restriction to separable perfectoid Tate algebras. Furthermore, as the closed disc $\overline{\mathbb{D}}$ and the Tate twist $\Z(1)$ are $\omega_1$-compact, the category of motivic sheaves on $\nu^{-1}X$ are full subcategories of $\ob{D}(X_{\ob{arc,\omega_1}}, \Z)$ where $X_{\ob{arc,\omega_1}}$ is the light arc-site of $X$. We note that light arc-stacks over $\F_p$ allow the necessary geometry appearing in the geometrization program of the local Langlands correspondence \cite{fargues2021geometrization}, \cite{scholze2025geometrization}. For example, we can speak about $\mathrm{Spd} \Q_p, \mathrm{Bun}_G,\mathrm{Div}^1_{\overline{\F}_p}, \ldots \in \Cat{ArcStk}_{\F_p}$.
\end{remark}

\begin{definition}
  \label{sec:light-arc-stacks-2-light-arc-stacks-over-q-p}
  We let $\Cat{ArcStk}_{\Q_p}$ be the category of light arc-stacks over $\mathcal{M}_{\ob{arc}}(\Q_p)$, or equivalently, $\Cat{ArcStk}_{\Q_p}$ is the category of hypersheaves on the arc-site of separable perfectoid spaces over $\Q_p$.
\end{definition}

From now on, \textit{all arc-stacks will implicitly assumed to be light}.
\\

As examples of light arc-stacks, we discuss those associated with condensed anima.

\begin{example}
  \label{sec:light-arc-stacks-1-condensed-anima-and-arc-stacks}
  Let $S$ be a light profinite set.
  Then we can define the arc-stack $\underline{S}$ sending a separable perfectoid ring $A$ to $\Hom_{\mathrm{cont}}(\mathcal{M}(A),S)$.
  Clearly, the functor $S\mapsto \underline{S}$ sends hypercovers of light profinite sets to arc-hypercovers and preserves fiber products, so that we get a colimit-preserving, left-exact functor
  \[
    \underline{(-)}\colon \Cat{CondAni}\to \Cat{ArcStk}.
  \]
We claim that if $T$ is a condensed set, then this functor is given by 
  \[
    A\mapsto F_T(A):=\Map_{\Cat{CondAni}}(\mathcal{M}(A),T)
  \]
  on separable perfectoid rings $A$.
  We first note that $\underline{T}$ is $0$-truncated by \cite[Proposition 6.5.16.(4)]{lurie_higher_topos_theory}, and that $F_T(-)$ takes $0$-truncated values, as $T$ is a condensed set.  
  As arc-covers of separable perfectoid rings induce covers, i.e., surjections, on $\mathcal{M}(-)$ (and $\mathcal{M}(B\otimes_AC)$ surjects onto $\mathcal{M}(B)\times_{\mathcal{M}(A)}\mathcal{M}(C)$ for any Banach rings $A,B,C$) the functor $F_T$ is an arc-sheaf on separable perfectoid spaces. 
  It suffices therefore to see that $T\mapsto F_T$ defined on sheaves of sets on light profinite sets by the same formula preserves colimits as a functor to arc-sheaves of sets and that $F_T=\underline{T}$ if $T$ is light profinite.
  The last assertion is clear (by definition of $\underline{(-)}$), and the first is clear as arc-covers induce surjections on Berkovich spaces.

  We caution the reader that in general $\underline{T}\ncong F_T$ if $T$ is not a condensed set: If $T=BG$ for a finite discrete group $G$, then for $A=K$ any separable perfectoid field, we would get that $F_T(A)=\Map_{\Cat{CondAni}}(\ast,T)=T$ while $\Map_{\Cat{ArcStk}}(\Marc(A),\underline{T})$ is the $1$-groupoid of \'etale $G$-torsors over $\Spa(K)$, and clearly not every $G$-torsor over $\Spa(K)$ is trivial for any such $K$. 
\end{example}

\begin{example}\label{ExaTopologicalSpaceCondensedAnima}
  We continue with the notation of \cref{sec:light-arc-stacks-1-condensed-anima-and-arc-stacks}.
  Let $ \Cat{AffPerfd}^{\mathrm{TD}}_{\omega_1,X}$ be the category of separable totally disconnected affinoid perfectoid spaces (over some implicit totally disconnected affinoid perfectoid space $X$). Each of the two functors in the adjunction 
\[
\mathcal{M}(-)\colon \Cat{AffPerfd}^{\mathrm{td}}_{\omega_1,X} \rightleftarrows \colon \Cat{Prof}^{\mathrm{light}} \colon \underline{(-)}
\]
sends hypercovers to hypercovers and thus, by \cite[Lemma A.3.6]{mann2022p}, this adjunction naturally extends to an adjunction
\[
(-)_{\mathrm{cond}}\colon \Cat{ArcStk}_X  \rightleftarrows \Cat{CondAni} \colon \underline{(-)} .
\]
In particular, the functor $(-)_{\mathrm{cond}}$ is colimit preserving (by \cite[Lemma A.3.6]{mann2022p} as well) and we call it the \textit{underlying condensed anima} of an arc-stack.
In fact, this assertion can be checked after base change to a totally disconnected affinoid perfectoid base.
Composing with the left adjoint $(-)_{\mathrm{top}}\colon \Cat{CondAni}\to \Cat{Top}$ from condensed anima to topological spaces (see \cite[Proposition 1.7]{scholze_lectures_on_condensed_mathematics})  one recovers the underlying topological space of an arc-stack.
\end{example}

We finish by discussing the quasi-finite dimensional variant. In the following, we denote by $\A^{n,\diamond}_{\Z_p}$ the $n$-dimensional affine line over $\Z_p$, regarded as an arc-stack over $\Z_p$.

\begin{definition}
\label{def-qfd-light-arc-stack}\
\begin{enumerate}

\item A morphism of arc-stacks $f\colon Y\to X$ over $\Z_p$ is \textit{quasi-finite dimensional} (\textit{qfd} for short) if  $f$ is quasi-pro-\'etale over some relative affine space $\A^{n}_{X}:=X\times_{\Marc(\Z_p)} \A^{n,\diamond}_{\Z_p}$.

\item  Let $f\colon A\to B$ be a morphism of separable Banach Tate algebras over $\Z_p$. Then $f$ is said to be \textit{quasi-finite dimensional} if the  induced map $\Marc(B)\to \Marc(A)$ is quasi-finite dimensional.

\item  We define the category of \textit{qfd  arc-stacks}  over $\Z_p$,
  \[
    \Cat{ArcStk}^\qfd_{\Z_p}:=\widehat{\Shv}(\Cat{AffPerfd}_{\Z_p,\omega_1}^\qfd),
  \]
   as the category of hypersheaves on the arc-site of separable qfd perfectoid spaces over $\Z_p$.
   \end{enumerate}
  \end{definition}

\begin{remark}\label{RemQProetEt}
  Any affinoid perfectoid space $X$ admits a pro-\'etale cover by a strictly totally disconnected space $X^{\prime}$.
  If in addition $X$ is qfd then $X^{\prime}$ is qfd as well, and so is the \v{C}ech nerve of $X'\to X$.
  This implies that qfd strictly totally disconnected perfectoid spaces define a basis for the arc-topology of qfd affinoid perfectoid spaces.

Hence,  one can introduce the notions of \'etale, finite \'etale and quasi-pro-\'etale morphisms of qfd arc-stacks analogously to those for ordinary arc-stacks: namely, a morphism $Y\to X$ of qfd arc-stacks is \'etale, finite \'etale, quasi-pro-\'etale if the pullback along all qfd totally disconnected space $X^{\prime}\to X$ is \'etale, finite \'etale, pro-\'etale respectively.
\end{remark}

\begin{example}\label{ExamDifferentArcStkZpFp}
  Let $A$ be a separable perfectoid ring over $\Z_p$, then $A$ is qfd over $\Z_p$ if and only if there is a morphism of rings $f\colon \Z_p[X_1,\ldots, X_n]\to A$ which induces a quasi-pro-\'etale map at the level of arc-stacks.
  If $A$ is defined over $\Q_p$, then this is equivalent to saying that there is some $r>0$ such that there is a map of Banach algebras $\Q_p\langle X_1,\ldots, X_n \colon |X_i|\leq r \rangle\to A$ which is quasi-pro-\'etale at the level of arc-stacks (here $\Q_p\langle X_1,\ldots, X_n \colon |X_i|\leq r \rangle$ denotes the ring of overconvergent functions on the disc of radius $r$).
  On the other hand, suppose that $A$ lives over $\F_p$ and that $\pi$ is a pseudo-uniformizer of $A$.
  Then, by extending the map $f$ to $\Z_p[\pi, X_1,\ldots, X_n]$, the algebra $A$ is qfd over $\F_p$ if and only if it is qfd over $\F_p((\pi^{1/p^{\infty}}))$ if and only if there is some $r>0$ and a map of Banach $\F_p((\pi^{1/p^{\infty}}))$-algebras  $\F_p((\pi^{1/p^{\infty}}))\langle X_1,\ldots, X_n \colon |X_i|\leq r \rangle \to A$ which is quasi-pro-\'etale on arc-stacks.

Furthermore, a separable perfectoid ring over $\Z_p$ is qfd if and only if its tilt is qfd over $\F_p$, in fact, this follows from the fact that the perfected affine line $\A_{\Z_p}^{1,\mathrm{perf}}\to \A_{\Z_p}^1$ obtained by taking $p$-th power roots of the coordinate is quasi-pro-\'etale. In particular, we have an equivalence of sites
\[
\Cat{AffPerfd}^{\mathrm{qfd}}_{\Z_p,\omega_1} = \big(\Cat{AffPerfd}^{\ob{qfd}}_{\F_p,\omega_1}\big)_{/ \Marc(\Z_p)}
\] 
and therefore an equivalence of topoi
\[
\Cat{ArcStk}^\qfd_{\Z_p}= \big(\Cat{ArcStk}^{\qfd}_{\F_p}\big)_{/\Marc(\Z_p)}. 
\]
\end{example}

\begin{remark}\label{remRelationQfdWithArc}
Since the inclusion $\Cat{AffPerfd}^{\qfd}_{\Z_p,\omega_1}\subset \Cat{AffPerfd}_{\Z_p,\omega_1}$ preserves fiber products and arc-covers, we have an induced geometric morphism of topoi $\Cat{ArcStk}_{\Z_p}\to \Cat{ArcStk}_{\Z_p}^{\qfd}$ whose left adjoint is the unique left exact colimit preserving functor $\Cat{ArcStk}^{\qfd}_{\Z_p}\to \Cat{ArcStk}_{\Z_p}$ that sends $\Marc(A)$ to itself for $A$ a separable qfd perfectoid ring over $\Z_p$.  
\end{remark}

\subsection{Gelfand stacks}
\label{sec:tdstacks}
While light arc-stacks over $\Q_p$, as discussed in \cref{sec:light-arc-stacks}, form the target of our perfectoidization functor, Gelfand stacks form the source. We will introduce this variant of analytic stacks in this subsection. \medskip

Let us first recall the general construction of the $\infty$-category of analytic stacks, \cite{AnStacks}.\footnote{In the course \cite{AnStacks}, the definition was slightly off, and is corrected here.} In the following, we denote by $\Cat{AnRing}$ the category of analytic rings, and use the 6-functor formalism on $\Cat{AnRing}$.

We left Kan extend the functor $A\to \ob{D}(A)$ (with $!$-pullback functoriality), from the category $\Cat{AnRing}^{!}$ of analytic rings with $!$-able maps towards $\infty$-categories, to a functor $X\mapsto \ob{D}^!(X)$ from the presheaf category $\mathcal{P}(\Cat{AnRing}^{!,\op})$ of functors from $\Cat{AnRing}^{!}$ to anima.

\begin{definition}[\cite{AnStacks}]\label{xhs92mk} 
Let $\mathcal{C}\subset \Cat{AnRing}^{\op}$ be a small subcategory of analytic rings, stable under pushouts of analytic rings and finite products; let $\mathcal C^!$ be the non-full subcategory consisting of all objects, and only the $!$-able maps.

A map $f\colon Y\to X$ in $\mathcal{P}(\mathcal{C}^!)$ is a \textit{$!$-equivalence} if any pullback of any iterated diagonal of $f$ is sent to an isomorphism under $\ob{D}^!$.

The category of \textit{analytic stacks over $\mathcal{C}$} is the full subcategory $\Cat{AnStk}(\mathcal{C}) \subset \mathcal{P}(\mathcal{C})$ obtained from presheaves by localizing at the countably presented $!$-equivalences. More precisely, a presheaf $X\in \mathcal{P}(\mathcal{C})$ is an analytic stack if for any $!$-equivalence $f\colon Y'\to X'$ between $\omega_1$-compact objects\footnote{This condition is only put in to avoid possible set-theoretic issues, and irrelevant in practice.} in $\mathcal P(\mathcal C^!)$, one has $f^\ast\colon \mathrm{Map}_{\mathcal{P}(\mathcal{C}^!)}(X',X)\cong \mathrm{Map}_{\mathcal{P}(\mathcal{C}^!)}(Y',X)$.
\end{definition}

\begin{remark}
  \label{sec:gelfand-stacks-remark-description-of-shriek-equivalences}
  Let $\{A\to B_i\}_i$ be a finite collection of $!$-able maps of analytic rings, and let $f\colon Y\subset X=\mathrm{AnSpec}(A)$ be the presheaf image of $\bigsqcup_i \mathrm{AnSpec}(B_i)\to \mathrm{AnSpec}(A)$, i.e.~the geometric realization of the \v{C}ech nerve.
  Then the diagonal of $f$ is an isomorphism, and so $f$ is a $!$-equivalence if and only if $\ob{D}$ satisfies $!$-descent along $f$.
  (Indeed, $!$-descent can be written as a colimit in $\Cat{Pr}^L$, as in the display below, and this colimit commutes with base change.) This is precisely the class of $!$-covers considered in \cite{AnStacks}, and we see that all analytic stacks are in particular sheaves with respect to $!$-covers. Conversely, any $!$-hypersheaf is an analytic stack. Indeed, for this it suffices that $!$-equivalences are $\infty$-connective in the $\infty$-topos of sheaves with respect to $!$-covers. By stability under diagonals, it suffices to see that $!$-equivalences are surjective. This can be checked locally, so assume $f: Y\to X=\mathrm{AnSpec}(A)$ is a $!$-equivalence. Write $Y=\varinjlim_i \mathrm{AnSpec}(B_i)$ as a colimit of affines. By $!$-descent, we have $\ob{D}(A)\cong \varprojlim_i \ob{D}^!(B_i)$, or equivalently
\[
\varinjlim_{i} \ob{D}_!(B_i)\to \ob{D}(A)
\]
is an equivalence in $\Cat{Pr}^L_{\ob{D}(A)}$. In particular, the compact object $A\in \ob{D}(A)$ can be written as a colimit of objects in the image of $\ob{D}(B_i)\to \ob{D}(A)$ for varying $i$. By compactness, it can be written as a retract of a finite such colimit; and a finite such colimit only involves finitely many indices $i$. The resulting finitely many maps $\{A\to B_i\}_i$ then form a $!$-cover over which $f$ admits a splitting.
\end{remark}

For example -- ignoring size issues -- for $\mathcal C=\Cat{AnRing}^{\op}$, we recover the $\infty$-topos of analytic stacks $\Cat{AnStk}$ of \cite{AnStacks}.

\begin{remark}\label{RemarkDStarDescentAnStk}
It is clear from the construction that $X\mapsto \ob{D}^!(X)$ is a sheaf in analytic stacks. We remark here that the functor $X\mapsto \ob{D}^{*}(X)$ also defines a sheaf in analytic stacks. This follows formally by taking linear duals in $\Cat{Pr}^L$.  Let us give a sketch of the proof of this fact. First, we note that the functor sending an analytic ring $A$ to $\ob{D}^*(A)$ is a sheaf for the $!$-topology. Indeed, this follows from \cite[Proposition 6.3.6 (a) and (b)]{SolidNotes}. It is now left to see that if $Y\to X$ is a $!$-equivalence, then the natural map $\ob{D}^*(X)\to \ob{D}^*(Y)$ is an equivalence. By $!$-descent on the target, we can assume that $X=\AnSpec(A)$ is affinoid, we can then write $Y=\varinjlim_{i} \AnSpec(B_i)$ as a colimit of $!$-able rings over   $A$. Consider the pullback map of linear categories
\[
F\colon \Cat{Pr}^L_{\ob{D}(A)}\to \Cat{Pr}^L_{Y}=\varprojlim_i \Cat{Pr}^L_{\ob{D}(B_i)}.
\]
Since the map $f\colon A\to B_i$ is $!$-able, it has both a left and right adjoint given by $f_!$. Thus, $F$ also has a left and right adjoint. The left adjoint sends a cocartesian section $(M_i)_{i\in I}$ of $\ob{D}(B_i)$-linear categories to the colimit $\varinjlim_{i} M_i$ in $\Cat{Pr}^L_{\ob{D}(A)}$ along lower $!$-maps. The right adjoint sends $(M_i)_{i\in I}$ to $\varprojlim_{i\in I} M_i$ along pullback maps. By projection formula, the functor $F$ is fully faithful if and only if the natural map
\[
\varinjlim_{i\in I} \ob{D}_!(B_i)\to \ob{D}(A)
\]
is an equivalence, where the transition maps are given by lower $!$-maps. By passing to right adjoints this is equivalent to the natural map 
\[
\ob{D}(A)\to \varprojlim_{i\in I} \ob{D}^!(B_i) = \ob{D}^!(Y)
\]
to be an equivalence, which  holds as $Y\to X$ is a $!$-equivalence. Thus, as $F$ is fully faithful, by looking at its right adjoint we deduce that 
\[
\ob{D}(A)\to \varprojlim_{i\in I} \ob{D}^*(B_i)= \ob{D}^*(Y)
\]
is also an equivalence, proving the desired $\ob{D}^*$-descent. 
\end{remark}

\begin{remark} The class of $!$-equivalences is by definition stable under passing to diagonals and base change. By \cite[Theorem 4.3.3]{left-exact-loc}, it follows that $\Cat{AnStk}(\mathcal C)$ is an $\infty$-topos; more precisely, it is a left exact accessible localization of $\mathcal P(\mathcal C)$.\footnote{In \cite{AnStacks}, it was erroneously not ensured that the class of $!$-equivalences is stable under diagonals; this is corrected here.}
\end{remark}

\begin{remark}\label{rem:map-to-AnStk} Assume that $\mathcal T$ is another $\infty$-topos presented as a left-exact accessible localization of a presheaf topos $\mathcal P(\mathcal D)$.
  Assume moreover that we have a functor $F\colon\mathcal D\to \mathcal C^!$, thus inducing a morphism of presheaf topoi $\mathcal P(\mathcal D)\to \mathcal P(\mathcal C^!)\to \mathcal P(\mathcal C)$.
  To check that the composite functor $\mathcal P(\mathcal D)\to \mathcal P(\mathcal C)\to \Cat{AnStk}(\mathcal C)$ factors over a, necessarily unique, left-exact colimit-preserving functor $\mathcal T\to \Cat{AnStk}(\mathcal C)$, we have to see that for any morphism $f\colon Y\to X$ in $\mathcal P(\mathcal D)$ that becomes an isomorphism in $\mathcal T$, the image $F(f)$ of $f$ in $\mathcal P(\mathcal C)$ is a $!$-equivalence.
  However, as any pullback of any diagonal of $f$ also becomes an isomorphism in $\mathcal T$, it suffices to see that for any such $f$ with $X\in \mathcal D$, the functor
\[
\ob{D}^!(F(X))\to \ob{D}^!(F(Y))
\]
is an equivalence.
Moreover, it is enough to check this for a generating class of such $f$. For example, if $\mathcal T$ is the $\infty$-topos of hypersheaves for some Grothendieck topology, this can be taken to be the corresponding class of hypercovers.
In practice, this often amounts to verifying the following.
Assume for simplicity that there exists a generating class $\mathcal{S}$ of hypercovers in $\mathcal{P}(\mathcal{D})$, which map to augmented simplicial \emph{affine} analytic stacks with proper transition maps.
Then it suffices to see that for any simplicial analytic stack $\AnSpec(A_\bullet)\to \AnSpec(A)$, which is in the image of $\mathcal{S}$, the natural functor
\[
\Phi_{A_\bullet}\colon \ob{D}(A)\to \mathrm{Tot}(\ob{D}^!(A_\bullet))
\]
is fully faithful, or equivalently, the pro-object $\{\mathrm{Tot}_{\leq n} A_\bullet\}$ is pro-constant.
Indeed, if $A\to A_\bullet$ is a \v{C}ech nerve, then $\Phi_{A_\bullet}$ is even an equivalence (as $A_0$ is a descendable $A$-algebra) and thus $X\in \mathcal{P}(\mathcal{D})\mapsto \ob{D}^!(F(X))$ is a sheaf.
By \cite[Proposition A.3.21]{mann2022p} the assumed fully faithfulness of $\Phi_{A_\bullet}$ implies then that this sheaf is a hypersheaf, as desired (and thus a posteriori that the $\Phi_{A_\bullet}$ are equivalences).

\end{remark}

We can now make the following definition, following the general pattern of \cref{xhs92mk}.

\begin{definition}\label{xhs9wj}\
  \begin{enumerate}
  \item A Gelfand ring $A$ is called \textit{separable} if $A^u$ is a separable Banach algebra over $\Q_p$.
  \item The category of Gelfand stacks is defined as
  $$\Cat{GelfStk}:=\Cat{AnStk}(\Cat{GelfRing}_{\omega_1}^{\op})$$
  where $\Cat{GelfRing_{\omega_1}}\subseteq \Cat{AnRing}_{\Q_{p,\solid}}$ is the full subcategory of separable Gelfand rings\footnote{One should restrict to an essentially small category of separable Gelfand rings, e.g. the full subcategory of $\kappa$-compact separable Gelfand rings for $\kappa$ an infinite regular cardinal. We will ignore these set-theoretic issues in the rest of the paper.}.
    \item For $A\in \Cat{GelfRing}_{\omega_1}$, we denote by $\ob{GSpec}(A)$ the functor corepresented by $A$ on separable Gelfand rings.
  \end{enumerate}
\end{definition}

We note that the category of separable Gelfand rings is stable under pushouts and products as required in \cref{xhs92mk},\ cf.\ \cref{xnbsyw}.

\begin{lemma}
  \label{sec:totally-disc-stacks-morphism-of-topoi}
  The inclusion $\Cat{GelfRing}_{\omega_1}\subseteq \Cat{AnRing}_{\Q_{p,\solid}}$ preserves pushouts and $!$-equivalences, and hence defines a morphism of $\infty$-topoi $\Cat{AnStk}_{\Q_{p,\solid}} \to  \Cat{GelfStk}$, whose right adjoint $(-)^{\Gelf}$ sends $X\in \Cat{AnStk}$ to the functor $A\in \Cat{GelfRing}_{\omega_1}\mapsto X(\ob{AnSpec}(A))$. The left adjoint $(-)_\An$ is the unique left exact colimit preserving functor sending $\ob{GSpec}(A)$ for $A\in \Cat{GelfRing}_{\omega_1}$ to $\ob{AnSpec}(A)$.
\end{lemma}
\begin{proof}
  This is formally implied by the fact that the inclusion $\Cat{GelfRing}_{\omega_1}\to \Cat{AnRing}_{\Q_{p,\solid}}$ preserves pushouts and $!$-equivalences.
\end{proof}

 We note that the functor $(-)_\An$ is probably not fully faithful, the problem being that a $!$-cover of a Gelfand ring by an arbitrary analytic ring may not be refined by a $!$-cover from a Gelfand ring.\footnote{For example, $\Q_p\to \Z_p((x))^\wedge_p[1/p]$ is such a $!$-cover.}

 As the site $\Cat{GelfRing}_{\omega_1}^\op$ is subcanonical, the restriction of $(-)_\An$ to affinoid Gelfand stacks, i.e., the essential image of $\Cat{GelfRing}_{\omega_1}\to \Cat{GelfStk}$, is fully faithful.

\begin{remark}
  \label{sec:totally-disc-stacks-qcoh-on-td-stacks}
  The $6$-functor formalism $$\ob{D}(-)\colon \Cat{AnStk}_{\Q_{p,\solid}}\to \Cat{Pr}^L_{\ob{D}(\Q_{p,\solid})}$$ on analytic stacks over $\AnSpec(\Q_{p,\solid})$, pulls back (via composition with $(-)^\An\colon \Cat{GelfStk}\to \Cat{AnStk}_{\Q_{p,\solid}}$) to a $6$-functor formalism on Gelfand stacks. In particular, we obtain the notions of $!$-able/suave/prim/...  maps between Gelfand stacks (see \cite{heyer20246functorformalismssmoothrepresentations}).
\end{remark}

\begin{lemma}
  \label{sec:totally-disc-stacks-td-betti-stacks}
  The functor $(-)_{\ob{Betti}}\times_{\ob{AnSpec}(\Z)} \ob{AnSpec}(\Q_{p,\solid})\colon \Cat{CondAni}\to \Cat{AnStk}_{\Q_{p,\solid}}$ factors naturally through a colimit-preserving functor $$(-)_{\ob{Betti}}^\Gelf\colon \Cat{CondAni}\to \Cat{GelfStk}$$ sending a light profinite set $S$ to the Gelfand spectrum $\ob{GSpec}(C^{\lc}(S,\Q_p))$ of the solid  $\Q_p$-algebra of locally constant functions $S\to \Q_p$. Moreover, $(-)_{\mathrm{Betti}}^\Gelf$ is the left adjoint in a morphism of $\infty$-topoi $(-)_{\mathrm{cond}}\colon\Cat{GelfStk} \to \Cat{CondAni}$.
\end{lemma}
\begin{proof}
  By \cref{rem:map-to-AnStk}, it suffices to see that given a hypercover $S_\bullet\to S$ of light profinite sets with $A:=C^{\lc}(S,\Q_p)\to A_\bullet:=C^{\lc}(S_\bullet,\Q_p)$, one has $!$-descent $\ob{D}(A)\cong \mathrm{lim}^! \ob{D}(A_\bullet)$. For this, one has to see that the pro-system of partial totalizations of $\{\varprojlim_{[n]\in \Delta_{\leq k}} A_\bullet\}_k$ is pro-isomorphic to $A$. This can already be checked with $\mathbb Z$-coefficients, and noting that cofiber $f_k$ of the map from $A$ to the $k$-th totalization is discrete, has countable cohomologies and lives in cohomological degrees $\geq k$. This implies that the map $f_k\to A$ is zero (using that $S$ is profinite).
\end{proof}

\begin{remark}
  \label{sec:gelfand-stacks-notation-for-betti-stack}
   To lighten notation, from now on we will write $(-)_{\Betti}\colon \Cat{CondAni}\to \Cat{GelfStk}$ for the Betti realization functor of condensed anima in Gelfand stacks.
\end{remark}

Next, we note that there is a well-behaved category of perfect modules on Gelfand stacks.

\begin{lemma}\label{LemmaDescentPerfectComplexes}
  Let $a\leq b$ be two integers.
  Let $\mathcal{F}$ be the prestack of categories sending a Gelfand ring $A$ to the category of perfect $A$-modules with amplitude $[a,b]$.
  Then the unique colimit preserving extension of $\mathcal{F}$ to a functor $\mathcal{P}(\Cat{GelfRing}^\op)\to \Cat{Cat}_\infty$ inverts $!$-equivalences.
  In particular, $\mathcal{F}$ extends uniquely to a sheaf of categories $\Perf^{[a,b]}(-)$ on $\Cat{GelfStk}$.
\end{lemma}
\begin{proof}
  The proof is similar to the proof of \cite[Corollary VI.1.4]{ScholzeRLL}, up to having to take care of the $!$-equivalences. We first show that $\mathcal{F}$ is a sheaf for the $!$-topology.
  As perfect complexes on a Gelfand ring $A$ embed into the full category of quasi-coherent sheaves $D(A)$, it suffices to see that being a perfect complex is a $!$-local property.
  Thus, let $M \in \ob{D}(A)$ be an object which is locally for the $!$-topology a perfect module of amplitude $[a,b]$.
  Since dualizability is a local property, $M$ is automatically dualizable.
  Since $A$ is Gelfand, it is Fredholm (\Cref{xh28sj}), hence the dualizable $A$-module $M$ is a perfect  $A$-module.
  The descent of the perfect amplitude follows from \cite[Corollary VI.1.3]{ScholzeRLL}.
  To handle the inversion of $!$-equivalences, we can note by \Cref{RemarkDStarDescentAnStk} that $A\mapsto D(A)$ (with $*$-pullbacks) is a sheaf of categories on $\Cat{GelfStk}$ (i.e., its extension to presheaves inverts $!$-equivalences), and that $\Perf^{[a,b]}(A)$ embeds fully faithfully into $D(A)$.
  Thus, $\Perf^{[a,b]}(-)$ satisfies the full-faithfulness assumption of \cite[Proposition A.3.21]{mann2022p} (for hypercovers of Gelfand rings whose geometric realization is a $!$-equivalence), and as it is $!$-sheaf (as we have shown before) the same argument implies then that $\Perf^{[a,b]}(-)$ inverts any $!$-equvivalence.
\end{proof}

\begin{definition}
Let $X$ be a Gelfand stack. Let $a\leq b$ be two integers. An object $M\in \ob{D}(X)$ is a \textit{perfect module of amplitude $[a,b]$}  if for all separable Gelfand ring $A$ and map $f\colon \ob{GSpec} A\to X$ of Gelfand stacks,  the pullback $f^*M$ is a perfect module of amplitude $[a,b]$. A perfect module of amplitude $[0,0]$ is also called a \textit{vector bundle}.
\end{definition}

\begin{remark}\label{RemarkPerfectDescent}
By \cref{LemmaDescentPerfectComplexes} it suffices to test whether a module $M$ on a Gelfand stack $X$ is a perfect module after pulling back along a jointly surjective family of morphisms of stacks $\{\GSpec(A_i)\to X\}_{i}$.
\end{remark}

Finally, we introduce a quasi-finite dimensional version of Gelfand stacks.

\begin{definition}\label{xhs9wjqfd}\
  \begin{enumerate}
  \item A  separable Gelfand ring $A$ is called \textit{quasi-finite dimensional}  (or \textit{qfd}) if $A^u$ is a separable qfd Banach algebra over $\Q_p$. We let $\Cat{GelfRing}_{\omega_1}^\qfd \subseteq \Cat{AnRing}_{\Q_{p,\solid}}$ be the full subcategory of separable qfd Gelfand rings.
  
  \item The category of qfd Gelfand stacks is defined as 
  \[
  \Cat{GelfStk}^\qfd:=\Cat{AnStk}((\Cat{GelfRing}_{\omega_1}^{\qfd})^{\op}).
  \]
    \item For $A\in \Cat{GelfRing}_{\omega_1}^\qfd$, we still denote by $\ob{GSpec}(A)$ the functor corepresented by $A$ on separable qfd Gelfand rings.
  \end{enumerate}
\end{definition}

Similarly as for qfd arc-stacks, cf.\ \cref{remRelationQfdWithArc}, we have a geometric morphism of $\infty$-topoi $\Cat{GelfStk}\to \Cat{GelfStk}^{\qfd}$ whose left adjoint is the unique left exact colimit preserving functor sending the spectrum $\GSpec (A)$ of a separable qfd Gelfand ring $A$ to itself. In particular, we can transfer the $6$-functor formalism of quasi-coherent sheaves on Gelfand stacks along the functor $\Cat{GelfStk}^{\qfd}\to \Cat{GelfStk}$.

\begin{example}\label{ExamCondAniqfd}
  Continuing with our series of examples arising from condensed anima, cf.\ \cref{sec:totally-disc-stacks-td-betti-stacks}, the functor $C^{\lc}(-,\Q_p)\colon \Prof^{\mathrm{light},\op}\to \Cat{GelfRing}_{\omega_1}$ of locally constant functions on light profinite sets factors through $\Cat{GelfRing}^{\qfd}_{\omega_1}$.
  Thanks to \textit{loc.\ cit.\ }we have an induced left exact colimit preserving functor
\[
(-)_{\Betti}^{\qfd}\colon \widehat{\Shv}(\Cat{Prof}^{\mathrm{light}})\to \Cat{GelfStk}^{\qfd}
\]
from the category of hypersheaves on profinite sets with values in anima towards  $\qfd$ Gelfand stacks.

We note that we have a commutative diagram of left exact colimit preserving functors
\[
\begin{tikzcd}
\widehat{\Shv}(\Cat{Prof}^{\mathrm{light}}) \ar[r] \ar[d,"(-)_{\Betti}^{\qfd}"'] & \Cat{CondAni} \ar[d, "(-)_{\Betti}"] \\
\Cat{GelfStk}^{\qfd} \ar[r] & \Cat{GelfStk},
\end{tikzcd}
\]
where the horizontal arrows are the natural left Kan extensions that are the identity on representable objects.
In order to avoid cumbersome notation, we shall keep writing $(-)^{\qfd}_{\Betti}$ as $(-)_{\Betti}$ and simply call it the Betti realization in qfd Gelfand stacks.
Throughout the text we will make explicit which incarnation of the Betti stack we are referring to. 
\end{example}

\subsection{Derived Berkovich spaces}
\label{sec:derived-berkovich-spaces}

The aim of this subsection is to construct a category of  \textit{derived Berkovich spaces over $\Q_{p,\solid}$}, as a full subcategory of the category of Gelfand stacks. This construction is completely analogous to that of derived Tate adic spaces of \cite[Section 2.7]{camargo2024analytic}, and will generalize partially proper rigid spaces and more generally (good) Berkovich spaces.

\begin{definition}\label{DefinitionBerkovichSpaces}\label{RemTopSpaceBerkovichSpace}
We let  $\mathcal{C}=\Cat{GelfStk}^{\ob{aff}}$ denote the category of affinoid Gelfand stacks, that is,  the category of  Gelfand stacks representable by separable Gelfand rings.

\begin{enumerate}

\item An \textit{analytic cover} of an affinoid Gelfand stack $X$ is a finite collection $\{Y_i\to X\}_{i\in I}$  of rational localizations of $X$, such that the closed cover $\{|Y_i|\to |X| \}$ is \textit{strict}, that is, it is refined by an open cover of the topological space $|X|$. Notice that any analytic cover in $\mathcal{C}$ is a $!$-cover of affinoid Gelfand stacks. 

\item Let  $\mathcal{C}_{\an}$ be the analytic site of affinoid Gelfand stacks and let $\Shv(\mathcal{C}_{\an})$ be the category of sheaves on $\mathcal{C}_{\an}$.  Any sheaf $Y\in \Shv(\mathcal{C}_{\an})$ has an underlying topological space $|Y|$ constructed as the colimit $|Y|:=\varinjlim_{X\in \mathcal{C}_{/Y}} |X|$ where $X$ runs over all affinoid Gelfand stacks over $Y$, and $|X|$ is the Berkovich spectrum of $X$.  Given a map $Z\to Y$ in $\Shv(\mathcal{C}_{\an})$ we have an induced map of topological spaces $|Z|\to |Y|$.

\item Let $U\subset |Y|$ be an open subspace, we define $Y_U\to Y$  to be the subsheaf whose values at $A\in \Cat{GelfRing}$ is given by  those maps $\GSpec (A)\to Y$ such that $\mathcal{M}(A)\to |Y|$ factors through $U$. A map $Z\to Y$ is called an \textit{open immersion} if it is equivalent to $Y_U\to Y$ for some open $U\subset Y$.

\item A sheaf $Y\in \Shv(\mathcal{C}_{\an})$ is called a \textit{derived Berkovich space over $\Q_{p,\solid}$} (or simply a derived Berkovich space) if there is an open cover $\{U_i\subset Y\}_{i\in I}$ of $Y$, and open immersions $U_i\to X_i$ where $X_i\in \mathcal{C}$ is an affinoid Gelfand stack. We let $\Cat{BerkSp}\subset \Shv(\mathcal{C}_{\an})$ be the full subcategory of derived Berkovich spaces. 

\item Given $Y$ a derived Berkovich space, a \textit{strict affinoid cover of $Y$} is a collection of  subsheaves $\{Y_i\subset Y\}_{i\in I}$ such that each $Y_i$ is representable by a separable Gelfand ring, and that the cover $\{|Y_i|\to |Y| \}_i$ is strict, that is, the cover can be refined by an open cover of $|Y|$.

\item By restricting all the previous constructions to the full subcategory $\mathcal{C}^{\qfd}\subset \mathcal{C}$ of qfd affinoid Gelfand stacks, we define in a similar fashion  the category $\Cat{BerkSp}^{\qfd}$ of \textit{quasi-finite dimensional Berkovich spaces}.  

\end{enumerate}
\end{definition}

\begin{remark}\label{RemarkComparisonTateSpaces}
There is a natural fully faithful functor $\Cat{GelfStk}^{\ob{aff}}\to \Cat{Aff}^b_{\Q_p}$ from affinoid Gelfand stacks to affinoid bounded spaces over $\Q_p$ (cf.\ \cite[Definition 2.7.20]{camargo2024analytic}). By definition, Berkovich analytic covers are refined by open covers of the underlying Berkovich spectrum which then produce analytic covers in the adic spectrum. This produces fully faithful embeddings 
\[
\Cat{BerkSp}^{\qfd}\hookrightarrow \Cat{BerkSp} \hookrightarrow \Cat{AdicSp}_{\Q_p} 
\]
from the category of (qfd) derived Berkovich spaces into the category of derived Tate adic spaces over $\Q_p$ (cf.\ \cite[Definition 2.7.22]{camargo2024analytic}). In particular, there are notions of  locally of finite presentation, smooth, \'etale, etc.\ for derived Berkovich spaces (see \cite[Definition 3.6.6]{camargo2024analytic}).
\end{remark}

The next proposition shows that derived Berkovich spaces form a full subcategory of Gelfand stacks.

\begin{proposition}\label{PropBerkovichAreStacks}
Let $Y\in \Cat{BerkSp}$ be a derived Berkovich space. Assume that $\varinjlim_i \mathrm{GSpec}(A_i)\to \mathrm{GSpec}(A)$ is a $!$-equivalence; or just that it induces an isomorphism after passing to arc-stacks, and $A=\mathrm{lim}_i A_i$. Then the natural map 
\[
Y(A)\to \varprojlim_i Y(A_i)
\]
is an equivalence. In particular, $Y$ defines a Gelfand stack. Furthermore, the functor $\Cat{BerkSp}\to \Cat{GelfStk}$ is fully faithful (resp.\ for $\Cat{BerkSp}^{\qfd}\hookrightarrow \Cat{GelfStk}^{\qfd}$).
\end{proposition}
\begin{proof}
Let $Y$ be a derived Berkovich space, we have a natural map of functors on Gelfand rings 
\[
Y(A)=\ob{Map}(\ob{GSpec} A,Y )\to \ob{Map}_{\ob{Top}}(\mathcal{M}(A), |Y|)
\]
for $A\in \Cat{GelfRing}$. The right hand side term satisfies arc-hyperdescent, and any $!$-equivalence gives rise to an arc-hypercover, namely, any map $A\to B$ of Gelfand rings for which the base change $B\otimes_A-$ is conservative is automatically an arc-cover. Evaluating at $A^{\bullet}$ and taking limits we get a commutative diagram  of anima
\[
\begin{tikzcd}
 Y(A) \ar[r] \ar[d] & \varprojlim_i Y(A_i) \ar[d]  \\ 
\ob{Map}_{\ob{Top}}(\mathcal{M}(A), |Y|) \ar[r] & \varprojlim_i\ob{Map}_{\ob{Top}}(\mathcal{M}(A_i), |Y|)
\end{tikzcd}
\]
where the bottom horizontal arrow is an equivalence. Thus, to show that the upper horizontal arrow is an isomorphism, it suffices to prove this over the fibers of $\ob{Map}_{\ob{Top}}(\mathcal{M}(A), |Y|) $. Therefore, we fix a map of topological spaces $f^{\ob{top}}\colon  \mathcal{M}(A)\to |Y|$ and consider the fibers 
\[
Y(A)_{/ f^{\ob{top}}}\to  \mathrm{lim}_i (Y(A_i))_{/f^{\ob{top}}}.  
\]
By taking a affinoid rational cover $Y_i$ of $Y$ whose pullback along $f^{\ob{top}}$ can be refined by a strict rational cover of $\mathcal{M}(A)$, we can formally reduce to the case where $Y=\ob{GSpec} B$ is itself affinoid. In that case, $Y(A)= \ob{Map}_{\ob{GelfRing}}(B,A)$ and the equivalence holds as $A=\mathrm{lim}_i A_i$.

It is left to show that the natural map $\Cat{BerkSp}\to \Cat{GelfStk}$ is fully faithful. For that, notice that any map $Y\to X$ in $\Cat{GelfStk}$ gives rise to a map of arc-stacks $Y^{\diamond}\to X^{\diamond}$ and so it gives rise to a map of topological spaces $|Y^{\diamond}|\to |X^{\diamond}|$. For a derived Berkovich space $Y$, one  has $|Y|=|Y^{\diamond}|$, namely, this holds for affinoid spaces and the general case follows from analytic descent.  Hence, if $X,Y$ are derived Berkovich spaces, we have a natural commutative triangle of maps of
\[
\begin{tikzcd}
\ob{Map}_{\ob{BerkSp}}(Y,X) \ar[r] \ar[rd] & \ob{Map}_{\Cat{GelfStk}}(Y,X)  \ar[d] \\ &  \ob{Map}_{\Cat{Top}}(|Y|,|X|).
\end{tikzcd}
\]
Therefore, to show that the upper horizontal map is an equivalence, it suffices to prove it on fibers over $\Cat{Map}_{\Cat{Top}}(|Y|,|X|)$. In that case, one fixes the map $f^{\ob{top}}\colon |Y|\to |X|$ of topological spaces, and by the same argument as before one formally reduces to the case when $Y$ and $X$ are affinoid where the equivalence is clear.
\end{proof}

Later, we will pay specific attention to \textit{rigid smooth} morphisms of derived Berkovich spaces. 

\begin{definition}\label{DefEtaleSmooth}
Let $f\colon Y\to X$ be a map  of derived Berkovich spaces over $\Q_p$. 

\begin{enumerate}

\item The map $f$ is called \textit{rigid \'etale} if, locally  in an open cover of $Y$ and $X$, $f$ admits a factorization along open immersions and finite \'etale maps.

\item The map $f$ is \textit{rigid smooth} if, locally in the analytic topology of $Y$ and $X$, it factors as a composite $Y\xrightarrow{g} \mathbb{A}^{d, \mathrm{an}}_{X}\to X$, where $g$ is rigid \'etale. 

\item The map $f$ is \textit{Berkovich \'etale} if locally in a strict closed cover of both $Y$ and $X$ (that is, a closed cover refined by an open cover), the map $f$ is a Berkovich \'etale map of affinoid Berkovich spaces as in \Cref{xjw92j}.

\item The map $f$ is \textit{Berkovich smooth} if locally in a strict closed cover of both $Y$ and $X$, it factors as a composite $Y\xrightarrow{g} \A^{d,\ob{\an}}_{X}\to X$ where $g$ is Berkovich \'etale. 

\end{enumerate}

\end{definition}

\begin{remark}
Any rigid \'etale (resp. rigid smooth map)  is Berkovich \'etale (resp. Berkovich smooth), but the converse does not hold: rational localizations are Berkovich \'etale but not rigid \'etale in general.   \end{remark}

Finally, we isolate an interesting subcategory of derived Berkovich spaces, equivalent to Gro\ss e-Kl\"onne's notion of dagger space (also called rigid space with overconvergent structure sheaf), \cite{grosse2000rigid}.

\begin{definition}\label{DefDaggerRigid}
A derived Berkovich space $Y$ over a (separable) non-archimedean field $K$ is called a \textit{$\dagger$-rigid space} if, locally on an affinoid cover, it is represented by a classical (i.e., static) quotient of an overconvergent Tate algebra $K\langle T_1,\ldots, T_n \rangle_{\leq r}$ for some $r>0$. 
If the structure morphism $Y \to \GSpec(K)$ is rigid smooth, we simply say that $Y$ is \textit{smooth}.
\end{definition}

For example, partially proper rigid spaces over $K$ are $\dagger$-rigid spaces.

\subsection{Arc- and $!$-topologies on perfectoid rings}
\label{sec:arc-and-!-topologies-on-perfectoid-rings}
To properly set up the theory of perfectoidization and de Rham stacks, it will be important to relate the arc-topology and the $!$-topology on perfectoid rings. It will even be useful to upgrade this to $\A_{\inf}$-coefficients, in the almost setup. This is what this rather technical subsection is devoted to. It ends with a discussion of the relation to these results to recent work developing a $6$-functor formalism for pro-\'etale $\Z_p$- or $\Q_p$-cohomology of rigid spaces, \cite{AMdescendPerfd}, \cite{ALBMFFCoho}. 

We start by introducing the following almost category:

\begin{definition}\label{DefAlmostAinfCat}
  Let $A$ be a perfectoid Tate ring over $\Z_p$ and let $\A_{\inf}(A)\colon= W(A^{\flat,\circ})$ be the ring of Witt vectors of the ring of power-bounded elements of its tilt.
  We see $\A_{\inf}(A)$ as a solid $\Z_p$-algebra.
  Let $\overline{A}:= A^{\circ}/A^{\circ\circ}$.
  By \cite[Lemma 10.3]{Bhatta} (or by a direct computation using reduction mod $p$), the map of solid rings $\A_{\inf}(A)\to W(\overline{A})$ is idempotent (here we use that the solid tensor product preserves connective $p$-complete modules, cf.\ \cite[Lemma 2.12.9]{mann2022p}).
  We define the category of solid almost $\A_{\inf}(A)$-modules to be the Verdier quotient
\[
\ob{D}^{a}(\A_{\inf}(A)):= \ob{D}(\A_{\inf}(A))/ \ob{D}(W(\overline{A})).
\]
\end{definition}

\begin{remark}\label{RemAlmostAinf}
Let $f\colon A\to B$ be a map of perfectoid Tate rings over $\Z_p$, then one has that $W(\overline{B})=W(\overline{A})\otimes_{\A_{\inf}(A)} \A_{\inf}(B)$. In particular, we have a natural base change map 
\[
f^*\colon \ob{D}^{a}(\A_{\inf}(A)) \to \ob{D}^a(\A_{\inf}(B))
\]
with conservative and $\ob{D}^{a}(\A_{\inf}(A)) $-linear right adjoint $f_*$ given by a forgetful functor.  Moreover, we have 
\[
\ob{D}^a(\A_{\inf}(B))=\ob{D}(\A_{\inf}(B))\otimes_{\ob{D}(\A_{\inf}(A))} \ob{D}^{a}(\A_{\inf}(A))
\]
using \cite[Remark 4.3]{AMdescendPerfd}.
Thus, one has an associated $6$-functor formalism on $A\mapsto \ob{D}^a(\A_{\inf}(A))$ compatible with the $6$-functor formalism of analytic rings, and the map $f$ gives rise to a proper map (having the induced analytic ring structure). 
\end{remark}

\begin{lemma}
  \label{sec:perf-analyt-de-2-arc-covers-are-good}
The functor $A\mapsto \ob{D}^{a,!}(\A_{\inf}(A))$ (with transition maps given by upper $!$-maps) on separable totally disconnected perfectoid rings satisfies arc-hyperdescent. In particular, if $A\to A_{\bullet}$ is an arc-hypercover of separable  perfectoid rings over $\Q_p$ with $A$ totally disconnected, then it is a $!$-equivalence.
\end{lemma}
\begin{proof}
  We use the criterion from \cref{rem:map-to-AnStk}, that is, we prove that the pro-system $(\ob{Tot}_{\leq k}(\A_{\inf}(A_{\bullet})))_k$  is pro-constant equal to $\A_{\inf}(A)$ in $\ob{D}^a(\A_{\inf}(A))$. By base changing along $\A_{\inf}(A)\to A$ one obtains the claim for the cover $A\to A_{\bullet}$. By almost arc-hyperdescent of $A\mapsto A^{\flat,\circ}$ (which follows formally from \cite[Proposition 8.8]{scholze_etale_cohomology_of_diamonds} and \cite[Lemma 4.6]{AMdescendPerfd}), we know that $\A_{\inf}(A)\to \ob{Tot}(\A_{\inf}(A_{\bullet}))$ is an isomorphism in $\ob{D}^{a}(\A_{\inf}(A))$.
  
 Thus,  it suffices now to see that $f_k\colon \A_{\inf}(A)\to \ob{Tot}_{\leq k}(\A_{\inf}(A_{\bullet}))$ admits a retract  in $\ob{D}^a(A^\circ)$ for all $k\gg 0$.  We note that the cofiber of $f_k$ is concentrated in cohomological degrees $\geq k$ thanks to the dual of \cite[Proposition 1.2.4.5]{lurie_higher_algebra}. Hence, it suffices to apply \cref{sec:perf-analyt-de-2-finite-cohom-dimension-for-separable-strictly-totally-disconnected}.
\end{proof}

\begin{lemma}
  \label{sec:perf-analyt-de-2-finite-cohom-dimension-for-separable-strictly-totally-disconnected} Let $A$ be a separable totally disconnected perfectoid Tate ring with pseudo-uniformizer $\varpi$, and let $R:=\A_{\inf}(A)$.
  Let $M,N\in \ob{D}(R)$ be static derived $(p,[\varpi])$-complete $R$-modules.
  Assume that $M$ is discrete and countable mod $(p,\varpi)$.
  Then $\mathrm{Ext}^i_{\ob{D}(R)}(M,N)=0$ for $i>5$.  
\end{lemma}
\begin{proof} In the following all tensor products, $\iHom$, quotients, etc.\ are derived.
  We have that
  \[
  \iHom_{R}(M,N)\otimes_{R} R/(p,[\varpi]) =  \iHom_{R}(M,N/(p,[\varpi])  = \iHom_{R/(p,\varpi)}(M/^{\L}(p,[\varpi]), N/^{\L}(p,\varpi)).
  \]
  Since $N$ is derived $(p,[\varpi])$-complete, then so is $ \iHom_{R}(M,N)$ and by the derived Nakayama's lemma it suffices to show that $ \iHom_{R/(p,\varpi)}(M/(p,[\varpi]), N/(p,\varpi))$ is $(-5)$-connective, i.e., its homotopy groups vanish in degrees $<-5$.
  Since $M$ is static, $M/(p,[\varpi])$ is supported in homological degrees $[0,2]$.
  Thus, it suffices to show that if $N$ and $M$ are static and countable  $A^{\circ}/\varpi=R/(p,\varpi)$-modules then $\ob{Ext}^i_{A^{\circ}/\varpi}(M,N)=0$ for $i>3$.
  To show that, it suffices to prove that $\ob{Ext}^i_{A^{\circ}}(M,N)=0$ for $i>2$ if $N$ and $M$  are $\varpi$-torsion and $M$ is countable. 
  
  We can find a short exact sequence of $A^{\circ}$-modules $0\to M^{\prime\prime}\to M^\prime\to M\to 0$ with $M^{\prime}\cong \widehat{\bigoplus}_{i\in \N} A^{\circ}$, and $M^{\prime\prime}$ $\varpi$-torsion free and (derived) $\varpi$-adically complete.
  Using that $M'/\varpi$ is a free $A^\circ/\varpi$-module, it suffices to show that $\mathrm{Ext}^i_{A^{\circ}}(M,N)=0$ for $i>1$ if $M$ is static derived $\varpi$-complete $\varpi$-torsion free with $M/\varpi$ countable (and $N$ killed by $\varpi$).
  We set $S:=A^\circ$.
  We claim that with these assumptions, $M/\varpi$ is a countable flat $S/\varpi$-module.
  Assuming this, as in the proof of \cref{xs2hd9}, we can use Lazard's theorem to write $M/\varpi$ as a countable filtered colimit of finite free $S/\varpi$-modules to see that $\mathrm{Ext}^i_{\ob{D}(S)}(M,N)\cong \mathrm{Ext}^i_{\ob{D}(S/\varpi)}(M/\varpi,N)=0$ for $i>1$; thus we are done.

  As $M/\varpi$ is countable, by assumption, it suffices to see that $M$ is $\varpi$-completely flat, i.e., for any (discrete, static) $S/\varpi$-module $K$, the (derived) tensor product $M\otimes_S K$ is concentrated in degree $0$.
  As $M\otimes_S K$ is an object of the classical derived category $\ob{D}^\delta(S/\varpi)$, this claim can be checked on stalks on $\ob{Spec}(S/\varpi)$.
  Given $x\in \ob{Spec}(S/\varpi)\cong \mathcal{M}(A)$, let $S_x$ be the filtered colimit of rational neighborhoods of $x$.
  As $A$ is totally disconnected, this colimit is a colimit of direct summands of $S$.
  Moreover, the $\varpi$-adic completion of $S_x$ is a valuation ring $T_x$ with pseudo-uniformizer $\varpi$.
  We can conclude that the tensor product $M\otimes_S S_x$ is a static, $\varpi$-torsion free $S_x$-module, and so its $\varpi$-adic completion $M_x$ is a static and $\varpi$-torsion $T_x$ module (which is still discrete mod $\varpi$).
  Because $T_x$ is a valuation ring, we can conclude that $M_x$ is a flat $T_x$-module. Let $(S/\varpi)_x$ be the Zariski localization of $S/\varpi$ at $x$. Then
  \[
    (M\otimes_S K)\otimes_{S/\varpi} (S/\varpi)_x\cong M_x\otimes_{T_x} (S/\varpi)_x
  \]
  is therefore concentrated in degree $0$ as desired.
\end{proof}

Our next goal is to show that for a separable perfectoid ring $A$ there is, under a finiteness condition on the Berkovich space $\mathcal{M}(A)$, a pro-\'etale $!$-cover $A\to A'$ where $A'$ is a  separable strictly totally disconnected perfectoid ring, so that the previous results can be extended beyond the totally disconnected case. One cannot hope to simply drop the finite-dimensionality assumption, see \Cref{sec:-covers-perfectoids-1-counter-example-to-cover-by-strictly-totally-disconnected} below.

\begin{lemma}\label{LemmaRationalLocalizationAlmost}
Let $A$ be a perfectoid ring and let $A\to B$ be a rational localization. Then the map $\A_{\inf}(A) \to \A_{\inf}(B)$ of solid almost analytic rings is idempotent, i.e., $\A_{\inf}(B)$ is an idempotent algebra in the almost category $\ob{D}^a(\A_{\inf}(A))$ of \cref{DefAlmostAinfCat}. Furthermore, if $\{\Marc(B_i)\to \Marc(A)\}_{i=1}^n$ is a rational cover of $A$, then the collection of almost  idempotent algebras $\{\A_{\inf}(A) \to  \A_{\inf}(B_i) \}_{i=1}^n$ is a closed cover of the smashing spectrum of $\ob{D}^a(\A_{\inf}(A))$, equivalently, the map $\A_{\inf}(A)\to \prod_{i=1}^d \A_{\inf}(B_i)$ is descendable in $\ob{D}^a(\A_{\inf}(A))$. 
\end{lemma}
\begin{proof}

We can assume without loss of generality that $A$ is a perfectoid ring over $\F_p$. We need to show the following two facts: 

\begin{enumerate}

\item Let $A\to B$ be a rational localization. Then the map $\A_{\inf}(A)\to \A_{\inf}(B)$ is idempotent. 

\item  Let $\{\Marc(B_i)\to \Marc(A)\}_{i=1}^n$ be a rational cover of $A$, then the collection of almost idempotent algebras $\{\A_{\inf}(A)\to \A_{\inf}(B)\}_{i=1}^n$ is a  closed cover of the smashing spectrum of $\ob{D}^a(\A_{\inf}(A))$.

\end{enumerate}

We first show (1). We want to see that the map
\[
\A_{\inf}(B)\otimes_{\A_{\inf}(A)} \A_{\inf}(B)\to \A_{\inf}(B)
\]
is an almost equivalence.  As the solid tensor product preserves connective derived $p$-complete objects \cite[Proposition 2.12.10]{mann2022p}, by derived Nakayama's lemma it suffices to prove modulo $p$, in which case we need to show that $g\colon B^{\circ}\otimes_{A^{\circ}} B^{\circ}\to B^{\circ}$ is an almost equivalence.  The left hand side is a tensor of perfect solid $\F_p$-algebra, so it is static by \cite[Lemma 3.16]{bhatt_scholze_projectivity_of_the_witt_vector_affine_grassmannian} (more precisely, all the values at profinite sets are perfect $\F_p$-algebras and so static) and therefore integrally perfectoid. Since the map $g$ is  an equivalence in generic fibers, it is an almost equivalence  integrally and we are done.

 We next show part (2). Let $\{A\to B_i\}_{i=1}^n$ be a rational cover of $\mathcal{M}(A)$, for $J\subset \{1,\ldots, n\}$ a finite subset let $B_J=\bigotimes_{j\in J} B_i$ where the tensor product is taken over $A$. To see that $\{\A_{\inf}(A))\to \A_{\inf}(B_i)\}_{i=1}^n$ is a closed cover in the smashing spectrum of $\ob{D}^a(\A_{\inf}(A))$, it suffices to show that  the natural map 
\[
\A_{\inf}(A)\to  \varprojlim_{J\in \{1,\ldots, n\}} \A_{\inf}(B_J)
\]
is an almost equivalence. Taking quotients mod $p$ this reduces to almost arc-descent of $\shf{O}^{\circ}$ for affinoid perfectoid  spaces.
\end{proof}

\begin{lemma}\label{LemAnDescentPerfectoids}
  Let  $A\to A^\prime$ be  a pro-finite \'etale cover  of separable perfectoid Tate rings.
  Then $\A_{\inf}(A)\to \A_{\inf}(A^{\prime})$ is descendable in $\ob{D}^a(\A_{\inf}(A))$ of index of descendability $\leq 16$.
\end{lemma}
\begin{proof}
  By arguing as in point (b) of the proof of \cite[Theorem 2.30]{AMdescendPerfd}, it suffices to show that for all $k\in \N$ the map $\A_{\inf}(A)/p^k\to \A_{\inf}(A^{\prime})/p^k$ is descendable of  index of descendability $\leq 8$.
  Then, by applying the same argument, it suffices to show that for any pseudo-uniformizer $\pi\in A^{\circ\circ}$, one has to show that $\A_{\inf}(A)/(p^k,[\pi])\to \A_{\inf}(A')/(p^k,[\pi])$ is descendable of index of descendability $\leq 4$.
  By hypothesis, we can write $A^{\prime}=\varinjlim_{n} A_n$ as a filtered colimit of finite \'etale covers $A\to A_n$.
  Since $\A_{\inf}(A^{\prime})/(p^k,[\pi])=\varinjlim_n \A_{\inf}(A_n)/(p^k,[\pi])$, by applying   \cite[Proposition 2.7.2]{mann2022p} it suffices to show that if $A\to A'$ is a finite \'etale cover then $\A_{\inf}(A)/(p^k,[\pi])\to \A_{\inf}(A^{\prime})/(p^k,[\pi])$ is descendable of index of descendability $\leq 2$, and by base change even just that $\A_{\inf}(A)/p^k\to \A_{\inf}(A^{\prime})/ p^k$ has index of descendability $\leq 2$.

  Thus, we can assume that $A\to A'$ is a finite \'etale cover.
  We claim that the quotient $Q=\ob{cofib}(\A_{\inf}(A)\to \A_{\inf}(A^{\prime}))$ is almost finite projective. 
  Suppose that this holds, then $Q/p^k$ is almost finite projective $\A_{\inf}(A)/p^k$-module for all $k$, and for any pseudo-uniformizer $\pi\in A^{\circ\circ}$ the surjective map $\A_{\inf}(A^{\prime})/p^k\to Q/p^k$ admits a splitting up to multiplication by $[\pi]$.
  Equivalently, the map $\iota\colon \A_{\inf}(A)/p^k\to \A_{\inf}(A^{\prime})/p^k$ admits a $\pi$-retract, that is, there is a map $g_{\pi}\colon \A_{\inf}(A^{\prime})/p^k \to\A_{\inf}(A)/p^k$ for which the composite with $\iota$ is the multiplication by $[\pi]$ on $\A_{\inf}(A)/p^k$.
  Suppose the claim holds, and let $K=\ob{fib}(\A_{\inf}(A)\to \A_{\inf}(A^{\prime}))$, then the map $K\to \A_{\inf}(A)/p^k$ is killed by $[\varpi]$ for all $\pi\in A^{\circ\circ}$.
  Writing $W(A^{\circ\circ})\otimes_{\A_{\inf}(A)}K$  as the $p$-completed filtered colimit of $([\pi]K)_{\pi\in A^{\circ\circ}}$ we have a short exact sequence 
 \[
 \begin{gathered}
  0\to  R^1\varprojlim_{\pi\in A^{\circ\circ}}  \pi_1\big(\Hom_{\A_{\inf}(A)/p^k }(([\pi]K)/p^k, \A_{\inf}(A)/p^k )  \to \pi_0(\Hom_{\A_{\inf}(A)/p^k }((W(A^{\circ\circ})\otimes_{\A_{\inf}(A)} K)/p^k, \A_{\inf}(A)/p^k ) \\  \to \varprojlim_{\pi\in A^{\circ\circ}}  \pi_0(\Hom_{\A_{\inf}(A)/p^k }(([\pi]K)/p^k, \A_{\inf}(A)/p^k )  \to 0,
 \end{gathered}
 \]
 where the right limit is not derived. The image of $\eta$ in the right term is zero (this is precisely the ``$[\pi]$-split for each $\pi$'' property discussed before), so it comes from the left $R^1\varprojlim$ term. Thus, by \cite[Lemma 2.7.1]{mann2022p} the map $\eta^{\otimes 2}\colon K^{\otimes_{\A_{\inf}(A)} 2}\to \A_{\inf}(A)$ is zero  modulo $p^k$, proving that $\A_{\inf}(A)/p^k\to \A_{\inf}(A^{\prime})/p^k$ is $(\leq 2)$-descendable as wanted.

 We now prove the claim, that is, that $Q=\ob{cofib}(\A_{\inf}(A)\to \A_{\inf}(A^{\prime}))$ is an almost finite projective $\A_{\inf}(A)$-module, equivalently that it is almost $p$-completely flat and almost finitely generated.
 By derived Nakayama's lemma (and the definition of $p$-complete flatness) this reduces to proving that the derived quotient $Q/p=\ob{cofib} (A^{\circ}\to A^{\prime, \circ})$ is an almost finite projective $A^{\circ}$-module (and thus in particular almost flat), and this follows from the almost purity theorem \cite[Theorem 10.9]{Bhatta}. Indeed, by almost purity we know that $Q/p$ is almost finitely presented, and that  for all pseudo-uniformizer $\pi$ we can find a $\pi$-retract $A^{\prime, \circ}\to A^{\circ}$, this is the same as having a $\pi$-section of $A^{\prime,\circ}\to Q/p$, proving that $Q/p$ is also almost flat.  
\end{proof}

\begin{proposition}\label{PropDescendableCoverFiniteDimPerfectoid}
Let $A$ be a separable affinoid perfectoid ring such that $\mathcal{M}(A)$ has finite cohomological dimension. Then there is a pro-\'etale descendable map of separable perfectoid rings $A\to A^{\prime}$  such that $\A_{\inf}(A)\to \A_{\inf}(A^{\prime})$ is descendable in $\ob{D}^a(\A_{\inf}(A))$ and $A^{\prime}$ is strictly totally disconnected. Moreover, let $g\colon A\to B$ be an arbitrary arc-cover with $B$ separable perfectoid, then $\A_{\inf}(A)\to \A_{\inf}(B)$ is descendable. By base change along $\A_{\inf}(A)\to A$, the map $A\to B$ is also descendable.
\end{proposition}
\begin{proof}
Let $A$ be as in the statement and let $\pi$ be a pseudo-uniformizer of $A$. By \cref{LemmaRationalLocalizationAlmost}, the functor mapping a closed rational open subspace $C\subset \mathcal{M}(A)$ to $\A_{\inf}(C)/(p^k,[\pi])$ is a sheaf  of idempotent algebras in the almost category $\ob{D}^a(\A_{\inf}(A)/(p^k,[\pi]))$. By \cite[Theorem A]{aoki2023sheavesspectrumadjunction} this functor produces a morphism of locales
\[
\psi\colon \ob{Sm}(\ob{D}^a(\A_{\inf}(A)/(p^k,[\pi])))\to \mathcal{M}(A),
\]
(here working modulo $(p^k,[\pi])$ is important for the compatibility with colimits of idempotent algebras).

Consider the algebra $A^{w}$ of \cref{xk29sm} and its uniform completion $A^{w,u}$.
It is separable totally disconnected perfectoid, and sits in a pullback diagram of arc-stacks
\begin{equation}\label{eq9maskd3}
\begin{tikzcd}
\Marc(A^{w,u}) \ar[r] \ar[d] & \underline{\mathcal{M}(A^{w,u})} \ar[d,"f"] \\ 
\Marc(A)   \ar[r, "g"] & \underline{\mathcal{M}(A)} 
\end{tikzcd}
\end{equation}
where the right terms are the compact Hausdorff Berkovich spectra  seen as arc-stacks.
As $\mathcal{M}(A) $ has finite cohomological dimension $d$, the  algebra $f_* 1\in \ob{D}(\mathcal{M}(A) , \Z)$ is descendable of index only depending on $d$, see \cref{xhsw82h}.
Pulling  back along $\psi$, and thanks to the pullback square \cref{eq9maskd3}, one has that $\psi^* f_* 1 = \A_{\inf}(A^{w,u})/(p^k,[\pi])$.
We deduce that the map $\A_{\inf}(A)/(p^k,[\pi])\to \A_{\inf}(A^{w,u})/(p^k,[\pi])$ is descendable in $\ob{D}^a(\A_{\inf}(A)/(p^k,[\pi]))$ with index of descendability only depending on $d$.
Taking limits on $p^k$ and $[\pi]$, and using the same argument as \cite[Theorem 2.30]{AMdescendPerfd}, we get that $\A_{\inf}(A)\to \A_{\inf}(A^{w,u})$ is descendable with index only depending on $d$.
Finally, by applying   \cref{sec:light-arc-stacks-1-cover-by-separable-totally-disconnected} (2) and  \cref{LemAnDescentPerfectoids}, we can produce a pro-finite \'etale  strictly totally totally disconnected cover  $A^{w,u} \to A^{'}$ such that $\A_{\inf}(A^{w,u})\to \A_{\inf}(A^{\prime})$ is descendable (with index of descendability independent of $d$).  Taking the composition $A\to A^{w,u}\to A^{\prime}$ we get the desired cover.

Next, we prove that if $A\to B$ is an arc-cover by a separable totally disconnected ring, then the map $\A_{\inf}(A)\to \A_{\inf}(B)$ is descendable of index only dependent on $d$.
Indeed, this can be proven after base change along $\A_{\inf}(A)\to \A_{\inf}(A^{\prime})$ (with $A^{\prime}$ constructed as before), in which case the descendability follows from \cref{sec:perf-analyt-de-2-arc-covers-are-good}.
\end{proof} 

\begin{remark}
  \label{sec:-covers-perfectoids-1-counter-example-to-cover-by-strictly-totally-disconnected}\
  \begin{enumerate}
  \item Assume that $A$ is a separable perfectoid ring with a descendable map $A\to B$ with $B$ a separable strictly totally disconnected perfectoid.
  Using the Artin--Schreier sequence on the tilt and that each open subset of $\Spa(B,B^{\circ})$ is a countable filtered union of quasi-compact open subsets (and hence of affinoids), one sees that each closed subset $Z$ of $\mathcal{M}(B)$ (or rather its pullback to a perfectoid subspace of $\mathrm{Spa}(B,B^{\circ})$) has vanishing $\F_p$-cohomology in degrees $\geq 2$.
  From this (and Artin--Schreier for the tilt again), one deduces, from the descendability assumption, that the same holds true for closed subsets of $\mathcal{M}(A)$.
  This implies that a finiteness condition in \cref{PropDescendableCoverFiniteDimPerfectoid} is necessary, and that $A':=\Q_p^\cyc\langle T_1^{1/p^\infty},T_2^{1/p^\infty},\dots \rangle$ does not admit a $!$-cover by a strictly totally disconnected perfectoid ring (e.g., $\prod_\N S^1\subseteq \mathcal{M}(A')$ and $\prod_\N S^1$ has unbounded cohomological dimension).
\item In fact, let $A$ be a separable Gelfand ring, and let $A\to B$ be a descendable map such that $\mathcal{M}(B)$ is totally disconnected.
  As in the previous discussion, we see that there exists $d\geq 0$ such that for \emph{any} open subspace $U\subseteq \mathcal{M}(A)$ the quasi-coherent cohomology (more precisely, the pushforward to $\GSpec(\Q_p)$) of the Gelfand stack $U_{\Betti}\times_{\mathcal{M}(A)_{\Betti}} \GSpec(A)$ is bounded by $d$.
  This in turn implies that the ring $A':=\Q_p\langle T_1,T_2,\ldots \rangle$ does not admit any descendable morphism $A'\to B'$ with $\mathcal{M}(B')$ profinite.
  Namely, for $n\geq 0$ we can consider the open space $V_n\subseteq Z_n:=\GSpec(\Q_p\langle T_1,\ldots, T_n\rangle)$, which is the complement of $\GSpec(\mathcal{O}_{Z_n,0})$ with $0\in |Z_n|$ the origin.
  Then the open subspace $U_n:=V_n\times_{\GSpec(\Q_p)}\GSpec(\Q_p\langle T_{n+1},\ldots \rangle)\subseteq Z_\infty:=\GSpec(\Q_p\langle T_1,\ldots \rangle)$ has cohomology in degree $n-1$ (as follows by an explicit computation), and this goes to $\infty$ if $n\to \infty$.
  \end{enumerate}
\end{remark}

  In the rest of this subsection, we make a small detour to explain the relation between the previous descent results obtained for perfectoid rings and the works on $6$-functor formalisms for quasi-coherent sheaves on Fargues--Fontaine curves of \cite{AMdescendPerfd} and \cite{ALBMFFCoho}.
  
  Let $\Cat{Perf}^{\ob{td}}_{\omega_1}$ be the category of separable totally disconnected perfectoid spaces over $\F_p$.
  Thanks to \cref{sec:perf-analyt-de-2-arc-covers-are-good} the functor
\[
\ob{D}^{a,!}(-,\A_{\inf})\colon  \Cat{Perf}^{\ob{td},\op}_{\omega_1}\to \Cat{Pr}^L
\]
sending $A\mapsto \ob{D}^a(\A_{\inf}(A))$, and transition maps given by upper $!$-maps, satisfies hyperdescent, and thus $\ob{D}^{a}(\A_{\inf}(-))$ also satisfies $*$-hyperdescent.
Hence, one can extend this functor via techniques as in \cite{ALBMFFCoho} to a $6$-functor formalism on light arc-stacks 
\begin{equation}\label{eq03las1j3}
\ob{D}^{a}(-,\A_{\inf})\colon \Cat{Corr}(\Cat{ArcStk}_{\F_p}, E)\to \Cat{Pr}^L
\end{equation}
where $E$ is a suitable class of arrows as in \cite[Theorem 3.4.11]{heyer20246functorformalismssmoothrepresentations} (including all maps of arc-stacks arising from rigid analytic varieties).

 Taking base change along $\A_{\inf}\to \widehat{\mathcal{O}}$ we recover a light version of the descent results of \cite{AMdescendPerfd}, up to the cost of working only with light solid sheaves on partially proper spaces.
One can also recover the $6$-functor formalism of quasi-coherent sheaves on $\mathcal{Y}_{[0,\infty)}$ of \cite{ALBMFFCoho}, or rather a light version restricted partially proper spaces, as follows. If $A$ is a separable perfectoid ring over $\F_p$ with pseudo-uniformizer $\varpi$ we define
\[
\Yc_{A}= \Spa(\A_{\inf}(A))\backslash V(p[\varpi]) 
\]
(see \cite[Section II.1]{fargues2021geometrization}), and regard it as a Berkovich space over $\Q_p$ (this is possible as $\Yc_{A}$ is partially proper as an adic space).
We can write then $\Yc_{A}$ as the union of Gelfand stacks $\GSpec \mathbb{B}_{[r,s]}(A)$ for $[r,s]\subset [0,\infty)$ an interval with rational ends, and $\mathbb{B}_{[r,s]}(A) = \A_{\inf}(A)\langle \frac{p}{[\varpi]^{1/s}}, \frac{[\varpi]^{1/r}}{p} \rangle[\frac{1}{[\varpi]}]$ (with the convention $[\varpi]^{1/r}/p=0$ if $r=0$).
 We have the following lemma: 
 
 \begin{lemma}\label{LemAlmostBIs}
 Let $A$ be a separable perfectoid $\F_p$-algebra with pseudo-uniformizer $\varpi$.
 Let $\ob{D}^a(\A_{\inf}(A))$ be the almost category of solid $\A_{\inf}(A)$-modules from \cref{DefAlmostAinfCat}.
  Then for $[r,s]\subset [0,\infty)$ the algebras $\mathbb{B}_{[r,s]}(A) \in \Cat{CAlg}(\ob{D}^a(\A_{\inf}(A)))$ are idempotent and we have a natural equivalence
 \[
 \ob{D}(\mathbb{B}_{[r,s]}(A)) = \ob{Mod}_{\mathbb{B}_{[r,s]}(A)}(\ob{D}^a(\A_{\inf}(A))),
 \] 
 where the left hand side is the category of $\Z$-solid $\mathbb{B}_{[r,s]}(A)$-modules. 
 \end{lemma}
 \begin{proof}
For the idempotency, it suffices to show that the map of solid algebras
\begin{equation}\label{eqIdempotentAinf}
\A_{\inf}(A)[\frac{1}{[\varpi]}]\to \mathbb{B}_{[0,s]}(A)=\A_{\inf}\langle \frac{p}{[\varpi]^{1/s}}\rangle[\frac{1}{\varpi}]
\end{equation}
is idempotent.
Indeed, since $\mathbb{B}_{[0,s]}$ is sousperfectoid it is a sheafy Tate algebra (with pseudo-uniformizer $[\varpi]$).
Thus, any rational localization $\mathbb{B}_{[0,s]}\to C$, in the sense of adic spaces, is idempotent for the relative solid tensor product (e.g.\ by \cite[Proposition 4.12]{andreychev2021pseudocoherent}, \cite[Proposition 2.3.22 (ii)]{mann2022p}).
Moreover, as the Banach module $C$ is nuclear over $\mathbb{B}_{[0,s]}$, the tensor product $C\otimes_{\mathbb{B}_{[0,s]}}C$ is already solid with respect to $C^+$ and thus we see that $C$ is idempotent as a $(\mathbb{B}_{[0,s]},\Z)_{\solid}$-algebra.

It is left to see that \eqref{eqIdempotentAinf} is idempotent.
By an explicit calculation, e.g., using the presentation in \cite[Proposition 11.2.1]{scholze2020berkeley}, we have a short exact sequence 
\[
0\to \A_{\inf}(A)\langle T\rangle [\frac{1}{[\varpi]}]   \xrightarrow{[\varpi]^{1/s}U-p} \A_{\inf}(A)\langle T\rangle [\frac{1}{[\varpi]}] \to \mathbb{B}_{[0,s]} \to 0
\]
where $\A_{\inf}(A)\langle T\rangle$ is the $(p,[\varpi])$-adic completion of the polynomial algebra over $\A_{\inf}(A)$. Thus, as the solid tensor product preserves $(p,[\varpi])$-complete connective objects, see \cite[Proposition 2.12.10]{mann2022p},  one easily deduces that 
\[
\A_{\inf}(A)\langle T\rangle\otimes_{\A_{\inf}(A)} \mathbb{B}_{[0,s]}= \mathbb{B}_{[0,s]}\langle T \rangle,
\]
with $\mathbb{B}_{[0,s]}\langle T \rangle$ the Tate algebra in one generator over $\mathbb{B}_{[0,s]}$. We deduce that the derived tensor product $\mathbb{B}_{[0,s]} \otimes_{\A_{\inf}(A)[\frac{1}{[\varpi]}]} \mathbb{B}_{[0,s]} $ is represented by the complex in homological degrees $[0,1]$ 
\[
\mathbb{B}_{[0,s]}\langle T\rangle \xrightarrow{[\varpi]^{1/s}T-p}   \mathbb{B}_{[0,s]}\langle T\rangle, 
\]
which is equal to $\mathbb{B}_{[0,s]}$ as $\frac{p}{[\varpi]^{1/s}}\in \mathbb{B}_{[0,s]}^{\circ}$.

 The claim about the almost category is easy since one can compute that $\mathbb{B}_{[r,s]}(A)\otimes_{\A_{\inf}(A)} W(\overline{A})=0$ for the solid tensor product (e.g. $[\varpi]$ is zero in $W(\overline{A})$ but a unit in $\mathbb{B}_{[r,s]}(A)$).
 \end{proof}
   
\begin{corollary}\label{CorDescentYConstruction}
The following hold: 
\begin{enumerate}
\item 
Let $A\to A_{\bullet}$ be an arc-hypercover of separable totally disconnected perfectoid $\F_p$-algebras (with fixed pseudo-uniformizer $\varpi$), then for all $[r,s]\subset (0,\infty)$ rational closed interval, the diagram $\mathbb{B}_{[r,s]}(A)\to \mathbb{B}_{[r,s]}(A_{\bullet})$ of Gelfand $\Q_p$-algebras is a $!$-equivalence.
 In particular, the functor $A\mapsto \Yc_{A}$ sends arc-covers of separable totally disconnected perfectoid $\F_p$-algebras to epimorphisms of Gelfand stacks. 

\item Let $\Yc^{\mathrm{arc}}\colon \Cat{ArcStk}_{\F_p}\to \Cat{GelfStk}$ be the extension of the functor $A \mapsto \Yc_A$ to light arc-stacks obtained by descent.
 Then, if $A$ is a separable perfectoid $\F_p$-algebra with $\mathcal{M}(A)$ finite dimensional, there is a natural equivalence of Gelfand stacks $\Yc^{\mathrm{arc}}_A\xrightarrow{\sim} \Yc_A$.

\end{enumerate}
\end{corollary}   
   \begin{proof}
Part (1) follows formally from \cref{LemAlmostBIs} after base changing along $\A_{\inf}(A)\to \mathbb{B}_{[r,s]}(A)$ the descendability results of \cref{sec:perf-analyt-de-2-arc-covers-are-good}.
 Part (2) follows by base change from the descendability results of \cref{PropDescendableCoverFiniteDimPerfectoid}. 
\end{proof}

\begin{remark}\label{RemDescentYFunctor}\
\begin{enumerate}
 \item The functor $\Yc^{\arc}$ sends $\F_p$ to $\GSpec(\Q_p)$ (cf.\ \cite[Theorem 6.3.1]{ALBMFFCoho}).
To see this, we claim first that the natural map
\begin{equation}\label{eq0k3lqms}
\Yc_{\Marc(K)\times_{\F_p} \Marc (K)}\to\Yc_{K}\times_{\GSpec(\Q_p)} \Yc_{K}
\end{equation}
is an isomorphism. 
 The fiber product $\Marc(K)\times_{\F_p} \Marc(K)$  identifies with the open punctured perfectoid disc $\mathring{\DD}^{\diamond, \times}_{K,\infty}$.
 On the other hand, one can identify $\Yc_{K}=\mathring{\DD}^{u,\times}_{\Q_p,\infty}$ with the open punctured perfectoid disc in the variable given by the Teichm\"uller $[\varpi]$ (where $u$ denotes the uniform completion of the Berkovich space).
 From this description, one easily deduces that
\[
\Yc_{\Marc(K)\times_{\F_p} \Marc (K)}\cong \Yc_{\mathring{\DD}^{\diamond,\times }_{K,\infty}}=\Yc_{K}\times_{\GSpec(\Q_p)} \mathring{\DD}^{u,\times}_{\Q_p,\infty} \cong \Yc_{K}\times_{\GSpec(\Q_p)} \Yc_K,
\]
and that this isomorphism is the one of \cref{eq0k3lqms}, proving the claim. Then one concludes using descent along the \v{C}ech nerve for $\Marc(K)$ for $K=\F_p((t^{1/p^\infty}))$.

\item One can check that the functor $\Marc(A)\mapsto \Yc_A$ preserves fiber products on perfectoid rings (and hence on perfectoid spaces).
  Using the previous discussion and base change along the covering $\Marc(K)\to \Marc(\F_p)$ shows that the functor $\Yc^{\arc}\colon \Cat{ArcStk}_{\F_p}\to \Cat{GelfStk}$ preserves finite products, and is thus left exact.

\end{enumerate}
\end{remark}

The functor $X\mapsto \ob{D}(\Yc^{\arc}_X)$ on $1$-truncated light arc-stacks is then a light variant of the $6$-functor formalism $S\mapsto \ob{D}_{(0,\infty)}(S)$ of \cite[Remark 4.2.4]{ALBMFFCoho}, but restricted to partially proper (and light) objects (and a different class of $!$-able maps\footnote{Strictly speaking, one needs to compare the $!$-functors, but for this one can use \cite[Theorem 1.3]{dauser2024uniqueness}.}).
 
 Notice that, in particular, we had no need to modify the category of solid quasi-coherent sheaves in order to guarantee arc-descent, as in \cite{{AMdescendPerfd}}.
 In fact, as we generally induce analytic ring structures from $\Q_{p,\solid}$ the modification is not necessary (\cite[Lemma 2.9]{AMdescendPerfd}).
 Nevertheless, the $6$-functor formalism of \cite{ALBMFFCoho} is more refined and in particular it can be evaluated in arbitrary adic spaces.
 For applications in $p$-adic geometry (such as Poincar\'e duality for smooth maps between partially proper rigid spaces as proven in \cite{colmez2024duality}, \cite{ALBMFFCoho}) and the $p$-adic Langlands program, both $6$-functor formalisms are equally good since all the relevant spaces are usually light and partially proper.
 An advantage of the $6$-functor formalism discussed here that it is directly related with the six functors of analytic stacks, e.g., making it easier to compare it to quasi--coherent sheaves on Hyodo--Kato stacks.

\subsection{Perfectoidization and the big de Rham stack}
\label{sec:perf-analyt-de-1}

 Given a separable Gelfand ring $A$, we let $\Marc(A)$ denote the arc-prestack corepresented by the ring $A$ in totally disconnected perfectoid rings. We note that, since arc-hypercovers of totally disconnected perfectoid rings are $!$-equivalences (\cref{sec:perf-analyt-de-2-arc-covers-are-good}), $\Marc(A)$ is actually an arc-stack. A major motivation for the use of Gelfand stacks is the perfectoidization, that we can define now.

\begin{lemma}
  \label{sec:totally-disc-stacks-construction-perfectoidization}\
  \begin{enumerate}

  \item There exists a unique morphism of $\infty$-topoi $\Cat{ArcStk}_{\Q_p} \to \Cat{GelfStk}$, whose left adjoint $(-)^\diamond$ sends $\ob{GSpec}(A)$, for $A\in \Cat{GelfRing}_{\omega_1}$, to $(\ob{GSpec}(A))^\diamond:= \Marc(A)$.

  \item If $A\in \Cat{GelfRing}_{\omega_1}$ the natural map $\Marc(A^u)\to \Marc(A)$ is an isomorphism.
  \end{enumerate}
\end{lemma}
 We call $(-)^\diamond$ the \textit{perfectoidization} functor.
\begin{proof}
  The last assertion is clear as each separable perfectoid $\Q_p$-algebra $B$ is uniform, so that any $A\to B$ factors over $A^u\to B$.
  We claim that the functor $A\mapsto \Marc(A)$ preserves terminal objects, fiber products and (its extension to presheaf topoi) sends $!$-equivalences to $\infty$-connective maps in the arc-topos. Given the claim the first assertion is formal.
  Preservation of terminal objects/fiber products is clear.
  To see that $!$-equivalences map to isomorphisms, we can show the stronger claim that hypercovers for the $!$-topology map to hypercovers for arc-stacks, and hence to isomorphisms in $\Cat{ArcStk}_{\Q_p}$ (indeed, any $\infty$-connective morphism, like a $!$-equivalence, cf., \Cref{sec:gelfand-stacks-remark-description-of-shriek-equivalences}, is a geometric realization of a hypercover, \cite[Lemma 6.5.3.5]{lurie_higher_topos_theory}).
  To check this stronger claim it suffices to check that $!$-covers map to arc-covers (as the functor $A\mapsto \Marc(A)$ preserves finite limits).
  Let $A\to B$ be a $!$-cover of Gelfand rings.
  It suffices to see that the map of Berkovich spaces $f\colon \mathcal{M}(B)\to \mathcal{M}(A)$ is surjective.
  Let $x\in \mathcal{M}(A)$, and let $A\to \kappa(x)$ be the  residue field of $A$ at $x$.
  The fiber of $f$ at $x$ is the Berkovich spectrum of the ring $B\otimes_A \kappa(x)$, since $A\to B$ is a $!$-cover, base change is conservative and $B\otimes_{A} \kappa(x)\neq 0$, then \cref{LemmaVanishingSpectrum} implies that $f^{-1}(x)\neq \emptyset$ and so $f$ is surjective. By stability of $!$-equivalences under diagonals, it then follows that they are sent to $\infty$-connective maps.
   \end{proof}

\begin{remark}\label{RemDiamondBerkovichSpace}
Let $Y$ be a derived Berkovich space, we can write $Y$ as an union $Y=\bigcup_{i\in I} Y_i$ where $Y_i$ are affinoid Berkovich spaces, and the transition maps are monomorphisms. Passing to perfectoidizations, we see that $Y^{\diamond}=\bigcup_i Y_i^{\diamond}$. But each $Y_i^{\diamond}$ is also equal to the diamond attached to a Banach ring (by taking uniform completions). This implies that $Y^{\diamond}$ is actually an  arc-sheaf of \emph{sets} and not just of anima.
\end{remark}

The right adjoint to the functor $(-)^\diamond$ is a version of the analytic de Rham stack.
As explained in the introduction to this section, we want to keep the name ``de Rham stack'' for a qfd variant.
We thus give it a different name.

\begin{definition}
  \label{sec:totally-disc-stacks-definition-de-rham-stack}\
  \begin{enumerate}
  \item  We define the \textit{big de Rham stack}
  \[
    (-)^{\dRall}\colon \Cat{ArcStk}_{\Q_p}\to \Cat{GelfStk}
  \]
   as the right adjoint of the functor $(-)^\diamond$ from \cref{sec:totally-disc-stacks-construction-perfectoidization}.
\item Given $X\in \Cat{GelfStk}$, we abuse notation and set $X^\dRall:=(X^\diamond)^\dRall$.

  \item More generally, if $Y\to X$ is a morphism in $\Cat{GelfStk}$, then we set $Y^{\dRall/X}:=Y^{\dRall}\times_{X^\dRall}X$, and call it the \textit{relative big de Rham stack} of $Y$ over $X$.   \end{enumerate}

\end{definition}

If $X\in \Cat{GelfStk}$ and $A\in \Cat{GelfRing}_{\omega_1}$, then by definition
\[
  \Hom_{\Cat{GelfStk}_{\Q_p}}(\ob{GSpec}(A),X^\dRall)\cong \Hom_{\Cat{ArcStk}_{\Q_p}}(\Marc(A), X^\diamond),
\]
which is different from $X(A^{\dagger-\mathrm{red}})$ or $X(A^u)$ in general.
The situation improves if we assume that $A^u$ is a separable totally disconnected perfectoid Tate ring.

\begin{proposition}
\label{prop:formula-functor-of-points-drall}
Let $X\in \Cat{GelfStk}$ and $A\in \Cat{GelfRing}_{\omega_1}$ with $A^u$ totally disconnected perfectoid. We have
\[
X^{\dRall}(A) =X(A^u).
\]
\end{proposition}
\begin{proof}
Indeed, by \cref{sec:totally-disc-stacks-construction-perfectoidization}, $\Marc(A)\cong \Marc(A^u)$, and
\cref{sec:perf-analyt-de-2-left-adjoint-to-perfectoidization} below implies that
\[
  \Hom_{\Cat{ArcStk}_{\Q_p}}(\Marc(A^u),X^\diamond)\cong \Hom_{\Cat{GelfStk}}(\ob{GSpec}(A^u),X)=X(A^u),
\]
as desired.
\end{proof}

\begin{proposition}
  \label{sec:perf-analyt-de-2-left-adjoint-to-perfectoidization}
The functor $(-)^\diamond\colon \Cat{GelfStk}\to \Cat{ArcStk}_{\Q_p}$ admits a left adjoint
$$\widehat{(-)}\colon \Cat{ArcStk}_{\Q_p}\to \Cat{GelfStk},$$ sending some $\Marc(A)$ with $A$ a separable totally disconnected perfectoid Tate ring, to $\ob{GSpec}(A)$. Furthermore, the restriction 
\[
\widehat{(-)}\colon \Cat{ArcStk}_{\Q_p}\to \Cat{GelfStk}_{/\widehat{\Marc}(\Q_p)}
\]
is the left adjoint of a morphism of $\infty$-topoi $\Cat{GelfStk}_{/\widehat{\Marc}(\Q_p)}\to \Cat{ArcStk}_{\Q_p}$.

\end{proposition}
\begin{proof}
  We note that separable totally disconnected perfectoid Tate rings over $\Q_p$ form a basis of $\Cat{ArcStk}_{\Q_p}$ (by \cref{sec:light-arc-stacks-1-cover-by-separable-totally-disconnected}).
  It suffices to see that on separable totally disconnected perfectoid rings the functor $\Marc(A)\mapsto \ob{GSpec}(A)$ sends arc-hypercovers to $!$-equivalences.
  This is the content of \cref{sec:perf-analyt-de-2-arc-covers-are-good}.
\end{proof}

\begin{example}
  \label{sec:perf-analyt-de-2-example-for-analytic-diamond}
  By commutation with colimits of $\widehat{(-)}$,  the Gelfand stack $X:=\widehat{\Marc}(\Q_p)$ admits the presentation $\ob{GSpec}(\C_p)/\Gamma$, where $\Gamma:=\ob{GSpec}(C(\mathrm{Gal}_{\Q_p},\Q_p))$ acts on $\C_p$ via the natural continuous action of the Galois group (here $C(\mathrm{Gal}_{\Q_p},\Q_p)$ denotes the $\Q_p$-Banach algebra of continuous functions on $\mathrm{Gal}_{\Q_p}$).
  In particular, vector bundles on $X$ identify with $v$-bundles on $\Marc(\Q_p)$ (or, equivalently, $\mathrm{Spa}(\Q_p)$ in the usual notation for adic spaces), and consequently, $X\neq \ob{GSpec}(\Q_p)$. 
\end{example}

\begin{proposition}
  \label{sec:perf-analyt-de-2-de-rham-stacks-fully-faithful}
  The functors $(-)^\dRall, \widehat{(-)} \colon \Cat{ArcStk}_{\Q_p}\to \Cat{GelfStk}$ are fully faithful.
\end{proposition}
\begin{proof}
  We first show fully faithfulness of $(-)^\dRall$.
  It suffices to see that for any $X\in \Cat{ArcStk}_{\Q_p}$ the natural map $(X^\dRall)^\diamond\to X$ is an isomorphism, when evaluated on $\Marc(A)$ with $A$ separable and totally disconnected perfectoid.
  But, by \cref{sec:perf-analyt-de-2-left-adjoint-to-perfectoidization},
  \[
    (X^{\dRall})^\diamond(\Marc(A))=X^{\dRall}(\ob{GSpec}(A))=X(\Marc(A^u)),
  \]
  which equals $X(\Marc(A))$ as $A=A^u$ is uniformly complete.

  The fully faithfulness of $\widehat{(-)}$ is now formally implied by the fully faithfulness of its twice-right adjoint $(-)^\dRall$.
  Indeed, we need to see that for $X,Z\in \Cat{ArcStk}_{\Q_p}$ the natural map
  \[
    \Hom_{\Cat{ArcStk}_{\Q_p}}((\widehat{X})^\diamond,Z)\to \Hom_{\Cat{ArcStk}_{\Q_p}}(X,Z)
  \]
  is an isomorphism.
  The first Hom can be rewritten as $\Hom_{\Cat{ArcStk}_{\Q_p}}(X,(Z^{\dRall})^{\diamond})$, and so the claim follows because the natural map $(Z^{\dRall})^{\diamond}\to Z$ is an isomorphism as $(-)^\dRall$ is fully faithful.
\end{proof}

 In fact, \Cref{prop:formula-functor-of-points-drall} holds for a larger class of Gelfand rings. First note the following general lemma.

\begin{lemma}\label{LemPointsdeRham}
Let $X\in \Cat{GelfStk}$ be a Gelfand stack and $A$ a separable Gelfand ring such that 
\begin{itemize}

\item  $A^u$ is perfectoid. 

\item There is a $!$-cover $A^u\to A'$ with $A'$ a separable totally disconnected perfectoid ring.

\end{itemize} 

Then there is a natural equivalence of anima $X^{\dRall}(A) = X(A^u)$.

\end{lemma}
\begin{proof}
By adjunctions, we know that 
\[
X^{\dRall}(A) = X^{\diamond}(\Marc(A))=X^{\diamond}(\Marc(A^{u}))=X(\widehat{\Marc}(A^u)).
\]
Thus, it suffices to show that the natural map 
\[
\widehat{\Marc}(A^u)\to \ob{GSpec} A^u
\]
is an equivalence. For this, like in the proof of \Cref{sec:perf-analyt-de-2-left-adjoint-to-perfectoidization}, it suffices to show that any arc-hypercover of $A^u$ by separable totally disconnected rings is a $!$-equivalence. This can be tested after pulling back along $!$-covers on $A^u$, in particular along $A^u\to A'$, where it follows from \cref{sec:perf-analyt-de-2-arc-covers-are-good}. 
\end{proof}

\begin{example}\label{ExampleCondensedAnimanodeRham}\

\begin{enumerate}
\item Let $T$ be a condensed anima and let $A$ be a separable Gelfand ring. By \cref{sec:light-arc-stacks-1-condensed-anima-and-arc-stacks} we have that $T^{\dRall}(\GSpec(A))= \underline{T}(\Marc(A))=T(\mathcal{M}(A))$. In particular, if $T=X$ is a metrizable compact Hausdorff space, then $X^{\dRall}(\GSpec (A)) = \Map_{\Cat{Top}}(\mathcal{M}(A),X)$ is the set of \textit{all} continuous maps $\mathcal{M}(A)\to X$.

\item  Let $X$ be a metrizable Hausdorff space.  Then the natural map $X_{\Betti}\to X^{\dRall}$ (defined via colimits from the case of light profinite sets) is an immersion as the diagonal is an equivalence (it identifies with the overconvergent diagonal in $(X\times X)_{\Betti}$, and thus with $X_{\Betti}$), but it is not necessarily an isomorphism if $X$ is not finite dimensional. Indeed, let $A=\GSpec (\Q_p\langle T_{n} : n\in \N \rangle)$ be the Tate algebra in countably many variables and let $X=\mathcal{M}(A)$ be the Berkovich space of $A$. We claim that $X_{\Betti}\to X^{\dRall}$ is not an isomorphism. For this, it suffices to see that the map $\GSpec (A)\to X^{\dRall}$ does not factor through $X_{\Betti}$. To prove the claim, recall from  \cref{xj29sj} that the map $\GSpec(A)\to X^{\dRall}$ factors through $X_{\Betti}$ if and only if there is a $!$-cover $A\to B$ and a factorization
\[
\begin{tikzcd}
\GSpec (B) \ar[r] \ar[d] & S_{\Betti} \ar[d] \\
\GSpec (A)  \ar[r]& X^{\dRall}
\end{tikzcd}
\]
with $S$ a profinite set.  If this holds, then the map of compact Hausdorff spaces  $S\to X$ is surjective and the fiber product $\GSpec (A)\times_{X^{\dRall}} S_{\rm Betti}$ is represented by an affinoid Gelfand stack $\GSpec (C)$. Hence, we have a factorization
\[
\GSpec (B)\to \GSpec (C) \to \GSpec (A)
\]
making $A\to C$  a $!$-cover with Berkovich space $\mathcal{M}(C)=S$ a light profinite set.
The existence of such map would imply, by base change, that given $B$ a separable Gelfand ring, there is a surjection $S\to \mathcal{M}(B)$ from a profinite set such that the fiber product $\GSpec(B)\times_{\mathcal{M}(B)^{\dRall}} S_{\Betti}$ is corepresented by an algebra $D$ such that
\begin{enumerate}[(i)]
\item  $D=\varinjlim_{n} B_n$ is a sequential colimit of finite disjoint union of rational localizations of $D$.
\item $\mathcal{M}(D)=S$ is profinite.
\item $B\to D$ is descendable.
\end{enumerate}
This contradicts \cref{sec:-covers-perfectoids-1-counter-example-to-cover-by-strictly-totally-disconnected}.

\item Let $A$ be a separable Gelfand ring. Then any map $f\colon \GSpec (A)\to X^{\dRall}$ gives rise to a pullback map of derived categories $\widehat{\ob{D}}(X, \Z)\to \widehat{\ob{D}}(\mathcal{M}(A), \Z)\to \ob{D}(A)$  where the $\widehat{\ob{D}}$ refers to the left completion of the derived categories of abelian sheaves on the topological spaces.  In particular, if  $\mathcal{M}(A)$ is  profinite, any such pullback functor preserves connective objects and by \cref{xj29sj} $f$ factors through $X_{\Betti}$. In other words, if $A$ is a separable Gelfand ring such that $\mathcal{M}(A)$ is profinite then $X_{\Betti}(A)=X^{\dRall}(A)$. By $!$-descent, if $A$ satisfies the condition of \cref{LemPointsdeRham} one also has $X_{\Betti}(A)=X^{\dRall}(A)$.

\end{enumerate}

In conclusion, the difference between $X_{\Betti}$ and $X^{\dRall}$ is only witnessed by Gelfand rings that do not admit a $!$-cover to a Gelfand ring $A$ with profinite Berkovich spectrum. In the rest of this and next sections we will show that under a quasi-finite dimensional assumption, a Gelfand ring admits such a $!$-cover.
\end{example}

\begin{corollary}
  \label{sec:-covers-perfectoids-1-arc-versus-shriek-cover}
Let $X\in \Cat{GelfStk}$ be a Gelfand stack and let $A$ be a separable Gelfand ring such that $A^u$ is perfectoid and $\mathcal{M}(A)$ has finite cohomological dimension. Then there is a natural equivalence of anima $X^{\dRall}(A) = X(A^u)$.
\end{corollary}
  \begin{proof}
This is a   direct consequence of \Cref{PropDescendableCoverFiniteDimPerfectoid} and \Cref{LemPointsdeRham}.
\end{proof}

We now define the qfd variant of the perfectoidization and give the official definition of the analytic de Rham stack.  Let $\Cat{ArcStk}^{\qfd}_{\Q_p}$ and $\Cat{GelfStk}^{\qfd}$ be the categories of qfd arc-stacks over $\Q_p$  and Gelfand stacks as in \cref{def-qfd-light-arc-stack} and \cref{xhs9wjqfd} respectively.

\begin{definition}\label{DefAndeRham}\
\begin{enumerate}

\item We define the \textit{perfectoidization} functor $$(-)^{\diamond}\colon \Cat{GelfStk}^{\qfd}\to \Cat{ArkStk}_{\Q_p}^{\qfd}$$  to be   the functor sending a qfd Gelfand stack to its restriction to qfd separable perfectoid $\Q_p$-algebras. Thanks to \cref{sec:perf-analyt-de-2-arc-covers-are-good} and  \cref{PropDescendableCoverFiniteDimPerfectoid} this functor sends qfd Gelfand stacks to qfd arc-stacks, it is left exact and commutes with colimits, and it is itself the left adjoint of a fiber products and colimits preserving functor $\widehat{(-)}\colon \Cat{ArkStk}_{\Q_p}^{\qfd} \to \Cat{GelfStk}^{\qfd}$ sending a qfd separable perfectoid affinoid $\Marc(A)$ to $\widehat{\Marc}(A)=\GSpec(A)$.

\item We define the  \textit{analytic de Rham stack functor} $$(-)^{\dR}\colon \Cat{ArcStk}^{\qfd}_{\Q_p} \to \Cat{GelfStk}^{\qfd}$$  to be  the right adjoint of the perfectoidization $(-)^{\diamond}\colon \Cat{GelfStk}^{\qfd}\to \Cat{ArcStk}^{\qfd}_{\Q_p}$. More explicitly, given $X$ a qfd arc-stack, its de Rham stack $X^{\dR}$ is the qfd Gelfand stack sending a qfd separable Gelfand ring $A$ to the anima $X^{\dR}(\GSpec (A))= X(\Marc(A))$.  For $X$ a qfd Gelfand stack we also denote $X^{\dR}:=(X^{\diamond})^{\dR}$.  Finally, for a morphism $Y\to X$ of qfd Gelfand stacks, we define its relative de Rham stack to be the pullback $Y^{\dR/X}:=Y^{\dR}\times_{X^{\dR}} X$ in $\Cat{GelfStk}^{\qfd}$.

\end{enumerate}
\end{definition}

In Sections \Cref{sec:the-analytic-dR-stack} and \Cref{Subsection:HyperdescentdR}, we shall prove, building on results about the $!$-topology of Gelfand rings proved in \cref{sec:furth-results-analyt-existence-of-covers}, that the analytic de Rham stack functor in qfd Gelfand stacks sends arc-hypercovers to $!$-equivalences, and we formally deduce that $(-)^{\dR}$ commutes with colimits.

The perspective that the analytic de Rham stack should be defined as a right adjoint of a perfectoidization functor suggests a close relationship between perfectoidness and de Rham stacks.
We end this subsection with a question due to one of us (S.), which makes this precise by suggesting a characterization of perfectoid Tate rings. 

\begin{question}
  \label{sec:char-perf-tate-conjecture-scholze-perfectoid}
  Let $A$ be a separable Gelfand ring, and $X=\ob{GSpec}(A)$. Then $A$ is a perfectoid Tate ring if and only if the natural map
  $$
A \to R\Gamma(X, \mathbb{G}_a^\dRall)
  $$
  is an isomorphism of solid abelian groups.
\end{question}

Here, the cohomology is taken on the topos $\Cat{GelfStk}$ with the $!$-topology, and $\mathbb{G}_a^\dRall$ is $(\mathbb{A}_{\Q_p}^{1,\diamond})^\dRall$ seen as a sheaf taking values in $\ob{D}(\Cat{CondAb})$. 

We observe that one direction is easy (see \cref{sec:char-perf-tate-vanishing-ga-dagger-cohomology-for-perfectoids} below): if $A$ is a separable perfectoid Tate ring over $\Q_p$, and $X=\ob{GSpec}(A)$, then $A \cong R\Gamma(X,\mathbb{G}_a^\dRall)$. In particular, if $A$ is a uniform Tate ring, then in \Cref{sec:char-perf-tate-conjecture-scholze-perfectoid} the cohomology can equivalently be taken in usual abelian groups (indeed, the condensed cohomology complex can be calculated by a complex of Banach spaces, because of the vanishing of higher $\mathbb{G}_a^\dRall$-cohomology on perfectoid rings and \Cref{LemBasisTopologyGelfandRings}(2)). 

In \cref{RemConjectureNilradicalVersion} we will give an equivalent formulation of this question in terms of the $\dagger$-nilradical and the almost vanishing of $\O^{\circ\circ}$-cohomology in the arc-site.

\begin{remark}
We note that if \cref{sec:char-perf-tate-conjecture-scholze-perfectoid} admits a positive answer, it would show that ``locally perfectoid implies perfectoid'', a long standing open question on perfectoids.
\end{remark}

\subsection{$!$-covers of separable Gelfand rings}
\label{sec:furth-results-analyt-existence-of-covers}

Let $X\in \Cat{GelfStk}$. We want to prove more difficult results on analytic de Rham stacks. Most notably, we want to show that:
\begin{itemize}

\item[I.] The natural morphism $X\to X^\dRall=(X^\diamond)^\dRall$ is surjective if $X=\GSpec(A)$ is an affinoid Gelfand stack, with $A$ a separable Gelfand ring which  is $\dagger$-formally smooth (cf. \Cref{sec:appr-gelf-rings-dagger-formally-smooth} below).

\item[II.] Let $X$  be a derived Berkovich space (cf. \cref{DefinitionBerkovichSpaces}),  then the Gelfand stack $X^{\dRall}$ agrees with the analytic de Rham stack of \cite{camargo2024analytic} (up to sending Gelfand stacks to analytic stacks).

\item[III.] Let $X$ be a qfd derived Berkovich space, then $X^{\dR}=X^{\dRall}$ (up to sending qfd Gelfand stacks to Gelfand stacks).

  \item[IV.] Let $Y\to Y'$ be an epimorphism of qfd arc-stacks. Then the map of de Rham stacks $Y^{\ob{dR}}\to Y^{\prime,\ob{dR}}$ is an epimorphism (seen as objects in $\Cat{GelfStk}^{\qfd}$). 
\end{itemize}

In order to prove these results, we need to introduce some technical but useful full subcategories of Gelfand rings. The class of nilperfectoids will play a role analogous to the one that semiperfectoid rings play in integral $p$-adic Hodge theory  \cite{Bhatta}. 

\begin{definition}\label{xkwjw8}
Let $A$ be a Gelfand ring.

\begin{enumerate}

\item  We say that $A$ is \textit{uniformly perfectoid} if $A^{u}$ is a perfectoid Banach algebra.   We let $\Cat{GelfRing}_{\mathbb{Q}_{p}}^{\ob{uperfd}}\subset \Cat{GelfRing}_{\mathbb{Q}_p}$ be the full subcategory of uniformly perfectoid Gelfand rings.

\item  We say that $A$ is \textit{totally disconnected} if $\mathcal{M}(A)$ is a profinite set. We say that $A$ is \textit{strictly totally disconnected} if it is totally disconnected and all its completed residue fields are algebraically closed fields. We let $\Cat{GelfRing}^{\ob{td}}_{\Q_p}\subset \Cat{GelfRing}_{\Q_p}$ (resp. $\Cat{GelfRing}^{\ob{std}}_{\Q_p}$) be the full subcategory of (strictly) totally disconnected Gelfand rings.   

\item Finally, we say that $A$ is \textit{nilperfectoid} if $A^{\dagger-\ob{red}}$ is a perfectoid ring.
  We let $\Cat{GelfRing}^{\ob{nilperfd}}_{\Q_p}\subset \Cat{GelfRing}_{\Q_p}$ be the full subcategory of nilperfectoid Gelfand ring.  
\item We denote by $\Cat{GelfRing}_{\mathbb{Q}_{p}}^{\ob{uperfd},\qfd}, \Cat{GelfRing}^{\ob{td},\qfd}_{\Q_p}, \Cat{GelfRing}^{\ob{std},\qfd}_{\Q_p}, \Cat{GelfRing}^{\ob{nilperfd},\qfd}_{\Q_p}$ the subcategories of these categories formed by qfd rings.

\end{enumerate}
\end{definition}

\begin{remark}\label{RemarkTDNilpRingIsPerfd}
Note that \cref{xhsjwks} shows that a strictly totally disconnected Gelfand ring $A$ such that $A^{\dagger-\mathrm{red}}=A^u$ is nilperfectoid. 
\end{remark}

The previous notions have the following stability properties.

\begin{lemma}\label{LemmaStabilityNilPerfd}
The following hold: 

\begin{enumerate}

\item The category $\Cat{GelfRing}^{\ob{nilperfd}}_{\Q_p}$ is stable under non-empty finite colimits.

\item  The category $\Cat{GelfRing}_{\mathbb{Q}_{p}}^{\ob{uperfd}}$ is stable under all (small) colimits of Gelfand rings.

\item  Let $A\in \Cat{GelfRing}^{\ob{td}}_{\Q_p}$ and $A\to B$ a map of separable Gelfand rings that induces a pro-\'etale map on arc-stacks, then  $B\in \Cat{GelfRing}^{\ob{td}}_{\Q_p}$.

\end{enumerate}

\end{lemma}
\begin{proof}

For part (1), for non-empty finite coproducts, if $A$ and $B$ are nilperfectoid, $A\to A^u$ and $B\to B^u$ are surjective, so the map $A\otimes_{\Q_{p,\solid}} B \to A^u\otimes_{\Q_{p,\solid}} B^u$ is surjective and thus we are reduced to show that  if $A$ and $B$  are perfectoid $\Q_p$-algebras then $A\otimes_{\Q_{p,\solid}} B$ is nilperfectoid. Indeed, $A^{\circ}\otimes_{\Z_{p,\solid}} B^{\circ}$ is a semiperfectoid ring, and generic fibers of semiperfectoid rings are nilperfectoid by \cref{sec:defin-main-prop-uniform-completion-and-perfectoidization}. The same argument applies for pushouts, replacing $\Q_p$ by a nilperfectoid $\Q_p$-algebra.

Part (2) follows from \cref{xhs892k} and part (1). Indeed, uniform completions of filtered colimits of perfectoid rings are clearly perfectoid, and part (1) implies that uniform completions of coproducts of perfectoids are also perfectoids.

Finally, for part (3), this follows from the fact that the map of arc-stacks $\Marc(B)\to \Marc(A)$ is pro-\'etale, and $A$ being totally disconnected this implies that $\Marc(B)$ is a profinite set too.  
\end{proof}

What makes nilperfectoid rings useful is the fact that they form a basis for the $!$-topology, as shown by the following lemma. This will for example be used to see that the natural map $\A^{1}_{\Q_p}\to \A^{1,\dRall}_{\Q_p}$ is an epimorphism, which was one of our basic desiderata.

\begin{proposition}\label{LemBasisTopologyGelfandRings}
Let $A$ be a separable Gelfand ring. 
\begin{enumerate}
\item  There is a descendable map $A\to A'$ inducing an isomorphism on uniform completions  where $(A')^{\dagger-\red}=A^{\prime,u}$. 

\item There is a descendable (and quasi-pro-\'etale after uniform completion) map $A\to A'$ where $A'$ is a separable nilperfectoid ring.

\item Suppose that $\mathcal{M}(A)$ has finite cohomological dimension. Then there is a descendable map $A\to A'$ where $A'$ is a totally disconnected nilperfectoid ring, and $\Marc(A')\to \Marc(A)$ is quasi-pro-\'etale.

\item Suppose that $A$ is totally disconnected. Then there is a ind-finite \'etale map $A\to A'$ with $A'$ separable strictly totally disconnected.  In particular, $A\to A'$ is descendable. 
\item Assume that $A$ is qfd. Then in items (1)-(4), $A^\prime$ can also be taken to be qfd.
\end{enumerate}
\end{proposition}
\begin{proof}

We first prove (1). For this, it suffices to construct a map $R:=\varinjlim_{n\in \N} \mathbb{Q}_p\langle T_1,\ldots, T_n \rangle\to A$ such that the associated map on uniform completions $R^u=\mathbb{Q}_p\langle T_{n}: n\in \N \rangle \to A^u$ is surjective. Indeed, the maps $\mathbb{Q}_p\langle T_1,\ldots ,T_n\rangle \to R^u $ are descendable of index $\leq 1$ as they admit a section. Then, by \cite[Proposition 2.7.2]{mann2022p} the map $R\to R^u$ is descendable of index $\leq 2$.  Thus, $A'=R^u\otimes_{R} A$ is a descendable $A$-algebra with same uniform completion and $A^\prime \to A^{\prime,u}=A^u$  is surjective. In particular,  $A^{\prime,\dagger-\ob{red}} \to A^{\prime,u}$ is surjective and since it is always injective (\Cref{xsu8rf}), it is an isomorphism. Finally, to produce such a map, as $A^u$ is a separable Banach algebra, there are countably many elements $x_n\in A^{u,\circ}$ such that the map $\Q_p\langle T_n: n\in \N \rangle\to A^u$, sending $T_n$ to $x_n$, is surjective. We can assume without loss of generality that each element $x_i$ lifts to  $A^{\leq 1}$, this produces a morphism of algebras $\varinjlim_{n\in \N} \Q_p\langle T_1,\ldots, T_n \rangle_{\leq 1}\to A^{\leq 1}$.  Each algebra $\Q_p \langle T_1,\ldots, T_n\rangle_{\leq 1}$ can be written as a countable filtered colimit of Banach algebras. Hence, after repeating the same construction as before for each of these Banach algebras, we can find countably many elements $y_n\in A^{\leq 1}$ that actually extend to a map of solid $\Q_p$-algebras
\[
\varinjlim_{n} \Q_p \langle T_1,\ldots, T_n\rangle \to A
 \]
 which surjects after taking uniform completion. This proves part (1).

For part (2), by the construction of part (1) we can assume that $A$ admits a map $\Q_p\langle T_n: n\in \N \rangle\to A$ that induces a surjection after uniform completion. The map $\Q_p\langle T_n: n\in \N \rangle\to \Q_p^{\ob{cyc}}\langle T^{1/p^{\infty}}_n : n\in \N\rangle$ is descendable as it admits as section as modules, taking base change along this map we get a descendable map $A\to A'$ such that $A^{\prime,\dagger-\ob{red}}=A^{u}$, and so that $A^{\prime,u}$ is a quotient of $ \Q_p^{\ob{cyc}}\langle T^{1/p^{\infty}}_n : n\in \N\rangle$. By \cref{sec:defin-main-prop-uniform-completion-and-perfectoidization}, $A^{\prime,u}$ is perfectoid, proving what we wanted.

Next, for part (3), by  \cref{xk29sm} we can find a Berkovich pro-\'etale map $A\to A'$ with $A'$ totally disconnected such that we have a cartesian square of Gelfand stacks
\[
\begin{tikzcd}
\ob{GSpec} A' \ar[r] \ar[d] & \mathcal{M}(A')_{\ob{Betti}} \ar[d] \\ 
\ob{GSpec} A \ar[r] & \mathcal{M}(A)_{\ob{Betti}}
\end{tikzcd}
\
\]
Furthermore, since  $\mathcal{M}(A)_{\ob{Betti}}$ has finite cohomological dimension, the map $\mathcal{M}(A')_{\ob{Betti}}\to \mathcal{M}(A)_{\ob{Betti}}$ is descendable by \cite[Proposition II.1.1]{ScholzeRLL}, 
and so is $A\to A'$. Hence we can assume that $A$ is totally disconnected. Then we can apply (2), noting that the map from (2) is pro-\'etale on uniform completion, to find the desired cover $A \to A^\prime$ by a totally disconnected nilperfectoid ring.

Finally, for part (4),  by applying \cref{x290sm} it suffices to construct such pro-finite \'etale cover to the uniform completion $A^u$ which then follows from the same construction of \cref{sec:light-arc-stacks-1-cover-by-separable-totally-disconnected} (2). The descendability follows from \cite[Proposition 2.7.2]{mann2022p} being a countable filtered colimit of finite \'etale maps. 

The last assertion (5) follows immediately by observing that the ring $A^\prime$ constructed in (1)-(4) is always such that $A\to A^\prime$ is quasi-pro-\'etale on arc-stacks.
\end{proof}

We deduce the following result saying that the two natural choices for defining the de Rham stack actually agree; for the big de Rham stack up to some additional technical assumption, for the analytic de Rham stack on qfd Gelfand stacks unconditionally.

\begin{proposition}\label{PropComparisonTwodeRhamStacks}
Let $X$ be a Gelfand stack whose restriction to \emph{all} separable perfectoid rings satisfies arc-hyperdescent\footnote{This condition does not always hold due to \cref{sec:-covers-perfectoids-1-counter-example-to-cover-by-strictly-totally-disconnected}, namely, we only can guarantee hyperdescent when restricted to totally disconnected perfectoid rings.}. Then, $X^{\dRall}$ is the sheafification for the $!$-topology of the pre-stack sending $A\in \Cat{GelfRing}_{\omega_1}$ to $X(A^{\dagger-\ob{red}})$. Similarly, let $X$ be a qfd Gelfand stack, then $X^{\dR}$ is the sheafification for the $!$-topology of the qfd Gelfand stack sending $A\in \Cat{GelfRing}^{\qfd}_{\omega_1}$ to $X(A^{\dagger-\red})$.
\end{proposition}

\begin{proof}
By \cref{LemBasisTopologyGelfandRings}, it suffices to show the statement when restricted to separable nilperfectoid rings (resp., if $X$ is qfd, when restricted to separable qfd nilperfectoid rings for $X$  qfd). But then, $A^{\dagger-\ob{red}}=A^u$ is actually perfectoid, and, by hypothesis, the restriction of $X$ to separable perfectoid rings satisfies arc-descent (resp., if $X$ is qfd, then any arc-cover  of qfd perfectoid rings is a $!$-cover by \cref{PropDescendableCoverFiniteDimPerfectoid}). This implies that for $A$ separable (qfd) nilperfectoid ring one has $X^{\ob{dR}}(A)=X(A^{\dagger-\ob{red}})$ proving the proposition.
\end{proof}

\begin{remark}
Note that for a Gelfand stack $X$ and a separable Gelfand ring $A$ with $A^u$ perfectoid, we know that $X^{\dRall}(A)=X(A^u)$ a priori only when $A$ is totally disconnected (\cref{prop:formula-functor-of-points-drall}) or when $A$ is such that $\mathcal{M}(A)$ has finite cohomological dimension (\Cref{sec:-covers-perfectoids-1-arc-versus-shriek-cover}). But to prove the previous proposition, we need to argue for a basis of the $!$-topology, such as nilperfectoid rings and they do not belong to these two classes of examples. This is why one is forced to put some assumption on $X$. On the other hand, in the qfd case, \cref{PropDescendableCoverFiniteDimPerfectoid} guarantees that the restriction of $X$ to qfd perfectoid rings always satisfies arc-descent, and so no hypothesis on $X$ is needed. This is another reason why we introduced the theory of the analytic de Rham stack in the qfd case.
\end{remark}

Using that nilperfectoid rings form a basis of the $!$-topology on separable Gelfand rings, we can prove the easy direction of \Cref{sec:char-perf-tate-conjecture-scholze-perfectoid}. We first reformulate it by noting that on the topos $\Cat{GelfStk}$ there exists the distinguished triangle
\[
  \mathbb{G}_a^\dagger\to \mathbb{G}_a\to \mathbb{G}_a^\dRall,
\]
whose value on a separable nilperfectoid Gelfand ring $A$ is given by the underlying anima\footnote{I.e., $\mathbb{G}_a^\dagger(A)=\Nil^\dagger(A)(\ast)$, where $\ast$ is a one-point set. We note that one can naturally upgrade $\mathbb{G}_a^\dagger, \mathbb{G}_a, \mathbb{G}_a^\dRall$ to functors to solid ablian groups.} of
\[
  \Nil^\dagger(A)\to A\to A^{\dagger-\red}=A^u.
\]
Consequently, \Cref{sec:char-perf-tate-conjecture-scholze-perfectoid} characterizes perfectoid rings in terms of the vanishing of $\mathbb{G}_a^\dagger$-cohomology. 
 
\begin{proposition}
  \label{sec:char-perf-tate-vanishing-ga-dagger-cohomology-for-perfectoids}
  Let $A$ be a separable perfectoid Tate ring over $\Q_p$, and $X=\ob{GSpec}(A)$. Then $R\Gamma(X,\mathbb{G}_a^\dagger)=0$. 
  \end{proposition}
\begin{proof}
  As $A$ is uniform perfectoid (and thus $H^i(X,\mathbb{G}_a^\dagger)=0$ for $i=0, 1$), it is sufficient to show that $H^i(X,\mathbb{G}_a^\dRall)=0$ for $i>0$. Let $A\to B$ be a $!$-cover. Refining using \Cref{LemBasisTopologyGelfandRings}(2) we may assume that $B$ is nilperfectoid. Then it suffices to see that
  \[
    A\to B^u\to (B\otimes_A B)^u\to \ldots
  \]
  is exact. But this is precisely the \v{C}ech complex for the arc-cover $A\to B^u$.  However, the structure sheaf on perfectoids satisfies arc-descent.  This shows the desired vanishing.
\end{proof}

As the proof of \cref{sec:char-perf-tate-vanishing-ga-dagger-cohomology-for-perfectoids} shows the $\mathbb{G}_a^{\dRall}$-cohomology calculates classical $v$-cohomology of the completed structure sheaf on (partially proper) rigid analytic varieties over $\mathbb{C}_p$. In particular, it is generally non-zero in higher degrees.

\begin{remark}\label{RemConjectureNilradicalVersion}
We also have a distinguished triangle
\[
 \mathbb{G}_a^\dagger\to \mathbb{G}_a^{< 1}\to \mathbb{G}_a^{< 1,\dRall},
\]
where $\mathbb{G}_a^{< 1}$ is the sheaf $A\mapsto A^{< 1}$,
so the vanishing of the cohomology of $\mathbb{G}_a^\dagger$ can also be reformulated as saying that the map
$$
R\Gamma(X,\mathbb{G}_a^{< 1}) \to R\Gamma(X,\mathbb{G}_a^{< 1,\dRall})
$$
is an isomorphism of solid abelian groups.

On the other hand, suppose that there is a pseudo-uniformizer $\pi$ of $A$ such that $|p|\leq |\pi|^p$. Then  $A$ being perfectoid is equivalent to the previous map being an isomorphism and the fact that $H^1(X,\mathbb{G}_a^{< 1,\dRall})$ is zero.
Indeed, the previous equivalence implies in particular that $A^{< 1}= H^0(X, \mathbb{G}_a^{< 1,\dRall}) = H^0(X^{\diamond}_{\proet}, \O^{\circ\circ})$ is $p$-complete and so that $A$ is a uniform Banach algebra.
The vanishing of the $H^1$-cohomology then implies that the Frobenius map $\varphi\colon A^{< 1}/p\to A^{< 1}/p$ is  surjective, which together with the existence of the pseudo-uniformizer $\pi$ implies that $A^{\leq 1}$ is integral perfectoid and so that $A$ is itself perfectoid: given $a\in A^{\leq 1}$ and a pseudo-uniformizer $\varpi\in A^{<1}$ with $|p|<|\varpi^p|$, there exists some $x\in A^{<1}$ with $x^p\equiv \varpi^p a$ in $A^{<1}/p$.
Then $ \varpi^{-1}x\in A^{\leq 1}$ (as can be checked at points on $\Spa(A)$, for example) and $(\varpi^{-1}x)^p\equiv a \mod A^{\leq 1}/ \varpi^{-1}p$.

Hence,  \Cref{sec:char-perf-tate-conjecture-scholze-perfectoid} says that this $H^1$-vanishing condition is automatic from the vanishing of the $\mathbb{G}_a^{\dagger}$-cohomology.  A natural question is whether the existence, locally in the analytic topology, of the pseudo-uniformizer $\pi$ is guaranteed by the vanishing of the $\mathbb{G}_a^{\dagger}$-cohomology.
\end{remark}

\subsection{de Rham stacks of derived Berkovich spaces}\label{ss:GelfandAnalyticSpaces}\label{ss:ApproxGelfand}
 
In this section we discuss the points (I), (II) and (III) of the beginning of \cref{sec:furth-results-analyt-existence-of-covers}. Before getting to the general case, let us discuss the example of the affine line.

\begin{example}\label{ExamAffineLineDR}
 Let $X=\A^{1,\an}_{\Q_p}$ be the analytic affine line over $\Q_p$.
 When we think of $X$ as a Gelfand stack, $X$ represents the functor sending a separable Gelfand ring $A$ to the anima $X(\GSpec(A)= A(*)$ given by the underlying algebra of $A$.
 Since $A=\bigcup_{r>0} A^{\leq r}$, we can write $\A^{\an}_{\Q_p}=\bigcup_{r>0} \mathbb{D}^{\leq r}_{\Q_p}$ where $\mathbb{D}^{\leq r}_{\Q_p}=\GSpec (\Q_p\langle T \rangle_{\leq r})$ is the overconvergent disc of radius $r$.
 We can then think of the analytic affine line as a suitable Gelfand stack which is glued from affinoid stacks with respect to the analytic topology, i.e., the Grothendieck topology of open subspaces of the Berkovich spectrum. 
 
To determine the big de Rham stack of $X$, by \cref{LemBasisTopologyGelfandRings} it suffices to evaluate it in nilperfectoid rings.
 Let $A$ be a nilperfectoid ring, then $X^{\dRall}(\GSpec(A))= A^u=A^{\dagger-\red}= A/\Nil^{\dagger}(A)$.
 Therefore the map $X\to X^{\dRall}$ is an epimorphism of Gelfand stacks, this means that $X$ is \textit{$\dagger$-formally smooth} in a suitable sense (see \cite[Section 3.7]{camargo2024analytic} for a definition in these lines adapted to bounded rings).
 If we see $X$ as a ring stack, denoted by $\mathbb{G}_{a}$, the map $\mathbb{G}_a\to \mathbb{G}_a^{\dRall}$ is an epimorphism of ring stacks, and its kernel is precisely the ideal $\mathbb{G}_a^{\dagger}=\GSpec (\Q_p\langle T\rangle_{\leq 0})$ representing the (underlying anima of the) nilradical. We deduce that
\[
\mathbb{G}_a^{\dRall}= \mathbb{G}_a/\mathbb{G}_a^{\dagger}
\]
where $\mathbb{G}_a^{\dagger}$ acts by translation.
 By the same discussion one has that $\mathbb{G}_a^{\dR}= \mathbb{G}_a/\mathbb{G}_a^{\dagger}$ as qfd Gelfand stacks. 
\end{example}

For studying the de Rham stack it will be necessary to have a good understanding of when $X\to X^\dRall$ is surjective.
 The following definition is a variant of \cite[Definition 3.7.2]{camargo2024analytic} (which we have to adapt as the definition of the analytic de Rham stack in this paper is different).

\begin{definition}
  \label{sec:appr-gelf-rings-dagger-formally-smooth} A morphism $Y\to X$ in $\Cat{GelfStk}$ with perfectoidization $Y^{\diamond}\to X^{\diamond}$ in $\Cat{ArcStk}_{\Q_p}$ is \textit{$\dagger$-formally smooth} (resp.\ \textit{$\dagger$-formally \'etale}) if for any separable  nilperfectoid ring $A$, the natural map
\begin{equation}\label{eq03jkl13r}
Y(A)\to X(A)\times_{X^{\diamond}(\Marc(A^u))} Y^{\diamond}(\Marc(A^u))
\end{equation}
is an epimorphism of anima (resp.\ an equivalence).
Similarly, a morphism $Y\to X$ in $\Cat{GelfStk}^{\qfd}$ with perfectoidization $Y^{\diamond}\to X^{\diamond}$ in $\Cat{ArcStk}^{\qfd}_{\Q_p}$  is  \textit{$\dagger$-formally smooth} (resp.\ \textit{$\dagger$-formally \'etale}) if for any separable qfd nilperfectoid ring $A$, the map \cref{eq03jkl13r} is an epimorphism of anima (resp.\ an equivalence).
\end{definition}

\begin{remark}
\cref{sec:appr-gelf-rings-dagger-formally-smooth} only rephrases the condition for the map $Y\to Y^{\dRall/X}$ to be a surjection of pre-stacks when restricted to separable nilperfectoid rings.
 Indeed, for $A$ a separable nilperfectoid ring, one has by adjunction $X^{\dRall}(A)=X^{\diamond}(\Marc(A))$.
 The same applies for the map $Y\to Y^{\dR/X}$ if $Y$ and $X$ are qfd Gelfand stacks.
\end{remark}

We will need this technical lemma in the following:

\begin{lemma}\label{Lemmapqjjpqwbfoqw}
Let $A$ be a solid $\Q_p$-algebra and let $\GSpec (A)\in \Cat{GelfStk}$ be the functor it represents in Gelfand stacks. Then its restriction to separable perfectoid rings satisfies arc-hyperdescent. 
\end{lemma}
\begin{proof}
This follows form the fact that if $B\to B^{\bullet}$ is an arc hypercover of perfectoid rings then $B=\ob{Tot}(B^{\bullet})$.  
\end{proof}

An immediate consequence of \cref{sec:appr-gelf-rings-dagger-formally-smooth} is the following surjectivity result for suitable $\dagger$-formally smooth algebras in the sense of \cite[Definition 3.7.2]{camargo2024analytic}.

\begin{corollary}\label{CorollaryDaggerSmoothComparison}
  \label{sec:dagg-form-smooth-presentation-of-analytic-de-rham-stack-}
  Assume that $A$ is a $\dagger$-formally smooth $\Q_p$-algebra in the sense of \cite[Definition 3.7.2]{camargo2024analytic}. Set $X:=\ob{GSpec}(A)$. Then the natural map $X\to X^{\dRall}$ is an epimorphism (resp. for $X^{\dR}$ if $A$ is qfd).
\end{corollary}
\begin{proof}
 By \cref{Lemmapqjjpqwbfoqw} and \cref{PropComparisonTwodeRhamStacks} we know that $X^{\dRall}$ is the sheafification of the functor sending a Gelfand ring $B$ to $X(B^{\dagger-\red})$. The formally smoothness definition of \cite[Definition 3.7.2]{camargo2024analytic} says precisely that the map of anima $X(B)\to X(B^{\dagger-\red})$ is an epimorphism. The corollary follows. 
\end{proof}

\begin{example}
  \label{sec:dagg-form-smooth-example-overconvergent-cohomology-of-rigid-disc}\
  \begin{enumerate}
  \item Let $A=\Q_p\langle T\rangle$ and consider $X=\ob{GSpec}(A)$, this is the realization of Huber's compactification of the affinoid disc in Gelfand stacks. In order to find a presentation of $X^{\dRall}$ as a quotient of a representable Gelfand stack we need a $\dagger$-formally smooth approximation of $A$. This is provided by $\widetilde{A}=\Q_p\langle T\rangle_{\leq 1}$ (e.g., by \cref{xjs9jw}). Indeed,  let $\widetilde{X}=\ob{GSpec} \widetilde{A}$, the ring $A$ is the uniform completion of $\widetilde{A}$ which yields that $X^{\dRall}=\widetilde{X}^{\dRall}$.  Then, by \cref{CorollaryDaggerSmoothComparison} the map $\widetilde{X}\to X^{\dRall}$ is an epimorphism. Concretely, by looking at the functor of points  one has that 
\[
X^{\dRall} = \mathbb{D}_{\Q_p}^{\leq 1}/\mathbb{G}_{a}^{\dagger} 
\]
where $\mathbb{D}_{\Q_p}^{\leq 1}=\widetilde{X}$.   In particular, the cohomology of $X^{\dRall}$ calculates \textit{overconvergent} de Rham cohomology of $X$, and not the honest, but pathological, de Rham cohomology of $X$. Since the ring $A$ is qfd, one also has $X^{\dR}= \mathbb{D}_{\Q_p}^{\leq 1}/\mathbb{G}_{a}^{\dagger} $ as qfd Gelfand stacks. 
   We note that by \Cref{sec:de-rham-stacks-counterexample-to-smoothness} the morphism $\Q_p\to \Q_p\langle T\rangle$ is \emph{not} $\dagger$-formally smooth.
   \item Let $S$ be a light profinite set. Then, the ring of locally constant functions $C^{\lc}(S,\Q_p)$ is a $\dagger$-formally \'etale $\Q_p$-algebra by \cref{sec:appr-gelf-rings-examples-of-dagger-formally-smooth-maps} below, with uniform completion $C(S,\Q_p)$ the space of $\Q_p$-valued continuous functions of $S$. We can conclude that
    \[
      (\ob{GSpec}(C(S,\Q_p)))^{\dRall}=(\ob{GSpec}(C^{\lc}(S,\Q_p)))^{\dRall}\cong \ob{GSpec}(C^{\lc}(S,\Q_p))=S_{\ob{Betti}}
    \]
    (resp. for the de Rham stack seen as qfd Gelfand stack).
  \item We have that
    \[
      (\ob{GSpec}(\C_p))^{\dRall}\cong (\ob{GSpec}(\overline{\Q}_p))^{\dRall}\cong \ob{GSpec}(\overline{\Q}_p),
    \]
    again using that the $\dagger$-formally \'etale $\Q_p$-algebra $\overline{\Q}_p$ has uniform completion $\C_p$ (resp.\ for $\GSpec (\C_p)^{\dR}$ as qfd Gelfand stack).
  \end{enumerate}

\end{example}

\begin{example}
  \label{sec:de-rham-stacks-counterexample-to-smoothness}
  The morphism $\Q_p\to \Q_p\langle T\rangle$ is \emph{not} $\dagger$-formally smooth. Indeed, assume that it is, and pick a descendable cover $h\colon \Q_p\langle T\rangle_{\leq 1}\to A$ with $A$ (qfd) nilperfectoid (the existence of such an $A$ is guaranteed by \Cref{LemBasisTopologyGelfandRings}).
  Then $ \Q_p\langle T\rangle_{\leq 1}\to A^u$ extends uniquely to a morphism $f\colon\Q_p\langle T\rangle$.
  If $\Q_p\to \Q_p\langle T\rangle$ were $\dagger$-formally smooth, there would exist a lift $g\colon \Q_p\langle T\rangle \to A$. We let $s:=g(T)$ and $t:=h(T)$, and consider the induced morphism
  \[
    \alpha\colon \Q_p\langle S\rangle\langle T\rangle_{\leq 1}\langle \varepsilon \rangle_{\leq 0}/(T-S-\varepsilon)\to A
  \]
  sending $S\mapsto s,\ T\mapsto t$ (note that $s-t\in \Nil^{\dagger}(A)$).
  Now,
  \[
    \Q_p\langle S\rangle\langle T\rangle_{\leq 1}\langle \varepsilon \rangle_{\leq 0}/(T-S-\varepsilon)\cong \Q_p\langle T\rangle\langle \varepsilon\rangle_{\leq 0},\ S\mapsto T-\varepsilon
  \]
  because the translation of a powerbounded element by a norm zero element is still powerbounded (in fact, it is sufficient to assume that $\varepsilon$ has norm $<1$).
  As $h\colon \Q_p\langle T\rangle_{\leq 1}\to A$ is descendable, we can conclude that $\Q_p\langle T\rangle_{\leq 1}\to \Q_p\langle T\rangle\langle \varepsilon\rangle_{\leq 0}$ is descendable, too.
  This in turn implies using $\Q_p\langle T\rangle\langle \varepsilon\rangle_{\leq 0}\otimes_{\Q_p\langle T\rangle_{\leq 1}}\Q_p\langle T\rangle \cong \Q_p\langle T\rangle\langle \varepsilon \rangle_{\leq 0}$
  that $\Q_p\langle T\rangle_{\leq 1}\to \Q_p\langle T\rangle$ is an isomorphism, which is not true.
\end{example}

Examples of $\dagger$-formally smooth/\'etale maps can be constructed as follows:

\begin{definition}\label{x2edaj3}
Let $f:A\to B$ be a morphism of Gelfand rings. We say that $f$ is  
\begin{enumerate}

\item  \textit{standard Berkovich smooth} if there is some $r>0$ such that $B$ is of the form $$B=A \langle T_1,\ldots, T_n \rangle_{\leq r}/^{\mathbb{L}} (f_1,\ldots, f_k),$$  with $k\leq n$ such that the determinant of the matrix $(\frac{\partial f_j}{\partial T_j})_{i,j=1}^k$  is invertible on $B$.

\item \textit{standard Berkovich \'etale} if it is standard Berkovich-smooth with $n=k$.
\end{enumerate}
\end{definition}

\begin{lemma}
  \label{sec:appr-gelf-rings-examples-of-dagger-formally-smooth-maps}\
  \begin{enumerate}
  \item Standard Berkovich smooth (resp.\ \'etale) maps of Gelfand rings are $\dagger$-formally smooth (resp.\ \'etale). In particular, $A\to A\langle T\rangle_{\leq 1}$ is $\dagger$-formally smooth.

\item Berkovich smooth (resp.\ \'etale) maps of Gelfand rings are, locally in a strict cover, standard Berkovich smooth (resp. standard $\dagger$-\'etale). In particular, they are $\dagger$-formally smooth.
  
    \item  $\dagger$-formally smooth (resp.\ \'etale) maps of Gelfand stacks are stable under pullbacks and compositions. 
    
    \item Let $\{Y_i\to X\}$ be a cofiltered limit  of $\dagger$-formally smooth (resp.\ \'etale) maps of (qfd) Gelfand stacks with $\dagger$-formally \'etale transitions maps. Then $\varprojlim_{i} Y_i\to X$ is $\dagger$-formally smooth.
  \end{enumerate}
\end{lemma}
\begin{proof}
  Part (1) follows from \cite[Proposition 3.4.7]{camargo2024analytic} and \cite[Proposition 3.7.4 (2)]{camargo2024analytic} as analytic locally any standard Berkovich smooth (resp. \'etale) map can be written (on associated analytic stacks) as a filtered colimit of solid smooth maps along solid \'etale maps, in the sense of \cite[Definition 3.5.5]{camargo2024analytic} (we note that if $X,Y$ are represented by Gelfand rings, then \cref{sec:appr-gelf-rings-dagger-formally-smooth} agrees with \cite[Definition 3.7.2]{camargo2024analytic}, up to restricting to nilperfectoids).
  Part (2) follows from \cite[Theorem 3.5.6]{camargo2024analytic} as, locally in rational covers, a Berkovich smooth map (resp. a Berkovich \'etale map) can be written as a colimit of formally smooth (resp. \'etale)  maps of solid finite presentation.
  Parts (3) and (4) are straightforward from the definitions.
\end{proof}

\begin{definition}\label{DefDaggerNeigh}
Let $Y$ be a  derived Berkovich space and let $C\subset |Y|$ be a locally closed subspace. We define the \textit{overconvergent neighbourhood of $C$ in $Y$} to be the  substack $Y^{\dagger_{C}}\subset Y$ for the analytic topology whose $A$-valued points for $A$ a separable Gelfand ring are 
\[
Y^{\dagger_{C}}(A) = \begin{cases} Y(A) & \mbox{ if } \mathcal{M}(A)\to C \subset |Y| \\ 
\emptyset & \mbox{ otherwise}. \end{cases}
\]

\end{definition}

\begin{lemma}\label{LemDaggerNeigh}
Let $Y$ be a (qfd) derived Berkovich space and $C\subset |Y|$ a locally closed subspace, then $Y^{\dagger_{C}}$ is a Gelfand stack.  If $C$ is locally Zariski closed then $Y^{\dagger_{C}}$ is a (qfd) derived Berkovich space.
\end{lemma}
\begin{proof}
The first claim follows from the fact that $Y$ is a Gelfand stack (\Cref{PropBerkovichAreStacks}) and the fact that $!$-equivalences are $\infty$-connective in the arc-topos. For the second claim, we can assume without loss of generality that $Y=\ob{GSpec} A$ is affinoid and that $C$ is the vanishing locus of some elements $\{f_i\}_{i\in I}$. In that case $Y^{\dagger_{C}}$ is the analytic spectrum of $A\langle f_i \rangle_{\leq 0}:= A\langle T_i \rangle_{\leq 0}/^{\mathbb{L}}(T_i-f_i)$, where $A\langle T_i \rangle_{\leq 0} = \varinjlim_{J\subset I} A\langle T_{j}: \: j\in J \rangle_{\leq 0}$ and $J$ runs over finite subsets of $I$.
\end{proof}

Next, we prove some important properties of the de Rham stack of derived Berkovich spaces. 

\begin{proposition}\label{PropdeRhamBerkovich}
Let $Y$ be a derived Berkovich space. Then it satisfies the hypothesis of \cref{PropComparisonTwodeRhamStacks}. Furthermore, the following hold: 
\begin{enumerate}

\item  Let $Z\to X$ be a map of derived Berkovich spaces whose associated map $Z^{\diamond}\to X^{\diamond}$ of arc-stacks is an immersion. Then, $$Z^{\dRall/X}= X^{\dagger_{Z}}.$$ In particular, $Z^{\dRall}=(X^{\dagger_Z})^{\dRall}$ (resp. for the qfd de Rham stack if $Z\to X$ is a morphism of qfd derived Berkovich spaces).

\item Let $Y\to X$ be a map of derived Berkovich spaces. The \v{C}ech nerve of the map $Y\to Y^{\dRall/X}$ is the simplicial derived Berkovich space $ (Y^{\times_{X} \bullet+1})^{\dagger_{\Delta Y}}$ where $\Delta Y\subset |Y^{\times_{X} \bullet+1}|$ is the diagonal immersion (resp. for the qfd de Rham stack if $Y\to X$ is a morphism of qfd derived Berkovich spaces).

\item Let $\Cat{Aff}^b_{\Q_p}$ be the category of bounded affinoid stacks over $\Q_p$ (\cite[Definition 2.6.10 and 3.2.10]{camargo2024analytic}).
  Let $$\Cat{TateStk}_{\Q_p}:=\Cat{AnStk}(\Cat{Aff}^b_{\Q_p})$$ be the category of Tate stacks over $\Q_p$.\footnote{In \cite[3.2.10]{camargo2024analytic} the author considered sheaves for the $\ob{D}$-topology. Since we are working with analytic rings, the $\ob{D}$ and $!$-topology are the same and, up to taking sheafification for $!$-equivalences, both definitions of Tate stacks agree.}
  Let $F\colon \Cat{GelfStk}\to \Cat{TateStk}$ be the left Kan extension of the inclusion $\Cat{GelfStk}^{\ob{aff}}\subset \Cat{GelfStk}$ along the map $\Cat{GelfStk}^{\ob{aff}}\to  \Cat{TateStk}$.
  Then, for $X$ a derived Berkovich space,  the natural map
\[
F(X)^{\mathrm{Tate}\text{-}\dR}\to F(X^{\dRall}).
\]
is an isomorphism. Namely, for derived Berkovich spaces the big de Rham stack and the de Rham stack of \cite{camargo2024analytic}, denoted here $(-)^{\mathrm{Tate}\text{-}\dR}$,  agree.

\item Let $X$ be a qfd Berkovich space, and consider the left exact colimit preserving functor $F\colon \Cat{GelfStk}^{\qfd}\to \Cat{GelfStk}$ inducing the identity on qfd affinoid Gelfand stacks. Then, the natural map
\[
F(X^{\dR})\to X^{\dRall}
\]
is an isomorphism. In other words, the qfd and big de Rham stacks agree for $X$.

\end{enumerate} 

\end{proposition}
\begin{proof}
We first have to show that the restriction of $Y$ to perfectoid rings satisfies arc-hyperdescent; this follows from \cref{PropBerkovichAreStacks}.

Next, we prove (1). Let $A$ be a nilperfectoid  ring, then 
\[
Z^{\dRall/X}(A) =  X(A) \times_{X^{\diamond}(\Marc(A^u))}Z^{\diamond}(\Marc(A^u)).
\]
As $Z^{\diamond}\to X^{\diamond}$ is an immersion, $Z^{\diamond}(\Marc(A^u))\subset X^{\diamond}(\Marc(A^u))$ is a subset (see \cref{RemDiamondBerkovichSpace})  and $Z^{\dRall/X}(A) \subset X(A)$ is a full subanima. A map $\ob{GSpec }A\to X$ will factor through $Z^{\dR/X}(A)$ if and only if its restriction to the perfectoid ring $A^u$ factors through $Z$, which is equivalent to the fact that $\Marc(A^u)\to X^{\diamond}$ factors through $Z^{\diamond}$. As $Z^{\diamond}\to X^{\diamond}$ is an immersion, this is equivalent to saying that $\mathcal{M}(A)\to |X|$ factors through $|Z|$, proving that $Z^{\dRall/X}=X^{\dagger_{Z}}$ as wanted.

Now we prove (2), i.e.  we compute the \v{C}ech nerve $Z^{\bullet}$ of $Y\to Y^{\dRall/X}$, it fits in a cartesian diagram 
 \[
 \begin{tikzcd}
  Z^{\bullet} \ar[r]\ar[d] &  Y^{\times_{X} \bullet+1} \ar[d] \\  
  Y^{\dRall/X} \ar[r,"\Delta"] & (Y^{\times_X \bullet+1})^{\dR/X} 
 \end{tikzcd}
 \]
 where the bottom horizontal map is the diagonal map. By part (1) the pullback is nothing but $(Y^{\times_{X} \bullet+1})^{\dagger_{\Delta Y}}$ proving what we wanted.

For parts (3), all the variants of de Rham stacks satisfy analytic descent (this can be checked by evaluating the respective functor of points description) and so it suffices to prove the claim for $X=\GSpec(A)$ an affinoid Berkovich space. In this situation,  we can always find a surjection $A'\to A$ from a bounded algebra $A'$ that is $\dagger$-formally smooth in the sense of \cite[Definition 3.7.2]{camargo2024analytic} (e.g. by $A$ as a quotient of the algebras $\Q_p\langle \mathbb{N}[S] \rangle_{\leq r}$ for $S$ light profinite and using \cref{xsu8rf}). Let $I=\ker (A'\to A)$,  and let $A[S_i]\to I$ be a family of jointly surjective maps of $A$-modules. By taking base of $ A'[\mathbb{N}[S_i] ]\to A'$ along $A'[\mathbb{N}[S_i]] \to A'\langle \mathbb{N}[S_i] \rangle_{\leq 0}$ for all $i$ (which are $\dagger$-formally \'etale thanks to \cref{xsu8rf}), we can assume without loss of generality that $A'\to A$ is an isomorphism in $\dagger$-reductions and therefore in uniform completions. In particular, $A'$ is a separable Gelfand ring.  Let $X'=\GSpec(A')$, then the map $X\to X'$ is an equivalence in perfectoidizations, and we have that $X^{\dRall}=X^{\prime, \dRall}$. On the other hand, by \cref{CorollaryDaggerSmoothComparison} and the fact that $A'$ is $\dagger$-formally smooth we have that
\[
X^{\dRall}=\varinjlim_{[n]\in \Delta^{\op}} (X^{'\times n+1, \dagger_{\Delta {X^{\prime}}}})
\]  
is the geometric realization of the overconvergent neighbourhood of the diagonals of the \v{C}ech nerve of $X'\to \GSpec(\Q_p)$. Since the functor $F$ from Gelfand stacks to Tate stacks preserves colimits by constructions, the same presentation holds for $F(X^{\dRall})$, which then agree with $(FX^{\prime})^{\ob{Tate}\text{-}\dR}$ since $FX^{\prime}\to (FX^{\prime})^{\ob{Tate}\text{-}\dR}$ is an epimorphism of Tate stacks by \cite[Proposition 5.2.3]{camargo2024analytic} (2). Finally, Kashiwara's lemma \cite[Corollary 5.2.5]{camargo2024analytic} shows that $(FX)^{\ob{Tate}\text{-}\dR}=(FX^{\prime})^{\ob{Tate}\text{-}\dR}$ proving what we wanted.

Part (4) follows from the same argument of part (3) after replacing (locally in the analytic topology) $X$ by a $\dagger$-formally smooth derived Berkovich space.
\end{proof}

\begin{remark}\label{RemKashiwara}
One should think of \cref{PropdeRhamBerkovich}  (1) as a very general version of Kashiwara's lemma, namely, if the  subspace $Z\to X$ is (locally) Zariski closed, it recovers Kashiwara's lemma for analytic $D$-modules in the sense of  \cite[Corollary 5.2.5]{camargo2024analytic}.
\end{remark}

The following result is helpful to present de Rham stacks as quotients of Berkovich spaces by overconvergent equivalence relations.

\begin{proposition}\label{PropQuotientdROverconvergentEquiv}
Let $Y,X$ be Berkovich spaces and $Y^{\diamond}\to X^{\diamond}$ a morphism of diamonds over $\Marc(\Q_p)$ with \v{C}ech nerve $Y^{\diamond, \bullet/X^{\diamond}}\to X^{\diamond}$, let $Y^{\bullet}$ be the \v{C}ech nerve over $\GSpec(\Q_p)$.  Then the natural map $Y^{\diamond, \bullet/X^{\diamond}}\to Y^{\bullet,\diamond}$ is an immersion and we have a pullback diagram 
\begin{equation}\label{eqqw013o1wd113es1}
\begin{tikzcd}
Y^{\bullet/X^{\dRall}} \ar[r] \ar[d]  & Y^{\bullet} \ar[d] \\
(Y^{\diamond, \bullet/X^{\diamond}})^{\dRall} \ar[r] & Y^{\bullet,\dRall}
\end{tikzcd}
\end{equation}
where $Y^{\bullet/X^{\dRall}}$ is the \v{C}ech nerve of $Y\to X^{\dRall}$.  In particular, $Y^{\bullet/X^{\dRall}}$ is the overconvergent neighbourhood of $Y^{\bullet}$ along the inclusion of topological spaces  $|Y^{\diamond, \bullet/X^{\diamond}}|\subset |Y^{\bullet}|$. 
\end{proposition}
\begin{proof}
The fact that $Y^{\diamond, \bullet/X^{\diamond}}\to Y^{\bullet,\diamond}$ is an immersion is clear since the morphism $Y^{\diamond}\to X^{\diamond}$ is locally separated on the analytic topology of $Y$ and $X$.  The fact that \eqref{eqqw013o1wd113es1} if a pullback diagram is formal from the pullback diagram 
\[
\begin{tikzcd}
(Y^{\diamond, \bullet/X^{\diamond}})^{\dRall} \ar[r] \ar[d] & Y^{\bullet,\dRall} \ar[d] \\
X^{ \dRall}\ar[r] & X^{\bullet,\dRall}
\end{tikzcd}
\]where $X^{\bullet}$ is the \v{C}ech nerve of $X\to \GSpec(\Q_p)$. Finally, the identification of $Y^{\bullet/X^{\dRall}}$ with the overconvergent neighbourhood of $|Y^{\diamond, \bullet/X^{\diamond}}|\subset |Y^{\bullet}|$ follows from \cref{PropdeRhamBerkovich} (1).
\end{proof}

We finish the discussion of the de Rham stack of derived Berkovich spaces  by proving an important descent result for the de Rham stack of $\dagger$-rigid spaces (cf. \Cref{DefDaggerRigid}).

\begin{theorem}
\label{thm:prim-descendable-berkovich}
Let $K$ be a complete non-archimedean field over $\Q_p$, separable as a $\Q_p$-Banach space, and $X$ a $\dagger$-rigid space over $K$. Then the natural map $X\to X^{\dRall/K}$ is prim and descendable, and thus an epimorphism of Gelfand stacks. In particular, we have a presentation as Gelfand stack
\[
X^{\dRall/K}= \varinjlim_{[n]\in \Delta^{\op}} (X^{\times_K n+1})^{\dagger_{\Delta X}}.
\]
The same holds for $X^{\dR/K}$ if $K$ is in addition qfd. 
\end{theorem}
\begin{proof}
This follows from the same proof of \cite[Theorem 5.4.1]{camargo2024analytic} where the only difference is that we use \cref{sec:appr-gelf-rings-examples-of-dagger-formally-smooth-maps} for proving that if $X$ is Berkovich smooth, then $X\to X^{\dRall/K}$ is an epimorphism.
\end{proof}

\begin{remark}
Later, in \Cref{sec:cohom-smooth-maps-descendability-for-smooth-maps}, we will improve on this result in the smooth case (even in a relative situation) by proving that the morphism is even descendable, with a bound for the index of descendability in terms of the dimension.

A standard consequence of \Cref{thm:prim-descendable-berkovich} is that in the smooth case, with the notations above, the coherent cohomology of the analytic de Rham stack of $X$ relative to $K$ and the de Rham cohomology of $X$ over $K$ coincide. This will be proven in \Cref{cor:cohomology-de-rham-stack-and-de-rham-cohomology} after some convenient preparations on six functors for the analytic de Rham stack.
\end{remark}

Having explicitly described the de Rham stack for rigid spaces, let us discuss an important example that illustrates some features of the new formulation of the theory. 

\begin{example}
  \label{sec:dagg-form-smooth-example-perfected-gm}
Let $\mathbb{G}_{m,\Q_p}^{\an}$ be the analytic multiplicative group seen as a Berkovich space and let $\mathbb{G}_{m, \ob{perf},\Q_p}^{\an,u}$ be the perfectoid multiplicative group obtained as the uniform completion (on affinoids) of the ``decompleted limit'' $\mathbb{G}_{m, \ob{perf},\Q_p}^{\an}:=\varprojlim_{x\mapsto x^p} \mathbb{G}_{m,\Q_p}^{\an}$, i.e., the limit in Gelfand stacks. Then, one has that $\mathbb{G}_{m, \ob{perf},\Q_p}^{\an,u,\dR}=\mathbb{G}_{m, \ob{perf},\Q_p}^{\an,\dR}$. Furthermore, as the transition maps are finite \'etale, \cref{sec:appr-gelf-rings-examples-of-dagger-formally-smooth-maps} implies that $\mathbb{G}_{m, \ob{perf},\Q_p}^{\an}$ is $\dagger$-formally smooth, and then
\begin{equation}\label{eqPerfectoidGm}
\mathbb{G}_{m, \ob{perf},\Q_p}^{\an,u,\dR}= \mathbb{G}_{m, \ob{perf},\Q_p}^{\an}/ \mathbb{G}_{m, \ob{perf},\Q_p}^{\dagger},
\end{equation}
where $\mathbb{G}_{m, \ob{perf},\Q_p}^{\dagger}\subset \mathbb{G}_{m, \ob{perf},\Q_p}$ is the overconvergent neighbourhood at $1$.
Notice that 
\[
\Q_p\langle  T-1\rangle_{\leq 0}\to \varinjlim_{n}{\Q_p}\langle T^{1/p^n}-1\rangle_{\leq 0}
\]
is actually an isomorphism as $T^{1/p^n}= (1+ (T-1))^{1/p^n}=\sum_{k} \binom{1/p^n}{k} (T-1)^k$ converges in $\Q_p\langle  T-1\rangle_{\leq 0}$ since $|T-1|=0$.
In particular,  $\mathbb{G}_{m, \ob{perf},\Q_p}^{\dagger}\cong\mathbb{G}_{m,\Q_p}^{\dagger}$.
The equation \cref{eqPerfectoidGm} also implies that the  compactly supported  Rham cohomology of $\mathbb{G}_{m, \ob{perf},\Q_p}^{\an,u}$, defined as the $!$-pushforward to $\GSpec(\Q_p)$, is the colimit of the compactly supported Rham cohomologies of $\mathbb{G}_{m,\Q_p}$ along the maps $x\mapsto x^p$ (see \cref{sec:d-modules-quasi} for more general results).
 These maps induce equivalences of de Rham cohomologies, so we deduce that 
\[
R\Gamma_{c,\dR}(\mathbb{G}_{m, \ob{perf},\Q_p}^{\an,u})= R\Gamma_{c,\dR}(\mathbb{G}_{m,\Q_p}^{\an}) = \Q_p[-2] \oplus \Q_p \frac{dT}{T} [-1].
\]
Hence, even though the perfectoid space $\mathbb{G}_{m, \ob{perf},\Q_p}^{\an,u,\dR}$ has no differentials and a priori there is no reasonable notion of \textit{algebraic} $D$-modules over it, there is a well defined notion of de Rham cohomology and \textit{analytic $D$-modules} that can be computed by a suitable approximation  along smooth rigid spaces.
\end{example}

\subsection{Arc-descent for the analytic de Rham stack}
\label{sec:the-analytic-dR-stack}
\label{sec:furth-results-analyt}

In this section, we prove a descent result, \Cref{sec:furth-results-analyt-1-surjections-on-de-rham-stack},  for the analytic de Rham stack of qfd Gelfand stacks, which is a preliminary step towards the main descent result established later (\Cref{TheoMaindeRham2}).

 Proving \Cref{sec:furth-results-analyt-1-surjections-on-de-rham-stack} will require some further preliminaries. 
  To motivate this technical discussion, as well as to explain the qfd assumption required in it, let us first informally sketch the idea.
 If $f\colon Y\to X$ is an epimorphism of light arc-stacks, we would like $Y^{\dRall} \to X^{\dRall}$ to be an epimorphism of Gelfand stacks.
 In other words, if $A$ is a separable Gelfand ring with a map $g\colon \GSpec(A)\to X^{\dRall}$, we want to see that there is a $!$-cover $\GSpec(A^\prime) \to \GSpec(A)$ and a lift $\GSpec(A^\prime) \to Y^{\dRall}$.
 By the definition of the analytic de Rham stacks, this means starting with a map $\Marc(A^u) \to X$ and trying to find a $!$-cover $\GSpec (A^\prime) \to \GSpec (A)$ and a lift $\Marc (A^{\prime,u}) \to Y$.
 Since $f$ is an epimorphism of arc-stacks, we can certainly find an arc-cover $A^u \to B$ of separable perfectoid rings with a lift $\Marc(B) \to Y$.
 The difficulty is to ensure that this arc-cover can be chosen as the uniform completion of a $!$-cover of $A$.
 If $A$ is furthermore assumed to be strictly totally disconnected nilperfectoid, we show in below that this can be achieved: roughly one writes the given arc-cover as the uniform completion of a  countable filtered colimit of rational localizations $A_i$ of Tate algebras over $A$ which are arc-covers after passing to the uniform completion.
 Since $A$ is strictly totally disconnected, the arc-cover $A^u \to A_i^u$ splits; since $A$ is nilperfectoid and $A\to A_i$ $\dagger$-formally smooth, the splitting lifts to a splitting of $A \to A_i$, yielding in the colimit descendability of $A \to A^\prime:=\varinjlim A_i$.

Unfortunately, although separable nilperfectoid rings form a basis for the $!$-topology on separable Gelfand rings, we cannot a priori reduce to the case where $A$ is strictly totally disconnected. We know however by \Cref{LemBasisTopologyGelfandRings} that any separable qfd Gelfand ring admits a $!$-cover by a strictly totally disconnected nilperfectoid ring. Hence we can make the above argument work if we restrict our framework to  qfd arc-stacks.

\begin{remark}
  \label{sec:perf-analyt-de-3-comparison-to-real-case}
 In the archimedean case, the perfectoidization functor $(-)^\diamond\colon \Cat{GelfStk}^{\qfd}\to \Cat{ArcStk}_{\Q_p}^{\qfd}$ is the analog  of the functor $\psi^\ast$ from \cite[Section V.3.1]{ScholzeRLL} (replacing light arc-stacks by condensed anima).
 This supports the viewpoint that the real analog of a perfectoid Tate ring is the ring of continuous functions on a compact Hausdorff space.
 We note that in the notation of \textit{loc. cit.}  $\psi^*$ admits a right adjoint  $\psi_\ast=\pi^\ast$ commuting with colimits (the second being the  analog of $(-)^\dR$).
 We note furthermore that analytic de Rham  stacks over $\Q_p$ are much richer than the Betti stacks of the underlying topological spaces (in contrast to analytic Riemann--Hilbert, \cite[Theorem II.3.1]{ScholzeRLL}).
 The basic reason lies in the fact that the class of algebraically closed Banach fields over $\Q_p$ is much richer than that over $\R$.
 In particular, the totally disconnected rings $A$ of \cite[Definition V.2.6]{ScholzeRLL} are (up to working with separable Banach algebras) qfd over $\mathbb{C}$ in the sense that $\Marc(A)\to  \Marc(\mathbb{C})$ is quasi-pro-\'etale (being only a map of profinite sets).
\end{remark}

Our first intermediate goal is to prove \Cref{corollary-smooth-cover} below. The statements used in its proof are inspired by analogous arguments in the context of perfectoid spaces (e.g.\ \cite[Proof of Proposition 14.6]{scholze_etale_cohomology_of_diamonds}). 

\begin{lemma}\label{xsh29k}
Let $A\to B$ be a morphism of separable Gelfand rings. Then there is a sequential diagram of morphisms of separable Gelfand $A$-algebras $B_n\to B$ over $A$ with $B_n$ given by rational localizations of affine spaces over $A$ such that  $\varinjlim_{n} B_n$ is $\dagger$-formally smooth over $A$ and
\[
(\varinjlim_{n} B_n)^u\to  B^u
\]
is an isomorphism. In particular, the map of Berkovich spectra
\[
\mathcal{M}(B)\to \varprojlim_{n}\mathcal{M}(B_n)
\]
is an equivalence. 
\end{lemma}
\begin{proof}
Let $B^u$ be the uniform completion of $B$. Since $B$ is separable, $B^u$ is a separable Banach space and has a countable basis over $\mathbb{Q}_p$. By \cref{xsu8rf} we can then take elements $\{e_{n}\}_{n\in \mathbb{N}}$ in $B$ mapping to a Banach basis in $B^{u,\leq 1}$. Therefore, the elements $e_n$ give rise to a map of Gelfand rings 
\[
C:=\varinjlim_{n} A\langle T_1,\ldots, T_n \rangle_{\leq 1}\to B
\]
sending $T_i\mapsto e_i$. It is an epimorphism after uniform completions. Thus, we have by \Cref{x5134hs}  a Zariski closed embedding of Berkovich spectra $\mathcal{M}(B)\hookrightarrow \mathcal{M}(C)$ given by the vanishing locus of the ideal $I=\ob{fib}(C\to B^u)$. Set $\widetilde{B}_n= A\langle T_1,\ldots, T_n \rangle_{\leq 1}$ so that we have that $\mathcal{M}(C)=\varprojlim_{n} \mathcal{M}(\widetilde{B}_n)$. The closed subspace $\mathcal{M}(B)\subset \mathcal{M}(C)$ is then a limit of the Zariski closed subspaces $Z_n \subset \mathcal{M}(\widetilde{B}_n)$ given by the vanishing locus of the kernel of $\widetilde{B}_n\to C\to B^u$. Then, we can write each $Z_n$ as a countable filtered colimit of rational localizations $\widetilde{B}_{n,m}$ of $\widetilde{B}_n$. In total, after modifying the indexes, we have a filtered family of algebras $\{B_{n,m}\}_{n,m\in \mathbb{N}}$ where each $B_{n,m}$ is a rational localization of the overconvergent Tate algebra $\widetilde{B}_{n}$. We claim that the system $\{B_{n,m}\}_m$ satisfies the properties of the lemma. First, the map $C\to B$ factors through   $C\to \varinjlim_{n,m} B_{n,m}\to B$  and so  $g:\varinjlim_{n} B_{n,m}\to B$ induces an epimorphism at the level of uniform completions. Since $\varprojlim_{n,m} \mathcal{M}(B_{n,m}) = \varprojlim_n \mathcal{M}(Z_n) =\mathcal{M}(B)$, \cref{x5134hs} implies that $g$ is an equivalence after uniform completions. By construction each $B_{n,m}$ is a rational localization of an overconvergent Tate algebra, so it remains to see that its colimit is $\dagger$-formally smooth. But $\varinjlim_{n} B_{n,m}$ is also a colimit of $C$ along rational localizations, hence a colimit along $\dagger$-formally \'etale maps. Since $C$ is $\dagger$-formally smooth (\cref{sec:appr-gelf-rings-examples-of-dagger-formally-smooth-maps}) the claim follows from \cref{sec:appr-gelf-rings-examples-of-dagger-formally-smooth-maps}.
\end{proof}

The following lemma is a decompleted version of \cite[Lemma 4.9]{scholze2024berkovichmotives}.

\begin{lemma}\label{xs2ifm}
Let $A$ be a separable strictly totally disconnected nilperfectoid ring and let $f:A\to B$ be a $\dagger$-formally smooth map where $B$ is a finite product of Berkovich rational localizations of an affine space over $A$. If $f$ is an arc-cover then there is a morphism of separable Gelfand rings $s:B\to A$, which is a section of $f$. 
\end{lemma}
\begin{proof}
Consider the map of uniform completions $f^u:A^u\to B^u$. Then $B^u$ is the underlying ring of an adic rational localization of an affine space over $A^u$. Since $A^u$ is a strictly totally disconnected perfectoid ring, \cite[Lemma 4.9]{scholze2024berkovichmotives} gives a section $\overline{s}:B^u\to A^u$. Then, we have a commutative square
\[
\begin{tikzcd}
 B \ar[r]  & B^u \ar[d, shift left = 1.5 ex, "\overline{s}"] \\ 
 A  \ar[r]  \ar[u,"f"]& A^u \ar[u, shift left =1.5ex, "f^u"]
\end{tikzcd}
\]
But the map $A\to A^u$ is a $\dagger$-\textit{thickening}, namely,  as $A$ is nilperfectoid, we have a short exact sequence $\Nil^\dagger(A)\to A\to A^{u}$. Since $f$ is $\dagger$-formally smooth, we can lift $\overline{s}$   to a section $s:B\to A$ as wanted.
\end{proof}

\begin{lemma}\label{xvsa9w}
Let $A$ be a separable strictly totally disconnected nilperfectoid ring and let $B=\varinjlim_{n} B_n$ be a  sequential colimit of Gelfand $A$-algebras satisfying the properties of \cref{xsh29k}. Suppose that $A^u\to B^u$ is an arc-cover, then $A\to B$ is $\leq 2$-descendable. 
\end{lemma}
\begin{proof}
By assumption each $B_n$ is a finite product of rational localizations on affine spaces over $A$. By \cref{xs2ifm} the map $A\to B_n$ has a section and so it is descendable of index $\leq 1$. Then \cite[Proposition 2.7.2]{mann2022p} implies that the colimit is descendable of index $\leq 2$. 
\end{proof}

The following statement follows directly from the results just proved. 
\begin{proposition}
\label{corollary-smooth-cover}
Let $A\to B$ be a morphism of separable Gelfand rings, which induces an arc-cover on uniform completions. Assume that $A$ is strictly totally disconnected nilperfectoid. One can find a separable Gelfand $A$-algebra $A^\prime$ with $A^{\prime, u} \cong B^u$ such that $A \to A^\prime$ is $\dagger$-formally smooth and descendable.
\end{proposition}

\begin{remark}
  \label{sec:arc-descent-analytic-case-that-a-is-qp}
  We note that \cref{corollary-smooth-cover} holds true if $A=\Q_p$ and $B$ a separable Gelfand ring for easier reasons: by \cref{xsh29k} there exists some $\dagger$-formally smooth $\Q_p$-algebra $A'$ such that $(A')^u=B^u$.
  If $A^u\to B^u$ is an arc-cover, then $B^u\neq 0$, which implies that $A=\Q_p\to B\to B^u$ is descendable (as it admits a section).
  This in turn implies that $A\to A'$ is descendable as well. 
\end{remark}

After the previous preparations we come to our descent result. 

\begin{theorem}
  \label{sec:furth-results-analyt-1-surjections-on-de-rham-stack}
Let $Y\to X$ be an epimorphism of qfd arc-stacks.
 Then the map $Y^{\dR}\to X^{\dR}$ of qfd Gelfand stacks is an epimorphism.  
\end{theorem}
\begin{proof}
Let $A$ be a qfd Gelfand ring with a map $f\colon \GSpec(A)\to X^{\dR}$.
 We want to show that there is a $!$-cover $\GSpec(B) \to \GSpec(A)$ with $B$  a qfd Gelfand ring and a lift to a commutative diagram
\[
\begin{tikzcd}
\GSpec(B)  \ar[r] \ar[d] & \GSpec(A)  \ar[d] \\
Y^{\dR} \ar[r] & X^{\dR}.
\end{tikzcd}
\]
By adjunction, the map $\GSpec (A)\to X^{\dR}$ factors through $\GSpec(A) \to \Marc(A)^{\dR}\to X^{\dR}$, and by taking pullbacks we can assume without loss of generality that $X=\Marc(A)^{\dR}$.
 By \cref{LemBasisTopologyGelfandRings} we can even assume that $A$ is a strictly totally disconnected nilperfectoid ring.
  In that case, as $Y\to X$ is an arc-cover and $X$ is qcqs, we can assume without loss of generality that $Y$ is the arc-sheaf associated to a qfd separable perfectoid $A^u$-algebra.
 Using \cref{corollary-smooth-cover}, resp.\ \cref{sec:arc-descent-analytic-case-that-a-is-qp}, we can assume that $A$ is $\dagger$-formally smooth, and that $Y=\Marc(B)$ is the arc-stack of a descendable $A$-algebra $B$.
 Thus, we get a diagram 
\[
\begin{tikzcd}
\GSpec(B) \ar[r] \ar[d] & \GSpec(A)  \ar[d] \\
Y^{\dR} \ar[r] & X^{\dR}
\end{tikzcd}
\]
where the upper horizontal map is a $!$-cover.
This finishes the proof of the theorem.\footnote{
We note that we even arranged that in the diagram the right vertical arrow is an epimorphism (as a consequence of $A$ being $\dagger$-formally smooth), and thus the above proof illustrates well how our previous results allow to access analytic de Rham stacks in general.}
\end{proof}

\begin{remark}

  \label{sec:main-descent-result-1-descent-for-the-big-de-rham-stack}
  Analyzing the proof of \Cref{sec:furth-results-analyt-1-surjections-on-de-rham-stack} one gets some descent results for $(-)^{\dRall}$. Let $Y\to X$ be an epimorphism of light arc-stacks over $\Q_p$. Then $Y^\dRall\to X^\dRall$ is an epimorphism if one of the following conditions hold:
  \begin{enumerate}
  \item $Y\to X$ is \'etale,
  \item $Y\cong \varprojlim_{n\in \N} Y_n$ with $Y_n\to X$ finite \'etale,
  \item $X$ is \emph{quasi-finitary} in the following sense: for any separable nilperfectoid ring $A$ and any map $\Marc(A)\to X$ there exists a $!$-cover $A\to A'$ and a factorization $\Marc(A')\to Z\to X$ of $\Marc(A')\to \Marc(A)\to X$ where $Z$ is qfd over $\Marc(\Q_p)$.
  \end{enumerate}
  Indeed, in the first case one can use that \'etale morphisms lift along $\dagger$-reductions (and the lift is again an \'etale covering if one starts with an \'etale cover).
  In the second case, one can use that the lift is again finite \'etale and that finite \'etale coverings have index of descendability bounded by $\leq 2$ (as follows by pullback from the underlying rings).
  In the third case, one can mimic the proof of \Cref{sec:furth-results-analyt-1-surjections-on-de-rham-stack}, and use that $Z$ admits a $!$-cover by a strictly totally disconnected perfectoid space (which can be pulled back to a $!$-cover of $\GSpec(A')$).
  This reduces to the case where $A$ is strictly totally disconnected nilperfectoid, where the same proof applies.

  Several examples of interesting light arc-stacks are quasi-finitary, e.g., stack quotients $X=U/\mathbb{D}$ where $U$ is qfd over $\Q_p$, and $\mathbb{D}\subseteq \mathbb{A}^m_{\Q_p}$ is an open unit disc acting on $U$. Namely, in this case any morphism $\Marc(A)\to X$ with $A$ nilperfectoid lifts to $U$ as affinoid perfectoid spaces have vanishing higher $\mathbb{D}$-cohomology.
  
\end{remark}

\begin{remark}
  \label{sec:arc-descent-analytic-classifying-stacks}
  Let $G$ be a locally light profinite group.
  Applying \cref{sec:furth-results-analyt-1-surjections-on-de-rham-stack} to the epimorphism $\ast\to B\underline{G}$, we get that
  \[
  (B\underline{G})^{\dR}=\big(\GSpec(\Q_p)\big) / G_{\Betti} = BG^{\sm}
  \]
  where $G^{\sm}=(\underline{G})^\dR$ is the Berkovich space with underlying topological space $G$ and sheaf of functions given by the locally constant functions of $G$ with values in $\Q_p$ (cf.\ \cref{sec:dagg-form-smooth-example-overconvergent-cohomology-of-rigid-disc}).
  Thanks to \cite[Proposition 5.3.10]{heyer20246functorformalismssmoothrepresentations} the category $\ob{D}(BG^{\sm})$ is the derived category of solid $\Q_p$-linear smooth representations of $G$, and  by \cite[Corollary 5.3.3]{heyer20246functorformalismssmoothrepresentations} any such classifying stack is $!$-able over $\GSpec(\Q_p)$ (note that we can apply \textit{loc. cit.} as the local cohomological dimension of a locally profinite group over a $\Q$-algebra is zero).
  In \cref{ExamCondensedAnimadeRham} we will prove more generally that the de Rham stack of a condensed anima is given by its Betti stack.
  Instead of \cref{sec:furth-results-analyt-1-surjections-on-de-rham-stack} this will rely on the stronger descent statement \cref{TheoHyperdescentdR}
\end{remark}

\newpage
\section{$D$-modules and $6$-functors on analytic de Rham stacks}
\label{sec:d-modules-quasi}

In this section we show that there are enough $!$-able maps for de Rham stacks, and that those arising from smooth maps of derived Berkovich spaces are also cohomologically smooth.
 These results are important by themselves, since they lead to finiteness and duality results for the de Rham cohomology of smooth $\dagger$-rigid spaces (via \Cref{cor:cohomology-de-rham-stack-and-de-rham-cohomology}), but they also  play a key role in the new proof of the $p$-adic monodromy theorem.

\subsection{$!$-able maps of de Rham stacks; cohomological properness and smoothness}

 The $6$-functor formalism on de Rham stacks is obtained by composing the functors
\[
\Cat{ArcStk}^{\qfd}_{\Q_p}\to \Cat{GelfStk}^{\qfd}\to \Cat{AnStk}
\]
given by the de Rham stack and sending a qfd Gelfand stack to its associated analytic stack, with the $6$-functors of quasi-coherent sheaves
\[
\Cat{Corr}( \Cat{AnStk}, E)\to \Cat{Pr}^L
\]
where $E$ is the class of $!$-able arrows of analytic stacks.
 By pulling back $E$, we get a suitable class of $!$-able maps in $\Cat{ArcStk}^{\qfd}_{\Q_p}$ for the de Rham stack.
 In this section we produce enough examples of arrows in $E$ for the de Rham stack.
 
 Let us first recall the following useful lemma.
 
 \begin{lemma}\label{LemShierkLocalTarget}
Let $(\mathcal{C},E)$ be a  geometric set up  with fiber products and $\ob{D}$ a presentable\footnote{This means that $\ob{D}$ as a functor to $\mathrm{Cat}_\infty$ factors over $\Cat{Pr}^L$.} $6$-functor formalism  on $(\mathcal{C},E)$. Suppose that $\mathcal{C}$ has a final object $*$, has finite disjoint unions, and that any of the inclusions $*\hookrightarrow *\sqcup *$ is $!$-able (so \cite[Lemma 3.4.13]{heyer20246functorformalismssmoothrepresentations} holds).  We endow $\mathcal{C}$ with the $\ob{D}$-topology of \cite[Example 3.4.5 (b)]{heyer20246functorformalismssmoothrepresentations}. Let $\ob{D}\colon \Cat{Corr}(\Shv(\mathcal{C}),E^{\prime})\to \Cat{Pr}^L$ be the extension of $\ob{D}$ to a $6$-functor formalisms for sheaves for the $\ob{D}$-topology of \cite[Theorem 3.4.11]{heyer20246functorformalismssmoothrepresentations}.  Let $f\colon Y\to X$ be a morphism in $\Shv(\mathcal{C})$ such that there is an epimorphism $\bigsqcup_i X_i\to X$ with $X_i\in \mathcal{C}$ such that the pullbacks $Y_i\to X_i$ are $!$-able. Then $f$ is $!$-able.
\end{lemma}
\begin{proof}
We want to apply \cite[Theorem 3.4.11 (ii)]{heyer20246functorformalismssmoothrepresentations}, that is, given \textit{any} map $X^{\prime}\to X$ with $X^{\prime}$ in $\mathcal{C}$, the pullback $Y^{\prime}\to X^{\prime}$ is $!$-able. By \cite[Theorem 3.4.11 (iii)]{heyer20246functorformalismssmoothrepresentations} it suffices to show this after a $\ob{D}$-cover of $X$, namely, by definition a $\ob{D}$-cover satisfies universal $!$-descent.   Since $\bigsqcup_i X_i\to X$ is an epimorphism, after replacing $X^{\prime}$ by a $\ob{D}$-cover, we can assume that it factors through $X^{\prime}\to \bigsqcup_i X_i\to X$. By hypothesis, each $Y_i\to X_i$ is $!$-able, and by \cite[Lemma 3.4.13]{heyer20246functorformalismssmoothrepresentations}  their disjoint union $\bigsqcup_i Y_i\to \bigsqcup_i X_i$ is also $!$-able, by taking pullbacks we deduce that $Y^{\prime}\to X^{\prime}$ is $!$-able as wanted. 
\end{proof}

We list some simple examples:

\begin{lemma}\label{LemmProetaleShierkdR}
The following hold:
\begin{enumerate}

 \item  Let $f\colon Y\to X$ be an qcqs quasi-pro-\'etale map of qfd arc-stacks over $\Q_p$, then $f^{\dR}\colon Y^{\dR}\to X^{\dR}$ is $!$-able.

\item Let $j\colon U\subset X$ be an open immersion of qfd arc-stacks, then $j^{\dR}\colon U^{\dR}\to X^{\dR}$ is $!$-able and cohomologically \'etale.

\item Let $f\colon Y\to X$ be a map of qfd arc-stacks such that there is an open cover $\{f_i:Y_i\subset Y\}$ with $f_{i}^{\dR}\colon Y^{\dR}_i\to X^{\dR}$ $!$-able. Then $f^{\dR}\colon Y^{\dR}\to X^{\dR}$ is $!$-able.

\item The map $\A^{1,\dR}_{\Q_p}\to \GSpec(\Q_p)$ is $!$-able.

\end{enumerate}
\end{lemma}
\begin{proof}
By \cref{LemShierkLocalTarget}, and by arc-descent of the de Rham stack,  we can check that the maps are $!$-able locally in the arc-topology of the target.
 Thus, for (1)-(3) we can assume without loss of generality that $X$ is the spectrum of a qfd strictly totally disconnected perfectoid ring. 

 (1) By \cite[Lemma 4.5]{scholze2024berkovichmotives} we have a pullback diagram of qfd arc-stacks
\[
\begin{tikzcd}
 Y \ar[r] \ar[d] & X \ar[d] \\
 \underline{|Y|} \ar[r] & \underline{|X|}
\end{tikzcd}
\]
where the bottom arrow is the map of the underlying Berkovich spaces seen as arc-stacks.
 Since the formation of the de Rham stack commutes with limits, it suffices to see that $\underline{|Y|}^{\dR}\to \underline{|X|}^{\dR}$ is $!$-able, but this is clear since both $|Y|$ and $|X|$  are profinite sets, and their de Rham stacks are just their Betti stacks.

 (2) By the same argument of part (1), one reduces to see that an open immersion of a profinite set is $!$-able and cohomologically \'etale, this is \cite[Lemma 4.8.2]{heyer20246functorformalismssmoothrepresentations}.

 (3) This follows from the fact that  $!$-able maps are $!$-local in source \cite[Theorem 3.4.11 (iii)]{heyer20246functorformalismssmoothrepresentations}, and  suave descent in \cite[Lemma 4.7.1]{heyer20246functorformalismssmoothrepresentations}.

 (4) This is an analogue of \cite[Lemma 5.2.8 (1)]{camargo2024analytic}.  Using the $!$-local property in the source of \cite[Theorem 3.4.11 (iii)]{heyer20246functorformalismssmoothrepresentations} and prim descent along descendable maps \cite[Lemma 4.7.4]{heyer20246functorformalismssmoothrepresentations}, it suffices to see that the map $\mathbb{A}^{1}_{\Q_p}\to \mathbb{A}^{1,\dR}_{\Q_p}$ is prim and descendable.
We have $\mathbb{A}^{1,\dR}_{\Q_p}\cong \mathbb{A}^{1}_{\Q_p}/\mathbb{G}_a^\dagger$ (by a qfd version of \cref{ExamAffineLineDR}), and so by base change it suffices to show the same for the map $g\colon \GSpec(\Q_p)\to \GSpec(\Q_p)/\mathbb{G}_a^\dagger$.
Primness is clear, as $\mathbb{G}_a^\dagger$ is proper over $\GSpec(\Q_p)$, and hence $g$ is represented in proper, affinoid Gelfand stacks.
Set $R:=\mathcal{O}(\mathbb{G}_a^\dagger)=\Q_p\langle T\rangle_{\leq 0}$, which is a Hopf algebra through the addition.
Viewing $R$ as the regular representation over itself, we have the short exact sequence
\[
  0\to \Q_p\to R\overset{\frac{\partial}{\partial T}}{\to} R\to 0
\] of $R$-comodules, i.e., as quasi-coherent sheaves on $\GSpec(\Q_p)/\mathbb{G}_a^\dagger$.
This implies that $g$ is descendable.
A more general assertion is discussed in \cite[Lemma 4.3.6]{camargo2024analytic}.
\end{proof}

From the stability properties of \cref{LemmProetaleShierkdR} we can already exhibit plenty of $!$-able maps that will be more than enough for  the main applications of this paper. 

\begin{definition}\label{DefLQFDim}

 A map $f\colon Y\to X$ of qfd arc-stacks over $\Z_p$ is said to be of \textit{locally of quasi-finite dimension} (or \textit{lqfd}) if there is a strict closed cover $Y_i \subset Y$ such that the maps $Y_i\to X$ are qcqs and quasi-pro-\'etale over some relative affine space $Y_i\to \mathbb{A}^{d,\diamond}_{X}$ (with $d$ possibly depending on $i$).

\end{definition}

\begin{proposition}\label{LemmLqfd}
Let $f\colon Y\to X$ be a morphism of qfd arc-stacks over $\Q_p$ which arc-locally on $X$ is lqfd. Then  $f^{\dR}\colon Y^\dR\to X^\dR$ is $!$-able. If in addition $f$ is proper, then $f^{\dR}$ is cohomologically proper in the sense of \cite[Definition 4.6.1]{heyer20246functorformalismssmoothrepresentations}.
\end{proposition}
\begin{proof}
By \Cref{LemShierkLocalTarget}, we can assume that $f$ is lqfd. By \cref{LemmProetaleShierkdR} (3) we can prove that $f^{\dR}$ is $!$-able locally in the analytic topology of $Y$, in particular we can localize along strict closed covers (as those are refined by an open cover), and we can assume without loss of generality that $f$ is qcqs and quasi-pro-\'etale over a relative affine line $\mathbb{A}^{d,\diamond}_{X}$.
  By \cref{LemmProetaleShierkdR} (1) we know that the map $g^{\dR} \colon Y^{\dR}\to \mathbb{A}^{d,\dR}_{X}$ is $!$-able, and by \cref{LemmProetaleShierkdR} (4) the map $ \mathbb{A}^{d,\dR}_{X}\to X^{\dR}$ is $!$-able, proving what we wanted.

For the claim about proper maps, we first note that a closed immersion $Z\subset X$ of arc-stacks is cohomologically proper, namely it arises as the pullback of a closed immersion of condensed anima $\underline{Z_{\ob{cond}}}\to \underline{X_{\ob{cond}}}$ along the map $X\to \underline{X_{\ob{cond}}}$ (cf. \cref{ExaTopologicalSpaceCondensedAnima}), and these are cohomologically proper by \cite[Lemma 4.8.2 (i)]{heyer20246functorformalismssmoothrepresentations}.
 Hence, if $f$ is proper, the diagonal is a closed immersion, and it suffices to show that $f^{\dR}$ is prim.
  By \cite[Lemma 4.5.7]{heyer20246functorformalismssmoothrepresentations} we can prove that $f^{\dR}$ is prim locally in the arc-topology on $X$.
 We can then assume that $X$ is qfd strictly totally disconnected.
  By \cite[Lemma 4.5.8]{heyer20246functorformalismssmoothrepresentations} we can prove that $f^{\dR}$ is prim locally on a descendable cover of the source, thus, as $f$ is proper, we can then pass to a finite strict   closed cover of $Y$ and assume that $Y\to X$ is qcqs and quasi-pro-\'etale over an affine space $\mathbb{A}^d_{X}$ over $X$.
  Since $Y$ is qcqs, the map $Y\to\mathbb{A}^d_{X}$ factors through a polydisc, that we can assume to be of radius $1$.
 Hence, since prim maps are stable under composition, we can reduce the claim to proving that a qcqs pro-\'etale map of arc-stacks is prim, and that the map $\overline{\mathbb{D}}^{1,\dR}_{\Q_p}\to \GSpec(\Q_p)$ from the de Rham stack of the closed disc is prim.

For qcqs and quasi-pro-\'etale maps, we can assume again that $X$ is strictly totally disconnected, and  by \cite[Lemma 4.5]{scholze2024berkovichmotives} the map $Y\to X$ arises from the pullback of profinite sets $|Y|\to |X|$.
 But the map $|Y|^{\dR}\to |X|^{\dR}$ is prim being a map of Betti stacks of profinite sets, and so given by a map of analytic rings with induced structure.
 For the claim about the closed disc $f\colon \overline{\mathbb{D}}^{1,\dR}_{\Q_p}\to \GSpec(\Q_p)$, by  \cref{sec:dagg-form-smooth-example-overconvergent-cohomology-of-rigid-disc} (1)   $\overline{\mathbb{D}}^{1,\dR}_{\Q_p} = \mathbb{D}^{1,\leq 1,\dR}_{\Q_p}$  is the de Rham stack of the overconvergent disc.
 Hence, the map $g\colon \mathbb{D}^{1,\leq 1}_{\Q_p}\to \mathbb{D}^{1,\leq 1,\dR}_{\Q_p}$ is an epimorphism of Gelfand stacks, descendable (as can be deduced by pullback from the case for $\mathbb{A}^1_{\Q_p}$, cf.\ \cref{LemmProetaleShierkdR}), and by \cite[Lemma 4.5.8 (ii)]{heyer20246functorformalismssmoothrepresentations} it suffices to see that $\mathbb{D}^{1,\leq 1}_{\Q_p}\to \GSpec(\Q_p)$ is prim which is clear as it arises from a map of analytic rings with induced structure.
\end{proof}

\begin{remark}\label{RemShierkMaps}
All maps between (partially proper) rigid spaces belong to those of \cref{DefLQFDim}, but not all the stacky maps relevant in the geometrization of the local Langlands correspondence.
 To extend the class of $!$-able maps one has to keep applying \cite[Theorem 3.4.11]{heyer20246functorformalismssmoothrepresentations} for the different classes one would like to consider.
 For example, thanks to \cite[Lemma 4.7.1]{heyer20246functorformalismssmoothrepresentations} one can extend the class of $!$-able maps to quotients of lqfd maps by \textit{suave equivalence relations}, this is nothing but a variant of the so called \textit{Artin stacks} of \cite[Definition IV.1.1]{fargues2021geometrization} for the de Rham stack.
 \end{remark}

\begin{definition}\label{DefRigSmoothRigEtale}
Let $f\colon Y\to X$ be a morphism of qfd arc-stacks over $\Z_p$. We say that $f$ is \textit{\'etale} if, locally in the arc-topology of  $X$ and the analytic topology of $Y$, $f$ factors as a composite of open immersions (\cite[Definition 4.21]{scholze2024berkovichmotives}) and finite \'etale maps.
  We say that $f$ is \textit{smooth} if, locally in the arc-topology of $X$ and the analytic topology of $Y$, $f$ factors as a composite $ Y\xrightarrow{g} \mathbb{A}^d_{X}:=\mathbb{A}^{d,\diamond}_{\Z_p}\times_{\Marc(\Z_p)}X\to X$ where $g$ factors as a composite of finite \'etale maps and open immersions. 
\end{definition}

For example, a rigid smooth, resp.\ rigid \'etale, morphism of derived Berkovich spaces over $\Q_p$ in the sense of \Cref{DefEtaleSmooth}, induces a smooth, resp. \'etale, morphism on the associated arc-stacks.

The following theorem is a variant of \cite[Theorem 1.0.12]{camargo2024analytic}.

\begin{theorem}
  \label{sec:geom-prop-analyt-1-cohomological-smoothness-for-smooth-rigid-spaces}
  Let $f\colon Y\to X$ be a smooth morphism of qfd arc-stacks over $\Q_p$.
  Then the map $Y^{\dR}\to X^{\dR}$ is cohomologically smooth, with dualizing sheaf given by $1_{Y^\dR}[2d]$ if $f$ is of pure dimension $d$.
  If $f$ is \'etale then $f^{\dR}$ is in addition cohomologically \'etale. 
\end{theorem}
\begin{proof}
Any \'etale map of qfd arc-stacks has open diagonal, namely, this can be checked locally in the arc-topology of $X$ and in the analytic topology of $Y$, where it reduces to the claim on perfectoid spaces which is obvious. Therefore, by definition of cohomologically \'etale \cite[Definition 4.6.1]{heyer20246functorformalismssmoothrepresentations}, we only need to show that smooth morphisms are cohomologically smooth on the de Rham stack.
 This property can be checked locally in the arc-topology of $X$ thanks to \cref{sec:furth-results-analyt-1-surjections-on-de-rham-stack} and \cite[Lemma 4.5.7]{heyer20246functorformalismssmoothrepresentations}, so we can assume that $X$ is a perfectoid space and then $Y$ is (the arc-stack of) a derived Berkovich space. The property can also be checked locally in the analytic topology thanks to \cref{PropdeRhamBerkovich}(1) and \cite[Lemma 4.5.8 (i)]{heyer20246functorformalismssmoothrepresentations}. Localizing further for the arc-topology on $X$ if necessary, we see that we can by definition of smoothness and base change ultimately reduce to the following three cases: 
\begin{enumerate}[(i)]
\item An open immersion $U\to X$.

\item A finite \'etale map $Y\to X$

\item The affine line $f\colon \mathbb{A}^{1,\dR}_{\Q_p}\to \GSpec(\Q_p)$. 
\end{enumerate} 
 
The case of open immersions is \Cref{LemmProetaleShierkdR}. For $Y\to X$ finite \'etale, as $X$ is totally disconnected we have $Y^{\dR}=X^{\dR}\times_{|X|_{\Betti}} |Y|_{\Betti}$ as qfd Gelfand-stacks and $|Y|_{\Betti}\to |X|_{\Betti}$ is cohomologically \'etale being isomorphic to a finite disjoint union of clopen maps.   Finally, the case of the affine line follows from the presentation $\mathbb{A}^{1,\dR}_{\Q_p}= \mathbb{A}^{1}_{\Q_p}/ \mathbb{G}_a^{\dagger}$, and the fact that both $\mathbb{A}^{1}_{\Q_p}$ and $B\mathbb{G}_a^{\dagger}$ are suave over $\GSpec(\Q_p)$ (the former being a smooth rigid space, the last thanks to Cartier duality).

For the description of the dualizing sheaf, we refer to \cite[Theorem 5.3.7]{camargo2024analytic}\footnote{Note that \textit{loc. cit.} even proves the result for the filtered de Rham stack, where an additional twist appears, which disappears when one pullbacks to the de Rham stack.} (or \cref{sec:-able-maps-first-chern-classes-de-rham-cohomology}): by $\mathbb{A}^1$-invariance of cohomology of the de Rham stack, the computation reduces, by a by now standard deformation to the normal bundle argument, to the computation of (the pullback by the zero-section of) the dualizing sheaf for the structure morphism of a vector bundle over $Y$, which can be done by reduction the universal case and using Cartier duality. \end{proof}

\begin{remark}
  \label{sec:-able-maps-relation-to-d-modules}
 Through the relation between quasi-coherent sheaves on the analytic de Rham stacks and $D$-modules (see the forthcoming paper \cite{AnDModRJRC}), \Cref{sec:geom-prop-analyt-1-cohomological-smoothness-for-smooth-rigid-spaces} gives Poincar\'e duality for analytic $D$-modules.
 \end{remark}

\begin{remark}
\label{rmk:dr-stack-disc-torus-not-coh-smooth}
We warn the reader, particularly those accustomed to working with adic spaces, that the de Rham stacks of the affinoid disc $\mathbb{D}_{\Q_p}$ and the affinoid torus $\mathbb{T}_{\Q_p}$ are \textit{not} cohomologically smooth over $\GSpec(\Q_p)$. This does not contradict  \Cref{sec:geom-prop-analyt-1-cohomological-smoothness-for-smooth-rigid-spaces}, since as Berkovich spaces these spaces are not rigid smooth. This also does not contradict the presentations
(cf. \Cref{sec:dagg-form-smooth-example-overconvergent-cohomology-of-rigid-disc})
 \[
\mathbb{D}_{\Q_p}^{\dR} = \mathbb{D}_{\Q_p}^{\leq 1}/\mathbb{G}_{a}^{\dagger}, \qquad \mathbb{T}_{\Q_p}^{\dR} = \mathbb{T}_{\Q_p}^\dagger/\mathbb{G}_{a}^{\dagger}
\]
(here, $\mathbb{T}_{\Q_p}^{\dagger}$ denotes the overconvergent torus\footnote{A perhaps more coherent, but too confusing, notation would have been $\DD_{\Q_p}^{=1}$.}) since, even though $\GSpec(\Q_p)/\mathbb{G}_a^\dagger$ is cohomologically smooth, $\mathbb{D}_{\Q_p}^{\leq 1}, \mathbb{T}_{\Q_p}^\dagger$ are not cohomologically smooth. In fact, if the de Rham stack of $\mathbb{T}_{\Q_p}$ were cohomologically smooth, since it is also prim over $\GSpec(\Q_p)$, pushforward along the structure morphism $\mathbb{T}_{\Q_p}^{\dR} \to \GSpec(\Q_p)$ would preserve perfect modules. But this not true, as it follows combining \cref{dRcoeff} with the fact that there exist vector bundles with integrable connection on $\mathbb{T}_{\Q_p}^{\dagger}$ with infinite-dimensional de Rham cohomology, \cite[Lemma A.1.12]{poineau2024convergence}. As we don't know whether there exist vector bundles with integrable connection on $\mathbb{D}_{\Q_p}^{\leq 1}$ with infinite-dimensional de Rham cohomology, we argue instead as follows to show that $f:\mathbb{D}_{\Q_p}^{\dR}\to \GSpec(\Q_p)$ is not suave. Consider the factorization of $f$
$$\mathbb{D}_{\Q_p}^{\dR}\overset{g}{\longrightarrow} \mathbb{D}_{\Q_p}^{\leq 1}/\mathring{\mathbb{D}}_{\Q_p}^{\leq 1/2}\overset{h}{\longrightarrow} \GSpec(\Q_p).$$
One has $\mathbb{D}_{\Q_p}^{\leq 1}/\mathring{\mathbb{D}}_{\Q_p}^{\leq 1/2}=\mathbb{D}_{\Q_p}^{\leq 1,\dR}/\mathring{\mathbb{D}}_{\Q_p}^{\leq 1/2,\dR}$. By \cref{sec:geom-prop-analyt-1-cohomological-smoothness-for-smooth-rigid-spaces}, we know that $\mathring{\mathbb{D}}_{\Q_p}^{\leq 1/2,\dR}\to \GSpec(\Q_p)$ is suave, and therefore $g$ is also suave (as its fibers are isomorphic to $\mathring{\mathbb{D}}_{\Q_p}^{\leq 1/2,\dR}$). If $f$ were suave, we would also have that $h$ is suave by \cite[Lemma 4.5.8]{heyer20246functorformalismssmoothrepresentations}. Now, consider the projection map $q\colon \mathbb{D}^{\leq 1}_{\Q_p}\to \mathbb{D}_{\Q_p}^{\leq 1}/\mathring{\mathbb{D}}_{\Q_p}^{\leq 1/2}$; it is suave being a quotient by a smooth rigid equivalence relation, and, as $\mathbb{D}^{\leq 1}_{\Q_p}$ is prim over $\GSpec(\Q_p)$ (being affinoid with induced structure), one has that $q_! 1\in \ob{D}(\mathbb{D}_{\Q_p}^{\leq 1}/\mathring{\mathbb{D}}_{\Q_p}^{\leq 1/2})$ is prim by  \cite[Lemma 4.5.16]{heyer20246functorformalismssmoothrepresentations}. But then, \cite[Lemma 4.5.16]{heyer20246functorformalismssmoothrepresentations} and the suaveness of $h$ would imply that $h_! q_! 1 =\mathcal{O}(\mathbb{D}^{\leq 1}_{\Q_p})$ is suave on $\GSpec(\Q_p)$ which is the same as being dualizable. This is absurd, since $\Q_{p,\solid}$ is Fredholm and the only dualizable objects in $\ob{D}(\Q_{p,\solid})$ are perfect modules.
\end{remark}

\begin{remark}
  \label{sec:-able-maps-first-chern-classes-de-rham-cohomology}
  Using excision for de Rham cohomology (\cref{sec:cohom-prop-de-excision}) one can also calculate the dualizing complex in \cref{sec:geom-prop-analyt-1-cohomological-smoothness-for-smooth-rigid-spaces} via the general paradigm in \cite[Theorem 5.7.7]{zavyalov2023poincaredualityabstract6functor}.
  This necessitates a theory of first Chern classes (as defined in \cite[Definition 5.2.4, Definition 5.2.8]{zavyalov2023poincaredualityabstract6functor}). We present here a stack-theoretic construction of a weak theory of first Chern classes for de Rham cohomology, as it will turn out useful for the analogous construction for Hyodo--Kato cohomology in \cref{sec:de-rham-fargues}. It then amounts to a standard computation to upgrade this to a theory of first Chern classes (cf. \cref{LemComputationP1dRFF}).

  \begin{enumerate}
   \item In the following, we work over $\Q_p$.
  Let $$c_1^{\dR}\colon (\mathbb{G}_m^\diamond)^\dR=\Gm/\mathbb{G}_m^\dagger\to B\Ga$$ be the composition of the map $\mathbb{G}_m^\dR\to B\mathbb{G}_{m}^{\dagger}$ (classifying the extension $0\to \mathbb{G}_m^\dagger\to \Gm\to \mathbb{G}_m^\dR\to 0$) with the logarithm $\log\colon \mathbb{G}^\dagger\to \Ga$.
  We note that $c_1^{\dR}$ is naturally a morphism of animated abelian groups as $\mathbb{G}_m$ is an \textit{animated abelian group stack}, i.e., it naturally extends to a finite coproduct preserving  functor $\Cat{FinFreeAb} \to \Cat{GelfStk}$ from finite free abelian groups to Gelfand stacks.
  Given a qfd arc-stack $X$ over $\Q_p$, we now get a map
  \[
    \mathrm{Map}(X,B \mathbb{G}_m^\diamond)\to \mathrm{Map}(X^{\dR},(B \mathbb{G}_m^\diamond)^{\dR}) \xrightarrow{B c_1} \mathrm{Map}(X^{\dR},B^2\Ga)\cong H^2(X^\dR,\mathcal{O})
  \]
  of animated abelian groups (with addition defined via the second factor), which defines a weak theory of first Chern classes (up to passing to $\ob{D}(\Z)$ by arc-sheafifying both sides as $\ob{D}(\Z)$-valued presheaves in $X$).
  Here, the first map is a morphism of animated abelian groups because $X\mapsto X^\dR$ preserves finite products.

  \item One checks that this abstractly defined first Chern classes agrees (on smooth partially proper rigid analytic varieties) with the usual first Chern class coming from the $d\log$-map $\Gm\to \Omega^1,\ \lambda\mapsto d\log(\lambda):=\frac{d\lambda}{\lambda}$. Let us give here an argument using the filtered de Rham stack (cf.\ \cref{de-rham-comparison-via-filtered-dR-stack}).

  Using e.g. Cartier duality (\cite[Theorem 4.1.13]{camargo2024analytic}), the cohomology of $B\mathbb{G}_{a}$ is naturally identified with  $\Q_p\oplus \Q_p dx [1]$ where $x$ is the coordinate of $\mathbb{G}_{a}$ (we think of $\Q_p dx$ as the cotangent bundle  $T^*_{\mathbb{G}_a}|_0$ at $0$).
  Similarly, the cohomology of $\mathbb{G}_{m}^{\dR}$ is isomorphic to $\Q_p\oplus \Q_p\frac{dt}{t}$ where $t$ is the coordinate of $\mathbb{G}_m$.
  Hence, we want to see that the map $c_1^{\dR}:\mathbb{G}_m^{\dR}\to B\mathbb{G}_a$ pulls back $dx$ to $\frac{dt}{t}$.
  Namely, we can choose the isomorphism $H^1(-,\mathcal{O})\cong \mathrm{Map}(-,B \mathbb{G}_a)$ so that $dx$ maps to the identity element of $B^2 \mathbb{G}_a$ under the morphisms
  \[
\mathrm{Map}(B^2 \mathbb{G}_a, B^2 \mathbb{G}_a)\leftarrow \mathrm{Map}(B \mathbb{G}_a, B \mathbb{G}_a)\cong H^1(B \mathbb{G}_a,\mathcal{O}).
\]
  For that, consider the filtered de Rham stack $\mathbb{G}_m^{\dR,+}$ over $S=\mathbb{A}^1/\mathbb{G}_m$.
  As $\mathbb{G}_m$ is a smooth group, its filtered de Rham stack can be written as the quotient of groups over $S$
 \[
 \mathbb{G}_m^{\dR,+}=(\mathbb{G}_m\times S)/\mathbb{G}^{\dagger}_m(-1)
 \]
  where $\mathbb{G}_m^{\dagger}(-1)$ sends a map $\GSpec A\to S$ corresponding to a virtual Cartier divisor $I\to A$ to the group
  \[\mathbb{G}_m^{\dagger}(-1)(\GSpec A\to S)=\ker( \mathbb{G}_m( A ) \to \mathbb{G}_m(A/\Nil^{\dagger} \otimes_A I))\cong 1+\Nil^{\dagger}(A)\cdot I.
  \]
  We have a logarithm map $\log\colon \mathbb{G}_m^{\dagger}(-1)\to \mathbb{G}_a(-1)$, producing a morphism of animated abelian group stacks over $S$
  \[
  c_1^{\dR,+}\colon \mathbb{G}_m^{\dR,+}\to B\mathbb{G}_a(-1),
  \]
  this map refines the Chern class $c_1^{\dR}$ of the de Rham stack, and lifts it to an element in the $\Fil^1$ of de Rham cohomology for the Hodge filtration. Now, since the $\Fil^1$ of de Rham cohomology of $\mathbb{G}_m$ is just given by differentials, the pullback of  $dx$ along $c_1^{\dR,+}$ is completely determined by the pullback along the zero section $0\colon \GSpec(\Q_p) \to  S=\mathbb{A}^1/\mathbb{G}_m$, and therefore by the Chern class of the Hodge stack $\mathbb{G}_m^{\mathrm {Hdg}} = \mathbb{G}_m/ \mathbb{G}_m^{\dagger}$, where the quotient is via the trivial action. The latter is given by the composite map
  \[
 c_1^{\mathrm {Hdg}}\colon  \mathbb{G}_m^{\mathrm {Hdg}} \to B\mathbb{G}_m^{\dagger} \to B \mathbb{G}_a
  \]
  where the first map is the natural projection on the classifying stack, and the second map is given by the logarithm $\log$.
  Finally, the cohomology of $B\mathbb{G}_m^{\dagger}$ is naturally identified with $\Q_p\oplus \Q_p\frac{dt}{t}$  (we think of $\Q_p \frac{dt}{t}$ as $T^*_{\mathbb{G}_m}|_1$) and it is clear that the map  $\log \colon \mathbb{G}_m^{\dagger}\to  \mathbb{G}_a$ induces the pullback map of differentials  $d\mathrm{log}\colon T^*_{\mathbb{G}_a}|_0\to T^*_{\mathbb{G}_m}|_1$ at the level of cotangent bundles at the identity, sending $dx$ to $\frac{dt}{t}$ as wanted.

  \item
  We note here a general procedure to construct (weak) theories of first Chern classes for cohomology theories defined by a ring stack $\mathcal{R}$ over a base $\mathcal{S}$ (via suitable versions of the general procedure of transmutation, cf.\ \cite[Remark 2.3.8]{FGauges}, say, for $p$-adic formal schemes or rigid analytic varieties).
  Namely, the essential datum is a morphism $$c_1^{\mathcal{R}}\colon \mathbb{G}_m^{\mathcal{R}}\to \mathbb{V}_{\mathcal{S}}(\mathcal{O}\langle 1\rangle)[-1])$$ of animated abelian group objects over the base $\mathcal{S}$ (that one might think of as a \textit{geometric Chern class}), where $\mathcal{O}\langle 1\rangle$ is the respective Tate twist for the cohomology theory, i.e., a distinguished invertible object on $\mathcal{S}$.
 For example, if $\mathcal{R}=\mathbb{G}_a^{\dR}=\Ga/\Ga^\dagger$, which recovers de Rham cohomology, one has $\mathcal{S}=\GSpec(\Q_p)$ and $\mathcal{O}\langle 1\rangle =\Q_p[2]$.
  \end{enumerate}
\end{remark}

\subsection{Cohomology of the analytic de Rham stack and de Rham cohomology}
\label{sec:coh-dR-stack-and-dR-coh}

One of course expects the cohomology of the analytic de Rham stack to be de Rham cohomology in good situations. We verify in this subsection that this is indeed the case.

\begin{proposition}
\label{cor:cohomology-de-rham-stack-and-de-rham-cohomology}
Let $K$ be a complete non-archimedean field over $\Q_p$, separable as a $\Q_p$-Banach space, and let $X$ be a $\dagger$-rigid space over $K$ which is Berkovich smooth over $K$ (in the sense of \Cref{DefEtaleSmooth}).
Then there is a natural isomorphism
$$
\Gamma(X^{\dRall/K},\mathcal{O}) \cong R\Gamma_{\rm dR}(X/K).
$$
The same holds for $X^{\dR}$ if $K$ is in addition qfd. 
\end{proposition}
Here, $R\Gamma_{\rm dR}(X/K)$ is the hypercohomology of the de Rham complex $\Omega^{\bullet}_{X/K}$ of $X$ over $K$, and $\Gamma(X^{\dRall/K},\mathcal{O})$ is the global sections of the structural sheaf of $X^{\rm DR/K}$ over $\GSpec K$.
Examples satisfying the assumption on $X$ are smooth partially proper rigid analytic variety, the overconvergent closed disc or variants.
\begin{proof}
The proof is a variant of the proof that de Rham cohomology of a smooth scheme over a characteristic zero field is computed by the cohomology of the structure sheaf on the infinitesimal site, found in \cite[Remark 3.7]{bhatt_dejong_crystalline_cohomology_and_de_rham_cohomology}. 

By \Cref{PropdeRhamBerkovich} and \Cref{thm:prim-descendable-berkovich}, the morphism $X \to X^{\dRall/K}$ is surjective, with \v{C}ech nerve the simplicial derived Berkovich space $X^{\dagger,\bullet}:= (X^{\times_{K} \bullet+1})^{\dagger_{\Delta X}}$ where $\Delta X\subset |X^{\times_{K} \bullet+1}|$ is the diagonal immersion.
The cosimplicial ring $\mathcal{O}_{X^{\dagger,\bullet}}$ admits a de Rham complex over $K$ (degree-wise an overconvergent version of the classical continuous de Rham complex rigid analytic varieties over $K$), giving rise to the $\Delta \times \mathbb{N}^{\rm op}$-indexed diagram $(\Omega_{X^{\dagger,n}/K}^{\leq m})_{[n] \in \Delta, m \in \mathbb{N}^{\rm op}}$ of complexes of $K$-linear sheaves on the analytic site of $X$, with obvious arrows.
It contains as subdiagrams the cosimplicial ring $\mathcal{O}_{X^{\dagger,\bullet}}$ corresponding to $\Delta \times \{0\}$ and the de Rham complex $\Omega_{X/K}^{\leq \bullet}$ corresponding to $\{[0]\} \times \mathbb{N}^{\rm op}$.
The inclusions of these subdiagrams give rise to natural maps of $K$-linear sheaves on the analytic site of $X$
$$
 \underset{[n] \in \Delta} \lim \mathcal{O}_{X^{\dagger,\bullet}}  \leftarrow \underset{([n],m) \in\Delta \times \mathbb{N}^{\rm op}} \lim \Omega_{X^{\dagger,n}/K}^{\leq m} \to \Omega_{X/K}^{\bullet}.
 $$
 The cohomology on the analytic site of $X$ of the left term computes the cohomology of the analytic de Rham stack relative to $K$, as we just recalled, while the cohomology on $X$ of the right term is the de Rham cohomology of $X$ over $K$.
 Hence one is done if one can show that the two maps are quasi-isomorphisms. Observe that the limit defining the middle term can be computed in two ways, by first taking a limit over $\mathbb{N}^{\rm op}$ and then over $\Delta$, or the other way around:
  $$
 \underset{([n],m) \in\Delta \times \mathbb{N}^{\rm op}} \lim \Omega_{X^{\dagger,n}/K}^{\leq m} =  \underset{([n] \in\Delta } \lim   \Omega_{X^{\dagger,n}/K}^{\bullet}  = \underset{m} \varprojlim \left(\underset{[n] \in \Delta} \lim \Omega_{X^{\dagger,n}/K}^{\leq m} \right).
 $$
 Thus it suffices to show the following statements:
 \begin{enumerate}
 \item For any arrow $\alpha\colon [n_1] \to [n_2]$ in $\Delta$, the induced map $\alpha: \Omega_{X^{\dagger,n_1}/K}^\bullet \to \Omega_{X^{\dagger,n_2}/K}^\bullet$ is a quasi-isomorphism.
\item For any integer $m\geq 1$, $\underset{[n]\in \Delta} \lim \Omega_{X^{\dagger,n}/K}^m=0$.
\end{enumerate}
For (1), it suffices to deal with the map $[0] \to [n], 0 \mapsto 0$, for $[n] \in \Delta$. Analytic locally on $X$, we can find (by definition of standard Berkovich smoothness) a morphism $X \to \mathbb{A}_K^{d,\mathrm{an}}$ giving an presentation as in \cref{x2edaj3}.
Via Taylor expansions, this morphisms yields an isomorphism 
$$
X^{\dagger,n} \cong X \times \mathbb{G}_a^{\dagger, dn}.
$$
Now the claim follows from Poincar\'e's lemma for $\mathbb{G}_a^{\dagger}$, which follows from \cite[Lemma 26]{TammeLazardIso} after taking colimits of open discs of radius $r$ as $r\to 0$.

Since for any $[n] \in \Delta$ and $m \in \mathbb{N}$
$$
\Omega_{X^{\dagger,n}/K}^m = \mathcal{O}_{X^{\dagger,n}} \otimes_{\mathcal{O}_{X^{n+1}}} \Omega_{X^{n+1}/K}^m,
$$
part (2) can be proven as in \cite[Lemma 2.15]{bhatt_dejong_crystalline_cohomology_and_de_rham_cohomology} via an argument using local charts of $X$.  

\end{proof}

\begin{remark}
\label{de-rham-comparison-via-filtered-dR-stack}
Another way to prove \Cref{cor:cohomology-de-rham-stack-and-de-rham-cohomology}, which we learnt from \cite[Theorem 2.3.6]{FGauges}, uses the \textit{filtered analytic de Rham stack}. Let us recall the definition of the filtered analytic de Rham stack, which can be constructed similarly to \cite[Definition 5.2.2]{camargo2024analytic}.
  We note first that the qfd Gelfand stack $\mathbb{A}^{1, \mathrm{an}}_{\Q_p}/\mathbb{G}_m^{\rm an}$ classifies generalized Cartier divisors on a qfd Gelfand ring $A$, i.e., pairs of a line bundle $L$ with an $A$-linear map $s\colon L\to A$.
  Now, given any morphism $g\colon W\to Z$ of qfd Gelfand stacks, one can consider its filtered de Rham stack $W^{\dRall/Z,+}$ relative to $Z$ as the $!$-sheafification of the functor sending a qfd Gelfand ring $A$ over $Z$ with a generalized Cartier divisor $s\colon L\to A$ to\footnote{The natural structure of $\mathrm{Nil}^\dagger(A)\otimes_A L\to A$ as an animated $A$-algebra can be deduced from \cite[Remark 9.13]{course_algebraic_d_modules}.}
  \[
    W(\mathrm{cofib}(\mathrm{Nil}^\dagger(A)\otimes_AL\to A)).
  \]
  In particular, there exists a natural morphism $W^{\dRall/Z,+}\to Z\times_{\GSpec(\Q_p)}\mathbb{A}^{1, \mathrm{an}}_{\Q_p}/\mathbb{G}_m^{\rm an} $.
  Assuming that $g$ is $\dagger$-formally smooth, the natural map $W\times_{\GSpec(\Q_p)} \mathbb{A}^{1}_{\Q_p}/\mathbb{G}_m\to W^{\dRall/Z,+}$ is an epimorphism.
  The pullback of $W^{\dRall/Z,+}$ to $Z\times_{\GSpec(\Q_p)}\GSpec(\Q_p)/\mathbb{G}_m^{\rm an}$ is the relative Hodge stack $W^{\mathrm{Hdg}/Z}$ of $W$ over $Z$, while its pullback to $Z = Z\times_{\GSpec(\Q_p)} \mathbb{G}_m^{\rm an}/\mathbb{G}_m^{\rm an}$ is the relative de Rham stack $W^{\dRall/Z}$.

  Back to the notations of \Cref{cor:cohomology-de-rham-stack-and-de-rham-cohomology}, the idea is to prove the stronger statement that the pushforward along $X^{\dRall/K,+}\to  \mathbb{A}^{1,\rm an}_{K}/\mathbb{G}_m^{\rm an}$ calculates filtered de Rham cohomology.
  Here we use the map of analytic stacks $ g\colon \mathbb{A}^{1,\rm an}_{K}/\mathbb{G}_m^{\rm an}\to \mathbb{A}^{1,\rm alg}_{K}/\mathbb{G}_m^{\rm alg}$ from the analytic to the algebraic stacks to compare with classical filtered modules.
  One can prove that, at least for complete filtered objects, the pullback along $g$ is fully faithful; this amounts to proving that the pullback along the map of analytic stacks $g'\colon \widehat{\mathbb{A}}^{1}_{K}/\mathbb{G}_m^{\rm an}\to \widehat{\mathbb{A}}^{1}_{K}/\mathbb{G}_m^{\rm alg}$ is fully faithful (with $\widehat{\mathbb{A}}^{1}_{K}$ to be the formal completion at $0$ of the affine line seen as an Ind-affinoid analytic stack), which one can show by noticing that $g'$ is prim and that $g'_* 1=1$.
  
  More precisely, one proves that the pushforward of the unit along $f^{\dRall,+}\colon X^{\dRall/K,+}\to \mathbb{A}^{1,\rm an}_K/\mathbb{G}^{\rm an}_K\times |X|_{\Betti}$ is given by the Hodge-filtered de Rham complex, seen as an object on $\mathbb{A}^{1,\rm an}_K/\mathbb{G}^{\rm an}_K\times |X|_{\Betti}$.
  To show this, one first considers the pullback square
  \[
  \begin{tikzcd}
  X^{\rm Hdg/K} \ar[r,"f^{\rm Hdg}"] \ar[d] & \GSpec K /\mathbb{G}^{\rm an}_K \times |X|_{\Betti} \ar[d, "\iota"] \\ 
  X^{\dRall/K,+}   \ar[r,"f^{\dRall,+}"] & \mathbb{A}^{1,\rm an}_K/\mathbb{G}^{\rm an}_K \times |X|_{\Betti}.
  \end{tikzcd} 
  \] 
  Since $\iota$ is suave (being the inclusion of a Cartier divisor, and hence a local complete intersection), we have suave base change for lower $*$ and obtain that 
  \[
  \iota^* f^{\dRall,+}_* 1= f^{\rm Hdg}_* 1. 
  \]
  Using Cartier duality between the analytic Hodge stack of $X$ and the analytic cotangent bundle \cite[Theorem 4.3.13]{camargo2024analytic}, one deduces that 
  \[
  f^{\rm Hdg}_* 1= \bigoplus_{i=0}^d \Omega^i_X (i) [-i]
  \]
  where $d$ is the dimension of $X$ and the twist $(i)$ is the $i$-fold tensor product of the standard representation of $\mathbb{G}_m^\an$.
  A local computation using \'etale charts proves that the filtration of  $f^{\dRall,+}_* 1$ is  complete, exhaustive and supported in degrees $[0,d]$.
  Since the degree $i$ term $\Omega^i_X (i)$ sits in cohomological degree $i$ for the natural $t$-structure on sheaves on the topological space $|X|$, the Beilinson $t$-structure on complete filtered objects yields that $f^{\dRall,+}_* 1$ is given by a complex of $K$-vector spaces on $|X|$ of the form 
  \begin{equation}\label{eqComplexDRComplexfiltered}
  \mathcal{O}_X \xrightarrow{\delta} \Omega^1_X \xrightarrow{\delta} \cdots \xrightarrow{\delta} \Omega^d_X,
  \end{equation}
  sitting in cohomological degrees $[0,d]$.
  We  are left to show  that \eqref{eqComplexDRComplexfiltered} is given by the de Rham complex of $X$.
  To prove that, one shows by an explicit computation via local \'etale coordinates that the map $\mathcal{O}_X \xrightarrow{\delta} \Omega^1_X$ is a  derivation of solid $K$-vector spaces.
  On the other hand, \eqref{eqComplexDRComplexfiltered} has the structure of a cdg algebra (being a commutative algebra for the heart of Beilinson $t$-structure), and   by the universality of the de Rham complex there is a unique morphism of cdg algebras from $(\Omega^{\bullet}_X,d)\to (\Omega^{\bullet}_X, \delta)$. One then checks that this map is an isomorphism by a local computation using \'etale charts.

 For more details on the argument we refer the reader to \cite[Theorem 2.3.6]{FGauges} and to \cite[\S 9.3.2]{course_algebraic_d_modules}: these references work in the algebraic setting, but the same works  in the analytic case.
\end{remark}

\begin{remark}\label{dRcoeff}
As expected, one can also extend \Cref{cor:cohomology-de-rham-stack-and-de-rham-cohomology} to coefficients. More precisely, keeping the notation of \textit{loc. cit.}, one can show that
\begin{equation*}\label{vbwc}
 \VB(X^{\dRall/K})\cong\{\text{vector bundles with integrable connection on }X/K\},
\end{equation*}
and given $(\mathcal{E}, \nabla)$ a vector bundle with integrable connection on $X$, denoting by $(\mathcal{E}, \nabla)^{\dRall}$ the corresponding vector bundle on $X^{\dRall/K}$, we have a natural isomorphism
$$\Gamma(X^{\dRall/K}, (\mathcal{E}, \nabla)^{\dRall})\cong R\Gamma_{\dR}(X/K, (\mathcal{E}, \nabla)).$$
\end{remark}
 A similar statement holds for $X^{\dR}$ if $K$ is in addition qfd. This can be shown to follow from the fact that there is a natural fully faithful embedding of the category $\ob{D}(X^{\dRall/K})$ into the category of quasi-coherent sheaves on the \textit{algebraic de Rham stack} of $X$ relative to $K$, cf. \cite[Definition 5.1.1, Proposition 5.2.11]{camargo2024analytic} (which, in addition, induces an equivalence on perfect modules/vector bundles), together with a comparison between the category of perfect modules/vector bundles on the algebraic de Rham stack of $X$ relative to $K$, and the category of crystals in perfect modules/vector bundles on the infinitesimal site of $X/K$ (this can be checked using a by now standard argument, cf. \cite[Theorem 3.4]{bhatt2023crystalschernclasses}, \cite{GuoCrystalline}).

\subsection{An abstract result in $6$-functor formalisms}

In this small subsection we prove one technical result on $6$-functor formalisms that will be used  later. The reader can use this subsection as blackbox.

\begin{lemma}\label{TechnicalLemma6Functors}
Let $(\mathcal{C},E)$ be a geometric setup and $\ob{D}$ a presentable $6$-functor formalism on  $(D,E)$. Let $S\in \mathcal{C}$ and consider a sequential limit diagram of $!$-able maps over $S$
\[
X_{\infty}\to \cdots \to X_{n+1}\to X_n\to \cdots \to X_0 
\]
with arrows $f_{m\to n}\colon X_m\to X_n$ for $m\geq n \in \N\cup\{\infty\}$. Suppose that the following conditions hold: 

\begin{itemize}
\item[(a)] All the maps $f_{m\to n}$ are prim for $m\geq n\in \N\cup \{\infty\}$. 

\item[(b)] Given $n\in \N$  the natural map $\varinjlim_{m\geq n} f_{m\to n,*}1\to  f_{\infty\to n,*} 1$ of objects in $\ob{D}(X_n)$ is an isomorphism. 

\item[(c)] There is a prim map  $g\colon Y_{\infty}\to X_{\infty}$ satisfying universal $\ob{D}^*$- and $\ob{D}^!$-descent such that we have an equivalence in the kernel category $\ob{K}_{\ob{D},S}$ (see \cite[Definition D.4.1]{heyer20246functorformalismssmoothrepresentations})
\[
Y_{\infty}\times_{X_{\infty}} Y_{\infty} = \varprojlim_{n} Y_{\infty}\times_{X_n} Y_{\infty}. 
\]

\end{itemize}

Then the following hold: 

\begin{enumerate}

\item  We have $X_{\infty}=\varprojlim_{n} X_{n}$ in the kernel category $\ob{K}_{\ob{D},S}$. In particular,
\[
\ob{D}(X_{\infty})=\varprojlim_{n} \ob{D}_!(X_n).
\]

\item An object $P\in \ob{D}(X_{\infty})$ is suave over $S$ if and only if $P_n:=f_{\infty\to m,!} P\in \ob{D}(X_n)$ is suave over $S$ for all $n$. In that case,   the suave dual is given by $\ob{SD}_{S}(P) = \varinjlim_{n} f_{\infty\to n}^{\flat} \ob{SD}_S(P_n) $ where $f_{\infty\to n}^{\flat}$ is the left adjoint of $f_{\infty\to n,!}$ (naturally isomorphic to $ \delta_{\infty\to n} \otimes f^*_{\infty\to n}$ where $\delta_{\infty\to n}$ is the codualizing sheaf of $f_{\infty\to n}$).
\end{enumerate}

\end{lemma}
\begin{proof}
Let us first prove (1). By definition of a limit in $2$-categories (\cite[Definition D.4.1]{heyer20246functorformalismssmoothrepresentations}),  we need to show that for all $Z\in \mathcal{C}_{/S,E}$, the natural map of morphism categories 
\begin{equation}\label{eq01j3o2d3d}
\ob{Fun}_{\ob{K}_{\ob{D},S}}(Z, X_{\infty})\to \varprojlim_{n} \ob{Fun}_{\ob{K}_{S}}(Z, X_{n})
\end{equation}
is an equivalence. By definition of the kernel category, $\ob{Fun}_{\ob{K}_{S}}(Z, Y)= \ob{D}(Y\times_S Z)$ and the transition maps of \cref{eq01j3o2d3d} are given by lower $!$-maps.
 Notice that the conditions (a) and (b) are preserved under pullbacks along maps $h\colon Z\to S$ in $\mathcal{C}$ thanks to proper base change.
 If $h$ is $!$-able,  condition (c) is also preserved along the pullback map of kernel categories  $h^*\colon \ob{K}_{\ob{D},S}\to \ob{K}_{\ob{D},Z}$ as this functor is a right adjoint by \cite[Lemma 4.2.7]{heyer20246functorformalismssmoothrepresentations}, and thus preserves limits.
  Therefore, after taking base changes, it suffices to show that the natural map 
\[
F_!\colon \ob{D}(X_{\infty})\to \varprojlim_{n} \ob{D}_!(X_n)
\]
along lower $!$-maps is an equivalence of categories (note that the objects in $\ob{K}_{\ob{D},S}$ are given by $!$-able maps to $S$). 

The functor $F_!$ is given by the limit of the functors $f_{\infty\to n,!}\colon  \ob{D}(X_{\infty})\to \ob{D}_!(X_n)$.
 As each $f_{\infty\to n}$ is prim, $f_{\infty\to n,!}$ has a left adjoint $f_{\infty\to n}^{\flat}$ and   by \cite[Lemma D.4.7 (i)]{heyer20246functorformalismssmoothrepresentations} $F_!$ has a left adjoint  given by $F^{\flat}=\varinjlim_{n} f_{\infty\to n}^{\flat}$. 

To prove part (1) of the lemma it suffices to show that both the unit and counit of the adjunction is an equivalence. 

Since the $f_{m\to n}\colon X_m\to X_n$ maps are prim,  the map $f_{m\to n,!}\colon X_m\to X_n$ in $\ob{K}_{\ob{D},S}$ is a right adjoint.
 Passing to duals one sees that $f_{m\to n}^*\colon X_n\to X_m$ is a left adjoint, and then passing to its right adjoint we obtain the map $f_{m\to n,*}\colon X_n\to X_m$  whose composite with $\ob{D}$ is the lower $*$-map.
 The passage from $f_{m\to n,!}$ to $f_{m\to n,*}$ is the following composition of equivalences of $2$-categories
\[
\ob{K}_{\ob{D},S}^{R} \xrightarrow{\sim} (\ob{K}_{\ob{D},S}^{L})^{\op} \xrightarrow{\sim} (\ob{K}_{\ob{D},S}^{R,\ob{co},\op})^{\op}= \ob{K}_{\ob{D},S}^{R,\ob{co}}
\]
where the first equivalence corresponds to passing to the dual  (since all objects in $\ob{K}_{\ob{D},S}$  are self duals, see \cite[Proposition 4.1.4]{heyer20246functorformalismssmoothrepresentations}), and the second equivalence is the passage to the right adjoint of \cite[Theorem D.3.17]{heyer20246functorformalismssmoothrepresentations}.
 Therefore,  to show that $X_{\infty}=\varprojlim_{n} X_n$, we can either show the lower $!$ or lower $*$ maps.
 In other words, it suffices to show that the functor 
\[
F_*\colon \ob{D}(X_{\infty})\to \varprojlim_{n} \ob{D}_*(X_n)
\]
given by lower $*$-maps is an equivalence.
 In this case, $F_*=\varprojlim_{m}(f_{m\to n,*})$ and its left adjoint is $F^*=\varinjlim_m f^{*}_{m\to n}$.

\textit{Unit.} We need to prove that the map $ 1\to  F_{*} F^{*}$ is an equivalence as endofunctors of  $\varprojlim_{n} \ob{D}_*(X_n)$.
 By the projection formula, since all maps are prim, it suffices to evaluate this at the unit object of $\varprojlim_{n} \ob{D}_*(X_n)\cong\varinjlim_{n} \ob{D}^*(X_n)$.
 The unit object of this category in the presentation $\varprojlim_{n} \ob{D}_*(X_n)$ is precisely the cocartersian section given by $(\varinjlim_{m\geq n} f_{m\to n,*} 1)_{n}$.
 Therefore, the unit being an equivalence is precisely the equivalence 
\[
\varinjlim_{m\geq n} f_{m\to n,*} 1 \xrightarrow{\sim} f_{\infty\to n,*} 1
\] 
for all $n\in \N$, which holds by condition (b).

\textit{Counit.} We need to prove that the map $ F^{*} F_{*}\to 1$ is an equivalence as endofunctors of $\ob{D}(X_{\infty})$.
  Since the map $g\colon Y_{\infty}\to X_{\infty}$ satisfies universal  $\ob{D}^!$-descent and is prim, we can prove that after right composing with the functor $g_{*}\colon \ob{D}(Y_{\infty})\to \ob{D}(X_{\infty})$.
 Since $g$ satisfies universal $\ob{D}^*$-descent, we can also prove this after left composing with $g^*$.
 Therefore, it suffices to show that the  natural map of endofunctors  $  g^*F^*F_* g_*\to g^*g_*$ of $\ob{D}(Y_{\infty})$ is an equivalence.
  For all $n \in \N\cup \{\infty\}$ consider the pullback diagram 
\[
\begin{tikzcd}
Y_{\infty}\times_{X_n} Y_{\infty}  \ar[r,"\pi_{2,n}"] \ar[d,"\pi_{1,n}"] & Y_{\infty} \ar[d,"f_{\infty\to n}\circ g"] \\ 
Y_{\infty} \ar[r, "f_{\infty\to n}\circ g"] & X_{n}.
\end{tikzcd}
\]
Applying limits along $n$, by proper base change we get a commutative diagram 
\[
\begin{tikzcd}
\varprojlim_{n} \ob{D}_*(Y_{\infty}\times_{X_n} Y_{\infty})  \ar[d,"\varprojlim_{n} \pi_{1,n,*}"] & \ar[l,"\varinjlim_{n} \pi_{2,n}^*"'] \ob{D}(Y_{\infty}) \ar[d,"F_*g_*"]\\
\ob{D}(Y_{\infty})  & \ar[l,"g^*F^*"'] \varprojlim_n \ob{D}_*(X_n)
\end{tikzcd}
\]
Similarly, we have a commutative diagram 
\[
\begin{tikzcd}
\ob{D}(Y_{\infty}\times_{X_{\infty}} Y_{\infty})   \ar[d,"\pi_{1,\infty,*}"] & \ar[l,"\pi_{2,\infty}^*"'] \ob{D}(Y_{\infty}) \ar[d,"g_*"] \\ 
\ob{D}(Y_{\infty})  & \ar[l, "g^*"'] \ob{D}(X_{\infty}).
\end{tikzcd}
\]
But by hypothesis (c), the natural map $\ob{D}(Y_{\infty}\times_{X_{\infty}}Y_{\infty})\to \varprojlim_n \ob{D}(Y_{\infty}\times_{X_{n}}Y_{\infty})$ is an equivalence,  this yields the equivalence of the counit, and finishes the proof of part (1).

Part (2) follows formally from \cite[Corollary D.4.9]{heyer20246functorformalismssmoothrepresentations}.
 Indeed, by part (1) we have an equivalence $X_{\infty}=\varprojlim_n X_n$.
 Let $P\in \ob{D}(X_{\infty})$ be a suave object, then, since the map $f_{\infty\to n}\colon X_{\infty}\to X_n$ is prim, the object $f_{\infty\to n,!} P\in \ob{D}(X_{n})$ is suave for all $n\in \N$.
  Conversely, if $f_{\infty\to n,!} P\in \ob{D}(X_n)$ is suave for all $n\in \N$, then   \cite[Corollary D.4.9]{heyer20246functorformalismssmoothrepresentations} implies that $P$ is itself suave. The formula for the suave dual of $P$ is a consequence of \cite[Proposition D.4.8]{heyer20246functorformalismssmoothrepresentations}.
\end{proof}

\begin{remark}
A statement closely related to \Cref{TechnicalLemma6Functors} is \cite[Proposition 1.35]{mikami2025finiteness}.
\end{remark}

\subsection{Cohomology of analytic de Rham stacks: some properties}
We establish some basic properties of the cohomology of the analytic de Rham stack. We start with the fact that it is finitary and disc-invariant.

\begin{proposition}\label{LemComputationdeRhamLimit}
Let $A$ be a separable Gelfand ring and let $B_{\infty}=\varinjlim_n B_n$ be a countable colimit of separable $A$-algebras. Set $X=\GSpec (A)$ and $Y_n=\GSpec(B_n)$ for $n\in \N\cup\{\infty\}$. Consider the natural maps of  de Rham stacks  $f_{n}\colon Y_n^{\dRall}\to X^{\dRall}$ for $n\in \N\cup\{\infty\}$. Then the natural map
\begin{equation}\label{eq03k1emfd}
\varinjlim_n f_{n,*}1 \to f_{\infty,*} 1
\end{equation}
is an equivalence.
\end{proposition}
\begin{proof}
The passage to the de Rham stack is invariant under uniform completion, thus, thanks to \cref{xsh29k} we can assume without loss of generality that $A$ is a countable colimit of rational localizations of affine spaces over $\Q_p$, that $A\to B_0$ is a rational localization of an affine space over $A$, and that for $n\geq 1$ the map $B_n\to B_{n+1}$ is also a rational localization of an affine space.  In particular, all the algebras are $\dagger$-formally smooth, and the maps $Y_n\to Y_n^{\dRall}$ and $X\to X^{\dRall}$ are epimorphisms.

 Let $X^{\bullet}\to X$ be the \v{C}ech nerve of $X\to X^{\dRall}$ and $f_{n}^{\bullet}\colon Y_{n}^{\dRall/X^{\bullet}}\to X^{\bullet}$ the base change of $Y_n^{\dRall}\to X^{\dRall}$ along $X^{\bullet}$ for $n\in \N\cup\{\infty\}$. Thus, to prove the claim it suffices to show the following claims: 
 
 \begin{itemize}

\item[(a)] For all $\alpha\colon [a]\to [b]$ in $\Delta^{\op}$ the natural map 
\[
\alpha^*  f_{\infty,*}^{[b]} 1 \to f_{\infty,*}^{[a]}1
\] 
 is an equivalence in $\ob{D}(X^{[a]})$.
 
 \item[(b)] The map $\varinjlim_{n\in \N} f_{n,*}^{[0]} 1\to f_{\infty,*}^{[0]} 1$ is an equivalence. 
 
 \end{itemize}
 
Indeed, the claim (a) implies that the $\Delta$-section $(f_{\infty,*}^{[a]} 1)_{[a]\in \Delta}$ in $\ob{D}^*(X^{\bullet})$ is cocartesian and so that it gives rise to the object $f_{\infty,*} 1$. In particular, if $g\colon X\to X^{\dR}$, we have base change 
\[
g^*f_{\infty,*} 1 = f_{\infty,*}^{[0]} 1.
\]
Then, as $g$ is an epimorphism, to show that \cref{eq03k1emfd} is an equivalence it suffices to pullback along $g^*$. The maps $f_{n}$ are prim thanks to \cref{LemmLqfd}, and $f_{n,*}$ satisfies base change, in which case we have to prove precisely claim (b).

Now, (a) and (b) follow from the fact that if $f\colon A\to B$ is of the form as in \cref{xsh29k}, and $f^{\dRall/A}\colon (\GSpec(B))^{\dRall/\GSpec(A)}\to \GSpec(A)$, then $f^{\dRall/A}_* 1$ is given by the de Rham complex of $B$ over $A$, cf. \Cref{cor:cohomology-de-rham-stack-and-de-rham-cohomology}: the formation of de Rham complexes commutes with filtered colimits and base change in the base.
\end{proof}

\begin{lemma}
  \label{sec:-able-maps-disc-invariance-de-rham-cohomology}
  Let $f\colon Y\to \Marc(\Q_p)$ be given by $Y=\Marc(\Q_p\langle T\rangle), Y=\mathbb{D}^{\circ}_{\Q_p}$ or $Y=\A^{1}_{\Q_p}$.
  Then $f^{\dR,\ast}\colon \ob{D}(\GSpec(\Q_p))\to \ob{D}(Y^\dR)$ is fully faithful.
  Moreover, the same assertion holds true after any base change of $f$ to a qfd arc-stack. 
\end{lemma}
\begin{proof}
  Writing the open unit disc $\mathbb{D}^{\circ}_{\Q_p}$ or the affine line $\A^1_{\Q_p}$ as a union of closed disc (and thus $\ob{D}(-)$ as an inverse limit along $\ast$-pullback), the claim reduces to the case $Y=\Marc(\Q_p\langle T\rangle)$ (also after any base change).
  In this case, we know from  \cref{LemmLqfd} that $Y^{\dR}\to \GSpec(\Q_p)$ is cohomologically proper, so that the formation of $f^\dR_\ast$ commutes with any base change.
  Using the projection formula (again valid by cohomologically properness), we reduce therefore to checking that the natural map $1\to f^{\dR}_{\ast}(1)$ is an isomorphism, which is the classical fact that the overconvergent de Rham cohomology of the closed unit disc is given by $\Q_p$ in degree $0$, and vanishes in higher degrees.
\end{proof}

We note that cohomology of the de Rham stack can be calculated via excision. This can be justified geometrically as we explain now.

\begin{remark}
  \label{sec:cohom-prop-de-excision}
  By \cref{ExamCondensedAnimadeRham} and a qfd version of \cref{ExaTopologicalSpaceCondensedAnima} there exists a natural map $Z^\dR\to (\underline{Z_{\mathrm{cond}}})^\dR\cong Z_{{\mathrm{cond}},\Betti}$ of qfd Gelfand stacks for any qfd arc-stack $Z$ (here $(-)_{\Betti}$ refers to the qfd version of the Betti stack).
  In particular, one can pullback any excision triangle on $Z_{{\mathrm{cond}}}$ to an excision triangle on the de Rham stack.
  More concretely, if $Z=X^\diamond$ for a derived Berkovich space, then $Z_{{\mathrm{cond}}}$ is the condensed set associated with a locally compact Hausdorff space, and one can use open-closed decompositions of the latter. 
\end{remark}

We next turn to a result, \Cref{PropApproxFFStk}, giving a precise meaning to the idea that the analytic de Rham stack of a diamond, and its category of quasi-coherent sheaves, can be computed via an arbitrary  approximation of the space by smooth rigid varieties, as long as one can control the dimensions. To implement it, we need a good control on the map $Y \to Y^{\dR/X}$ for a rigid smooth morphism $Y\to X$. The proof of \cref{sec:geom-prop-analyt-1-cohomological-smoothness-for-smooth-rigid-spaces} shows that the map $Y\to Y^{\dR/X}$ is descendable (as this reduces via \'etale localization to the descendability of $\mathbb{A}^{1,\dR}_{\Q_p}\to \GSpec(\Q_p)$, which was proven in \cref{LemmProetaleShierkdR}). This argument does however not control the index of descendability. We now show that the index is controlled by the relative dimension.

\begin{proposition}
  \label{sec:cohom-smooth-maps-descendability-for-smooth-maps}
  Let $f\colon Y\to X$ be a Berkovich smooth morphism of qfd derived Berkovich spaces over $\Q_p$, which is of relative dimension $\leq d$ for some $d\in \N$. Then the map $Y\to Y^{\dR/X}$ is descendable of index $\leq d+1$.
\end{proposition}
\begin{proof}
  The statement   essentially follows from the Poincar\'e lemma for the global functions of $\mathbb{G}_a^{\dagger}$, cf. \cref{sec:cohom-smooth-maps-descendable-for-complex-manifolds}.
   To make this precise, we will use the filtered analytic de Rham stack introduced, for a morphism $g: W \to Z$ of qfd Gelfand stacks, in \Cref{de-rham-comparison-via-filtered-dR-stack}.
   We consider it in the case $g\colon W:=Y\to Z:=Y^{\dR/X}$. Note that $h:=g^{\dR/Z,+}\colon W^{\dR/Z,+}\to Z\times_{\GSpec(\Q_p)} \mathbb{A}^1_{\Q_p}/\mathbb{G}_m$ is an isomorphism over the open $Z\times_{\GSpec(\Q_p)} \mathbb{G}_m/\mathbb{G}_m\cong Y^{\dR/X}$.
  In particular, we can see that the object $1_{Y^{\dR/X}}$ naturally acquires a filtration (or rather upgrades naturally to an object $\mathcal{E}$ in $Z\times_{\GSpec(\Q_p)}\mathbb{A}^1_{\Q_p}/\mathbb{G}_m$) via the pushforward along $h$ (we note that this object is different from $1_{Z\times_{\GSpec(\Q_p)} \mathbb{A}^1_{\Q_p}/\mathbb{G}_m}$, which would correspond to the trivial filtration on $1_{Y^{\dR/X}}$).
  The prim morphism $g$ is not Berkovich smooth, but an inverse limit of such along open immersions (at least after base change along the descendable map $Y\to Y^{\dR/X}$).
  This is sufficient to conclude that the relative Hodge cohomology of $g$ can be calculated as in the Berkovich smooth case, i.e., via $g_\ast(\Omega^i_{W/Z})$.
  Now, by the yoga of filtrations via sheaves on $\mathbb{A}^1_{\Q_p}/\mathbb{G}_m$, we see that the associated graded on $1_{Z}$ (defined by $\mathcal{E}$) is given by the graded Hodge cohomology of $g$.
  In particular, this filtration is complete and exhaustive, and of length $\leq d+1$.
  Altogether, we can conclude that $g\colon Y\to Y^{\dR/X}$ is descendable of index $\leq d+1$.
  Indeed, as the unit $1_{Y^{\dR/X}}$ has a filtration with at most $d+1$ non-zero gradeds given by objects pushed forward along $g$, one checks that the $d+1$-fold composition of maps in $\ob{D}(Y^{\dR/X})$, whose pullback to $Y$ is zero, vanishes.
  Applying this to the map $\mathrm{fib}(1_{Y^{\dR/X}}\to g_\ast(1_{Y}))\to 1_{Y^{\dR/X}}$ shows the claim.
\end{proof}

Over fields, we can improve \cref{sec:cohom-smooth-maps-descendability-for-smooth-maps} to the non-smooth case by d\'evissage. 

\begin{proposition}\label{DescentRigidSpacedRStak}
  Let $K$ be a non-archimedean field, which is qfd over $\Q_p$, and let $X/K$ be a  $\dagger$-rigid space of Krull dimension $\leq d$.
  Then the map $f\colon X\to X^{\dR/K}$ is prim and descendable of index $\leq \frac{(d+1)(d+2)}{2}$.
  The same holds for the relative big de Rham stack and an arbitrary non-archimedean field $K$.
\end{proposition}
\begin{proof}
  We follow the proof of \cite[Theorem 5.4.1]{camargo2024analytic} keeping track the descendability index.
  First, we can assume without loss of generality that $X$ is a reduced $\dagger$-rigid space.
  By \cite[Lemma 1.2.1]{MR1697371} we can  find a filtration of $X$ by   Zariski closed  subspaces $X=Z_0\supset Z_1\supset \cdots \supset Z_d$   with $C_i:=Z_i\backslash Z_{i+1}$  such that the reduced $\dagger$-rigid space structure on $C_i$ is Berkovich smooth over $K$ of dimension $\leq d$.
  In particular, by \Cref{sec:cohom-smooth-maps-descendability-for-smooth-maps} each map   $C_i\to C_i^{\dR}$ is descendable of index $\leq d$.
  If $X^{\dagger_{C_i}}$ is the overconvergent neighbourhood of $C_i$ in $X$, the factorization $C_i\to X^{\dagger_{C_i}}\to C_i^{\dR}$ also shows that $X^{\dagger_{C_i}}\to C_i^{\dR}$ is descendable of index $\leq d+1$. The proposition follows  by induction from \cref{LemmaDevisageDescendable} below using that $(d+1)+d+\ldots +1=\frac{(d+1)(d+2)}{2}$.
\end{proof}

\begin{lemma}\label{LemmaDevisageDescendable}
Let $\mathcal{C}$ be  a stable symmetric monoidal category and let $1\to A$ be an idempotent algebra. Let $j^*\colon \mathcal{C}\to \mathcal{C}(U):=\mathcal{C} /\Mod_A(\mathcal{C})$ and $\iota^*\colon \mathcal{C}\to \mathcal{C}(Z)=\Mod_{A}(\mathcal{C})$ be the pullback functors along the open and closed localizations defined by $A$ respectively.  Let $B\in \ob{CAlg}(\mathcal{C})$ be a commutative algebra such that  $j^* B$ is descendable of index $\leq e$ and $\iota^* B$ is descendable of index $\leq d$. Then $B$ is descendable of index $\leq e+d$.  
\end{lemma}
\begin{proof}
Let $K=\ob{fib}(1\to B)$ and consider the fiber sequence 
\[
j_! j^* K \to K \to \iota_* \iota^* K. 
\]
By assumption and \cite[Definition 2.6.7]{mann2022p}, we know that 
\[
j^* K^{\otimes e}\to 1_{U}
\]
and
\[
\iota^*K^{\otimes d}\to 1_{Z}
\]
are zero. Consider the commutative diagram  for $f\in \N$
\[
\begin{tikzcd}
j_! j^* K^{\otimes f} \ar[r] \ar[d] & K^{\otimes f} \ar[r] \ar[d] & \iota_* \iota^* K^{\otimes f} \ar[d] \\ 
j_! 1_U  \ar[r] & 1 \ar[r] & \iota_* 1_{Z} 
\end{tikzcd}
\]
Taking $f=d$  the composite $K^{\otimes d}\to 1 \to \iota_* 1_Z$ is homotopic  to $K^{\otimes d}\to  \iota_*\iota^* K^{\otimes d}\to \iota_* 1_Z$ which is zero, and so $K^{\otimes d}\to 1$ factors through $K^{\otimes d}\to j_! 1_{U}$.
Tensoring with $K^{\otimes e}$ we get the composite map 
\[
K^{\otimes d+e}\to j_! K^{\otimes e}\to j_! 1_U \to 1
\]
which is zero, proving that $B$ is $(d+e)$-descendable as wanted. 
\end{proof}

\begin{remark}
  \label{sec:cohom-smooth-maps-descendable-for-complex-manifolds}
  Let $X$ be a complex manifold of dimension $d$. By analytic Riemann--Hilbert (\cite[Theorem II.3.1]{ScholzeRLL}) $X^\dR\cong X(\C)_{\Betti}$ (with the Betti stack taken over $\C_{\mathrm{gas}}$). Now, the Poincar\'e lemma provides an exact complex
  \[
    0\to \C\to \mathcal{O}_X\to \Omega^1_{X/\C}\to \ldots \to \Omega^d_{X/\C}\to 0
  \]
  of sheaves on $X(\C)$, i.e., of quasi-coherent sheaves on $X(\C)_\Betti$, which resolves the unit $\C$ by $\mathcal{O}_X$-modules. As $\mathcal{O}_X=g_\ast 1$ for $g\colon X\to X^\dR$, we can deduce that $g$ is prim and descendable of index $\leq d+1$.
    We note that the proof of \cref{sec:cohom-smooth-maps-descendability-for-smooth-maps} works in this way as well, and the corresponding object to $\mathcal{E}$ would incorporate the naive (or Hodge) filtration on the de Rham complex (seen as an object in $\ob{D}(X(\C)_{\Betti}\times_{\C} \mathbb{A}^{1,\an}_{\C}/\mathbb{G}_{m,\C}^{\mathrm{an}})$).

\end{remark}

\begin{lemma}\label{LemLimitsBerkovichSpaces}
  Let $(X_{n})_{n\in \N}$ be a projective diagram of derived Berkovich spaces such that, locally on a strict cover of $X_0$, all maps are affine.
  Let $X_{\infty}=\varprojlim_{n} X_n$ be the limit in the category of derived Berkovich spaces, then $X_{\infty}=\varprojlim_{n} X_n$ in the kernel category $\ob{K}_{\ob{D},\Q_p}$ of Gelfand stacks over $\Q_{p}$. 
\end{lemma}
Here, $K_{\ob{D},\Q_p}$ denotes the kernel category (\cite[Definition 4.1.3.(b)]{heyer20246functorformalismssmoothrepresentations}) associated with the $6$-functor formalism $\ob{D}(-)$ on Gelfand stacks.
\begin{proof}
All the transition maps $f_{m\to n}\colon X_{m}\to X_n$ are prim by construction, and  we have an identification $f_{m\to n,!}=f_{m\to n,*}$.  By \cite[Proposition 4.3.3]{heyer20246functorformalismssmoothrepresentations} and proper base change, it suffices to show the lemma locally in the analytic topology of $X_n$ (for a fixed $n$), and by taking a strict closed cover assume that $X_n$ and so all the $X_{m}$ for $m\geq n$ are affinoid. Set $X_n=\GSpec(A_n)$ with $n\in \N\cup\{\infty\}$ so that $A_{\infty}=\varinjlim_n A_n$. By \cite[Definition D.4.1]{heyer20246functorformalismssmoothrepresentations}, to show that $X_{\infty}=\varprojlim_n X_n$ in the kernel category, we have to show that for all Gelfand rings $B$ over $\Q_p$ with $Z=\GSpec(B)$ one has that
\[
\ob{Fun}_{\ob{K}_{\ob{D},\Q_p}}(Z, X_{\infty})\to \varprojlim_{n} \ob{Fun}_{\ob{K}_{\ob{D},\Q_p}}(Z, X_{n})
\]
is an equivalence. This amounts to proving that 
\[
\ob{D}(Z\times_{\GSpec(\Q_p)}  X_{\infty}) \to  \varprojlim_{n} \ob{D}_*(Z\times_{\GSpec(\Q_p)}  X_{n})
\]
where the transition maps are given by lower $*$-functors. This translates to the fact that the functor $A\mapsto \ob{D}^*(A)$ from $\Cat{AnRing}\to \Cat{Pr}^L$ commutes with colimits which is true (\cite[Proposition 4.1.14]{SolidNotes}). 
\end{proof}

\begin{proposition}\label{PropApproxFFStk}
  Let $X^{\diamond}_{\infty}=\varprojlim_{n} X_n^{\diamond}$ be a qfd arc-stack over $\Q_p$ obtained as the limit of  partially proper rigid spaces of dimension $\leq d$ for a fixed $d$ and affine transition maps. Then $X_{\infty}\to X_{\infty}^{\dR}$ is a descendable cover and   $X^{\dR}_{\infty}=\varprojlim_{n} X_n^{\dR}$ in the kernel category $\ob{K}_{\ob{D}, \GSpec(\Q_p)}$.
  \end{proposition}
\begin{proof}
  Let $X_{\infty}=\varprojlim_n X_n$ be the limit as Berkovich spaces, then $X_{\infty}\to X_{\infty}^{\dR}$ is prim, and we have to show that this map is descendable.
  Since each $X_n$ is a partially proper rigid space of dimension $\leq d$, we have that $f_n\colon X_n \to X_n^{\dR}$ is descendable of index  $\leq a:=\frac{(d+1)(d+2)}{2}$ thanks to \Cref{DescentRigidSpacedRStak}.
  Let $g_n\colon X_{\infty}^{\dR}\to X_n^{\dR}$, then  $g_n^{*} f_{n,*} 1 \in \ob{D}(X_{\infty}^{\dR})$ is descendable of index $\leq a$, and by taking colimits \cite{mann2022p} implies that $\varinjlim_{n} g_{n}^* f_{n,*}1$ is descendable of index $\leq 2a$.
  On the other hand, \cref{LemLimitsBerkovichSpaces} shows that $X_{\infty}=\varprojlim_n X_n$ in the kernel category over $\GSpec(\Q_p)$, hence, if  we denote $f_{\infty}\colon X_{\infty}\to X_{\infty}^{\dR}$, one deduces that $f_{\infty,* }1 = \varinjlim_{n} g_n^* f_{n,*} 1$, proving that $f_{\infty}$ is a descendable cover as wanted. 

Next, for the equivalence $X_{\infty}^{\dR}=\varprojlim_{n} X_n^{\dR}$ in the kernel category we apply \cref{TechnicalLemma6Functors}. Condition (a) of the lemma is automatic, condition (c) holds thanks to the map $X_{\infty}\to X_{\infty}^{\dR}$, finally, condition (b) follows from \cref{LemComputationdeRhamLimit} after passing to a strict closed cover of $X_0$.
\end{proof}

As a consequence, we deduce that perfect modules on de Rham stacks have a natural $t$-structure with heart given by vector bundles.

\begin{proposition}\label{PropPerfectBerkovichSpaces}
  Let $X$ be a qfd arc-stack over $\Q_p$, then the category $\ob{Perf}(X^{\dR})$ of perfect modules on the de Rham stack of $X$ has a $t$-structure with the following properties:
  \begin{enumerate}
  \item an object $P\in \ob{Perf}(X^{\dR})$ is connective if for any map $f\colon \GSpec(A)\to X^{\dR}$ from a qfd affinoid Gelfand stack, the pullback of $P$ along $f$ is connective,
  \item an object $P\in \ob{Perf}(X^\dR)$ is coconnective if its dual $P^\vee$ is connective, 
  \item for any morphism $Y\to X$ the functor $f^\ast\colon \ob{Perf}(X^{\dR})\to \ob{Perf}(Y^{\dR})$ is $t$-exact.
  \end{enumerate}
 Furthermore, the heart of  $\ob{Perf}(X^{\dR})$ is the category $\VB(X^{\dR})$ of vector bundles on $X^{\dR}$.
\end{proposition}
\begin{proof}
  By descent, it suffices to show the existence of the $t$-structure   in the case that $X=\Marc(A)$ for a qfd Gelfand $\Q_p$-algebra $A$ (as each morphism from $\GSpec(A)$ to a de Rham stack will factor over the canonical map $\GSpec(A)\to \Marc(A)^\dR$). Indeed, the definition of the connective and coconnective objects of (1) and (2) satisfy $!$-descent by \Cref{LemmaDescentPerfectComplexes}, and thus we only need to prove that they give rise to a $t$-structure in de Rham stacks. 
  
  Moreover, it suffices to show the claim $!$-locally on $A$, i.e., the existence of the $t$-structure as well as the exactness of pullback.
  In particular, we may assume that $X=\mathcal{M}_\arc(A)$ is pro-\'etale over an affine space.
  Indeed, there exists a quasi-pro-\'etale morphism $\Marc(A)\to \A^{d,\diamond}_{\Q_p}$ factoring over some disc $\mathbb{D}^{d,\leq r,\diamond}_{\Q_p}$, and after replacing $\GSpec(A)$ by the pullback of $\GSpec(A)\to \mathbb{D}^{d,\leq r}_{\Q_p}$ along a descendable morphism $Y\to \mathbb{D}^{d,\leq r}_{\Q_p}$ with $Y$ strictly totally disconnected (\cref{LemBasisTopologyGelfandRings}), we can find a presentation of $X=X_\infty$ as in \cref{PropApproxFFStk}  (namely, one uses that a quasi-pro-\'etale morphism over $Y$ is pro-finite \'etale).

  We first claim that if $X=\GSpec(A)$ for a qfd Gelfand ring $A$, then the unit of $X^\dR$ is a compact object in $\ob{D}(X^\dR)$ (this implies that each perfect module, aka dualizable object, on $X^\dR$ is compact).
 This is true for $X^\diamond$ replaced by a closed disc, and hence it suffices to show that for a countable inverse limit $Y_\infty=\varprojlim_n Y_n\to Y$ of qcqs quasi-pro-\'etale maps of qfd arc-stacks, the pushforward $f_\ast\colon Y_\infty^\dR\to Y^\dR$ preserves colimits.
 By \cref{sec:furth-results-analyt-1-surjections-on-de-rham-stack} we may assume that $Y$ is qfd and strictly totally disconnected.
 Then $Y_\infty\to Y$ is a countable inverse limit of finite \'etale maps (using that it is an inverse limit of qcqs maps), which then implies that for any qfd affinoid Gelfand ring $B$ with a map $\GSpec(B)\to Y^\dR$ the fiber product $Y_\infty^\dR\times_{Y^\dR} \GSpec(B)$ is represented by a countable pro-finite \'etale $B$-algebra $B_\infty$.
 In particular, the $\ast$-pushforward along $\GSpec(B_{\infty})\to \GSpec(B)$ preserves colimits.
 As it also commutes with base change, the claim follows.

 Assume now that $X=\mathcal{M}_\arc(A)$ is pro-\'etale over some fixed $\A^n_{\Q_p}$.
  We can then write $X^\dR$ as a filtered limit of $X_i^\dR$ with $X_i\to \GSpec(\Q_p)$ affinoid and Berkovich-\'etale over $\A^n_{\Q_p}$ (i.e., as compositions of rational localizations and finite \'etale maps).
  In particular, we note that $X_i$ and all terms in the \v{C}ech nerve for $X_i\to X_i^\dR$ are static and affinoid, and the same holds for the limit $\widetilde{X}$ of the $X_i$ in Gelfand stacks (note that $\widetilde{X}^\diamond=X$, and so $\widetilde{X}^\dR=X^\dR$).
  We claim that this implies that $\Perf(X^\dR)$ is the filtered colimit in $\Cat{Cat}_\infty$ along $\ast$-pullbacks of the categories $\Perf(X_i^\dR)$ (here $\Perf(-)\subseteq \ob{D}(-)$ denotes the subcategory of perfect complexes).
  Fully faithfulness of $\varinjlim_i \Perf(X^\dR_i)\to \Perf(X^\dR)$ follows from the compactness of perfect complex in $\ob{D}(-)$.
  For essential surjectivity, let $P\in \Perf(X^\dR)$, and assume that $P$ is of amplitude $[a,b]$.
  We have that $\Perf^{[a,b]}(X^\dR)$ is the limit over $\Delta$ of $\Perf^{[a,b]}({\widetilde{X}}_\bullet)$, where $\widetilde{X}_\bullet$ is the \v{C}ech complex for the covering $\widetilde{X}\to X^\dR$ (we note that $\widetilde{X}$ is $\dagger$-formally smooth).
  Because we bounded the amplitude and all terms in $\widetilde{X}_\bullet$ are static affinoid Gelfand rings, the resulting limit only depends on the limit over $\Delta_{\leq n}$ for some $n\geq 0$ (because the limit is over categories which are uniformly truncated).
  The same reasoning applies with $\widetilde{X}$ replaced by some $X_i$.
  Now filtered colimits in $\Cat{Cat}_\infty$ commute with finite limits, and over affinoid Gelfand stacks perfect complexes can be approximated: if $B=\varinjlim_iB_i$ is a filtered colimit of Gelfand rings $B_i$, then by \cite[Corollary A.5.9]{heyer20246functorformalismssmoothrepresentations} $\ob{D}(B)^\omega=\varinjlim_i D(B_i)^\omega$ in $\Cat{Cat}_\infty$, where $(-)^\omega$ denotes the compact objects (here we use that $\ob{D}(B_i)$ is compactly generated for each $i$).
  As perfect complexes are dualizable and compact, this implies $\Perf(B)\cong \varinjlim_i \Perf(B_i)$ in $\Cat{Cat}_\infty$.

  Altogether, we get that we can find $i$ and some perfect complex $P_i$ on $X_i^\dR$ of amplitude $[a,b]$ whose pullback to $X^\dR$ is $P$.

  Now for $X_i$ a partially proper smooth rigid space over $\Q_p$, it is classical that an algebraic $D$-module on $X_i$, which is coherent as an $\mathcal{O}_{X_i}$-module, is a vector bundle.\footnote{Here, we are implicitly using that, for $X$ a partially proper rigid space over $\Q_p$, there is a natural fully faithful embedding of $\ob{D}(X^\dR)$ into the category of quasi-coherent sheaves on the algebraic de Rham stack of $X$, cf. \cite[Proposition 5.2.11]{camargo2024analytic}, which preserves perfect modules.} Applying this to the top cohomology of a perfect module, and induct, shows that for any perfect module $P$ on $X_i^\dR$ its pullback $Q$ to $X_i$ has all cohomology objects given by vector bundles.
  This implies that the truncations of $Q$ descend to $X_i^\dR$, and hence define the desired $t$-structure with the desired properties, i.e., description of (co)connective objects (we note that the pullback of perfect modules whose cohomology objects are vector bundles preserves the truncations).
\end{proof}

The following proposition is useful to construct more $!$-able maps for the de Rham stack.

\begin{proposition}\label{CorDescendableMapsqfla}
Let $X$ be a qcqs arc-stack over $\Q_p$ with a quasi-pro-\'etale map $X\to \A^{d,\diamond}_{\Q_p}$.   The following hold:

\begin{enumerate}

\item There is a quasi-pro-\'etale cover $\Marc(A)\to X$ with $A$ a  strictly totally disconnected nilperfectoid ring such that:

\begin{itemize}

\item The morphism $\GSpec(A)\to \Marc(A)^{\dR}$ is prim and descendable.

\item  The composite map $\GSpec(A)\to \Marc(A)^{\dR}\to X^{\dR}$ is prim and descendable.

\end{itemize}

\item Let $f\colon Y\to X$ be a  $!$-able map of  qfd arc-stacks which is an arc-cover, then $f$ satisfies universal $\ob{D}^*$ and $\ob{D}^!$-descent, for the 6-functor formalism on de Rham stacks. If in addition $Y=\Marc(B)$ with $B$ a qfd separable Gelfand ring, then $f^{\dR}_* 1$ is descendable.

\end{enumerate}

\end{proposition}

\begin{proof}
Let us prove (1).
 By definition, there is a quasi-pro-\'etale map $X\to \A^{d,\diam}_{\Q_p}$.
 Since $X$ is qcqs,  after rescaling we can assume that it factors through the closed disc $\overline{\DD}^{d, \diam}_{\Q_p}$.
 Since $\overline{\DD}^{d, \diam}_{\Q_p}$ is finite dimensional,   using \Cref{xhs82j} we can find a Berkovich pro-\'etale descendable morphism
 $$
 \GSpec(B) \to \DD^{d, \leq 1}_{\Q_p}
 $$
with $\Marc(B)$  totally disconnected. Since this morphism is Berkovich pro-\'etale, it is $\dagger$-formally \'etale and thus, since the map $\DD^{d,\leq 1}_{\Q_p}\to \overline{\DD}^{d,\dR}_{\Q_p}$ is a descendable cover thanks to \Cref{thm:prim-descendable-berkovich}, we deduce that $\GSpec(B) \to \Marc(B)^\dR$ is also a descendable cover by base change. Replacing $B$ by a further ind-finite \'etale algebra, we may even assume that $B$ is strictly totally disconnected, see \cref{LemBasisTopologyGelfandRings}. Thus, since $\Marc(B)$ is strictly totally disconnected, the fiber product $X\times_{\overline{\DD}^{d, \diamond}} \Marc(B)$ is pro-\'etale over $\Marc(B)$, and so it can be written as a countable inverse limit of $\Marc(B_i)$, with $B \to B_i$ finite \'etale (using  \cite[Lemma 4.5]{scholze2024berkovichmotives} and the fact that finite \'etale $B^{u}$-algebras lift uniquely to $B$, by \Cref{x290sm}). Define 
 $$
 B' = \varinjlim_i B_i.
 $$
 Since $B'$ is ind-finite \'etale over $B$, the morphism $\GSpec(B') \to \GSpec(B')^\dR$ is again a descendable cover, proving the first item. Moreover, we claim that $\Marc(B)^\dR \to (\overline{\DD}^{d, \diam}_{\Q_p})^\dR$ is descendable. By pullback, this will imply that $\Marc(B')^\dR \to X^\dR$ is descendable, and so by composition the map $\GSpec (B') \to \Marc(B')^{\dR}\to X^{\dR}$ will also be descendable, proving the second item as well. To prove the descendability of $\Marc(B)^\dR \to (\overline{\DD}^{d, \diam}_{\Q_p})^\dR$, since we have a commutative diagram 
  \[
\begin{tikzcd}
\GSpec(B)  \ar[r] \ar[d] & \DD^{d,\leq 1}_{\Q_p} \ar[d] \\ 
\Marc(B)^{\dR} \ar[r] & \ (\overline{\DD}^{d, \diam}_{\Q_p})^\dR
 \end{tikzcd}
\]
 and since $\DD^{d,\leq 1}_{\Q_p}\to (\overline{\DD}^{d, \diam}_{\Q_p})^\dR$ is descendable, it is enough to see that $\GSpec(B) \to  \DD^{d,\leq 1}_{\Q_p}$ is descendable, but this follows from its construction.
 
 Finally, we apply \Cref{LemBasisTopologyGelfandRings} (1).  This produces a descendable cover $B' \to A$, which is an isomorphism on uniform completions and satisfies $(A)^{\dagger-\mathrm{red}}=A^{u}= B^{\prime,u}$.   Since $B' \to A$ is descendable and an isomorphism on uniform completions, we see that the morphisms $\GSpec(A) \to \Marc(A)^\dR=\Marc(B')^{\dR}$ and $\Marc(A)^\dR\to X^\dR$ are descendable.    Finally, \Cref{lemma:berkovich-spaceprofinite-residue-fields-closed-implies-perfectoid} implies that $A$ is nilperfectoid.

Now we prove (2).
  Since $f$ is an arc-cover, there is an arc-cover $Y^{\prime}\to X$ from a qfd strictly totally disconnected perfectoid space and a factorization $Y^{\prime}\to Y\to X$.
 Thus, we can assume without loss of generality that $X$ is itself qfd strictly totally disconnected.
 By \cref{LemmLqfd} $f^{\dR}$ is proper, so it suffices to show that $f^{\dR}_* 1$ is descendable in $X^{\dR}$. Using part (1), we can assume without loss of generality that $X=\Marc(A)$ is strictly totally disconnected nilperfectoid ring.

We are then reduced to proving the claim whenever $Y\to \Marc(A)$ is an arc-cover of qfd strictly totally disconnected perfectoid spaces, with $A$ a nilperfectoid ring with $\GSpec(A)\to \Marc(A)^{\dR}$ a descendable cover.
 By \cref{xsh29k},  we can assume that $Y=\Marc(B)$ with $A\to B$ a map of Gelfand rings that can be written as a sequential colimit of rational localizations of affine spaces over $A$. Since $Y \to \Marc(A)$ is qfd, we can uniformly fix the dimension of the affine spaces using a factorization $Y\to \A^{d,\diamond}_{\Marc(A)}\to \Marc(A)$, so that the map $\GSpec(B)\to \Marc(B)^{\dR/ \GSpec(A)}$ remains descendable by looking at the Hodge filtration of the de Rham complex, like in \Cref{sec:cohom-smooth-maps-descendability-for-smooth-maps}.
 As $\GSpec(A)\to \Marc(A)^{\dR}$ is descendable, the composite  $\GSpec(B)\to \Marc(B)^{\dR/ \GSpec(A)}\to \Marc(B)^{\dR}$ is also descendable.
 In total, we have a diagram 
\[
\begin{tikzcd}
\GSpec(B) \ar[r] \ar[d] & \GSpec(A) \ar[d] \\
\Marc(B)^{\dR} \ar[r] & \Marc(A)^{\dR}
 \end{tikzcd}
\]
where the vertical maps are descendable, $A$ is strictly totally disconnected nilperfectoid, and $A\to B$ can be written as in \cref{xsh29k}.
 Thanks to  \cref{xvsa9w} the map $A\to B$ is descendable, and then so is  $\Marc(B)^{\dR} \to \Marc(A)^{\dR}$ as wanted. 
\end{proof}

\begin{remark}
\label{rk:refined-not-std-but-berk-pro-et}
The proof of \Cref{CorDescendableMapsqfla} (1) shows the following: let  $X$ be a qcqs arc-stack over $\Q_p$ with quasi-pro-\'etale map $X\to \overline{\DD}^{d,\diamond}_{\Q_p}$. Then there is a  Berkovich pro-\'etale map $\GSpec(B) \to \DD^{d,\leq 1}_{\Q_p}$ with $B$ strictly totally disconnected (denoted $B'$ in the proof of  \cref{CorDescendableMapsqfla}),  and a factorization of arc-stacks 
\[
\begin{tikzcd}
\Marc(B) \ar[d] \ar[rd] & \\
X \ar[r]& \overline{\DD}^{d,\diamond}_{\Q_p}
\end{tikzcd}
\]
where the map $\Marc(B)\to X$ is pro-\'etale. In particular, $B$ is $\dagger$-formally smooth (cf. \cref{sec:appr-gelf-rings-examples-of-dagger-formally-smooth-maps}), and the map  $\GSpec(B)\to \Marc(B)^{\dR}$ is prim and descendable (by the proof of \cref{CorDescendableMapsqfla} (1)), and then so is the map $\GSpec(B)\to X^{\dR}$. If $X=\Marc(A)$  is affinoid,  writing $\Marc(B)=\varprojlim_i \Marc(A_i)$ as a   countable limit of \'etale maps of arc-stacks, with $A\to A_i$ its corresponding Berkovich \'etale map, we have a fiber product
\[
\begin{tikzcd}
\GSpec(\varinjlim_{i} A_i) \ar[r] \ar[d] & \GSpec(A) \ar[d] \\
\Marc(B)^{\dR} \ar[r] & \Marc(A)^{\dR}.
\end{tikzcd}
\]
Finally, for general $X$, thanks to \cref{PropQuotientdROverconvergentEquiv}, the \v{C}ech nerve of the map $\GSpec(B)\to X^{\dR}$ is given by the overconvergent neighbourhood of the closed immersion $\Marc(B)^{\bullet/X}\to  \Marc(B^{\bullet})$ in $\GSpec (B^{\bullet})$, where $\GSpec(B^{\bullet})$ is the \v{C}ech nerve of $\GSpec B$ over $\Q_p$, and $\Marc(B)^{\bullet/X}$ is the \v{C}ech nerve of $\Marc(B)\to X$.
Since $\Marc(B)\to X$ is pro-\'etale and $\Marc(B)$ is totally disconnected, the simplicial space $\Marc(B)^{\bullet/X}$ consists of strictly totally disconnected affinoid perfectoids with pro-\'etale transition maps.
 Hence, we  have that 
 \[
 \GSpec(B)^{\bullet/X^{\dR}}= \GSpec (B_{\bullet}^{\dagger/X})
 \]
 where $B_{\bullet}^{\dagger/X}$ is a cosimplicial ring of strictly totally disconnected Gelfand rings with Berkovich ind-\'etale transition maps  (even ind-finite \'etale), which are Berkovich pro-\'etale over the disc $\overline{\DD}^d_{\Q_p}$. We conclude that 
 \begin{equation}\label{eqdeRhamGeoRealizationBasicNucelar}
X^{\dR}=|\GSpec (B_{\bullet}^{\dagger/X})|. 
 \end{equation}
\end{remark}

\subsection{de Rham stacks and nuclearity}\label{subsection:DeRham-and-nuclearity}

We explain how to define a notion of (locally) nuclear quasi-coherent sheaves on de Rham stacks, and some of their stability properties.
Besides its independent interest, the results established here will be used in the proof of arc-hyperdescent for the analytic de Rham stack in the next subsection (\cref{Subsection:HyperdescentdR}).
But since this section is rather technical, we point out that little of it is used in the rest of the paper (only in the proofs of \Cref{TheoHyperdescentdR} and \Cref{sec:fully-faithf-psiast-1-case-for-f})
. We will make use of the basic results on nuclear modules recalled in \Cref{appendix:solid-functional-analysis}.

The following lemma gives some permanence properties of nuclear algebras in Gelfand rings. 

\begin{lemma}\label{LemmaNuclearGelfand}
Let $ \Cat{GelfRing}^{\ob{bnuc}}\subset \Cat{GelfRing}^{\ob{nuc}}\subset \Cat{GelfRing}$ be the full subcategory of basic nuclear and nuclear Gelfand $\Q_p$-algebras respectively. The following hold:

\begin{enumerate}

\item $\Cat{GelfRing}^{\ob{nuc}}$ is stable under colimits in $\Cat{GelfRing}$.  Similarly, $\Cat{GelfRing}^{\ob{bnuc}}$ is stable under countable colimits in $\Cat{GelfRing}$.

\item Let $r\in \R_{\ge 0}$.
  Then the algebra of overconvergent functions $\Q_p\langle  T \rangle_{\leq r}$ of radius $r$ lies in $\Cat{GelfRing}^{\ob{bnuc}}$.

\item $\Cat{GelfRing}^{\ob{nuc}}$ and $\Cat{GelfRing}^{\ob{bnuc}}$ are stable under rational localizations.

\end{enumerate}

\end{lemma}
\begin{proof}
  Part (1) follows from the stability of nuclear objects under colimits and under tensor products, \cref{basicnuclemma}.
  Part (2) follows from the observation that $\Q_p\langle T\rangle_{\leq r}$ is a light DNF space (see \cref{defDNF}, \cref{LemmaDNF}).
  Part (3) follows directly from the previous two items.
\end{proof}

Using  the permanence properties of (basic) nuclear Gelfand algebras of \cref{LemmaNuclearGelfand}, the descent of \cref{LemmaDescentNuclear} and the general definition \cref{xhs92mk}, we can define a suitable category of \textit{nuclear Gelfand stacks} together with suitable full subcategories of \textit{locally nuclear modules}.

\begin{definition}\label{DefinitionNuclearGelfandStks}
Let $\Cat{GelfRing}^{\ob{nuc}}_{\omega_1} \subset \Cat{GelfRing}_{\omega_1}$  be the full subcategory of separable nuclear  Gelfand rings. We define the category of \textit{nuclear Gelfand stacks} to be
$$\Cat{GelfStk}^{\ob{nuc}}:= \Cat{AnStk}((\Cat{GelfRing}^{\ob{nuc}}_{\omega_1})^{\op}).
$$
 Similarly, we let
 $$\Cat{GelfStk}^{\ob{qfd}, \ob{nuc}}=\Cat{AnStk}((\Cat{GelfRing}^{\ob{qfd}, \ob{nuc}})^{\op})
 $$ denote the category of \textit{qfd nuclear Gelfand stacks} (here $\Cat{GelfRing}^{\ob{qfd},\ob{nuc}}$ denotes the category of nuclear Gelfand rings, which are additionally qfd).
 Given $X\in \Cat{GelfStk}^{\ob{nuc}}$, we let $\Cat{Nuc}^{\ob{loc}}(X)\subset \ob{D}(X)$ be the full subcategory of modules  whose pullback along all maps $\GSpec(A)\to X$ with $A\in \Cat{GelfRing}^{\ob{nuc}}_{\omega_1}$ is a nuclear $A$-module (resp.\ for $X$ a qfd nuclear Gelfand stack).
\end{definition}

\begin{remark}\label{Permanence of nuclearity}
Let  $X\in \Cat{GelfStk}^{\ob{nuc}}$. Thanks to the descent results of \cref{LemmaDescentNuclear}, the category $\Cat{Nuc}^{\ob{loc}}(X)$ is a sheaf on nuclear Gelfand stacks, and if $X=\GSpec(A)$ is affinoid nuclear then $\Cat{Nuc}(X)=\Cat{Nuc}(A)$.  On the other hand, thanks to the stability of nuclear modules along pullback maps (by definition) the full subcategories $\Cat{Nuc}^{\ob{loc}}(X)\subset \ob{D}(X)$  are stable under pullbacks. 
\end{remark}

Thanks to \cref{Permanence of nuclearity} we know that $*$-pullback maps along nuclear Gelfand stacks preserve locally nuclear objects.
It is natural to ask whether $!$-pullbacks or $*$/$!$-pushforwards preserve locally nuclear objects.\footnote{We note that, similarly to \cref{sec:totally-disc-stacks-qcoh-on-td-stacks}, we can talk about $!$-ability etc.\ of nuclear Gelfand stacks.}
We can prove some stability results in that direction, allowing one to change the source and the target. 

\begin{lemma}\label{LemmaPreservationLocallyNuclear}
Let $f\colon Y\to X$ be a morphism of nuclear Gelfand stacks. 

\begin{enumerate}

\item   Let $g\colon Y'\to Y$ be an epimorphism with \v{C}ech nerve $Y^{\prime,\bullet}$.
  Suppose that $X$ is corepresented by a nuclear Gelfand ring and  that $*$-pushforward along $Y^{\prime,n}\to X$ for each $[n]\in \Delta$ preserves locally nuclear objects.
  Then  $f_*$ preserves locally nuclear objects.

\item Suppose that $f$ is $!$-able and that there is an epimorphism $X^{\prime}\to X$ of nuclear Gelfand stacks with pullback $f^{\prime}\colon Y^{\prime}=Y\times_{X} X^{\prime}\to  X^{\prime}$ such that $f^{\prime}_!$ preserves locally nuclear objects.
  Then $f_!$ preserves locally nuclear objects. 

\item Let $g\colon Y'\to Y$ be a $\ob{D}^!$-cover with \v{C}ech nerve $Y^{\prime,\bullet}$.
  Suppose that locally nuclear sheaves satisfy $\ob{D}^!$-descent along $g$ (in particular, they are stable under upper $!$ and lower $!$-functors for the maps $Y^{\prime,\bullet}\to Y$ and $Y^{\prime,n}\to Y^{\prime, m}$ for $\alpha\colon [m]\to [n]$ in $\Delta$).
  Then the following hold:

\begin{itemize}
\item[(a)] The functor $f^!\colon \ob{D}(X)\to \ob{D}(Y)$ preserves locally nuclear objects if and only if the functor $(f\circ g)^!\colon \ob{D}(X)\to \ob{D}(Y^{\prime})$ preserves locally nuclear objects. 

\item[(b)] The functor $f_!\colon \ob{D}(Y)\to \ob{D}(X)$ preserves locally nuclear objects if and only if $(f\circ g)_!\colon \ob{D}(Y^{\prime})\to \ob{D}(X)$ preserves locally nuclear objects. 

\end{itemize} 

\item Suppose that $X=\GSpec(A)$ is nuclear affinoid.
  Suppose we are given a morphism $g\colon Y^{\prime}\to Y$  satisfying universal $\ob{D}^*$- and  $\ob{D}^!$-descent, with \v{C}ech nerve $Y^{\prime,\bullet}$, and such that  $Y^{\prime,\bullet}=\GSpec (B_{\bullet})$ is nuclear affinoid.
  Then the following hold:

\begin{enumerate}

\item[(a)]  $f_*$ preserves locally nuclear objects.

\item[(b)] If the rings $B_{\bullet}$ are basic nuclear $\Q_p$-algebras, then locally nuclear objects satisfy $!$-descent along $g$, and $g_*$ preserves locally nuclear objects. If in addition $A$ is basic nuclear,   both $f_!$ and $f^!$ preserve locally nuclear objects.  

\end{enumerate}

\end{enumerate}
\end{lemma}
\begin{proof}
(1) By construction, locally nuclear sheaves are stable under $*$-pullbacks of nuclear Gelfand stacks. Since $f$ is an epimorphism, we have that
\[
\ob{D}(Y)=\ob{Tot}(\ob{D}^{*}(Y^{\bullet})).
\]
Let $g^{\bullet}\colon Y^{\prime, \bullet}\to Y$. Then, we have an equivalence of functors $\ob{D}(Y)\to \ob{D}(X)$
\[
f_*(-)\xrightarrow{\sim} \ob{Tot}( (f\circ g^{\bullet})_*  (g^{\bullet})^*(-) ). 
\]
The preservation of locally nuclear objects by $f_*$  then follows from the preservation of locally nuclear objects of $(f\circ g^{\bullet})_* $ and \cref{LemmaCOuntableLimitNuclear}.

(2)   By construction,  locally nuclear sheaves  satisfy descent along epimorphisms of nuclear Gelfand stacks. Hence, an object $M\in\ob{D}(X)$ is locally nuclear if and only if its pullback to $\ob{D}(X')$ is so.
  The statement follows from proper base change.

(3)  By assumption, the  equivalence of categories
\[
\ob{D}(Y)=\ob{Tot}(\ob{D}^!(Y^{\prime,\bullet}))
\]
restricts to an equivalence on locally nuclear objects. Moreover, an object $M\in \ob{D}(Y)$ is locally nuclear if and only if its upper $!$-pullback in $\ob{D}(Y')$ is locally nuclear. This proves part (3.a) of the statement. For part (3.b),  let $g^{\bullet}\colon Y^{\prime,\bullet}\to Y$,  it suffices to note that there is a natural equivalence of functors $\ob{D}(Y)\to \ob{D}(X)$
\[
\varinjlim_{[n]\in \Delta^{\op}} (f\circ g^{\bullet})_! g^{\bullet, !} \xrightarrow{\sim } f_!,
\]
and that locally nuclear objects are stable under colimits. 

(4) Part (4.a) follows from Part (1) and \cref{LemmaPermanence}. For part (4.b), we want to see that  the natural equivalence
\begin{equation}\label{eqo02o3mqeldqw}
\ob{D}(Y)\xrightarrow{\sim} \ob{Tot}(\ob{D}^!(Y^{\prime,\bullet}))
\end{equation}
along upper $!$-morphisms preserve locally nuclear modules. Once we have proven this, (4.b) will follow from (3), \cref{LemmaPermanence} for lower-$!$ and \cref{LemmaNuclearUppershierk} for upper-$!$.  By \cref{LemmaNuclearUppershierk} and since the terms $Y^{\prime,\bullet}$ are affinoids represented by basic nuclear algebras, the right hand side term of \cref{eqo02o3mqeldqw} leaves stable the categories of nuclear modules. Hence, it suffices to show that an object  $M\in \ob{D}(Y)$ is locally nuclear if and only if $g^!M\in \ob{D}(Y^{\prime})$  is  nuclear. Since $g$ satisfies universal $*$-descent, we know that $M$ is nuclear if and only if $g^*M$ is nuclear, and by applying (2) and (1), we deduce that $g_*$ preserves locally nuclear objects (note that $g_!=g_*$ because $g$ is represented in affinoid Gelfand rings).
This implies by $!$-descent that $M$ is locally nuclear if $g^!M$ is  nuclear. 

Now, consider the pullback diagram 
\[
\begin{tikzcd}
Y^{\prime,\bullet+1} \ar[d,"g^{\prime,\bullet}"] \ar[r,"\alpha_0^{\bullet}"] & Y^{\prime,\bullet} \ar[d, "g^{\bullet}"]  \\ 
Y^{\prime} \ar[r,"g"] & Y
\end{tikzcd}
\]
where $\alpha_0^{n}\colon [n]\to [n+1]$ is the inclusion in the last $n$ terms.   Let $M$ be a locally nuclear module over $Y$ with associated  cocartesian section $(M_n)_{n}\in \ob{D}(Y^{\prime,\bullet})$ (and upper $\ast$-functors).
Then by proper base change (for upper $!$ and lower $*$ functors) we have that
\[
g^!M= g^!( \ob{Tot}( g^{\bullet}_{\ast}M_{\bullet} )) =  \ob{Tot}(g^{\prime,\bullet}_* \alpha^{\bullet,!}_{0} M_{\bullet}). 
\]
If $M$ is locally nuclear, we deduce that $g^!M$ is nuclear by a combination of  \cref{LemmaNuclearUppershierk} for $\alpha_0^{\bullet,!}$, \cref{LemmaPermanence} for $g^{\prime,\bullet}_*$ and \cref{LemmaCOuntableLimitNuclear} for the totalization. 
\end{proof}

For future reference we need an analogue of \cref{LemmaPreservationLocallyNuclear} for $\omega_1$-compact objects.

\begin{lemma}\label{LemmaStabilityOmega1Compact} 
Let $f\colon Y\to X$ be a  morphism of Gelfand stacks.  The following hold: 

\begin{enumerate}

\item Suppose that $f$ is an epimorphism and that its \v{C}ech nerve $Y^{\bullet}\to X$ is corepresented by basic nuclear Gelfand $\Q_p$-algebras. Then $\omega_1$-compact objects are stable under $f^*$, $f_*$ and satisfy $\ob{D}^*$-descent along $f$.

\item Suppose that there is an epimorphism $g\colon X'\to X$ as in (1). Let $f'\colon Y'\to X'$ be the base change of $f$ along $g$. Suppose that $f$ is $!$-able,  then $f_{!}$ preserves $\omega_1$-compact objects if and only of $f'_{!}$ does. 

\item Suppose that there is an epimorphism $Y'\to Y$ as in (1) and let $f'\colon Y'\to X$. Then $f^*$ preserves $\omega_1$-compact objects if and only if $f^{\prime,*}$ does so.

\item Suppose that $X$ is the spectrum of a  basic nuclear Gelfand $\Q_p$-algebra and that there is an epimorphism $g\colon Y'\to Y$ as in (1). Then $f^*$ and $f_*$ preserve $\omega_1$-compact objects.

\end{enumerate}

\end{lemma} 
\begin{proof}

(1) Since $f$ is an epimorphism, we have that
\[
\ob{D}(X)=\ob{Tot}(\ob{D}^*(Y^{\bullet}))
\]along $*$-pullback maps. In particular, we have for $M,N\in \ob{D}(X)$ that 
\[
\Hom_{X}(N,M)= \ob{Tot} (\Hom_{Y^{\bullet}}( N_{\bullet}, M_{\bullet}))
\]
where $N_{\bullet},M_{\bullet}\in \ob{D}(Y^{\bullet})$ are the corresponding pullbacks of $N$ and $M$ respectively.
This implies that if $N\in \ob{D}(X)$ is such that $f^*N$ is $\omega_1$-compact, then so is $N$.
Conversely, since $f$ is prim (it is represented in affinoid stacks),   $f_*$ preserves colimits, and if $M\in \ob{D}(X)$ is $\omega_1$-compact then so is $f^*M\in \ob{D}(Y)$. From this one  deduces $*$-descent for $\omega_1$-compact objects along $Y\to X$.
It is left to show that $f_*$ preserves $\omega_1$-compact objects.
By the previous discussion, it suffices to show that $f^*f_*$ preserves $\omega_1$-compact objects, by proper base change along the pullback $\ob{pr}_i\colon Y\times_XY\to Y$ the functor $f^*f_*$ is naturally equivalent to $\ob{pr}_{2,*}\ob{pr}_1^*$, is is a composite of a   pullback  an forgetful functor for basic nuclear rings.
The preservation of $\omega_1$-compact modules follows from \cref{LemmaPermanenceOmega1Compact}.

(2) This follows from the $\ob{D}^*$-descent of $\omega_1$-compact modules along $g$ of part (1), and by proper base change for $f_!$.

(3) This follows from  the $\ob{D}^*$-descent of $\omega_1$-compact objects along $Y'\to Y$ of part (1).

(4)  By (1), the functor $g^*$ reflects $\omega_1$-compact objects.
 Hence, $f^*$ preserves $\omega_1$-compact objects if and only if $(f\circ g)^*$ does so, but this is a base change along basic nuclear algebras and the claim follows from \cref{LemmaPermanenceOmega1Compact}.
  On the other hand, for proving that $f_*$ preserves $\omega_1$-compact objects,  let $g^{\bullet}\colon Y^{\prime, \bullet}\to Y$ be the \v{C}ech nerve of $g$. By assumption the terms $Y^{\prime,\bullet}$ are corepresented by basic nuclear Gelfand $\Q_p$-algebras.
   On the other hand, we have an equivalence of functors $\ob{D}(Y)\to \ob{D}(X)$ given by 
\[
f_*=\ob{Tot}( (f\circ g^{\bullet})_* g^{\bullet,*}(-)).
\]
By (1) we know that $g^{\bullet,*}$ preserve $\omega_1$-compact objects, and by \cref{LemmaPermanenceOmega1Compact} so does the forgetful functor $(f\circ g^{\bullet})_*$.
 Since $\omega_1$-compact objects are stable under countable limits, we deduce that $f_*$ preserves $\omega_1$-compact objects as wanted. 
\end{proof}

We apply these general considerations to analytic de Rham stacks. There are evident pullback morphisms of $\infty$-topoi
\begin{equation}\label{eq01klmqaswdq}
\begin{tikzcd}
\Cat{GelfStk}^{\ob{qfd},\ob{nuc}}  \ar[r, "f^*"] \ar[d] & \Cat{GelfStk}^{\ob{nuc}} \ar[d] \\ 
\Cat{GelfStk}^{\ob{qfd}}  \ar[r] \ar[d, "(-)^{\diamond}"] & \Cat{GelfStk} \ar[d,"(-)^{\diamond}"] \\
\Cat{ArcStk}^{\ob{qfd}} \ar[r] & \Cat{ArcStk}.
\end{tikzcd}
\end{equation}
Thanks to the proof of \cref{LemBasisTopologyGelfandRings}, and the fact that Banach $\Q_p$-modules are nuclear (\cref{ExampleBanach}), the category of nuclear Gelfand rings admits a basis by nuclear nilperfectoid rings, and the category of qfd nuclear Gelfand rings admit a basis by nuclear nilperfectoid  strictly totally disconnected rings.
Hence, we can define a de Rham stack $(-)^{\dR,\ob{nuc}}\colon \Cat{ArcStk}^{\qfd}\to \Cat{GelfStk}^{\ob{qfd}, \ob{nuc}}$ as the right adjoint of the perfectoidization $(-)^{\diamond, \ob{nuc}}\colon \Cat{GelfStk}^{\ob{qfd},\ob{nuc}}\to \Cat{ArcStk}^{\ob{qfd}}$.
Concretely, the functor $(-)^{\dR,\ob{nuc}}$ sends an arc-stack $X$ to the functor sending $A$ a separable qfd nilperfectoid strictly totally disconnected nuclear $\Q_p$-algebra to the anima $X(A^{\dagger-\ob{red}})=X(A^u)$. We have proven the following:
\begin{proposition}\label{PropDeRhamNuclear}
Let $f_*\colon \Cat{GelfStk}^{\ob{qfd}}\to \Cat{GelfStk}^{\ob{qfd},\ob{nuc}}$ be the geometric morphism of $\infty$-topoi arising from the inclusion of qfd nuclear Gelfand rings in qfd Gelfand rings. We have a natural commutative diagram
\[
\begin{tikzcd}
\Cat{ArcStk}^{\ob{qfd}} \ar[r, "(-)^{\dR}"] \ar[rd, "(-)^{\ob{dR}, \ob{nuc}}"'] & \Cat{GelfStk}^{\ob{qfd}} \ar[d,"f_*"] \\
 &   \Cat{GelfStk}^{\ob{qfd}, \ob{nuc}}
\end{tikzcd}
\]
Moreover, the functor $(-)^{\ob{dR},\ob{nuc}}$  sends epimorphisms of qfd arc-stacks to epimorphisms of qfd nuclear Gelfand stacks.  
\end{proposition}
\begin{proof}
The first statement follows by passing to right adjoints in the left column of \cref{eq01klmqaswdq}. The fact that $(-)^{\ob{dR},\ob{nuc}}$ sends epimorphisms to epimorphisms follows from the fact that $X^{\dR, \rm nuc}$ sends $A$ a separable qfd nilperfectoid strictly totally disconnected nuclear $\Q_p$-algebra to the anima $X(A^{\dagger-\ob{red}})=X(A^u)$, and the same proof of arc-descent of \cref{sec:furth-results-analyt-1-surjections-on-de-rham-stack} after noticing that all the constructions there leave stable the category of nuclear Gelfand rings. 
\end{proof}

Later in \cref{CorodeRhamNuclearSame} we will prove that the de Rham and nuclear de Rham stacks are essentially the same, in particular that they have the same category of quasi-coherent sheaves.
In preparation to that statement, let us prove this property for affinoid rings. 

\begin{lemma}\label{Lemqqj3qw}
Keep the notation of \cref{PropDeRhamNuclear}. Let $A$ be a qfd Gelfand ring and let $X=\Marc(A)$ be its associated arc-stack. Then the natural map of qfd Gelfand stacks 
\[
f^*X^{\dR,\ob{nuc}}\to f^*f_* X^{\dR} \to X^{\dR}
\] 
is an equivalence of Gelfand stacks.
\end{lemma}
\begin{proof}
  Thanks to \cref{xsh29k} and the independence of de Rham stacks under uniform completions, we can assume without loss of generality that $A=\varinjlim_n A_n$ is a filtered colimit of rational localizations of affine spaces. In that case $A$ is $\dagger$-formally smooth as in \cref{sec:appr-gelf-rings-dagger-formally-smooth}, and by \cref{CorollaryDaggerSmoothComparison} and \cref{PropdeRhamBerkovich} (2) we can write $X^{\dR}$  (resp.\ $X^{\dR,\ob{nuc}}$) as the geometric realizations of the overconvergent diagonals of the \v{C}ech nerve of $\GSpec(A)\to \GSpec(\Q_p)$ in qfd (resp.\ qfd nuclear) Gelfand stacks.
  The lemma follows since $f^*$ commutes with colimits, in particular with geometric realizations, and since it is the identity in nuclear affinoid Gelfand stacks.
\end{proof}

Finally, we prove stability properties of locally nuclear and $\omega_1$-compact objects on de Rham stacks under the six functors, \Cref{PropDescentNuclearomega1CompactdR}. First, we establish the following useful criterion.  

\begin{lemma}\label{LemmaDescentArcProet}
Let $X$ be a qcqs arc-stack over $\Q_p$ with a quasi-pro-\'etale map $X\to \overline{\mathbb{D}}^{d,\diamond}_{\Q_p}$. Let $B$ be a strictly totally disconnected qfd basic nuclear $\Q_p$-algebra endowed with a $!$-cover $h\colon \GSpec B\to X^{\dR}$ (cf. \cref{rk:refined-not-std-but-berk-pro-et}). The following hold: 

\begin{enumerate}

\item The \v{C}ech nerve $Y^{\bullet}$ of $h\colon Y=\GSpec B\to X^{\dR}$ consists of affinoid Gelfand stacks corepresented by basic nuclear $\Q_p$-algebras. 

\item $\omega_1$-compact objects satisfy $\ob{D}^*$-descent along $h$. Moreover, $h_*$ preserves $\omega_1$-compact objects, equivalently, $h^!$ commutes with $\omega_1$-filtered colimits. 

\item Locally nuclear objects satisfy $\ob{D}^!$-descent along $h$, and $h_*$ preserves locally nuclear objects. 

\item Locally nuclear  objects on $X^{\dR}$ are stable under countable limits. 

\end{enumerate}
\end{lemma}
\begin{proof}
The existence of the ring $B$ is discussed in \cref{rk:refined-not-std-but-berk-pro-et}.  We proceed to prove the statements (1) to (4).

(1) By \cref{PropQuotientdROverconvergentEquiv} the \v{C}ech nerve $Y^{\bullet}$ of $h$ is given by the overconvergent neighbourhoods of $\Marc(B)^{\bullet/X^{\diamond}} \subset \Marc(B^{\bullet})$ in $\GSpec(B^{\bullet})$, where $\Marc(B)^{\bullet/X^{\diamond}}$ is the \v{C}ech nerve of $\Marc(B)\to X^{\diamond}$, and $\GSpec(B^{\bullet})$ is the \v{C}ech nerve of $\GSpec B\to \GSpec(\Q_p)$.
  In particular, for all morphisms $\alpha\colon [n]\to [m]$ in $\Delta$, the associated map $\alpha\colon Y^{m}\to Y^{n}$ is quasi-pro-\'etale at the level of arc-stacks, and since $Y^{\diamond}$ is strictly totally disconnected, it is actually pro-\'etale.
  In particular,  for any map $\alpha\colon [0]\to [n]$, the map $\alpha\colon Y^{n,\diamond}\to Y^{\diamond}$ can be written as a countable limit of finite \'etale maps, and the pullback $Y^{n}=Y\times_{Y^{\dR}} Y^{n,\dR}$ is a countable limit of finite \'etale maps over $Y$, proving the claim.  

(2) This follows by \cref{LemmaStabilityOmega1Compact} (1) and part (1) above.

(3) This follows from \cref{LemmaPreservationLocallyNuclear} (4.b) and part (1) above.

(4) The stability of locally nuclear objects under countable limits follows from part (3), the commutation of $h^!$ with limits, and the stability under countable limits of nuclear modules of \cref{LemmaCOuntableLimitNuclear}.  
\end{proof}

Before addressing the general case of \Cref{PropDescentNuclearomega1CompactdR}, we prove the qcqs case.

\begin{lemma}\label{LemmaBasicNuclear}
Let $f\colon Y\to X$ be a morphism of qcqs arc-stacks over $\Q_p$ admitting quasi-pro-\'etale maps to finite dimensional affine spaces over $\Q_p$.  Let $f^{\dR}\colon Y^\dR\to X^{\dR}$ be its associated map of de Rham stacks. 
\begin{enumerate}

\item The functor $f^{\dR, !}$ preserves locally nuclear objects and commutes with $\omega_1$-filtered colimits.  

\item  The functor $f^{\dR}_!=f^{\dR}_*$ preserves locally nuclear objects and $\omega_1$-compact objects. 

\item The functor $f^{\dR,*}$ preserves $\omega_1$-compact objects. 

\end{enumerate}  
  
  In particular, if $f$ is an arc-cover, $f^{\dR}$ satisfies $\ob{D}^!$-descent for locally nuclear objects, and $\ob{D}^*$-descent for $\omega_1$-compact objects. 
 
\end{lemma}
\begin{proof}
By assumption on $Y$ and $X$,  we can find a commutative diagram  of arc-stacks
\[
\begin{tikzcd}
Y^{\diamond} \ar[r]\ar[d] & \overline{\mathbb{D}}^{d+e,\diamond}_{\Q_p} \ar[d]\\ 
X^{\diamond} \ar[r] & \overline{\mathbb{D}}^{d,\diamond}_{\Q_p}
\end{tikzcd}
\]
where the horizontal maps are quasi-pro-\'etale, and the right vertical map is the projection map onto the first $d$-components. 
 By \cref{LemmaDescentArcProet} there is a strictly totally disconnected algebra $A$, Berkovich pro-\'etale over $\mathbb{D}^{d,\leq 1}_{\Q_p}$, and a pro-\'etale cover of arc-stacks $\Marc(A)\to X^{\diamond}$ over $\overline{\mathbb{D}}^{d,\diamond}_{\Q_p}$, such that $\GSpec(A)\to X^\dR$ is a $!$-cover. Taking the base change $Y^{\prime,\diamond}=Y^{\diamond}\times_{X^{\diamond}} \Marc(A)$, we have a quasi-pro-\'etale map 
 \[
 Y^{\prime,\diamond}\to \overline{\mathbb{D}}^{e,\diamond}_{A}. 
 \]
 Following the same construction of \cref{rk:refined-not-std-but-berk-pro-et} relative to $\overline{\mathbb{D}}^{e,\diamond}_{A}$,   we can find a strictly totally disconnected Berkovich pro-\'etale algebra $B$ over $\DD^{d+e,\leq 1}_{A}$, and a pro-\'etale cover $\Marc(B)\to Y^{\prime, \diamond}$ of arc-stacks over $\overline{\DD}^{e,\diamond}_{A}$ such that $\GSpec(B)\to Y^{\prime, \dR}$ is a $!$-cover.
   In total, we get the diagram 
\[
\begin{tikzcd}
\GSpec(B) \ar[r,"g"] & \GSpec(A)\times_{X^{\dR}}  Y^{\dR} \ar[r,"\pr_2"] \ar[d,"\pr_1"]  &   Y^{\dR}  \ar[d,"f^{\dR}"] \\ 
 &  \GSpec(A) \ar[r,"h"] & X^{\dR}  
\end{tikzcd}
\]
where $g$ and $h$ are prim and descendable, and such that $B$ and $A$ are basic nuclear $\Q_p$-algebras.  We now prove the statements of the lemma.

(1) By \cref{LemmaPreservationLocallyNuclear} (3.a) and \cref{LemmaDescentArcProet} (1), it suffices to show that basic nuclear objects are preserved under the functor
\[
g^! \circ (\pr_2)^! \circ f^!\colon \ob{D}(X^{\dR})\to \ob{D}(B).
\]
This functor is is naturally equivalent to $\Hom_A(B, h^!(-))$. Let $M\in \ob{D}(X^{\dR})$ be locally nuclear, by  \cref{LemmaDescentArcProet} we know that $h^!M$ is a nuclear $A$-module. Then, as $B$ and $A$ are basic nuclear, we see that  $\Hom_A(B,h^!M)$  is a nuclear $B$-module thanks to \cref{LemmaNuclearUppershierk}. The commutation with $\omega_1$-filtered colimits will follow from the preservation of $\omega_1$-compact objects in (2).

(2) To prove that $f^{\dR}_!=f^{\dR}_*$ preserves locally nuclear and $\omega_1$-compact objects, by \cref{LemmaStabilityOmega1Compact} (2) (for $\omega_1$-compact) and \cref{LemmaPreservationLocallyNuclear} (2) (for locally nuclear), it suffices that locally nuclear objects are preserved under the map
\[
\pr_{1,*}\colon \GSpec(A)\times_{X^{\dR}}  Y^{\dR}\to \GSpec(A). 
\]
Then, using \cref{LemmaStabilityOmega1Compact} (1) (for $\omega_1$-compact) and \cref{LemmaPreservationLocallyNuclear} (1) (for locally nuclear), it suffices to see that the forgetful functor $\ob{D}(B)\to \ob{D}(A)$ preserves $\omega_1$-compact and locally nuclear objects. This follows from \cref{LemmaPermanenceOmega1Compact} and \cref{LemmaPermanence} respectively as both $A$ and $B$ are basic nuclear.

(3) By \cref{LemmaDescentArcProet} (2) it suffices to show that the pullback along $\GSpec(B)\to Y^{\dR}\to X^{\dR}$ preserves $\omega_1$-compact objects. But this also factors as $\GSpec(B)\to \GSpec(A)\to X^{\dR}$ and by the same lemma the pullback along $\GSpec(A)\to X^{\dR}$ preserves $\omega_1$-compact objects. Thus, it suffices to see that the base change along $A\to B$ preserves $\omega_1$-compact modules, but this is clear since both $A$ and $B$ are basic nuclear $\Q_p$-algebras.
\end{proof}

\begin{remark}
Note that in general $f_!^{\dR/X}$ does not preserve compact objects without further assumptions on $f$ (indeed, combined with the preservation in nuclearity of \Cref{LemmaBasicNuclear}(2), that would imply preservation of dualizable objects, which is not necessarily true). 
\end{remark}

In the following we say that a (derived) Berkovich space $X$ is \textit{lqfd} if the structure morphism $X\to \GSpec(\Q_p)$ is lqfd in the sense of \cref{DefLQFDim}.

\begin{lemma}\label{LemmaNuclearCompactLocalTopology}
  Let $X$ be a lqfd Berkovich space over $\Q_p$ such that $|X|$ admits a countable basis of open subsets.
  The following hold:
\begin{enumerate}
\item Let $\iota\colon Z\to X$ be a closed immersion with $Z$  and $\iota^{\dR}\colon Z^{\dR}\to X^{\dR}$ the morphism at the level of de Rham stacks. 

\begin{enumerate}

\item  $\iota^{\dR,*}$ preserves locally nuclear objects and $\omega_1$-compact objects. 

\item $\iota_*^{\dR}$ preserves locally nuclear objects and $\omega_1$-compact objects.

\item $\iota^{\dR,!}$ preserves locally nuclear objects and commutes with $\omega_1$-filtered colimits.

\end{enumerate}

\item Let $j\colon U\to X$ be an open immersion and $j^{\dR}\colon U^{\dR}\to X^{\dR}$ be the associated morphism of de Rham stacks. 

\begin{enumerate}

\item   $j^{\dR,*}$ preserves locally nuclear  objects and $\omega_1$-compact objects. 

\item  $j_*^{\dR}$ preserves locally nuclear objects and commutes with $\omega_1$-filtered colimits.

\item  $j_{!}^{\dR}$ reserves locally nuclear  objects and $\omega_1$-compact objects. 

\end{enumerate}

\item Write  $X=\bigcup_n Z_n$ as  a countable union of closed subspaces with inclusions $\iota_n\colon Z_n\to X_n$. Let $M\in \ob{D}(X^{\dR})$.

\begin{enumerate}

\item   $M$ is $\omega_1$-compact if and only if $\iota_{n,*}^\dR\iota^{\dR,*}_nM$ is $\omega_1$-compact for all $n\in \N$. 

\item $M$ is locally nuclear if and only if $\iota_{n,*}^{\dR}\iota^{\dR,!}_n M$ is locally nuclear for all $n\in \N$, if and only if $\iota_*^{\dR}\iota^{\dR,*}M$ is locally nuclear for all $n\in \N$.

\end{enumerate}

\item Write  $X=\bigcup_n U_n$ as  a countable union of open subspaces with inclusions $j_n\colon U_n\to X_n$. Let $M\in \ob{D}(X^{\dR})$.

\begin{enumerate}

\item   $M$ is $\omega_1$-compact if and only if $j_{n,!}^{\dR}j^{\dR,*}_nM$ is $\omega_1$-compact for all $n\in \N$. 

\item $M$ is locally nuclear if and only if $j_{n,!}^{\dR}j^{\dR,*}_n M$ is locally nuclear for all $n$, if and only if $j_{n,*}^{\dR}j_n^{\dR,*}M$ is locally nuclear for all $n\in \N$.

\end{enumerate}

\end{enumerate}
\end{lemma}
\begin{proof}
(1) We prove the items in order.
\begin{enumerate}
 \item[(a)] It is clear that $\iota^{*,\dR}$ preserves locally nuclear objects. The fact that it preserves $\omega_1$-compact objects follows from the fact that its right adjoint $\iota_{*}^{\dR}$ commutes with colimits.

 \item[(b)] Consider a strict cover $X=\bigcup_{i\in I} C_i$ by closed affinoid subspaces of $X$ with inclusions $f_i\colon C_i\to X$, and $I$ a countable set. To see that  $\iota_{*}^{\dR}$ preserves nuclear objects,  it suffices to see that $f_i^{\dR,*} \iota_{*}^{\dR} $ preserves nuclear objects for all $i\in I$. By proper base change this reduces to the case when $X$ is itself affinoid and $Z\subset X$ is a closed subspace, and then follows from \cref{LemmaBasicNuclear} (2).  (Note that this part does not use the countability assumption on $I$.) To see that $\iota_{*}^{\dR}$ preserves $\omega_1$-compact objects,  since $I$ is countable, part (1.a) shows that an object $M\in \ob{D}(X^{\dR})$ is $\omega_1$-compact provided that $M|_{C_i}$ is $\omega_1$-compact for all $C_i$. Thus, by proper base change again, we can assume without loss of generality that $X$ is affinoid in which case it follows from  \cref{LemmaBasicNuclear} (2).
  
\item[(c)]  The fact that $\iota^{\dR,!}$ commutes with $\omega_1$-filtered colimits follows from the fact that its left adjoint $\iota^{\dR}_*$ preserves $\omega_1$-compact objects, cf.\ (1.b). To see that $\iota^{\dR,!}$ preserves locally nuclear modules, consider a strict cover of $X=\bigcup_i C_i$ as in part (1.b). By parts (1.a) and (1)b), $\iota^{\dR,!}$ preserves locally nuclear objects if and only if $\iota_*^{\dR}\iota^{\dR,!}$ preserves locally nuclear objects: indeed, one implication is a direct consequence of (1.b) and for the converse implication, note that
\[
\iota^{\dR,!}=\iota^{\dR,*}\iota_*^{\dR} \iota^{\dR,!}
\]
preserves locally nuclear objects if $\iota_*^{\dR} \iota^{\dR,!}$ does, by (1.a).  By $\ob{D}^*$-descent of locally nuclear objects, it suffices to see that the pullback  $f_{i}^{\dR,*}  \iota_{*}^{\dR}\iota^{\dR,!} $ preserves locally nuclear objects  for all $i\in I$.  Since the cover is strict, there is some strict inclusion
$$
h\colon C_i \subset V\subset   Y
$$
with $g\colon Y\subset X$ a closed affinoid subspace, and $j: V\subset X$ an open subspace. By proper  base change along $V\subset X$,  we have an equivalence of functors $\ob{D}(X^{\dR})\to \ob{D}(C_i^{\dR})$:
\begin{equation}\label{eqpmapwmqow}
\begin{aligned}
f_i^{\dR,*} \iota_{*}^{\dR} \iota^{\dR,!}
    &= h^{\dR,*} j^{\dR,*} \iota_{*}^{\dR} \iota^{\dR,!} \\
    &= h^{\dR,*} \iota_{Z\cap V\subset V,!}^{\dR} j_{Z\cap V\subset Z}^{\dR,*} \iota^{\dR,!} \\
    &= h^{\dR,*} \iota_{Z\cap V\subset V,!}^{\dR} \iota_{Z\cap V\subset V}^{\dR,!} j^{\dR,*}.
\end{aligned}
\end{equation}
where $\iota_{K\subset S}\colon K\to S$ is the inclusion along a closed map,  and $j_{W\subset S} \colon  W \to S$ is an inclusion along an open map. Let $\iota_{Y}\colon Y\to X$ be the inclusion, note that after applying the functors of the right term of \cref{eqpmapwmqow} to $ \id_{X^{\dR}}\to \iota_{Y,*}^{\dR}\iota_{Y}^{\dR,*} $ we get an equivalence. Hence, without loss of generality, we can assume that $Y=X$ and we are reduce to prove the statement when $Y$ is qfd affinoid. This follows from \cref{LemmaBasicNuclear}(1).

 \end{enumerate}

(2) Let $Z = X\backslash U$ be the closed complement of $U$ with inclusion $\iota\colon Z\to X$.

\begin{enumerate}

\item[(a)] It is clear that the pullback $j^{\dR,*}$ preserves locally nuclear  objects. To see that it preserves $\omega_1$-compact objects, it equivalent to show that its right adjoint  commutes with $\omega_1$-filtered colimits, this follows from (2.b) below.

\item[(b)] We have a fiber sequence of functors in $\ob{D}(X^{\dR})$
\[
\iota_{*}^{\dR} \iota^{\dR,!} \to \id_{X_{\dR}}\to j_*^{\dR} j^{\dR,*}.
\]
Since $\iota^{\dR,!}$ commutes with $\omega_1$-filtered colimits by (1.b), and since $\iota_*^{\dR}$  and $j^{\dR,*}$ preserve colimits, we deduce that $j_{*}^{\dR}$ preserves $\omega_1$-filtered colimits objects. Similarly, since both $\iota^{\dR,!}$ and $\iota_*^{\dR}$ preserve locally nuclear objects by (1), then so does $j_*^{\dR}j^{*,\dR}$. It is left to see that any locally nuclear object on $U^{\dR}$ arises as a pullback of a locally nuclear object on $X^{\dR}$. This will follow if $j_{!}^{\dR}$ preserves locally nuclear objects which will be proved in (2.c) below.

\item[(c)] Let $X=\bigcup_{i \in I}C_i$  be a countable strict cover of $X$ by rational subspaces. To see that $j_!^\dR$ preserves locally nuclear objects it suffices to do it after pullback along $C_i\to X$. Similarly, as $I$ is countable, to see that $j_!^\dR$ preserves $\omega_1$-compact objects it suffices to do it after pullback along $C_i\subset X$ (using (1.a)). Hence, by proper base change we can assume without loss of generality that $X$ is qfd affinoid. Let $\GSpec B \to X^{\dR}$ be a $!$-cover as in \cref{LemmaDescentArcProet}, as $B$ is basic nuclear and strictly totally disconnected, the pullback of $Z$ gives rise to a pullback diagram
\[
\begin{tikzcd}
\GSpec B_Z \ar[r] \ar[d] & Z^{\dR} \ar[d] \\ 
\GSpec B \ar[r]  & X^{\dR}
\end{tikzcd}
\]
where $B\to B_Z$ is an ind-finite \'etale map since $Z\to X^{\dR}$ is quasi-pro-\'etale. In particular, $B_Z$ is also basic nuclear.  Let $V= \GSpec B \backslash \GSpec B_Z$ be the open complement  in $\GSpec B$ and let $f\colon V\to \GSpec B$ be the inclusion. 
 To see that $j_!$ preserves locally nuclear objects (resp.\ $\omega_1$-compact objects), it suffices to base change along $\GSpec B\to X^{\dR}$, and therefore it suffices to see that $f_!$ preserves locally nuclear objects (resp.\ $\omega_1$-compact objects). We can write $V=\bigcup_n \GSpec B_n$ as a countable union of clopen  affinoid subspaces of $\GSpec B$ (as $\mathcal{M}(B)$ is light profinite any open subset is a countable union of clopen subspaces). Hence, for $M\in \ob{D}(V)$, we have that 
\[
f_!M =\varinjlim_{n}  f_{n,*} j^*M_n
\]
where $f_n\colon \GSpec B_n\to \GSpec B$ is the clopen immersion, As the algebras $B_n$ are basic nuclear, \cref{LemmaPermanenceOmega1Compact} and \cref{LemmaPermanence} implies that $f_!$ preserves $\omega_1$-compact objects and locally nuclear objects as wanted. 
\end{enumerate}

(3) Part (3.a) follows from the preservation of $\omega_1$-compact objects under $\iota^{*,\dR}$ and $\iota_{*,\dR}$ under closed immersions $\iota\colon Z\to X$, and $\ob{D}^*$-descent for strict closed covers  of quasi-coherent sheaves on de Rham stacks. Similarly, part (3.b) follows from the preservation of locally nuclear objects along the maps $\iota^!$, $\iota^*$ and $\iota_*$, and the $\ob{D}^*$ and $\ob{D}^!$-descent for strict closed covers of quasi-coherent sheaves on de Rham stacks.

(4) This follows  from the same argument as for (3).
\end{proof}

\begin{proposition}\label{PropDescentNuclearomega1CompactdR}
Let $f\colon Y\to X$ be a morphism of locally qfd  Berkovich spaces such that $Y$ and $X$ can be written as countable unions of qcqs Berkovich spaces. Let $f^{\dR}\colon Y^{\dR}\to X^{\dR}$ be the associated morphism of de Rham stacks. Then the following hold:
\begin{enumerate}
\item The functors $f^{\dR,!}$, $f_!^{\dR}$, $f_*^{\dR}$ and $f^{\dR,*}$ preserve locally nuclear objects. In particular, if $f$ is an arc-cover then  locally nuclear objects satisfy $\ob{D}^*$ and $\ob{D}^!$-descent.

\item The functors $f^{\dR,*}$ and $f_!^{\dR}$  preserve $\omega_1$-compact objects. In particular, if $f$ is proper $\omega_1$-compact objects satisfy $\ob{D}^*$-descent.

\end{enumerate}

\end{proposition}
\begin{proof}
Thanks to \cref{LemmaNuclearCompactLocalTopology}, locally nuclear  objects on de Rham stacks satisfy $\ob{D}^*$ and $\ob{D}^!$-descent on closed and open covers of $Y$ and $X$ respectively.  Similarly, $\omega_1$-compact objects satisfy $\ob{D}^*$-descent for closed covers and $\ob{D}^!$-descent for open covers. Thus,  for the functors $f^{!,\dR}$ and $f^{*,\dR}$, we can localize both $Y$ and $X$ in the analytic topology and assume that  they are qfd affinoid, in which case the statement follows from \cref{LemmaBasicNuclear}. For the functors $f_!^{\dR}$ and $f_*^{\dR}$ and locally nuclear objects, we can use either $\ob{D}^*$ or $\ob{D}^!$-descent on the target, and assume that $X$ is qfd affinoid. For $f_{!}^{\dR}$ and $\omega_1$-compact objects, we can use $\ob{D}^*$-descent and assume that $X$ is affinoid as well. Thus, without loss of generality let us assume that $X$ is affinoid. Next,   for the functor $f_{!}^{\dR}$, by writing  $Y=\bigcup_n Y_n$ as a strict union of  closed affinoid subspaces with interiors $\mathring{Y}_n$ still covering $Y_n$, by $\ob{D}^*$-descent on $Y$, and the fact that $\omega_1$-compact objects and locally nuclear  objects are stable under countable colimits, it suffices to consider the case $Y=\mathring{Y}_n$ for some $n$. Hence, we can assume that $Y\subset  \overline{Y}\to X$ is an open subspace in a qfd affinoid  $\overline{Y}$ mapping to $X$. By \cref{LemmaNuclearCompactLocalTopology} (2.c), the extension by zero along $Y\subset \overline{Y}$ preserves locally nuclear objects and $\omega_1$-compact objects. It follows from \cref{LemmaBasicNuclear} that $f_!^{\dR}$ does so.  Finally, we prove that $f_*^{\dR}$ preserves locally nuclear objects.  Since $Y=\bigcup_n Y_n$ and locally nuclear objects on $Y$ satisfy $\ob{D}^*$-descent (\cref{LemmaNuclearCompactLocalTopology} (3)), and  locally nuclear objects on $X^{\dR}$ are stable under countable limits (cf. \cref{LemmaDescentArcProet} (4)), it suffices to consider the case where $Y=Y_n$, i.e. when $Y$ is qfd affinoid. In this situation the proposition follows from \cref{LemmaBasicNuclear}.
\end{proof}

\subsection{Arc-hyperdescent for the de Rham stack}\label{Subsection:HyperdescentdR}
In this subsection, we finally prove the promised commutation of the formation of the de Rham stack with colimits, \Cref{TheoMaindeRham2}.  The key statement is the arc-hyperdescent of the formation of the analytic de Rham stack.

\begin{theorem}\label{TheoHyperdescentdR}
Let $X_{\bullet}\to X$ be an arc-hypercover of qfd affinoid perfectoid rings over $\Q_p$. Then the augmented simplicial diagram $X_{\bullet}^{\dR}\to X^{\dR}$ satisfies universal $\ob{D}^*$ and $\ob{D}^!$-descent in qfd Gelfand stacks. More precisely,
\begin{enumerate}
\item for all qfd Gelfand stacks $Y$ and maps $Y\to X^{\dR}$, the  pullback diagram $X^{\dR/Y}_{\bullet}\to Y$ satisfies $\ob{D}^*$-descent, that is, the natural map
\[
\ob{D}(Y)\to \ob{Tot}(\ob{D}(X_{\bullet}^{\dR/Y}))
\]
is an equivalence, where the transition maps are given by $*$-pullback maps,

\item for all $Y\in \Cat{GelfStk}^{\qfd}$ and maps $Y\to X^{\dR}$, the pullback diagram $X^{\dR/Y}_{\bullet}\to Y$ satisfies  $\ob{D}^!$-descent, that is, the natural map 
\[
\ob{D}(Y) \to  \ob{Tot}(\ob{D}^!(X_{\bullet}^{\dR/Y}))
\]
is an equivalence, where transition maps are given by $!$-pullback maps. 

\end{enumerate} 
\end{theorem}

A key ingredient in the proof of \cref{TheoHyperdescentdR} is the arc-hyperdescent of de Rham cohomology, that we will prove in \cref{LemmaDescentdRCoho}.

\begin{lemma}\label{LemmaCoconnectivedeRhamCoho}
  Let $Y\to X$ be a morphism of affinoid qfd Gelfand stacks and let $U\subset Y$ be an open substack.
  Write $Y=\GSpec (B)$ and $X=\GSpec (A)$.
  Let $f\colon U\to X$ be the restricted morphism and $f^{\dR/X}\colon U^{\dR/X}\to X$ the induced morphism on the relative Rham stack.
  Finally, suppose that $A$ is static and that any rational localization of affine spaces $\mathbb{A}^d_{A}$ remains static (e.g., $A$ could be a filtered colimit of qfd sous-perfectoid rings).
  Then $f_*^{\dR/X}  1$ and $f_!^{\dR/X} 1$ are coconnective.
\end{lemma}
\begin{proof}
  By \cref{xsh29k} we can assume without loss of generality that $B$ is $\dagger$-formally smooth over $A$, and that it can be written as a sequential colimit $B=\varinjlim_n B_n$ with $B_n$ a rational localization over an affine space over $A$.
  In particular, $B$ and its \v{C}ech nerve $A\to B^{\bullet}$  over $A$ remain static.
  Let $B^{\bullet,\dagger}$ be the simplicial ring of overconvergent neighbourhood of the diagonal of the \v{C}ech nerve of $Y\to X$. Notice that by the assumption on $A$, the ring $B^{\bullet,\dagger}$ is static.
  Let $Z=Y\backslash U$ be the closed complement. We can write $Z=\bigcap_n Z_n$ as a countable intersection of closed subspaces $Z_n$ which are finite unions of rational localizations of $Y$.
  By \cref{PropdeRhamBerkovich} (2) we have an equivalence of Gelfand stacks 
\[
| \GSpec B^{\bullet,\dagger}| \xrightarrow{\sim } Y^{\dR/X}. 
\]
Thus, if $U^{\bullet}\subset \GSpec B^{\bullet,\dagger}$ are the preimages of $U$, we have that 
\[
|U^{\bullet}|=U^{\dR/X}.
\]
In particular, we have an isomorphism in $\ob{D}(A)$
\[
f^{\dR/X}_* 1= \ob{Tot}(\Gamma(U^{\bullet}, \mathcal{O}))
\]
which is clearly a coconnective object as each $\Gamma(\Gamma(U^{\bullet},\mathcal{O}))$ is coconnective.

On the other hand,   we have fiber sequences  of cohomology with compact supports
\[
\Gamma_c(U^{\dR/X},\mathcal{O})\to \Gamma(Y^{\dR/X},\mathcal{O})\to \varinjlim_{n} \Gamma(Z_n^{\dR/X},\mathcal{O}).
\]
Since $\Gamma(Y^{\dR/X},\mathcal{O})$ is coconnective by the previous part, and the $\Gamma(Z_n^{\dR/X},\mathcal{O})$ are coconnective by the previous point after descending from  a rational cover, we deduce that $f_{!}^{\dR/X} 1$ is also coconnective as desired. 
\end{proof}

\begin{lemma}\label{LemmaDescentdRCoho}
  Let $Y_{\bullet}\to Y$ be an arc-hypercover of qfd affinoid perfectoid rings over $\Q_p$ and let $A$ be a static qfd Gelfand ring with a map $X:=\GSpec(A)\to Y^{\dR}$. Suppose that rational localizations of relative affine spaces  $\mathbb{A}^d_{A}$ over $A$ remain static.
  Consider the relative de Rham stacks $Y_{\bullet}^{\dR/X}$ of $Y_\bullet\times_Y X$ and let $\Gamma(Y_{\bullet}^{\dR/X},\mathcal{O})$ be the relative de Rham cohomology of $Y_{\bullet}\times_{Y} X$ over $X$.
  Then the natural map 
\begin{equation}\label{eqo0jo3rmqwdqd}
A\to  \ob{Tot}(\Gamma(Y_{\bullet}^{\dR/X},\mathcal{O}))
\end{equation}
is an equivalence. 
\end{lemma}
\begin{proof}
  By \cref{LemmaCoconnectivedeRhamCoho} the de Rham cohomologies $\Gamma(Y_{\bullet}^{\dR/X}, \mathcal{O})$ are coconnective. Let $n\in \N$.
  By the dual of \cite[Proposition 1.2.4.5]{lurie_higher_algebra}, for $k<n$ there is a natural isomorphism of solid $\Q_p$-modules
\[
\pi_{-k}(\ob{Tot}_{\leq n}(\Gamma(Y_{\bullet}^{\dR/X},\mathcal{O})))=\pi_{-k}(\ob{Tot}(\Gamma(Y_{\bullet}^{\dR/X},\mathcal{O}))). 
\]
Let $\widetilde{Y}_{\bullet}$ be the $n$-coskeleton of the $\Delta^{\op}_{\leq n}$-subdiagram $Y_{\bullet}$. By \cite[Lemma 6.5.3.9]{lurie_higher_topos_theory} and the arc-descent of the de Rham stack of \cref{sec:furth-results-analyt-1-surjections-on-de-rham-stack}, the natural map of geometric realizations $|\widetilde{Y}_{\bullet}^{\dR/X}|\to X$   is an isomorphism and we have that
\[
A=\ob{Tot}(\Gamma(\widetilde{Y}^{\dR/X}_{\bullet}, \mathcal{O})). 
\]
Hence, we deduce that  for $k<n$
\[
\begin{gathered}
\pi_{-k}(A)\cong \pi_{-k} (\ob{Tot}(\Gamma(\widetilde{Y}^{\dR/X}_{\bullet}, \mathcal{O}))) \cong \pi_{-k}(\ob{Tot}_{\leq n}(\Gamma(\widetilde{Y}^{\dR/X}_{\bullet}, \mathcal{O}))) \cong \pi_{-k}(\ob{Tot}_{\leq n}(\Gamma(Y_{\bullet}^{\dR/X},\mathcal{O})))) \\ 
\cong \pi_{-k}(\ob{Tot}(\Gamma(Y_{\bullet}^{\dR/X},\mathcal{O})))).
\end{gathered}
\]
Since $n$ was arbitrary, we deduce that the natural map  \eqref{eqo0jo3rmqwdqd} is an equivalence, proving what we wanted. 
\end{proof}

We will also need the following lemmata.

\begin{lemma}\label{LemmaExtEtalesprelim}
  Let $d\in \N$, $d \geq 1$. Let $N,M$ be  static $A=\Q_p\langle T_1,\ldots, T_d \rangle_{\leq 1}$-modules with $N$ an $\omega_1$-compact static $A$-module.
  Then for $i>d+2$ we have
\[
\underline{\Ext}^i_A(N,M)=0.
\]
\end{lemma}
\begin{proof}
The morphism of solid algebras $\Q_p[T_1,\ldots, T_d]\to \Q_p\langle T_1,\ldots, T_d \rangle_{\leq 1}=A$ is idempotent. Hence, we have (in the following the $\Hom$'s are always derived)
\[
\iHom_A(N,M)=\iHom_{\Q_p[T_1,\ldots,T_d]}(N,M)\cong \iHom_{\Q_p}(N,M)\otimes_{\Q_p[T_1,\ldots, T_d,S_1,\ldots, S_d]} \Q_p[T_1,\ldots, T_d][-d]
\]
where the morphism $\Q_p[T_1,\ldots, T_d,S_1,\ldots, S_d]\to \Q_p[T_1,\ldots, T_d]$ sends, $T_i,S_i\mapsto T_i$, and the  $\Q_p[T_1,\ldots, T_d,S_1,\ldots, S_d]$-action on $\Hom_{\Q_p}(N,M)$ is via the action of $\Q_p[T_1,\ldots, T_d]$ on $N$, and the action of $\Q_p[S_1,\ldots, S_d]$ on $M$.
Therefore, to prove the claim, it suffices to show that $\iHom_{\Q_p}(N,M)$ is concentrated in cohomological degrees $[0,2]$.
The algebra $A$ is $\omega_1$-compact as a solid $\Q_p$-module (e.g., being  a light DNF space).
Thus, $N$ is an $\omega_1$-compact solid $\Q_p$-module by \cref{LemmaPermanenceOmega1Compact}.
Then, we can write $N=\varinjlim_n N_n$ as a countable colimit of finitely presented solid $\Q_p$-modules, and by taking $R\lim_n$ it suffices to show that for $N$ finitely presented we have that $\iHom_{\Q_p}(N,M)$ is in cohomological degrees $[0,1]$.
In this case, $N$ sits in a right exact sequence
\[
\Q_{p,\solid}[S'] \to \Q_{p,\solid}[S]\to N\to 0
\]
where $S$ and $S'$ are  light profinite sets.
Now, by \cref{lemmaSmith}, the image of $\Q_{p,\solid}[S'] $ in $\Q_{p,\solid}[S]$ is a light Smith space.
Hence, after modifying $S'$, we can assume without loss of generality that we have a short exact sequence
\[
0\to \Q_{p,\solid}[S'] \to \Q_{p,\solid}[S]\to N\to 0.
\]
This implies that $N$ has projective dimension $\leq 1$, and so that $\iHom_{\Q_p}(N,M)$ is in cohomological degrees $[0,1]$ as desired.
\end{proof}

The following lemma improves on the previuos one.

\begin{lemma}\label{LemmaExtEtales}
Let $d\in \N$ and let  $f\colon \GSpec(A)\to  \A^d_{\Q_p}$ be a Berkovich \'etale map.  Let $N\in \ob{D}(A)^{\heartsuit}$ be a  static $\omega_1$-compact solid $A$-module, and let $M\in \ob{D}(A)^{\heartsuit}$ be an arbitrary static object. Then 
\[
\underline{\ob{Ext}}^i_{A}(N,M)=0
\]
for $i>d+2$.
\end{lemma}
\begin{proof}
We can suppose without loss of generality that $f$ factors through $$\mathbb{D}^{d,\leq 1}_{\Q_p}=\GSpec(\Q_p\langle T_1,\ldots, T_d \rangle_{\leq 1}).$$ When $A=\Q_p\langle T_1,\ldots, T_d \rangle_{\leq 1}$ itself, we know the result by \Cref{LemmaExtEtalesprelim}. Thus, it suffices to show if $A\to B$ is a Berkovich \'etale map of Gelfand rings, $N,M$ are static $B$-modules such that $\iHom_{A}(N,M)[k]$ is connective for some $k\in \N$, then $\iHom_{B}(N,M)[k]$ is connective. 

In this situation, the map of arc-stacks $\Marc(B)\to \Marc(A)$ is \'etale, this makes the diagonal map $\Marc(B)\to \Marc(B)\times_{\Marc(A)} \Marc(B)=\Marc(B\otimes_A B)$ a clopen immersion. On the other hand, since $A\to B$ is Berkovich \'etale, it is $\dagger$-formally \'etale as in \cref{sec:appr-gelf-rings-dagger-formally-smooth} and  we have a cartesian diagram 
\[
\begin{tikzcd}
\GSpec(B) \ar[r] \ar[d] & \GSpec(A) \ar[d] \\
\Marc(B)^\dR \ar[r] & \Marc(A)^\dR.
\end{tikzcd}
\]
Similarly, we have that $\GSpec (B\otimes_A B) = \Marc(B\otimes_A B)^\dR \times_{\Marc(A)^\dR} \GSpec(A)$. This produces a cartesian diagram
\[
\begin{tikzcd}
\GSpec(B) \ar[r] \ar[d] & \GSpec (B\otimes_A B )\ar[d] \\
\Marc(B)^\dR \ar[r] & \Marc(B\otimes_A B)^\dR
\end{tikzcd}
\]
proving that $\GSpec(B)\to \GSpec(B\otimes_A B)$ is a clopen immersion. Therefore, by considering the corresponding idempotent element associated to $\Marc(B)\subset \Marc(B\otimes_A B)$,  there is a $B\otimes_A B$-linear retract $B\to B\otimes_A B$. Tensoring with $N$ we have a retract of $B$-modules $N = B\otimes_B N\to (B\otimes_A B)\otimes_B N = B\otimes_A N$, where $B$ acts on the left tensor in $B\otimes_A N$. Thus, the following map admits a retract
\[
\iHom_{B}(N,M) \to \iHom_{B}(B\otimes_A N, M)= \iHom_{A}(N,M),
\]
proving that $\iHom_B(N,M)[k]$ is connective as desired. 
\end{proof}

We are ready to prove the arc-hyperdescent of de Rham stacks.

\begin{proof}[Proof of \cref{TheoHyperdescentdR}]
We keep the notation of the theorem.  Consider the topos $\Cat{ArcStk}^{\qfd}_{/X}$ and consider the following functors 
\[
F\colon  \Cat{Perfd}^{\ob{qfd},\ob{aff}}_{/X} \to \Cat{Cat}_{\infty}
\]
 and 
\[
G\colon \Cat{Perfd}^{\qfd,\ob{aff}}_{/X}\to \Cat{Cat}_{\infty}
\]
sending a qfd affinoid perfectoid $\Marc(A)$ over $X$ to 
\[
F(A)= \ob{D}^*(\Marc(A)^{\dR}\times_{X^{\dR}} Y) \mbox{ and } G(A)=\ob{D}^!(\Marc(A)^{\dR}\times_{X^{\dR}} Y),
\]
where the transition maps for $F$ are given by pullback maps, and the transition maps for $G$ are given by upper $!$-maps (notice that in the case of $G$ the map $\Marc(A)^{\dR}\to X^{\dR}$ is $!$-able by \cref{LemmLqfd}). 
We want to prove that both $F$ and $G$ are hypersheaves for the arc-topology on $\Cat{ArcStk}^{\qfd}_{/X}$. Thanks to \cref{sec:furth-results-analyt-1-surjections-on-de-rham-stack} (for $*$-pullbacks in $F$) and \cref{CorDescendableMapsqfla} (for $!$-pullbacks in $G$) we know that these functors are sheaves. 
Hence, by \cite[Proposition A.3.21]{mann2022p}, it suffices to show that for any hypercover $X_{\bullet}\to X$ by qfd affinoid perfectoids, the natural maps 
\begin{equation}\label{eqHyperSheafReduction1}
F(X)\to \ob{Tot}(F(X_{\bullet})) \mbox{ and } G(X)\to \ob{Tot}(G(X_{\bullet})) 
\end{equation}
are fully faithful. 

By \Cref{rk:refined-not-std-but-berk-pro-et} (a variant of \cref{CorDescendableMapsqfla}(1)), there exists a static qfd Gelfand ring $A$ and a prim and descendable map $g\colon X'=\GSpec(A)\to X^{\dR}$. Hence the map $g$ is of universal $\ob{D}^*$ and $\ob{D}^!$-descent (\cite[Lemma 4.7.4]{heyer20246functorformalismssmoothrepresentations}). Hence, to prove that \eqref{eqHyperSheafReduction1} are fully faithful, by $\ob{D}^*$ or $\ob{D}^!$-descent respectively, it suffices to prove it after base change along $X'\to X^{\dR}$ and assume that we have a factorization $Y\to X'\to X^{\dR}$.

 Let $f_{n}\colon X_{n}^{\dR/Y}\to Y$ be the structural map. Fully faithfulness for the functor $F$ is equivalent to the following: given $M\in \ob{D}(Y)$ the natural map 
 \begin{equation}\label{eq10ef1wd1o3}
 M\to \ob{Tot}(f_{n,*} f^{*}_n M)
 \end{equation}
 is an isomorphism. 
Fully faithfulness for $G$ is equivalent to the following: given $M\in \ob{D}(Y)$ the natural map 
\begin{equation}\label{eqo01newfqow}
\varinjlim_{[n]\in \Delta^{\op}} f_{n,!} f^{!}_n M \to  M
\end{equation}
is an isomorphism.  Since the $f_{n,*}$ are cohomologically proper we have $f_{n,*}=f_{n,!}$, and thanks to proper base change the right terms of  \eqref{eq10ef1wd1o3} is equal to $\ob{Tot} (f_{n,*}1 \otimes M)$.
Similarly, the left term of \eqref{eqo01newfqow} is the geometric realization of the diagram 
\[
f_{\bullet,*}f^{!}_{\bullet} M = f_{\bullet, *} \Hom_{X_{n}^{\dR/Y}}(1, f_n^! M) =\Hom_{Y}(f_{\bullet,!} 1, M )= \Hom_{Y}(f_{\bullet,*} 1, M ).
\]
Therefore, following the same argument of \cite[Lemma 4.7.4 (iii)]{heyer20246functorformalismssmoothrepresentations}, to prove fully faithfulness it suffices to prove that the cosimplicial diagram $(f_{\bullet,*} 1)_{\Delta^{\op}}$ of commutative algebras in $\ob{D}(Y)$ is pro-isomorphic to $1_{Y}$, that is, the augmented cosimplicial diagram $1_{Y}\to f_{\bullet,*} 1$ is descendable. To prove this, by base change, it suffices to prove it when $Y=X'$. Furthermore, by the construction in \cref{rk:refined-not-std-but-berk-pro-et}, we can assume that $A=\varinjlim_{n} A_n$ is a countable filtered colimit of Berkovich \'etale maps over some fixed affine space $\A^d_{\Q_p}$.

We want to prove the following claim: 
\begin{claim}
Let $Y=\GSpec(A)$ as before and consider the cosimplicial diagram of relative de Rham cohomologies $\Gamma(X_{\bullet}^{\dR/Y}, \mathcal{O})$ in $\ob{D}(A)$. Then the tower of totalizations of the  cosimplicial $A$-object $\Gamma(X_{\bullet}^{\dR/Y}, \mathcal{O})$ is pro-constant isomorphic to $A$.
\end{claim}

To prove the claim,  by \cref{LemmaDescentdRCoho} we know that 
\[
A=\ob{Tot}(\Gamma(X_{\bullet}^{\dR/Y}, \mathcal{O})). 
\]
Furthermore, by the proof of \cref{LemmaDescentdRCoho}, the terms $\Gamma(X_{\bullet}^{\dR/Y},\mathcal{O})$ are coconnective $A$-modules, and given $k\in \N$,  the cofiber  $Q_k$
\[
A\to \ob{Tot}_{\leq k} (\Gamma(X_{\bullet}^{\dR/Y}, \mathcal{O})) \to Q_k
\]
is $(k-1)$-coconnective (i.e., sitting in cohomological degrees $\geq k-1$).
We want to show that there is some $k\in \N$ such that the map $Q_k\to A[1]$ is zero in $\ob{D}(A)$.   By \cref{LemmaBasicNuclear}(2) and \Cref{LemmaDescentArcProet}(2), the cohomologies $\Gamma(X_{\bullet}^{\dR/Y}, \mathcal{O})$ are $\omega_1$-compact $A$-modules. Thus, it suffices to show that there is some $k\in \N$ such that if $N$ is a static $\omega_1$-compact $A$-module and $M\in \ob{D}(A)^{\heartsuit}$ is arbitrary, we have $\Ext^i_A(N,M)=0$ for $i>k$. Writing $A$ as the colimit of the rings $A_n$, we have that 
\[
\varprojlim_n \Hom_{A_n}(N,M) = \varprojlim_n \Hom_{A}(N\otimes_{A_n} A,M) = \Hom_{A}(\varinjlim_n N\otimes_{A_n} A,M)= \Hom_{A}(N,M).
\]
Hence, by taking derived $R^i\lim$, it suffices to show that there exists a uniform $k\in \N$ such that $\Ext^i_{A_n}(N,M)=0$ for $i>k$ for $N$ static $A_n$-module. Since $A_n$ is Berkovich \'etale over $\A_{\Q_p}^d$ (with $d$ independent of $n$), the bound of the Ext groups follows from \cref{LemmaExtEtales}. This finishes the proof.
\end{proof}

\begin{theorem}\label{TheoMaindeRham2}
  The analytic de Rham stack functor $(-)^{\dR}\colon \Cat{ArcStk}_{\Q_p}^{\qfd}\to \Cat{GelfStk}^{\qfd}$  is  a pullback of morphism of $\infty$-topoi, i.e., it is left exact and  commutes with colimits.
  In particular, if $X_{\bullet}\to X$ is an arc hypercover of qfd arc-stacks, then the natural map  $|X_{\bullet}^{\dR}|\to X^{\dR}$ is an equivalence.
\end{theorem}
\begin{proof}
We only need to show that the de Rham stack functor $(-)^{\dR}\colon \Cat{ArcStk}_{\Q_p}^{\qfd}\to \Cat{GelfStk}^{\qfd}$, which is defined as the right adjoint of perfectoidization, commutes with colimits.
 This follows formally from \cref{TheoHyperdescentdR} as we explain now.
 Let $\iota\colon \Cat{Perfd}^{\qfd,\mathrm{aff}}_{\Q_p}\subset \Cat{GelfStk}^{\ob{qfd}, \mathrm{nil}}$ be the categories of qfd perfectoid and nilperfectoid affinoid Gelfand stacks.
 The inclusion $\iota$ has a right adjoint given by sending a nilperfectoid qfd Gelfand affinoid stack $\GSpec(A)$ to its uniform completion (equivalently, to its perfectoidization) $\Marc(A^u)$.
 Let $f\colon \Cat{GelfStk}^{\ob{qfd}, \mathrm{nil}} \to \Cat{Perfd}^{\qfd,\mathrm{aff}}_{\Q_p}$ be this right adjoint.
 Passing to presheaves on anima, we get the following adjunction

\begin{equation}\label{adjj}
\begin{tikzcd}
  \mathcal{P}(\Cat{Perfd}^{\qfd,\mathrm{aff}}_{\Q_p}) \ar[r, shift left, "f^*"] & \ar[l, shift left,"f_*"] \mathcal{P}(\Cat{GelfStk}^{\ob{qfd}, \mathrm{nil}})
\end{tikzcd}
\end{equation}

Where the functor $f^*$ sends a presheaf $Y\colon \Cat{Perfd}^{\qfd,\mathrm{aff},\op}_{\Q_p}\to \Cat{Ani}$ to the presheaf on qfd affinoid nilperfectoid Gelfand stacks given by $f^*Y(\GSpec(A)) = Y(A^u)$, that is, $f^* Y=X^{pre\dR}$ is the analytic de Rham prestack.
 For $X\in  \mathcal{P}(\Cat{GelfStk}^{\ob{qfd}, \mathrm{nil}})$ and $A$ a qfd perfectoid ring, one has that
\[
(f_*X) (\Marc(A)) = X(\Marc(A)^{pre\dR}).  
\]
Notice, however, that since we are using nilperfectoid rings, $\Marc(A)^{pre\dR}=\Marc(A)^{\dR}$ for $A$ perfectoid, namely, its de Rham pre-stack restricted to nilperfectoid rings already satisfies $!$-descent.

We claim that the adjunction \eqref{adjj} restricts to an adjunction
$$
\begin{tikzcd}
  \Cat{ArcStk}^{\qfd}_{\Q_p} \ar[r, shift left, "f^*"] & \ar[l, shift left,"f_*"] \Cat{GelfStk}^{\qfd}.
\end{tikzcd}
$$
This would imply that $f^*=(-)^{\dR}$ preserves colimits as wanted.
 To show this, it suffices to prove that $f^*$ sends a qfd arc-stack to a qfd Gelfand stack (resp.\  $f_*$ sends a qfd Gelfand stack to a qfd arc-stack).
 The claim for $f^*$ follows from (the proof of) \cref{sec:totally-disc-stacks-construction-perfectoidization} as any $!$-equivalence of separable Gelfand rings induces an arc-hypercover on uniform completions.
  The claim for $f_*$ follows from \cref{TheoHyperdescentdR}.
\end{proof}

\begin{corollary}\label{CorodeRhamNuclearSame}
Keep the notation of \cref{PropDeRhamNuclear}. Let $X\in \Cat{ArcStk}^{\ob{qfd}}$ be a qfd arc-stack,  then the natural morphism of de Rham stacks 
\[
f^*X^{\dR,\ob{nuc}}=f^*f_* X^{\dR}\to  X^{\dR}
\]
is an equivalence of qfd Gelfand stacks. In other words, the analytic de Rham stack $X^{\dR}$ admits a natural refinement to a qfd nuclear Gelfand stack $X^{\dR,\ob{nuc}}$. In particular, it has a well defined full subcategory $\Cat{Nuc}^{\ob{loc}}(X)\subset \ob{D}(X)$ of locally nuclear objects.   
\end{corollary}
\begin{proof}
The same argument as in the proof of \cref{TheoMaindeRham2} implies that the nuclear de Rham functor
\[
(-)^{\dR,\ob{nuc}}\colon \Cat{ArcStk}_{\Q_p}^{\qfd}\to \Cat{GelfStk}^{\qfd, \ob{nuc}}
\] 
commutes with colimits. Since the pullback map $f^*\colon \Cat{GelfStk}^{\qfd, \ob{nuc}}\to \Cat{GelfStk}^{\qfd}$ also commutes with colimits, the corollary follows from the affinoid case of \cref{Lemqqj3qw}. 
\end{proof}

\begin{example}\label{ExamCondensedAnimadeRham}
Let us give one useful consequence of the commutation of colimits of the de Rham stack.
  Recall from \cref{ExamCondAniqfd} that we have a left exact Betti realization of hypersheaves on light profinite sets on qfd Gelfand stacks
\[
(-)_{\Betti}\colon \widehat{\Shv}(\Cat{Prof})\to \Cat{GelfStk}^{\qfd}
\] 
given by the unique colimit preserving functor that sends a light profinite set $S$ to $S_{\Betti}=\GSpec\big( C^{\lc}(S, \Q_p)\big)$.
 We also have a realization of hypersheaves on light profinite sets on qfd arc-stacks
\[
\underline{(-)}\colon \widehat{\Shv}(\Cat{Prof})\to \Cat{ArcStk}_{\Q_p}^{\qfd}
\]
given by  the same construction of \cref{sec:light-arc-stacks-1-condensed-anima-and-arc-stacks}.
 The composition of $\underline{(-)}$
 with the de Rham stack $(-)^{\dR}\colon \Cat{ArcStk}_{\Q_p}^{\qfd}\to \Cat{GelfStk}^{\qfd}$ is the unique colimit preserving functor that sends a light profinite set $S$ to $(\underline{S})^{\dR}=S_{\Betti}$ (thanks to \cref{sec:dagg-form-smooth-example-overconvergent-cohomology-of-rigid-disc} (3) and \cref{PropdeRhamBerkovich} (4)).
 Therefore,  we have a natural equivalence $(-)_{\Betti}= \underline{(-)}^{\dR}$.

  When composing along the left Kan extension $\Cat{GelfStk}^{\qfd}\to \Cat{GelfStk}$, the composite functor  $\widehat{\Shv}(\Cat{Prof})\to \Cat{GelfStk}$ factors  through the Betti stack for condensed anima $(-)_{\Betti}\colon \Cat{CondAni}\to \Cat{GelfStk}$, in particular, the $6$-functor formalism of the de Rham stack contains all the $!$-able functors of condensed anima of \cite[Definition 3.5.17]{heyer20246functorformalismssmoothrepresentations} for coefficients over a field in characteristic zero.

\end{example}

\subsection{Cohomological smoothness of the de Rham stack of the punctured perfectoid
disc}
\label{subsubsec:de-rham-stack-punctured-perfectoid-opend-disc}
We prove the  cohomological smoothness of the de Rham stack of the  perfectoid punctured open disc in \cref{sec:geom-prop-analyt-1-cohom-smoothness-of-perfectoid-unit-disc}.
At first glance this might look surprising: inverse limits of cohomologically smooth maps are rarely cohomologically smooth again (e.g., the uncompleted perfectoid open unit disc is not cohomologically smooth).
However, this result has to be expected to be true, e.g., in comparison with the $\ell$-cohomological smoothness of positive slope Banach-Colmez spaces for $\ell\neq p$.

\begin{proposition}
  \label{sec:geom-prop-analyt-1-cohom-smoothness-of-perfectoid-unit-disc}
Let $X=\mathring{\mathbb{D}}^{\times,\diamond}_{\infty}=\varprojlim_{x\mapsto x^p} \mathring{\mathbb{D}}^{\times,\diamond}_{\mathbb{Q}_p}$ be the open punctured pre-perfectoid unit disc seen as an arc-stack. Then, the map $X^{\dR}\to \GSpec(\mathbb{Q}_p)$ is cohomologically smooth with dualizing sheaf isomorphic to $1_{X}[2]$.
\end{proposition}
\begin{proof}
   We will apply the criterion of \cref{TechnicalLemma6Functors}.
  Thanks to \Cref{PropApproxFFStk}, we know that
  $$\mathring{\DD}^{\times,\dR}_{\infty}=\varprojlim_{x\mapsto x^p} \mathring{\DD}^{\times,\dR}$$
   in the category of kernels $\ob{K}_{\ob{D}}$ of Gelfand stacks.
   
    Since the statements to be proved are local for the $!$-topology on $\GSpec(\Q_p)$, we can base change along $\Q_p\to \Q_p^{\cyc}$ and consider relative de Rham stacks over the completed cyclotomic extension. For the rest of this proof, all the spaces considered are base changed to $\Q_p^{\rm cyc}$, even though we do not write it, for simplicity of notation.

  Denote  by $\mathring{\DD}^{\times}_{n}$  the open punctured unit disc with variable $T^{1/p^n}$ over $\Q_p^\cyc$, and  let $\mathring{\DD}^{\times}_{\infty}=\varprojlim_{n} \mathring{\mathbb{D}}^{\times}_{n}$ be the limit of derived Berkovich spaces. Notice that $\mathring{\DD}^{\times}_{\infty}$ is the uncompleted perfectoid disc. In $\ob{K}_{\ob{D}, \Q_p^\cyc}$, we can rewrite the above formula for the de Rham stack as 
  \[
\mathring{\DD}^{\times,\dR}_{\infty}=\varprojlim_{n} \mathring{\DD}^{\times,\dR}_{n}
\]

The next step is to rewrite $\mathring{\DD}^{\times,\dR}_{n}$ as a suitable limit of Gelfand stacks; we can assume that $n=0$ and simply write $\mathring{\DD}^{\times}$.

\begin{claim}
We have the identification $$\mathring{\DD}^{\times,\dR} = \varprojlim_{r\to 0} \mathring{\DD}^{\times}/ (1+ \DD^{\leq r})$$ as objects in the category of kernels over $\Q_p^{\cyc}$, where $\DD^{\leq r}$ is the overconvergent closed disc of radius $r$.
\end{claim}
\begin{proof}[Proof of the claim]
  Let $X_{\infty}=\DD^{\circ,\times,\dR} $ and $X_{n}= \mathring{\DD}^{\times}/ (1+ \DD^{\leq 1/2^n})$  where we have normalized the norm so that $|p|=1/2$. For $m\geq n\in \N \cup\{\infty\}$ let us write $h_{m\to n}\colon X_m\to X_n$.
  We have that $X_n=X_{n}^{\dR}$. Indeed, we recall that we have an epimorphism $\mathbb{G}_m\to \mathbb{G}_m^{\dR}$ of groups, whose kernel is $\mathbb{G}_m^{\dagger}$, the overconvergent neighbourhood  of $\mathbb{G}_m$ at $1$. Thus, $\mathbb{G}_m^{\dR}=\mathbb{G}_m/\mathbb{G}_m^{\dagger}$. Taking pullbacks along the map $\mathring{\DD}^{\times \dR}\to \mathbb{G}_m^{\dR}$ one has the presentation
  \[
  \mathring{\DD}^{\times \dR}= \mathring{\DD}^{\times}/ \mathbb{G}_m^{\dagger}.
  \]
Similarly, one has the presentation 
\[
(1+ \DD^{\leq 1/2^n})^{\dR}= (1+ \DD^{\leq 1/2^n})/ \mathbb{G}_m^{\dagger}
\]  
and therefore $X_n=X_n^{\dR}$.

  On the other hand,  the map of arc-stacks $X_m^{\diamond}\to X_n^{\diamond}$ are proper, hence \cref{sec:geom-prop-analyt-1-cohomological-smoothness-for-smooth-rigid-spaces} implies that $h_{m\to n}$ is cohomologically proper.
  It is easy to see that the natural map $1\to h_{\infty\to n,*}1 $ is an equivalence in $\ob{D}(X_n)$, namely, one can check this after pulling back along $\mathring{\DD}^{\times}\to X_n$, in which case the claim reduces to proving that the de Rham cohomology of the overconvergent disc $1+\overline{\DD}^{\dagger}(1/2^n)$ is trivial. Indeed, this follows from the pullback square
  \[
  \begin{tikzcd}
  ((1+ \DD^{\leq 1/2^n})/\mathbb{G}_m^{\dagger} ) \times  \mathring{\DD}^{\times}   \ar[r] \ar[d] &  \mathring{\DD}^{\times} \ar[d] \\ 
\mathring{\DD}^{\times}/\mathbb{G}_m^{\dagger} \ar[r]  & \mathring{\DD}^{\times}/ (1+ \DD^{\leq 1/2^n})
  \end{tikzcd}
  \]
  where the upper horizontal map is the projection map, and the left vertical map is given by multiplication.

  By projection formula,  it follows that $h_{\infty\to n}^* \colon \ob{D}(X_n)\to \ob{D}(X_{\infty})$  is fully faithful, and so are the $h_{m\to n}^*$.
  Thus, we also have $h_{m\to n,*}1=1$. This verifies the conditions (a) and (b) of \cref{TechnicalLemma6Functors}. Condition (c) is also easily verified by pulling back along $\mathring{\DD}^{\times}\to X_{\infty}$  using \cref{LemLimitsBerkovichSpaces}.
\end{proof}

Thanks to the previous claim and \cref{TechnicalLemma6Functors}, to prove suaveness of $\mathring{\DD}^{\times,\dR}_{\infty}  $ we are reduced to prove that given $n\geq 1$ there is some $m=m(n)$ fitting in a diagram
\[
\begin{tikzcd}
\mathring{\DD}^{\times,\dR}_{\infty} \ar[rd, "f_{\infty}"']  \ar[r] & \mathring{\DD}^{\times,\dR}_{m} \ar[d,"f_m"] \\ 
& {\mathring{\DD}^{\times}/(1+\DD^{\leq 1/2^n}) }
\end{tikzcd}
\] 
such that the natural map $f_{m,*}1\to f_{\infty,*} 1$ is an equivalence (with $f_{\infty}$ and $f_m$ as in the previous diagram). Indeed, since $f_{m}$ is prim and $\mathring{\DD}^{\times,\dR}_{m} $ is suave over $\Q_p^{\cyc}$ by \cref{sec:geom-prop-analyt-1-cohomological-smoothness-for-smooth-rigid-spaces}, this would imply that $f_{\infty,*} 1$ is suave.
But then, we can write $\mathring{\DD}^{\times,\dR}_{\infty}$ as limit of the quotients $\mathring{\DD}^{\times}_n/(1+\DD^{\leq 1/2^k}) $  in the kernel category for $n,k\to\infty$, and so $1\in \ob{D}(\mathring{\DD}^{\times,\dR}_{\infty} )$ will be suave as well (by \cref{TechnicalLemma6Functors}). By base change it suffices to prove the following claim:
\begin{claim}
  For $m\in \N$, let $\overline{\DD}_m=\GSpec (\Q_p^{\cyc}\langle T^{1/p^m}\rangle)$ be the compactification of the open Tate  unit disc over $\Q_p^{\cyc}$ of radius $1$ and coordinate $T^{1/p^m}$.  Let $\overline{\DD}_{\infty}=\varprojlim_{m} \overline{\DD}_m$ be the limit in Gelfand stacks, that is the uncompleted  perfectoid closed unit  disc on the variables $T^{1/p^{\infty}}$.
  Consider the following diagrams of Gelfand stacks
\[
\begin{tikzcd}
\overline{\DD}_{\infty}/\mathbb{G}_m^{\dagger} \ar[r] \ar[rd, "f_{\infty}"']& \overline{\DD}_{m}/\mathbb{G}_m^{\dagger}  \ar[d, "f_m"] \\
 & \overline{\DD}/ (1+ \DD^{\leq 1/2^n})
 \end{tikzcd}
\]
Then for $m\geq n-1$, the natural map $f_{m,*} 1 \to f_{\infty,*}1$ is an equivalence. 
\end{claim}
\begin{proof}[Proof of the claim]
  Let us write $K=\Q_p^{\cyc}$.
  The logarithm map induces an isomorphism of groups $\log \colon 1+ \DD^{\leq 1/2^n}\xrightarrow{\sim } \DD^{\leq 1/2^n}$ (for $n\geq 2$ if $p=2$).   By \cite[Theorem 4.3.8]{camargo2024analytic},  one has a Cartier duality between $\DD^{\leq 1/2^n}$ and $\mathring{\DD}^{\leq 2^n}$, which yields an equivalence of categories via the Fourier--Mukai transform
\[
\ob{D}(\GSpec(K)/ \DD^{\leq 1/2^n})\cong \ob{D}(\mathring{\DD}^{\leq 2^n}).
\]
Similarly, under the logarithm map $\log\colon \mathbb{G}_m^{\dagger}\to \Ga^{\dagger}$ one has an equivalence of categories (cf. \cite[Theorem 4.3.13]{camargo2024analytic})
\[
\ob{D}(\GSpec(K) / \Ga^{\dagger})\cong \ob{D}(\Ga^{\an}).
\]
Let $U$ be the coordinate of $\Ga^{\an}$ seen as the Cartier dual of $\Ga^{\dagger}\cong \mathbb{G}_m^{\dagger}$.
Then $\shf{O}(\overline{\DD}_{m}) = K\langle T^{1/p^m} \rangle$ is endowed with an endomorphism  $U$  which is identifies with the logarithmic derivation $T\partial_{T}$, resp. for $\shf{O}(\overline{\DD}_{\infty})=\varinjlim_{k} K\langle  T^{1/p^k}\rangle$. This endomorphism corresponds to the action of $\mathbb{G}_m^{\dagger}$ on the discs by multiplication, and this allows us to see both $\shf{O}(\overline{\DD}_{m})$ and  $\shf{O}(\overline{\DD}_{\infty})$ as quasi-coherent sheaves on $\mathbb{G}_a^{\an}$ via Cartier duality.

Consider the cofiber  $P=\ob{cofib}( K\langle T^{1/p^m} \rangle \to \varinjlim_{k} K\langle T^{1/p^k} \rangle \rangle)$ seen as a quasi-coherent sheaf of $\mathbb{G}_a^{\an}$.   To prove the claim it suffices to show that the pullback of  $P$ to   $\mathring{\DD}^{\leq 2^n}\subset \Ga^{\an}$ is zero. To show this, it suffices to prove the same for the quotients
\[
P_k= K\langle T^{1/p^{m+k}} \rangle/ K\langle T^{1/p^m}\rangle= \widehat{\bigoplus_{\substack{r\in p^{-(m+k)} \N \\  v_p(r)\leq  - m-1}}} K \cdot T^{r}.
\] 
Clearly $(T\partial_T) T^r = rT^r$, thus the spectral decomposition of $T\partial_T$ on $P_k$ has slopes $\geq m+1$ (i.e. eigenvalues $\lambda$ of norm $|\lambda|\geq |p^{-(m+1)}|=2^{m+1}$).
Taking $m\geq n-1$, we see that the action of $U=T\partial_T$ on $P_k$ localizes to an action of the rational localization $\Ga^{\an}(|U|\geq 2^n)$,  and so its pullback to $\Ga(|U|<2^n)=\DD^{\circ}(2^{n})$ vanishes.
This proves the claim.
\end{proof}

It is left to identify the dualizing sheaf of $\mathring{\DD}^{\times,\dR}_{\infty} $.
By \cref{TechnicalLemma6Functors} it is the colimit along the upper $\flat$-pullbacks of the maps $\mathring{\DD}^{\times,\dR}_{\infty} \to \mathring{\DD}^{\times,\dR}_{n}$ of the dualizing sheaves of the finite level discs.
Since these maps are cohomologically proper and \'etale, the upper $\flat$ and $*$-pullbacks agree.
Then, \cite[Theorem 3.5.7]{camargo2024analytic} implies that the dualizing sheaf of $X_n=\mathring{\DD}^{\times,\dR}_{n}$ identifies with $1_{X_n}[2]$. The proposition follows.
\end{proof}

\begin{remark}
Heuristically, suaveness of $(\mathring{\DD}_{\infty,\Q_p^{\rm cyc}}^\diamond)^\dR$ (which implies suaveness of $(\mathring{\DD}_{\infty}^{\times \diamond})^\dR$) follows from the fact that it is Cartier dual to $\mathbb{A}_{\Q_p^{\rm cyc}}^{1,\dR}/\Q_p^{\rm sm}$ and the fact that the zero section of $\mathbb{A}_{\Q_p^{\rm cyc}}^{1,\dR}/\Q_p^{\rm sm}$ is prim. Establishing this Cartier duality statement properly would require some work, so we opted for a more direct approach.
\end{remark}

\begin{corollary}\label{CoroKeyCasesSuave}
The map $\mathbb{G}_{a,\infty}^{\dR}\to \GSpec (\Q_p)$ is suave with invertible dualizing sheaf sitting in cohomological degree $2$,  where $\mathbb{G}_{a,\infty}^{\diamond}=\varprojlim_{x\mapsto x^p} \mathbb{G}_{a,\Q_p}^{\diamond}$.
\end{corollary}
\begin{proof}
By  \cite[Lemma 4.5.7]{heyer20246functorformalismssmoothrepresentations} we can prove the statement after base change along $\GSpec (\C_p)\to \GSpec(\Q_p)$. Let $T^{1/p^{\infty}}$ denote the compatible system of $p$-th power roots of the coordinate of $\mathbb{G}_a$.
  In that case, we can write $\mathbb{G}_{a,\infty}^{\diamond} = \bigcup_{r>0} \mathring{\DD}^{\diamond}_{\infty}(r)$ as an union of perfectoid open discs of radius $r$ obtained by rescaling the variable $T^{1/p^{\infty}}$ in the open unit disc by a fixed sequence of $p$-th power roots of an element $a\in \C_p$ with norm $|a|=r$.
 Therefore, it suffices to show that $\mathring{\DD}^{\dR}_{\infty}$ is suave with invertible dualizing sheaf in degree $2$.
 By the same rescaling argument, we can multiply the coordinate of the open punctured disc $\mathring{\DD}^{\times,\diamond}_{\infty}$ and deduce that $\mathbb{G}_{m,\infty}^{\dR}$ is suave with the desired dualizing sheaf in degree $2$ thanks to \cref{sec:geom-prop-analyt-1-cohom-smoothness-of-perfectoid-unit-disc}.
 Since $\mathring{\mathbb{D}}_{\infty}^{\dR}\cong \varprojlim_{x\mapsto x^p}(1+\mathring{\mathbb{D}}(1))^\dR \subset \mathbb{G}^{\dR}_{m,\infty}$  is an open immersion, we deduce that $\mathring{\mathbb{D}}_{\infty}^{\dR}$ is suave with desired dualizing sheaf, proving what we wanted. 
\end{proof}

\newpage
\section{Hyodo--Kato stacks}
\label{sec:de-rham-fargues}

 In the previous sections, we have built a robust formalism of analytic de Rham stacks and discussed the six operations on their categories of quasi-coherent sheaves. In this section, we put this formalism in action, by introducing the de Rham stacks of relative Fargues--Fontaine curves, that we call \textit{Hyodo--Kato stacks}.
 We prove some useful geometric properties of them, and establish a basic understanding of their $6$-functor formalism. In \cref{sec:p-adic-monodromy}, we will use these results to give a new proof of the $p$-adic monodromy theorem.

\subsection{Definition and first properties}
\label{sec:defin-first-prop}

Let $\Cat{ArcStk}^{\qfd}_{\F_p}$ be the category of qfd arc-stacks over $\F_p$, cf. \cref{def-qfd-light-arc-stack} and \cref{ExamDifferentArcStkZpFp}.
 We make the following definition.

\begin{definition}
  \label{sec:defin-first-prop-1-definition-de-rham-ff-stacks}
  Let $X\in \Cat{ArcStk}_{\F_p}$,  we define the following qfd arc-stacks over $\Q_p$:
  \begin{enumerate}
 
  \item The \textit{open punctured curve over $X$}
  $$\Yc_{X}^\diamond:=X\times_{\Marc(\F_p)} \Marc(\Q_p).$$
  \item The \textit{Fargues--Fontaine curve over $X$}
  $$\FF_{X}^\diamond:= \Yc_{X}^\diamond/\varphi_{X}^\Z$$
  where $\varphi_X$ is the Frobenius of $X$.
  \end{enumerate}
\end{definition}
  
\begin{definition}

Let $X\in \Cat{ArcStk}_{\F_p}$, we define the following qfd Gelfand stacks:

\begin{enumerate}

\item  The de Rham stack of open punctured curve over $X$
$$\Yc_{X}^{\dR}:= (\Yc_{X}^\diamond)^\dR.$$

\item  The \textit{Hyodo--Kato stack of $X$}
$$X^\HK:= (\FF_X^\diamond)^\dR.$$

\end{enumerate}

For $X\in \Cat{GelfStk}^{\qfd}$ we write $\Yc_{X}^{\dR}:= \Yc^{\dR}_{X^{\diamond}}$ and $X^\HK := (X^{\diamond})^\HK$.

\end{definition}

If $f: Z\to X$, we will often denote by $f^{\mathcal{Y}^\dR}:\Yc_{Z}^{\dR}\to \Yc_{X}^{\dR}$ and $f^\HK:Z^{\HK}\to X^{\HK}$ the respective induced morphisms.

\begin{remark}
In a sequel to this paper, we will explain in detail the relation between the coherent cohomology of $X^\HK$, when $X$ is a smooth rigid space over a $p$-adic field, and Hyodo--Kato cohomology in the sense of Colmez--Nizio\l{} \cite{colmez_niziol_basic}. This will justify our choice of terminology.
\end{remark}

\begin{remark}\label{RemInterpretationConstruction}
Let us give a more explicit moduli description of the de Rham stack of open punctured curves.
 Let $X\in \Cat{ArcStk}_{\mathbb{F}_p}^{\qfd}$, and let $A$ be a qfd nilperfectoid ring.
 Then
\[
\Yc_{X}^{\dR}(A)=  \Yc_X^\diamond(\Marc(A^u)) = X(A^{u,\flat})
\]
(here in the middle, we consider morphisms over $\Marc(\Q_p)$).
Indeed, the functor $\Yc_{(-)}^{\dR}$ is the composition
\[
\Cat{ArcStk}^{\qfd}_{\F_p} \xrightarrow{-\times_{\Marc(\F_p)} \Marc(\Q_p)} \Cat{ArcStk}^{\qfd}_{\Q_p} \xrightarrow{(-)^{\dR}} \Cat{GelfStk}^{\qfd}
\]
which is right adjoint to the composition 
\[
\Cat{ArcStk}^{\qfd}_{\Q_p}\xrightarrow{(-)^{\diamond}} \Cat{ArcStk}^{\qfd}_{\Q_p} \xrightarrow{G} \Cat{ArcStk}^{\qfd}_{\F_p}
\]
where the functor $G$ is the forgetful functor along $\Marc(\Q_p)\to \Marc(\F_p)$, or equivalently, the unique colimit preserving functor sending an affinoid perfectoid space $\Marc(B)$ over  $\Q_p$ to its tilt  $\Marc(B^{\flat})$ in characteristic $p$.
\end{remark}

Recall from the end of \Cref{sec:arc-and-!-topologies-on-perfectoid-rings} the Gelfand stack $\Yc_A$, for $A$ a separable affinoid perfectoid $\mathbb{F}_p$-algebra, which led by descent to the definition of the Gelfand stack $\Yc_X^{\rm arc}$ for any $X\in \Cat{ArcStk}_{\mathbb{F}_p}$ (\cref{CorDescentYConstruction}). We introduce the qfd variant of this construction.

\begin{definition}\label{Def}
We define
\begin{equation*}\label{YY}
 \Yc\colon \Cat{ArcStk}^{\qfd}_{\F_p}\to \Cat{GelfStk}^{\qfd}
\end{equation*}
 as the left Kan extension of the functor sending a separable qfd affinoid perfectoid $\Marc(A)$ to the Berkovich space $\mathcal{Y}_{A}$.
\end{definition}

Equivalently, after the discussion of \cref{RemDescentYFunctor}, the functor in \cref{Def} is the unique left exact and colimit preserving functor on qfd arc-stacks sending $\Marc(A)$ to $\Yc_{A}$ for $A$ a separable qfd perfectoid $\F_p$-algebra.

\begin{remark}\label{RemarkSeveralAdjoints}
 We have the following diagram (where any upper functor is the left adjoint to the lower functor)
\[
\begin{tikzcd}
\Cat{ArcStk}^{\qfd}_{\F_p} \ar[r," \mathcal{Y}_{(-)}", shift left = 4ex] \ar[r,"\mathcal{Y}_{(-)}^{\dR}", shift right =4ex] & \Cat{GelfStk}^{\qfd} \ar[l,"(-)^{\diamond}"'].
\end{tikzcd}
\]
 We note that the functors $\mathcal{Y}_{(-)}$ and $\mathcal{Y}_{(-)}^{\dR}$ are both pullback functors of morphisms of $\infty$-topoi (i.e. they are left exact and colimit preserving). Moreover, the perfectoidization functor $(-)^{\diamond}$ becomes a pullback functor of morphism of topoi when restricted to a morphism $(-)^{\diamond}\colon \ob{GelfStk}^{\qfd}\to \left(\Cat{ArcStk}^{\qfd}_{\F_p}\right)_{/\Marc(\Q_p)}=\Cat{ArcStk}^{\qfd}$.
\end{remark}

\begin{remark}
Note that for $X\in \Cat{ArcStk}_{\mathbb{F}_p}^{\rm qfd}$, the perfectoidization $\Yc_X^\diamond$ of the qfd Gelfand stack $\Yc_X$ coincides with the qfd arc-stack introduced with the same name in \Cref{sec:defin-first-prop-1-definition-de-rham-ff-stacks}: there is no conflict of notation.

Note also that the functor $\Yc$ differs from the functor $\widehat{\Yc^\diamond}$: heuristically, in terms of quasi-coherent sheaves, for a qfd arc-stack $X$, we are not considering all $\widehat{\mathcal{O}}$-modules on the arc-stack $\Yc_X^\diamond= X \times \Marc(\Q_p)$, but are only using arc-descent in the $X$-direction, not the $\Marc(\Q_p)$-direction.
\end{remark}

 Following the previous construction, we can define a relative version of the Hyodo--Kato stacks.

\begin{definition}\label{DefFFStackRel}
Let $Z\to X$ be a map of qfd arc-stacks over $\F_p$, we define the stack $\Yc^{\dR/ \Yc_{X}}_Z$ to be the pullback
\[
\begin{tikzcd}
\Yc^{\dR/ \Yc_{X}}_Z \ar[r] \ar[d] & \Yc_{X} \ar[d] \\ 
\Yc_Z^{\dR} \ar[r] & \Yc_{X}^{\dR}.
\end{tikzcd}
\]
To lighten notation we will also write $\Yc_{Z}^{\dR/X}$ instead of $\Yc_{Z}^{\dR/\Yc_{X}}$.
 Similarly, we define the \textit{relative Hyodo--Kato stack} of $Z$ over $X$  $$Z^{\HK/\FF_{X}}$$ to be the quotient by the action of the Frobenius of $Z$ on $\Yc^{\dR/X}_Z$ (so that $Z^{\HK/\FF_X}$ lives over $\FF_{X}$).
\end{definition}

\begin{remark}
 In \cite{le2023rham}, a motivic definition of  ``de Rham cohomology over the Fargues--Fontaine curve'' was given and named \textit{de Rham--Fargues--Fontaine cohomology}. We expect the coherent cohomology of the relative Hyodo--Kato stacks of \cref{DefFFStackRel} to capture the latter cohomology theory, as well as the one introduced in \cite{lebras2018} and further studied in \cite{bosco2023}.
\end{remark}

\begin{remark}\label{RemAllStacksTogether}
Let $X$ be a qfd Gelfand stack over $\Q_p$, by adjunction we have maps of Gelfand stacks $\widehat{X}\to \Yc_X$ corresponding to the fixed untilt of the Fargues--Fontaine curve provided from the fact that $X$ is defined over $\Q_p$.
 We have a commutative diagram of qfd Gelfand stacks
\[
\begin{tikzcd}
\widehat{X} \ar[r] \ar[d] & \Yc_X \ar[dd] \\
X 	\ar[d]	&  \\ 
X^{\dR} \ar[r] & \Yc_{X}^{\dR}. 
\end{tikzcd}
\]
\end{remark}

\begin{remark}\label{excisionYdR}
 A useful tool that we will repeatedly use is excision, whose validity can be justified geometrically as in \cref{sec:cohom-prop-de-excision}: for any qfd arc-stack $X$, the stack $\mathcal{Y}_X^{\dR}$ lives over the (qfd-version of the) Betti stack of $X_{\mathrm{cond}}$.
\end{remark}

Let $X$ be a qfd arc-stack over $\F_p$, it is natural to ask when the natural map of Gelfand stacks $\Yc_X\to \Yc^{\dR}_X$ is an epimorphism  for the $!$-topology, or even a descendable cover.
 It turns out that it is not always the case, e.g. if $X= \overline{\T}_{\Q_p}$ is a closed torus over $\Q_p$, that is, the adic compactification of the adic torus.
 Then $\Yc_{X}\to \Yc_{X}^{\dR}$ is not an epimorphism. Indeed, if it were the case, by base change along the point at infinity $\GSpec (\Q_p) \to \Yc_{\Marc(\Q_p)}$, from \cref{RemAllStacksTogether}  we would deduce that the map of Gelfand stacks $\overline{\T}_{\Q_p}\to \overline{\T}_{\Q_p}^{\dR}$ from the closed torus over $\Q_p$ to its de Rham stack is an epimorphism.
 But this cannot be the case,  since the latter is the same as the de Rham stack of an overconvergent torus, and we could conclude that the overconvergent de Rham cohomology of $\overline{\T}_{\Q_p}$ agrees with its de Rham cohomology which is certainly false (the former is finite dimensional and the last is non-separated).
 
 For partially proper rigid spaces, this problem disappears, and then the map to the de Rham stack is indeed prim and descendable for a large variety of examples in $p$-adic geometry.

 \begin{proposition}\label{PropepiFF}
Let $X$ be a partially proper rigid space over a non-archimedean field $K/\Q_p$ of finite $\ob{dim.trg}$.
 Let $W\to X$ be a countable inverse limit of finite \'etale covers.  Then, the map of qfd Gelfand stacks
\begin{equation}\label{eqjwkqwejqeq}
\Yc_{W^{\diamond}}\to \Yc_{W^{\diamond}}^{\dR/\Marc(K)}
\end{equation}
is an epimorphism of Gelfand stacks,  prim and (locally in the analytic topology) descendable.
 In particular, it satisfies universal $\ob{D}^*$- and $\ob{D}^!$-descent. If in addition $K$ is pro-\'etale over $\Q_p$, then the same holds for the map of qfd Gelfand stacks
 \[
 \Yc_{W^{\diamond}}\to \Yc_{W^{\diamond}}^{\dR}.
 \]
\end{proposition}
\begin{proof}
First, note that since $K$ has finite $\ob{dim.trg}$, it is pro-\'etale over a finitely generated non-archimedean field (whose underlying $\Q_p$-Banach space is separable), and so it is  quasi-finite dimensional in the sense of \cref{def-qfd-light-arc-stack}.
 Furthermore, by base change along $K\to K^{\cyc}$ and \cref{CorDescendableMapsqfla}, we can assume without loss of generality that $K$ is perfectoid.

 The statement is local in the analytic topology of $X$.
 Hence, we can assume without loss of generality that $X$ is a Zariski closed subspace of $\mathring{\DD}^n_K$ and that $W=\varprojlim_{n}  X_n$ is a countable limit along finite \'etale covers of $X$.
 We first verify primness. By \cite[Lemma 4.5.7]{heyer20246functorformalismssmoothrepresentations}, we can pullback along the perfectoid open disc $\mathring{\DD}^{\diam}_{K, \infty}\to \mathring{\DD}^{\diam}_K$, and show that the claim holds for the induced maps of arc-stacks $\widetilde{W}=\varprojlim_n \widetilde{X}_n\to \widetilde{X}$.
 In that case, both $\widetilde{W}$ and $\widetilde{X}$ are perfectoid spaces.
  Moreover, $\Yc_{\widetilde{W}^{\diam}}$ is represented by a Berkovich space and the self fiber product of the map $ \Yc_{\widetilde{W}^{\diam}}\to \Yc_{\widetilde{W}^{\diamond}}^{\dR/K^{\flat}}$ is the overconvergent neighbourhood of the diagonal in $\Yc_{\widetilde{W}^{\diam}}\times_{\Yc_K} \Yc_{\widetilde{W}^{\diam}}$ which is represented in Berkovich spaces, see \cref{PropdeRhamBerkovich}.
 This shows that \cref{eqjwkqwejqeq} is  prim.

  It is left to prove that it is descendable. We can consider a stratification of $X$ by locally Zariski closed subspaces where the reduction of each component is smooth over $K$. More precisely, we can write $X=\bigcup_n \overline{X}_n =\bigcup_n X_n$ as a countable union of open Stein subspaces $X_n$ along strict inclusions, such that their closure $\overline{X}_n$ is an affinoid subspace of $X$. The affinoid subspaces $\overline{X}_n$ admit such a stratification being qcqs, and up to changing $X$ by $X_n$, we can assume that $X$ has such a stratification as well.  By an excision argument, cf. \cite[Lemma 5.4.4]{camargo2024analytic} (and see \cref{excisionYdR} for a justification of the excision property), we can assume without loss of generality that $X$ itself is smooth.

Now, as $X$ is smooth, the immersion $X\to \mathring{\DD}_{K}$ is a local regular immersion, and locally we can write $X$ as a pullback 
\[
\begin{tikzcd}
X \ar[d] \ar[r] & \mathring{\DD}_{K}^{m+d} \ar[d] \\ 
\GSpec(K) \ar[r,"0"] &  \mathring{\DD}_K^m.
\end{tikzcd}
\]
In particular, we have an \'etale map of partially proper rigid spaces $X\to \mathring{\DD}_{K}^{d}$ that, after taking a further open analytic cover, will factor as a composite of finite \'etale maps and open immersions.
 Let us now write $W_{\infty}^{\diam}$ and  $X_{n,\infty}^{\diam}$ for the pullback of $W^{\diam}$ and $X_n^{\diam}$  along the perfectoid open disc $1+\mathring{\DD}_{K,\infty}^{\diam}\to 1+\mathring{\DD}_{K}^{\diam}$ obtained by extracting $p$-th power roots of $1-T$, where $T$ is the variable of $\mathring{\DD}_{K}^{\diam}$. It suffices to prove the result for $W_\infty$ in place of $W$, by \Cref{CorDescendableMapsqfla}.
   We have a diagram with pullback squares
\[
\begin{tikzcd}
\Yc_{W^{\diamond}_{\infty}} \ar[r] \ar[rd] & \varprojlim_{n} \Yc_{X^{\diam}_{n,\infty}} \ar[r] \ar[d] & \Yc_{X^{\diam}_{n,\infty}} \ar[r] \ar[d] &  {\Yc_{\mathring{\DD}_{K,\infty}^{d,\diam}}} \ar[d] \\
& {\Yc_{W_{\infty}^{\diam}}^{\dR/K^{\flat}}} \ar[r] & {\Yc_{X^{\diam}_{n,\infty}}^{\dR/K^{\flat}}} \ar[r] &  {\Yc_{\mathring{\DD}_{K,\infty}^{d,\diam}}^{\dR/K^{\flat}}}
\end{tikzcd}
\] 
where the limit is as Berkovich spaces (i.e., a colimit at the level of structural algebras).
 The map $\Yc_{W^{\diamond}_{\infty}}\to \Yc_{X_{n,\infty}^{\diamond}}$  is descendable of index $\leq 16$ by \cref{LemAnDescentPerfectoids}.
 Hence, the map $\Yc_{W^{\diamond}_{\infty}}\to \varprojlim_{n} \Yc_{X^{\diam}_{n,\infty}}$ is descendable by \cite[Proposition 2.7.2]{mann2022p}.
 Therefore, by applying \textit{loc.cit.\ }again and taking pullbacks from the disc to $X$, it suffices to show that the map $\Yc_{\mathring{\DD}_{K,\infty}^{d,\diam}} \to \Yc_{\mathring{\DD}_{K,\infty}^{d,\diam}}^{\dR/K^{\flat}}$ is descendable.
 But  $\Yc_{\mathring{\DD}_{K,\infty}^{d,\diam}} = \Yc_{K^{\flat}}\times_{\GSpec(\Q_p)} \mathring{\DD}^{d,u}_{\Q_p,\infty}$, and one is reduced to show that 
 \[
 \mathring{\DD}^u_{\Q_p,\infty}\to \mathring{\DD}^{\dR}_{\Q_p,\infty}.
 \]
 is descendable (where the upper script $u$ means the uniform completion).
 For that, we write $ \mathring{\DD}_{\Q_p,\infty}= \varprojlim_n  \mathring{\DD}_{\Q_p,n}$ as limit of open discs centered at $0$ and variable $(T-1)^{1/p^n}$.
 By \cref{PropApproxFFStk}, the map $\mathring{\DD}_{\Q_p,\infty}\to \mathring{\DD}^{\dR}_{\Q_p,\infty}$ is descendable, and it suffices to see that the map $h\colon  \mathring{\DD}_{\Q_p,\infty}^u\to \varprojlim_n \mathring{\DD}_{\Q_p,n}$ is also descendable.
 But each map $\mathring{\DD}_{\Q_p,\infty}^u\to  \mathring{\DD}_{\Q_p,n}$ is descendable of index $\leq 1$ as the underlying sheaves of rings have a section as modules, and it follows from \cite[Proposition 2.7.2]{mann2022p} that the map $h$ is descendable of index $\leq 2$, finishing the proof of the first part of the proposition.

 It is left to prove that, if $K$ is pro-\'etale over $\Q_p$, then the map $\Yc_{W^{\diamond}}\to \Yc^{\dR}_{W^{\diamond}}$ is prim and (locally in the analytic topology) descendable. By the previous step and base change, it suffices to prove this for the map $\Yc_{K}\to \Yc^{\dR}_K$. By arc-descent, it suffices to prove this for $K=\C_p$ as both $\Yc_{\C_p}\to \Yc_{K}$ and $\Yc_{\C_p}^{\dR}\to \Yc_{K}^{\dR}$ are descendable (which is easy to deduce from the presentation $\Yc_{K}=\Yc_{\C_p}/\underline{\Gal_K}$ and $\Yc_{K}^{\dR}=\Yc^{\dR}_{\C_p}/\Gal^{\sm}_{K}$).  For any finite extension $L/\Q_p^{\Kum}$, with $\Q_p^{\Kum}=\Q_p(p^{1/p^{\infty}})^u$ the Kummer extension, the map 
 \[
 \Yc_{\C_p}\to \Yc_{L}
 \]
 is descendable with descendable index independent of $L$ (\Cref{LemAnDescentPerfectoids}), and hence $\Yc_{\C_p}\to \varprojlim_{L/\Q_p^{\Kum}} \Yc_{L}$ (where $L$ runs over finite extensions of $\Q_p^{\Kum}$ is descendable. On the other hand, we have a  cartesian square 
 \[
 \begin{tikzcd}
 \varprojlim_{L/\Q_p^{\Kum}} \Yc_L \ar[r]  \ar[d] & \Yc_{\Q_p^{\Kum}}   \ar[d] \\ 
 \GSpec(\Q_p)  \ar[r] & B (\Gal_{\Q_p^{\Kum}}^{\sm})
 \end{tikzcd}
 \]
 and similarly for de Rham stacks. Since $\varprojlim_{L/\Q_p^{\Kum,u}}\Yc_{L}^{\dR}=\Yc_{\C_p}$,    it suffices to show that $\Yc_{\Q_p^{\Kum,u}}\to \Yc^{\dR}_{\Q_p^{\Kum,u}}$ is descendable. 
 This follows from \Cref{explicit-description-drff-stack-of-qp}(3) and the fact that the pre-perfectoid disc $\mathring{\mathbb{D}}_{\Q_p,\infty}^u$ is descendable over the decompleted pre-perfectoid disc $\mathring{\mathbb{D}}_{\Q_p,\infty}=\varprojlim_{n} \mathring{\mathbb{D}}_{\Q_p,n}$, since the latter is a countable colimit of discs $\mathring{\mathbb{D}}_{\Q_p,n}$ with variable $T^{1/p^n}$, and that $f\colon \mathring{\mathbb{D}}_{\Q_p,\infty}^u\to \mathring{\mathbb{D}}_{\Q_p,n}$ has index of descendability $\leq 1$ since $f_* 1$ has a section as a module.
\end{proof}

\subsection{Examples of Hyodo--Kato stacks and their cohomology; Chern classes}\label{sss:ExamplesdRFFStacks}

In this subsection, we provide examples of Hyodo--Kato stacks and their quasi-coherent cohomology, referred to as \textit{Hyodo–Kato cohomology}. We hope this will give the reader an idea of how to work with such objects. As an application, we show that one has a theory of first Chern classes for Hyodo--Kato cohomology.
 These examples will also be relevant in the proof of the $p$-adic monodromy theorem.
 To start with, we discuss the example of $\Marc(\Q_p)$.

 \begin{lemma}\label{explicit-description-drff-stack-of-qp}\

\begin{enumerate}

\item Let $K=\F_p((\pi^{1/p^{\infty}}))$.
 Then there is a natural Frobenius equivariant equivalence of Gelfand stacks 
\[
\Yc_{K}\cong\mathring{\DD}^{\times,u}_{\Q_p,\infty}\cong (\varprojlim_{X\mapsto X^p} \mathring{\DD}^{\times}_{\Q_p})^u
\]
where the right hand side is the Berkovich space given by the pre-perfectoid open and punctured disc over $\Q_p$ on the variable $X=[\pi]$, and $\varphi(X)=X^p$ (here the superscript $u$ refers to the uniform completion of the Berkovich space).
   In particular, we have a Frobenius equivariant isomorphism of de Rham stacks
\[
\Yc_K^{\dR}\cong \varprojlim_{X\mapsto X^p} \mathring{\DD}^{\times,\dR}_{\Q_p}.
\]

\item  Let $\Q_p^{\cyc}$ denote the cyclotomic extension of $\Q_p$. Let $\epsilon=(\zeta_{p^n})_n$ be a compatible sequence of $p$-th power roots of $1$.
  Then, the isomorphism $\Q_p^{\cyc, \flat} \cong \F_p((\epsilon^{1/p^{\infty}}-1))$ of perfectoid fields induces a  Frobenius and $\Z_p^{\times,\sm}$-equivariant isomorphism of de Rham stacks
\[
\Yc_{\mathbb{Q}_p^{\cyc,\flat}}^{\dR}\cong (\varprojlim_{q\mapsto q^p} (1+ \mathring{\mathbb{D}}_{\mathbb{Q}_p}^{\dR})) \backslash \{1\},
\]
where $1+ \mathring{\mathbb{D}}_{\mathbb{Q}_p}$ is the rigid generic fiber of the formal spectrum of the ring underlying the $q$-crystalline prism, i.e., $\mathbb{Z}_p\llbracket q-1\rrbracket$ with $\varphi(q)=q^p$, and the transition maps are giving by rising to the $p$-th power.
Under this identification we have $q=[\epsilon]$.
  The $\mathbb{Z}_p^{\times }$-equivariant action on the LHS term is via the identification   $a\in \mathbb{Z}_p^{\times}\cong \ob{Gal}(\mathbb{Q}_p^{\cyc}/\mathbb{Q}_p)$, the action   on the RHS is given by the formula $\sigma\cdot q= q^{\chi^{\cyc}(\sigma)}$ for $\sigma\in \Z_p^{\times}$.
 In particular,
\[
\Yc^{\dR}_{\Marc(\Q_p)}\cong  \big( (\varprojlim_{p} (1+\mathring{\mathbb{D}}_{\mathbb{Q}_p} )^{\dR}) \backslash \{1\} \big) / \mathbb{Z}_p^{\times, \sm}
\]
and 
\[
 \Marc(\Q_p)^\HK \cong  \big( (\varprojlim_{p} (1+\mathring{\mathbb{D}}_{\mathbb{Q}_p} )^{\dR}) \backslash \{1\} \big) / \mathbb{Q}_p^{\times, \sm}.
\]
\item Let $E/\mathbb{Q}_p$ be a finite extension of $\mathbb{Q}_p$.
  For $\varpi$ a uniformizer of $E$ let $E^{\Kum}=E(\varpi^{1/p^{\infty}})^u$ be the uniform completion of the Kummer extension.
  Let $E_0\subset E$ be the maximal unramified extension of $E$ with residue field $\F_q$.
  Let $\varpi^{\flat}=(\varpi^{1/p^n})_n\in E^{\Kum,\flat}$.
  Then, $E^{\Kum,\flat}\cong \F_q((\varpi^{\flat,1/p^{\infty}}))$ and we have a natural Frobenius equivariant isomorphism
\[
\Yc_{E^{\Kum, \flat}}^{\dR} \cong \varprojlim_{X\mapsto X^p} \mathring{\mathbb{D}}_{E_{0}}^{\times,\dR},
\]
where the right hand side is the de Rham stack of the perfection of the punctured rigid generic fiber of the formal spectrum of the ring underlying the Breuil-Kisin prism, i.e., $\shf{O}_{E_0}\llbracket X\rrbracket$ with $\varphi(X)= X^{p}$.
Under this equivalence $X=[\varpi^{\flat}]$.

\end{enumerate}
\end{lemma}
\begin{proof}

  For part (1), observe that $\A_{\inf}(K)\cong \Z_p\llbracket X^{1/p^{\infty}}\rrbracket$ with $X=[\pi]$ the Teichm\"uller lift of $\pi$, where on the latter ring the Frobenius acts as $\varphi(X)=X^p$.
  Then, under this isomorphism, $\Yc_K$ identifies Frobenius equivariantly with a pre-perfectoid punctured open unit disc in the variable $X$ as in the statement.
  The statement about de Rham stacks follows from the fact the the formation of de Rham stacks commutes with limits.

  For part (2), we observe that the tilt of $\mathbb{Z}_p^{\cyc}$ identifies with the ring $\mathbb{F}_p\llbracket \epsilon^{1/p^{\infty}}-1\rrbracket$ where $\epsilon= (\zeta_{p^{n}})_n$ is a compactible sequence of $p$-th power roots of unit.
  We then have an isomorphism as arc-stacks over $\Marc (\mathbb{Q}_p)$
\[
\Yc^{\diamond}_{\mathbb{Z}_p^{\cyc, \flat}}= \Marc (\mathbb{Z}_p^{\cyc})\times_{\Marc(\overline{\F}_p)} \Marc( \mathbb{Q}_p)  \cong (\varprojlim_{x\mapsto x^p} 1+\mathring{\mathbb{D}}_{\mathbb{Q}_p})^{\diamond}.
\]
Under this identification, $q=[\epsilon]$ and the action of Frobenius is determined by $\varphi(q)=q^p$.
The action of $\Z_p^{\times}=\Gal(\mathbb{Q}_p^{\cyc}/\mathbb{Q}_p)$  is determined by
\[
\sigma(q)=q^{\chi^{\cyc}(\sigma)}.
\]
Then, by pulling back to $\Marc(\mathbb{Q}_p^{\cyc,\flat})$, we obtain
\[
\Yc^{\diam}_{\mathbb{Q}_p^{\cyc,\flat}} \cong ( \varprojlim_{x\mapsto x^p}  1+\mathring{\mathbb{D}}_{\mathbb{Q}_p})^{\diamond} \backslash \{1\}.
\]
By passing to de Rham stacks we obtain the first assertion.
Then, for the last assertion of part (2), it suffices to observe that $\Yc_{\Marc(\Q_p)}^{\diam} = \Yc_{\mathbb{Q}_p^{\cyc,\flat}}^{\diamond} /\underline{ \mathbb{Z}_p^{\times}}$ where $\mathbb{Z}_p^{\times}$ is identified with the Galois group of $\mathbb{Q}_p^{\cyc}$ via the cyclotomic character, and that $\Marc(\Q_p)^{\HK}=\Yc_{\Q_p^{\cyc,\flat}}/\Q_p^{\times,\mathrm{sm}}$ where in addition $p^{\Z}$ is identified with the action of Frobenius. We note that quotients pass to the de Rham stacks as its formation commutes with colimits (\cref{TheoMaindeRham2}), and the de Rham stack of a truncated condensed anima is its Betti stack (\cref{ExamCondensedAnimadeRham}).

For part (3), we note that the tilt of $\shf{O}_{E^{\Kum}}$ identifies with $\mathbb{F}_q\llbracket\varpi^{\flat,1/p^{\infty}}\rrbracket$.
By taking Witt vectors, one finds  that
\[
\A_{\inf}(E^{\Kum})\cong\shf{O}_{E_0}\llbracket X^{1/p^{\infty}}\rrbracket
\]
is the perfection of the Breuil--Kisin prism.  Passing to the locus where $|Xp|\neq 0$ we see that $\Yc_{E^{\Kum,\flat}}$ is a perfectoid punctured open unit disc over $E_{0}$. Thus, as arc-stacks we have that
\[
\Yc_{E^{\Kum,\flat}}^{\diamond}\cong \varprojlim_{x\mapsto x^p} \mathring{\mathbb{D}}^{\times, \diamond}_{E_0}.
\]
We deduce the desired statements by passing to de Rham stacks.
\end{proof}

A property we want to prove is the disc-invariance of Hyodo--Kato cohomology. For this, we first need invariance of usual (overconvergent) de Rham cohomology for the pre-perfectoid disc, cf.\ \cref{sec:-able-maps-disc-invariance-de-rham-cohomology}.

\begin{lemma}
  \label{sec:key-computations-de-1-disc-invariance-dr}
    Let $f\colon Y\to \Marc(\Q_p)$ be given by $Y=\Marc(\Q_p\langle T^{1/p^\infty}\rangle), Y=\varprojlim_{x\mapsto x^p}\mathbb{D}^{\circ}_{\Q_p}$ or $Y=\varprojlim_{x\mapsto x^p}\A^{1}_{\Q_p}$.
  Then, $f^{\dR,\ast}\colon \ob{D}(\GSpec(\Q_p))\to \ob{D}(Y^\dR)$ is fully faithful.
  Moreover, the same assertion holds true after any base change of $f$ to a qfd arc-stack. 
\end{lemma}
\begin{proof}
  As in \cref{sec:-able-maps-disc-invariance-de-rham-cohomology}, we can reduce to the case that $Y=\Marc(\Q_p\langle T^{1/p^\infty}\rangle)$ by writing the pre-perfectoid open disc, or the pre-perfectoid affine line as a union of $Y$'s (also after base change).
  Set $Y_n=\GSpec(\Q_p\langle T^{1/p^n}\rangle_{\leq 1})$ for $n\geq 1$, $\widetilde{Y}=\varprojlim_{n}Y_n=\GSpec(\varinjlim_{n} \Q_p\langle T^{1/p^n}\rangle_{\leq 1})$ (so that $\widetilde{Y}^\diamond=Y$). Because for $n\geq 1$ the map $Y_n\to Y_n^\dR\cong Y_n/\mathbb{G}_a^\dagger$ is descendable of index $\leq 2$, we can use the argument of \cref{PropApproxFFStk} to deduce that $\widetilde{Y}\to \widetilde{Y}^\dR$ is descendable. From here, we can deduce that that $f^\dR\colon \widetilde{Y}^\dR\cong Y^\dR\to \GSpec(\Q_p)$ is cohomologically proper. In particular, the formation of $f^\dR_{\ast}$ commutes with base change. As in \cref{sec:-able-maps-disc-invariance-de-rham-cohomology}, it suffices therefore to check that the natural map $1\to f^\dR_{\ast}(1)$ is an isomorphism. Using \cref{LemComputationdeRhamLimit}, this follows as in \cref{sec:-able-maps-disc-invariance-de-rham-cohomology} from the classical computation of overconvergent de Rham cohomology of a closed disc.
\end{proof}

In the following, all schemes in characteristic $p$ are implicitly regarded as qfd arc-stacks over $\F_p$.

\begin{lemma}
  \label{lemDiscInvariancedRFF}\
  \begin{enumerate}
   \item Let $f\colon X=\Marc(\F_p[T])\to \Marc(\F_p)$. Then, the functor $f^{\HK,\ast}\colon \ob{D}(\Marc(\F_p)^{\HK}) \to \ob{D}(X^{\HK})$ is fully faithful. The same holds true after any base change of $f$.
 \item Let $Z=\Marc(\F_p((\pi)))$, and $g\colon W\to Z$ for $W=\Marc(\F_p((\pi))\langle t\rangle)$, $W=\mathbb{D}^{\circ}_{\F_p((\pi))}$ or $W=\mathbb{A}^{1}_{\F_p((\pi))}$. Then, $g^{\HK,\ast}\colon \ob{D}(Z^\HK)\to \ob{D}(W^{\HK})$ is fully faithful. The same holds true after any base change of $g$.
  \end{enumerate}
 Similar assertions hold for $\mathcal{Y}^{\dR}_{(-)}$.
\end{lemma}
\begin{proof}
The case for $(-)^\HK$ follows from the case for $\mathcal{Y}^\dR_{(-)}$ by passing to Frobenius-equivariant objects. By \cref{RemDescentYFunctor} we know $\mathcal{Y}_{\Marc(\F_p)}^\dR\cong \GSpec(\Q_p)$. For part (1), we can calculate that $\mathcal{Y}_X^\dR\cong (\varprojlim_{x\mapsto x^p}\mathbb{A}^{1}_{\Q_p})^\dR$. Hence, the statement reduces to \cref{sec:key-computations-de-1-disc-invariance-dr}. One can argue similarly to prove part (2).
\end{proof}

Before continuing with other computations of Hyodo--Kato cohomology, we make the following convenient definition.

\begin{definition}
 For $n\in \Z$, we denote by $\mathcal{O}(n)$ the vector bundle on $\Marc(\F_p)^{\HK}$ corresponding to the isocrystal $(\Q_p,p^{-n}\varphi)$ under the isomorphism $\Marc(\F_p)^\HK \cong \GSpec(\Q_p)/\varphi^\Z$ (\cref{RemDescentYFunctor}).
\end{definition}

\begin{lemma}
  \label{sec:key-computations-de-1-computation-for-gm}
  Let $f\colon \mathbb{G}_{m, \F_p}^{\diamond}=\Marc(\F_p[T,T^{-1}])\to \Marc(\F_p)$. Then, $f^{\HK}_{\ast}\mathcal{O}\cong \mathcal{O}[0]\oplus \mathcal{O}(-1)[-1]$. A similar assertion holds for $f_*^{\mathcal{Y}^{\dR}}$.
\end{lemma}
\begin{proof}
  We can identify $X\times_{\Marc(\F_p)}\Marc(\Q_p)\cong \mathbb{G}_{m,\infty,\Q_p}^{\diamond}:=\varprojlim_{x\mapsto x^p} \mathbb{G}_{m,\Q_p}^{\diamond}$. More precisely, the coordinate on $\mathbb{G}_{m,\Q_p}^{\diamond}$ is given by $[T^\flat]:=(T,T^{1/p},T^{1/p^{2}},\ldots)$. We note that the de Rham cohomology of $\mathbb{G}_{m,\Q_p}$ identifies with the de Rham cohomology of any overconvergent annulus in it. Writing $\mathbb{G}_{m,\Q_p}$ as an increasing union of overconvergent annuli (and hence the cohomology as an essentially constant inverse limit), we can use \cref{LemComputationdeRhamLimit} to see that $f^\HK_{\ast}(1)$ can be calculated as the Frobenius-equivariant object given by the inverse limit of $f_{n,\ast}^\dR(1)$ for $f_{n}\colon \mathbb{G}_{m,\Q_p}\to \GSpec(\Q_p)$. As the de Rham cohomology of $\mathbb{G}_{m,\Q_p}$ is given by $\Q_p[0]\oplus \Q_p[-1]$ with the class in degree $1$ generated by $d\log([T^\flat]):=\frac{1}{[T^\flat]}d[T^\flat]$, on which Frobenius acts by multiplication by $p$, we get the result.
\end{proof}

\begin{remark}\label{sec:key-computations-de-1-invariance-over-q-p}\
\begin{enumerate}
 \item We note that \Cref{lemDiscInvariancedRFF} yields invariance of Hyodo--Kato cohomology with respect to the pre-perfectoid disc or pre-perfectoid affine line.
  However, invariance with respect to the morphisms $X\to \Marc(\Q_p)$, with $X=\GSpec(\Q_p\langle T\rangle)^\diamond$, $X=\mathbb{D}^{\circ, \diamond}_{\Q_p}$ or $X=\mathbb{A}^{1,\diamond}_{\Q_p}$, is true as well.
  Indeed, as in \cref{lemDiscInvariancedRFF} or \cref{sec:-able-maps-disc-invariance-de-rham-cohomology} one reduces to the case that $X=\GSpec(\Q_p\langle T\rangle)^\diamond$ and the statement that the higher pushforwards for $\mathcal{Y}_X^\dR\to \mathcal{Y}_{\Marc(\Q_p)}^\dR$ vanish (and are the unit in degree $0$).
  Using excision, this reduces to the calculation of the pushforward for $X$ replaced by $X\setminus \{0\}$, and then to the calculation of the pushforward for $\mathcal{Y}^\dR_{\Marc(\Q_p\langle T,T^{-1}\rangle)}\to \mathcal{Y}_{\Marc(\Q_p)}^\dR$.
  Now, $\mathcal{Y}^\dR_{\Marc(\Q_p^{\cyc}\langle T,T^{-1}\rangle)}\cong \mathcal{Y}^\dR_{\Marc(\Q_p^\cyc\langle T^{\pm 1/p^\infty}\rangle)}/\Z_p(1)^{\mathrm{sm}}$,
  $$
 \Yc^{\diamond}_{\Marc(\Q_p^\cyc\langle T^{\pm 1/p^\infty}\rangle)} = \Yc_{\Q_p^{\cyc, \flat}}^{\diamond}\times_{\Marc(\Q_p)} \Marc(\Q_p\langle T^{\pm 1/p^\infty}\rangle)
 $$
  and thus, by \cref{sec:key-computations-de-1-computation-for-gm}, the pushforward of $1$ along $\mathcal{Y}^\dR_{\Marc(\Q_p^\cyc\langle T^{\pm 1/p^\infty}\rangle)}\to \mathcal{Y}_{\Marc(\Q_p^\cyc)}^{\dR}$ is generated in homological degree $-1$ by the class $\omega:=\frac{1}{[T^\flat]}d[T^\flat]$ with $T^\flat=(T, T^{1/p},\ldots)$.
  Now, one can observe that $\Z_p(1)^\sm$ acts trivially on $\omega$ because $g[T^\flat]=\varepsilon^g[T^\flat]$ for $\varepsilon=(1,\zeta_p,\ldots )$ and $g\in \Z_p\cong \Z_p(1)$. Descending from $\Q_p^\cyc$ to $\Q_p$, we obtain the desired statement.
  \item Invariance of Hyodo--Kato cohomology with respect to the morphism $f\colon\A^{1,\diamond}_{\Z_p}\to \Marc(\Z_p)$ is true as well, as it follows by excision along
  $$\Marc(\Q_p)\subset \Marc(\Z_p) \supset \Marc(\F_p).$$
  Indeed, by \cref{sec:defin-first-prop-3-smoothness-primness-of-analytic-de-rham-stacks} the morphism $f^{\HK}$ is cohomologically smooth, thus $f^{\HK, *}$ commutes with base change, and one can argue by excision, \cref{excisionYdR}, as claimed.
  Over $\Marc(\F_p)$ one uses \cref{lemDiscInvariancedRFF} and over $\Marc(\Q_p)$ the previous discussion. As consequence, an analog of \cref{sec:key-computations-de-1-computation-for-gm} also holds for the arc-stack associated to the qfd arc-stack $\mathbb{G}_{m, \Z_p}^{\diamond}$ over $\Marc(\Z_p)$.

\end{enumerate}
\end{remark}

\begin{remark}\label{HKgoodred}
We note that the Hyodo--Kato cohomology over $\Z_p$ can be easily compared to the one over $\Q_p$. Namely, let $X$ be a qfd arc-stack over $\Z_p$, and denote $X_{\Q_p}= X\times_{\Marc(\Z_p)}\Marc(\Q_p)$. Then, the following diagram is cartesian
 $$
  \begin{tikzcd}\label{comparison}
  X_{\Q_p}^{\HK} \ar[r, "g^\prime"] \ar[d, "f_{\Q_p}^{\HK}"'] & X^{\HK}\ar[d, "f^{\HK}"] \\
  \Q_p^{\HK}   \ar[r, "g"] & \Z_p^{\HK}
  \end{tikzcd}
  $$
  (as $Y\mapsto Y^{\HK}$ commutes with fiber products), and the natural morphism
 $$(f_{\Q_p}^{\HK})_*g^{\prime *}\to g^*(f^{\HK})_*$$
 is an isomorphism of functors, since $g$ is suave (in fact, $g$ is an open immersion as $\Marc(\Q_p)\subset \Marc(\Z_p)$ is an open immersion, cf. \cref{excisionYdR}). In particular, pulling back $g$ to $\GSpec(\Q_p)$, we see that $(f^\HK)_*\mathcal{O}$ and $(f_{\Q_p}^{\HK})_*\mathcal{O}$ have isomorphic underlying solid $\Q_p$-modules.
\end{remark}

Next, we check that Hyodo--Kato cohomology has a theory of first Chern classes. This will be used in particular to identify the dualizing sheaf in \cref{sec:defin-first-prop-3-smoothness-primness-of-analytic-de-rham-stacks}.

\begin{definition}[Tate twist]\label{Tatetwist}
 We call \textit{Tate twist} on $\Marc(\F_p)^{\HK}$ the complex $\mathcal{O}\langle 1 \rangle:= \mathcal{O}(1)[2]$. Given a qfd arc stack $X$ over $\F_p$, we will also call Tate twist the pullback of the latter complex along $X^\HK\to \Marc(\F_p)^\HK$ (and we will abuse notation denoting it the same way).
\end{definition}

\begin{lemma}\label{LemComputationP1dRFF}
  The 6-functor formalism of Hyodo--Kato stacks on qfd arc-stacks over $\Z_p$ has a strong theory of first Chern classes $c_1^{\HK}$, in the sense of \cite[Definition 5.2.8]{zavyalov2023poincaredualityabstract6functor}.
  In particular, for any integer $d\ge 1$, and for any qfd arc-stack $X$ over $\Z_p$, denoting $f\colon\mathbb{P}_X^{d}\to X$ the natural projection of arc-stacks, we have that the morphism
\begin{equation}\label{projbundleformula}
 \sum_{k=0}^d c_1^k\langle d-k \rangle: \bigoplus_{k=0}^d \mathcal{O}\langle d-k \rangle \longrightarrow f_{*}^\HK \mathcal{O}\langle d\rangle
\end{equation}
is an isomorphism; here, $c_1$ denotes the first Chern class of the universal line bundle. A similar assertion holds for $\mathcal{Y}_{(-)}^{\dR}$.

\end{lemma}
\begin{proof}
 First, we need to construct a weak theory of first Chern classes, \cite[Definition 5.2.4]{zavyalov2023poincaredualityabstract6functor}. In our case, this amounts to constructing a natural transformation
 \begin{equation}\label{weakchern}
  \Gamma(-,\mathbb{G}_{m,\Z_p}^{\diamond}[1])\to \Gamma((-)^\HK,\mathcal{O}\langle 1 \rangle)
 \end{equation}
 of arc-sheaves with values in $\ob{D}(\Z)$; here, $\mathbb{G}_{m,\Z_p}^{\diamond}$ denotes the arc-stack over $\Marc(\Z_p)$ associated to $\mathbb{G}_{m, \Z_p}$. By \cref{TheoMaindeRham2}, the right hand side satisfies arc-descent, so that it is sufficient to construct a natural transformation
\[
\tau_{\leq 0}\Gamma(-,\mathbb{G}_{m,\Z_p}^\diamond[1])\to \tau_{\leq 0}\Gamma((-)^\HK,\mathcal{O}\langle 1 \rangle)
\]
of sheaves of animated abelian groups (namely, the arc-sheafification of $\mathbb{G}_{m,\Z_p}^\diamond=\tau_{\leq 0}\Gamma(-,\mathbb{G}_{m,\Z_p}^\diamond[1])$ as animated abelian group valued presheaf yields $\Gamma(-,\mathbb{G}_{m,\Z_p}^{\diamond}[1])$).
For this, we can follow the general discussion in \cref{sec:-able-maps-first-chern-classes-de-rham-cohomology}(3) with $\mathcal{S}=\Z_p^\HK$, $\mathcal{R}=\A^{1,\HK}_{\Z_p}$, and Tate twist $\mathcal{O}\langle 1\rangle$ on $\Z_p^\HK$.
Thus, we have to construct a morphism
\begin{equation}\label{c1HK}
 c_1^\HK\colon \mathbb{G}_{m,\Z_p}^\HK\to B\Ga(1)
\end{equation}
of animated abelian group Gelfand stacks  over $\Z_p^\HK$.
We let
\[
\mathbb{G}_{m,\Z_p}^\perf:=\varprojlim\limits_{x\mapsto x^p} \mathbb{G}_{m,\Z_p}.
\]
We have a short exact sequence $0\to \Z_p(1)^\diamond\to \mathbb{G}_{m,\Z_p}^{\perf,\diamond}\to \mathbb{G}_{m,\Z_p}^\diamond \to 0$ of (animated) abelian group arc-stacks over $\Marc(\Z_p)$, which yields a ($\varphi$-equivariant) short exact sequence
\[
0\to \mathcal{Y}_{\Z_p(1)^\diamond}^\dR\to \mathcal{Y}_{\mathbb{G}_{m,\Z_p}^{\perf,\diamond}}^\dR\to \mathcal{Y}_{\mathbb{G}_{m,\Z_p}^\diamond}^\dR\to 0
\]
of abelian group Gelfand stacks over $\mathcal{Y}_{\Marc(\Z_p)}^\dR$.
We have a natural epimorphism
\[
\mathbb{G}_{m,\mathcal{Y}_{\Marc(\Z_p)}^\dR}^{\perf}\to \mathcal{Y}_{\mathbb{G}_{m,\Z_p}^{\perf,\diamond}}^\dR
\]
of abelian group Gelfand stacks over $\mathcal{Y}_{\Marc(\Z_p)}^\dR$, which induces a short exact sequence
\begin{equation}
  \label{eq:4}
0\to \widetilde{\Z_p(1)} \to \mathbb{G}_{m,\mathcal{Y}_{\Marc(\Z_p)}^\dR}^{\perf}\to \mathcal{Y}_{\mathbb{G}_{m,\Z_p}^{\perf,\diamond}}^\dR\to 0
\end{equation}
of abelian group Gelfand stacks over $\mathcal{Y}_{\Marc(\Z_p)}^\dR$.
Here,
\[
\widetilde{\Z_p(1)}:={\mathcal{Y}_{\Z_p(1)^\diamond}^\dR}_{\times_{\mathcal{Y}_{\mathbb{G}_{m,\Z_p}^{\perf,\diamond}}^\dR}}\mathbb{G}_{m,\mathcal{Y}_{\Marc(\Z_p)}^\dR}^{\perf}.
\]
We note that the image of the natural map
\[
  \widetilde{\Z_p(1)}\to \mathbb{G}_{m,\mathcal{Y}_{\Marc(\Z_p)}^\dR}^{\perf}
\]
lands in $1+\mathring{\mathbb{D}}^\perf_{\mathcal{Y}_{\Marc(\Z_p)}^\dR}:=\varprojlim\limits_{x\mapsto x^p}(1+\mathring{\mathbb{D}}_{\mathcal{Y}_{\Marc(\Z_p)}^\dR})$ (this boils down to the fact that if $(1,x_1,\ldots)$ is a $p$-power compatible system, then its Teichm\"uller lift lies in $1+\mathring{\mathbb{D}}$).
Using the $\varphi$-equivariant map of abelian groups  $$\widetilde{\log}\colon (1+\mathring{\mathbb{D}}^\perf_{\mathcal{Y}_{\Marc(\Z_p)}^\dR})\to \mathbb{G}_{a,\mathcal{Y}_{\Marc(\Z_p)}^\dR}(1),\ (y_0,y_1,\ldots)\mapsto \log(y_0)$$ we can define $c_1^\HK$ in \eqref{c1HK}, as the morphism induced by the ($\varphi$-equivariant) composition
\begin{equation}\label{eqGeometricChernClass}
\mathcal{Y}_{\mathbb{G}_{m,\Z_p}^{\perf,\diamond}}^\dR \to B\widetilde{\Z_p(1)}\to B\mathbb{G}_{a,\mathcal{Y}_{\Marc(\Z_p)}^\dR}(1),
\end{equation}
where the first map classifies the extension \cref{eq:4}.

Next, we want to check that the morphism \cref{projbundleformula} is an isomorphism.\footnote{We note that by excision (\cref{excisionYdR}) and \cite[Theorem 5.5.7]{zavyalov2023poincaredualityabstract6functor}, we could even reduce to the case $d=1$. However, the proof below is not significantly harder for a general $d\ge 1$.} For this, after a base change, it suffices to show the following two assertions:

 \begin{enumerate}[(a)]
  \item Denote $h\colon B\mathbb{G}_{m, \Z_p}^\diamond\to \Marc(\Z_p)$ the natural morphism. We have
 \begin{equation}\label{HKBG}
  h^\HK_\ast\mathcal{O}=\bigoplus_{k\ge 0} \mathcal{O}(-k)c_1^{k}
 \end{equation}
where $c_1=d\log([T^\flat])\in H^1((\mathbb{G}_{m,\Z_p})^\HK,\mathcal{O})\cong H^2((B\mathbb{G}_{m,\Z_p})^\HK,\mathcal{O})$ sits in homological degree $-2$ (we note that this class is independent of the choice of a coordinate on $\mathbb{G}_{m,\Z_p}$). Moreover, under this identification, the morphism $\mathbb{P}_{\Z_p}^d \to B\mathbb{G}_{m, \Z_p}$, classifying the universal line bundle on $\mathbb{P}^d_{\Z_p}$, induces an isomorphism between the Hyodo--Kato cohomology of $\mathbb{P}_{\Z_p}^d$ and the direct summand  $\bigoplus_{k=0}^d \mathcal{O}(-k)c_1^{k}$ of \eqref{HKBG}.

 \item The map given by applying \cref{weakchern} to $\mathbb{P}_{\Z_p}^d$, and taking $H^0$, sends the universal line bundle on $\mathbb{P}_{\Z_p}^d$ to the class $c_1=d\log([T^\flat]) \in H^2((\mathbb{P}_{\Z_p}^d)^\HK, \mathcal{O})$.

\end{enumerate}

 Point (a) follows from the computation of the Hyodo--Kato cohomology of $\mathbb{G}_{m, \Z_p}$, and $\A_{\Z_p}^1$-invariance, \cref{sec:key-computations-de-1-invariance-over-q-p}(2). Indeed, the Hyodo--Kato cohomology of  $B\mathbb{G}_{m, \Z_p}$ can be calculated via approximation by the morphisms $\mathbb{P}^{n}_{\Z_p}\to B\mathbb{G}_{m, \Z_p}$, classifying the universal line bundle, with $(n-1)$-connected fiber $\mathbb{A}^{n+1}_{\Z_p}\setminus \{0\}$ (cf. the general argument for calculating the de Rham cohomology of $B \mathrm{GL}_n$, \cite[Th\'eorème 1.4]{Gros1990}; see also \cite[Corollary 8.2]{scholze2024berkovichmotives} for a related calculation).

  For point (b), we first note that, thanks to \cref{HKgoodred}, it can be deduced from the analogous assertion over $\Q_p$. Then, we note that the construction of $c_1^{\HK}$ is compatible with the construction of $c_1^{\dR}$ in \cref{sec:-able-maps-first-chern-classes-de-rham-cohomology}(1). Therefore, point (b) follows from the discussion in  \cref{sec:-able-maps-first-chern-classes-de-rham-cohomology}(2) on the analogous identification in the de Rham case.
\end{proof}

\begin{remark}
 \label{sec:key-computations-de-1-compactly-supported-cohomology}
 In \cref{sec:key-computations-de-1-invariance-over-q-p}(2), \cref{LemComputationP1dRFF} we calculated the Hyodo--Kato cohomology of $\mathbb{G}_{m,\Z_p}$ and $\mathbb{P}^1_{\Z_p}$, but this implies the result for \emph{compactly supported} Hyodo--Kato cohomology by passing to duals.

 Indeed, by excision and \cref{sec:defin-first-prop-3-smoothness-primness-of-analytic-de-rham-stacks} (or just \cref{sec:geom-prop-analyt-1-cohom-smoothness-of-perfectoid-unit-disc}), we can conclude a priori that the compactly supported Hyodo--Kato cohomology is a perfect module (cf. \cref{CorFinitenessdRFFcoho}), and hence dual to the Hyodo--Kato cohomology.
\end{remark}

\begin{remark}\label{RemarkComputationBG}
Let us sketch a different way to compute the Hyodo--Kato cohomology of $B\mathbb{G}_m$ over $\Z_p$, relying on abstract properties of the category of kernels and $6$-functor formalisms. An advantage of the following approach is that it can be generalized to compute the Hyodo--Kato cohomology of stacks of the form $B\mathbb{G}$ for $\mathbb{G}$ a reductive group.\footnote{Moreover, the same argument should generalize to any cohomology theory factoring through Berkovich motives \cite{scholze2024berkovichmotives}, working in the symmetric monoidal $(\infty, 2)$-category of kernels of Berkovich motives, studied in ongoing work of Aoki. We recall, however, also \cite[Corollary 8.2]{scholze2024berkovichmotives}.} Let $g\colon \mathbb{G}_m^{\HK}\to \Z_p^{\HK}$,  $h\colon B\mathbb{G}_m^{\HK}\to \Z_p^{\HK}$, and consider the zero section $e\colon \mathbb{Z}^{\HK}\to B\mathbb{G}_m^{\HK}$. The map $e$ is suave as its fibers are isomorphic to $\mathbb{G}_m^{\HK}$, in particular the functor $e^*$ is monadic (as it preserves limits, colimits and it is conservative). Then, we have that
\[
\ob{D}(B\mathbb{G}_m^{\HK}) = \Mod_{e^*e_{\natural}} (\ob{D}(\Z_p^{\HK}))
\] 
where $e_{\natural}$ is the left adjoint of $e^*$. As both the functors $e_{\natural}$ and $e^*$ are $\ob{D}(\Z_p^{\HK})$-linear (as they arise from and adjunction in the kernel category over $\Z_p^{\HK}$), we actually have that 
\[
\ob{D}(B\mathbb{G}_m^{\HK}) = \Mod_{e^*e_{\natural} 1}  (\ob{D}(\Z_p^{\HK}))
\]
where $e^*e_{\natural} 1\in \ob{Alg}(\ob{D}(\Z_p^{\HK}))$ is the algebra arising from the monad valued at $1$.  Using abstract nonsense in the kernel category, one can show that $e^*e_{\natural} 1$ is precisely the algebra arising from the homology of $\mathbb{G}_m^{\HK}$, that is the algebra $g_{\natural} 1$ with product given by convolution\footnote{One can even show that $g_{\natural} 1$ has a natural structure of animated ring object  over $\ob{D}(\Z_p^{\HK})$ using the animated group structure of $\mathbb{G}_m$.} with respect to the multiplicative structure of $\mathbb{G}_m^{\HK}$; for instance, the underlying sheaf can be computed by proper base change along the cartesian diagram
\[
\begin{tikzcd}
\mathbb{G}_m^{\HK} \ar[r] \ar[d] & \Z_p^{\HK} \ar[d] \\ 
\Z_p^{\HK} \ar[r] &  B\mathbb{G}_m^{\HK}.
\end{tikzcd}
\]
Now, the underlying sheaf of $g_{\natural} 1$ is given by $\mathcal{O}\oplus \mathcal{O}(1)[1]$ (where the class of $\mathcal{O}(1)[1]$ is constructed precisely using the map \cref{eqGeometricChernClass} above); this follows from $g_{\natural}1= g_!g^!1$, \cref{sec:key-computations-de-1-compactly-supported-cohomology}, and an explicit computation of the dualizing sheaf in this case. One can also identify $g_{\natural} 1$ with the square zero extension\footnote{For example, by noticing that $\Gamma(\Z_p^{\HK}, \mathcal{O}(i))=0$ for $i\neq 0$ so that there are no endomorphisms or extensions of vector bundles of different slope.} of $\mathcal{O}$ by $\mathcal{O}(1)[1]$.  Now, we have a natural identification of the cohomology
\[
h^{\HK}_*\mathcal{O} = \End_{\mathcal{O}(1)[1]\oplus \mathcal{O}}(\mathcal{O})
\] 
as objects over $\ob{D}(\Z_p^{\HK})$. Moreover, by Koszul duality, one computes that
\[
\End_{\mathcal{O}(1)[1]\oplus \mathcal{O}}(\mathcal{O})= \oplus_{n\geq 0} \mathcal{O}(-n)[-2n]
\]
and, more precisely, one recovers \eqref{HKBG}.
\end{remark}

 \subsection{Cohomological properties of Hyodo--Kato stacks}\label{ss:CohoPropertiesdRFF}

In this section, we apply the key suaveness property of the open punctured perfectoid disc of \cref{sec:geom-prop-analyt-1-cohom-smoothness-of-perfectoid-unit-disc} to deduce that rigid smooth morphisms of arc-stacks give rise to cohomologically smooth morphisms of Hyodo--Kato stacks.

Recall the notions of \'etale and smooth morphisms of qfd arc-stacks from \Cref{DefRigSmoothRigEtale}.

\begin{theorem}
  \label{sec:defin-first-prop-3-smoothness-primness-of-analytic-de-rham-stacks}
  Let $f:Y\to X$ be a map of qfd arc-stacks over $\Z_p$ and let $f^{\HK}\colon Y^{\HK}\to X^{\HK}$ be its associated morphism of Hyodo--Kato stacks.
  \begin{enumerate}
   \item Suppose that, locally in the arc-topology of $X$, $f$ is lqfd (see \cref{DefLQFDim}), then $f^{\HK}$ is $!$-able.
 Furthermore, if $f$ is proper then $f^{\HK}$ is cohomologically proper.
\item Suppose that $f$ is smooth of pure relative dimension $d$, then $f^{\HK}$ is cohomologically smooth.
  More precisely, it is suave and its dualizing sheaf is given by the $d$-th tensor power of the Tate twist (\cref{Tatetwist}).
 Furthermore, if $f$ is \'etale, then it is cohomologically \'etale. 
  \end{enumerate}
  Similar assertions hold for $f^{\mathcal{Y}^\dR}\colon \mathcal{Y}_Y^{\dR}\to \mathcal{Y}_X^{\dR}$.
\end{theorem}
\begin{proof}
We prove the statements for $f^{\mathcal{Y}^\dR}$, as the ones for $f^{\HK}$ follow from the latter (except for the calculation of the dualizing complex).

The $!$-ability is clear thanks to \Cref{LemShierkLocalTarget}, \cref{LemmProetaleShierkdR} and \cref{sec:geom-prop-analyt-1-cohomological-smoothness-for-smooth-rigid-spaces}, and properness follows from \cref{LemmLqfd}.
   Any \'etale map of qfd arc-stacks has open diagonal, namely, this can be checked locally in the arc-topology of $X$ and in the analytic topology of $Y$, where it reduces to the claim on perfectoid spaces which is obvious.
 Therefore, by definition of cohomologically \'etale \cite[Definition 4.6.1]{heyer20246functorformalismssmoothrepresentations}, we only need to show that smooth morphisms are cohomologically smooth.
 This property can be checked locally in the arc-topology of $X$ thanks to \cref{TheoMaindeRham2} and \cite[Lemma 4.5.7]{heyer20246functorformalismssmoothrepresentations}, and locally in the analytic topology on $Y$ thanks to \cref{sec:geom-prop-analyt-1-cohomological-smoothness-for-smooth-rigid-spaces}, \cite[Lemma 4.5.8 (i)]{heyer20246functorformalismssmoothrepresentations}.
  Therefore, we can assume without loss of generality that $X$ is qfd affinoid perfectoid over $\Z_p$ and $Y\to X$ is smooth.
  Moreover, applying \cref{sec:geom-prop-analyt-1-cohomological-smoothness-for-smooth-rigid-spaces}, \cite[Lemma 4.5.8 (i)]{heyer20246functorformalismssmoothrepresentations} again it suffices to show that $\mathcal{Y}_{\mathbb{A}^1_{X}}^{\dR}\to \mathcal{Y}_X^{\dR}$ is suave.
  By taking a base change and a further analytic cover we can also just reduce to the case of  $X=\Marc(\Z_p)$ and $Y=\mathbb{G}_{m,\Z_p}$, by further arc-descent even to the case $X=\Marc(\Z_p^{\cyc})$.

 Let $X=\Marc(\Z_p^{\cyc})$ and  $\mathbb{G}^{\ob{perf}}_{m,X}=\varprojlim_{x\mapsto x^p} \mathbb{G}_{m,X}$ be the perfection of the multiplicative group. One has an equivalence of arc-stacks over $\Q_p$
 \[
 \Yc^{\diamond}_{\mathbb{G}^{\ob{perf}}_{m,X}} = \Yc_X^{\diamond}\times_{\Marc(\Q_p)} \mathbb{G}^{\ob{perf},\diamond}_{m,\Q_p}
 \]
 where, if $T^{\flat}$ is the coordinate of $\mathbb{G}_{m,X}^{\perf}$, then the coordinate of the multiplicative group of the right term is the Teichm\"uller $[T^{\flat}]$.
 By \cref{CoroKeyCasesSuave} we deduce that $\Yc^{\dR}_{\mathbb{G}_{m,X}^{\perf}}\to \Yc^{\dR}_X$ is suave.
 By construction, the kernel of the map of groups $\mathbb{G}_{m,X}^{\perf,\diamond}\to \mathbb{G}_{m,X}^{\diamond}$   is the arc-stack given by the Tate module $T_p\mu_{p^{\infty},X}$ of the arc-stack of $p$-th powers roots of unit over $X$.
 As $X$ contains $p$-th power roots of unit, after fixing a compatible system $(\zeta_{p^{k}})_k$, we have a map of groups $\underline{\Z_p}\times X\to T_p \mu_{p^{\infty}}^{\diamond}$.
 Hence, we have a natural map $g\colon \mathbb{G}_{m,X}^{\perf,\diamond}/ \underline{\Z_p}\to \mathbb{G}_{m,X}^{\diamond}$.
 Passing to $\mathcal{Y}_{(-)}^{\dR}$ we get a map
 \begin{equation}\label{eq91j3ormqwdr}
 g^{\mathcal{Y}^{\dR}}\colon \Yc^{\dR}_{\mathbb{G}^{\ob{perf}}_{m,X}} / \Z_p^{\sm}\to  \Yc^{\dR}_{\mathbb{G}_{m,X}}.
 \end{equation}
 The left term is suave over $\Yc^{\dR}_X$ by the previous discussion and the fact that classifying stacks of Betti stacks of locally profinite groups are suave in characteristic zero, and the map $g^{\mathcal{Y}^{\dR}}$ is also prim.
 Therefore, to prove that $ \Yc^{\dR}_{\mathbb{G}_{m,X}}$ is suave over $\Yc_{X}^{\dR}$ it suffices to show that the natural map $1\to g^{\mathcal{Y}^{\dR}}_* 1$ is an equivalence.
 Indeed, by \cite[Lemma 4.5.16 (i)]{heyer20246functorformalismssmoothrepresentations}, $g^{\mathcal{Y}^{\dR}}_* 1$ is suave over $\Yc_{X}^{\dR}$.

To prove the claim we can argue via excision along $$j\colon \Marc(\Q_p^{\cyc})\subset \Marc(\Z_p^{\cyc}) \supset \Marc(\F_p) \colon \iota.$$
 The pullback along $j$ of $g$ is already an isomorphism, so there is nothing to show.
 It is left to prove the claim after pulling back along $\iota$, in that case the action of $\Z_p$ in $\mathbb{G}_{m}^{\perf,\diamond}$ is trivial, but the map $\mathbb{G}_{m,\F_p}^{\perf,\diamond}\to \mathbb{G}_m^{\diamond}$ is an isomorphism being the perfectoidization of a scheme in characteristic $p$.
 Since taking smooth $\Z_p$-group cohomology is exact  in characteristic zero, we deduce that the pullback along $\iota$ of \cref{eq91j3ormqwdr} is an isomorphism, finishing the proof of suaveness.

Finally, the computation of the dualizing sheaf follows from (the proof of) \cite[Theorem 5.7.7]{zavyalov2023poincaredualityabstract6functor}, recalling that the 6-functor formalism of the de Rham stacks of relative open punctured curves $\mathcal{Y}_{(-)}^{\dR}$ satisfies excision (\cref{excisionYdR}) and has a (strong) theory of first Chern classes (\cref{LemComputationP1dRFF}).
 This finishes the proof of the theorem. 
\end{proof}

\begin{corollary}\label{CorFinitenessdRFFcoho}
Let $f:Y\to X$ be a  proper smooth morphism of Berkovich spaces over $\Q_p$.
 Then, $$f^{\HK}_*\colon \ob{D}(Y^{\HK})\to \ob{D}(X^{\HK})$$ sends perfect modules to perfect modules. Similarly, $f_\ast^{\mathcal{Y}^{\dR}}\colon \ob{D}(\mathcal{Y}_Y^{\dR}) \to \ob{D}(\mathcal{Y}_{X}^{\dR})$ sends perfect modules to perfect modules.
\end{corollary}
 \begin{proof}
We prove the case of $f_\ast^{\HK}$, the argument for $f_\ast^{\mathcal{Y}^{\dR}}$ being identical.
 By \cref{sec:defin-first-prop-3-smoothness-primness-of-analytic-de-rham-stacks} the map $f^{\HK}$ is suave.
 We observe that $\FF_{Y}^{\diamond}\to \FF_{X}^{\diamond}$ is proper, as $f$ is proper. Then, \cref{LemmLqfd} implies that $f^{\HK}$ is cohomologically proper.
 In particular, $f^{\HK}_*$ sends $f^{\HK}$-suave objects to suave objects in $\ob{D}(X^{\HK})$ (\cite[Lemma 4.5.16 (ii)]{heyer20246functorformalismssmoothrepresentations}), and since any dualizable object in $\ob{D}(Y^{\HK})$ is $f^{\HK}$-suave (\cite[Corollary 4.5.18 (i)]{heyer20246functorformalismssmoothrepresentations}), one deduces that $f^{\HK}_*$ sends dualizable objects to dualizable objects, applying this time \cite[Corollary 4.5.18 (ii)]{heyer20246functorformalismssmoothrepresentations} to the identity.
 Then, by \cref{LemmaDescentPerfectComplexes} one deduces that dualizable objects in $\ob{D}(X^{\HK})$ are perfect modules, proving what we wanted.
 \end{proof}

 In future work, we will study an extension of the finiteness result \cref{CorFinitenessdRFFcoho} above for $f^{\HK}_*$, in the case $f:Y\to \GSpec(\Q_p)$ is the adic compactification of a qcqs rigid analytic variety over $\Q_p$.
\newpage
\section{The $p$-adic monodromy theorem}
\label{sec:p-adic-monodromy}

 The main goal of this section is to show that the category of vector bundles on $\Q_p^{\HK}$ is equivalent to Fontaine's category of $(\varphi, N, \Gal_{\Q_p})$-modules. Along the way, we give a new proof of the (local) $p$-adic monodromy theorem of Andr\'e, Kedlaya, and Mebkhout, \cite{Andre_monodromie}, \cite{Kedlaya_monodromie}, \cite{Mebkhout_monodromie}. Our proof is ``geometric'' in flavour, and it does not make reference to the classical theory of $p$-adic differential equations, but rather to the theory of arc-sheaves, Banach--Colmez spaces, and properties of the Hyodo--Kato stacks.
 \medskip

 For this whole section, we fix an algebraic closure $\overline{\Q}_p$ of $\Q_p$, and its completion $\C_p$.
 Consequently, we let $\overline{\F}_p$ be the residue field of $\mathcal{O}_{\C_p}$, and $\Q_p^\un\subseteq \overline{\Q}_p$ the maximal unramified extension of $\Q_p$.
 Given a subextension $L\subseteq \overline{\Q}_p$, we let $\Gal_L:=\Gal(\overline{\Q}_p/L)$ be its Galois group.

 \subsection{Statement of the $p$-adic monodromy theorem}
 \label{sec:p-adic-monodromy-2-the-p-adic-monodromy-theorem}
 In a simplified form, our version of the $p$-adic monodromy theorem reads as follows:

\emph{For any subextension $L\subseteq \overline{\Q}_p$, the category of vector bundles on $L^\HK$ is naturally equivalent to the category of $(\varphi, N,\mathrm{Gal}_L)$-modules over $\Q_p^\un$.  }

Here, we write $L^\HK:=\Marc(L)^\HK$, and note that it only depends on the uniform completion of $L$.

In this subsection, we will outline the contents of the next subsections, and how they will combine to a proof of (a slightly strengthened form of) the $p$-adic monodromy theorem.

We will first construct a suitable functor from $(\varphi, N,\mathrm{Gal}_L)$-modules over $\Q_p^\un$ to vector bundles on $L^\HK$ in \Cref{sec:construction-p-adic-construction-of-the-functor}. We note that it is easy to define a functor from $(\varphi,\Gal_L)$-modules over $\Q_p^\un$ to vector bundles on $L^\HK$.
By Galois descent (\cref{TheoMaindeRham2}), we have $L^\HK\cong \C_p^\HK/\Gal_L^{\rm sm}$, and for $\C_p$ (or equivalently $\overline{\Q}_p$) we can use the ($\Gal_{\Q_p}^{\sm}$-equivariant) morphism
\[
  \C_p^\HK \to \overline{\F}_p^\HK \cong \GSpec(\Q_p^\un)/\varphi^\Z
\]
(using an easy variant of \cref{RemDescentYFunctor}).
Clearly, vector bundles on $\GSpec(\Q_p^\un)/\varphi^\Z$ are equivalent to $\varphi$-modules over $\Q_p^\un$, so that pullback along the previous morphism yields the desired functor.
Hence, the main task in \cref{sec:construction-p-adic-construction-of-the-functor} will be to incorporate the monodromy operator $N$, and we will do this by constructing a $\Gal_{\Q_p}^{\sm}$-equivariant morphism (later $\Psi$ will be denoted by $\Psi_{\overline{\Q}_p,\varpi}$)
\[
  \Psi\colon \C_p^\HK \to B_{\overline{\F}_p^{\HK}}\mathbb{V}(-1) 
\]
with $B_{\overline{\F}_p^{\HK}}\mathbb{V}(-1) $ the classifying stack over $\overline{\F}_p^{\HK}$ of the (pullback of the) geometric vector bundle $\mathbb{V}(-1)$ associated with the simple isocrystal $D_{-1}=(\Q_p,p\varphi)$ of slope $-1$ (\cref{sec:construction-p-adic-definition-kottwitz-gerbe}).
In fact, vector bundles on $B_{\overline{\F}_p^{\HK}}\mathbb{V}(-1) $ are naturally equivalent to $(\varphi,N)$-modules over $\Q_p^\un$ (\cref{sec:construction-p-adic-description-of-vector-bundles-on-kottwitz-stack}).

We are now in a position to state precisely our version of the $p$-adic monodromy theorem.

\begin{theorem}[{$p$-adic monodromy theorem}]\label{localmonodromy}
  The pullback along $\Psi\colon \C_p^\HK\to B_{\overline{\F}_p^{\HK}}\mathbb{V}(-1)  $ defines a $t$-exact equivalence
  \[
    \Psi^\ast\colon \Perf(B_{\overline{\F}_p^{\HK}}\mathbb{V}(-1))\overset{\sim}{\longrightarrow} \Perf(\C_p^\HK).
  \]
  In particular, for any $L\subseteq \overline{\Q}_p$ the category of vector bundles on $L^\HK$ is naturally equivalent to the category of $(\varphi,N,\Gal_L)$-modules over $\Q_p^\un$.
 \end{theorem}

In particular, we see that each vector bundle on $\C_p^\HK$ is a successive extension of vector bundles, which are pulled back from $\overline{\F}_p^{\HK}\cong\GSpec(\Q_p^\un)/\varphi^\Z$.
 
We can generalize this property into the notion of \textit{unipotence}.

 \begin{definition}\label{unipotence}
Let $\F$ be a pro-finite \'etale extension of $\F_p$ and let $\Q_{p,\F}$ be the maximal algebraic unramified extension of $\Q_p$ with residue field $\F$.  Given $X\in \Cat{ArcStk}_{\F}^\qfd$, we say that a vector bundle on $X^\HK$ is \textit{unipotent} (with respect to $\mathbb{F}$) if it is a successive extension of vector bundles isomorphic to the pullback of a vector bundle on $\GSpec(\Q_{p,\F})/\varphi^{\Z}$ along the map
 $$X^\HK\to  \F^{\HK}=(\GSpec(\Q_{p,\F})/\varphi^{\Z}).$$

Let $X\in \Cat{ArcStk}$. We say that a vector bundle on $X^\HK$ is \emph{quasi-unipotent} if there exists a finite \'etale covering $Y\to X$ such that after pullback to $Y^{\HK}$ it becomes unipotent with respect to a finite extension $\mathbb{F}$ of $\mathbb{F}_p$. 
 \end{definition}

 For example, \cref{localmonodromy} implies that any vector bundle on $\Q_p^{\HK}$ is quasi-unipotent (see \cref{atclassicalpoints} for a more general statement).

\begin{example}
  \label{sec:local-p-adic-notation-for-o-lambda}
  Let $\lambda = r/s$ be a rational number with $s>0$ and $r,s$ coprime integers.
  We let $D_{\lambda}$ be the simple $\varphi$-module over $\Q_p$ of slope $\lambda$, i.e., the $\varphi$-module of dimension $s$ over $\Q_p$ and Frobenius map
\[
\varphi=\left(\begin{matrix}
0 & 1 & \cdots & 0 \\
\vdots & \vdots & \ddots & \vdots \\
 0 & 0 & \cdots  & 1 \\
p^{r} & 0 &  \cdots & 0
\end{matrix} \right).
\]
We can pullback the simple isocrystal $D_{-\lambda}$ (seen as a vector bundle on $\GSpec(\Q_p)/\varphi^\Z$) to a unipotent vector bundle $\mathcal{O}(\lambda)$ on $X^{\HK}$.
In particular, we note that unipotence in \cref{unipotence} does \emph{not} mean that there exists a finite filtration by the \emph{unit} on $X^{\HK}$.
For later relevance, we note that we will usually abuse notation and denote by $\mathcal{O}(\lambda)$ as well the pullback to $\FF_X$ (if $X$ is, say, perfectoid).
Similarly, we will denote by $D_\lambda$ as well the pullback of $D_\lambda$ to $\GSpec(\Q_p^\un)/\varphi^\Z$.
 \end{example}

\begin{remark}\label{classicalmonodromy}
 Before sketching our proof of \cref{localmonodromy} further, we note that from it (or rather, from its variant replacing $\C_p$ by $\Q_p^{\cyc}$, that we will also show below) one can deduce the classical statement of the local $p$-adic monodromy theorem for $p$-adic differential equations over the  Robba ring, asserting that any $(\varphi, \partial)$-module over the Robba ring is quasi-unipotent (\cite[Theorem 1.1]{Kedlaya_monodromie}). Indeed, it is well-known that the latter can be reformulated as saying that the category of $(\varphi,\partial)$-modules over the (cyclotomic) Robba ring is equivalent to the category of $(\varphi,N,\Gal_{\Q_p^{\rm cyc}})$-modules (see e.g.\cite[\S 0.2.3]{colmez2003conjectures}\footnote{A small typo in \textit{loc.cit.}: a Frobenius is missing one side, when stating the equivalence.}). Moreover, we will show in \cref{lemma:relation-p-adic-differential-equations} below that the category of vector bundles on $\Q_p^{\cyc, \HK}$ is equivalent to the category of  $(\varphi,\partial)$-modules over the  Robba ring.\footnote{Alternatively, one can check that the notion of quasi-unipotence in \cref{unipotence} matches the notion of quasi-unipotence in \cite{Kedlaya_monodromie}, under the equivalence in \cref{lemma:relation-p-adic-differential-equations}.}

\end{remark}

We recall that in Kedlaya's proof of the $p$-adic monodromy theorem, one needs two main ingredients.
The first ingredient is the fact that $(\varphi,\partial)$-modules over the Robba ring admit slope decompositions with semi-stable graded pieces (\cite[Theorem 1.2]{Kedlaya_monodromie}).
The second ingredient is the theorem of Tsuzuki \cite[Theorem 3.4.10]{zbMATH01199189} implying that semi-stable $(\varphi, \partial)$-modules over the Robba ring are quasi-unipotent.
Our proof for essential surjectivity of $\Psi^\ast$ in \cref{localmonodromy}, i.e., \cref{sec:tsuzukis-theorem-1-essential-surjectivity}, also makes use of an analogue of the first ingredient, in the form of the classification of vector bundles on the Fargues--Fontaine curve (of which Kedlaya's slope theory is an antecedent).
But we completely reprove Tsuzuki's theorem by geometric means, avoiding any deeper theory of $p$-adic differential equations.

Our proof of the fully faithfulness of $\Psi^\ast$ in \cref{sec:fully-faithf-psiast} uses all the geometric theory of analytic de Rham stacks that we have developed before, and applies it to variants of the morphism $\Psi\colon \C_p^\HK\to B_{\overline{\F}_p^{\HK}}\mathbb{V}(-1)$.
More precisely, using rather formal approximation and descent arguments, we reduce fully faithfulness of $\Psi^\ast$ (on perfect modules) to fully faithfulness of $\Psi^\ast_L$, where
\[
  \Psi_L\colon L^\HK\to B\mathbb{V}(-1)\times_{\GSpec(\Q_p)/\varphi^{\Z}}\GSpec(\Q_p^{\un})/(\varphi^{\Z}\times \Gal_L^{\sm})
\]
is a variant of $\Psi$ for a finite extension $L$ of $\Q_p(p^{1/p^\infty})\Q_p^\un$ of $\Q_p$.  After a further twist, we can make the geometry of $L^\HK$, which is a quotient by Frobenius of a punctured pre-perfectoid open unit disc, and $\Psi_L$ reasonably explicit, so that we can check fully faithfulness of $\Psi_L^\ast$ by an explicit calculation (\cref{sec:fully-faithf-psiast-1-case-for-f}).
A critical ingredient is here the suaveness result in \cref{sec:geom-prop-analyt-1-cohom-smoothness-of-perfectoid-unit-disc}.

\subsection{Relation with the Robba ring}\label{Sec:RelationWithRobba} 
 
We briefly  deviate from the proof of our version of the $p$-adic monodromy theorem in order to compare our theory with the more classical theory over the Robba ring (as anticipated in \cref{classicalmonodromy}). The reader interested in the proof of \cref{localmonodromy} can jump directly to next subsection.
Moreover, we note that the main result of this subsection, \cref{lemma:relation-p-adic-differential-equations}, is essentially a well-known spreading out/approximation argument (whose details are unfortunately a bit tedious to present).
 
\begin{definition}\label{DefRobbaRing}
  Let $\Z_p\llbracket X \rrbracket$ be the power series ring over $\Z_p$.
  The rigid generic fiber of $\Spf (\Z_p\llbracket X \rrbracket)$ is the punctured open unit disc $\mathring{\mathbb{D}}_{\Q_p}$ (seen, e.g., as a qfd Gelfand stack).
  For $0<r<s<1$ consider the affinoid overconvergent annulus 
\[
\overline{\mathbb{D}}_{[r,s]}\subset \mathring{\mathbb{D}}_{\Q_p}
\]
of closed internal radius $r$ and closed external radius $s$. We also consider the locally closed overconvergent annulus
\[
\overline{\mathbb{D}}_{[r,1)}\subset \mathring{\mathbb{D}}_{\Q_p}
\]
of closed internal radius $r$ and open external radius $1$.  The \textit{Robba ring} is defined as   the  ring of \textit{overconvergent functions in the boundary of $\mathring{\mathbb{D}}_{\Q_p}$}:
\begin{equation}\label{eqRobbaRingPresentation}
\mathcal{R}=\varinjlim_{r\to 1} \mathcal{O}(\overline{\mathbb{D}}_{[r,1)})=\varinjlim_{r\to 1} \varprojlim_{r<s<1} \mathcal{O}(\overline{\mathbb{D}}_{[r,s]}). 
\end{equation}
\end{definition}

In classical $p$-adic Hodge theory, one is interested in $\varphi$-modules over the Robba ring $\mathcal{R}$. There are, however, different choices of $\varphi$, serving different purposes. To package all of them together, we make the following definition:

\begin{definition}\label{DefDeltaStructureRobba}
We define a \textit{$\delta$-structure} on the Robba ring $\mathcal{R}$ to be a $\delta$-structure on $\Z_p\llbracket X \rrbracket$, or  equivalently, a lift of Frobenius $\varphi$ on $\Z_p\llbracket X \rrbracket$. We call the pair $(\mathcal{R}, \varphi)$ a \textit{Robba ring with $\delta$-structure}.
\end{definition}

\begin{example}\label{ExRobbaRings}
There are two main important examples in $p$-adic Hodge theory:
\begin{enumerate}

\item The \textit{cyclotomic Robba ring} with Frobenius $\varphi(X)=(1+X)^p-1$, that is, $\varphi$ is the multiplication by $p$ on the multiplicative formal group law on $\Z_p \llbracket X \rrbracket $.

\item The \textit{Kummer Robba ring} with Frobenius $\varphi(X)=X^p$.

\end{enumerate}
\end{example}

The following lemma takes care of the dynamical properties of such a Frobenius lift. 

\begin{lemma}\label{LemmaDynamicsFrobenius}
Let $p^{-1}=|p| <r<s<1$,  and let $\varphi$ be a $\delta$-structure on the Robba ring. The following hold: 
\begin{enumerate}

\item $\varphi$ restricts to an endomorphism of $\mathring{\mathbb{D}}_{\Q_p}$. 

\item The morphism $\varphi$ has the following dynamical properties  on annuli:
\[
\varphi\colon \overline{\mathbb{D}}_{[r^{1/p},s^{1/p}]}\to \overline{\mathbb{D}}_{[r,s]}
\]
and 
\[
\varphi\colon  \overline{\mathbb{D}}_{[r^{1/p}, 1)}\to  \overline{\mathbb{D}}_{[r, 1)}.
\]
\item The morphism $\varphi$ restricts to an endomorphism of the Robba ring $\mathcal{R}$. 
\end{enumerate}
\end{lemma}
\begin{proof}
  The Frobenius lift $\varphi$ is determined by a power series with $\Z_p$-coefficients $\varphi(X)$ such that $\varphi\equiv X^p \mod p\Z_p\llbracket X \rrbracket$. In particular, if $|X|<1$ then $|\varphi(X)|<1$ and so $\varphi$ restricts to an endomorphism of $\mathring{\mathbb{D}}_{\Q_p}$ proving (1).
  To prove (2), notice that if $p^{-1}=|p|<r $  and $|X|\geq r^{1/p}$ then $|\varphi(X)|\geq r$ thanks to the ultrametric property and  $|\varphi(X)-X^p|\leq |p|$ (in fact, $|\varphi(X)|=|X|^p$, which shows that $\varphi^{-1}(\overline{\mathbb{D}}_{[r,s]})=\overline{\mathbb{D}}_{[r^{1/p},s^{1/p}]}$).
  Similarly, $|X|\leq s^{1/p}$ will imply $|\varphi(X)|\leq s$. Finally, part (3) follows from (2) and the presentation of the Robba ring in \cref{eqRobbaRingPresentation}.
\end{proof}

Before giving examples of $\delta$-structure on Robba rings, let us discuss the relation with decompletions of the punctured pre-perfectoid open disc over $\Q_p$.

\begin{lemma}
Let $\varphi$ be a $\delta$-structure on the Robba ring, and consider the Berkovich space 
\[
\mathring{\mathbb{D}}^{\varphi-\rm perf}_{\Q_p}=\varprojlim_{\varphi} \mathring{\mathbb{D}}_{\Q_p}.
\]
The following hold:
\begin{enumerate}
\item The perfectoidization of   $\mathring{\mathbb{D}}^{\varphi-\rm perf}_{\Q_p}$ is naturally isomorphic to the pre-perfectoid open unit disc  $\mathring{\mathbb{D}}^{\perf, \diamond}_{\Q_p}$ with a variable $T$ satisfying $T\equiv X\mod p$, that is the locus $|T|<1$ of the arc-stack $\Marc(\Z_p\llbracket T^{1/p^{\infty}}\rrbracket )$. Under this identification, the Frobenius in $\mathring{\mathbb{D}}^{\varphi-\rm perf}_{\Q_p}$ is identified with the Frobenius $\varphi(T)=T^p$.
\item The underlying topological space $|\mathring{\mathbb{D}}^{\varphi-\rm perf}_{\Q_p}|=|\mathring{\mathbb{D}}^{\perf, \diamond}_{\Q_p}|$ has a unique $\varphi$-fixed point $0$. 
\item For $p^{-1}=|p|<r$ we have an equality of rational  subspaces of $|\mathring{\mathbb{D}}^{\varphi-\rm perf}_{\Q_p}|$:
\[
\{|T|\geq r\}=\{|X|\geq r\} \mbox{ and } \{|T|\leq r\}=\{|X|\leq r\}. 
\]

\end{enumerate}
\end{lemma}
\begin{proof}
Let $A$ be the completion of $\varinjlim_{\varphi }\Z_p \llbracket X \rrbracket $ for the $p$-adic topology, it is a perfect $p$-complete $\delta$-ring whose special fiber is isomorphic to $\varinjlim_n \F_p \llbracket X^{1/p^{\infty}} \rrbracket $. Thus, taking $T=[X]$ to be the Teichm\"uller  lift of $X\in \F_p\llbracket X \rrbracket $, we have that $A$ is the $p$-completion of the colimit $\varinjlim_n  \Z_p \llbracket T^{1/p^n}\rrbracket $. Passing to arc-stacks, inverting $p$,  and  taking the locus $\{|X|<1\}=\{|T|<1\}$ we get point (1).   For point (2), by construction the Frobenius $\varphi$ acts on the variable $T$ by $\varphi(T)=T^p$, therefore $|\varphi(T)|_x=|T|_x^p$  for all $x\in |\mathring{\mathbb{D}}^{\perf, \diamond}_{\Q_p}|$. Thus, $\varphi$ acts totally discontinuously on the locus $\{|T|\neq 0\}$, and its unique fixed point is $T=0$. Finally, part (3) follows from the fact that $T\equiv X \mod p$ in $A$.
\end{proof}

\begin{definition}\label{DefEquivariantModuleRobbaRing}
Let $(\mathcal{R}, \varphi)$ be a Robba ring with $\delta$-structure. We define the following notions of $\varphi$-equivariant vector bundles:
\begin{enumerate}

\item Let $p^{-1}=|p|<r<1$, a $\varphi$-equivariant  vector bundle on $\overline{\mathbb{D}}_{[r,1)}$ is a vector bundle on the coequalizer qfd Gelfand stack 
\[
\ob{coeq}(\overline{\mathbb{D}}_{[r^{1/p},1)} \underset{ \varphi}{\overset{\iota_1}{\rightrightarrows} } \overline{\mathbb{D}}_{[r,1)})
\]
where $\iota_1$ is the inclusion $\overline{\mathbb{D}}_{[r^{1/p},1)}\subset \overline{\mathbb{D}}_{[r,1)}$, and $\varphi$ is the Frobenius $\varphi \colon \overline{\mathbb{D}}_{[r^{1/p},1)}\to \overline{\mathbb{D}}_{[r,1)}$. We let $\ob{VB}(\overline{\mathbb{D}}_{[r,1)})^{\varphi}$ denote the category of $\varphi$-equivariant vector bundles on $\overline{\mathbb{D}}_{[r,1)}$.

\item Let $p^{-1}=|p|<r<1$, a $\varphi$-equivariant  vector bundle on $\overline{\mathbb{D}}_{[r,r^{1/p}]}$ is a vector bundle on the coequalizer stack
\[
\ob{coeq}(\overline{\mathbb{D}}_{[r^{1/p},r^{1/p}]} \underset{\iota_2 \circ \varphi}{\overset{\iota_1}{\rightrightarrows} } \overline{\mathbb{D}}_{[r,r^{1/p}]})
\]
where $\iota_1$ is the inclusion $\overline{\mathbb{D}}_{[r^{1/p},r^{1/p}]}\subset \overline{\mathbb{D}}_{[r,r^{1/p}]}$, $\varphi$ is the Frobenius $\varphi\colon \overline{\mathbb{D}}_{[r^{1/p},r^{1/p}]}\to \overline{\mathbb{D}}_{[r,r]}$ and $\iota_2$ is the inclusion $\overline{\mathbb{D}}_{[r,r]}\subset \overline{\mathbb{D}}_{[r,r^{1/p}]}$. We let $\ob{VB}(\overline{\mathbb{D}}_{[r,r^{1/p}]})^{\varphi}$ denote the category of $\varphi$-equivariant vector bundles on $\overline{\mathbb{D}}_{[r,r^{1/p}]}$. 

\item  A $\varphi$-equivariant vector bundle on $\mathcal{R}$ is a finite projective $\mathcal{R}$-module $M$ together with a $\varphi$-semilinear Frobenius endomorphism  $F\colon M\to M$ such that its linearization
\[
 M\otimes_{\mathcal{R},\varphi} \mathcal{R}\to  M
\]
is an isomorphism. We let $\ob{VB}(\mathcal{R})^{\varphi}$ denote the category of $\varphi$-equivariant vector bundles over $\mathcal{R}$. 
\end{enumerate}
\end{definition}

\begin{remark}\label{RemPreciseDefinitionPhiModules}
Let us mention concretely what $\varphi$-equivariant vector bundles are in the cases (1) and (2) of  \Cref{DefEquivariantModuleRobbaRing}. For (2), a $\varphi$-equivariant vector bundle on $\overline{\mathbb{D}}_{[r,1)}$ is the datum of a vector bundle $V$ on $\overline{\mathbb{D}}_{[r,1)}$, together with an isomorphism 
\[
\alpha\colon \varphi^* V \cong V_{\overline{\mathbb{D}}_{[r^{1/p},1)}}
\]
of vector bundles on $\overline{\mathbb{D}}_{[r^{1/p},1)}$.  Similarly, a $\varphi$-equivariant vector bundle on $\overline{\mathbb{D}}_{[r,r^{1/p}]}$ is the datum of a vector bundle $V$ on $\overline{\mathbb{D}}_{[r,r^{1/p}]}$ together with an isomorphism 
\[
\alpha\colon \varphi^* V|_{\overline{\mathbb{D}}_{[r^{1/p},r^{1/p}]}} \cong V|_{\overline{\mathbb{D}}_{[r^{1/p},r^{1/p}]}}
\]  
of vector bundles on $\overline{\mathbb{D}}_{[r^{1/p},r^{1/p}]}$. 
\end{remark}

\begin{proposition}\label{PropRelationPhiModules}
Let $(\mathcal{R},\varphi)$ be a Robba ring with $\delta$-structure. There are natural equivalences of  categories of $\varphi$-equivariant vector bundles   
\[
\ob{VB}\bigg(\big(\mathring{\mathbb{D}}^{\varphi-\rm perf}_{\Q_p}\backslash 0 \big)/\varphi^\Z\bigg) \cong \varinjlim_{n,\varphi}\ob{VB}(\overline{\mathbb{D}}_{[r^{1/p^n},r^{1/p^{n+1}}]})^{\varphi}  \cong \varinjlim_{n,\varphi} \ob{VB}(\overline{\mathbb{D}}_{[r^{1/p^n},1)})^{\varphi}  \cong \ob{VB}(\mathcal{R}^{\rm perf} )^{\varphi} \cong   \ob{VB}(\mathcal{R})^{\varphi},
\]
where $\mathcal{R}^{\rm perf}$ is the perfection of the Robba ring, and $\ob{VB}(\mathcal{R}^{\rm perf} )^{\varphi}$ is the category of $\varphi$-equivariant vector bundles on $\mathcal{R}^{\rm perf}$, and the transition maps in the colimits are given by pullbacks along the Frobenius $\varphi$.  
\end{proposition}
 \begin{proof}
We construct the equivalences from left to right. Let $p^{-1}=|p|<r<1$, and consider the decompleted perfect annulus $\overline{\mathbb{D}}^{\varphi-\rm perf}_{[r,r^{1/p}]}= \varprojlim_{n,\varphi}\overline{\mathbb{D}}_{[r^{1/p^n}, r^{1/p^{n+1}}]} $ given by the loci of $\mathring{\mathbb{D}}^{\varphi-\rm perf}_{\Q_p}$ where $r\leq |X|\leq r^{1/p}$ (and where the transition maps in the limit are given by Frobenius).  Define a $\varphi$-equivariant vector bundle on $\overline{\mathbb{D}}^{\varphi-\rm perf}_{[r,r^{1/p}]}$ as in \Cref{DefEquivariantModuleRobbaRing} (2), that is, a vector bundle in the stacky coequalizer 
\[
\ob{coeq}(\overline{\mathbb{D}}^{\varphi-\rm perf}_{[r^{1/p},r^{1/p}]} \underset{\iota_2 \circ \varphi}{\overset{\iota_1}{\rightrightarrows} }\overline{\mathbb{D}}^{\varphi-\rm perf}_{[r,r^{1/p}]}).
\]
Let $\ob{VB}(\overline{\mathbb{D}}^{\varphi-\rm perf}_{[r,r^{1/p}]})^{\varphi}$ denote the category of $\varphi$-equivariant vector bundles on $\overline{\mathbb{D}}^{\varphi-\rm perf}_{[r,r^{1/p}]}$.

\item[Step 1.] We construct first a natural equivalence 
\[
\ob{VB}\bigg(\big(\mathring{\mathbb{D}}^{\varphi-\rm perf}_{\Q_p}\backslash 0 \big)/\varphi^\Z\bigg) \cong  \ob{VB}(\overline{\mathbb{D}}^{\varphi-\rm perf}_{[r,r^{1/p}]})^{\varphi}. 
\]
As one might expect, this is nothing but a spreading out argument. Indeed, the two compositions
\[
\overline{\mathbb{D}}^{\varphi-\rm perf}_{[r^{1/p},r^{1/p}]} \underset{\iota_2 \circ \varphi}{\overset{\iota_1}{\rightrightarrows} }\overline{\mathbb{D}}^{\varphi-\rm perf}_{[r,r^{1/p}]}\to  \big(\mathring{\mathbb{D}}^{\varphi-\rm perf}_{\Q_p}\backslash 0 \big)/\varphi^\Z
\]
are identified, and so we have a morphism of stacks 
\[
\ob{coeq}(\overline{\mathbb{D}}^{\varphi-\rm perf}_{[r^{1/p},r^{1/p}]} \underset{\iota_2 \circ \varphi}{\overset{\iota_1}{\rightrightarrows} }\overline{\mathbb{D}}^{\varphi-\rm perf}_{[r,r^{1/p}]})\to  \big(\mathring{\mathbb{D}}^{\varphi-\rm perf}_{\Q_p}\backslash 0 \big)/\varphi^\Z. 
\]
Taking pullbacks this produces a symmetric monoidal functor 
\begin{equation}\label{eq0j13e1omwdp1}
\ob{VB}(\big(\mathring{\mathbb{D}}^{\varphi-\rm perf}_{\Q_p}\backslash 0 \big)/\varphi^\Z)\to \ob{VB}(\overline{\mathbb{D}}^{\varphi-\rm perf}_{[r,r^{1/p}]})^{\varphi}. 
\end{equation}
Concretely, given a vector bundle $V$ on $\big(\mathring{\mathbb{D}}^{\varphi-\rm perf}_{\Q_p}\backslash 0 \big)/\varphi^\Z$, its restriction to $\mathring{\mathbb{D}}^{\varphi-\rm perf}_{\Q_p}\backslash 0 $ is a $\varphi$-equivariant vector bundle, i.e., a vector bundle $V$ together with an isomorphism $\varphi^* V \cong V$. This restricts to the datum of an object in $\ob{VB}(\overline{\mathbb{D}}^{\varphi-\rm perf}_{[r,r^{1/p}]})^{\varphi}$ by \Cref{RemPreciseDefinitionPhiModules}. The equivalence of \eqref{eq0j13e1omwdp1} follows from \Cref{LemmSpreadOut}.

\item[Step 2.] By definition, we have an equivalence of categories 
\[
\ob{VB}(\overline{\mathbb{D}}^{\varphi-\rm perf}_{[r,r^{1/p}]})^{\varphi}=\ob{eq}\big( \ob{VB}(\overline{\mathbb{D}}^{\varphi-\rm perf}_{[r,r^{1/p}]})  \underset{(\iota_2 \circ \varphi)^*}{\overset{\iota_1^*}{\rightrightarrows} }  \overline{\mathbb{D}}^{\varphi-\rm perf}_{[r^{1/p},r^{1/p}]}\big).
\]
Since the annulus $\overline{\mathbb{D}}^{\varphi-\rm perf}_{[r,s]}$ is affinoid and equal to $\varprojlim_{n,\varphi} \overline{\mathbb{D}}_{[r^{1/p^n},s^{1/p^n}]}$, we have
\[
\ob{VB}(\overline{\mathbb{D}}^{\varphi-\rm perf}_{[r,s]})=\varinjlim_{n,\varphi^*} \ob{VB}(\overline{\mathbb{D}}_{[r^{1/p^n},s^{1/p^n}]})
\]
as categories (we note that on rings the filtered colimits is \emph{uncompleted}).
Thus, as equalizers commute with filtered colimits, we find that
\[
\ob{VB}(\overline{\mathbb{D}}^{\varphi-\rm perf}_{[r,r^{1/p}]})^{\varphi}=\varinjlim_{n,\varphi^*} \ob{VB}(\overline{\mathbb{D}}_{[r^{1/p^n},r^{1/p^{n+1}}]})^{\varphi}. 
\]
\item[Step 3.] We claim that the pullback along the map of coequalizers
\[
\ob{coeq}(\overline{\mathbb{D}}_{[r^{1/p},r^{1/p}]} \rightrightarrows \overline{\mathbb{D}}_{[r,r^{1/p}]} ) \to  \ob{coeq}(\overline{\mathbb{D}}_{[r^{1/p},1)} \rightrightarrows \overline{\mathbb{D}}_{[r,1)} ) 
\]
gives rise to an equivalence of vector bundles
\[
\ob{VB}(\overline{\mathbb{D}}_{[r^{1/p^n},r^{1/p^{n+1}}]})^{\varphi}\cong \ob{VB}(\overline{\mathbb{D}}_{[r^{1/p^n},1)})^{\varphi}.
\] 
Indeed, an inverse of the pullback is constructed by taking a $\varphi$-equivariant vector bundle $V$ on $\overline{\mathbb{D}}_{[r,r^{1/p}]}$ and send it to the $\varphi$-equivariant vector bundle $\widetilde{V}$ on $\overline{\mathbb{D}}_{[r,1)}$ constructed as follows: 

\begin{itemize}

\item The restriction of $\widetilde{V}$ to $\overline{\mathbb{D}}_{[r^{1/p^n}, r^{1/p^{n+1}}]}$ is $\varphi^{n,*} V$ where $\varphi^n\colon\overline{\mathbb{D}}_{[r^{1/p^n}, r^{1/p^{n+1}}]} \to \overline{\mathbb{D}}_{[r, r^{1/p]}}$, and the descent datum along the intersections $\overline{\mathbb{D}}_{[r^{1/p^n}, r^{1/p^{n+1}}]}\cap \overline{\mathbb{D}}_{[r^{1/p^{n+1}}, r^{1/p^{n+2}}]}=\overline{\mathbb{D}}_{[r^{1/p^{n+1}}, r^{1/p^{n+1}}]}$ is given by the $\varphi^{n,*}$-pullback of the isomorphism $\varphi^* V \cong V$ on $\overline{\mathbb{D}}_{[r^{1/p}, r^{1/p}]}$.

\item The $\varphi$-structure on $\widetilde{V}$ is obtained via the isomorphisms 
\[
\varphi^* \widetilde{V}|_{[r^{1/p^{n+1}}, r^{1/p^{n+2}}]} = \varphi^* (\varphi^{n,*} V )=\varphi^{n+1,*} V \cong \widetilde{V}|_{[r^{1/p^{n+1}}, r^{1/p^{n+2}}]}.
\] 
\end{itemize}
Passing to the colimit along Frobenius pullbacks we get 
\[
\varinjlim_{n,\varphi^*}\ob{VB}(\overline{\mathbb{D}}_{[r^{1/p^n},r^{1/p^{n+1}}]})^{\varphi}  \cong \varinjlim_{n,\varphi^*} \ob{VB}(\overline{\mathbb{D}}_{[r^{1/p^n},1)})^{\varphi}. 
\]

\item[Step 4.] Let $0<r'<r<1$, the space $\mathbb{D}_{(r', 1)}$ is Stein and by \cite[Corollary 2.6.8 and Proposition 2.6.13]{KedlayaLiuRelativeII} there is a $\otimes$-equivalence of categories of vector bundles 
\[
\ob{VB}(\mathbb{D}_{(r', 1)})\cong \ob{VB}(\mathcal{O}(\mathbb{D}_{(r', 1)})) 
\]
obtained by passing to global sections and localizing on the Stein  space.  Taking colimits as $r'\to r$, we get an equivalence of categories of vector bundles
\[
\ob{VB}(\overline{\mathbb{D}}_{[r, 1)}) \cong \ob{VB}(\mathcal{O}(\overline{\mathbb{D}}_{[r, 1)})). 
\]
Passing to Frobenius equivariant objects, and taking limits along Frobenius, we get isomorphisms 
\[
\varinjlim_{n,\varphi^*} \ob{VB}( \overline{\mathbb{D}}_{[r^{1/p^n}, 1)})^{\varphi} \cong  \ob{VB}( \varinjlim_{\varphi} \mathcal{O}(\overline{\mathbb{D}}_{[r^{1/p^n}, 1)}))^{\varphi}.
\]
Taking colimits as $r\to 1$, and noticing that the restriction maps on the left term are isomorphisms for $r<r'<1$ by \cref{LemmSpreadOut}, we deduce that 
\[
\varinjlim_{n,\varphi^*} \ob{VB}( \overline{\mathbb{D}}_{[r^{1/p^n}, 1)})^{\varphi} \cong \ob{VB}(\varinjlim_{r\to 1} \varinjlim_{\varphi} \mathcal{O}(\overline{\mathbb{D}}_{[r^{1/p^n}, 1)}))^{\varphi} = \ob{VB}(\mathcal{R}^{\rm perf})^{\varphi}. 
\]

\item[Step 5.] Finally, we have a base change functor $\ob{VB}(\mathcal{R})^{\varphi}\to \ob{VB}(\mathcal{R}^{\rm perf})^{\varphi}$ of $\varphi$-equivariant modules. We claim that it is an equivalence. Indeed, since $\mathcal{R}^{\rm perf}=\varinjlim_{\varphi} \mathcal{R}$ we have an equivalence of $\varphi$-equivariant vector bundles
\[
\ob{VB}(\mathcal{R}^{\perf})^{\varphi}=\varinjlim_{\varphi} \ob{VB}(\mathcal{R})^{\varphi},
\]
where transition maps are given by pullback along Frobenius. But by definition of $\varphi$-equivariant module,  the pullback map $\varphi^* \colon \ob{VB}(\mathcal{R})^{\varphi} \to \ob{VB}(\mathcal{R})^{\varphi} $ is an equivalence of categories. The claim follows. 
 \end{proof}
 
 The following spreading out argument was used in the proof of \Cref{PropRelationPhiModules}.

 \begin{lemma}[Spreading out argument]\label{LemmSpreadOut}
 Let $X$ be a Gelfand  stack with an endomorphism $\varphi\colon X\to X$. Suppose that $X$ admits a ``radius'' map $r\colon X\to (0,1)_{\Betti}$ such that $\varphi$ is equivariant with rising to the $p$, that is, such that we have a commutative diagram of analytic stacks 
 \[
 \begin{tikzcd}
 X  \ar[r, "\varphi"] \ar[d,"r"]& X \ar[d,"r"] \\ 
 (0,1)_{\Betti}  \ar[r,"x\mapsto x^{p}"]& (0,1)_{\Betti}.
 \end{tikzcd}
 \]
For $I\subset (0,1)$ an interval let us write $X_{I}=X\times_{(0,1)_{\Betti}} I_{\Betti}$. Suppose that $\varphi$ is an isomorphism on $X$. Then the natural maps of stacks 
\begin{equation}\label{eq023koqw23r}
\ob{coeq}(X_{[r^{1/p},r^{1/p}]} \underset{\iota_2 \circ \varphi}{\overset{\iota_1}{\rightrightarrows}} X_{[r,r^{1/p}]} )\to \ob{coeq}(X_{[r^{1/p},1)}\underset{\varphi}{\overset{j}{\rightrightarrows}}X_{[r,1)})\to X/\varphi^\Z
\end{equation}
are equivalences.

 \end{lemma}
 \begin{proof}
 Set $S^1=(0,1)/\{x\mapsto x^p\}$ be the circle obtained by taking the quotient of $(0,1)$ along the $p$-th power map. Then, we have the following isomorphisms of locally compact Hausdorff spaces
 \[
 S^1=\ob{coeq}(* \underset{r}{\overset{r^{1/p}}{\rightrightarrows}}  [r,r^{1/p}]) = \ob{coeq}( [r^{1/p},1)  \underset{x\mapsto x^p}{\overset{j}{\rightrightarrows}}  [r,1))
 \]
which pass to coequalizers at the level of Betti stacks. Taking pullbacks along the quotient of the radius morphism $r\colon X/\varphi^\Z\to S^1_{\Betti}$, and since colimits are universal in an $\infty$-topoi (i.e., they commute with pullbacks), we deduce the equivalences of stacks \eqref{eq023koqw23r}. 
\end{proof}

 \begin{corollary}\label{lemma:relation-p-adic-differential-equations}
 The category of vector bundles on $\Marc(\F_p((\pi^{1/p^{\infty}})))^{\HK}$ is equivalent to the category of $(\varphi,\partial)$-modules over \text{any} Robba ring with $\delta$-structure  as in \Cref{DefDeltaStructureRobba}.
\end{corollary}
\begin{proof}
This follows from \Cref{PropRelationPhiModules} and \Cref{explicit-description-drff-stack-of-qp} (which describes $\Marc(\F_p((\pi^{1/p^{\infty}})))^{\HK}$ as the analytic de Rham stack of $\big(\mathring{\mathbb{D}}^{\varphi-\rm perf}_{\Q_p}\backslash 0 \big)/\varphi^\Z$) noticing that vector bundles on the analytic de Rham stack are identified with vector bundles with flat connection, i.e., $\partial$-modules. 
\end{proof}

 \subsection{Construction of the $p$-adic monodromy functor}
\label{sec:construction-p-adic-construction-of-the-functor}

We begin by defining for each $L\subseteq \overline{\Q}_p$ a Gelfand stack geometrizing Fontaine's category of $(\varphi,N, \ob{Gal}_{L})$-modules over $\Q_p^{\un}$.
We recall from \cref{sec:local-p-adic-notation-for-o-lambda} that for each $\lambda\in \Q_p$ we have the $\varphi$-module $D_\lambda$ over $\Q_p^\un$ of slope $\lambda$, seen as vector bundle on $\GSpec(\Q_p^\un)/\varphi^\Z$.

\begin{definition}
  \label{sec:construction-p-adic-definition-kottwitz-gerbe}
  Let $L\subseteq \overline{\Q}_p$ be an extension of $\Q_p$.
  \begin{enumerate}
  \item Given $\lambda \in \mathbb{Q}$ we let $\mathbb{V}(\lambda)$ be the geometric vector bundle over $\GSpec(\Q_p)/\varphi^\Z$ defined by $D_\lambda$.
    Equivalently, it is the Gelfand stack over $\GSpec(\Q_p)/\varphi^\Z$ given by the relative analytic spectrum of the algebra $\ob{Sym}_{\GSpec(\Q_p)/\varphi^\Z} (D_{-\lambda})$ over $\GSpec(\Q_p)/\varphi^\Z$.

  \item We set
  $$\Font_L:= B\mathbb{V}(-1)\times_{\GSpec(\Q_p)/\varphi^{\Z}}\GSpec(\Q_p^{\un})/(\varphi^{\Z}\times \Gal_L^{\sm})$$
 and call it \textit{Fontaine's gerbe} attached to $L$.
Here, $\Gal_L^\sm$ acts on $\Q_p^\un$ naturally and this action commutes with $\varphi$ (the notation $(-)^\sm$ refers to the discussion after \cref{ExamCondensedAnimadeRham}).
\end{enumerate}
\end{definition}

Clearly, the stacks in \Cref{sec:construction-p-adic-definition-kottwitz-gerbe} admit obvious qfd variants, and we will use the same notation for these qfd versions.
We note that the \textit{Gelfand} stack represented by the symmetric algebra $\ob{Sym}_{\GSpec(\Q_p)/\varphi^\Z} (D_{-1})$ is a relative \textit{analytic} vector bundle, i.e., $!$-locally on the target isomorphic to the projection $\A^{1,\an}_{\Q_p}\to \GSpec(\Q_p)$.

\begin{remark}
  \label{sec:construction-p-adic-description-of-vector-bundles-on-kottwitz-stack}
  We note that the datum of a vector bundle on $\GSpec(\Q_p^\un)/(\varphi^\Z\times \Gal_L^\sm)$ is the same as the datum of a $(\varphi, \ob{Gal}_{L} )$-module over $\mathbb{Q}^{\un}_p$, where $\varphi$ is a $\mathbb{Q}^{\un}_p$-semilinear operator, and the action of $\ob{Gal}_{L}$ is smooth (and commutes with the $\varphi$-action).
  Thus, a vector bundle on Fontaine's gerbe $\Font_L$ is the datum of a Fontaine's $(\varphi, N, \Gal_L)$-module over $\mathbb{Q}_p^{\un}$, where $N$ is a nilpotent   $(\varphi, \ob{Gal}_L)$-equivariant map
 \begin{equation}\label{mdr}
  N\colon V\to V(-1):=V\otimes_{\Q_p}D_1.
 \end{equation}
 This follows, for example, from the Cartier duality for analytic vector bundles over $\Q_p$ as discussed in \cite[Section 4.3.2]{camargo2024analytic}.
 We note that after fixing a basis vector $e$ of $D_{1}$ with $\varphi(e)=pe$, \cref{mdr} is the same datum as an endomorphism $N\colon V\to V$ which is $\Gal_L$-equivariant and such that $N \varphi = p\varphi N$.
\end{remark}

 Now, we want to construct for any $L\subseteq \overline{\Q}_p$ a map of stacks (natural in $L$)
 \begin{equation}\label{monmap}
  \Psi_{L, \varpi}\colon L^\HK \to \Font_L
 \end{equation}
 that will depend on the choice of a pseudo-uniformizer $\varpi\in \Q_p$.
 \cref{localmonodromy} will then be the statement that the pullback functor $\Psi_{L,\varpi}^*$ induces an equivalence on perfect modules.
 We note that $L^\HK\cong \C_p^\HK/\Gal_L^{\rm sm}$, and similarly $\Font_L\cong\Font_{\overline{\Q}_p}/\Gal_L^{\rm sm}$, and hence it suffices to construct a Galois equivariant map $\Psi_{\overline{\Q}_p,\varpi}$.

We first note, as already mentioned, that the natural map $\overline{\F}_p\to \C_p^\flat$ induces a Galois-equivariant map
\[
  \C_p^\HK\to \overline{\F}_p^\HK\cong \GSpec(\Q_p^\un)/\varphi^\Z.
  \]
Thus, in order to construct a map as in \cref{monmap} for $K=\overline{\Q}_p$,  we need to construct a ($\Gal_{\Q_p}^{\sm}$-equivariant) extension
\begin{equation}\label{nonsplit}
0\to \shf{O}\to \mathcal{E}_{\varpi}\to \shf{O}(-1)\to 0
\end{equation}
of vector bundles on $\C_p^\HK$.
Indeed, the stack $\Font_{\overline{\Q}_p}$, as a Gelfand stack over $\GSpec(\Q_p^\un)/\varphi^\Z$, classifies vector bundle extensions $0\to \mathbb{V}(-1)\to \mathcal{E}\to \mathbb{V}(0)\to 0$, or equivalently extensions $0\to \mathbb{V}(0)\to \mathcal{E}^\vee\to \mathbb{V}(1)\to 0$.
Now, we use that $D_1$ pulls back to $\mathcal{O}(-1)$ on $\C_p^\HK$.

Luckily, there exists an interesting extension of $\shf{O}(-1)$ by $\shf{O}$, and this extension is (essentially) unique.

\begin{lemma}\label{choices}
  Let $L\subseteq \overline{\Q}_p$.
  Then
  \[
    H^i(L^\HK,\shf{O}(j))\cong
    \begin{cases}
      \Q_p, & i=j=0 \\
      \Q_p, & i=1, j=1\\
      0, & \text{otherwise.}
    \end{cases}
  \]
\end{lemma}
\begin{proof}
  We first assume that $\widehat{L}$ is perfectoid and $\widehat{L}^\flat\cong \F_q((\pi^{1/p^\infty}))$ for some $q=p^f$.
  Using \Cref{explicit-description-drff-stack-of-qp}(1), we see that
  \[
    \mathcal{Y}_L^\dR\cong (\varprojlim_{x\mapsto x^p}\mathring{\mathbb{D}}^{\times}_{\Q_{p^f}})^\dR,
  \]
  where $\Q_{p^f}$ is the unramified extension of $\Q_p$ of degree $f$.
  In particular, we can calculate
  \[
    H^i(\mathcal{Y}_L^\dR,\mathcal{O})\cong
    \begin{cases}
      \Q_{p^f},& i=0 \\
      \Q_{p^f} \frac{dx}{x},& i=1 \\
      0,& i\geq 2,
    \end{cases}
  \]
  where $x=[\pi]$ is the chosen coordinate of $\mathring{\mathbb{D}}^{\times}_{\Q_{p^f}}$.
  Thus, as $\varphi$-modules over $\Q_{p^f}$ we get $D_0\otimes_{\Q_p}\Q_{p^f}$ in degree $0$ and $D_{1}\otimes_{\Q_p}\Q_{p^f}$ in degree $1$.
  
  From here, it is easy to calculate the cohomology $H^i(L^\HK,\shf{O}(j))$ as the cohomology of the derived $\varphi$-invariants of $R\Gamma(\mathcal{Y}_L^\dR,\mathcal{O})\otimes_{\Q_{p}} D_{-j}$.
  Namely, one can observe that there are no non-trivial Ext-groups between $\varphi$-modules of different slopes.
  The result is
  \[
    H^i(L^\HK,\shf{O}(j))\cong
    \begin{cases}
      \Q_p,& i=j=0 \\
      \Q_p,& i=1, j=1\\
      0,& \textrm{ otherwise }
    \end{cases}
  \]
  as desired.
  The case for a general field $L$ with $\widehat{L}$ perfectoid follows now (using the approximation arguments in \cref{sec:fully-faithf-psiast-1-reduction-to-smaller-extensions}) by passing to the colimit of subfields $L'$ as before (noting that for a finite extension $L'\subseteq L''$ the pullback morphism $H^\ast(\mathcal{Y}^\dR_{L'},\shf{O})\to H^\ast(\mathcal{Y}^\dR_{L''},\shf{O})$ is an isomorphism, e.g., by the existence of norm maps).
  To pass to all $L\subseteq \overline{\Q}_p$, we can argue via descent from $\overline{\Q}_p$.
  We have to see therefore that the non-zero classes in $H^i(\C_p^\HK,\mathcal{O}(j))$ are invariant under $\Gal_{\Q_p}$.
  In degree $0$ this is clear, and in degree $1$ one can verify this by a calculation (namely, the class $\frac{dx}{x}$ is invariant under multiplication).
  Alternatively, \Cref{ConstructionMonodromy} below yields a Galois equivariant section of $H^1(\C_p^\HK,\mathcal{O}(1))$.
  This finishes the proof.
\end{proof}

 We proceed by constructing an explicit non-split  extension as in \cref{nonsplit}.
 For this, we need the following lemma:

\begin{lemma}
  \label{sec:construction-p-adic-map-to-quotient-by-unit-disc}
Let $\varpi\in \mathbb{Q}_p$ be a pseudo-uniformizer. Let us fix a sequence $\varpi^{\flat}=(\varpi^{1/p^n})_n$ of $p$-th roots of $\varpi$ in $\mathbb{C}_p$. Consider the map of qfd arc stacks over $\mathbb{Q}_p$
\[
\Phi_{\varpi^{\flat}}\colon \mathcal{Y}_{\mathbb{C}_p^{\flat}}^{\diamond}\to  \mathring{\mathbb{D}}^{\times,\diamond }
\]
induced by the topologically nilpotent unit $[\varpi^{\flat}]$ on $\mathcal{Y}_{\mathbb{C}_p}$.  Then for all $\sigma\in \ob{Gal}_{\mathbb{Q}_p}$ and $u\in 1+p\Z_p$ the composite  maps
\[
\mathcal{Y}_{\mathbb{C}_p^{\flat}}^{\diamond}\xrightarrow{\Phi_{\sigma((u\varpi)^{\flat})}}  \mathring{\mathbb{D}}^{\times,\diamond } \to \mathring{\mathbb{D}}^{\times,\diamond }/(1+\mathring{\mathbb{D}}^{\diamond })
\]
are equal, and they factor through a map
\[
\Phi_{\varpi} \colon  \mathcal{Y}_{\mathbb{Q}_p^{\diamond}}^{\diamond} \to \mathring{\mathbb{D}}^{\times,\diamond }/(1+\mathring{\mathbb{D}}^{\diamond})
\]
only depending on the class $\varpi\in p\Z_p/(1+p\Z_p)$.  Furthermore, this map is $\varphi$-equivariant where the action on the RHS is given by sending the coordinate $x$ to $x^p$.  Passing to de Rham stacks, we get a  $\varphi$-equivariant map
 \[
 \Phi_{\varpi}^{\dR}\colon \mathcal{Y}_{\mathbb{Q}_p^{\diamond}}^{\dR} \to (\mathring{\mathbb{D}}^{\times }/(1+\mathring{\mathbb{D}}))^{\dR} = \mathring{\mathbb{D}}^{\times }/(1+\mathring{\mathbb{D}}) .
 \]
\end{lemma}
\begin{proof}
  Let us first notice that $\mathring{\mathbb{D}}^{\times}/ (1+\mathring{\mathbb{D}})$ is isomorphic to its de Rham stack (cf. \Cref{sec:construction-p-adic-de-rham-stacks-of-finitary-sheaves} for a more general statement).
  Indeed, using \cref{TheoMaindeRham2} this follows from the fact that $\mathbb{G}_m^{\dagger}\subset   1+\mathring{\mathbb{D}}$ and that $\mathring{\mathbb{D}}^{\times,\dR}= \mathring{\mathbb{D}}^{\times}/\mathbb{G}_m^{\dagger}$ as $\mathring{\mathbb{D}}^{\times}$ is an open $\mathbb{G}_m^{\dagger}$-stable subspace of $\mathbb{G}_m^{\an}$.

Consider then the map
\[
 \mathcal{Y}_{\mathbb{C}_p}^{\diamond} \xrightarrow{\Phi_{\varpi^{\flat}}} \mathring{\mathbb{D}}^{\times,\diamond}\to \mathring{\mathbb{D}}^{\times,\diamond}/ (1+\mathring{\mathbb{D}}^{\diamond}).
\]
It is clear that the composite identifies $\Phi_{\varpi^{\flat}}$ and $\Phi_{(u\varpi)^{\flat}}$ for $u\in 1+p\Z_p$.  We want to show that it is $\Gal_{\Q_p}$-equivariant.
For this, consider the action map
\[
m\colon \Gal_{\mathbb{Q}_p}\times \mathcal{Y}^{\diamond}_{\mathbb{C}_p^{\flat}} \to  \mathcal{Y}^{\diamond}_{\mathbb{C}_p^{\flat}}.
\]
We want to see that the maps $[\varpi^{\flat}] \circ m$ and $[\varpi^{\flat}] \circ \pr_{2}$ are equal.
Let $S=\Spa(R)$ be a qfd perfectoid affinoid space over $\mathbb{Q}_p$ with a map $f\colon S\to \Gal_{\mathbb{Q}_p}\times \mathcal{Y}^{\diamond}_{\mathbb{C}_p^{\flat}}$ over $\mathbb{Q}_p$.
This is equivalent to a map of rings
\[
g\colon \mathbb{A}_{\inf}(\mathcal{O}_{\mathbb{C}_p})\to R,
\]
which maps $p$ and $[\varpi^{\flat}]$ to pseudo-uniformizers of $R$, and a map $C(\Gal_{\mathbb{Q}_p}, \mathbb{Q}_p)\to R$.
Together these maps give rise to a map
\[
h\colon  C(\Gal_{\mathbb{Q}_p}, \mathbb{A}_{\inf}(\mathcal{O}_{\mathbb{C}_p}))\to R.
\]
(as the LHS is the tensor product of $C(\Gal_{\mathbb{Q}_p}, \mathbb{Q}_p)$ with $\mathbb{A}_{\inf}(\shf{O}_{\mathbb{C}_p})$ over $\Z_p$).
The composite of  $f$ with $[\varpi^{\flat}] \circ \pr_2$ is induced by the pseudo-uniformizer $g([\varpi^{\flat}])\in R$.
On the other hand, the map $ [\varpi^{\flat}] \circ  m \circ  f$ arises from the pseudo-uniformizer $f(\mathrm{Orb}([\varpi^{\flat}]))$ where $\mathrm{Orb}([\varpi^{\flat}]) \in C(\Gal_{\mathbb{Q}_p}, \mathbb{A}_{\inf}(\shf{O}_{\mathbb{C}_p)})$ is the orbit $\sigma\mapsto \sigma([\varpi^\flat])$ of $[\varpi^{\flat}]$.
We know that for $\sigma\in \Gal_{\mathbb{Q}_p} $
  \[
 \sigma([\varpi^{\flat}])/  [\varpi^{\flat}] \equiv 1 \mod (p, [p^{\flat}]^{ 1/(p-1)})
  \]
  as $\sigma$ acts on $\varpi^{1/p^n}$ by multiplication by a $p^n$-th root of unit (with $p^{\flat}=(p^{1/p^n})_n$).
  Therefore, the quotient
\[
f(\mathrm{Orb}([\varpi]^{\flat}))/g([\varpi]^{\flat})
\]
lands in the locus $1+ (p, g([p^{\flat}]^{1/(p-1)})) R^{\leq 1}\subset 1+R^{<1}$.
This implies that the two maps $[\varpi^{\flat}]\circ m$ and $[\varpi^{\flat}]\circ \pr_2$ from $\Gal_{\mathbb{Q}_p}\times \mathcal{Y}_{\mathbb{C}_p^{\flat}}^{\diamond}$  to $(\mathring{\mathbb{D}}^{\times,\diamond}/ (1+\mathring{\mathbb{D}}^{\diamond}))$ are the same proving what we wanted.
\end{proof}

\begin{remark}
  \label{sec:construction-p-adic-de-rham-stacks-of-finitary-sheaves}
  The de Rham stacks of quotients like $\mathring{\mathbb{D}}^{\times }/(1+\mathring{\mathbb{D}})$ can be analyzed more generally:
  let $X=U/D$ be a quotient with $U\subseteq \A^{n}_{\Q_p}$ open, and $D\subseteq \A^n_{\Q_p}$ an open unit disc acting on $U$ by translations seen as a qfd arc-stack (or more generally a finitary qfd arc-sheaf, i.e., a qfd analog of \cite[Definition 4.22]{scholze2024berkovichmotives}, \cite[Proposition 3.12]{gerth2024hodge}).
  Then $X^\dR\cong U/D$ with now $U, D$ seen as qfd Gelfand stacks.
  Indeed, this follows from \cref{TheoMaindeRham2} and the fact that $U\times_{\GSpec(\Q_p)} D\to U\times_{\GSpec(\Q_p)}U$ is \'etale.
  We note that $X$ is quasi-finitary in the sense of \cref{sec:main-descent-result-1-descent-for-the-big-de-rham-stack}, and hence we can deduce from \cref{sec:main-descent-result-1-descent-for-the-big-de-rham-stack} that the same holds for the big de Rham stack $X^\dRall\cong U/D$ (now $U,D$ are seen as Gelfand stacks).
\end{remark}

We are ready to exhibit the key construction of an extension as in \cref{nonsplit}.

\begin{construction}\label{ConstructionMonodromy}
  Let $x$ be the coordinate of $X':=\mathring{\mathbb{D}}_{\mathbb{Q}_p}^{\times}$, and let $\mathcal{E}'$ be given by the rank $2$-vector bundle $\shf{O}_{X'} \cdot 1 \oplus \shf{O}_{X'} \cdot \ell_x$.\footnote{$\ell_x$ is a formal basis vector, but should be thought of as $\log x$.}
  Then we equip $\mathcal{E}'$ with the following $1+\mathring{\mathbb{D}}$-equivariant structure: given $R$ a  Gelfand ring with map $f\colon \GSpec R\to X'=\mathring{\mathbb{D}}_{\mathbb{Q}_p}^{\times}$, the action of $1+R^{<1}$  on $f^*\mathcal{E}'$ given by the trivial action on the basis element $1$ and by sending $\ell_x$  to
\[
\ell_{rx}:= \ell_x + \log r\cdot 1
\]
for $r\in 1+R^{<1}$, where $\log (r)$ arises from the logarithm homomorphism $\log\colon 1+\mathring{\mathbb{D}}\to \mathring{\mathbb{D}}$. In addition, by defining $\varphi(\ell_x)= \ell_{x^p} = p \ell_x$ we can make $\mathcal{E}$ a $\varphi$-equivariant vector bundle on $X= \mathring{\mathbb{D}}^{\times }/(1+\mathring{\mathbb{D}})$, where $\varphi$ is the operator sending $x$ to $x^p$. 

Note that we have a natural extension of the $\varphi$-equivariant vector bundles on $X$
\[
0\to \shf{O}_X\cdot 1\to \mathcal{E} \to \shf{O}_{X}(-1)\to 0,
\]
where $\shf{O}_X(-1)=\shf{O}_X\cdot \ell_x$, and $\varphi_{\shf{O}_X(-1)}=p\varphi_{\shf{O}_X}$. The bundle $\mathcal{E}'$ is then  endowed with a $1+\mathring{\mathbb{D}}$-equivariant connection obtained by the derivation of $1+\mathring{\mathbb{D}}$ on $\mathcal{E}$. Concretely, this connection  $\nabla$ is trivial on $1$ and sends
\[
\nabla(\ell_x)= \frac{dx}{x}. 
\]

Let $\varpi\in \mathbb{Q}_p$ be a pseudo-uniformizer.
By taking the pullback along the map $\Phi_{\varpi}\colon \mathcal{Y}_{\mathbb{Q}_p^{\diamond}}^{\dR} \to X$ from \cref{sec:construction-p-adic-map-to-quotient-by-unit-disc} we get a $\varphi$-equivariant vector bundle on $\mathcal{Y}_{\mathbb{Q}_p^{\diamond}}^{\dR}$ which descends to an extension
\begin{equation}\label{explicitext}
 0\to \shf{O}\to \mathcal{E}_{\varpi}\to \shf{O}(-1)\to 0
\end{equation}
of vector bundles on $\Q_p^{\HK}$.
As already remarked, this extension is classified by a map of stacks
\[
\Psi_{\Q_p,\varpi}\colon\Q_p^{\HK} \to \Font_{\Q_p}.
\]
Pullback along $\Font_L\to \Font_{\Q_p}$ defines the desired map $\Psi_{L,\varpi}\colon L^\HK\to \Font_{L}$ for any $L\subseteq \overline{\Q}_p$.
\end{construction}

Next, we note that the isomorphism class of the extension of vector bundles \cref{explicitext} can be also realized as the first Hyodo--Kato cohomology of a Tate elliptic curve. More precisely:

\begin{lemma}\label{LemComputationTateCurve}
  Let $\varpi\in \Q_p$ be a pseudo-uniformizer.
  Let $f\colon X_{\varpi}=\mathbb{G}_{m,\Q_p}^\diamond/\varpi^{\Z}\to \Marc(\Q_p)$ be Tate's elliptic curve over $\Q_p$, seen as a qfd arc-stack (here, $\mathbb{G}_{m,\Q_p}=\mathbb{G}_{m,\Q_p}^{\rm an}$).
  We have a natural isomorphism
 \begin{equation*}
  (f^{\HK})_*\mathcal{O}\cong \mathcal{O}\oplus \mathcal{E}_{\varpi}^{-}[-1]\oplus \mathcal{O}(-1)[-2]
 \end{equation*}
 of perfect modules over $\Q_p^{\HK}$,
 where $\mathcal{E}^{-}_{\varpi}$ identifies with the opposite of the  extension \eqref{explicitext}.
 \end{lemma}
 \begin{proof}

   Let $q$ be the coordinate of the open punctured unit disc $Y'=\mathring{\mathbb{D}}^{\times}$.
   The idea of the proof is to compare the Hyodo--Kato cohomology of $X_{\varpi}$ with the de Rham cohomology of the Tate curve $\mathcal{E}_{q}= (\mathbb{G}_m \times_{\Q_p} Y' )/q^{\Z}$ over $Y'$ and explicitly compute this last one.
   To compare these cohomologies, consider the qfd arc stack
\[
\mathcal{Y}^{\diamond}_{\mathbb{G}_{m,\Q_p}}=\mathbb{G}_{m,\Q_p}^{\diamond}\times_{\Marc(\F_p)} \Marc(\Q_p) 
\]
sending a qfd perfectoid ring $A$  over $\F_p$ to a tuple $(A_1^{\sharp}, A_2^{\sharp},a)$ with $A_i^{\sharp}$ untilts of $A$, and $a\in A_1^{\sharp,\times}$ a unit of the first untilt.
We have a morphism of  arc-stacks over $\Marc(\Q_p)$
\[
\alpha\colon \mathcal{Y}^{\diamond}_{\mathbb{G}_{m,\Q_p}}\to \mathbb{G}_{m,\Q_p}^{\diamond}/(1+\mathring{\mathbb{D}}^{\diamond})
\]
whose values at a qfd perfectoid $A$ over $\F_p$ sends a tuple $(A_1^{\sharp},A_2^{\sharp},a)$ to the tuple $(A_2^{\sharp},\theta_2([a^{\flat}]))$, where $a^{\flat} = (a^{1/p^n})_n$ is a compatible sequence of $p$-th power roots of $a$ in $A_1^{\sharp}$ (that we can guarantee locally in the arc-topology of $A$), and $\theta_2\colon \A_{\inf}(A)\to A_2^{\sharp}$ is Fontaine's map with respect to the second untilt $A_2^{\sharp}$ of $A$.
We note that any two choices $a^{\flat}$ and $a^{\flat,'}$ of systems of compatible $p$-th power roots of $a$  satisfy that $\frac{a^{\flat}}{(a')^{\flat}}$ is of norm $<1$ (namely, on geometric points it consists of compatible sequences of $p$-th power roots of $1$), and so the map $\theta_2([a^{\flat}])$ is well defined up to an element of $1+\mathring{\mathbb{D}}(A^\sharp_2)$.
We also see from the definition that $\alpha$ is a morphism of abelian groups.
 
  Thus, we have a morphism of arc-stacks 
\begin{equation}
\label{eqo0qjwoqwdpomwd}
\beta\colon \mathcal{Y}^{\diamond}_{\mathbb{G}_{m,\Q_p}} \to   \mathbb{G}_{m,\Q_p}^{\diamond}/(1+\mathring{\mathbb{D}}^{\diamond}) \times_{\Marc(\Q_p)} \mathcal{Y}_{\Q_p}^{\diamond}.
\end{equation}

\begin{claim}
Let $f\colon \mathcal{Y}^{\diamond}_{\mathbb{G}_m,\Q_p} \to \mathcal{Y}_{\Q_p}^{\diamond}$ and $g\colon \mathbb{G}_{m,\Q_p}^{\diamond}/(1+\mathring{\mathbb{D}}^{\diamond}) \times_{\Marc(\Q_p)} \mathcal{Y}_{\Q_p}^{\diamond}\to \mathcal{Y}_{\Q_p}^{\diamond}$ be the natural morphisms. Then the natural map  of de Rham cohomologies 
\begin{equation}\label{eqwo9joq3noqwdf}
g^{\dR}_*1 \to f^{\dR}_* 1
\end{equation}
is an isomorphism in $\ob{D}(\mathcal{Y}_{\Q_p}^{\dR})$.
\end{claim}
\begin{proof}
We have a correspondence of arc-stacks over $\mathcal{Y}_{\Q_p}^{\diamond}$
\begin{equation}\label{eq0jqo3malwdasd}
\begin{tikzcd}
\mathcal{Y}_{\mathbb{G}_{m,\Q_p}^{\perf}}^{\diamond}\cong   \mathbb{G}^{\perf,\diamond}_{m,\Q_p} \times_{\Marc(\Q_p)} \mathcal{Y}_{\Q_p}^\diamond \ar[r, "\beta'"] \ar[d] &  \mathbb{G}_{m,\Q_p}^{\diamond}/(1+\mathring{\mathbb{D}}^{\diamond}) \times_{\Marc(\Q_p)} \mathcal{Y}_{\Q_p}^{\diamond} \\ 
\mathcal{Y}_{\mathbb{G}_{m,\Q_p}}^{\diamond} & 
\end{tikzcd}
\end{equation}
where $\mathbb{G}_{m,\Q_p}^{\perf,\diamond}$ is the inverse limit of multiplication by $p$ on $\mathbb{G}_{m,\Q_p}$, the vertical map is the natural map arising from $\mathbb{G}_{m,\Q_p}^{\perf,\diamond}\to \mathbb{G}_{m,\Q_p}^{\diamond}$ after applying the functor  $\mathcal{Y}_{(-)}$, and the horizontal map arises from the natural projection map $\mathbb{G}_{m,\Q_p}^{\perf,\diamond}\to \mathbb{G}_{m,\Q_p}^{\diamond}/(1+\mathring{\mathbb{D}}^{\diamond})$ (induced by choosing a $p$-power compatible system of a coordinate on $\mathbb{G}_{m,\Q_p}$).
Thanks to \cref{sec:key-computations-de-1-computation-for-gm} and \cref{sec:key-computations-de-1-invariance-over-q-p}, the horizontal and vertical maps in \eqref{eq0jqo3malwdasd} induce an isomorphism on de Rham cohomology over $\mathcal{Y}_{\Q_p}^{\dR}$, thus proving the claim.
\end{proof}

Fix a pseudo-uniformizer $\varpi\in \Q_p$, let $Y=\mathring{\mathbb{D}}^{\times}/(1+\mathring{\mathbb{D}})$ and consider the map of arc-stacks  $\Phi_{\varpi,\Q_p}\colon \mathcal{Y}_{\Q_p}^{\diamond}\to Y^{\diamond}= \mathring{\mathbb{D}}^{\times,\diamond}/(1+\mathring{\mathbb{D}}^{\diamond})$ from \cref{sec:construction-p-adic-map-to-quotient-by-unit-disc}.

We note that after fixing $\varpi^\flat$ the automorphism $\sigma$ of multiplication by $\varpi$ on $\mathbb{G}_{m,\Q_p}$ lifts to the automorphism $\sigma'$ by multiplication by $\varpi^\flat$ on $\mathbb{G}_{m,\Q_p}^{\mathrm{perf}}$.
Now, the morphism $(\mathrm{Id}\times \Phi_{\varpi,\Q_p})\circ \beta'$ is equivariant for the automorphism on $\mathcal{Y}^\diamond_{\mathbb{G}_{m,\Q_p}^{\mathrm{perf}}}$ induced by $\sigma'$ and the automorphism induced by $(x,q)\mapsto (qx,q)$ on $\mathbb{G}_{m,\Q_p}\times \mathring{\mathbb{D}}^\times$.
Indeed, unraveling the constructions this amounts to the observation that $x$ is mapped to $[T^\flat]$ for a coordinate $T$ on $\mathbb{G}_{m,\Q_p}$ (with a choice of $p^n$-th roots), $q$ is mapped to $[\varpi^\flat]$ and $\sigma'([T^\flat])=[\varpi^\flat][T^\flat]$.

Thus, we have a commutative diagram of morphism of stacks
\begin{equation}\label{eqsojo3qwd}
\begin{tikzcd}
\mathcal{Y}_{\mathbb{G}_{m,\Q_p}/\varpi^{\Z}}^{\diamond} \ar[r] & \mathbb{G}^{\diamond}_{m,\Q_p}/(1+\mathring{\mathbb{D}}^{\diamond})[\varpi^{\flat}]^{\Z} \times_{\Marc(\Q_p)} \mathcal{Y}_{\Q_p}^{\diamond} \ar[r] \ar[d] &\mathbb{G}^{\diamond}_{m,\Q_p}/(1+\mathring{\mathbb{D}}^{\diamond})q^{\Z}\times_{\Marc(\Q_p)} Y^{\diamond} \ar[d]\\ 
  & \mathcal{Y}_{\Q_p}^{\diamond}  \ar[r] &  Y^{\diamond}
\end{tikzcd}
\end{equation}
where the square is a pullback square.
We deduce from the claim \eqref{eqwo9joq3noqwdf} and the previous computation that the relative de Rham cohomology of $\widetilde{g}\colon \mathbb{G}^{\diamond}_{m,\Q_p}/(1+\mathring{\mathbb{D}}^{\diamond})[\varpi^{\flat}]^{\Z}\times_{\Marc(\Q_p)} \mathcal{Y}_{\Q_p}^{\diamond}\to \mathcal{Y}_{\Q_p}^{\diamond}$ is equivalent to the Hyodo--Kato cohomology of the Tate curve, after modding out by the Frobenius.
It is left to compute  $\widetilde{g}^{\dR}_{*} 1$. Thanks to the pullback square of  \eqref{eqsojo3qwd}, it suffices to compute this cohomology over $Y=\mathring{\mathbb{D}}^{\times}/(1+\mathring{\mathbb{D}})$. Furthermore, since the  open disc is contractible for the de Rham cohomology, pullback along  $\mathring{\mathbb{D}}^{\times,\dR}\to \mathring{\mathbb{D}}^{\times}/(1+\mathring{\mathbb{D}})$ is fully faithful, and thus it suffices to compute the relative de Rham  cohomology of $\mathbb{G}^{\diamond}_{m,\Q_p}/(1+\mathring{\mathbb{D}}^{\diamond})q^{\Z}\times_{\Q_p} \mathring{\mathbb{D}}^{\times,\dR}\to \mathring{\mathbb{D}}^{\times,\dR}$.
Using again that the open disc is contractible, this is the same as the de Rham cohomology of the Tate elliptic curve $\mathcal{E}_q:=(\mathbb{G}_{m}\times_{\Q_p} \mathring{\mathbb{D}}^{\times})/q^{\Z} \to  \mathring{\mathbb{D}}^{\times}$. 

We now compute the de Rham cohomology of the Tate elliptic curve over the open  punctured disc, with its connection $\nabla^{\mathrm{GM}}$ (the ``Gauss--Manin connection''). For that, let $x$ denote the coordinate of $\mathbb{G}_m$ and $q$ the coordinate of $\mathring{\mathbb{D}}^{\times}$. Let $A=\mathcal{O}(\mathbb{G}_m\times_{\Q_p} \mathring{\mathbb{D}}^{\times})$, and consider the following notations:
\begin{itemize}
\item  $\nabla_x $ is the logarithmic connection on the variable $x$ sending $f(x,q)$ to $\nabla_x(f) =  x\partial_x(f) (x,q) \frac{dx}{x}$.

\item  $\nabla_q$ is the logarithmic connection on the variable $q$ sending $f(x,q)$ to $\nabla_q(f)=q\partial_q(f)(x,q)\frac{dx}{x}$.

\item $(m^*f)(x,q)= f(qx,q)$ is the pullback along multiplication by the second variable. 

\end{itemize}

Thus, the relative de Rham cohomology of the Tate curve $h\colon \mathcal{E}_q\to \mathring{\mathbb{D}}^{\times}$ is computed via  the cohomology of $q^{\Z}$ acting on the de Rham complex of $\mathbb{G}_m$, that is, via the complex on  $\mathring{\mathbb{D}}^{\times}$
\[
0\to \mathcal{O}(\mathbb{G}_m)\otimes_{\Q_{p,\solid}} \mathcal{O}_{\mathring{\mathbb{D}}^{\times}}   \xrightarrow{(m^*-1, \nabla_x)} \mathcal{O}(\mathbb{G}_m)\otimes_{\Q_{p,\solid}} \mathcal{O}_{\mathring{\mathbb{D}}^{\times}} \oplus \Omega^1_{\mathbb{G}_m/\Q_p}(\mathbb{G}_m)\otimes_{\Q_{p,\solid}} \mathcal{O}_{\mathring{\mathbb{D}}^{\times}} \xrightarrow{(-\nabla_x,m^*-1)} \Omega^1_{\mathbb{G}_m/\Q_p}(\mathbb{G}_m)\otimes_{\Q_{p,\solid}} \mathcal{O}_{\mathring{\mathbb{D}}^{\times}} \to 0.
\]
In order to compute the $D$-module structure of this complex over $\mathring{\mathbb{D}}^{\times}$, it suffices to understand the action of the derivation $\nabla_q=q\partial_q(-) \frac{dq}{q}$ through the Gauss--Manin connection.
For that, we note the following identities of derivations acting on an element $f(x,q)\in \mathcal{O}_{\mathbb{G}_m\times_{\Q_p} \mathring{\mathbb{D}}^{\times}}$:
\begin{itemize}
\item  $\nabla_q \nabla_x= \nabla_x \nabla_q$ (clear as $x,q$ are two independent coordinates).

\item $m^* \nabla_x = \nabla_x m^*$. Indeed, we have that 
\[
m^* \nabla_x(f) = m^* ( x\partial_x(f)(x,q) \frac{dx}{x})= qx \partial_x(f)(qx,q) \frac{dx}{x} =  \nabla_x (f(qx,q))= \nabla_x m^* (f).
\]

\item $\nabla_q m^* = m^* (x\partial_x(-)) \frac{dq}{q} + m^* \nabla_q$. Indeed, we have that 
\[
\begin{aligned}
\nabla_q m^*(f)= q\partial_q( f(qx,q)) \frac{dq}{q} = \big( qx(\partial_x(f)(qx,q))+ q\partial_q(f) (qx,q) \big)\frac{dq}{q}  \\ 
= m^* ((x\partial_x)(f)  \frac{dq}{q}) + m^* \nabla_q(f). 
\end{aligned}
\]

\end{itemize}

The previous equations yield the following commutative diagram of complexes
\begin{equation}\label{eqw0joweddcwsdq}
\adjustbox{scale=0.9,center}{
\begin{tikzcd}[row sep=1.5em, column sep=3.5em]
0 \ar[r]&\mathcal{O}(\mathbb{G}_m)\otimes_{\Q_{p,\solid}} \mathcal{O}_{\mathring{\mathbb{D}}^{\times}}   \ar[r, "{(m^*-1,\nabla_x)}"] \ar[d,"q\partial_q"] &\mathcal{O}(\mathbb{G}_m)\otimes_{\Q_{p,\solid}} \mathcal{O}_{\mathring{\mathbb{D}}^{\times}} \oplus \Omega^1_{\mathbb{G}_m/\Q_p}(\mathbb{G}_m)\otimes_{\Q_{p,\solid}} \mathcal{O}_{\mathring{\mathbb{D}}^{\times}} \ar[r,"{(-\nabla_x,m^*-1)}"] \ar[d,"\mathcal{L}"] &\Omega^1_{\mathbb{G}_m/\Q_p}(\mathbb{G}_m)\otimes_{\Q_{p,\solid}} \mathcal{O}_{\mathring{\mathbb{D}}^{\times}}  \ar[r] \ar[d,"q\partial_q"] & 0 \\ 
0\ar[r] &\mathcal{O}(\mathbb{G}_m)\otimes_{\Q_{p,\solid}} \mathcal{O}_{\mathring{\mathbb{D}}^{\times}}   \ar[r, "{(m^*-1,\nabla_x)}"] &\mathcal{O}(\mathbb{G}_m)\otimes_{\Q_{p,\solid}} \mathcal{O}_{\mathring{\mathbb{D}}^{\times}} \oplus \Omega^1_{\mathbb{G}_m/\Q_p}(\mathbb{G}_m)\otimes_{\Q_{p,\solid}} \mathcal{O}_{\mathring{\mathbb{D}}^{\times}} \ar[r,"{(-\nabla_x,m^*-1)}"] &\Omega^1_{\mathbb{G}_m/\Q_p}(\mathbb{G}_m)\otimes_{\Q_{p,\solid}} \mathcal{O}_{\mathring{\mathbb{D}}^{\times}}  \ar[r] & 0, 
\end{tikzcd}%
}
\end{equation}
where
\[\mathcal{L}\colon \mathcal{O}(\mathbb{G}_m)\otimes_{\Q_{p,\solid}} \mathcal{O}_{\mathring{\mathbb{D}}^{\times}} \oplus \Omega^1_{\mathbb{G}_m/\Q_p}(\mathbb{G}_m)\otimes_{\Q_{p,\solid}} \mathcal{O}_{\mathring{\mathbb{D}}^{\times}}\to \mathcal{O}(\mathbb{G}_m)\otimes_{\Q_{p,\solid}} \mathcal{O}_{\mathring{\mathbb{D}}^{\times}} \oplus \Omega^1_{\mathbb{G}_m/\Q_p}(\mathbb{G}_m)\otimes_{\Q_{p,\solid}} \mathcal{O}_{\mathring{\mathbb{D}}^{\times}}
\]
is the matrix of differential operators $\left(\begin{matrix} q\partial_q & -   m^* \\ 0 & q\partial_q \end{matrix}\right)$ sending $(f(x,q), g(x,q)\frac{dx}{x})$ to the pair 
\[
\mathcal{L}(f(x,q),g(x,q) \frac{dx}{x}):= ( q\partial_q(f) - m^*(g), q\partial_q(g) \frac{dx}{x}). 
\]
The commutativity of the right square is clear, and for the left square we compute for $f(x,q)\in \mathcal{O}(\mathbb{G}_m)\otimes_{\Q_{p,\solid}} \mathcal{O}_{\mathring{\mathbb{D}}^{\times}}$ using the above relations:
\begin{align*}
  \mathcal{L}(m^\ast(f),x\partial_x(f)\frac{dx}{x})= & (q\partial_q(m^\ast f)-m^\ast x\partial_x(f), q\partial_q(x\partial_x(f))\frac{dx}{x})\\
  = & (m^\ast(x\partial_x(f))+m^\ast(q\partial_q(f))-m^\ast x\partial_x(f), q\partial_q(x\partial_x(f))\frac{dx}{x}) \\
  = & (m^\ast(q\partial_q(f), q\partial_q(x\partial_x(f))\frac{dx}{x}).
\end{align*}
This implies the commutativity of the left square.
We note that the calculation shows that $\mathcal{L}$ is the \emph{unique} morphism making both diagrams commute.
In particular, the action of $q\partial_q(-)$ via the Gauss--Manin connection on the relative de Rham cohomology has to be given via $\mathcal{L}$ in cohomological degree $1$.

We make the action on the de Rham classes now explicit.
First, we note that multiplication by $p$ on the de Rham cohomology of the Tate elliptic curve $h\colon \mathcal{E}_q\to \mathring{\mathbb{D}}^{\times}$ is given by $p^i$ on the cohomology group in degree $i$, and hence this action provides a splitting of $h^\dR_\ast(1)$.
We deduce that $h^\dR_* 1 = H^0(h_*^{\dR} 1)\oplus H^1(h_*^{\dR}1)[-1] \oplus H^2(h_*^{ \dR}1)[-2]$.
We know already concrete basis vectors for the $H^i(h_\ast^\dR(1))$, and we can use the complex \eqref{eqw0joweddcwsdq} to calculate the Gauss--Manin connection $\nabla^{\mathrm{GM}}$ on them:
\begin{itemize}
\item $H^0(h_*^{\dR} 1)= \mathcal{O}_{\mathring{\mathbb{D}}^{\times}}\cong 1\otimes \mathcal{O}_{\mathring{\mathbb{D}}^\times}$ with the trivial connection $d\colon \mathcal{O}_{\mathring{\mathbb{D}}^{\times}}\to \Omega^1_{\mathring{\mathbb{D}}^\times},\ g\mapsto dg$ and the (trivial) semilinear action of $\varphi$ sending $1\mapsto 1$.

\item  $H^1(h_*^{\dR}1)= \mathcal{O}_{\mathring{\mathbb{D}}^{\times}}\oplus \mathcal{O}_{\mathring{\mathbb{D}}^{\times}}\frac{dx}{x}$ with connection given by $\nabla^{\mathrm{GM}}(\frac{dx}{x})=-1$ as $\mathcal{L}(0,1\cdot \frac{dx}{x})= (-1,0)$, and the action of Frobenius determined by $\varphi(1)=1$ and $\varphi(\frac{dx}{x})=p\frac{dx}{x}$.
  This produces the opposite of the extension \eqref{explicitext}.
\item $H^2(h_*^{\dR}1) = \mathcal{O}_{\mathring{\mathbb{D}}^{\times}} \frac{dx}{x}$ with trivial connection and action by Frobenius $\varphi(\frac{dx}{x})=p\frac{dx}{x}$. This is the vector bundle with flat connection $\mathcal{O}(-1)$.
\end{itemize}
This finishes the proof of the lemma. 
 \end{proof}

\subsection{Fully faithfulness of $\Psi^\ast_{L,\varpi}$}
\label{sec:fully-faithf-psiast}

In this subsection, we want to show that the functor
\[
  \Psi^\ast_{L,\varpi}\colon \Perf(\Font_L)\to \Perf(L^\HK)
\]
is fully faithful for any $L\subseteq \overline{\Q}_p$.

We note that the morphism $\Psi_{L,\varpi}\colon L^\HK\to \Font_L$ has bad geometric properties (it is neither prim nor suave): its base change to $\GSpec(\Q_p^\un)/\varphi^\Z$ is an analytic vector bundle over the huge space $\C_p^\HK$.
In particular, the pushforward $\Psi_{L,\varpi,\ast}$ on the full category of quasi-coherent sheaves does not commute with base change, making it difficult to verify that $\Psi_{L,\varpi,\ast}\circ \Psi^\ast_{L,\varpi}$ is the identity.
For this reason, we will make several reduction steps.

We start with a general observation (see the proof of \Cref{PropPerfectBerkovichSpaces} for related arguments).

\begin{lemma}
  \label{sec:fully-faithf-psiast-1-perfect-complexes-on-ff-curves-are-compact}
  Let $L\subseteq \overline{\Q}_p$.
  Then each perfect module on $L^\HK$ is a compact object in $\ob{D}(L^\HK)$.
\end{lemma}
\begin{proof}
  The statement holds more generally for $\Marc(A)^\HK$ for any affinoid qfd arc-stack $\Marc(A)$ over $\F_p$ (and then also for quotients of such by actions of light profinite groups).
  Indeed, using that quasi-pro-\'etale morphisms become inverse limits of pro-finite \'etale morphisms over a strictly totally disconnected base (and hence their $\ast$-pushfoward commutes with colimits on Hyodo--Kato stacks), one can reduce to the case of a closed disc in a relative affine space over $\F_p((t))$.
  Here, one can use explicit presentations of the Hyodo--Kato stacks to show that the pushforward $\Marc(A)^\HK\to \GSpec(\Q_p)$ preserves direct sums.
  From here, one deduces that $1\in \ob{D}(\Marc(A)^\HK)$ is compact, which implies that every dualizable object, i.e., perfect module, is compact.
\end{proof}

We can use \cref{sec:fully-faithf-psiast-1-perfect-complexes-on-ff-curves-are-compact} to deduce the following.

\begin{lemma}
  \label{sec:fully-faithf-psiast-1-reduction-to-smaller-extensions}
  Let $L\subseteq \overline{\Q}_p$.
  Assume that $\Psi^\ast_{L,\varpi}\colon \Perf(\Font_L)\to \Perf(L^\HK)$ is fully faithful (resp.\ an equivalence).
  Then for any extension $L'/L$ in $\overline{\Q}_p$, the functor
  \[
    \Psi^\ast_{L',\varpi}\colon \Perf(\Font_{L^{\prime}})\to \Perf(L^{\prime,\HK})
  \]
  is fully faithful (resp.\ an equivalence).
\end{lemma}
\begin{proof}
  We note that the morphism $B\Gal_{L'}^\sm\to B\Gal_L^\sm$ of qfd Gelfand stacks over $\GSpec(\Q_p)$ is pro-finite \'etale in the sense that after pullback along any morphism $\GSpec(A)\to B\Gal_L^\sm$ the pullback $\GSpec(A)\times_{B\Gal_L^\sm}B\Gal_{L'}^\sm\cong \GSpec(B)$ is represented by a (countable) filtered colimit of finite \'etale $A$-algebras $A_i$.
  In fact, one can take $A_i$ with $\GSpec(A_i)\cong \GSpec(A)\times_{B\Gal_L^\sm} B\Gal_{L_i}^\sm$ with $L_i\subseteq L$ finite over $K$.
Applying \Cref{TechnicalLemma6Functors} to the morphisms $\GSpec(\Q_p)\to B\Gal_{L_i}^{\sm}$ with $L\subset L_i\subset L'$ finite extensions of $L$,  whose hypothesis are verified thanks to  \cref{LemLimitsBerkovichSpaces}   and after noticing that the pushforward of $1$ along $B\Gal_{L'}^{\sm}\to B\Gal_{L_i}^{\sm}$ is $1$ being smooth cohomology in characteristic zero (using that profinite groups have no cohomology for smooth representations over $\Q_p$),    we can conclude that $B\Gal_{L'}^\sm\cong \varprojlim_{i} B\Gal_{L_i}^\sm$ in the category of kernels.
  In particular, $\ob{D}(B\Gal_{L'}^\sm)\cong \varprojlim_{i}\ob{D}(B\Gal_{L_i}^\sm)$, along $\ast$-pushforward.
  Thus, $\ob{D}(B\Gal_{L'}^\sm)\cong \varinjlim_{i}\ob{D}(B\Gal_{L_i}^\sm)$ via $\ast$-pullback (with filtered colimit taken in $\Cat{Pr}^L$).
  We note that perfect modules on $B\Gal_{L'}^\sm$ are compact, and similarly for $L_i$.
  This compactness and the expression of $\ob{D}(B\Gal_{L'}^\sm)$ as a filtered colimit in $\Cat{Pr}^L$ implies that each perfect module on $\ob{D}(B\Gal_{L'}^\sm)$ is the pullback of a compact object (and then of a perfect module, by descent of perfect modules, \cref{LemmaDescentPerfectComplexes}) on some $B\Gal_{L_i}^\sm$ (we note that the categories here are compactly generated, so that \cite[Corollary A.5.9]{heyer20246functorformalismssmoothrepresentations} applies).
  Almost the same argument holds after base change along $L^\HK\to B\Gal_L^\sm$ or $\Font_L\to B\Gal_L^\sm$.
  Ccompactness of perfect modules and compact generation can easily be checked for $\Font_{L'}$ and the $\Font_{L_i}$, which then implies $\Perf(\Font_{L'})\cong \varinjlim_i\Perf(\Font_{L_i})$.
  For ${L'}^\HK$ we need to argue differently.
  But using \cref{sec:fully-faithf-psiast-1-perfect-complexes-on-ff-curves-are-compact} the same argument as in \Cref{PropPerfectBerkovichSpaces} shows $\Perf(L^{\prime, \HK})\cong \varinjlim_{i} \Perf({L_i}^{\HK})$ as desired.

  Now for $L_i/K$ finite, the fully faithfulness of $\Psi^\ast_{L,\infty}$ (resp.\ equivalence) implies fully faithfulness of $\Psi^\ast_{L'_i,\infty}$ (resp.\ equivalence) by passing to module categories for finite \'etale algebra, and then the general case follows by passage to a colimit.
\end{proof}

Moreover, we can argue by descent.

\begin{lemma}
  \label{sec:fully-faithf-psiast-1-descent}
  Let $L\subseteq L'\subseteq \overline{\Q}_p$.
  Assume that $\Psi^\ast_{L',\varpi}\colon \Perf(\Font_{L'})\to \Perf(L^{\prime,\HK})$ is fully faithful (resp.\ an equivalence).
  Then $\Psi^\ast_{L,\varpi}\colon \Perf(\Font_L)\to \Perf(L^{\HK})$ is fully faithful (resp.\ an equivalence).
\end{lemma}
\begin{proof}
  We note that we get a natural morphism $\Psi^\bullet$ between the \v{C}ech nerves of the morphisms
  \[
    L^{\prime,\HK}\to L^\HK
  \]
  resp.\
  \[
    \Font_{L^\prime}\to \Font_L.
  \]
  In fact, both these \v{C}ech nerves are pulled back from the \v{C}ech nerve of the morphism $B\Gal_{L'}^\sm\to B\Gal_L^\sm$ along the morphisms $L^\HK\to B\Gal_L^\sm$ resp.\ $\Font_L\to B\Gal_L^\sm$.
  In particular, the morphisms in these \v{C}ech nerves are all pro-finite \'etale.

  It suffices to show that for each $i\in \Delta$ the pullback $\Psi^{i,\ast}$ is fully faithful (resp.\ an equivalence) on perfect modules.
  This follows from the proof of \Cref{sec:fully-faithf-psiast-1-reduction-to-smaller-extensions}, which generalizes to show that fully faithfulness (resp.\ equivalence) of $\Psi^\ast_{L',\varpi}$ on perfect modules implies the same after base change along any countable pro-finite \'etale morphism $Z\to B\Gal_{L'}^\sm$, e.g., $Z=B\Gal_{L'}^\sm\times_{B\Gal_L^\sm} \cdots \times_{B\Gal_L^\sm} B\Gal_{L'}^\sm$ with respect to the first projection. 
\end{proof}

However, the geometric properties of $\Psi_{L,\varpi}$ are still bad, so we need a further reduction.

\begin{lemma}
  \label{sec:fully-faithf-psiast-1-getting-rid-of-galois-group}
  Let $L\subseteq L'\subseteq \overline{\Q}_p$.
  If $\Q_p^\un\subseteq L'$ and $L'/L$ is Galois, set \[\Font_{L'/L}:=(B\mathbb{V}(-1)\times_{\GSpec(\Q_p)/\varphi^\Z}\GSpec(\Q_p^\un)/\varphi^\Z)/\Gal(L'/L)^\sm.\]
  Let $\Phi_{L'/L}\colon \Font_{L}\to \Font_{L'/L}$ be the natural morphism, and $\Psi_{L'/L,\varpi}:=\Phi_{L'/L}\circ \Psi_{L,\varpi}$.
  Then, $$\Phi^\ast_{L'/L}\colon \ob{D}(\Font_{L'/L})\to \ob{D}(\Font_L)$$ is fully faithful.

  In particular, $\Psi^\ast_{L,\varpi}$ is fully faithful on perfect modules if and only if for each $L'/L$ Galois with $\Q_p^\un\subseteq L'$ and $L'/L\Q_p^\un$ finite, the morphism $\Psi^\ast_{L'/L,\varpi}$ is fully faithful on perfect modules.
\end{lemma}
We note that as $\Q_p^\un\subseteq L'$, $\Gal(L'/L)$ acts on $\Q_p^\un$.
\begin{proof}
  Let $g\colon B\Gal_L^\sm\to B\Gal(L'/L)^\sm$ be the natural morphism.
  Then $g$ is cohomologically proper, and the vanishing of higher cohomology of smooth representations of profinite groups on $\Q$-vector spaces implies that $g_\ast(1)=1$.
  In particular, $g^\ast$ is fully faithful and the same holds then using cohomological properness of $g$ for the pullback $\Phi_{L'/L}$ of $g$.

  The final statement follows because $B\Gal_L^\sm$ is the (countable) inverse limit of the $B \Gal(L'/L)^\sm$ in the category of kernels (as follows, for example, by using \cref{TechnicalLemma6Functors} and the same argument in the proof of \cref{sec:fully-faithf-psiast-1-reduction-to-smaller-extensions}),    and as consequence one has that $\Perf(\Font_L)$ is the filtered colimit in $\mathrm{Cat}_\infty$ of $\Perf(\Font_{L'/L})$, along fully faithful transition functors.
\end{proof}

\begin{remark}
  \label{sec:fully-faithf-psiast-2-reductions-summary}
  From the reductions we have made we can now conclude that the following statements are equivalent:
  \begin{enumerate}
   
  \item for all $L\subseteq \overline{\Q}_p$ the functor $\Psi^\ast_{L,\varpi}\colon \Perf(\Font_L)\to \Perf(L^\HK)$ is fully faithful (resp.\ an equivalence),
    \item there exists some $L\subseteq \overline{\Q}_p$ such that the functor $\Psi^\ast_{L,\varpi}\colon \Perf(\Font_L)\to \Perf(L^\HK)$ is fully faithful (resp.\ an equivalence),
  \item there exists some $L\subseteq \overline{\Q}_p$ containing $\Q_p^\un$, such that for each finite extension $L'/L$ the pullback
    \[
     \Phi_{L'}^\ast\colon \Perf(\GSpec(\Q_p^\un)/\varphi^\Z\times_{\GSpec(\Q_p)/\varphi^\Z} B\mathbb{V}(-1))\to \Perf(L^{\prime,\HK})
    \]
    along 
    $$
    \Phi_{L'} \colon L^{\prime,\HK}\to \GSpec(\Q_p^\un)/\varphi^\Z\times_{\GSpec(\Q_p)/\varphi^\Z} B\mathbb{V}(-1)
    $$
    is fully faithful.
  \end{enumerate}
  Indeed, by \Cref{sec:fully-faithf-psiast-1-reduction-to-smaller-extensions} and  \Cref{sec:fully-faithf-psiast-1-descent} the first two statements are equivalent. In fact, it suffices to see that $\Psi_{\overline{\Q}_p,\varpi}^\ast$  or  $\Psi_{\Q_p,\varpi}^*$ induce   equivalences of perfect modules.   To get rid of the $\Gal_L^\sm$-action, fix some $L$ as in the third statement.
  By \cref{sec:fully-faithf-psiast-1-getting-rid-of-galois-group} we need to see that $\Psi_{L'/L,\varpi}^\ast\colon \Perf(\Font_{L'/L})\to \Perf(L^\HK)$ is fully faithful on perfect modules for any $L'/L$ finite and Galois.
  However, this can be checked (by adapting the argument in \cref{sec:fully-faithf-psiast-1-descent}) after pulling back along the morphism $$\GSpec(\Q_p^\un)/\varphi^\Z\times_{\GSpec(\Q_p)/\varphi^\Z} B\mathbb{V}(-1)\to \Font_{L'/L},$$ which yields the map  $\Phi_{L'} \colon L^{\prime,\HK}\to \GSpec(\Q_p^\un)/\varphi^\Z\times_{\GSpec(\Q_p)/\varphi^\Z} B\mathbb{V}(-1)$.
  \end{remark}

We will apply the above considerations to $L:=\Q_p^\un\Q_p^\Kum$, where $\Q_p^\Kum=\Q_p(p^{1/p^\infty})=\bigcup_{n\geq 0}\Q_p(p^{1/p^n})$.
We also fix a finite extension $L'/L$, with associated morphism
\[
  \Phi_{L'}\colon L^{\prime,\HK} \to \GSpec(\Q_p^\un)/\varphi^\Z\times_{\GSpec(\Q_p)/\varphi^\Z} B\mathbb{V}(-1).
\]
The geometry of the map $\Phi_{L'}$ is now much better than that of $\Psi_{L,\varpi}$, and we can prove the following.

\begin{proposition}
  \label{sec:fully-faithf-psiast-1-case-for-f}
  The morphism $\Phi:=\Phi_{L'}\colon L^{\prime,\HK}\to \GSpec(\Q_p^\un)/\varphi^\Z\times_{\GSpec(\Q_p)/\varphi^\Z} B\mathbb{V}(-1)$ is suave.
  Let $\Phi_{\natural}$ be the left adjoint of the pullback functor $\Phi^*$ (on the full category of quasi-coherent sheaves).   Then the natural map
\[
\Phi_{\natural} 1\to  1
\]
is an equivalence.
In particular, the pullback
\[
\Phi^* \colon \ob{D}(\GSpec(\Q_p^\un)/\varphi^\Z\times_{\GSpec(\Q_p)/\varphi^\Z} B\mathbb{V}(-1)) \to \ob{D}(L^{\prime,\HK})
\]
is fully faithful (and hence its restriction to perfect modules, too).
\end{proposition}
\begin{proof}
Note that since $L'$ is a finite extension over $L=\Q_p^{\un}\Q_p^{\rm Kum}$, we have an isomorphism of perfectoid fields $\widehat{L}^{\prime, \flat}\cong \overline{\F}_p((\pi^{1/p^{\infty}}))$ (namely, it is a finite separable extension of $\widehat{L}^{\flat}=\overline{\F}_p((\varpi^{1/p^{\infty}}))$ with $\varpi=(p^{1/p^{n}})_n$, and any such extension arises from a finite separable extension of $\F_q((\varpi))$ for some $q=p^k$ which is itself isomorphic to $\F_{q'}((\pi))$ for suitable $\pi$ being a local field).

  The suaveness of $\Phi$ can be checked after pullback along $\GSpec(\Q_p^\un)\to \GSpec(\Q_p^\un)/\varphi^\Z\times_{\GSpec(\Q_p)/\varphi^\Z} B\mathbb{V}(-1)$, which is a $\mathbb{V}(-1)$-torsor $\mathcal{T}$ over $\Yc_{L'}^\dR$.
  In particular, the morphism $\mathcal{T}\to \Yc_{L'}^\dR$ is cohomologically smooth.
  Moreover, $\Yc_{L'}^\dR$ identifies with the de Rham stack of a pre-perfectoid punctured open unit disc over $\Q_p^\un$ (\cref{explicit-description-drff-stack-of-qp}), and hence the map to $\GSpec(\Q_p^\un)$ is suave by \cref{sec:geom-prop-analyt-1-cohom-smoothness-of-perfectoid-unit-disc}.
  
  To prove that the natural map $\Phi_{\natural} 1\to 1$ is an equivalence, by suave base change (\cite[Remark 4.5.15(1)]{heyer20246functorformalismssmoothrepresentations}) along the following cartesian diagram
 $$
 \begin{tikzcd}
 \mathcal{T} \ar[r, "\Phi'"]\ar[d] & \GSpec(\Q_p^\un) \ar[d] \\
 \Yc_{L'}^\dR \ar[r, "\Phi"] & \GSpec(\Q_p^\un)/\varphi^\Z\times_{\GSpec(\Q_p)/\varphi}B\mathbb{V}(-1)
\end{tikzcd}
 $$
 it suffices to prove that the natural map $\Phi'_{\natural} 1\to 1$ is an equivalence.

 After picking an isomorphism $\widehat{L'}^\flat\cong \overline{\F}_p((\pi^{1/p^\infty}))$, we can write $\mathcal{Y}^\dR_{L'}\cong Z:=\varprojlim_{x\mapsto x^p}(\mathring{\mathbb{D}}^{\times}_{\Q_p^\un})^\dR$ with chosen coordinate $x=[\pi]$, cf.\ \Cref{explicit-description-drff-stack-of-qp}(1).
 We know that $\mathcal{T}$ has to be a non-trivial torsor over $Z$ (e.g., as its pullback to $\mathcal{Y}_{\C_p}^\dR$ is non-trivial, and we calculated all possible extensions in \cref{choices}).
 We can further note (using the explicit \cref{ConstructionMonodromy}) that $\mathcal{T}$ descents to a $\mathbb{V}(-1)$-torsor $\mathcal{P}$ over $X^\dR$ for $X:=\mathring{\mathbb{D}}^{\times}_{\Q_p}$, and that it suffices to show that the natural map $g_\natural 1\to 1$ is an isomorphism, where $g\colon \mathcal{P}\to \GSpec(\Q_p)$ (indeed, this result will then also hold for $\mathcal{T}$ by passing to the colimit over $x\mapsto x^p$, and after base change to $\Q_p^\un$).

We claim that we can reduce to showing that the natural map $1\to g_*1$ is an isomorphism. Indeed, $g_\natural 1=g_!(\omega_f)$ is a basic nuclear object in $\ob{D}(\Q_{p, \solid})$ by \Cref{PropDescentNuclearomega1CompactdR}(1) and we have $g_{\ast}1= \underline{\ob{Hom}}(g_\natural 1, 1)$, so assuming these properties the claim will follow if we know that the functor $\underline{\ob{Hom}}(-,1)$ is conservative on basic nuclear objects in $\ob{D}(\Q_{p,\solid})$. But any basic nuclear object is represented by a complex of DNF spaces by \cref{PropBasicNuclearQp}.
 We recall here that the derived dual of a DNF is a Fr\'echet space in degree $0$ (using \cite[A.18]{bosco2021p}).
 Then we can use the exact (underived) duality result \cite[Theorem 3.40]{jacinto2021solid} between dual nuclear Fr\'echet and nuclear Fr\'echet spaces  on the terms of the complexes, to see that if the dual complex of Fr\'echet spaces is acyclic, then the original complex of dual nuclear Fr\'echet spaces is acyclic. This proves the claim.

 The fact that $1\to g_\ast(1)$ is an isomorphism follows from \cref{nasty} below. This finishes the proof.
\end{proof}

\begin{lemma}\label{nasty} Consider $X=\mathring{\DD}^{\times}_{\Q_p}$, the punctured open disc over $\Q_p$ with coordinate $x$ and let us write $\nabla$ for its  connection.   Let $h\colon\mathcal{P}\to X^\dR$ be a non-split $\mathbb{V}(-1)$-torsor.
  \begin{enumerate}
  \item For $0< \epsilon <1$, let $X_\epsilon$ be the overconvergent annulus of radius $[\epsilon,1-\epsilon]$.
    There exists some $c\in \Q_p^\times$ such that $h_\ast(1)$ identifies with the following $\mathcal{O}(X)$-algebra $(\mathcal{F},\nabla_c)$ with   flat connection\footnote{We use the terminology ``module with flat connection'' here as a more intuitive for a $\mathbb{G}_a^\dagger$-equivariant sheaf on $X=\mathring{\mathbb{D}}^{\times}$.} satisfying the Leibniz rule\footnote{More precisely, the pullback of $h_\ast(1)$ to $X$ is the quasi-coherent sheaf on $X$ associated with the $\mathcal{O}(X)$-module $\mathcal{F}$.}:
    $$
    \mathcal F:= \lim_{\epsilon \to 0} \{\sum_{n\ge 0}a_n\cdot (\ell_x)^n \text{ with } a_n\in \mathcal O(X_\epsilon)\ |\ \textit{ for all } r>0, \|a_n\|r^n \to 0 \text{ as } n \to \infty\}\cong\mathcal{O}(X\times_{\Q_p}\A^{1,\an}_{\Q_p})
    $$
    and $\nabla_c(\ell_x)=c \frac{dx}{x}$, and $\nabla_{c|\mathcal{O}(X)}=x \partial_x(-) \frac{dx}{x}$.
    \item The natural map $\Q_p\to R\Gamma(\mathcal{P},\mathcal{O})$ is an isomorphism.
  \end{enumerate}
 \end{lemma}
 \begin{proof}
   We note that $\mathcal{P}\times_{X^\dR}X$ is a trivial $\mathbb{V}(-1)$-torsor as $X$ is Stein.
   Choosing an isomorphism with $X\times_{\GSpec(\Q_p)}\mathbb{A}^{1,\an}_{\Q_p}$, and denoting by $\ell_x$ the coordinate on $\mathbb{A}^{1,\an}_{\Q_p}$, the first assertion follows by using the explicit construction of a non-split $\mathbb{V}(-1)$-torsor as provided by \cref{ConstructionMonodromy}.

   Given the first statement, the second amounts to showing that 
   \[
     \Q_p\overset{\sim}{\to} R\Gamma_{\dR}(X, (\mathcal F, \nabla_c)),
   \]
   where the right-hand side denotes the de Rham complex of $(\mathcal{F},\nabla_c)$.
   As $X$ and $\mathcal{P}$ are Stein spaces and $X$ is of dimension $1$, the right-hand side is calculated by the complex
   \[
     \mathcal{F}\overset{\nabla_c}{\longrightarrow}\mathcal{F}\cdot dx.
   \]
   Thus, we have to show that $\ob{ker} \nabla_c\cong \Q_p$ and $\ob{coker} \nabla_c = 0$ on underlying sets (indeed, $\mathcal{F}$ is a Fr\'echet space and the exactness of a complex of Fr\'echet spaces can be checked on the underlying sets thanks to the open mapping theorem).
  
   We may assume that $c=1$, by replacing $\ell_x$ by $\frac{1}{c}\ell_x$.
   Thus, we set $\nabla:=\nabla_1$.
   Let $a=\sum_{m\geq 0}a_m\ell_x^m\in \mathcal{F}$ with $a_m\in \mathcal{O}(X)$.
   Then we can write
   \[
    a_m=\sum_{n\in\Z}a_{n,m}x^n 
  \]
  with $a_{n,m}\in \Q_p$, so that
  \[
    a=\sum\limits_{n\in \Z, m\geq 0}a_{n,m}x^n\ell_x^m.
  \]
  We set
  \[
    \Phi(a):=[\sum\limits_{n\in \Z, n\neq 0, m\geq 0}a_{n,m} x^n(\sum\limits_{k=0}^m(-1)^k\frac{k!\binom{m}{k}}{n^{k+1}}\ell_x^{m-k})]+\sum\limits_{m>0} \frac{a_{0,m-1}}{m}\ell_x^m.
  \]
  This expression defines (a priori) only an element in $\prod\limits_{n\in \Z,m\geq 0} \Q_p x^n\ell_x^m$.
  Then we formally can check that
  \[
    x\nabla(\Phi(a))=a dx.
  \]
  Indeed, this can be checked on each summand according to powers of $x$: write $\Phi(a)=\sum_{n\in \Z} c_n $, with $c_n \in x^n \prod_{m\geq 0} \Q_p \ell_x^m$ for each $n$. 
  If $n=0$, then $c_0=\sum\limits_{m>0} \frac{a_{0,m-1}}{m}\ell_x^m$ and thus
    $$x\partial_x(c_0)=\sum\limits_{m>0} a_{0,m-1} \ell_x^{m-1}= \sum\limits_{m\geq 0} a_{0,m} \ell_x^{m}.$$
     If $n\neq 0$, then $c_n=\sum\limits_{m\geq 0}a_{n,m}x^n\sum_{k=0}^m(-1)^k\frac{k!\binom{m}{k}}{n^{k+1}}\ell_x^{m-k}$. 
  We calculate 
  \begin{align*}
    x\partial_x(c_n)=& x\partial_x(\sum\limits_{m\geq 0}a_{n,m}x^n\sum_{k=0}^m(-1)^k\frac{k!\binom{m}{k}}{n^{k+1}}\ell_x^{m-k})\\
    =& \sum\limits_{m\geq 0}a_{n,m}nx^n(\sum_{k=0}^m(-1)^k\frac{k!\binom{m}{k}}{n^{k+1}}\ell_x^{m-k})+a_{n,m}x^n(\sum_{k=0}^{m-1}(-1)^k\frac{k!\binom{m}{k}}{n^{k+1}}(m-k)\ell_x^{m-1-k})\\
    =& \sum\limits_{m\geq 0}a_{n,m}nx^n(\frac{1}{n}\ell_x^m+\sum_{k=1}^m (-1)^k\frac{k!\binom{m}{k}}{n^{k+1}}\ell_x^{m-k})+a_{n,m}x^n(\sum_{k=0}^{m-1}(-1)^k\frac{k!\binom{m}{k}}{n^{k+1}}(m-k)\ell_x^{m-1-k})\\
    =& \sum\limits_{m\geq 0}a_{n,m}(x^n\ell_x^m+x^n\sum\limits_{k=0}^{m-1}((-1)^{k+1}n\frac{(k+1)!\binom{m}{k+1}}{n^{k+2}}+ (-1)^k\frac{k!\binom{m}{k}(m-k)}{n^{k+1}})\ell_{x}^{m-k-1})\\
    =& \sum\limits_{m\geq 0}a_{n,m}x^n\ell_x^m
  \end{align*}
  by using the formula $(k+1)!\binom{m}{k+1}=k!\binom{m}{k}(m-k)$.

  We now show that $\Phi(a)\in \mathcal{F}$.
  If this is done, then we can conclude that $\nabla$ is surjective as desired.
  We note that
  \begin{equation}
    \label{eq:3}
    \Phi(a)=\sum\limits_{n\in \Z \backslash\{0\},\ m\in \N} (\sum_{k\in \N}\frac{k!\binom{m+k}{m}}{n^{k+1}}a_{n,m+k})x^n\ell_x^m+\sum_{m>0}\frac{a_{0,m-1}}{m}\ell_x^{m}.
  \end{equation}
  The condition that a sum $b=\sum\limits_{n\in \Z, m\in \N}b_{n,m}x^n\ell_x^m$ lies in $\mathcal{F}$ is equivalent to the conditions
  \[
   |b_{n,m}|\gamma_1^{|n|}\gamma_2^m\to 0, 
 \]
 for $n,m\to \infty$ and all $\gamma_1<1$, $\gamma_2<\infty$, or $-n,m\to \infty$ and all $\gamma_1<\infty$, $\gamma_2<\infty$.
 
 We now check these conditions for $b:=\Phi(a)$.
 We note that in \cref{eq:3} the second summand converges as desired (because it is the integral of the analytic function $\sum_{m\geq 0}a_{0,m}\ell_x^m$ on $\mathbb{A}^1_{\Q_p}$).
 For the first summand, we note that
 \[
   |b_{n,m}|\gamma^{|n|}_{1}\gamma_2^m\leq \mathrm{sup}_{k\geq 0}\{\frac{|a_{n,m+k}|}{|n|^{k+1}}\gamma_1^{|n|}\gamma_2^m\}=\mathrm{sup}_{k\geq 0}\{|a_{n,m+k}|\gamma_1^{|n|}\gamma_2^{m+k}\frac{1}{|n|^{k+1}\gamma_2^k}\}.
 \]
 The cases $n\to \infty$ and $n\to -\infty$ are similar, so we assume that $n,m\to \infty$ (and thus $\gamma_1<1$).
 We note that $\frac{1}{\gamma_2^k}$ is bounded if $k\to \infty$ and $\gamma_2>1$.
 We write $\gamma_1=\gamma_1^\prime(1-\delta)$ for some $\delta>0$.
 Then $\frac{(1-\delta)^n}{|n|^k}\to 0$ because $|n|^{-k}$ grows like $p^{k\mathrm{log}_p(n)}$.
 Thus, (using that $|a_{n,m}|\gamma_1^n\gamma_2^m$ is bounded for $n,m\to \infty$ because $a\in \mathcal{F}$) we see that $\mathrm{sup}_{k\geq 0}\{|a_{n,m+k}|\gamma_1^n\gamma_2^{m+k}\frac{1}{|n|^{k+1}\gamma_2^k}\}$
 remains bounded as well.
 This implies that $|b_{n,m}|\gamma_1^n\gamma_2^m$ remains bounded for $n,m\to \infty$, and hence goes to zero by making $\gamma_1, \gamma_2$ slightly smaller.

 We now check that $\ker(\nabla_c)=\Q_p$.
 Thus we assume that $a=\sum_{n\in \Z,m\in \N}a_{n,m}x^n\ell_x^m$ satisfies $\nabla(a)=0$.
 We calculate
 \begin{align}
   x\partial_x(a)= & \sum_{n\in \Z,m\geq 1} a_{n,m}n x^n\ell_x^m+a_{n,m}mx^n\ell_x^{m-1} \\
   = & \sum_{n\in \Z,m\geq 0} (na_{n,m}+(m+1)a_{n,m+1})x^{n-1}\ell_x^{m}. 
 \end{align}
 Thus, $a_{n,m+1}=\frac{-n}{m+1}a_{n,m}$ for $n\in \Z, m\in \N$, and hence
 \[
   a_{n,m}=\frac{(-1)^mn^m}{m!}a_{n,0}.
 \]
 We can conclude
 \[
   a=\sum_{n\in \Z}\sum_{m\in \N} a_{n,m}x^n\ell_x^m=\sum_{n\in \Z}a_{n,0}\sum_{m\in \N}\frac{(-n\ell_x)^m}{m!}x^n=\sum_{n\in \Z}a_{n,0}\mathrm{exp}(-n\ell_x)x^n.
 \]
 Because $a\in \mathcal{F}\cong \mathcal{O}(X\times_{\GSpec(\Q_p)}\mathbb{A}^1_{\Q_p})\cong \mathcal{O}(X)\otimes_{\Q_{p,\solid}}\mathcal{O}(\A^1_{\Q_p})$ (using \cite[Corollary A.67.(i)]{bosco2021p}) we can conclude that
 \[
   a_{n,0}\exp(-n\ell_x)
 \]
 has to converge on $\A^{1}_{\Q_p}$ (with coordinate $\ell_x$).
 But this happens for $a_{n,0}\exp(-n\ell_x)$ if and only if $n=0$ or $a_{n,0}=0$.
 This implies that $a=a_{0,0}$ is constant as desired. 
 \end{proof}

\subsection{Tsuzuki's theorem}
\label{sec:tsuzukis-theorem}

With the results of \cref{sec:fully-faithf-psiast} we have shown the fully faithfulness in \cref{localmonodromy}, and in this section we want to prove essential surjectivity. Before entering into details of the theorem we sketch the idea of the proof.

Thanks to \cref{sec:fully-faithf-psiast-2-reductions-summary,} we know that the pullback along $L^{\HK}\to \Font_L$ is fully faithful, so we are only left to show essential surjectivity.
For that, by descent we can easily reduce to the case of $\C_p^{\HK}$, and write the Hyodo--Kato stack as a quotient of $\FF_{\C_p}$ by the overconvergent diagonal $\Delta^{\dagger}_{\C_p}$ of $\FF_{\C_p}\times_{\Q_p^{\un}/\varphi^\Z} \FF_{\C_p}$ (or, equivalently in $\FF_{\C_p}\times_{\Q_p}\FF_{\C_p}$).
This overconvergent diagonal can be written as a limit of Fargues--Fontaine $\FF_{U}$ where $U$ is an open neighbourhood of the diagonal of the fiber product $\Marc(\C_p)\times_{\Marc(\F_q)} \Marc(\C_p)$ for $q=p^k$ and $k\to \infty$.  Thus, using the uniqueness of  Harder--Narasimhan  filtrations on vector bundles on Fargues--Fontaine curves, we can descend the slope filtration from $\FF_{\C_p}$ to $\C_p^{\HK}$.
This reduces the problem of essential surjectivity to the semi-stable case.
Now, for $\mathcal{V}$ a semi-stable vector bundle on $\C_p^{\HK}$, its pullback to $\FF_{\C_p}$ is of the form $\mathcal{O}(\lambda)^{\oplus n}$ for some $\lambda \in \Q$ and $n\in \N$.
But then, the automorphism group of $\mathcal{O}(\lambda)^{\oplus n}$ is a locally light profinite group $A$, so that the descent datum is determined by an object in  the kernel of the morphism of groups $\underline{A}(\Delta^{\dagger}_{\C_p}) \to \underline{A}(\Marc(\C_p)))$.
This morphism is the same as 
\[
\varinjlim_{U} \underline{A}(U) \to A
\]
where $U$ runs through the open neighbourhoods of the diagonals of $\Marc(\C_p)\times_{\Marc(\F_q)} \Marc(\C_p)$ as above.
Then the essential surjectivity would hold if one can prove that there is a cofinal family of those $U$ such that they are connected, namely, in that situation $\underline{A}(U)=A$ and the descent datum will be trivial.

 However, this would prove that any semi-stable vector bundle on $\C_p^{\HK}$ is of the form $\mathcal{O}(\lambda)^{\oplus n}$, and by the fully faithfulness of \cref{sec:fully-faithf-psiast-2-reductions-summary},   we would have that all semi-stable $\varphi$-modules over $\Q_p^{\un}$  are of the form $D_{\lambda}^{\oplus n}$ which is false: $\varphi$-vector bundles on $\Q_p^{\un}$  are the colimit of $\varphi$-vector bundles on finite unramified  extensions $\Q_q$, and a $\varphi$-vector bundle $V$ on $\Q_q$  of slope zero is the same as an \'etale $\Q_p$-local system of $\Spec(\F_q)$. 
  Saying that $V$  becomes trivial on $\Q_p^{\un}$ as a $\varphi$-module is the same as asking that it is trivial after some  finite extension $\Q_q\to \Q_{q'}$, but this would mean that the morphism of groups $\Gal_{\F_q}\to \GL_n(\Q_p)$ has   finite image, and there are clearly examples of such morphisms which have infinite image. 
   Note however, that by Kottwitz classification of $\varphi$-modules on $\breve{\Q}_p$,  any $\varphi$-vector bundle over $\breve{\Q}_p$ of slope $0$ is actually trivial.

    The previous discussion suggests that we need to slightly modify the strategy as follows: we prove the analogue of \cref{localmonodromy} after extending coefficients to $\GSpec(\breve{\Q}_p)\cong \FF_{\overline{\F}_p}$. 
   More precisely, let 
\[
\C_p^{\HK/\FF_{\overline{\F}_p}}= \C_p^{\HK}\times_{\FF_{\overline{\F}_p}^{\dR}} \FF_{\overline{\F}_p}
\]
and consider the morphism of stacks
\[
g\colon \C_{p}^{\HK/\FF_{\overline{\F}_p}}\to B_{\Q_p^{\un}}\mathbb{V}(-1)\times_{\GSpec(\Q_p^\un)/\varphi^\Z}\GSpec(\breve{\Q}_p)/\varphi^{\Z}.
\] 
Applying the same strategy as before, we will able to show that pullback along $g$ gives rise to an equivalence of perfect modules; this time we encounter that $ \C_{p}^{\HK/\FF_{\overline{\F}_p}}$ is the quotient of $\FF_{\C_p}$ by the overconvergent diagonal $\Delta_{\C_p/\breve{\Q}_p}$ in $\FF_{\C_p}\times_{\FF_{\overline{\F}_p}} \FF_{\C_p}$.
  This time, this overconvergent diagonal is written as a colimit of  Fargues--Fontaine curves $\FF_{U}$ with $\Marc(\C_p)\subset U \subset \Marc(\C_p)\times_{\Marc(\overline{\F}_p)}\Marc(\C_p)$.
  We do not know whether we can find such neighbourhoods  to be connected, but instead we prove a weaker connectedness statement that will suffice for proving the claim.
 Finally, we are able to descent back from $\breve{\Q}_p$ to $\mathbb{Q}_p^{\un}$ finishing the proof of \Cref{localmonodromy}.

 In order to finish the proof of \cref{localmonodromy} we need a version over $\breve{\Q}_p$.

\begin{proposition}[$p$-adic monodromy theorem: $\breve{\Q}_p$-version]\label{PropTsuzukiBreve}
  Let $L/\breve{\Q}_p$ be an algebraic extension and set $\breve{\Font}_{L}:=B\mathbb{V}(-1)\times_{\GSpec(\Q_p)/\varphi^{\Z}}\GSpec (\breve{\Q}_p)/(\varphi^{\Z}\times \Gal_{L}^{\sm}).$
  Consider the natural morphism of Gelfand stacks
\[
\Psi_L\colon L^{\HK/ \FF_{\overline{\F}_p}} \to \breve{\Font}_L.
\]
Then, the $*$-pullback defines an equivalence of categories
\[
\Psi_L^*\colon \ob{Perf}(\breve{\Font}_L) \xrightarrow{\sim} \ob{Perf}(L^{\HK/ \FF_{\overline{\F}_p}}).
\]
\end{proposition}
\begin{proof}
By the same arguments of \cref{sec:fully-faithf-psiast-1-reduction-to-smaller-extensions} and  \cref{sec:fully-faithf-psiast-1-descent}, the following are equivalent: 

\begin{enumerate}[(i)]

\item There is an algebraic extension $L/\breve{\Q}_p$ such that the functor $\Psi_L^*$ is fully faithful (resp.\ an equivalence) on perfect modules.

\item For all algebraic extension $L/\breve{\Q}_p$ the functor $\Psi^*_L$ is fully faithful (resp.\ an equivalence) on perfect modules.

\end{enumerate}

Therefore, in order to prove fully faithfulness we can work with $L/\Q^{\Kum}\breve{\Q}_p$ a finite extension over the Kummer extension of $\breve{\Q}_p$, in which case the same argument of \cref{sec:fully-faithf-psiast-1-getting-rid-of-galois-group} reduces to proving that the pullback along the map 
\[
L^{\HK/\FF_{\overline{\F}_p}}\to B_{\FF_{\overline{\F}_p}} \mathbb{V}(-1)
\]
is fully faithful.
But then the tilt $\widehat{L}^{\flat}$ is isomorphic to $\overline{\F}_p((\pi^{1/p^{\infty}}))$ and $L^{\HK/\FF_{\overline{\F}_p}}$ is isomorphic to a punctured open unit disc over $\breve{\Q}_p$.
We deduce fully faithfullness  by base changing \cref{sec:fully-faithf-psiast-1-case-for-f} along $\GSpec(\breve{\Q}_p)/\varphi^\Z\to \GSpec(\Q^\un_p)/\varphi^\Z$. 

From the previous we deduce that $\Psi^*_{L}$ is fully faithful for all algebraic extension $L/\breve{\Q}_p$.
We shall prove that the functor is essentially surjective.
For that, we reduce to the case of $\widehat{L}=\C_p$ by Galois descent.  
Thanks to \cref{PropPerfectBerkovichSpaces} we can reduce the essential surjectivity to the case of vector bundles.
Let $\mathcal{V}$ be a vector bundle on $\C_p^{\HK/\FF_{\overline{\F}_p}}$.
 By \cref{PropepiFF} the morphism of Gelfand stacks 
\[
\FF_{\C_p}  \to \C_p^{\HK/\FF_{\overline{\F}_p}} 
\] is an epimorphism. 
Thus, we can write  $\C_p^{\HK/\FF_{\overline{\F}_p}}$ as the quotient of $\FF_{\C_p}$ by its overconvergent diagonal $\Delta^{\dagger}_{\C_p/\breve{\Q}_p}$ over $\FF_{\overline{\F}_p}= (\GSpec \breve{\Q}_p)/\varphi^\Z$. 
Since the functor $X\mapsto \FF_{X}$ preserves fiber products of arc-stacks, we see that  
\[
\Delta^{\dagger}_{\C_p/\breve{\Q}_p}=\varprojlim_{ \Delta \Marc(\C_p) \subset  U\subset \Marc(\C_p)\times_{\Marc(\overline{\F}_p)} \Marc(\C_p)} \FF_{U}.
\]
The pullback of $\mathcal{V}$ to $\FF_{\C_p}$, denoted in the same way to lighten notation, has a Harder--Narasimhan filtration $\Fil^{\bullet} \mathcal{V}$. We want to show that the latter descends to a Harder--Narasimhan filtration on $\FF_{\C_p}^{\HK/\FF_{\overline{\F}_p}}$.
For this, the descent datum along the two projections $\pi_i:\Delta^{\dagger}_{\C_p/\breve{\Q}_p}\to \FF_{\C_p}$, for $i=1, 2$, extends to an isomorphism  on $\FF_{U}$ for some compact neighbourhood of the diagonal $ \Delta\Marc(\C_p)\subset U \subset \Marc(\C_p)\times_{\Marc(\overline{\F}_p)} \Marc(\C_p)$.
The pullbacks $\pi_i^*\mathcal{V}$, for $i=1, 2$, admit global Harder--Narasimhan filtrations (being the restriction of the pullbacks along the two projection maps $\FF_{\Marc(\C_p)\times_{\Marc(\overline{\F}_p)} \Marc(\C_p)}\to \FF_{\C_p}$), and, by uniqueness, we have that $\Fil^{\bullet} \pi_1^* \mathcal{V}= \Fil^{\bullet} \pi_2^* \mathcal{V}$.
Therefore,  $\Fil^{\bullet} \mathcal{V}$ is stable under the descent datum of $\mathcal{V}$ and it descends to a Harder--Narasimhan filtration on $\C_p^{\HK/\FF_{\overline{\F}_p}}$, as desired.

Thus, in order to prove essential surjectivity of the functor $\Psi^*_{\C_p}$, we can assume without loss of generality that $\mathcal{V}$ is a semi-stable vector bundle.
Then, to show that $\mathcal{V}$ lies in the essential image of $\Psi^*_L$, we can prove (via approximation) the statement for a semi-stable vector bundle $\mathcal{V}$ over $L^{\HK/\FF_{\overline{\F}_p}}$ where $L/\Q_p^{\Kum}\breve{\Q}_p$ is a finite extension.
Suppose that at geometric points $\mathcal{V}$ arises from the isocrystal $D_{-\lambda}^{\oplus n}$ for $\lambda\in \Q$ and $n\in \N$.
By  \cref{PropepiFF}  the morphism $\FF_{L}\to L^{\HK/\FF_{\overline{\F}_p}}$ is still an epimorphism and, by passing to the associated pro-\'etale torsor of trivializations of $\mathcal{V}$, the restriction of $\mathcal{V}$ to $\FF_{L}$ gives rise to a morphism of arc-stacks
 \[
 \Marc(L)\to B \underline{\GL_n(\mathbb{H}_{\lambda})}
 \]
 where $\mathbb{H}_{\lambda}$ is the $p$-adic Lie group of automorphisms of $D_{-\lambda}$.
 Therefore, there exists a finite tamely ramified extension $K/L$ and a pro-$p$-extension $L'/K$ with Galois group a compact $p$-adic Lie group such that $\mathcal{V}$  is isomorphic to $\mathcal{O}(\lambda)^{\oplus n}$ over $L'$. Replacing $L$ by $L'$, we are therefore, by descent, reduced to proving the following claim:
 \begin{claim}
   Let $L/\Q_p^{\Kum}\breve{\Q}_p$ be  an algebraic Galois extension whose Galois group is a compact $p$-adic Lie group.
   Let $\mathcal{V}$ be a vector bundle on $L^{\HK/\FF_{\overline{\F}_p}}$ such that the pullback to $\FF_L$ is isomorphic to $\mathcal{O}(\lambda)^{\oplus n}$ for some $\lambda\in \Q$ and $n\in \N$.
   Then $\mathcal{V}$ is isomorphic to $\mathcal{O}(\lambda)^{\oplus n}$ on $L^{\HK/\FF_{\overline{\F}_p}}$.
   In particular, it belongs to the essential image of $\Psi^*_{L}$. 
 \end{claim}
\begin{proof}[Proof of the claim:]
Let $\Delta^{\dagger}_{L/\breve{\Q}_p}$ be the overconvergent diagonal of the morphism $\FF_{L}\to \FF_{\overline{\F}_p}$. It can be written as the limit
\[
\Delta^{\dagger}_{L/\breve{\Q}_p}= \varprojlim_{L\subset U\subset \Marc(L)\times_{\Marc(\overline{\F}_p)} \Marc(L)} \FF_U
\] where $U$ runs over affinoid neighbourhoods of the diagonal. Thus, after fixing an isomorphism between $\mathcal{V}$ and $\mathcal{O}(\lambda)^{\oplus n}$ on $\FF_{L}$, the descent datum of $\mathcal{V}$ is given by an element in
\[
\ker(\underline{A}(\Delta^{\dagger}_{L/\breve{\Q}_p})\to \underline{A}(\Marc(L)) ) = \ker(\varinjlim_{L\subset U\subset \Marc(L)\times_{\Marc(\overline{\F}_p)} \Marc(L)} \underline{A}(U)\to A)
\]
with $A=\GL_n(\mathbb{H}_{\lambda})$.
The claim will follow if there is a cofinal system of neighbourhoods that are connected.
We have an isomorphism of tilts $L^{\flat}\cong \overline{\F}_p((\pi^{1/p^{\infty}}))$ by \cite[2.1.3. Th\'eor\`eme (ii)]{wintenberger1983corps} and \cite[4.3.4. Corollaire]{wintenberger1983corps} (note that these apply because $L/\breve{\Q}_p$ is Galois with Galois group a compact $p$-adic Lie group, and hence strictly APF, \cite[1.2.2. Exemples]{wintenberger1983corps}), and thus we get that
\[
\Marc(L)\times_{\Marc(\overline{\F}_p)} \Marc(L) \cong \mathring{\mathbb{D}}^{\times,\perf,\diamond}_L
\]
is isomorphic to the perfectoid punctured open unit disc over $L$.
Moreover, the diagonal $\Marc(L)$ defines an $L$-rational point of the punctured open unit disc over $L$, and because quasi-compact rigid analytic varieties over $L$ have only finitely many connected components we see that such a cofinal system of connected neighborhoods exists.
\end{proof}
 This finishes the proof of \cref{PropTsuzukiBreve}.
\end{proof}

We deduce our interpretation of Tsuzuki's theorem over $\Q_p^{\un}$.

\begin{theorem}
  \label{sec:tsuzukis-theorem-1-essential-surjectivity}
  Let $L\subseteq \overline{\Q}_p$ be an algebraic extension of $\Q_p$.   Then, the functor $$\Psi^\ast_{L,\varpi}\colon \Perf(\Font_L)\to \Perf(L^\HK)$$ is essentially surjective (and thus an equivalence).
\end{theorem}
\begin{proof}
  By \cref{sec:fully-faithf-psiast-1-reduction-to-smaller-extensions} and \cref{sec:fully-faithf-psiast-1-descent} we can prove essential surjectivity for any intermediate algebraic extension $\Q_p\subset L\subset \overline{\Q}_p$. Let us take $L=\Q_p^{\Kum}\Q_p^{\un}$ to be the Kummer extension of $\Q_p$ over  $\Q_p^{\un}$.
  Let $\mathcal{V}$ be a vector bundle on $L^{\HK}$. Let $L'/L$ be a finite Galois extension and consider the morphism
\[
\Psi_{L'/L,\varpi}\colon L^{\HK}\to \Font_{L'/L}
\]
as in \cref{sec:fully-faithf-psiast-1-getting-rid-of-galois-group}.
Thanks to \cref{sec:fully-faithf-psiast-1-case-for-f} the map $\Psi_{L'/L,\varpi}$ is suave and $\Psi_{L'/L,\varpi}^*$ is fully faithful. Thus,  by \cref{sec:fully-faithf-psiast-1-case-for-f}, to prove essential surjectivity of $\Psi^*$ on perfect modules it suffices to show that there is a finite Galois extension $L'/L$ such that the natural map
\[
 \mathcal{V}\to  \Psi_{L'/L,\varpi}^*\Psi_{L'/L,\varpi,\natural} \mathcal{V}
\]
is an equivalence.
This can be checked after base change along the epimorphism of Gelfand stacks $\breve{\Font}_{L'/L} \to\Font_{L'/L}$  obtained by base changing from $\Q_p^{\un}$ to $\breve{\Q}_p$. In this case we have a cartesian diagram
\[
\begin{tikzcd}
L^{\HK/\FF_{\mathbb{\F}_p}} \ar[r,"g"] \ar[d,"\widetilde{f}"] &  \breve{\Font}_{L'/L} \ar[d,"f"]  \\
L^{\HK} \ar[r,"\Psi_{L'/L,\varpi}"] & \Font_{L'/L},
\end{tikzcd}
\]
and  by proper base change one reduces to proving that there exists a finite extension $L'/L$ such that the natural map $\widetilde{f}^* \mathcal{V}\to  g^*g_{\natural}  \widetilde{f}^* \mathcal{V}$ is an isomorphism. This follows from \cref{sec:fully-faithf-psiast-1-getting-rid-of-galois-group} and \cref{PropTsuzukiBreve}.
\end{proof}

\begin{remark}
  \label{sec:tsuzukis-theorem-1-description-of-varphi-n-modules-over-q-p-un}
  As explained at that beginning of this section, the categories of vector bundles on $\GSpec(\breve{\Q}_p)/\varphi^\Z$ and $\GSpec(\Q_p^\un)/\varphi^\Z$ are \emph{not} equivalent.
  Let $G:=\varprojlim_{H\subseteq \Gal_{\F_p} \textrm{ open }}\GSpec(C(H,\Q_p))$ for $C(H,\Q_p)$ the Hopf algebra of continuous functions $H\to \Q_p$.
  Then $G$ is a group object in Gelfand stacks, acting on $\GSpec(\breve{\Q}_p)/\varphi^\Z$ with quotient $\GSpec(\Q_p^\un)/\varphi^\Z$.
\end{remark}

 As a consequence of \cref{localmonodromy}, we obtain the following.

 \begin{corollary}\label{atclassicalpoints}
   Let $X$ be a partially proper rigid space over $\Q_p$.
   Let $\mathcal{V}$ be a vector bundle on $X^\HK$.
   For any classical point $x\in X$, there exists an open neighborhood $U\subset X$ of $x$ and a finite extension $L/\Q_p$ such that the pullback of $\mathcal{V}$ to $U_L^\HK$ is unipotent (with respect to the residue field of $L$).
 \end{corollary}
 \begin{proof}
 First, we show that given $Y$ a partially proper rigid space over $\C_p$, and $\mathcal{V}$ a vector bundle on $Y^{\HK}$, then, for any point $y$ of $Y$ whose completed residue field $\widehat{k(y)}$ is isomorphic to $\C_p$, there exists an open neighborhood $U\subset X$ of $y$ such that the pullback of $\mathcal{V}$ to $U^\HK$ is unipotent (with respect to $\overline{\mathbb{F}}_p$).
 This follows from \cref{localmonodromy} via an approximation argument.
 More precisely, it suffices to show that
 $$\VB(\Marc(\widehat{k(y)})^{\HK})= \varinjlim_{U \ni y}\VB(U^{\HK}).$$
 Moreover, thanks to \cref{PropPerfectBerkovichSpaces}, it suffices to prove a similar statement for perfect modules. Restricting to an affinoid neighborhood of $y$, and passing to a (light) profinite torsor over it, we can reduce to the affinoid perfectoid case (cf. the proof of \cref{sec:fully-faithf-psiast-1-getting-rid-of-galois-group}). Then the claim follows from \cref{LemLimitsBerkovichSpaces}, recalling that the Fargues--Fontaine curve attached to an affinoid perfectoid $\Marc(A)$ over $\Q_p$ is a Berkovich space, and observing that,  considering the map $\Marc(A)^\HK\to \mathcal{M}(A)_{\Betti}$, each closed subset of $|\Marc(A)^{\HK}|$ has a fundamental system of neighborhoods of the form $|U^\HK|$.

 Then, the statement follows from what we have just shown, again via approximation, using the classicality of the point $x\in X$.
 \end{proof}

 \begin{remark}
   \label{sec:tsuzukis-theorem-1-remark-on-general-c}
   Let $C/\Q_p$ be an algebraically closed field of finite transcendence degree $d\geq 1$ (up to completion), so that $C$ is a qfd Banach algebra over $\Q_p$.
   One might wonder if any vector bundle on $C^{\HK}$ is unipotent (with respect to $\overline{\mathbb{F}}_p$), or even if it is possible to classify vector bundles on $C^{\HK}$ by a category of $(\varphi,N_0,\ldots, N_d)$-modules over $\Q_p^\un$.
   This is however too optimistic: each vector bundle on $C^{\HK}$ spreads out to a vector bundle on $X^{\HK}$ for a smooth partially proper rigid analytic variety $X$ over $\Q_p$ with a morphism $x\colon \Spa(C)\to X$.
   As a consequence, a vector bundle on $C^{\HK}$ does \emph{not} necessarily admit a Harder--Narasimhan filtration (e.g., if the Harder--Narasimhan polygon for the pullback to $\FF_X$ jumps at $x$) even though its pullback to $\FF_C$ does.
   Looking at the argument in \cref{PropTsuzukiBreve}, we see that $\FF_C\to C^{\HK}$ is therefore not surjective.
   Still there is the chance that a version of Tsuzuki's theorem can hold for $C$:
   each vector bundle $\mathcal{E}$ on $C^{\HK}$, which spreads out to a \emph{semi-stable} vector bundle $\mathcal{E}'$ on $X^{\HK}$ for a smooth partially proper rigid analytic variety $X$ over $\Q_p$ with a morphism $x\colon \Spa(C)\to X$, is pulled back along the morphism $C^{\HK}\to \GSpec(\Q_p^\un)/\varphi^\Z$.
 \end{remark}

 \begin{remark}[de Rham implies potentially semi-stable]\label{dR=pst} \
 \begin{enumerate}
  \item As shown by Berger, \cite{Berger_monodromie}, the local $p$-adic monodromy theorem implies that any de Rham $p$-adic Galois represention is automatically potentially log-crystalline (i.e.\ potentially semi-stable).
 This was later reproved by Colmez, \cite{Colmez_deRham}, and Fargues--Fontaine, \cite{Farguesf}, using different methods.
 Rephrasing Berger's proof in terms of Hyodo--Kato stacks, this result can be deduced as follows from \cref{localmonodromy}.
 Using Sen theory, one can show that there is an equivalence of categories

 \begin{equation}\label{VBdR}
  \VB^{\dR}(\FF_{\Q_p})\cong \VB(\Q_p^{\HK,+})
 \end{equation}
 where the LHS denotes the category of \textit{de Rham} $\Gal_{\Q_p}$-equivariant vector bundles on $\FF_{\C_p}$, \cite[Définition 10.4.5]{Farguesf}, and $\Q_p^{\HK,+}$ denotes the stack defined as the pushout of the diagram $$\Q_p^{\HK}\leftarrow \Q_p^{\dR}\to \Q_p^{\dR, +}$$ (here, we write $\Q_p^{\dR, +}$ for the Hodge filtered de Rham stack of $\Q_p$).
 By \cref{localmonodromy} the category $\VB(\Q_p^{\HK})$ is equivalent to the category of $(\varphi, N, \Gal_{\Q_p})$-modules over $\Q_p^{\un}$; in particular, the category $\VB(\Q_p^{\HK,+})$ is equivalent to the category of \textit{filtered}  $(\varphi, N, \Gal_{\Q_p})$-modules over $\Q_p^{\un}$.
 On the other hand, by \cite[Proposition 10.6.7]{Farguesf}, the latter category is equivalent to the category $\VB^{\mathrm{plog}}(\FF_{\Q_p})$ of \textit{potentially log-crystalline} $\Gal_{\Q_p}$-equivariant vector bundles on $\FF_{\C_p}$.
 Thus, putting everything together, we deduce from \cref{VBdR} that we have an equivalence of categories
 \begin{equation}\label{VBdR2}
  \VB^{\dR}(\FF_{\Q_p})\cong \VB^{\mathrm{plog}}(\FF_{\Q_p}),
 \end{equation}
 which is the generalization of Berger's theorem proven by Fargues--Fontaine, \cite[Théorème 10.6.10]{Farguesf}.
 An advantage of this strategy is that it is amenable to generalizations to the relative case.
 In fact, we expect that the equivalence \cref{VBdR} can be extended replacing $\Q_p$ with a partially proper smooth rigid space $X$ over $\Q_p$, and that one can deduce from \cref{atclassicalpoints}, Shimizu's theorem, \cite[Theorem 1.1]{Shimizu_monodromy}.

 \item We recall that it is a major open problem in $p$-adic Hodge theory whether, given a de Rham pro-étale $\Z_p$-local system $\mathbb{L}$ on a smooth rigid space $X$ over $\Q_p$, there exists a finite étale extension $Y\to X$ such that the pullback of $\mathbb{L}$ to $Y$ is semi-stable at each classical point (see e.g.  \cite[Remark 1.4]{LiuZhu_RiemannHilbert}).
 This is known to be true only in very special cases; e.g. it is known for local systems with a single (constant) Hodge--Tate weight, \cite[Theorem 1.6]{Shimizu_constancy}.
 In view of the previous remark, we can translate this question in terms of Hyodo--Kato stacks, as follows.
 Let $X$ be a partially proper rigid space over $\Q_p$.
 Is it true that, given a vector bundle $\mathcal V$ on $X^{\HK}$, there exists a finite étale extension $Y\to X$ such that the pullback of $\mathcal{V}$ to $Y^{\HK}$ is unipotent at each classical point of $Y$?
 \end{enumerate}
 \end{remark}

\clearpage
\appendix
\section{Complements on solid functional analysis}\label{appendix:solid-functional-analysis}
Building on results of Clausen--Scholze, we collect here  various useful statements in solid functional analysis over $\Q_{p, \solid}$, revolving around the notion of nuclearity and $\omega_1$-compactness.

As in the main body of the paper, we work in the light solid setup, and all rings are animated. Given a solid $\Q_{p, \solid}$-algebra $A$, we implicitly endow it with the analytic ring structure induced from $\Q_{p, \solid}$ and denote by $\ob{D}(A)$ the derived $\infty$-category of complete $A$-modules. \medskip

We begin by recalling some terminology.

\begin{definition}\label{defSmith}
 A  solid $\Q_p$-module is called a \textit{Smith space} if it is isomorphic to $\Q_p\otimes_{\Z_p} \prod_I \Z_p$, for some set $I$. We call a Smith space \textit{light}, if $I$ can be chosen to be countable.
\end{definition}

We recall that a solid $\Q_p$-module is a light Smith space if and only if it is isomorphic to $\Q_{p, \solid}[S]$ for some light profinite set $S$.

\begin{lemma}\label{lemmaSmith}
The class of Smith spaces, resp.\ light Smith spaces, is stable under extensions, closed subobjects, and quotient by closed subobjects.
\end{lemma}
\begin{proof}
 See e.g. \cite[Proposition 3.9]{jacinto2021solid}, \cite[Lemma A.26]{bosco2021p}.
\end{proof}

\begin{lemma}\label{LemmaPermanenceOmega1Compact}
Let $A$ be a solid $\Q_p$-algebra whose underlying solid  $\Q_p$-module is $\omega_1$-compact. Let $M\in \ob{D}(A)$, then the following are equivalent:
\begin{enumerate}
\item $M$ is an $\omega_1$-compact $A$-module.

\item $M$ is an $\omega_1$-compact solid $\Q_p$-module.

\item $\pi_k(M)$ is a static $\omega_1$-compact $\pi_0(A)$-module  (in the abelian category of solid $\pi_0(A)$-modules) for all $k\in \Z$.

\end{enumerate}
\end{lemma}
\begin{proof}
  Let us first prove that (1) $\Leftrightarrow$ (2). Suppose that $M$ is $\omega_1$-compact as solid $\Q_p$-module, then $M$ is the geometric realization of the $\omega_1$-compact objects $A^{\bullet+1}\otimes_{\Q_{p,\solid}} M$ and so $\omega_1$-compact as $A$-module, since the compact objects in $\ob{D}(\Q_{p,\solid})$ are stable under the solid tensor product.
  Conversely, let $M$ be an $\omega_1$-compact $A$-module, then if $N\in \ob{D}(\Q_{p,\solid})$ we have that 
\[
\iHom_{\Q_p}(M,N)= \iHom_{A}(M,\iHom_{\Q_p}(A,N)).
\]
Hence, as $A$ is $\omega_1$-compact as a solid $\Q_p$-module, the functor $\iHom_{\Q_p}(M,-)$ commutes with $\omega_1$-filtered colimits making $M$ an $\omega_1$-compact solid $\Q_p$-module.

Next, for proving the equivalence with (3), by the previous equivalences (applied to $A$ and $\pi_0(A)$) we can assume without loss of generality that $A=\Q_p$.
Suppose that $M$ is an $\omega_1$-compact solid $\Q_p$-module.
Then, $M=\varinjlim_I P_i$ can be written as a countable filtered colimit of compact objects $\ob{D}(\Q_{p,\solid})$.
Each $P_i$ is represented by a finite complex of light Smith spaces.
In particular, by \cref{lemmaSmith}, this implies that each cohomology group $\pi_k(P_i)$ is a quotient of Smith spaces and so compact as a static solid $\Q_{p}$-module.
This implies that $\pi_k(M)=\varinjlim_i \pi_k(P_i)$ is a countable colimit of compact $\Q_p$-modules and so $\omega_1$-compact as a static solid $\Q_{p,\solid}$-module.
Conversely, let $M\in \ob{D}(\Q_{p,\solid})$ be such that $\pi_k(M)$ is $\omega_1$-compact as a static solid $\Q_p$-module for all $k\in \Z$.
First, by writing $M=\varinjlim_{n} \tau_{\geq -n} M_n$ as a countable colimit of its right truncations, we can assume without loss of generality that $M$ is connective.
In that case, as $\pi_0(M)$ is $\omega_1$-compact, we can find a countably family of light profinite sets $S_{0,j}$ and a map 
\[
Q_0:=\bigoplus_{j} \Q_{p,\solid}[S_j] \to M
\]
surjective on $\pi_0$.
Let $M_1$ be the cofiber of this map, then $M_1$ is $1$-connective and its cohomology groups are still $\omega_1$-compact (this follows because the kernel of $Q_0\to \pi_0(M)$ admits again a surjection by a countable sum of Smith spaces).
By an inductive argument we see that $M$ is represented by a complex of the form 
\[
\cdots \to Q_1\to Q_0\to 0
\]
where each $Q_i$ is a countable direct sum of compact projective solid $\Q_p$-modules. From this presentation, one can see that $M$ is $\omega_1$-compact proving what we wanted.
\end{proof}

We recall the following definition of nuclear modules. 

\begin{definition}\label{DefNuclearModules}
Let $A$ be a solid $\Q_p$-algebra. 
\begin{enumerate}

\item A morphism $N\to M$ in $\ob{D}(A)$ is \textit{trace class} if it lands in the image of the map 
\[
\pi_0(\iHom_A(N,A)\otimes_A M )(*)\to \pi_0 \Hom_A(N,M).  
\]

\item A solid $A$-module $M$ is \textit{nuclear} if for all compact $A$-modules $P$ the natural morphism 
\[
\iHom_A(P,A)\otimes_{A} M\to \iHom_A(P,M)
\]
is an equivalence. We let $\Cat{Nuc}(A)\subset \ob{D}(A)$ be the full subcategory of nuclear $A$-modules.

\item   An object $M\in \ob{D}(A)$ is called \textit{basic nuclear} if it is of the form $M=\varinjlim_{n} M_n$ where each $M_n\in \ob{D}(A)$ is compact and the maps $M_n\to M_{n+1}$ are trace class.  We let $\Cat{Nuc}^{\ob{basic}}(A)\subset \ob{D}(A)$ be the full subcategory of basic nuclear $A$-modules.

\end{enumerate}
\end{definition}

\begin{example}\label{ExampleBanach}
  Banach spaces are nuclear solid $\Q_p$-modules (see e.g. \cite[Corollary A.50]{bosco2021p}). More generally, Fr\'echet spaces over $\Q_p$ are nuclear solid $\Q_p$-modules (\cite[Proposition A.64]{bosco2021p}).
\end{example}

 Let us collect some properties of the category of (basic) nuclear modules.

\begin{lemma}\label{basicnuclemma}\
\begin{enumerate}
 \item The subcategory $\ob{Nuc}(A)\subset \ob{D}(A)$ is stable under colimits, and its $\omega_1$-compact objects are precisely the basic nuclear objects.
  \item The categories  $\Cat{Nuc}^{\ob{basic}}(A)\subset \Cat{Nuc}(A)\subset \ob{D}(A)$ are stable under tensor products.
  \item If $A\to B$ is a morphism of solid $\Q_p$-algebras and $N$ is a nuclear (resp.\ basic nuclear)  $A$-module, then  $B\otimes_A N$  is a nuclear (resp.\ basic nuclear) $B$-module.
\end{enumerate}
\end{lemma}
\begin{proof}
 For part (1) see \cite[Proposition 13.13]{scholze-analytic-spaces}, part (2) is \cite[Theorem  8.6]{clausenscholzecomplex}, and part (3) \cite[Proposition 2.3.22]{mann2022p}.
\end{proof}

\begin{lemma}\label{LemmaPermanence}
Let $A$ be a solid $\Q_p$-algebra which is nuclear as solid $\Q_p$-module and let $M\in \ob{D}(A)$. The following are equivalent: 
\begin{enumerate}

\item $M$ is a nuclear $A$-module.

\item $M$ is nuclear as a $\Q_p$-module.

\item $\pi_i(M)$ is a nuclear $A$-module for all $i\in \Z$.
\end{enumerate}
\end{lemma}
\begin{proof}
 See e.g. the proof of \cite[Theorem A.17(iii), Remark A.18]{bosco2023}.
\end{proof}

\begin{lemma}\label{LemmaCOuntableLimitNuclear}
Let $A$ be a nuclear solid $\Q_p$-algebra. Then $\Cat{Nuc}(A)\subset \ob{D}(A)$ is stable under countable limits.
\end{lemma}
\begin{proof}
See \cite[Theorem A.17(i)]{bosco2023}.
\end{proof}

\begin{lemma}\label{LemmaNuclearUppershierk}
Let $A$ be a  nuclear solid $\Q_p$-algebra and let $M$ be an $\omega_1$-compact  $A$-module. Then the functor
\[
\iHom_{A}(M,-)\colon \ob{D}(A)\to \ob{D}(A)
\]
preserves nuclear modules.
\end{lemma}
\begin{proof}
By assumption, we can write $M=\varinjlim_i P_i$ as a countable filtered colimit  of  compact projective $A$-modules. Then for $N\in \ob{Nuc}(A)$ we have that
\[
\iHom_{A}(M,N)=\varprojlim_i \iHom_A(P_i,N)=\varprojlim_i (\iHom_A(P_i,A)\otimes_{A} N),
\]
where  the first equivalence  is clear, and the second  follows from nuclearity of $N$. Since $A$ is a nuclear $\Q_p$-algebra, $\iHom_A(P_i,A)\otimes_{A} N$ is a nuclear $A$-module, and \cref{LemmaCOuntableLimitNuclear} implies that $\iHom_{A}(M,N)$ is nuclear proving what we wanted.
\end{proof}

The most important source of basic nuclear objects is the following class of solid $\Q_p$-modules:

\begin{definition}\label{defDNF}
A solid $\Q_p$-module is called  \textit{dual nuclear Fr\'echet space} if it can written as a countable filtered colimit of Smith spaces along injective trace class transition maps. If the Smith spaces in the system can be chosen to be light, then we say that it is a \textit{light dual nuclear Fr\'chet space}. For short, we will also refer to it as \textit{DNF space} (resp. \textit{light DNF space}).
\end{definition}

\begin{lemma}\label{LemmaDNF}
Let $M$ be a static and quasi-separated solid $\Q_p$-module. The following are equivalent:

\begin{enumerate}

\item  $M$ is basic nuclear.

\item $M$ is a light DNF space.

\end{enumerate}

Moreover, let $f\colon N\to M$ be a morphism of light DNF spaces. Then both $\ob{ker}(f)$ and $\ob{im}(f)$ are light DNF spaces.

\end{lemma}
\begin{proof}
  It is clear that any light DNF space is basic nuclear, static and quasi-separated solid $\Q_p$-module.
  Conversely, suppose that $M$ is a static, quasi-separated and basic nuclear solid $\Q_p$-module.
  We can write $M=\varinjlim_n P_n$ as a filtered colimit of compact projective $\Q_{p,\solid}$-modules along trace class transition maps. Since $M$ is quasi-separated, the kernel $Q_n=\ob{ker}(P_n\to M)$ is closed and, by \cref{lemmaSmith}, it is a compact projective $\Q_{p,\solid}$-module. In particular, there is some $m\gg n$ such that $P_n\to P_m$ factors through $P_n/Q_n $. Hence, after passing through a subsequence and taking quotients, we can assume that the transition maps in the colimit $M=\varinjlim_n P_n$ are injective and of trace class.
  This is precisely a light DNF space.

  For the last claim, let $f\colon N\to M$ be a morphism of light DNF spaces.
  Then $M,N$ are quasi-separated and the kernel $K\subset N$ is a closed subspace.
  By writing $N=\varinjlim_n P_n$ as a filtered colimit of injective trace class maps of compact projective objects, we see that 
\[
K=\varinjlim_{n} (P_n\cap K)
\]
and that $P_n\cap K\subset P_n$ is a compact projective module, by \cref{lemmaSmith}.
Thus, $K$ is $\omega_1$-compact.
Now,  by \cite[Corollary 3.38]{jacinto2021solid}  we can also write $N=\varinjlim_{n} B_n$ as a filtered colimit of Banach spaces along injective trace class maps, hence $K=\varinjlim_{n} B_n \cap K$ is also a filtered colimit of Banach spaces, making it nuclear.
We deduce that $K$ is both nuclear and $\omega_1$-compact, and so it is basic nuclear.
This implies that it is a light DNF space (using the proven equivalence of (1) and (2)) as it is quasi-separated.
Finally, $\ob{im}(f)=N/K$ is a quasi-separated and basic nuclear module, so it is light DNF as well.
\end{proof}

\begin{proposition}\label{PropBasicNuclearQp}
Let $M\in \ob{D}(\Q_{p,\solid})$, the following are equivalent:
\begin{enumerate}

\item $M$ is basic nuclear.

\item  $M$ can be represented by a complex whose terms are all light DNF spaces.

\item For each $i\in \Z$ the static $\Q_{p,\solid}$-module $\pi_i(M)$ is a quotient of light DNF spaces.

\end{enumerate}
\end{proposition}
\begin{proof}
  This is an analog of \cite[Theorem 8.15]{clausenscholzecomplex}.
  The implication (1) $\Rightarrow$ (2) follows from the same proofs of \cite[Lemma 8.7 and Proposition 8.9]{clausenscholzecomplex} since any compact object $P$ of $\ob{D}(\Q_{p,\solid})$ can be represented  as a complex $P^{\bullet}$ of compact projective modules of the form $\Q_{p,\solid}[S]$ with $S$ light profinite; here the key input of \cite[Lemma 8.7]{clausenscholzecomplex} is that $\iHom_{\Q_p}(\Q_{p,\solid}[S], \Q_p)=C(S,\Q_p)$ sits in degree $0$ and is flat.
  The first assertion is clear since $\Q_{p,\solid}[S]$ is projective, and for the second assertion see e.g. \cite[Corollary A.28]{bosco2021p}.

  The implication (2) $\Rightarrow$ (3) is obvious.
  Finally, for the implication (3) $\Rightarrow$ (1), note that  if $\pi_i(M)$ is a quotient of light DNF spaces, then it is basic nuclear by \cref{LemmaDNF}.
  Then combining \cref{LemmaPermanenceOmega1Compact}  (for $\omega_1$-compactness) and \cref{LemmaPermanence}  (for nuclearity) one deduces that $M$ is basic nuclear.
\end{proof}

Nuclear modules satisfy the following descent result. For the notion of $!$-equivalence used here, see \Cref{xhs92mk}.

\begin{lemma}\label{LemmaDescentNuclear}
Let $I$ be a countable $\infty$-category, and let $\varinjlim_{i\in I} \AnSpec(A_i)\to \mathrm{AnSpec}(A)$ be a $!$-equivalence of solid $\Q_p$-algebras such that $A$ and all $A_i$ are nuclear over $\Q_p$. Then the natural functor
\[
\Cat{Nuc}(A)\rightarrow \varprojlim_{i\in I} \Cat{Nuc}(A_{i})
\]
is an equivalence, where the transition maps are given by $*$-pullbacks.

Suppose that in addition that $A$ and all $A_i$ are basic nuclear $\Q_p$-algebras. Then the natural functor
\[
\Cat{Nuc}(A)\rightarrow \varprojlim_{i\in I} \Cat{Nuc}^!(A_{i})
\]
is an equivalence, where the transition maps are given by $!$-pullbacks. 
\end{lemma}

\begin{proof}
  Since we have descent for solid quasi-coherent sheaves (as we have a $!$-equivalence), and the nuclear categories are full subcategories of all solid modules (stable under $\ast$-pullback), we have an inclusion $\Cat{Nuc}(A)\subset \varprojlim_i \Cat{Nuc}(A_i)$.
  To see that this inclusion is essentially surjective, let $M_i$ be a cocartesian section in $\Cat{Nuc}(A_i)$.
  By \cref{LemmaPermanence} each $M_i$ is nuclear as $A$-module, and by \cref{LemmaCOuntableLimitNuclear} the limit of all $M_i$ remains nuclear being a countable limit of nuclear modules.

  For the claim about $!$-descent, by assumption all the $A_i$ are basic nuclear, and in particular they are $\omega_1$-compact as $A$-modules.
  By \cref{LemmaNuclearUppershierk} the upper $!$-functors $\iHom_{A}(A_i,-)$ preserve nuclear $A$-modules, and so the natural equivalence (coming from the input being a $!$-equivalence)
\[
\ob{D}(A)\xrightarrow{\sim }\varprojlim_{i\in I} \ob{D}^!(A_i)
\]
restricts to fully faithful functor 
\[
\Cat{Nuc}(A)\subset \varprojlim_{i\in I} \Cat{Nuc}^!(A_{i}).
\]
To see that this morphism is essentially surjective, it suffices to see that if $M_i$ is a cocartesian section of $\varprojlim_{i\in I} \Cat{Nuc}^!(A_i)$ then $\varinjlim_{i\in I^{\op}} M_{i}$ is a nuclear $A$-module. This follows again from \cref{LemmaCOuntableLimitNuclear}.
\end{proof}

\bibliographystyle{alpha}

\bibliography{biblio}

\begin{thebibliography}{ALBM24}

\bibitem[ABFJ22]{left-exact-loc}
Mathieu Anel, Georg Biedermann, Eric Finster, and Andr\'{e} Joyal.
\newblock Left-exact localizations of {$\infty$}-topoi {I}: {H}igher sheaves.
\newblock {\em Adv. Math.}, 400:Paper No. 108268, 64, 2022.

\bibitem[ALBM24]{ALBMFFCoho}
Johannes Ansch\"utz, Arthur-C\'esar Le~Bras, and Lucas Mann.
\newblock {A 6-functor formalism for solid quasi-coherent sheaves on the
  Fargues-Fontaine curve}.
\newblock {\em Preprint, arXiv:2412.20968}, 2024.

\bibitem[AM24]{AMdescendPerfd}
Johannes Ansch\"utz and Lucas Mann.
\newblock Descent for solid quasi-coherent sheaves on perfectoid spaces.
\newblock {\em Preprint, arXiv:2403.01951}, 2024.

\bibitem[And02]{Andre_monodromie}
Yves Andr\'{e}.
\newblock Filtrations de type {H}asse-{A}rf et monodromie {$p$}-adique.
\newblock {\em Invent. Math.}, 148(2):285--317, 2002.

\bibitem[And21]{andreychev2021pseudocoherent}
Grigory Andreychev.
\newblock {P}seudocoherent and perfect complexes and vector bundles on analytic
  adic spaces.
\newblock {\em Preprint, arXiv:2105.12591}, 2021.

\bibitem[Aok23]{aoki2023sheavesspectrumadjunction}
Ko~Aoki.
\newblock The sheaves-spectrum adjunction.
\newblock {\em Preprint, arXiv:2302.04069}, 2023.

\bibitem[Aok24]{aoki2024cohomology}
Ko~Aoki.
\newblock On cohomology of locally profinite sets.
\newblock {\em Preprint, arXiv:2411.05995}, 2024.

\bibitem[Bd11]{bhatt_dejong_crystalline_cohomology_and_de_rham_cohomology}
Bhargav {Bhatt} and Aise~Johan {de Jong}.
\newblock {{C}rystalline cohomology and de {R}ham cohomology}.
\newblock {\em arXiv e-prints}, page arXiv:1110.5001, Oct 2011.

\bibitem[Ber93]{berkovich1993etale}
Vladimir~G Berkovich.
\newblock {\'E}tale cohomology for non-{A}rchimedean analytic spaces.
\newblock {\em Publications Math{\'e}matiques de l'IH{\'E}S}, 78:5--161, 1993.

\bibitem[Ber02]{Berger_monodromie}
Laurent Berger.
\newblock Repr\'{e}sentations {$p$}-adiques et \'{e}quations
  diff\'{e}rentielles.
\newblock {\em Invent. Math.}, 148(2):219--284, 2002.

\bibitem[Bha17]{bhatt2017lecture}
Bhargav Bhatt.
\newblock Lecture notes for a class on perfectoid spaces.
\newblock Available at
  \url{https://www.math.ias.edu/~bhatt/teaching/mat679w17/lectures.pdf}, 2017.

\bibitem[Bha22]{FGauges}
Bhargav Bhatt.
\newblock Prismatic {$F$}-{Gauges}.
\newblock Available at
  \url{https://www.math.ias.edu/~bhatt/teaching/mat549f22/lectures.pdf}, 2022.

\bibitem[Bha23]{bhatt2023crystalschernclasses}
Bhargav Bhatt.
\newblock {C}rystals and {C}hern classes.
\newblock {\em Preprint, arXiv:2310.00676}, 2023.

\bibitem[Bos21]{bosco2021p}
Guido Bosco.
\newblock On the $p$-adic pro-\'{e}tale cohomology of {D}rinfeld symmetric
  spaces.
\newblock {\em Preprint, arXiv:2110.10683}, 2021.

\bibitem[Bos23]{bosco2023}
Guido Bosco.
\newblock Rational $p$-adic {H}odge theory for rigid-analytic varieties.
\newblock {\em Preprint, arXiv:2306.06100}, 2023.

\bibitem[BS17]{bhatt_scholze_projectivity_of_the_witt_vector_affine_grassmannian}
Bhargav Bhatt and Peter Scholze.
\newblock {P}rojectivity of the {W}itt vector affine {G}rassmannian.
\newblock {\em Invent. Math.}, 209(2):329--423, 2017.

\bibitem[BS22]{Bhatta}
Bhargav Bhatt and Peter Scholze.
\newblock {Prisms and prismatic cohomology}.
\newblock {\em Annals of Mathematics}, 196(3):1135 -- 1275, 2022.

\bibitem[CGN24]{colmez2024duality}
Pierre Colmez, Sally Gilles, and Wies{\l}awa Nizio{\l}.
\newblock Duality for $ p $-adic geometric pro-\'{e}tale cohomology {I}: a
  {F}argues-{F}ontaine avatar.
\newblock {\em Preprint, arXiv:2411.12163}, 2024.

\bibitem[CN25]{colmez_niziol_basic}
Pierre Colmez and Wies{\l}awa Nizio{\l}.
\newblock On the cohomology of {$p$}-adic analytic spaces, {I}: {T}he basic
  comparison theorem.
\newblock {\em J. Algebraic Geom.}, 34(1):1--108, 2025.

\bibitem[Col03]{colmez2003conjectures}
Pierre Colmez.
\newblock Les conjectures de monodromie $p$-adiques.
\newblock {\em Ast{\'e}risque-Soci{\'e}t{\'e} Math{\'e}matique de France},
  290:53--102, 2003.

\bibitem[Col08]{Colmez_deRham}
Pierre Colmez.
\newblock Espaces {V}ectoriels de dimension finie et repr\'{e}sentations de de
  {R}ham.
\newblock Number 319, pages 117--186. 2008.
\newblock Repr\'{e}sentations $p$-adiques de groupes $p$-adiques. I.
  Repr\'{e}sentations galoisiennes et $(\phi,\Gamma)$-modules.

\bibitem[Con99]{MR1697371}
Brian Conrad.
\newblock Irreducible components of rigid spaces.
\newblock {\em Ann. Inst. Fourier (Grenoble)}, 49(2):473--541, 1999.

\bibitem[CS22]{clausenscholzecomplex}
Dustin {Clausen} and Peter {Scholze}.
\newblock {L}ectures on {C}omplex {G}eometry.
\newblock Available at
  \url{https://people.mpim-bonn.mpg.de/scholze/Complex.pdf}, 2022.

\bibitem[CS24]{AnStacks}
Dustin Clausen and Peter Scholze.
\newblock {Lectures on Analytic Stacks}.
\newblock Available at
  \url{https://www.youtube.com/playlist?list=PLx5f8IelFRgGmu6gmL-Kf_Rl_6Mm7juZO},
  2024.

\bibitem[DK24]{dauser2024uniqueness}
Adam Dauser and Josefien Kuijper.
\newblock Uniqueness of six-functor formalisms.
\newblock {\em Preprint, arXiv:2412.15780}, 2024.

\bibitem[FF18]{Farguesf}
Laurent Fargues and Jean-Marc Fontaine.
\newblock {C}ourbes et fibr\'{e}s vectoriels en th\'{e}orie de {H}odge
  {$p$}-adique.
\newblock {\em Ast\'{e}risque}, (406):xiii+382, 2018.

\bibitem[FS21]{fargues2021geometrization}
Laurent Fargues and Peter Scholze.
\newblock {G}eometrization of the local {L}anglands correspondence.
\newblock {\em Preprint, arXiv:2102.13459}, 2021.

\bibitem[Ger24]{gerth2024hodge}
Lucas Gerth.
\newblock A {H}odge--{T}ate decomposition with rigid analytic coefficients.
\newblock {\em Preprint, arXiv:2411.07366}, 2024.

\bibitem[GK00]{grosse2000rigid}
Elmar Grosse-Kl\"{o}nne.
\newblock Rigid analytic spaces with overconvergent structure sheaf.
\newblock {\em J. Reine Angew. Math.}, 519:73--95, 2000.

\bibitem[GR03]{gabber2003almost}
Ofer Gabber and Lorenzo Ramero.
\newblock {\em {A}lmost ring theory}, volume 1800 of {\em Lecture Notes in
  Mathematics}.
\newblock Springer-Verlag, Berlin, 2003.

\bibitem[Gro90]{Gros1990}
Michel Gros.
\newblock Régulateurs syntomiques et valeurs de fonctions {$L$} $p$-adiques
  {I}.
\newblock {\em Inventiones mathematicae}, 99(2):293--320, 1990.

\bibitem[Guo21]{GuoCrystalline}
Haoyang Guo.
\newblock Crystalline cohomology of rigid analytic spaces.
\newblock {\em Preprint, arXiv:2112.14304}, 2021.

\bibitem[HM24]{heyer20246functorformalismssmoothrepresentations}
Claudius Heyer and Lucas Mann.
\newblock {6-Functor Formalisms and Smooth Representations}.
\newblock {\em Preprint, arXiv:2410.13038}, 2024.

\bibitem[Ked04]{Kedlaya_monodromie}
Kiran~S. Kedlaya.
\newblock A $p$-adic local monodromy theorem.
\newblock {\em Ann. Math. (2)}, 160(1):93--184, 2004.

\bibitem[KL15]{kedlaya_liu_relative_p_adic_hodge_theory_foundations}
Kiran~S. Kedlaya and Ruochuan Liu.
\newblock {R}elative {$p$}-adic {H}odge theory: {F}oundations.
\newblock {\em Ast\'{e}risque}, (371):239, 2015.

\bibitem[KL16]{KedlayaLiuRelativeII}
Kiran~S. Kedlaya and Ruochuan Liu.
\newblock Relative p-adic {Hodge} theory, {II}: {Imperfect} period rings.
\newblock Preprint, {arXiv}:1602.06899 [math.{NT}] (2016), 2016.

\bibitem[LB18]{lebras2018}
Arthur-César Le~Bras.
\newblock Overconvergent relative de {R}ham cohomology over the
  {F}argues-{F}ontaine curve.
\newblock {\em Preprint, arXiv:1801.00429}, 2018.

\bibitem[LBV23]{le2023rham}
Arthur-C{\'e}sar Le~Bras and Alberto Vezzani.
\newblock The de {R}ham--{F}argues--{F}ontaine cohomology.
\newblock {\em Algebra \& Number Theory}, 17(12):2097--2150, 2023.

\bibitem[Lur09]{lurie_higher_topos_theory}
Jacob Lurie.
\newblock {\em {H}igher {T}opos {T}heory ({A}{M}-170)}, volume 189.
\newblock Princeton University Press, 2009.

\bibitem[Lur17]{lurie_higher_algebra}
Jacob Lurie.
\newblock {H}igher {A}lgebra.
\newblock Available at \url{https://www.math.ias.edu/~lurie/}, 2017.

\bibitem[LZ17]{LiuZhu_RiemannHilbert}
Ruochuan Liu and Xinwen Zhu.
\newblock Rigidity and a {R}iemann-{H}ilbert correspondence for {$p$}-adic
  local systems.
\newblock {\em Invent. Math.}, 207(1):291--343, 2017.

\bibitem[Man22]{mann2022p}
Lucas Mann.
\newblock A $p$-adic 6-{F}unctor {F}ormalism in {R}igid-{A}nalytic {G}eometry.
\newblock {\em Preprint, arXiv:2206.02022}, 2022.

\bibitem[Mat16]{mathew2016galois}
Akhil Mathew.
\newblock The {G}alois group of a stable homotopy theory.
\newblock {\em Advances in Mathematics}, 291:403--541, 2016.

\bibitem[Meb02]{Mebkhout_monodromie}
Z.~Mebkhout.
\newblock Analogue {$p$}-adique du th\'{e}or\`eme de {T}urrittin et le
  th\'{e}or\`eme de la monodromie {$p$}-adique.
\newblock {\em Invent. Math.}, 148(2):319--351, 2002.

\bibitem[Mik25]{mikami2025finiteness}
Yutaro Mikami.
\newblock Finiteness and duality of cohomology of {$(\varphi,\Gamma)$}-modules
  and the 6-functor formalism of locally analytic representations.
\newblock {\em Preprint, arXiv:2504.01780}, 2025.

\bibitem[PP24]{poineau2024convergence}
J{\'e}r{\^o}me Poineau and Andrea Pulita.
\newblock The convergence {N}ewton polygon of a $ p $-adic differential
  equation {V}: local index theorems.
\newblock {\em Preprint, arXiv:1309.3940}, 2024.

\bibitem[RC24a]{camargo2024analytic}
Juan~Esteban Rodr\'iguez~Camargo.
\newblock The analytic de {R}ham stack in rigid geometry.
\newblock {\em Preprint, arXiv:2401.07738}, 2024.

\bibitem[RC24b]{SolidNotes}
Juan~Esteban Rodr\'iguez~Camargo.
\newblock Notes on solid geometry.
\newblock Avilable at
  \url{https://sites.google.com/view/jerodriguezcamargo/courses-and-seminars?authuser=0},
  2024.

\bibitem[RC25]{course_algebraic_d_modules}
Juan~Esteban Rodr\'iguez~Camargo.
\newblock Notes on {$D$}-modules via derived algebraic stacks.
\newblock Available at
  \url{https://sites.google.com/view/jerodriguezcamargo/courses-and-seminars?authuser=0},
  2025.

\bibitem[RJRC]{AnDModRJRC}
Joaquin Rodrigues~Jacinto and Juan~Esteban Rodr\'iguez~Camargo.
\newblock Analytic {$D$}-modules and solid locally analytic representations.
\newblock In preparation.

\bibitem[RJRC21]{jacinto2021solid}
Joaquin Rodrigues~Jacinto and Juan~Esteban Rodr\'iguez~Camargo.
\newblock {S}olid locally analytic representations of $ p $-adic {L}ie groups.
\newblock {\em Preprint, arXiv:2110.11916}, 2021.

\bibitem[Sch18a]{scholze_etale_cohomology_of_diamonds}
Peter Scholze.
\newblock {\'E}tale cohomology of diamonds.
\newblock {\em Preprint, arXiv:1709.07343}, 2018.

\bibitem[Sch18b]{scholze_lectures_on_condensed_mathematics}
Peter Scholze.
\newblock {L}ectures on condensed mathematics.
\newblock Available at
  \url{https://www.math.uni-bonn.de/people/scholze/Condensed.pdf}, 2018.

\bibitem[Sch20]{scholze-analytic-spaces}
Peter Scholze.
\newblock Lectures on {Analytic} {Geometry}.
\newblock Available at
  \url{https://www.math.uni-bonn.de/people/scholze/Analytic.pdf}, 2020.

\bibitem[Sch23]{scholze6functors}
Peter Scholze.
\newblock {Six-Functor Formalisms}.
\newblock Available at
  \url{https://people.mpim-bonn.mpg.de/scholze/SixFunctors.pdf}, 2023.

\bibitem[Sch24a]{scholze2024berkovichmotives}
Peter Scholze.
\newblock {Berkovich Motives}.
\newblock {\em Preprint, arXiv:2412.03382}, 2024.

\bibitem[Sch24b]{ScholzeRLL}
Peter Scholze.
\newblock {Geometrization of the real Local Langlands correspondence: draft
  version}.
\newblock Available at
  \url{https://people.mpim-bonn.mpg.de/scholze/RealLocalLanglands.pdf}, 2024.

\bibitem[Sch25]{scholze2025geometrization}
Peter Scholze.
\newblock Geometrization of the local {L}anglands correspondence, motivically.
\newblock {\em Preprint, arXiv:2501.07944}, 2025.

\bibitem[Shi18]{Shimizu_constancy}
Koji Shimizu.
\newblock Constancy of generalized {H}odge-{T}ate weights of a local system.
\newblock {\em Compos. Math.}, 154(12):2606--2642, 2018.

\bibitem[Shi22]{Shimizu_monodromy}
Koji Shimizu.
\newblock A {$p$}-adic monodromy theorem for de {R}ham local systems.
\newblock {\em Compos. Math.}, 158(12):2157--2205, 2022.

\bibitem[{Sta}25]{stacks-project}
The {Stacks project authors}.
\newblock The stacks project.
\newblock \url{https://stacks.math.columbia.edu}, 2025.

\bibitem[SW20]{scholze2020berkeley}
Peter Scholze and Jared Weinstein.
\newblock {B}erkeley lectures on $p$-adic geometry.
\newblock In {\em Berkeley Lectures on $p$-adic Geometry}. Princeton University
  Press, 2020.

\bibitem[Tam15]{TammeLazardIso}
Georg Tamme.
\newblock On an analytic version of {L}azard's isomorphism.
\newblock {\em Algebra Number Theory}, 9(4):937--956, 2015.

\bibitem[Tsu98]{zbMATH01199189}
Nobuo Tsuzuki.
\newblock Slope filtration of quasi-unipotent overconvergent {$F$}-isocrystals.
\newblock {\em Ann. Inst. Fourier}, 48(2):379--412, 1998.

\bibitem[Win83]{wintenberger1983corps}
Jean-Pierre Wintenberger.
\newblock Le corps des normes de certaines extensions infinies de corps locaux;
  applications.
\newblock In {\em Annales scientifiques de l'Ecole Normale Superieure},
  volume~16, pages 59--89, 1983.

\bibitem[Zav23]{zavyalov2023poincaredualityabstract6functor}
Bogdan Zavyalov.
\newblock Poincar\'e duality in abstract 6-functor formalisms.
\newblock {\em Preprint, arXiv:2301.03821}, 2023.

\bibitem[Zel24]{zelich2024faithfully}
Ivan Zelich.
\newblock Faithfully flat ring maps are not descendable.
\newblock {\em Preprint, arXiv:2405.08124}, 2024.

\end{thebibliography}


\end{document}